\documentclass{article}
\usepackage[T1]{fontenc}
\usepackage[latin1]{inputenc}
\usepackage{amsthm, fullpage}
\usepackage{amsmath,enumerate}
\usepackage{amssymb}
\usepackage{amsfonts}
\usepackage{eucal}
\usepackage[all]{xy}
\usepackage{hyperref}

\usepackage{mathrsfs}

\DeclareMathOperator{\Mod}{Mod}
\DeclareMathOperator{\coh}{Coh}

\DeclareMathOperator{\ind}{Ind}
\DeclareMathOperator{\indcoh}{IndCoh}
\DeclareMathOperator{\Ker}{Ker}
\DeclareMathOperator{\Coker}{Coker}
\DeclareMathOperator{\Hom}{Hom}

\newtheorem{prop}{Proposition}[subsection]
\newtheorem{theo}[prop]{Theorem}
\newtheorem{coro}[prop]{Corollary}
\newtheorem{lemm}[prop]{Lemma}
\newtheorem{lem}[prop]{Lemma}
\newtheorem*{lemm*}{Lemma}

\newtheorem{exs}[prop]{Examples} 

\theoremstyle{definition}

\newtheorem{empt}[prop]{}
\newtheorem{dfn}[prop]{Definition}
\newtheorem{rem}[prop]{Remark}
\newtheorem{ntn}[prop]{Notation}
\newtheorem{ex}[prop]{Example}
\newtheorem*{rem*}{Remark}

\theoremstyle{thm}
\newtheorem{thm}[prop]{Theorem}
\newtheorem*{thm*}{Theorem}
\newtheorem*{lem*}{Lemma}
\newtheorem{cor}[prop]{Corollary}
\newtheorem*{cor*}{Corollary}
\newtheorem*{prop*}{Proposition}

\theoremstyle{dfn}
\newtheorem*{dfn*}{Definition}

\numberwithin{equation}{prop}

\newcommand{\riso}{ \overset{\sim}{\longrightarrow}\, }
\newcommand{\liso}{ \overset{\sim}{\longleftarrow}\, }

\newcommand{\Spec}{\mathrm{Spec}\,}
\newcommand{\Spf}{\mathrm{Spf}\,}

\renewcommand{\sp}{\mathrm{sp}}
\renewcommand{\det}{\mathrm{det}}

\renewcommand{\AA}{{\mathcal{A}}}

\newcommand{\cR}{{\mathcal{R}}}
\newcommand{\cA}{{\mathcal{A}}}
\newcommand{\cB}{{\mathcal{B}}}

\newcommand{\cD}{{\mathcal{D}}}
\newcommand{\cE}{{\mathcal{E}}}
\newcommand{\cF}{{\mathcal{F}}}
\newcommand{\cG}{{\mathcal{G}}}
\newcommand{\cH}{{\mathcal{H}}}

\newcommand{\cM}{{\mathcal{M}}}
\newcommand{\FF}{{\mathcal{F}}}
\newcommand{\B}{{\mathcal{B}}}

\newcommand{\E}{{\mathcal{E}}}
\newcommand{\G}{{\mathcal{G}}}
\renewcommand{\H}{{\mathcal{H}}}
\newcommand{\M}{{\mathcal{M}}}
\newcommand{\cN}{{\mathcal{N}}}
\newcommand{\D}{{\mathcal{D}}}
\newcommand{\I}{{\mathcal{I}}}
\newcommand{\cI}{{\mathcal{I}}}

\newcommand{\cL}{{\mathcal{L}}}

\renewcommand{\O}{{\mathcal{O}}}
\newcommand{\cO}{{\mathcal{O}}}

\newcommand{\fC}{{\mathfrak{C}}}
\newcommand{\fD}{{\mathfrak{D}}}
\newcommand{\fM}{{\mathfrak{M}}}
\newcommand{\fP}{{\mathfrak{P}}}
\newcommand{\fQ}{{\mathfrak{Q}}}
\newcommand{\fS}{{\mathfrak{S}}}
\newcommand{\fX}{{\mathfrak{X}}}
\newcommand{\fY}{{\mathfrak{Y}}}
\newcommand{\fU}{{\mathfrak{U}}}
\newcommand{\fZ}{{\mathfrak{Z}}}

\newcommand{\fA}{\mathfrak{A}}
\newcommand{\fB}{\mathfrak{B}}

\newcommand{\V}{\mathcal{V}}
\newcommand{\cV}{\mathcal{V}}
\newcommand{\W}{\mathcal{W}}

\renewcommand{\S}{\mathfrak{S}}
\newcommand{\T}{{\mathfrak{T}}}
\newcommand{\fT}{{\mathfrak{T}}}
\newcommand{\Y}{\mathfrak{Y}}
\newcommand{\ZZ}{\mathfrak{Z}}
\newcommand{\X}{\mathfrak{X}}

\newcommand{\fm}{\mathfrak{m}}
\newcommand{\U}{\mathfrak{U}}

\newcommand{\cP}{{\mathcal{P}}}
\newcommand{\cQ}{{\mathcal{Q}}}

\newcommand{\bbA}{\mathbb{A}}
\newcommand{\bbP}{\mathbb{P}}

\newcommand{\DD}{\mathbb{D}}
\renewcommand{\L}{\mathbb{L}}
\newcommand{\R}{\mathbb{R}}
\newcommand{\Q}{\mathbb{Q}}
\newcommand{\bbQ}{\mathbb{Q}}
\newcommand{\Z}{\mathbb{Z}}
\newcommand{\bbZ}{\mathbb{Z}}
\newcommand{\N}{\mathbb{N}}
\newcommand{\bbL}{\mathbb{L}}

\newcommand{\hdag}{  \phantom{}{^{\dag} }    }

\def\debrom{
\makeatletter
\renewcommand{\theenumi}{(\roman{enumi})}
\renewcommand{\labelenumi}{\theenumi}
\makeatother\begin{enumerate}}
\def\finrom{\end{enumerate}}

\begin{document}

\title{Arithmetic $\D$-modules over algebraic varieties of characteristic $p>0$}
\author{Daniel Caro}
\date{}

\maketitle

\tableofcontents

\bigskip
\begin{abstract}
Let $k$ be a field of characteristic $p>0$ not necessarily perfect.
Using Berthelot's theory of arithmetic $\D$-modules,
we construct a $p$-adic formalism of Grothendieck's six operations for realizable $k$-schemes of finite type. 
\end{abstract}

\section*{Introduction}

Let $\V$ be a complete discrete valuation ring of mixed characteristic $(0,p)$, 
$\pi$ be a uniformizer,
$k:= \V/\pi \V$ be its residue field and $K$ be its fraction field. 
In order to build a $p$-adic formalism of Grothendieck six operations for $k$-varieties (i.e. separated $k$-schemes of finite type), 
Berthelot introduced an arithmetic avatar of the theory of modules over the differential operators ring. 
The objects appearing in his theory are called arithmetic $\D$-modules or complexes of arithmetic $\D$-modules (for an introduction,
see \cite{Beintro2}). 

In the case where $k$ is perfect, within Berthelot's arithmetic $\D$-modules theory,
such a $p$-adic formalism was already known in different contexts. Let us describe these known cases. 
With N. Tsuzuki  (see \cite{caro-Tsuzuki}),
the author got such a formalism  for overholonomic $F$-complexes of arithmetic $\D$-modules 
(i.e. complexes together with a Frobenius structure)
over realizable $k$-varieties (i.e. 
$k$-varieties which can be embedded into a proper formal $\V$-scheme). 
Another example was given later (do not focus on the publication date) with 
holonomic $F$-complexes of arithmetic $\D$-modules over quasi-projective varieties (\cite{caro-stab-holo}).
In a wider geometrical context, T. Abe 
established a six functors formalism for admissible stacks, 
namely algebraic stacks of finite type with finite diagonal morphism
(see \cite[2.3]{Abe-Langlands}). The starting point of his work was the case of quasi-projective $k$-varieties. 
Again, some Frobenius structures are involved in his construction. 
Finally, without Frobenius structure, in \cite{caro-unip}, 
we explained how to build 
such a $p$-adic formalism of Grothendieck's six functors, e.g. with quasi-unipotent complexes of arithmetic $\D$-modules (see \cite{caro-unip}).

In this paper, we construct such a $p$-adic formalism for some realizable $k$-scheme of finite type, 
without the hypothesis of perfectness on $k$. 
Similarly to T. Abe in 
\cite[2.3]{Abe-Langlands},
we might a priori expect to extend from a six functors formalism for quasi-projective varieties,
a six functors formalism for admissible stacks.
But, we will focus our study to realizable $k$-varieties. 
We check in this paper that we can extend (with some slight changes) from what was made in the perfect case in \cite{caro-unip}
to the general case. 
Moreover, we give a second construction of a $p$-adic formalism of Grothendieck six functors (see below in the introduction for more details). 
A lot of preliminary results which were stated with the perfectness condition
are still valid without any change in the proof. In this case, most of the time (when we do not have found a better proof) 
we will only refer to some published proofs. 
We have also improved some results, e.g. the relative duality isomorphism of the form \ref{rel-dual-isom} was not known in this context
for coherent complexes. These improvements sometimes simplify the presentation.
We have also given some simpler proofs of known results. For instance,
the base change isomorphism is a consequence of some basic properties of the exterior tensor products whose study here 
is given with details. 
For completeness, 
it also happens that we incorporate 
some unpublished proofs of Berthelot 
(e.g. see \ref{letterBerthelotCaro2007} which was a letter of Berthelot to the author in 2007).
Since these results are scattered in the literature, 
this paper might at least help the reader to better understand the 
steps (in the right order) in order to get a six functors formalism in Berthelot's theory of arithmetic $\D$-modules.
Even if we tried to make this paper as self-contained as possible (at least for the statements), 
we have used freely \cite{Be1},
\cite{Be2}. Moreover, if this is not already the case,
we advise the reader to read \cite{Beintro2} which contains a lot of fundamental properties that 
we sometimes use (in that case, we have tried to give each time some reference).
\bigskip

Let us clarify the content of the paper. 
Let  $\fP$ be a separated smooth formal $\V$-scheme (for the $p$-adic topology).
The special fiber of $\fP$, the $k$-variety equal to its reduction modulo $\pi$, is denoted by $P$.
In the first chapter we recall Berthelot's notion of derived categories of inductive systems of arithmetic $\D$-modules on $\fP$.
Some objects in theses categories will give our coefficients satisfying a six functors formalism. 
Two Berthelot's notions are fundamental in theses categories :  that of quasi-coherence and that of coherence. 
In the second chapter, we study the localization functor outside a divisor $T$ of $P$ and the forgetful functor of a divisor $T$ of $P$. 
We check both functors preserve the quasi-coherence.
Next, we give a coherence stability criterion involving a change of divisors
which is one fundamental property of the theory (see \ref{limTouD}).
In the third chapter, we give the construction of the extraordinary inverse image $f ^!$ by a morphism $f$, 
the pushforward $f _+$ and the duality. 
We also recall the stability of the coherence under the pushforward by a proper morphism
and under the pullback by a smooth morphism. 
The forth chapter deals with the relative duality isomorphism and the adjoint paire 
$(f _+, f ^!)$ for a proper morphism $f$, which was the work of Virrion. In fact, we also study and prove
with many details the case where the morphism $f$ is a closed immersion by using the fundamental local isomorphism.
The adjunction morphisms in this case are very explicit and will be used in the fifth chapter.

Let 
$T$ be the (support of a) divisor of $P$,
let $X$ be a smooth closed subscheme of $P$ of  codimension $r$, 
 $Y := X \setminus T$. 
In the fifth chapter, we define the notion of coherent arithmetic $\D$-modules 
over $(Y,X)/\V$ (see \ref{rem-ind-ariYXP}). 
Roughly speaking, first choose $(X _\alpha)$ an affine open covering of $X$, 
and for each $\alpha$ choose a smooth formal $\V$-scheme $\X _\alpha$ which is a lifting of $X _\alpha$.
Then, a coherent arithmetic $\D$-module 
over $(Y,X)/\V$ is the data of a family of coherent arithmetic $\D$-module on $\fX _\alpha$
with overconvergent singularities along $T \cap X _\alpha$ together with glueing isomorphisms
satisfying a cocycle condition.
We check that we have a canonical equivalence of categories between that
of coherent arithmetic $\D$-modules 
over $(Y,X)/\V$ and that of coherent arithmetic $\D$-modules on $\fP$
with overconvergent singularities along $T$ and support in $X$
(see Theorem \ref{prop1}).
This extends Berthelot's theorem of his arithmetic version of Kashiwara theorem
appearing in the classical  $\D$-modules theory.
In the sixth chapter, 
we construct a fully faithful functor denoted by $\sp _+$ from the category of overconvergent isocrystals over $(Y,X)/\V$
to that of arithmetic $\D$-modules over $(Y,X)/\V$. Its essential image has an explicit description (see \ref{ntn-dfnsp+}). 
The functor $\sp _+$ is some kind of pushforward under $\sp \colon \fP _K \to \fP$, 
the specialisation morphism from the rigid analytic variety associated to $\fP$.
In the seventh chapter, we study 
the standard properties of  exterior tensor products and some consequences. 
In the arithmetic $\D$-modules case, this is slightly more technical since for instance the base is not always a field 
and then exterior tensor products are not always exact (see \ref{rem-ext-prod}).
In the eighth chapter, we adapt Berthelot's proof of the coherence of the constant coefficient with overconvergent singularities
$\O _{\fP} (\hdag T) _{\Q}$. The first step is to check explicitly this coherence when 
$T$ is a strict smooth crossing divisor (more precisely, see \ref{NCDgencoh}). 
In the general context, using de Jong's desingularisation theorem, we reduce by descent to the SNCD case. 
Two key ingredients which makes possible the descent are the isomorphism
$\sp _+ (\smash{\O} _{]X[ _{\fP}}) 
\riso 
\mathcal{H} ^r \R \sp _* \underline{\Gamma} ^\dag _X ( \O _{\fP _K})$ (see \ref{coro-sp+jdagO}),
and the commutation of $\sp _+$ with the duality.

In the ninth chapter, we introduce the
local cohomological functor with strict support over a subvariety $Y$ of $P$ that we denote by 
$\R \underline{\Gamma} ^\dag _Y$. 
Roughly speaking, to define the functor
$\R \underline{\Gamma} ^\dag _Y$
we reduce by devissage 
to the case where $Y$ is open. 
Again by devissage, we reduce to the case where $Y$ is the complementary of a divisor $T$ of $P$.
In this later case, $\R \underline{\Gamma} ^\dag _Y$ is 
the  localisation outside the divisor $T$ functor (that was defined previously in the second chapter). We should also clarify that 
one key ingredient to be able to define this local cohomological functor is the coherence of the constant coefficient. 
In the next chapter, 
we check that the expected properties satisfied by local cohomological functors are still valid, e.g. its commutation 
with pushforwards and extraordinary pullbacks. We also check some base change isomorphism (see \ref{theo-iso-chgtbase})
and a relative duality isomorphism that extends that of Virrion (we replace the properness hypothesis of the morphism $f$
by the properness via $f$ of the support of our complexes).

In the eleventh chapter, we adapt the construction given in  \cite{caro-unip}
of a formalism of Grothendieck six functors. 
For instance, we change the notion of data of coefficients
by replacing the base $\V$ by some fixed perfectification of $\V \to \V ^\flat$
(see the definition \ref{DVR}) in the sense of \ref{special-filtered-pre}. 
It seems unavoidable to fix such a perfectification.
We also introduce a notion of {\it special descent of the base}. 
Such notion allows us to use de Jong's desingularization theorem 
as if the base field was perfect. 
In fact, we improve the construction given in  \cite{caro-unip}
since we also consider the stability under the cohomological functors
$\cH ^{r}$ which allows us to endow 
our triangulated categories with a canonical t-structure.
Via Theorem \ref{dfnquprop} and the example \ref{ex-datastableevery},
we explain how to build a data of coefficient 
which contains the constant coefficient, 
which is 
local,
stable under devissages, direct summands, 
local cohomological functors, 
pushforwards, extraordinary pullbacks, 
base change, tensor products, duals,
cohomology
and
special descent of the base.
In order to get such stable data but which moreover contains convergent isocrystals 
on smooth $k$-varieties, we propose a second construction 
which uses a little more external tensor products (see Theorem \ref{dfnqupropbis}). 
Recall that since we do not have some Frobenius structures, 
then it is impossible to put every overconvergent isocrystals in our categories. 

Finally, in the last chapter, we get a 
$p$-adic formalism of Grothendieck six operations over couples of $k$-varieties $(Y,X)$ which 
can be enclosed into a frame of the form
$(Y,X,\fP)$ where $\fP$ is a realizable (i.e. which can be embedded into a proper smooth scheme over $\V$) 
smooth formal scheme over $\V$, 
$X$ is a closed subscheme of the special fiber of $\fP$ and $Y$ is an open of $X$.
For an enough stable data of coefficients $\fC$, a coefficient 
over $(Y,X,\fP)$ is a coefficient over $\fP$ with support in $X$ and having overconvergent singularities along $X \setminus Y$
(i.e. which is isomorphic under its image via $\R \underline{\Gamma} ^\dag _{X \setminus Y}$).
We prove the independence with respect to the choice of the frame 
enclosing $(Y,X)$ of the coefficients of such $\fC$
over $(Y,X,\fP)$  (\ref{ind-CYW}). 
When $\fC$ is stable under cohomology, 
this independence preserves t-structures. 
When $X$ is proper over $k$, then 
the coefficients of such $\fC$
over $(Y,X)$ is independent (up to canonical equivalence of categories) 
of the choice of such proper varieties $X$ enclosing $Y$. 
This yields a formalism of Grothendieck's six operations over $k$-varieties 
(which can be enclosed into a frame).
 
\subsection*{Acknowledgment}
The author was supported by the IUF.

\section*{Notation}

Let $\V$ be a complete discrete valuation ring of mixed characteristic $(0,p)$, 
$\pi$ be a uniformizer,
$k:= \V/\pi \V$ be its residue field and $K$ its fraction field. 
We set $\S := \Spf \V$, the $p$-adic formal scheme (in the sense of Grothendieck's terminology of EGA I.10) 
associated with $\V$.
A formal $\V$-scheme $\X$
or formal $\S$-scheme $\X$ means a 
$p$-adic formal scheme endowed with 
a structural morphism of $p$-adic formal schemes
$\X \to \Spf \V$.
Schemes and formal schemes are supposed to be separated and quasi-compact.
Formal schemes will be denoted with gothic letters and their special fiber with the associated roman letter. 
Unless otherwise stated, 
a subscheme of the special fiber of a formal $\V$-scheme
are always supposed to be reduced.

Sheaves will be denoted with calligraphic letters and their global sections with the associated straight letter. 
By default, a module means a left module. 
We denote $p$-adic completions with some hat   and if 
$\E$ is an abelian sheaf of groups, we set $\E _{\Q}:= \E \otimes _{\Z} \Q$. 
Let $\cR$ be a commutative sheaf on $X$.
By convention, $\cR$-algebras are always unital and associative. 
Let  $\AA, \cB$ be $\cR$-algebras. 
Unless otherwise stated, an $\cA$-module is a left $\cA$-module
and an $(\cA,\cB)$-bimodule is an $\cR$-module endowed with a compatible structure of left $\cA$-module and right $\cB$-module.
If $*$ is one of the symboles $+$, $-$, or $\mathrm{b}$, 
$D ^* ( \AA )$ (resp. $D ^* ( \AA,\cB )$) means the derived category of the complexes of 
$\AA$-modules (resp. of $( \AA,\cB )$-bimodules) satisfying the corresponding condition of vanishing of cohomological spaces. 
When we would like to clarify between right and left,  we will write
$D ^* ( {}^l \AA )$ or $D ^* ( {}^r \AA )$
(resp. $D ^* (  {}^l \cA, {}^l \cB )$ or $D ^* ({}^r \cA, {}^r \cB )$).
We denote by $D ^{\mathrm{b}} _{\mathrm{coh}} ( \AA )$
the subcategory of  $D  ( \AA )$
of bounded and coherent complexes.

When 
$f \colon \X \to \cP $
is a smooth morphism of formal  $\V$-schemes,
for any integer $i \in \N$,
we denote by  $ f _{i}   \colon X _i \to P _i$
the induced morphism modulo $\pi ^{i+1}$.

\section{Derived categories of inductive systems of arithmetic $\D$-modules}
\label{ntn-tildeD(Z)}

Let $\fP $ be a smooth formal scheme over $\S $
and $T$ be a divisor of $P$.
Divisors of $P$ will be supposed to be reduced divisors (in our context, this is not really less general). 
Remark that since $P$ is regular, then Weil divisors correspond to Cartiel divisors. Hence, in our context, 
a divisor is determined by its irreducible components. 
To reduce the amount of notation, we set 
$\smash{\widehat{\D}} _{\fP /\S } ^{(m)} (T):=
\widehat{\B} ^{(m)} _{\fP } ( T)  \smash{\widehat{\otimes}} _{\O _{\fP}} \smash{\widehat{\D}} _{\fP /\S } ^{(m)}$, 
where $\widehat{\B} ^{(m)} _{\fP} ( T) $ is the sheaf constructed in
\cite[4.2]{Be1}
and
$\smash{\D} _{\fP /\S } ^{(m)}$ is the sheaf of differential operators of level $m$ over $\fP /\S $
(see \cite[2.2]{Be1}).
We fix  $\lambda _0\colon \N \to \N$ an increasing map such that 
$\lambda _{0} (m) \geq m$ for any $m \in \N$. 
We set 
$\widetilde{\B} ^{(m)} _{\fP} ( T):= \widehat{\B} ^{(\lambda _0 (m))} _{\fP} ( T)$ 
et
$\smash{\widetilde{\D}} _{\fP /\S } ^{(m)} (T):=
\widetilde{\B} ^{(m)} _{\fP} ( T)  \smash{\widehat{\otimes}} _{\O _{\fP}} \smash{\widehat{\D}} _{\fP /\S } ^{(m)}$.
Finally, we set 
$\smash{\D} _{P  _i /S  _i} ^{(m)} (T):= \V / \pi ^{i+1} \otimes _{\V} \smash{\widehat{\D}} _{\fP /\S  } ^{(m)} (T) 
=
\B ^{(m)} _{P _i} ( T)  \otimes _{\O _{P _i}} \smash{\D} _{P  _i/S  _i} ^{(m)}$
and
$\smash{\widetilde{\D}} _{P  _i/S  _i} ^{(m)} (T):=\widetilde{\B} ^{(m)} _{P _i} ( T)  \otimes _{\O _{P _i}} \smash{\D} _{P  _i/S  _i} ^{(m)}$.

\subsection{Localisation of derived categories of inductive systems of arithmetic $\D$-modules}

\begin{empt}
[Berthelot's localized categories of the form $\smash{\underrightarrow{LD}} _{\Q}$]
\label{loc-LM}
We recall below some constructions of Berthelot of \cite[4.2.1 and 4.2.2]{Beintro2} 
which are still valid by adding singularities along a divisor. 
We have the inductive system of rings 
$\smash{\widetilde{\D}} _{\fP /\S } ^{(\bullet)}(T) 
: =
(\smash{\widetilde{\D}} _{\fP /\S } ^{(m)}(T) )_{m\in \N}$. 
We get the derived categories
$D ^{\sharp} ( \smash{\widetilde{\D}} _{\fP /\S } ^{(\bullet)}(T))$,  
where $\sharp \in \{\emptyset, +,-, \mathrm{b}\}$.
The objects of $D ^{\sharp}( \smash{\widetilde{\D}} _{\fP /\S } ^{(\bullet)}(T))$
are denoted by
 $\E ^{(\bullet)}= (\E ^{(m)} , \alpha ^{(m',m)})$, 
where $m,m'$ run over non negative integers such that  $m' \geq m$,
where $\E ^{(m)} $ is a complex of $\smash{\widetilde{\D}} _{\fP /\S } ^{(m)}(T)$-modules
and $\alpha ^{(m',m)} \colon \E ^{(m)}\to \E ^{(m')}$ are $\smash{\widetilde{\D}} _{\fP /\S } ^{(m)}(T)$-linear morphisms.

\begin{itemize}
\item Let $M$ bet the filtrant set (endowed with the canonical order) 
of increasing maps $\chi \colon \N \to \N$. 
For any map $\chi \in M$, we set
$\chi ^{*} (\E ^{(\bullet)}) := (\E ^{(m)} , p ^{\chi (m') -\chi (m)}\alpha ^{(m',m)})$.
We obtain the functor
$\chi ^{*} \colon D ( \smash{\widetilde{\D}} _{\fP /\S } ^{(\bullet)}(T))\to D ( \smash{\widetilde{\D}} _{\fP /\S } ^{(\bullet)}(T))$ 
as follows:
if $f ^{(\bullet)} \colon \E ^{(\bullet)} \to \FF ^{(\bullet)}$ is a morphism of $D ( \smash{\widetilde{\D}} _{\fP /\S } ^{(\bullet)}(T))$, 
then the morphism of level $m$ of  $\chi ^{*} f ^{(\bullet)} 
\colon 
\chi ^{*} (\E ^{(\bullet)}) 
\to 
\chi ^{*} (\FF ^{(\bullet)})$ 
is $f ^{(m)}$.
If $\chi _1, \chi _2 \in M$, we compute
$\chi _1 ^* \circ \chi _2 ^* = (\chi _1 +\chi _2)^*$, and in particular $\chi _1 ^*$ and $\chi _2 ^*$ commute.
Moreover, if  $\chi _1 \leq \chi _2$, then we get the morphism 
$\chi _1 ^* ( \E ^{(\bullet)})\to \chi _2 ^* ( \E ^{(\bullet)})$ defined at the level $m$ by  
$p ^{\chi _2 (m) -\chi _1(m)}\colon \E ^{(m)} \to \E ^{(m)}$. 
A morphism $f ^{(\bullet)} \colon \E ^{(\bullet)} \to \FF ^{(\bullet)}$ of $D ( \smash{\widetilde{\D}} _{\fP /\S } ^{(\bullet)}(T))$
is an {``ind-isogeny''} if there exist  $\chi \in M$ 
and a morphism
$g ^{(\bullet)} \colon \FF ^{(\bullet)} \to \chi ^{*} \E ^{(\bullet)}$ of $D ( \smash{\widetilde{\D}} _{\fP /\S } ^{(\bullet)}(T))$
such that 
$g ^{(\bullet)}\circ f ^{(\bullet)}$ and $\chi ^{*} (f ^{(\bullet)}) \circ g ^{(\bullet)}$ 
are the canonical morphisms described above (in the case $\chi _1 =0$ and $\chi _2=\chi$).
The subset of ind-isogenies is a multiplicative system (this follows from Proposition
\cite[I.4.2]{HaRD} and the analogue of Lemma \cite[1.1.2]{caro-stab-sys-ind-surcoh} still valid without the hypothesis that $k$ is perfect). 
The localisation of $D ^{\sharp} ( \smash{\widetilde{\D}} _{\fP /\S } ^{(\bullet)}(T))$
with respect to ind-isogenies is denoted by
$\smash{\underrightarrow{D}} ^{\sharp} _{\Q} ( \smash{\widetilde{\D}} _{\fP /\S } ^{(\bullet)}(T))$.

\item Let $L$ be the filtrant set of increasing maps 
$\lambda \colon \N \to \N$ such that $\lambda (m ) \geq m$.
For any $\lambda \in L$, we put
$\lambda ^{*} (\E ^{(\bullet)}) := (\E ^{(\lambda(m))} , \alpha ^{(\lambda(m'),\lambda(m))})_{m'\geq m}$.
When $\lambda _1, \lambda _2 \in L$, we compute
$\lambda _1 ^* \circ \lambda _2 ^* = (\lambda _1 \circ\lambda _2)^*$.
When $\lambda _1 \leq \lambda _2$, we have the canonical morphism 
$\lambda _1 ^* (\E ^{(\bullet)}) \to 
\lambda _2 ^* (\E ^{(\bullet)})$ defined at the level $m$ by the morphism
$\alpha ^{(\lambda _2(m),\lambda _1(m))} \colon 
\E ^{(\lambda _1(m))} \to \E ^{(\lambda _2(m))}$. 
Similarly to \cite[4.2.2]{Beintro2}, 
we denote by $\Lambda ^{\sharp}$ the set of morphisms
$f ^{(\bullet)} \colon \E ^{(\bullet)} \to \FF ^{(\bullet)}$
of 
$\smash{\underrightarrow{D}} ^{\sharp} _{\Q}  ( \smash{\widetilde{\D}} _{\fP /\S } ^{(\bullet)}(T))$
such that there exist $\lambda \in L$ and a morphism
$g ^{(\bullet)} \colon \FF ^{(\bullet)} \to \lambda ^{*} \E ^{(\bullet)}$ of $\smash{\underrightarrow{D}} _{\Q} ( \smash{\widetilde{\D}} _{\fP /\S } ^{(\bullet)}(T))$
such that 
the morphisms 
$g ^{(\bullet)}\circ f ^{(\bullet)}$ 
and $\lambda ^{*} (f ^{(\bullet)}) \circ g ^{(\bullet)}$ 
of $\smash{\underrightarrow{D}} ^{\sharp} _{\Q} ( \smash{\widetilde{\D}} _{\fP /\S } ^{(\bullet)}(T))$
are the canonical morphisms (i.e. we take $\lambda _1 =id$ and $\lambda _2=\lambda$).
The morphisms belonging to $\Lambda $ are called {``lim-isomorphisms''}.
We check that  $\Lambda ^{\sharp} $ is a multiplicative system
(again, use  \cite[I.4.2]{HaRD} and the analogue of Lemma \cite[1.1.2]{caro-stab-sys-ind-surcoh}). 
By localizing 
$\smash{\underrightarrow{D}}  ^{\sharp}_{\Q}
( \smash{\widetilde{\D}} _{\fP /\S } ^{(\bullet)}(T))$
with respect to  lim-isomorphisms 
we get a category denoted by
$\smash{\underrightarrow{LD}} ^{\sharp} _{\Q}
( \smash{\widetilde{\D}} _{\fP /\S } ^{(\bullet)}(T))$.

\item Let  $\chi _1\leq \chi _2 $ in $M$ and $\lambda _1 \leq \lambda _2$ in $L$.
We get by composition the canonical morphism
$\lambda  _1^{*} \chi  _1^{*} \to \lambda _2^{*} \chi _2^{*}$.
By considering  $\chi _1 \circ \lambda _1$ as an element of $M$, we get the equality
$\lambda _1 ^* \chi _1 ^* = (\chi _1 \circ \lambda _1) ^* \lambda _1 ^*$.
Let  $S ^{\sharp}$ be the set of morphisms 
$f ^{(\bullet)} \colon \E ^{(\bullet)} \to \FF ^{(\bullet)}$
of 
$D ^{\sharp} ( \smash{\widetilde{\D}} _{\fP /\S } ^{(\bullet)}(T))$
such that there exist $\chi \in M$, 
$\lambda \in L$ and a morphism
$g ^{(\bullet)} \colon \FF ^{(\bullet)} \to \lambda ^{*} \chi ^{*}\E ^{(\bullet)}$ of $D ( \smash{\widetilde{\D}} _{\fP /\S } ^{(\bullet)}(T))$
such that 
$g ^{(\bullet)}\circ f ^{(\bullet)}$ and $\lambda ^{*} \chi ^{*}(f ^{(\bullet)}) \circ g ^{(\bullet)}$ 
are the canonical morphisms. 
The elements of   $S ^{\sharp}$ are called {`` lim-ind-isogenies''}.
We check as usual that $S ^{\sharp}$ is a multiplicative system.

\end{itemize}
\end{empt}

\begin{empt}
\label{S=LDQ}
Similarly to \cite[1.1.5]{caro-stab-sys-ind-surcoh},
we check the canonical equivalence of categories
$S ^{\sharp -1} D ^{\sharp}
( \smash{\widetilde{\D}} _{\fP /\S } ^{(\bullet)}(T))
\cong
\smash{\underrightarrow{LD}} ^{\sharp} _{\Q} 
( \smash{\widetilde{\D}} _{\fP /\S } ^{(\bullet)}(T))$,
which is the identity over the objects.

\end{empt}

\begin{empt}
\label{HomLDQ}
Similarly to \cite[1.1.6]{caro-stab-sys-ind-surcoh},
for any  $\E ^{(\bullet)}, \FF ^{(\bullet)} \in \underrightarrow{LD} ^{\sharp} _{\Q} (\smash{\widetilde{\D}} _{\fP /\S } ^{(\bullet)} (T))$,
we have the equality
\begin{equation}
\label{4.2.2Beintro}
\mathrm{Hom} _{\underrightarrow{LD} ^{\sharp} _{\Q} (\smash{\widetilde{\D}} _{\fP /\S } ^{(\bullet)} (T))}
(\E ^{(\bullet)}, \FF ^{(\bullet)} )
=
\underset{\lambda \in L}{\underrightarrow{\lim}}~
\underset{\chi \in M}{\underrightarrow{\lim}}~
\mathrm{Hom} _{D ^{\sharp} (\smash{\widetilde{\D}} _{\fP /\S } ^{(\bullet)} (T))}
(\E ^{(\bullet)}, \lambda ^{*} \chi ^{*}\FF ^{(\bullet)} ).
\end{equation}

\end{empt}

\subsection{Point of view of a derived category of an abelian category}

\begin{empt}
\label{defi-M-L}
We denote by $M (\smash{\widetilde{\D}} _{\fP /\S } ^{(\bullet)} (T))$ the category of
$\smash{\widetilde{\D}} _{\fP /\S } ^{(\bullet)} (T)$-modules.
The $\smash{\widetilde{\D}} _{\fP /\S } ^{(\bullet)} (T)$-modules
are denoted by 
$\E ^{(\bullet)}= (\E ^{(m)} , \alpha ^{(m',m)})$, 
where $m,m'$ run through non negative integers  $m' \geq m$,
where $\E ^{(m)} $ is a $\smash{\widetilde{\D}} _{\fP /\S } ^{(m)}(T)$-module
and $\alpha ^{(m',m)} \colon \E ^{(m)}\to \E ^{(m')}$ are
$\smash{\widetilde{\D}} _{\fP /\S } ^{(m)}(T)$-linear morphisms.
For any  $\chi \in M$, we denote similarly to \ref{loc-LM} the object
$\chi ^{*} (\E ^{(\bullet)}) := (\E ^{(m)} , p ^{\chi (m') -\chi (m)}\alpha ^{(m',m)})$.
In fact, we get the functor
$\chi ^{*}\colon M (\smash{\widetilde{\D}} _{\fP /\S } ^{(\bullet)} (T)) \to M (\smash{\widetilde{\D}} _{\fP /\S } ^{(\bullet)} (T))$. 
Moreover, similarly to \ref{loc-LM},
for any $\lambda \in L$, we set 
$\lambda ^{*} (\E ^{(\bullet)}) := (\E ^{(\lambda(m))} , \alpha ^{(\lambda(m'),\lambda(m))})$.

Similarly to 
\ref{loc-LM},
we define the notion of ind-isogenies  (resp. of lim-ind-isogenies) of
$M (\smash{\widetilde{\D}} _{\fP /\S } ^{(\bullet)} (T))$
and we denote by 
$\smash{\underrightarrow{M}} _{\Q}  ( \smash{\widetilde{\D}} _{\fP /\S } ^{(\bullet)}(T))$
(resp. $S ^{-1}M (\smash{\widetilde{\D}} _{\fP /\S } ^{(\bullet)} (T))$)
the localization by  ind-isogenies (resp. par les lim-ind-isogenies). 
We define also the multiplicative system of lim-isomorphisms of 
$\smash{\underrightarrow{M}} _{\Q}  ( \smash{\widetilde{\D}} _{\fP /\S } ^{(\bullet)}(T))$
and we denote by 
$\underrightarrow{LM} _{\Q} (\smash{\widetilde{\D}} _{\fP /\S } ^{(\bullet)} (T))$
the corresponding localized category. 

\end{empt}

\begin{empt}
The results of \cite[1.2.1]{caro-stab-sys-ind-surcoh} are still valid in our context:
we check the canonical equivalence of categories
$S ^{-1}M (\smash{\widetilde{\D}} _{\fP /\S } ^{(\bullet)} (T))
\cong 
\underrightarrow{LM} _{\Q} (\smash{\widetilde{\D}} _{\fP /\S } ^{(\bullet)} (T))$.
Moreover, for any $\E ^{(\bullet)}, \FF ^{(\bullet)} \in \underrightarrow{LM} _{\Q} (\smash{\widetilde{\D}} _{\fP /\S } ^{(\bullet)} (T))$ 
\begin{equation}
\label{4.2.2BeintroLMQ}
\mathrm{Hom} _{\underrightarrow{LM} _{\Q} (\smash{\widetilde{\D}} _{\fP /\S } ^{(\bullet)} (T))}
(\E ^{(\bullet)}, \FF ^{(\bullet)} )
=
\underset{\lambda \in L}{\underrightarrow{\lim}}\;
\underset{\chi \in M}{\underrightarrow{\lim}}\;
\mathrm{Hom} _{M (\smash{\widetilde{\D}} _{\fP /\S } ^{(\bullet)} (T))}
(\E ^{(\bullet)}, \lambda ^{*} \chi ^{*}\FF ^{(\bullet)} ).
\end{equation}

\end{empt}

\begin{lemm}
\label{defi-M-L-Serre}
The category $\underrightarrow{LM} _{\Q} (\smash{\widetilde{\D}} _{\fP /\S } ^{(\bullet)} (T))$ is abelian 
and the multiplicative system of  lim-ind-isogenies of
$M (\smash{\widetilde{\D}} _{\fP /\S } ^{(\bullet)} (T))$ is saturated.
\end{lemm}

\begin{proof}
We can copy the proof of \cite[1.2.4]{caro-stab-sys-ind-surcoh}.
\end{proof}

\begin{empt}
We denote by  $M ( \smash{\D} ^\dag _{\fP } (\hdag T) _{\Q} )$ 
the abelian category of 
$\smash{\D} ^\dag _{\fP } (\hdag T) _{\Q}$-modules.
By tensoring by $\Q$ and next by applying the inductive limit on the level, we get the functor
$\underrightarrow{\lim} 
\colon
M
(\smash{\widetilde{\D}} _{\fP /\S } ^{(\bullet)}(T))
\to
M ( \smash{\D} ^\dag _{\fP } (\hdag T) _{\Q} )$.
Since this functor sends a lim-ind-isomorphism to an  isomorphism, it factorizes canonically through the functor 
\begin{equation}
\label{M-eq-lim}
\underrightarrow{\lim} 
\colon
\smash{\underrightarrow{LM}}  _{\Q}
(\smash{\widetilde{\D}} _{\fP /\S } ^{(\bullet)}(T))
\to
M ( \smash{\D} ^\dag _{\fP } (\hdag T) _{\Q} ).
\end{equation}
Similarly, we get 
\begin{equation}
\label{D-eq-lim}
\underrightarrow{\lim} 
\colon
\smash{\underrightarrow{LD}}  ^\mathrm{b} _{\Q}
(\smash{\widetilde{\D}} _{\fP /\S } ^{(\bullet)}(T))
\to
D ^\mathrm{b} ( \smash{\D} ^\dag _{\fP } (\hdag T) _{\Q} ).
\end{equation}
\end{empt}

\begin{prop}
\label{eqcatLD=DSM}
The canonical functor 
$D ^{\mathrm{b}} (\smash{\widetilde{\D}} _{\fP /\S } ^{(\bullet)} (T))
\to
D ^{\mathrm{b}}
(\underrightarrow{LM} _{\Q} (\smash{\widetilde{\D}} _{\fP /\S } ^{(\bullet)} (T)))$
of triangulated categories
induced by the functor of abelian categories
$M(\smash{\widetilde{\D}} _{\fP /\S } ^{(\bullet)} (T))
\to
\underrightarrow{LM} _{\Q} (\smash{\widetilde{\D}} _{\fP /\S } ^{(\bullet)} (T))$
factorizes canonically through the equivalence of triangulated categories
\begin{equation}
\label{eqcatLD=DSM-fonct}
\underrightarrow{LD} ^{\mathrm{b}} _{\Q} (\smash{\widetilde{\D}} _{\fP /\S } ^{(\bullet)} (T))
\cong 
D ^{\mathrm{b}}
(\underrightarrow{LM} _{\Q} (\smash{\widetilde{\D}} _{\fP /\S } ^{(\bullet)} (T))).
\end{equation}

\end{prop}

\begin{proof}
We can copy the proof of \cite[1.2.11]{caro-stab-sys-ind-surcoh}.
\end{proof}

\begin{empt}
\label{empt-diag-Hn-comp}
The equivalence \ref{eqcatLD=DSM-fonct} commutes with cohomological functors, 
i.e. we have for any $n\in \N$ the commutative diagram
\begin{equation}
\label{diag-Hn-comp}
\xymatrix{
{D ^{\mathrm{b}} (\smash{\widetilde{\D}} _{\fP /\S } ^{(\bullet)} (T))} 
\ar[r] ^-{}
\ar[d] ^-{\H ^n}
&
{\underrightarrow{LD} ^{\mathrm{b}} _{\Q} (\smash{\widetilde{\D}} _{\fP /\S } ^{(\bullet)} (T))} 
\ar[r] ^-{\cong}
\ar@{.>}[d] ^-{\H ^n}
& 
{D ^{\mathrm{b}} (\underrightarrow{LM}  _{\Q} (\smash{\widetilde{\D}} _{\fP /\S } ^{(\bullet)} (T))) } 
\ar[d] ^-{\H ^n}
\\ 
{M(\smash{\widetilde{\D}} _{\fP /\S } ^{(\bullet)} (T))} 
\ar[r] ^-{}
& 
{\underrightarrow{LM}  _{\Q} (\smash{\widetilde{\D}} _{\fP /\S } ^{(\bullet)} (T))} 
\ar@{=}[r] ^-{}
& 
{\underrightarrow{LM}  _{\Q} (\smash{\widetilde{\D}} _{\fP /\S } ^{(\bullet)} (T))} 
}
\end{equation}
where the middle vertical arrow is the one making 
commutative by definition the left square (see \cite[1.2.6]{caro-stab-sys-ind-surcoh}).
\end{empt}

\subsection{Coherence}

Similarly to \cite[2.2.1]{caro-stab-sys-ind-surcoh}, we have the following definition. 
\begin{dfn}
[Coherence up to lim-ind-isogeny]
\label{coh-loc-pres-fini-lim-ind-iso}
Let $\E ^{(\bullet)}$ be a
$\smash{\widetilde{\D}} _{\fP /\S } ^{(\bullet)} (T)$-module.
The module 
$\E ^{(\bullet)}$
is said to be a
$\smash{\widetilde{\D}} _{\fP /\S } ^{(\bullet)} (T)$-module 
of finite type up to  lim-ind-isogeny
if there exists an open covering 
$(\fP _i ) _{i\in I}$ of  $\fP$ 
such that, for any $i \in I$, 
there exists an exact sequence of 
$\underrightarrow{LM} _{\Q}  (\smash{\widetilde{\D}} _{\fP /\S } ^{(\bullet)} (T))$ of the form: 
$\left ( \smash{\widetilde{\D}} _{\fP  _i} ^{(\bullet)} (T  \cap P _i) \right) ^{r _i}
\to 
\E ^{(\bullet)} | \fP _i
\to 0$,
where $r _i\in \N$.
Similarly, we get the notion of
$\smash{\widetilde{\D}} _{\fP /\S } ^{(\bullet)} (T)$-module 
locally of finite presentation up to  lim-ind-isogeny 
(resp. coherence up lim-ind-isogeny). 
\end{dfn}

\begin{ntn}
\label{nota-(L)Mcoh}
We denote by 
 $\underrightarrow{LM} _{\Q, \mathrm{coh}} (\smash{\widetilde{\D}} _{\fP /\S } ^{(\bullet)} (T))$
the full subcategory of 
$\underrightarrow{LM} _{\Q} (\smash{\widetilde{\D}} _{\fP /\S } ^{(\bullet)} (T))$
consisting of 
coherent $\smash{\widetilde{\D}} _{\fP /\S } ^{(\bullet)} (T)$-modules up to lim-ind-isogeny.
\end{ntn}

\begin{prop}
\label{LQ-coh-stab}
The full subcategory 
$\underrightarrow{LM} _{\Q, \mathrm{coh}} ( \smash{\widetilde{\D}} _{\fP /\S } ^{(\bullet)} (T))$
of 
$\underrightarrow{LM} _{\Q}  (\smash{\widetilde{\D}} _{\fP /\S } ^{(\bullet)} (T))$
is stable by isomorphisms, 
kernels, cokernels, extensions.
\end{prop}

\begin{proof}
We can copy the proof of \cite[2.2.8]{caro-stab-sys-ind-surcoh}.
\end{proof}

\begin{ntn}
\label{ntn-cohDLM}
For any $\sharp  \in \{ 0,+,-, \mathrm{b}, \emptyset\}$, 
we denote by  
$D  ^\sharp _{\mathrm{coh}}
(\underrightarrow{LM} _{\Q} (\smash{\widetilde{\D}} _{\fP /\S } ^{(\bullet)} (T)))$
the full subcategory of 
$D ^\sharp
(\underrightarrow{LM} _{\Q} (\smash{\widetilde{\D}} _{\fP /\S } ^{(\bullet)} (T)))$
consisting of complexes $\E ^{(\bullet)}$ such that, for any $n \in \Z$, 
$\mathcal{H} ^{n} (\E ^{(\bullet)}) \in 
\underrightarrow{LM} _{\Q, \mathrm{coh}} (\smash{\widetilde{\D}} _{\fP /\S } ^{(\bullet)} (T))$
(see notation  \ref{nota-(L)Mcoh}). 
These objects are called coherent complexes of 
$D ^\sharp
(\underrightarrow{LM} _{\Q} (\smash{\widetilde{\D}} _{\fP /\S } ^{(\bullet)} (T)))$.
\end{ntn}

\begin{empt}
\label{cohDLMislocal}
By definition, the property that an object of 
$\underrightarrow{LM} _{\Q}  (\smash{\widetilde{\D}} _{\fP /\S } ^{(\bullet)} (T))$
is an object of 
$\underrightarrow{LM} _{\Q, \mathrm{coh}} ( \smash{\widetilde{\D}} _{\fP /\S } ^{(\bullet)} (T))$
is local in $\fP$. 
This yields that the notion of coherence of \ref{ntn-cohDLM} is local in  $\fP$, i.e. 
the fact that a complex $\E ^{(\bullet)}$ of 
$D ^\sharp
(\underrightarrow{LM} _{\Q} (\smash{\widetilde{\D}} _{\fP /\S } ^{(\bullet)} (T)))$
is coherent is local.
\end{empt}

\begin{dfn}
[Coherence in the sense of Berthelot]
\label{defi-LDQ0coh}
Let  $\sharp  \in \{\emptyset, +,-, \mathrm{b}\}$.
Let
$\E ^{(\bullet)} \in \smash{\underrightarrow{LD}} ^{\sharp} _{\Q}  ( \smash{\widetilde{\D}} _{\fP /\S } ^{(\bullet)}(T))$.
The complex
$\E ^{(\bullet)}$
is said to be coherent if 
there exist
$\lambda \in L$
and
$\FF ^{(\bullet)}\in \smash{\underrightarrow{LD}} _{\Q} ^{\sharp} (\lambda ^{*} \smash{\widetilde{\D}} _{\fP /\S } ^{(\bullet)}(T))$ 
together with an  isomorphism in 
$ \smash{\underrightarrow{LD}} _{\Q} ^{\sharp}  ( \smash{\widetilde{\D}} _{\fP /\S } ^{(\bullet)}(T))$ of the form
$\E ^{(\bullet)} \riso \FF ^{(\bullet) }$,
such that $\FF ^{(\bullet)}$ satisfies the following conditions:
\begin{enumerate}
\item For any $m \in \N$, $\FF ^{(m)} \in D  _{\mathrm{coh}} ^{\sharp} (\smash{\widetilde{\D}} _{\fP /\S } ^{(\lambda (m))} (T))$ ; 
\item For any  $0\leq m \leq m'$,  
the canonical morphism
\begin{equation}
\label{Beintro-4.2.3M}
\smash{\widetilde{\D}} _{\fP /\S } ^{(\lambda (m'))} (T) \otimes ^\L _{\smash{\widetilde{\D}} _{\fP /\S } ^{(\lambda (m))} (T)}
\FF ^{(m)} \to \FF ^{(m')} 
\end{equation}
is an isomorphism.
\end{enumerate}
\end{dfn}

\begin{ntn}
\label{nota-LDQ0coh}
Let $\sharp \in \{\emptyset, +,-, \mathrm{b}\}$.
We denote by  $\underrightarrow{LD}  ^{\sharp} _{\Q, \mathrm{coh}} (\smash{\widetilde{\D}} _{\fP /\S } ^{(\bullet)} (T))$
the strictly full subcategory of 
$\underrightarrow{LD}  ^{\sharp} _{\Q} (\smash{\widetilde{\D}} _{\fP /\S } ^{(\bullet)} (T))$
consisting of coherent complexes. 
\end{ntn}

\begin{prop}
\label{eqcat-limcoh}
\begin{enumerate}
\item The functor  \ref{M-eq-lim} induces the equivalence of categories
\begin{equation}
\label{M-eq-coh-lim}
\underrightarrow{\lim} 
\colon
\smash{\underrightarrow{LM}}  _{\Q, \mathrm{coh}}
(\smash{\widetilde{\D}} _{\fP /\S } ^{(\bullet)}(T))
\cong
\mathrm{Coh} ( \smash{\D} ^\dag _{\fP } (\hdag T) _{\Q} ),
\end{equation}
where $\mathrm{Coh} ( \smash{\D} ^\dag _{\fP } (\hdag T) _{\Q} )$ 
is the category of coherent
$\smash{\D} ^\dag _{\fP } (\hdag T) _{\Q}$-modules.

\item The functor 
\ref{D-eq-lim} induces the equivalence of triangulated categories
\begin{equation}
\label{eqcatcoh}
\underrightarrow{\lim} 
\colon 
D ^{\mathrm{b}} _{\mathrm{coh}} (\underrightarrow{LM} _{\Q} (\smash{\widetilde{\D}} _{\fP /\S } ^{(\bullet)} (T)))
\cong
D ^{\mathrm{b}} _{\mathrm{coh}}( \smash{\D} ^\dag _{\fP } (\hdag T) _{\Q} ).
\end{equation}

\item The equivalence of triangulated categories 
$\underrightarrow{LD} ^{\mathrm{b}} _{\Q} (\smash{\widetilde{\D}} _{\fP /\S } ^{(\bullet)} (T))
\cong 
D ^{\mathrm{b}}
(\underrightarrow{LM} _{\Q} (\smash{\widetilde{\D}} _{\fP /\S } ^{(\bullet)} (T)))$
of 
\ref{eqcatLD=DSM-fonct}
induces the equivalence of triangulated categories
\begin{equation}
\label{eqcatLD=DSM-fonct-coh}
\underrightarrow{LD} ^{\mathrm{b}} _{\Q, \mathrm{coh}} (\smash{\widetilde{\D}} _{\fP /\S } ^{(\bullet)} (T))
\cong
D ^{\mathrm{b}} _{\mathrm{coh}}
(\underrightarrow{LM} _{\Q} (\smash{\widetilde{\D}} _{\fP /\S } ^{(\bullet)} (T))).
\end{equation}
\end{enumerate}
\end{prop}

\begin{proof}
We can copy the proof of Theorems \cite[2.4.4, 2.5.7]{caro-stab-sys-ind-surcoh}.
\end{proof}

\begin{empt}
\label{thick-subcat}
\begin{enumerate}
\item 
\label{thick-subcat1}
Using \ref{LQ-coh-stab}, 
we get that 
$D ^{\mathrm{b}} _{\mathrm{coh}}
(\underrightarrow{LM} _{\Q} (\smash{\widetilde{\D}} _{\fP /\S } ^{(\bullet)} (T)))$ 
is a thick triangulated subcategory (some authors say saturated or épaisse) of 
$D ^{\mathrm{b}} 
(\underrightarrow{LM} _{\Q} (\smash{\widetilde{\D}} _{\fP /\S } ^{(\bullet)} (T)))$,
i.e. is a strict triangulated subcategory closed under direct summands. 
Hence, using \ref{eqcatLD=DSM-fonct} and 
\ref{eqcatLD=DSM-fonct-coh}, we get that 
$\underrightarrow{LD} ^{\mathrm{b}} _{\Q, \mathrm{coh}} (\smash{\widetilde{\D}} _{\fP /\S } ^{(\bullet)} (T))$
is a thick triangulated subcategory of 
$\underrightarrow{LD} ^{\mathrm{b}} _{\Q} (\smash{\widetilde{\D}} _{\fP /\S } ^{(\bullet)} (T))$. 
\item 
\label{thick-subcat2}
Using the same arguments, 
it follows from \ref{cohDLMislocal} the following local property : the fact that 
a complex of 
$\underrightarrow{LD} ^{\mathrm{b}} _{\Q} (\smash{\widetilde{\D}} _{\fP /\S } ^{(\bullet)} (T))$
is a coherent complex (i.e. 
a complex of 
$\underrightarrow{LD} ^{\mathrm{b}} _{\Q,\mathrm{coh}} (\smash{\widetilde{\D}} _{\fP /\S } ^{(\bullet)} (T))$)
is local in $\fP$. 

\end{enumerate}
\end{empt}

\begin{empt}
For any $n \in \N$, 
the cohomological functor 
$\H ^n \colon 
{\underrightarrow{LD} ^{\mathrm{b}} _{\Q} (\smash{\widetilde{\D}} _{\fP /\S } ^{(\bullet)} (T))}
\to 
{\underrightarrow{LM}  _{\Q} (\smash{\widetilde{\D}} _{\fP /\S } ^{(\bullet)} (T))} $
of \ref{diag-Hn-comp}
induces 
$\H ^n \colon 
{\underrightarrow{LD} ^{\mathrm{b}} _{\Q,\mathrm{coh}} (\smash{\widetilde{\D}} _{\fP /\S } ^{(\bullet)} (T))}
\to 
{\underrightarrow{LM}  _{\Q,\mathrm{coh}} (\smash{\widetilde{\D}} _{\fP /\S } ^{(\bullet)} (T))} $
and we have the commutative diagram (up to canonical isomorphism)
\begin{equation}
\label{diag-Hn-comp-coh}
\xymatrix{
{D ^{\mathrm{b}} (\underrightarrow{LM}  _{\Q,\mathrm{coh}} (\smash{\widetilde{\D}} _{\fP /\S } ^{(\bullet)} (T))) } 
\ar[d] ^-{\H ^n}
\ar[r] ^-{}
& 
{D ^{\mathrm{b}} _{\mathrm{coh}}(\underrightarrow{LM}  _{\Q} (\smash{\widetilde{\D}} _{\fP /\S } ^{(\bullet)} (T))) } 
\ar[d] ^-{\H ^n}
& 
{\underrightarrow{LD} ^{\mathrm{b}} _{\Q,\mathrm{coh}} (\smash{\widetilde{\D}} _{\fP /\S } ^{(\bullet)} (T))} 
\ar[d] ^-{\H ^n}
\ar[l] ^-{\cong}
\\ 
{\underrightarrow{LM}  _{\Q,\mathrm{coh}} (\smash{\widetilde{\D}} _{\fP /\S } ^{(\bullet)} (T))}
\ar@{=}[r] ^-{}
& 
{\underrightarrow{LM}  _{\Q,\mathrm{coh}} (\smash{\widetilde{\D}} _{\fP /\S } ^{(\bullet)} (T))}
\ar@{=}[r] ^-{}
& 
{\underrightarrow{LM}  _{\Q,\mathrm{coh}} (\smash{\widetilde{\D}} _{\fP /\S } ^{(\bullet)} (T)).}
}
\end{equation}
Indeed, the commutativity of the left square is obvious 
and that of the right one is almost tautological (see the commutative diagram \ref{diag-Hn-comp}).
\end{empt}

\subsection{Indcoherence}
We denote by 
$\D ^{(m)}$ either 
$\widetilde{\D} ^{(m)}_{\fP /\S } (T)$ 
or 
$\widetilde{\D} ^{(m)}_{\fP /\S } (T) _\Q$.
We denote by 
$\D $ either 
$\widetilde{\D} ^{(m)}_{\fP /\S } (T)$ 
or 
$\widetilde{\D} ^{(m)}_{\fP /\S } (T) _\Q$
or 
$\D ^\dag _{\fP /\S } (\hdag T) _\Q$.
We put 
$D^{(m)} := \Gamma (\fP, \D ^{(m)})$,
$D:= \Gamma (\fP, \D)$.

\begin{empt}
\label{IndCoh}
We denote by $\Mod ( \D)$ (resp. $\coh ( \D)$)
the abelian category of left $\D$-modules (resp. coherent left $\D$-modules).
We denote by $\iota \colon \coh ( \D) \to \Mod ( \D)$
the canonical fully faithful functor. 
Since 
$\Mod ( \D)$ admits small filtrant inductive limits,
from \cite[6.3.2]{Kashiwara-schapira-book}
we get a functor denoted by 
 $J \iota  \colon \ind ( \coh ( \D) )\to \Mod ( \D)$
 such that $J \iota$ commutes with small filtrant inductive limits and the composition 
$ \coh ( \D)\to 
  \ind ( \coh ( \D) )\to \Mod ( \D)$
  is isomorphic to $\iota$. 
  
Let $\E \in  \coh ( \D) $. 
 For any functor $\alpha \colon I \to  \Mod ( \D) $ with $I$ small and filtrant,
 the natural morphism 
 $\underrightarrow{\lim} \Hom _{\D} ( \E,\alpha)
 \to 
\Hom _{\D}( \E,  \underrightarrow{\lim} \alpha)$ is an isomorphism 
(indeed, since this is local we can suppose that $\E$ has finite presentation and then using the five lemma
we reduce to the case where $\E$ is free of finite type which is obvious). 
This means that 
$\coh ( \D)\subset \Mod ( \D) ^\mathrm{fp}$,
where $\Mod ( \D) ^\mathrm{fp}$ is the full subcategory of modules of finite presentation in 
 $ \Mod ( \D) $ (in the sense of the definition \cite[6.3.3]{Kashiwara-schapira-book}).
From 
 \cite[6.3.4]{Kashiwara-schapira-book}, 
 this implies that $J \iota$ is fully faithful. 
 We denote by $\indcoh  ( \D) $ the essential image of 
 $J \iota$.
 By definition, the category 
 $\indcoh  ( \D) $
 is the subcategory of 
  $ \Mod ( \D) $
  consisting of objects which are filtrante inductive limits of objects of 
  $\coh ( \D)$.
 Since $\fP $ is noetherian, 
 the category $\coh ( \D)$ is essentially small. 
 From  
 \cite[8.6.5.(vi)]{Kashiwara-schapira-book}, 
 this yields that 
 $\indcoh  ( \D) $ is a Grothendieck category.

We set $D ^\mathrm{b} 
_{\mathrm{indcoh}}
( \D)
:=D ^\mathrm{b} _{ \indcoh  ( \D)} ( \Mod (\D))$.

Replacing $\D$ by $D$, we define the categories 
$\Mod  ( D)$, $\coh  ( D)$, $\indcoh  ( D) $.
\end{empt}

\begin{lem}
We keep the notation of \ref{IndCoh}.

\begin{enumerate}
\item We have the equalities 
$\coh  ( \D)
=
 \Mod  ( \D) ^{\mathrm{fp}}
= \indcoh  ( \D) ^{\mathrm{fp}}$.

\item Suppose $\fP$ affine. 
We have the equalities
$\indcoh  ( D) 
=
\Mod  ( D)$,
$\coh  ( D)
= 
\Mod  ( D) ^{\mathrm{fp}}$.
\end{enumerate}

\end{lem}

\begin{proof}
1) We have already seen that 
$\coh ( \D)\subset \Mod ( \D) ^\mathrm{fp}$.
Moreover, 
since $J \iota$ commutes with small filtrant limits, we get
$\Mod  ( \D) ^{\mathrm{fp}}
\subset \indcoh  ( \D) ^{\mathrm{fp}}$.
Let $\E \in \indcoh  ( \D) ^{\mathrm{fp}}$.
In particular, $\E \in \indcoh  ( \D) $ and then there exists 
a functor $\alpha \colon I \to  \coh ( \D) $ with $I$ small and filtrant,
such that 
 $\E \riso \underrightarrow{\lim} \alpha $.
 Using the definition of finite presentation for $\E$ and $\alpha$, the  isomorphism
 $\E \riso \underrightarrow{\lim} \alpha $ implies that $\E$ is a direct summand of $\alpha (i)$ for some $i \in I$.
Hence, $\E$ is also coherent. 

2) Suppose $\fP$ affine. 
Since $D$ is coherent, 
then 
$\coh  ( D)
= 
\Mod  ( D) ^{\mathrm{fp}}$. We conclude following 
 \cite[6.3.6.(ii)]{Kashiwara-schapira-book}.
\end{proof}

\begin{lem}
\label{eqcatindcoh}
We keep the notation of \ref{IndCoh}.
We suppose $\fP$ affine. 
\begin{enumerate}
\item The functors $\D 
\otimes _{D} -$ and $\Gamma (\fP, -)$ 
induce quasi-inverse
equivalences of categories between
$ \indcoh  ( \D)$
and 
$ \Mod ( D)$
(resp. $ \coh  ( \D)$
and 
$ \coh  ( D)$).
Moreover
$ \coh  ( D)$
(resp. $ \coh  ( D ^{(m)})$)
is equal to category of 
finitely presented $ D$-module 
(resp. 
the category of 
finitely generated $ D ^{(m)}$-module).

\item For any 
$\E \in \indcoh  ( D)$, 
$q \geq 1$, 
$H ^q ( \fP, \E)  = 0$.
\end{enumerate}

\end{lem}

\begin{proof}
From \cite[VI.5.1-2]{sga4-1}, the functors  
$H ^q ( \fP, -) $ for $q \geq 0$ commutes with small filtrant inductive limits. 
This is also the case for the functor 
$\D 
\otimes _{D} -$.
Since an ind-coherent object is a small filtrant inductive limits of coherent modules, 
we reduce to the coherent case which is already known (see \cite[3]{Be1}). 
\end{proof}

\begin{prop}
\label{eq-cat-indcoh-m}
We keep the notation of \ref{IndCoh}.
We suppose $\fP$ affine. 
The canonical functor
\begin{equation}
\label{eq-cat-indcoh-m-funct}
D ^\mathrm{b} ( \indcoh  ( \D)) 
\to 
D ^\mathrm{b} 
_{\mathrm{indcoh}}
( \D)
\end{equation}
is an equivalence of categories.
\end{prop}

\begin{proof}
Since $D ^\mathrm{b} 
_{ \indcoh  ( \D)}
( \D)$
is generated as a triangulated category by the objects of $ \indcoh  ( \D)$,
there remains to see that the functor is fully faithful (use \cite[I.2.18]{borel}).
Since the functor \ref{eq-cat-indcoh-m-funct} is t-exact and is the identity over the objects,
when no confusion is possible, we omit writing it.
Let $\E,\E' \in D ^\mathrm{b} ( \indcoh  ( \D)) $.
Since the functors
$F = \R \mathrm{Hom} _{
D ^\mathrm{b} ( \indcoh  ( \D))}
( \E, -)$ 
and 
$G = \R \mathrm{Hom} _{
D ^\mathrm{b} ( \D)}
( \E, -)$ 
are way-out right, using
\cite[I.7.1.(iv)]{HaRD}, 
we reduce to check that 
the canonical morphism
$F (\E') \to G (\E')$ is an isomorphism
when $\E'$ is an injective object of $\indcoh  ( \D)$.
Again, using \cite[I.7.1.(iv)]{HaRD} (this time we use the functors
$\R \mathrm{Hom} _{
D ( \indcoh  ( \D))}
( -, \E')$ 
and 
$\R \mathrm{Hom} _{
D ( \D)}
( -, \E')$),
we reduce to the case where $\E$ is a free $\D$-module.
Hence, we reduce to check that the canonical morphism
$\R \mathrm{Hom} _{
D ( \indcoh  ( \D))}
( \E, \E')
\to 
\R \mathrm{Hom} _{
D ^\mathrm{b} ( \D)}
( \E, \E')$
is an isomorphism
when $\E$ is a free $\D$-module of rank one and $\E'$ is an injective object
of $\indcoh  ( \D)$.
By injectivity of $\E'$,
we get 
$\mathcal{H} ^{i }\R \mathrm{Hom} _{
D ( \indcoh  ( \D))}
( \E, \E')=
0$ if $i \not =0$, and 
$\mathcal{H} ^{0}\R \mathrm{Hom} _{
D ( \indcoh  ( \D))}
( \E, \E')
=
\Gamma (\fP, \E ')$.
Moreover, 
$\mathcal{H} ^{i }\R \mathrm{Hom} _{
D ( \D)}
( \E, \E')=
H ^i ( \fP, \E') $ for any $i\in \N$. 
Hence, we conclude using \ref{eqcatindcoh}.
\end{proof}

\begin{lem}
\label{Serresubcat-cohindcoh}
We suppose $\fP$ affine. The category 
$\coh  ( \D ^{(m)} )$ 
is a Serre subcategory of
$\indcoh  ( \D ^{(m)})$. 
\end{lem}

\begin{proof}
This is a consequence of \ref{eqcatindcoh}
and of the fact that 
$D ^{(m)}$ is noetherian (see \cite[3.3-3.4]{Be1}).
\end{proof}

\begin{rem}
It seems false that the category 
$\coh  ( \D ^\dag _{\fP /\S } (\hdag T) _\Q)$
is a Serre subcategory of
$\indcoh  ( \D ^\dag _{\fP /\S } (\hdag T) _\Q)$.

\end{rem}

\begin{lem}
We suppose $\fP$ affine. 
Let $\alpha \colon 
\E \twoheadrightarrow \FF$
be an epimorphism of 
$\indcoh ( \D ^{(m)})$ such that 
$\FF \in \coh (\D^{(m)})$. 
Then there exists $\G \subset \E$
such that 
$\G \in \coh (\D^{(m)})$
and 
$\alpha (\E) = \FF$.
\end{lem}

\begin{proof}
This is a consequence of \ref{eqcatindcoh}
and the fact that 
$D ^{(m)}$ is noetherian (see \cite[3.3-3.4]{Be1}).
\end{proof}

\begin{lem}
\label{lem-a-eqcat}
We suppose $\fP$ affine. 
Let 
$\alpha \colon 
\E \to \FF$ be a morphism of 
$C ^{\mathrm{b}} (\indcoh ( \D ^{(m)}))$ 
such that 
$\E \in C ^{\mathrm{b}} (\coh ( \D ^{(m)}))$.
Then there exists
a subcomplex 
$\G $ of $\FF$ such that 
$\G \hookrightarrow  \FF$ is a quasi-isomorphism,
$\alpha (\E) \subset \G$
and 
$\G \in C ^{\mathrm{b}} (\coh ( \D ^{(m)}))$.
\end{lem}

\begin{proof}
Replacing $\E$ by $\alpha (\E)$, we can suppose $\alpha$ is a monomorphism.
Using \ref{eqcatindcoh}, we can replace 
$\indcoh ( \D ^{(m)})$ by
$\Mod (D ^{(m)})$
and
$\coh ( \D ^{(m)})$
by 
$\coh ( D ^{(m)})$.
Then, we can proceed as in the part (a) of the proof of \cite[VI.2.11]{borel}.
\end{proof}

\begin{prop}
\label{eq-cat-coh-m}
We keep the notation of \ref{IndCoh}.
We suppose $\fP$ affine. 
The canonical functor
$$
D ^\mathrm{b} ( \coh  ( \D^{(m)})  )
\to 
D ^\mathrm{b} _{\mathrm{coh}}
( \D ^{(m)})$$ 
is an equivalence of categories.
\end{prop}

\begin{proof}
Since $D ^\mathrm{b} _{\mathrm{coh}}
( \D ^{(m)})$
is generated as a triangulated category by the objects of $ \coh (\D ^{(m)})$,
there remains to see that the functor is fully faithful (use \cite[I.2.18]{borel}).
Using \ref{eq-cat-indcoh-m}, we reduce to check
that 
$D ^\mathrm{b} ( \coh  ( \D^{(m)})  
\to 
D ^\mathrm{b} 
( \indcoh (\D ^{(m)}))$
is fully faithful.
Let 
$\E,\FF \in C ^\mathrm{b} ( \coh  ( \D^{(m)})  )$.
We have to check that 
$\mu 
\colon 
\mathrm{Hom} 
_{D ( \coh  ( \D ^{(m)}))}
( \E, \FF)
\to 
\mathrm{Hom} 
_{D ( \indcoh  ( \D ^{(m)}))}
( \E, \FF)$
is a bijection.
Let us check the surjectivity. 
Let $v \in \mathrm{Hom} 
_{D ( \indcoh  ( \D ^{(m)}))}
( \E, \FF)$.
Let
$\FF '\in C ^\mathrm{b} ( \indcoh  ( \D^{(m)})  )$, 
$\beta \colon \FF \to \FF'$ a quasi-isomorphism
of 
$C ^\mathrm{b} ( \indcoh  ( \D^{(m)})  )$,
$\alpha \colon \E \to \FF'$ a morphism
of 
$C ^\mathrm{b} ( \indcoh  ( \D^{(m)})  )$
such that $v$ is represented by $(\alpha , \beta)$.
Using \ref{Serresubcat-cohindcoh}, 
$\alpha (\E) + \beta (\FF) \in \coh ( \D^{(m)}) $. 
Hence, using \ref{lem-a-eqcat}, there exists 
a subcomplex 
$\FF'' $ of $\FF'$   such that 
$\FF'' \hookrightarrow  \FF'$ is a quasi-isomorphism,
$\alpha (\E) + \beta (\FF) \subset \FF'' $,
and 
$\FF'' \in C ^{\mathrm{b}} (\coh ( \D ^{(m)}))$.
Let $u \in \mathrm{Hom} 
_{D ( \coh  ( \D ^{(m)}))}
( \E, \FF)
$ be the element represented by $\alpha\colon \E \to \FF ''$
and $\beta \colon \FF \to \FF''$. 
Then $u$ is sent to $v$. 

Let us check the injectivity. 
Let 
$u 
\in 
\mathrm{Hom} 
_{D ( \coh  ( \D ^{(m)}))}
( \E, \FF)$ whose image in 
$\mathrm{Hom} 
_{D ( \indcoh  ( \D ^{(m)}))}
( \E, \FF)$
is zero. 
Let
$\FF '\in C ^\mathrm{b} ( \coh  ( \D^{(m)})  )$, 
$\beta \colon \FF \to \FF'$ a quasi-isomorphism
of 
$C ^\mathrm{b} ( \coh  ( \D^{(m)})  )$,
$\alpha \colon \E \to \FF'$ a morphism
of 
$C ^\mathrm{b} ( \coh  ( \D^{(m)})  )$
such that $u$ is represented by $(\alpha , \beta)$.
By hypothesis, there exists 
a quasi-isomorphism 
$\gamma \colon \FF' \to \FF''$ 
of
$C ^\mathrm{b} ( \indcoh  ( \D^{(m)})  )$
such that 
$\gamma \circ \alpha$ is homotopic to zero. 
The homotopy is given by a morphism 
of the form 
$k \colon \E \to \FF '' [-1]$.
Using \ref{lem-a-eqcat}, there exists 
a subcomplex 
$\G $ of $\FF''$   such that 
$\G \hookrightarrow  \FF''$ is a quasi-isomorphism,
$\gamma  (\FF') +k (\E)[1] \subset \G$,
and 
$\G \in C ^{\mathrm{b}} (\coh ( \D ^{(m)}))$.
Hence, $\gamma$ has the factorization
$\gamma ' \colon \FF' \to \G$ 
such that 
$\gamma ' \circ \alpha$ is homotopic to zero,
and 
$u$ is represented by
$\gamma ' \circ \alpha$ and 
$\gamma ' \circ \beta$.
Hence, $u=0$.
\end{proof}

\begin{cor}
\label{cor-eq-cat-coh-m}
We keep the notation of \ref{IndCoh}.
We suppose $\fP$ affine. 
The canonical functors
\begin{gather}
\label{cor-eq-cat-coh-m-1}
D ^\mathrm{b} ( \coh  ( \D ^\dag _{\fP /\S } (\hdag T) _\Q ))
\to 
D ^\mathrm{b} _{\mathrm{coh}}
(\D ^\dag _{\fP /\S } (\hdag T) _\Q),
\\
\label{cor-eq-cat-coh-m-2}
D ^{\mathrm{b}}  (\underrightarrow{LM} _{\Q,\mathrm{coh}} (\smash{\widetilde{\D}} _{\fP /\S } ^{(\bullet)} (T)))
\to 
\underrightarrow{LD} ^{\mathrm{b}} _{\Q, \mathrm{coh}} (\smash{\widetilde{\D}} _{\fP /\S } ^{(\bullet)} (T))
\end{gather}
are essentially surjective.
\end{cor}

\begin{proof}
Using \ref{eqcat-limcoh}, it is sufficient to check the first statement. 
Let $\E \in D ^\mathrm{b} _{\mathrm{coh}}
(\D ^\dag _{\fP /\S } (\hdag T) _\Q)$.
Following \ref{eqcat-limcoh}, there exists
$\E ^{(\bullet)}  \in 
\underrightarrow{LD}  ^{\mathrm{b}} _{\Q, \mathrm{coh}} (\smash{\widetilde{\D}} _{\fP /\S } ^{(\bullet)} (T))$
together with an isomorphism
$\underrightarrow{\lim} \E ^{(\bullet)} 
\riso 
\E$.
By definition \ref{defi-LDQ0coh},
there exist
$\lambda \in L$
and
$\FF ^{(\bullet)}\in \smash{\underrightarrow{LD}} _{\Q} ^{\mathrm{b}} (\lambda ^{*} \smash{\widetilde{\D}} _{\fP /\S } ^{(\bullet)}(T))$ 
together with an  isomorphism in 
$ \smash{\underrightarrow{LD}} _{\Q} ^{\mathrm{b}}  ( \smash{\widetilde{\D}} _{\fP /\S } ^{(\bullet)}(T))$ of the form
$\E ^{(\bullet)} \riso \FF ^{(\bullet) }$,
such that $\FF ^{(\bullet)}$ satisfies the conditions \ref{defi-LDQ0coh}.1 and \ref{defi-LDQ0coh}.2.
Hence, 
$\underrightarrow{\lim} \FF ^{(\bullet)} 
\riso 
\E$.
This yields
$\FF ^{(0)} \in D  _{\mathrm{coh}} ^{\mathrm{b}} (\smash{\widetilde{\D}} _{\fP /\S } ^{(\lambda (0))} (T))$,
and the isomorphism
$$  \D ^\dag _{\fP /\S } (\hdag T) _\Q  
 \otimes  _{\smash{\widetilde{\D}} _{\fP /\S } ^{(\lambda (0))} (T) _\Q }
\FF ^{(0)}  _\Q
\riso
\E .$$
Using \ref{eq-cat-coh-m}, 
there exists 
$\G ^{(0)}
\in 
D   ^{\mathrm{b}} (\coh (\smash{\widetilde{\D}} _{\fP /\S } ^{(\lambda (0))} (T)_\Q))$
together with an isomorphism
$\G ^{(0)}\riso 
\FF ^{(0)} _\Q$ in 
$D   ^{\mathrm{b}} (\smash{\widetilde{\D}} _{\fP /\S } ^{(\lambda (0))} (T)_\Q)$.
Set 
$\G :=  \D ^\dag _{\fP /\S } (\hdag T) _\Q  
 \otimes  _{\smash{\widetilde{\D}} _{\fP /\S } ^{(\lambda (0))} (T) _\Q }
\G ^{(0)}
\in 
D   ^{\mathrm{b}} (\coh ( \D ^\dag _{\fP /\S } (\hdag T) _\Q  ))
$. 
By applying the exact functor
$\D ^\dag _{\fP /\S } (\hdag T) _\Q  
 \otimes  _{\smash{\widetilde{\D}} _{\fP /\S } ^{(\lambda (0))} (T) _\Q }
-$
to 
$\G ^{(0)}\riso 
\FF ^{(0)} _\Q$, we get the isomorphism
$\G \riso 
\E$
of
$D ^\mathrm{b} _{\mathrm{coh}}
(\D ^\dag _{\fP /\S } (\hdag T) _\Q)$.
 
\end{proof}

\begin{prop}
Let $\fU:= \fP \setminus T$ be the open smooth formal $\S $-scheme complementary to $T$.
Let $\E \in  \indcoh  ( \D ^\dag _{\fP /\S } (\hdag T) _\Q)$.
If $\E | \fU \in  \coh  ( \D ^\dag _{\fU/\S,\Q} )$
then 
$\E \in  \coh  ( \D ^\dag _{\fP /\S } (\hdag T) _\Q)$.
\end{prop}

\begin{proof}
Since this is local, we can suppose that $\fP$ is affine. From \ref{eqcatindcoh},
we reduce to check that $E:= \Gamma (\fP, \E)$ lies in   $\coh  ( D ^\dag _{\fP /\S } (\hdag T) _\Q)$.
Since 
$\D ^\dag _{\fP /\S } (\hdag T) _\Q |\fU = 
\D ^\dag _{\fU/\S,\Q} $, 
with  \ref{eqcatindcoh} we get 
$\E |\fU \riso 
\D ^\dag _{\fU/\S,\Q} 
\otimes _{D ^\dag _{\fP /\S } (\hdag T) _\Q} E
\riso \D ^\dag _{\fU/\S,\Q} 
\otimes _{D ^\dag _{\fU/\S,\Q} } E'$, where 
$E':= D ^\dag _{\fU/\S,\Q} 
\otimes _{D ^\dag _{\fP /\S } (\hdag T) _\Q} E$.
Since  $\E |\fU$ is $\D ^\dag _{\fU/\S,\Q} $-coherent,
then using \ref{eqcatindcoh} we get that $E'$ is  $D ^\dag _{\fU/\S,\Q} $-coherent. 
Since the extension 
$D ^\dag _{\fP /\S } (\hdag T) _\Q \to 
D ^\dag _{\fU/\S,\Q} $ is faithfully flat, 
we are done.
\end{proof}

\section{Localization functor outside a divisor}

We keep the notation of chapter \ref{ntn-tildeD(Z)}.

\subsection{Tensor products, quasi-coherence, forgetful functor,
localization functor outside a divisor}

\begin{empt}
For any $\E,\FF \in D  ^-
(\overset{^\mathrm{l}}{} \smash{\widetilde{\D}} _{\fP /\S } ^{(m)} (T ))$
and $\M \in D  ^-
({}^r \smash{\widetilde{\D}} _{\fP /\S } ^{(m)} (T ))$,
we set:

\begin{gather} \notag
\M _i := \M \otimes ^\L _{\smash{\widetilde{\D}} _{\fP /\S } ^{(m)} (T )} \smash{\widetilde{\D}} _{P  _i/ S  _i} ^{(m)} (T ),\
\E _i := \smash{\widetilde{\D}} _{P  _i/ S  _i} ^{(m)} (T ) \otimes ^\L _{\smash{\widetilde{\D}} _{\fP /\S } ^{(m)} ( T )} \E,\\
\notag
\M \smash{\widehat{\otimes}} ^\L _{\smash{\widetilde{\B}} _{\fP} ^{(m)} (T )} \E :=
\R \underset{\underset{i}{\longleftarrow}}{\lim}\, ( \M _i \otimes ^\L  _{\widetilde{\B} _{P _i} ^{(m)} (T )} \E _i)
,\,
\E \smash{\widehat{\otimes}} ^\L _{\smash{\widetilde{\B}} _{\fP} ^{(m)} (T )} \FF :=
\R \underset{\underset{i}{\longleftarrow}}{\lim}\, ( \E _i \otimes ^\L  _{\widetilde{\B} _{P _i} ^{(m)} (T )} \FF _i),
\\
\M \smash{\widehat{\otimes}} ^\L _{\smash{\widetilde{\D}} _{\fP /\S } ^{(m)} (T )} \E :=
\R \underset{\underset{i}{\longleftarrow}}{\lim}\, ( \M _i \otimes ^\L  _{\widetilde{\D} _{P  _i/S  _i} ^{(m)} (T )} \E _i).
\end{gather}
\end{empt}

\begin{empt}
For any
$\E ^{(\bullet)} \in 
D ^{-}
(\overset{^\mathrm{l}}{} \smash{\widetilde{\D}} _{\fP /\S } ^{(\bullet)} (T ))$,
$\M ^{(\bullet)} \in 
D ^{-}
( {}^r \smash{\widetilde{\D}} _{\fP /\S } ^{(\bullet)} (T )  )$,
we set
\begin{gather}
\label{predef-otimes-coh0}
  \M ^{(\bullet)}
\smash{\widehat{\otimes}} ^\L _{\smash{\widetilde{\D}} ^{(\bullet)} _{\fP /\S }( T) }\E ^{(\bullet)}
:=
(\M ^{(m)}  \smash{\widehat{\otimes}} ^\L _{\widetilde{\D} ^{(m)} _{\fP /\S } ( T) } \E ^{(m)}) _{m\in \N}.
\end{gather}
For $? = r$ or $? = l$, 
we define the following tensor product bifunctor
\begin{align}
\label{predef-otimes-coh1}
-
 \smash{\widehat{\otimes}}
^\L _{\widetilde{\B} ^{(\bullet)}  _{\fP} ( T) }
 -
 \colon
D ^-
(\overset{^\mathrm{?}}{} \smash{\widetilde{\D}} _{\fP /\S } ^{(\bullet)} (T ))
\times 
D ^-
(\overset{^\mathrm{l}}{} \smash{\widetilde{\D}} _{\fP /\S } ^{(\bullet)} (T ))
&
\to 
D ^-
(\overset{^\mathrm{?}}{} \smash{\widetilde{\D}} _{\fP /\S } ^{(\bullet)} (T )),
\end{align}
by setting, for any
$\E ^{(\bullet)}\in 
D ^{-}
(\overset{^\mathrm{?}}{} \smash{\widetilde{\D}} _{\fP /\S } ^{(\bullet)} (T ))$,
$\FF ^{(\bullet)} \in 
D ^{-}
(\overset{^\mathrm{l}}{} \smash{\widetilde{\D}} _{\fP /\S } ^{(\bullet)} (T ))$,
\begin{equation}
\notag
\E ^{(\bullet)}
\smash{\widehat{\otimes}}
^\L _{\widetilde{\B} ^{(\bullet)}  _{\fP} ( T) }\FF ^{(\bullet)}
:=
(\E ^{(m)}  
\smash{\widehat{\otimes}}
^\L _{\widetilde{\B} ^{(m)}  _{\fP} ( T) }
\FF ^{(m)}) _{m\in \N}.
\end{equation}

When $T$ is empty, $\widetilde{\B} ^{(\bullet)}  _{\fP} ( T) $
will simply be denoted by 
$\O _{\fP} ^{(\bullet)}$, i.e. 
$\O _{\fP} ^{(\bullet)}$ is the subring of 
$\widehat{\D} ^{(\bullet)} _{\fP/\S}$ whose transition morphisms
are the identity of $\O _{\fP}$.

\end{empt}

\begin{ntn}
[Quasi-coherence and partial forgetful functor of the divisor]
Let $D \subset T$ be a second divisor.

\begin{itemize}
\item Let  $\E ^{(m)} \in 
D ^{\mathrm{b}}
(\overset{^\mathrm{l}}{} \smash{\widetilde{\D}} _{\fP /\S } ^{(m)} (T ))$.
Since  $\smash{\widetilde{\D}} _{\fP /\S } ^{(m)}(T)$ (resp. $\smash{\widetilde{\B}} _{\fP} ^{(m)}(T)$)
has not $p$-torsion,
using the Theorem \cite[3.2.2]{Beintro2} 
we get that 
$\E ^{(m)} $ is quasi-coherent in the sense of Berthelot 
as object of
$D ^{\mathrm{b}}
(\overset{^\mathrm{l}}{} \smash{\D} _{\fP} ^{(m)})$ (see his definition \cite[3.2.1]{Beintro2}) 
if and only if 
$\E ^{(m)} _0 \in D ^{\mathrm{b}} _{\mathrm{qc}}
(\O  _{P})$
and 
the canonical morphism 
$\E ^{(m)} \to 
\smash{\widetilde{\D}} _{\fP /\S } ^{(m)} (T ) 
\smash{\widehat{\otimes}} ^\L _{\smash{\widetilde{\D}} _{\fP /\S } ^{(m)} (T )} \E ^{(m)} $
(resp. $\E ^{(m)} \to 
\smash{\widetilde{\B}} _{\fP} ^{(m)} (T ) 
\smash{\widehat{\otimes}} ^\L _{\smash{\widetilde{\B}} _{\fP} ^{(m)} (T )} \E ^{(m)} $)
is an isomorphism. 
In particular, this does not depend on the divisor $T$.
We denote by 
$D ^{\mathrm{b}} _{\mathrm{qc}}
(\overset{^\mathrm{l}}{} \smash{\widetilde{\D}} _{\fP /\S } ^{(m)} (T ))$, 
the full subcategory of
$D ^{\mathrm{b}}
(\overset{^\mathrm{l}}{} \smash{\widetilde{\D}} _{\fP /\S } ^{(m)} (T ))$
of quasi-coherent complexes.
We get the
{\it partial forgetful functor of the divisor} 
$\mathrm{oub} _{D,T}\colon 
D ^{\mathrm{b}} _{\mathrm{qc}}
(\overset{^\mathrm{l}}{} \smash{\widetilde{\D}} _{\fP /\S } ^{(m)} (T))
\to
D ^{\mathrm{b}} _{\mathrm{qc}}
(\overset{^\mathrm{l}}{} \smash{\widetilde{\D}} _{\fP /\S } ^{(m)} (D ))$
which is induced by the canonical forgetful functor
$\mathrm{oub} _{D,T}\colon 
D ^{\mathrm{b}} 
(\overset{^\mathrm{l}}{} \smash{\widetilde{\D}} _{\fP /\S } ^{(m)} (T))
\to
D ^{\mathrm{b}} 
(\overset{^\mathrm{l}}{} \smash{\widetilde{\D}} _{\fP /\S } ^{(m)} (D ))$.

\item Similarly, we denote by
$D ^{\mathrm{b}} _{\mathrm{qc}}
(\overset{^\mathrm{l}}{} \smash{\widetilde{\D}} _{\fP /\S } ^{(\bullet)} (T ))$
the full subcategory of 
$D ^{\mathrm{b}}
(\overset{^\mathrm{l}}{} \smash{\widetilde{\D}} _{\fP /\S } ^{(\bullet)} (T ))$
of complexes $\E ^{(\bullet)} $ such that, for any $m\in \Z$,  
$\E ^{(m)} _0 \in D ^{\mathrm{b}} _{\mathrm{qc}}
(\O  _{P})$ 
and the canonical morphism 
$\E ^{(\bullet)}  \to 
\smash{\widetilde{\D}} _{\fP /\S } ^{(\bullet)} (T )
\smash{\widehat{\otimes}} ^\L _{\smash{\widetilde{\D}} _{\fP /\S } ^{(\bullet)} (T )} \E ^{(\bullet)} $
is an isomorphism of $D ^{\mathrm{b}}
(\overset{^\mathrm{l}}{} \smash{\widetilde{\D}} _{\fP /\S } ^{(\bullet)} (T ))$. 
We get the
{\it partial forgetful functor of the divisor} 
$\mathrm{oub} _{D,T}\colon 
D ^{\mathrm{b}} _{\mathrm{qc}}
(\overset{^\mathrm{l}}{} \smash{\widetilde{\D}} _{\fP /\S } ^{(\bullet)} (T))
\to
D ^{\mathrm{b}}  _{\mathrm{qc}}
(\overset{^\mathrm{l}}{} \smash{\widetilde{\D}} _{\fP /\S } ^{(\bullet)} (D ))$.

\item We denote by 
$\smash{\underrightarrow{LD}}  ^{\mathrm{b}} _{\Q,\mathrm{qc}}
( \smash{\widetilde{\D}} _{\fP /\S } ^{(\bullet)} (T ))$ 
the strictly full subcategory of 
$\smash{\underrightarrow{LD}}  ^{\mathrm{b}} _{\Q}
( \smash{\widetilde{\D}} _{\fP /\S } ^{(\bullet)} (T ))$
of complexes which are isomorphic in
$\smash{\underrightarrow{LD}}  ^{\mathrm{b}} _{\Q}
( \smash{\widetilde{\D}} _{\fP /\S } ^{(\bullet)} (T ))$ 
to a complex belonging to 
$D ^{\mathrm{b}}  _{\mathrm{qc}}
(\overset{^\mathrm{l}}{} \smash{\widetilde{\D}} _{\fP /\S } ^{(\bullet)} (T ))$.
Since the functor $\mathrm{oub} _{D,T}$
sends a lim-ind-isogeny to a lim-ind-isogeny, 
we obtain the factorization of the form : 
\begin{equation}
\label{def-oubDT}
\mathrm{oub} _{D,T}
\colon 
\smash{\underrightarrow{LD}}  ^{\mathrm{b}} _{\Q,\mathrm{qc}}
( \smash{\widetilde{\D}} _{\fP /\S } ^{(\bullet)} (T ))
\to
\smash{\underrightarrow{LD}}  ^{\mathrm{b}} _{\Q,\mathrm{qc}}
( \smash{\widetilde{\D}} _{\fP /\S } ^{(\bullet)} (D )).
\end{equation}

\item We still denote by 
$\mathrm{oub} _{D, T}\colon 
D ^{\mathrm{b}} (\D ^\dag _{\fP /\S } (\hdag T) _\Q) \to 
D ^{\mathrm{b}} (\D ^\dag _{\fP /\S } (\hdag D) _\Q)$
the partial forgetful functor of the divisor.

\end{itemize}
\end{ntn}

\begin{rem}
\begin{enumerate}
\item A morphism 
$\E ^{(\bullet)}
\to \cF ^{(\bullet)}$ 
of 
$D ^{\mathrm{b}}
(\overset{^\mathrm{l}}{} \smash{\widetilde{\D}} _{\fP /\S } ^{(\bullet)} (T ))$
is an isomorphism if and only if 
the induced morphism
$\E ^{(m)}
\to \cF ^{(m)}$ is an isomorphism of 
$D ^{\mathrm{b}}
(\overset{^\mathrm{l}}{} \smash{\widetilde{\D}} _{\fP /\S } ^{(m)} (T ))$ for every $m\in \Z$.

\item Let 
$\E ^{(\bullet)}
\in 
D ^{\mathrm{b}}
(\overset{^\mathrm{l}}{} \smash{\widetilde{\D}} _{\fP /\S } ^{(\bullet)} (T ))$.
Using the first remark, 
we check the property $\E ^{(\bullet)}
\in 
D ^{\mathrm{b}} _{\mathrm{qc}}
(\overset{^\mathrm{l}}{} \smash{\widetilde{\D}} _{\fP /\S } ^{(\bullet)} (T ))$
is equivalent to
the property that,  
for any $m \in \Z$, 
$\E ^{(m)}
\in 
D ^{\mathrm{b}} _{\mathrm{qc}}
(\overset{^\mathrm{l}}{} \smash{\widetilde{\D}} _{\fP /\S } ^{(m)} (T ))$.
Hence, the above definition of 
$\smash{\underrightarrow{LD}}  ^{\mathrm{b}} _{\Q,\mathrm{qc}}
( \smash{\widetilde{\D}} _{\fP /\S } ^{(\bullet)} (T ))$ 
corresponds to that of Berthelot's one formulated in 
 \cite[4.2.3]{Beintro2} without singularities along a divisor. 
\end{enumerate}

\end{rem}

\begin{lemm}
\label{lemm-def-otimes-coh1}
The bifunctor \ref{predef-otimes-coh1} induces
\begin{align}
\label{def-otimes-coh1}
-
 \smash{\widehat{\otimes}}
^\L _{\widetilde{\B} ^{(\bullet)}  _{\fP} ( T) }
 -
 \colon
 \smash{\underrightarrow{LD}}  ^- _{\Q}
(\overset{^\mathrm{?}}{} \smash{\widetilde{\D}} _{\fP /\S } ^{(\bullet)} (T ))
\times 
\smash{\underrightarrow{LD}}  ^- _{\Q}
(\overset{^\mathrm{l}}{} \smash{\widetilde{\D}} _{\fP /\S } ^{(\bullet)} (T ))
&
\to 
\smash{\underrightarrow{LD}}  ^- _{\Q}
(\overset{^\mathrm{?}}{} \smash{\widetilde{\D}} _{\fP /\S } ^{(\bullet)} (T )).
\end{align}
\end{lemm}

\begin{proof}
Let $\E ^{(\bullet)}\in 
D ^{-}
(\overset{^\mathrm{?}}{} \smash{\widetilde{\D}} _{\fP /\S } ^{(\bullet)} (T ))$,
$\FF ^{(\bullet)} \in 
D ^{-}
(\overset{^\mathrm{l}}{} \smash{\widetilde{\D}} _{\fP /\S } ^{(\bullet)} (T ))$.
Let $\chi \in M$, $\lambda \in L$.
By using a left  resolution of $\FF ^{(\bullet)}$ by flat $ \smash{\widetilde{\D}} _{\fP /\S } ^{(\bullet)} (T )$-modules
(remark also that 
a   $ \smash{\widetilde{\D}} _{\fP /\S } ^{(\bullet)} (T )$-module $\fP ^{(\bullet)}$ is flat if and only if 
 the $ \smash{\widetilde{\D}} _{\fP /\S } ^{(m)} (T )$-module $\fP ^{(m)}$
is flat for any $m$),
there exists a morphism
$
\E ^{(\bullet)}
\smash{\widehat{\otimes}}
^\L _{\widetilde{\B} ^{(\bullet)}  _{\fP} ( T) }
\lambda ^{*} \chi ^{*}  \FF ^{(\bullet)}
\to 
\lambda ^{*} \chi ^{*}  \left (\E ^{(\bullet)}
\smash{\widehat{\otimes}}
^\L _{\widetilde{\B} ^{(\bullet)}  _{\fP} ( T) }
 \FF ^{(\bullet)}\right) $ 
inducing the commutative canonical diagram
\begin{equation}
\notag
\xymatrix @ R=0,3cm{
{\E ^{(\bullet)}
\smash{\widehat{\otimes}}
^\L _{\widetilde{\B} ^{(\bullet)}  _{\fP} ( T) }
\FF ^{(\bullet)}
} 
\ar[r] ^-{}
\ar[d] ^-{}
& 
{\E ^{(\bullet)}
\smash{\widehat{\otimes}}
^\L _{\widetilde{\B} ^{(\bullet)}  _{\fP} ( T) }
\lambda ^{*} \chi ^{*}  \FF ^{(\bullet)} } 
\ar[d] ^-{}
\ar@{.>}[dl] ^-{}
\\ 
{\lambda ^{*} \chi ^{*}  \left (\E ^{(\bullet)}
\smash{\widehat{\otimes}}
^\L _{\widetilde{\B} ^{(\bullet)}  _{\fP} ( T) }
 \FF ^{(\bullet)}\right) } 
 \ar[r] ^-{}
& 
{ \lambda ^{*} \chi ^{*}  \left (\E ^{(\bullet)}
\smash{\widehat{\otimes}}
^\L _{\widetilde{\B} ^{(\bullet)}  _{\fP} ( T) }
\lambda ^{*} \chi ^{*} \FF ^{(\bullet)}\right) .} 
}
\end{equation}
This yields the factorization
$\E ^{(\bullet)}
 \smash{\widehat{\otimes}}
^\L _{\widetilde{\B} ^{(\bullet)}  _{\fP} ( T) }
 -
 \colon
\smash{\underrightarrow{LD}}  ^- _{\Q}
(\overset{^\mathrm{l}}{} \smash{\widetilde{\D}} _{\fP /\S } ^{(\bullet)} (T ))
\to 
\smash{\underrightarrow{LD}}  ^- _{\Q}
(\overset{^\mathrm{?}}{} \smash{\widetilde{\D}} _{\fP /\S } ^{(\bullet)} (T ))$
and then by functoriality, 
$$-
 \smash{\widehat{\otimes}}
^\L _{\widetilde{\B} ^{(\bullet)}  _{\fP} ( T) }
 -
 \colon
 D  ^- 
(\overset{^\mathrm{?}}{} \smash{\widetilde{\D}} _{\fP /\S } ^{(\bullet)} (T ))
\times 
\smash{\underrightarrow{LD}}  ^- _{\Q}
(\overset{^\mathrm{l}}{} \smash{\widetilde{\D}} _{\fP /\S } ^{(\bullet)} (T ))
\to 
\smash{\underrightarrow{LD}}  ^- _{\Q}
(\overset{^\mathrm{?}}{} \smash{\widetilde{\D}} _{\fP /\S } ^{(\bullet)} (T )).$$
Similarly, we get the factorization with respect to the first factor. 
\end{proof}

\begin{empt}
Let $D \subset T$ be a second divisor. 
For any
$\E ^{(\bullet)} \in 
D ^{-}
(\overset{^\mathrm{l}}{} \smash{\widetilde{\D}} _{\fP /\S } ^{(\bullet)} (D))$, 
similarly to \cite[1.1.8]{caro_courbe-nouveau} we get 
the commutative diagram in $D ^{-}
(\overset{^\mathrm{l}}{} \smash{\widetilde{\D}} _{\fP /\S } ^{(\bullet)} (T))$:
\begin{equation}
\label{ODdivcohe}
  \xymatrix @ R=0,4cm {
{(\widetilde{\B} ^{(m)} _{\fP} ( T)  \smash{\widehat{\otimes}} ^\L
_{\widetilde{\B} ^{(m)}  _{\fP} ( D) } \E ^{(m)}) _{m\in \N}}
\ar@{=}[r] ^-{\mathrm{def}} \ar[d] _\sim
&
{\widetilde{\B} ^{(\bullet)} _{\fP} ( T)  \smash{\widehat{\otimes}} ^\L
_{\widetilde{\B} ^{(\bullet)}  _{\fP} ( D) } \E ^{(\bullet)}}
\ar@{.>}[d] _\sim
\\
{(\widetilde{\D} ^{(m)} _{\fP } ( T)  \smash{\widehat{\otimes}} ^\L
_{\widetilde{\D} ^{(m)} _{\fP } ( D) } \E ^{(m)}) _{m\in \N}}
\ar@{=}[r] ^-{\mathrm{def}}
&
{\smash{\widetilde{\D}} ^{(\bullet)} _{\fP /\S }( T)  \smash{\widehat{\otimes}} ^\L
_{\smash{\widetilde{\D}} ^{(\bullet)} _{\fP /\S }( D) }\E ^{(\bullet)}=: (\hdag T, D) (\E ^{(\bullet)})}.
}
\end{equation}
As for Lemma \ref{lemm-def-otimes-coh1}, 
we get the the localization outside $T$ functor :
\begin{gather}
\label{hdag-def}
 (\hdag T ,\,D) 
 :=\smash{\widetilde{\D}} ^{(\bullet)} _{\fP /\S }( T)  \smash{\widehat{\otimes}} ^\L
_{\smash{\widetilde{\D}} ^{(\bullet)} _{\fP /\S }( D) }-
\colon
\smash{\underrightarrow{LD}} ^{-} _{\Q} ( \smash{\widetilde{\D}} _{\fP /\S } ^{(\bullet)}(D))
\to
\smash{\underrightarrow{LD}} ^{-} _{\Q} ( \smash{\widetilde{\D}} _{\fP /\S } ^{(\bullet)}(T)).
\end{gather}

\end{empt}

\subsection{Preservation of the quasi-coherence}
\label{section3.2}
Let $m' \geq m \geq 0$ be two integers,
$D ' \subset D \subset T$ be three (reduced) divisors of  $P$. 
We have the canonical morphisms
$\widetilde{\B} _{P _i} ^{(m)} (D ') \to \widetilde{\B} _{P _i} ^{(m)} (D )\to \widetilde{\B} _{P _i} ^{(m')} (T)$.
Similarly to the notation of \cite{Beintro2},
we denote by $D _{\Q,\mathrm{qc}} ^{-} (\smash{\widetilde{\B}} _{\fP} ^{(m)} (D))$ 
(resp. 
$D _{\Q,\mathrm{qc}} ^{-} (\widetilde{\B} ^{(m')} _{\fP} ( D)  \smash{\widehat{\otimes}} _{\O _{\fP}} \smash{\widehat{\D}} _{\fP /\S } ^{(m)})$)
the localization of the category 
$D ^{-} _{\mathrm{qc}} (\smash{\widetilde{\B}} _{\fP} ^{(m)} (D))$ 
(resp.  $D _{\mathrm{qc}} ^{-} (\widetilde{\B} ^{(m')} _{\fP} ( D)  \smash{\widehat{\otimes}} _{\O _{\fP}} \smash{\widehat{\D}} _{\fP /\S } ^{(m)})$)
by isogenies.

\begin{lemm}
\label{lem1-hdagDT}
\begin{enumerate}
\item The kernel of the canonical epimorphism
$\smash{\widetilde{\B}} _{\fP} ^{(m)} (D ) \widehat{\otimes} _{\O _{\fP} }
\smash{\widetilde{\B}} _{\fP} ^{(m')} (T )
\to 
\smash{\widetilde{\B}} _{\fP} ^{(m')} (T )$
is a quasi-coherent $\O _{P}$-module. 

\item The canonical morphism
$\smash{\widetilde{\B}} _{\fP} ^{(m)} (D ) 
\widehat{\otimes} ^{\L}  _{\O _{\fP} }
\smash{\widetilde{\B}} _{\fP} ^{(m')} (T )
\to 
\smash{\widetilde{\B}} _{\fP} ^{(m)} (D ) \widehat{\otimes} _{\O _{\fP} }
\smash{\widetilde{\B}} _{\fP} ^{(m')} (T )$
is an isomorphism.

\end{enumerate}
\end{lemm}

\begin{proof}
We can copy word by word the proof of  \cite[3.2.1]{caro-stab-sys-ind-surcoh}. 
\end{proof}

\begin{empt}
Let us clarify some terminology.

\begin{enumerate}
\item A morphism of rings 
$f \colon \AA \to \B$ is a $p ^n$-isogeny if there exists 
a morphisms of rings $g \colon \B \to \AA$ such that $f \circ g = p ^n id$
and
$g \circ f = p ^n id$.

\item A morphism $f \colon \AA \to \B$ of 
$D ^- ( \smash{\widetilde{\B}} _{\fP} ^{(m')} (T ))$
is a $p ^n$-isogeny if there exists a morphisms  $g \colon \B \to \AA$ of $D ^- ( \smash{\widetilde{\B}} _{\fP} ^{(m')} (T ))$
such that $f \circ g = p ^n id$
and
$g \circ f = p ^n id$.
\end{enumerate}

\end{empt}

\begin{prop}
\label{lem3-hdagDT}
The canonical homomorphisms of 
$D ^- ( \smash{\widetilde{\B}} _{\fP} ^{(m')} (T ))$
or respectively of rings
\begin{equation}
\label{lem3-hdagDT-iso}
\smash{\widetilde{\B}} _{\fP} ^{(m')} (T ) \to
\smash{\widetilde{\B}} _{\fP} ^{(m)} (D ) \widehat{\otimes} ^{\L}_{\smash{\widetilde{\B}} _{\fP} ^{(m)} (D' )}
\smash{\widetilde{\B}} _{\fP} ^{(m')} (T )
\to 
\smash{\widetilde{\B}} _{\fP} ^{(m)} (D ) \widehat{\otimes} _{\smash{\widetilde{\B}} _{\fP} ^{(m)} (D' )}
\smash{\widetilde{\B}} _{\fP} ^{(m')} (T )
\to 
\smash{\widetilde{\B}} _{\fP} ^{(m')} (T )
\end{equation}
are $p$-isogenies.
\end{prop}

\begin{proof}
We can copy word by word the proof of \cite[3.2.2]{caro-stab-sys-ind-surcoh}. 
\end{proof}

\begin{coro}
\label{rema-dim-coh-finie}
\begin{enumerate}
\item The functors of the form
$\B _{P _i} ^{(m')} (T) \otimes ^{\L}_{\O _{P _i} }- $ have cohomological dimension 
$1$. 
The functor $\smash{\widetilde{\B}} _{\fP} ^{(m')} (T) 
\widehat{\otimes} ^{\L} _{\O _{\fP} }-$
is way-out over
$D  ^{-} (\O _{\fP} )$
with bounded amplitude independent of $m'$ and $m$.

\item 
The functor 
$\smash{\widetilde{\B}} _{\fP} ^{(m')} (T) \widehat{\otimes} ^{\L} _{\smash{\widetilde{\B}} _{\fP} ^{(m)} (D)}-
\colon 
D _{\Q,\mathrm{qc}} ^{\mathrm{b}} (\smash{\widetilde{\B}} _{\fP} ^{(m)} (D))
\to 
D _{\Q,\mathrm{qc}} ^{\mathrm{b}} (\smash{\widetilde{\B}} _{\fP} ^{(m')} (T))$ 
is way-out
with bounded amplitude independent of $m'$ and $m$.
We have the factorization 
$\smash{\widetilde{\B}} _{\fP} ^{(m+\bullet)} (T) \widehat{\otimes} ^{\L} _{\smash{\widetilde{\B}} _{\fP} ^{(m)} (D)}-
\colon 
D _{\Q,\mathrm{qc}} ^{\mathrm{b}} (\smash{\widetilde{\B}} _{\fP} ^{(m)} (D))
\to
\smash{\underrightarrow{LD}} _{\Q,\mathrm{qc}} ^{\mathrm{b}} (\smash{\widetilde{\B}} _{\fP} ^{(m+\bullet)} (T ))$.

\item 
The functor
$$(\widetilde{\B} ^{(m')} _{\fP} ( T)  \smash{\widehat{\otimes}} _{\O _{\fP}} \smash{\widehat{\D}} _{\fP /\S } ^{(m)})
\widehat{\otimes} ^{\L} _{(\widetilde{\B} ^{(m)} _{\fP} ( D)  \smash{\widehat{\otimes}} _{\O _{\fP}} \smash{\widehat{\D}} _{\fP /\S } ^{(m)})}
-
\colon 
D _{\Q,\mathrm{qc}} ^{\mathrm{b}} (\widetilde{\B} ^{(m)} _{\fP} ( D)  \smash{\widehat{\otimes}} _{\O _{\fP}} \smash{\widehat{\D}} _{\fP /\S } ^{(m)})
\to 
D _{\Q,\mathrm{qc}} ^{\mathrm{b}} (\widetilde{\B} ^{(m')} _{\fP} (T)  \smash{\widehat{\otimes}} _{\O _{\fP}} \smash{\widehat{\D}} _{\fP /\S } ^{(m)})$$
is way-out
with bounded amplitude independent of $m'$ and $m$.
\end{enumerate}

\end{coro}

\begin{proof}
We can copy word by word the proof of \cite[3.2.3]{caro-stab-sys-ind-surcoh}. 
\end{proof}

\begin{empt}
\label{hdagT-nota}
With Corollary  \ref{rema-dim-coh-finie} which implies the stability of the boundedness of the cohomology,
we check the factorization of the functor of \ref{hdag-def} as follows:
\begin{gather}
\label{hdag-def-qc}
 (\hdag T ,\,D) 
 :=\smash{\widetilde{\D}} ^{(\bullet)} _{\fP /\S }( T)  \smash{\widehat{\otimes}} ^\L
_{\smash{\widetilde{\D}} ^{(\bullet)} _{\fP /\S }( D) }-
\colon
\smash{\underrightarrow{LD}} ^{\mathrm{b}} _{\Q,\mathrm{qc}} ( \smash{\widetilde{\D}} _{\fP /\S } ^{(\bullet)}(D))
\to
\smash{\underrightarrow{LD}} ^{\mathrm{b}} _{\Q,\mathrm{qc}} ( \smash{\widetilde{\D}} _{\fP /\S } ^{(\bullet)}(T)).
\end{gather}
We also write 
$ \E ^{(\bullet)} (\hdag D ,\,T) :=(\hdag T ,\,D) (\E ^{(\bullet)})$. 
This functor  $(\hdag T ,\,D)$ is 
 {\it the localization outside $T$ functor}. 
When  $D=\emptyset $, we omit writing it. 
We write in the same way the associated functor for coherent complexes:
\begin{equation}
\label{hdag-def-coh}
(\hdag T, D) := 
\D ^\dag _{\fP /\S } (\hdag T) _\Q \otimes _{ \D ^\dag _{\fP /\S } (\hdag D) _\Q} - \colon 
D ^\mathrm{b} _\mathrm{coh} ( \D ^\dag _{\fP /\S } (\hdag D) _\Q)
\to 
D ^\mathrm{b} _\mathrm{coh} (\D ^\dag _{\fP /\S } (\hdag T) _\Q).
\end{equation}
The  functor \ref{hdag-def-coh} is exact, which justifies the absence of the symbol 
$\L$. 
\end{empt}

\begin{prop}
\label{oub-pl-fid}
Let  $\E ^{(\bullet)}\in 
\smash{\underrightarrow{LD}} ^{\mathrm{b}} _{\Q,\mathrm{qc}}
 ( \smash{\widetilde{\D}} _{\fP /\S } ^{(\bullet)}(T))$.
\begin{enumerate}
\item The functorial in $\E ^{(\bullet)}$ canonical morphism :
\begin{equation}
\label{oub-pl-fid-iso1}
(\hdag T ,\,D) \circ \mathrm{oub} _{D,T} (\E ^{(\bullet)})
\to 
\E ^{(\bullet)}
\end{equation}
is an isomorphism of
$\smash{\underrightarrow{LD}} ^{\mathrm{b}} _{\Q,\mathrm{qc}} ( \smash{\widetilde{\D}} _{\fP /\S } ^{(\bullet)}(T))$.
\item The functorial in $\E ^{(\bullet)}$ canonical morphism :
\begin{equation}
\label{oub-pl-fid-iso2}
\mathrm{oub}_{D, T}  (\E ^{ (\bullet)} )
\to
\mathrm{oub}_{D, T} \circ (\hdag T, D) \circ \mathrm{oub}_{D, T}(\E ^{ (\bullet)} ) 
\end{equation}
is an isomorphism
of $\underrightarrow{LD} ^{\mathrm{b}}  _{\Q, \mathrm{qc}}  ( \smash{\widetilde{\D}} _{\fP /\S } ^{(\bullet)}(D))$.
\item 
\label{oub-pl-fid-iso3}
The functor
$\mathrm{oub} _{D,T}\colon 
\underrightarrow{LD} ^{\mathrm{b}}  _{\Q, \mathrm{qc}} ( \smash{\widetilde{\D}} _{\fP /\S } ^{(\bullet)}(T))
\to 
\underrightarrow{LD} ^{\mathrm{b}}  _{\Q, \mathrm{qc}} ( \smash{\widetilde{\D}} _{\fP /\S } ^{(\bullet)}(D))$
is fully faithful.
\end{enumerate}

\end{prop}

\begin{proof}
We can copy word by word the proof of \cite[3.2.6]{caro-stab-sys-ind-surcoh}. 
\end{proof}

\begin{coro}
\label{gen-oub-pl-fid}
Let $\E ^{(\bullet)}\in \smash{\underrightarrow{LD}} ^{\mathrm{b}} _{\Q,\mathrm{qc}} ( \smash{\widetilde{\D}} _{\fP /\S } ^{(\bullet)}(D))$.
The functorial in $\E ^{(\bullet)}$ canonical morphism 
\begin{equation}
\label{gen-oub-pl-fid-iso1}
(\hdag T ,\,D ') \circ \mathrm{oub} _{D',D} (\E ^{(\bullet)})
\to 
(\hdag T ,~D) (\E ^{(\bullet)})
\end{equation}
is an isomorphism of
$\smash{\underrightarrow{LD}} ^{\mathrm{b}} _{\Q,\mathrm{qc}} ( \smash{\widetilde{\D}} _{\fP /\S } ^{(\bullet)}(T))$.
\end{coro}

\begin{proof}
We can copy word by word the proof of \cite[3.2.7]{caro-stab-sys-ind-surcoh}. 
\end{proof}

\begin{ntn}
\label{nota-hag-sansrisque}
Let $D \subset T \subset T' $ be some divisors of $P$.
Following \ref{gen-oub-pl-fid},
by forgetting to write some forgetful functors, 
the functors $(\hdag T',~D) $ 
and
$(\hdag T', ~T) $ are canonically isomorphic over
$\smash{\underrightarrow{LD}} ^{\mathrm{b}} _{\Q ,\mathrm{qc}}
(\smash{\widetilde{\D}} _{\fP } ^{(\bullet)}(T))$.
Hence, we can simply write
$(\hdag T') $ in both case.

\end{ntn}

\begin{ntn}
\label{nota-tdf}
We denote by 
$D  ^{\mathrm{b}} _{\mathrm{tdf}}
( \smash{\widetilde{\D}} _{\fP /\S } ^{(\bullet)} (T ))$
the full subcategory of 
$D  ^{\mathrm{b}} 
(\smash{\widetilde{\D}} _{\fP /\S } ^{(\bullet)} (T ))$
consisting of complexes of finite Tor-dimension. 
We denote by 
$\smash{\underrightarrow{LD}}  ^{\mathrm{b}} _{\Q, \mathrm{qc}, \mathrm{tdf}}
(\smash{\widetilde{\D}} _{\fP /\S } ^{(\bullet)} (T ))$
the strictly full subcategory of 
$\smash{\underrightarrow{LD}}  ^{\mathrm{b}} _{\Q, \mathrm{qc}}
(\smash{\widetilde{\D}} _{\fP /\S } ^{(\bullet)} (T ))$
consisting of objects isomorphic in 
$\smash{\underrightarrow{LD}}  ^{\mathrm{b}} _{\Q, \mathrm{qc}}
(\smash{\widetilde{\D}} _{\fP /\S } ^{(\bullet)} (T ))$
to an object of $D  ^{\mathrm{b}} _{\mathrm{tdf}}
( \smash{\widetilde{\D}} _{\fP /\S } ^{(\bullet)} (T ))$.

\end{ntn}

\begin{coro}
\begin{enumerate}
\label{def-otimes-coh1&2qc}
\item 
The bifunctor \ref{def-otimes-coh1} factorizes throught the bifunctor 
\begin{align}
\label{def-otimes-coh1qc}
-
 \smash{\widehat{\otimes}}
^\L _{\widetilde{\B} ^{(\bullet)}  _{\fP} ( T) }
 -
\colon
\smash{\underrightarrow{LD}}  ^{\mathrm{b}} _{\Q, \mathrm{qc}}
(\overset{^\mathrm{?}}{} \smash{\widetilde{\D}} _{\fP /\S } ^{(\bullet)} (T ))
\times 
\smash{\underrightarrow{LD}}  ^{\mathrm{b}} _{\Q, \mathrm{qc}}
(\overset{^\mathrm{l}}{} \smash{\widetilde{\D}} _{\fP /\S } ^{(\bullet)} (T ))
&
\to 
\smash{\underrightarrow{LD}}  ^{\mathrm{b}}  _{\Q, \mathrm{qc}}
(\overset{^\mathrm{?}}{} \smash{\widetilde{\D}} _{\fP /\S } ^{(\bullet)} (T )).
\end{align}

\item 
\label{def-otimes-coh2qc}
With notation  \ref{nota-tdf}, 
we have the equality 
$\smash{\underrightarrow{LD}}  ^{\mathrm{b}} _{\Q, \mathrm{qc}}
(\smash{\widetilde{\D}} _{\fP /\S } ^{(\bullet)} (T ))
=\smash{\underrightarrow{LD}}  ^{\mathrm{b}} _{\Q, \mathrm{qc}, \mathrm{tdf}}
(\smash{\widetilde{\D}} _{\fP /\S } ^{(\bullet)} (T ))$.

\end{enumerate}
\end{coro}

\begin{proof}
We can copy word by word the proof of \cite[3.2.9]{caro-stab-sys-ind-surcoh} (for the second statement, the careful reader might notice in fact
we need the slightly more precise argument that the cohomological dimension of our rings can be bounded independently of the level $m$). 
\end{proof}

\begin{rem}
We have $\smash{\widetilde{\D}} _{\fP /\S } ^{(m)} (T ) \in 
D  ^{\mathrm{b}} _{\mathrm{tdf}}
( \smash{\widehat{\D}} _{\fP /\S } ^{(0)} )$, with flat amplitude bounded independently of the level $m$.
Indeed, 
$\D _{P /S } ^{(m)} \in 
D  ^{\mathrm{b}} _{\mathrm{tdf}}
( \D _{P /S } ^{(0)} )$, with flat amplitude bounded by 
the cohomological dimension of $ \D _{P /S } ^{(0)}$
(see \cite[I.5.9]{sga6}).
Since 
$\smash{\widehat{\D}} _{\fP /\S } ^{(m)} \in 
D  ^{\mathrm{b}} _{\mathrm{qc}}
( \smash{\widehat{\D}} _{\fP /\S } ^{(0)} )$,
then from \cite[3.2.3]{Beintro2} (still valid in our context), 
$\smash{\widehat{\D}} _{\fP /\S } ^{(m)} \in 
D  ^{\mathrm{b}} _{\mathrm{tdf}}
( \smash{\widehat{\D}} _{\fP /\S } ^{(0)} )$, with flat amplitude bounded by 
the cohomological dimension of $ \smash{\widehat{\D}} _{\fP /\S } ^{(0)}$
(in fact the proof of  \cite[3.2.3]{Beintro2} shows more precisely the preservation of flat amplitude).
Then, using 
\ref{rema-dim-coh-finie},
$\smash{\widetilde{\D}} _{\fP /\S } ^{(m)} (T ) \in 
D  ^{\mathrm{b}} _{\mathrm{tdf}}
( \smash{\widehat{\D}} _{\fP /\S } ^{(0)} )$, with flat amplitude bounded independently of the level $m$.
Hence, 
$\smash{\widetilde{\D}} _{\fP /\S } ^{(\bullet)} (T ) \in 
D  ^{\mathrm{b}} _{\mathrm{tdf}}
( \smash{\widehat{\D}} _{\fP /\S } ^{(\bullet)} )$.

\end{rem}

\begin{coro}
Let
$\M ^{(\bullet)}
\in \smash{\underrightarrow{LD}}  ^{\mathrm{b}} _{\Q, \mathrm{qc}}
(\overset{^\mathrm{?}}{} \smash{\widetilde{\D}} _{\fP /\S } ^{(\bullet)}(D))$,
and
$\E ^{(\bullet)} \in \smash{\underrightarrow{LD}}  ^{\mathrm{b}} _{\Q, \mathrm{qc}}
(\smash{\widetilde{\D}} _{\fP /\S } ^{(\bullet)}(D))$.
We have the canonical isomorphism in 
$\smash{\underrightarrow{LD}} ^{\mathrm{b}} _{\Q, \mathrm{qc}} 
(\overset{^\mathrm{?}}{}  \smash{\widetilde{\D}} _{\fP /\S } ^{(\bullet)}(T))$
of the form 
\begin{equation}
\label{hdagTDotimes}
 (\hdag T ,\,D) ( \M ^{(\bullet)} )
 \smash{\widehat{\otimes}} ^\L
_{\widetilde{\B} ^{(\bullet)}  _{\fP} ( T) } 
  (\hdag T ,\,D) ( \E ^{(\bullet)})
  \riso 
   (\hdag T ,\,D)  
   \left (
   \M ^{(\bullet)}
   \smash{\widehat{\otimes}} ^\L
_{\widetilde{\B} ^{(\bullet)}  _{\fP} ( D) } 
   \E ^{(\bullet)}
      \right ) .
\end{equation}
\end{coro}

\begin{proof}
Using the bounded quasi-coherence of our objects, 
this is straightforward from the associativity of the tensor products
(use the equivalence of categories of 
\cite[3.2.3]{Beintro2} to reduce to the case of usual tensor products of complexes).
\end{proof}

\begin{coro}
Let
$\M ^{(\bullet)}
\in \smash{\underrightarrow{LD}}  ^{\mathrm{b}} _{\Q, \mathrm{qc}}
(\overset{^\mathrm{?}}{} \smash{\widetilde{\D}} _{\fP /\S } ^{(\bullet)}(T))$,
and
$\E ^{(\bullet)} \in \smash{\underrightarrow{LD}}  ^{\mathrm{b}} _{\Q, \mathrm{qc}}
(\smash{\widetilde{\D}} _{\fP /\S } ^{(\bullet)}(T))$.
We have the isomorphism
\begin{equation}
\label{oubTDDotimes}
\mathrm{oub} _{D,T} ( \M ^{(\bullet)} )
 \smash{\widehat{\otimes}} ^\L
_{\widetilde{\B} ^{(\bullet)}  _{\fP} ( D) } 
\mathrm{oub} _{D,T} ( \E ^{(\bullet)})
  \riso 
\mathrm{oub} _{D,T}
   \left (
   \M ^{(\bullet)}
   \smash{\widehat{\otimes}} ^\L
_{\widetilde{\B} ^{(\bullet)}  _{\fP} ( T) } 
   \E ^{(\bullet)}
      \right ) .
\end{equation}

\end{coro}

\begin{proof}
Using \ref{oub-pl-fid-iso1}, we get 
$   \M ^{(\bullet)}
   \smash{\widehat{\otimes}} ^\L
_{\widetilde{\B} ^{(\bullet)}  _{\fP} ( T) } 
   \E ^{(\bullet)}
\riso
 \M ^{(\bullet)} 
 \smash{\widehat{\otimes}} ^\L
_{\widetilde{\B} ^{(\bullet)}  _{\fP} ( T) } 
\left (
\widetilde{\B} ^{(\bullet)}  _{\fP} ( T) 
 \smash{\widehat{\otimes}} ^\L
_{\widetilde{\B} ^{(\bullet)}  _{\fP} ( D) } 
 (\mathrm{oub} _{D,T} ( \E ^{(\bullet)}))
 \right ) $.
We conclude by associativity of the tensor product. 
\end{proof}

\subsection{Composition of localisation functors}

\begin{lemm}
\label{lem1-hdagT1T2}
Let $T, ~T'$ be two divisors of  $P$ 
whose irreducible components are distinct, 
$\U''$ the open set of  $\fP$ complementary to $T \cup T'$.

\begin{enumerate}
\item For any $i \in \N$, 
the canonical morphism
$\smash{\widetilde{\B}} _{P _i} ^{(m)} (T ) \otimes ^{\L}_{\O _{P _i} }
\smash{\widetilde{\B}} _{P _i} ^{(m)} (T' ) \to 
\smash{\widetilde{\B}} _{P _i} ^{(m)} (T ) \otimes _{\O _{P _i} }
\smash{\widetilde{\B}} _{P _i} ^{(m)} (T' )$
is an isomorphism.

\item The canonical morphism
$\smash{\widetilde{\B}} _{\fP} ^{(m)} (T ) \widehat{\otimes} ^{\L} _{\O _{\fP} }
\smash{\widetilde{\B}} _{\fP} ^{(m)} (T ')
\to
\smash{\widetilde{\B}} _{\fP} ^{(m)} (T ) \widehat{\otimes} _{\O _{\fP} }
\smash{\widetilde{\B}} _{\fP} ^{(m)} (T ')$
is an isomorphism
and 
the $\O _{\fP} $-algebra
$\smash{\widetilde{\B}} _{\fP} ^{(m)} (T ) \widehat{\otimes} _{\O _{\fP} }
\smash{\widetilde{\B}} _{\fP} ^{(m)} (T ')$
has no $p$-torsion.

\item The canonical morphism of  $\O _{\fP}$-algebras
$\smash{\widetilde{\B}} _{\fP} ^{(m)} (T ) \widehat{\otimes} _{\O _{\fP} }
\smash{\widetilde{\B}} _{\fP} ^{(m)} (T ')
\to 
j _* \O _{\U ''} $,
where $j\colon \U '' \hookrightarrow \fP$ is the inclusion,
is a monomorphism.

\item Let  $\chi,~\lambda \colon \N \to \N$ defined respectively by setting for any integer $m\in \N$ 
$\chi (m) := p ^{p-1}$ and $\lambda (m) := m +1$.
We have two canonical monomorphisms 
$\alpha ^{(\bullet)}
\colon
\smash{\widetilde{\B}} _{\fP} ^{(\bullet)} (T ) \widehat{\otimes} _{\O _{\fP} }
\smash{\widetilde{\B}} _{\fP} ^{(\bullet)} (T ')
\to 
\smash{\widetilde{\B}} _{\fP} ^{(\bullet)} (T \cup T')$
and
$\beta ^{(\bullet)}\colon 
\smash{\widetilde{\B}} _{\fP} ^{(\bullet)} (T \cup T')
\to 
\lambda ^{*} \chi ^{*} (
\smash{\widetilde{\B}} _{\fP} ^{(\bullet)} (T ) \widehat{\otimes} _{\O _{\fP} }
\smash{\widetilde{\B}} _{\fP} ^{(\bullet)} (T '))$
such that 
$\lambda ^{*} \chi ^{*} (\alpha  ^{(\bullet)})
\circ \beta  ^{(\bullet)}$ 
and
$\beta  ^{(\bullet)} \circ \alpha  ^{(\bullet)}$ 
are the canonical morphisms.
\end{enumerate}

\end{lemm}

\begin{proof}
We can copy word by word the proof of \cite[3.2.10]{caro-stab-sys-ind-surcoh}. 
\end{proof}

\begin{prop}
\label{hdagT'T=cup}
Let  $T', T $ be two divisors of $P$.
For any $\E ^{(\bullet)}
\in 
\underrightarrow{LD} ^{\mathrm{b}} _{\Q,\mathrm{qc}} (\smash{\widetilde{\D}} _{\fP /\S } ^{(\bullet)})$, 
we have the isomorphism
$(\hdag T ') \circ (\hdag T) (\E ^{(\bullet)})
\to
(T '\cup T) (\E ^{(\bullet)})$
functorial in  $T,~T',~\E ^{(\bullet)}$.
\end{prop}

\begin{proof}
Using \ref{lem1-hdagT1T2}, 
we can copy word by word the proof of \cite[3.2.11]{caro-stab-sys-ind-surcoh}. 
\end{proof}

\subsection{A coherence stability criterion by localisation outside a divisor}

\begin{thm}
\label{limTouD}
Let  $T' \supset T$ be a divisor,
$\E ^{(\bullet)}
\in 
\underrightarrow{LD} ^{\mathrm{b}} _{\Q, \mathrm{coh}} (\smash{\widetilde{\D}} _{\fP } ^{(\bullet)} (T))$
and
$\E := 
\underrightarrow{\lim}
\E ^{(\bullet)}
\in D ^{\mathrm{b}} _{\mathrm{coh}} (\smash{\D} ^\dag _{\fP } (\hdag T) _{\Q})$.
We suppose that the morphism 
$\E \to (\hdag T',T) (\E)$
is an isomorphism of 
$D ^{\mathrm{b}}  (\smash{\D} ^\dag _{\fP } (\hdag T) _{\Q})$.
Then, the canonical morphism
$\E ^{(\bullet)} \to 
(\hdag T',T) (\E ^{(\bullet)})$
is an isomorphism of 
$\underrightarrow{LD} ^{\mathrm{b}} _{\Q, \mathrm{coh}} (\smash{\widetilde{\D}} _{\fP } ^{(\bullet)}(T))$.
\end{thm}

\begin{proof}
We can copy \cite[3.5.1]{caro-stab-sys-ind-surcoh}.
\end{proof}

\begin{cor}
\label{coro1limTouD}
Let  $T' \supset T$ be a divisor,
$\E ^{\prime (\bullet)}
\in 
\underrightarrow{LD} ^{\mathrm{b}} _{\Q, \mathrm{coh}} (\smash{\widetilde{\D}} _{\fP } ^{(\bullet)} (T'))$
and
$\E ':= 
\underrightarrow{\lim}
\E ^{\prime (\bullet)}
\in D ^{\mathrm{b}} _{\mathrm{coh}} (\smash{\D} ^\dag _{\fP } (\hdag T ') _{\Q})$.
If
$\E '\in D ^{\mathrm{b}}  _{\mathrm{coh}} (\smash{\D} ^\dag _{\fP } (\hdag T) _{\Q})$,
then
$\E ^{\prime (\bullet)}
\in 
\underrightarrow{LD} ^{\mathrm{b}} _{\Q, \mathrm{coh}} (\smash{\widetilde{\D}} _{\fP } ^{(\bullet)}(T))$.
\end{cor}

\begin{coro}
\label{coro2limTouD}
Let  $T' \supset T$ be a divisor,
$\E \in D ^{\mathrm{b}}  _{\mathrm{coh}} (\smash{\D} ^\dag _{\fP } (\hdag T) _{\Q}) \cap 
D ^{\mathrm{b}}  _{\mathrm{coh}} (\smash{\D} ^\dag _{\fP } (\hdag T') _{\Q})$.
Let 
$\E ^{(\bullet)}
\in 
\underrightarrow{LD} ^{\mathrm{b}} _{\Q, \mathrm{coh}} (\smash{\widetilde{\D}} _{\fP } ^{(\bullet)} (T))$
and
$\E ^{\prime (\bullet)}
\in 
\underrightarrow{LD} ^{\mathrm{b}} _{\Q, \mathrm{coh}} (\smash{\widetilde{\D}} _{\fP } ^{(\bullet)} (T'))$
such that we have the $\smash{\D} ^\dag _{\fP } (\hdag T) _{\Q}$-linear isomorphisms of the form
$\underrightarrow{\lim}
\E ^{(\bullet)}
\riso 
\E$
and
$\underrightarrow{\lim}
\E ^{\prime (\bullet)}
\riso 
\E$.
Then, we have the isomorphism
$\E ^{(\bullet)} \riso \E ^{\prime (\bullet)}$
of
$\underrightarrow{LD} ^{\mathrm{b}} _{\Q, \mathrm{coh}} (\smash{\widetilde{\D}} _{\fP } ^{(\bullet)} (T))$.
\end{coro}

\begin{proof}
This is a straightforward consequence of \ref{coro1limTouD} and 
of the full faithfulness of the functor
$\underrightarrow{\lim}$ 
on 
$\underrightarrow{LD} ^{\mathrm{b}} _{\Q, \mathrm{coh}} (\smash{\widetilde{\D}} _{\fP } ^{(\bullet)} (T))$.
\end{proof}

\begin{prop}
\label{stab-coh-oub-DT}
Let  $T \subset D \subset T' $ be some divisors of  $P$.
\begin{enumerate}
\item Let 
$\E ^{(\bullet)}
\in 
\underrightarrow{LD} ^{\mathrm{b}} _{\Q, \mathrm{coh}} (\smash{\widetilde{\D}} _{\fP } ^{(\bullet)}(T))
\cap \underrightarrow{LD} ^{\mathrm{b}} _{\Q, \mathrm{coh}} (\smash{\widetilde{\D}} _{\fP } ^{(\bullet)} (T'))$.
Then 
$\E ^{(\bullet)}
\in 
 \underrightarrow{LD} ^{\mathrm{b}} _{\Q, \mathrm{coh}} (\smash{\widetilde{\D}} _{\fP } ^{(\bullet)} (D))$.
\item Let 
$\E \in 
D ^{\mathrm{b}} _{\mathrm{coh}}( \smash{\D} ^\dag _{\fP } (\hdag T) _{\Q} )
\cap 
D ^{\mathrm{b}} _{\mathrm{coh}}( \smash{\D} ^\dag _{\fP } (\hdag T') _{\Q} )$.
Then 
$\E 
\in 
D ^{\mathrm{b}} _{\mathrm{coh}}( \smash{\D} ^\dag _{\fP } (\hdag D) _{\Q} )$.
\end{enumerate}
\end{prop}

\begin{proof}
Using \ref{oub-pl-fid-iso1}, we check that the canonical morphism 
$(\hdag D,\,T) \circ 
\mathrm{oub} _{T ,D} 
( \mathrm{oub} _{D ,T'} (\E ^{(\bullet)}))
\to 
\mathrm{oub} _{D ,T'} (\E ^{(\bullet)})$
of $ \underrightarrow{LD} ^{\mathrm{b}} _{\Q, \mathrm{qc}} (\smash{\widetilde{\D}} _{\fP } ^{(\bullet)} (D))$
is an isomorphism.
Hence, we get the first assertion. 
Using \ref{coro1limTouD}, this yields the second one.
\end{proof}

\begin{ntn}
\label{fct-qcoh2coh}
Let $\fP$ and $\mathfrak{Q}$ be two smooth formal schemes over $\S $,
$T$ be a divisor of $P$,
$U$ be a divisor of $Q$,
and 
$\phi ^{(\bullet)}
\colon 
\smash{\underrightarrow{LD}} ^{\mathrm{b}} _{\Q,\mathrm{qc}}
( \smash{\widetilde{\D}} _{\fP /\S } ^{(\bullet)}(T))
\to 
\smash{\underrightarrow{LD}} ^{\mathrm{b}} _{\Q,\mathrm{qc}}
( \smash{\widetilde{\D}} _{\mathfrak{Q} /\S } ^{(\bullet)}(U))$
be a functor.
We denote by 
$\mathrm{Coh} _{T} ( \phi ^{(\bullet)})
\colon 
D ^\mathrm{b} _\mathrm{coh} ( \smash{\D} ^\dag _{\fP /\S } (\hdag T) _{\Q} ) \to
D ^\mathrm{b}  ( \smash{\D} ^\dag _{\mathfrak{Q} /\S } (\hdag U) _{\Q} )$
the functor 
$\mathrm{Coh} _{T} ( \phi ^{(\bullet)}) := \underrightarrow{\lim} \circ \phi ^{(\bullet)}\circ (\underrightarrow{\lim} _T ) ^{-1}$,
where $(\underrightarrow{\lim} _T ) ^{-1}$ is a quasi-inverse functor of 
the equivalence of categories
\begin{equation}
\label{eq-catLDBer-LD-D}
\underrightarrow{\lim} 
\colon 
\underrightarrow{LD} ^{\mathrm{b}}  _{\Q, \mathrm{coh}} (\smash{\widetilde{\D}} _{\fP } ^{(\bullet)} (T))
\overset{\ref{eqcatLD=DSM-fonct-coh}}{\cong} 
D ^{\mathrm{b}} _{\mathrm{coh}}
(\underrightarrow{LM} _{\Q} (\smash{\widetilde{\D}} _{\fP } ^{(\bullet)} (T)))
\overset{\ref{eqcatcoh}}{\cong}
D ^{\mathrm{b}} _{\mathrm{coh}}( \smash{\D} ^\dag _{\fP } (\hdag T) _{\Q} ).
\end{equation}
\end{ntn}

\begin{rem}
\label{rema-fct-qcoh2coh}
Let  $T \subset T'$ be a second divisor. 
Let  $\E \in 
D ^\mathrm{b} _\mathrm{coh} ( \smash{\D} ^\dag _{\fP } (\hdag T') _{\Q} ) 
\cap D ^\mathrm{b} _\mathrm{coh} ( \smash{\D} ^\dag _{\fP } (\hdag T) _{\Q} )$.
Using \ref{coro2limTouD},  
the corresponding objects 
of $\smash{\underrightarrow{LD}} ^{\mathrm{b}} _{\Q ,\mathrm{coh}}
(\smash{\widetilde{\D}} _{\fP } ^{(\bullet)}(T))$
and
$\smash{\underrightarrow{LD}} ^{\mathrm{b}} _{\Q ,\mathrm{coh}}
(\smash{\widetilde{\D}} _{\fP } ^{(\bullet)}(T'))$
(via the equivalence of categories \ref{eq-catLDBer-LD-D})
are isomorphic. 
With notation \ref{fct-qcoh2coh}, the functors 
$\mathrm{Coh} _{T} ( \phi ^{(\bullet)}) $
and
$\mathrm{Coh} _{T'} ( \phi ^{(\bullet)}) $
are then isomorphic over 
$D ^\mathrm{b} _\mathrm{coh} ( \smash{\D} ^\dag _{\fP } (\hdag T') _{\Q} ) 
\cap D ^\mathrm{b} _\mathrm{coh} ( \smash{\D} ^\dag _{\fP } (\hdag T) _{\Q} )$.
\end{rem}

\begin{rem}
\label{rem-hag-sansrisque}

\begin{itemize}

\item For any divisors $D \subset T$, 
we have the isomorphism of functors 
$\mathrm{Coh} _{D} ((\hdag T' ,~D) ) \riso (\hdag T' ,~D) $
(see notation \ref{hdagT-nota}) 
Hence, both notation of \ref{hdagT-nota} are compatible.

\item

Let $T$ and $D \subset D' $ be some divisors of $P$.
We obtain the functor
$(\hdag T ) := \mathrm{Coh} _{D} ((\hdag T) )
\colon D ^\mathrm{b} _\mathrm{coh} ( \smash{\D} ^\dag _{\fP } (\hdag D) _{\Q} ) 
\to D ^\mathrm{b} _\mathrm{coh} ( \smash{\D} ^\dag _{\fP } (\hdag T \cup D) _{\Q} ) $
(see notation \ref{nota-hag-sansrisque}).
With the remark \ref{rema-fct-qcoh2coh}, 
since the functors 
$\mathrm{Coh} _{D} ((\hdag T) )$ and
$\mathrm{Coh} _{D'} ((\hdag T) )$
are isomorphic over
 $D ^\mathrm{b} _\mathrm{coh} ( \smash{\D} ^\dag _{\fP } (\hdag D) _{\Q} ) 
\cap D ^\mathrm{b} _\mathrm{coh} ( \smash{\D} ^\dag _{\fP } (\hdag D') _{\Q} )$,
then it is not necessary to clarify $D$.

\end{itemize}

\end{rem}

\section{Extraordinary inverse image, direct image, duality}

\subsection{Definitions of the functors}
\label{subsection3.1}

Let  $f \colon \fP ^{\prime } \to \fP $ be a morphism of smooth formal schemes over $\S $,
$T$ and $T'$ be some divisors of respectively $P$ and $P'$ such that 
$f ( P '\setminus T' ) \subset P \setminus T$.
We define in this section the extraordinary inverse image and 
direct image by $f$ with overconvergent singularities along $T$ and $T'$,
and the dual functor.

We fix  $\lambda _0\colon \N \to \N$ an increasing map such that 
$\lambda _{0} (m) \geq m$ for any $m \in \N$. 
We set 
$\widetilde{\B} ^{(m)} _{\fP} ( T):= \widehat{\B} ^{(\lambda _0 (m))} _{\fP} ( T)$ 
et
$\smash{\widetilde{\D}} _{\fP /\S } ^{(m)} (T):=
\widetilde{\B} ^{(m)} _{\fP} ( T)  \smash{\widehat{\otimes}} _{\O _{\fP}} \smash{\widehat{\D}} _{\fP /\S } ^{(m)}$.
Finally, we set 
$\smash{\D} _{P  _i /S  _i} ^{(m)} (T):= \V / \pi ^{i+1} \otimes _{\V} \smash{\widehat{\D}} _{\fP /\S  } ^{(m)} (T) 
=
\B ^{(m)} _{P _i} ( T)  \otimes _{\O _{P _i}} \smash{\D} _{P  _i/S  _i} ^{(m)}$
and
$\smash{\widetilde{\D}} _{P  _i/S  _i} ^{(m)} (T):=\widetilde{\B} ^{(m)} _{P _i} ( T)  \otimes _{\O _{P _i}} \smash{\D} _{P  _i/S  _i} ^{(m)}$.
We use similar notation by adding some primes,
e.g. 
$\widetilde{\B} ^{(m)} _{\fP'} ( T'):= \widehat{\B} ^{(\lambda _0 (m))} _{\fP'} ( T')$ .

\begin{ntn}
\begin{enumerate}[(a)]
\item Since $f ^{-1} (T) \subset T'$, 
we get the canonical morphism 
$f ^{-1} \widetilde{\B} _{P _i} ^{(m)} ( T ) \rightarrow \widetilde{\B} _{P' _i} ^{(m)} ( T ')$.
Hence, the sheaf
$\widetilde{\B} _{P' _i} ^{(m)} ( T ') \otimes _{\O _{P ^\prime _i}} f _i ^* \D _{P  _i/ S  _i} ^{(m)}
\riso 
\widetilde{\B} _{P' _i} ^{(m)} ( T ') \otimes _{f ^{-1} \widetilde{\B} _{P _i} ^{(m)} ( T ) } f ^{-1} \widetilde{\D} _{P  _i/ S  _i} ^{(m)} (T)$
is endowed with a canonical structure 
of ($ \widetilde{\D} _{P  ^{\prime } _i/ S  _i} ^{(m)} ( T ')$, $ f  ^{-1} \widetilde{\D} _{P  _i/ S  _i} ^{(m)}(T)$)-bimodule.
We denote this bimodule by $\widetilde{\D} ^{(m)} _{P ^{\prime } _i \rightarrow P  _i/ S  _i} ( T' , T)$.

\item By $p$-adic completion, we get the following
$(\smash{\widetilde{\D}} _{\fP ^{\prime }/\S }  ^{(m)}(T') , f ^{-1} \smash{\widetilde{\D}} _{\fP / \S } ^{(m)} (T))$-bimodule :
$\smash{\widetilde{\D}} _{\fP ^{\prime }\rightarrow \fP / \S } ^{(m)} ( T' , T):=
\underset{\underset{i}{\longleftarrow}}{\lim}\, \widetilde{\D}_{P ^{\prime } _i \rightarrow P  _i/ S  _i} ^{(m)}( T' , T)$.

\item We get a $(\D ^{\dag } _{\fP ^{\prime }/ \S } (\hdag T' )_{\Q}, f ^{-1} \D ^{\dag }_{\fP ^{ }/ \S } (\hdag T) _{\Q})$-bimodule
by setting
$\D ^{\dag} _{\fP ^{\prime }\rightarrow \fP ^{ }/ \S } (\hdag T' ,T) _{\Q}:=\underset{\underset{m}{\longrightarrow}}{\lim}\,
\smash{\widetilde{\D}} _{\fP ^{\prime }\rightarrow \fP ^{ }/ \S } ^{(m)} ( T' ,T)_{\Q}$.

\end{enumerate}

\end{ntn}

\begin{empt}
[Extraordinary inverse image]
\label{ntn-Lf!+*}
\begin{enumerate}
\item  The extraordinary inverse image by $f$ with overconvergent singularities along $T$ and $T'$ 
is a functor of the form
$f  ^{!(\bullet)} _{T',T} \colon
\smash{\underrightarrow{LD}} ^{\mathrm{b}} _{\Q,\mathrm{qc}} ( \smash{\widetilde{\D}} _{\fP /\S } ^{(\bullet)}(T))
\to
\smash{\underrightarrow{LD}} ^{\mathrm{b}} _{\Q,\mathrm{qc}} ( \smash{\widetilde{\D}} _{\fP ^{\prime }/\S }  ^{(\bullet)}(T'))$
which is defined 
for any $\E ^{(\bullet)} \in \smash{\underrightarrow{LD}}  ^\mathrm{b} _{\Q, \mathrm{qc}}
( \smash{\widetilde{\D}} _{\fP /\S } ^{(\bullet)} (T ))$ by setting:
\begin{equation}
\notag
f _{T',T} ^{ !(\bullet)} ( \E ^{(\bullet)}) :=
\smash{\widetilde{\D}} ^{(\bullet)} _{\fP ^{\prime } \rightarrow \fP /\S } (T',T)
\smash{\widehat{\otimes}} ^\L _{f ^{-1} \smash{\widetilde{\D}} ^{(\bullet)} _{\fP /\S } (T)}
f ^{-1} \E ^{(\bullet)} [ d_{\fP ' /\fP}],
\end{equation}
where the tensor product is defined similarly to \ref{predef-otimes-coh0}.
\item The extraordinary inverse image by $f$ with overconvergent singularities along $T$ and $T'$ 
is also a functor of the form
$f  ^{ !} _{T',T} \colon
D ^\mathrm{b} _\mathrm{coh} ( \smash{\D} ^\dag _{\fP /\S } (\hdag T) _{\Q} )
\to 
 D ^\mathrm{b} ( \smash{\D} ^\dag _{\fP ^{\prime }/\S }  (\hdag T') _{\Q} )$
which is defined 
for any $\E \in D ^{\mathrm{b}} _{\mathrm{coh}} ( \D ^{\dag} _{\fP /\S } (\hdag T ) _{\Q})$ by setting:
\begin{equation}
\label{def-image-inv-extr}
f  ^{ !} _{T' , T} (\E ):=\D ^{\dag} _{\fP ^{\prime }\rightarrow \fP  }  ( \hdag T' , T ) _{\Q}
\otimes ^{\L} _{ f ^{-1} \D ^{\dag} _{\fP /\S } (\hdag T ) _{\Q}} f ^{-1} \E [d _{\fP '/\fP} ].
\end{equation}

\item Mostly when $f$ is smooth, we can also consider the functors
$\L f _{T',T} ^{*(\bullet)} 
:= 
f _{T',T} ^{ !(\bullet)} 
[ - d_{\fP ' /\fP}]$,
and
$\L f  ^{ *} _{T',T}  : = f  ^{ !} _{T',T} [ - d_{\fP ' /\fP}]$.
Beware that our notation might be misleading since 
$\L f _{T',T} ^{*(\bullet)} $ is not  necessarily a left derived functor of some functor (except for coherent complexes).
When $f$ is smooth, these  functors are t-exact over coherent complexes, and we denote them respectively
$f _{T',T} ^{*(\bullet)} $ and $f  ^{ *} _{T',T} $.

\item When $T '= f ^{-1} (T)$, we simply write respectively
$  f _{T} ^{ !(\bullet)} $, $f   _{T} ^!$, and $f _T ^*$. If moreover $T$ is empty, 
we write $  f^{ !(\bullet)} $, 
$f ^!$, and $f ^*$.

\end{enumerate}

\end{empt}

\begin{ntn}

\begin{enumerate}[(a)]
\item We define a ($f ^{-1} \widetilde{\D} _{P  _i/ S  _i} ^{(m)}(T)$, $ \widetilde{\D} _{P ^{\prime } _i} ^{(m)} ( T ')$)-bimodule by setting
$$\widetilde{\D} ^{(m)} _{P  _i \leftarrow P ^{\prime   }_i/ S  _i} ( T , T'):=
\widetilde{\B} _{P' _i} ^{(m)} ( T ') \otimes _{\O _{P ^\prime _i}}
\left (\omega _{P ^{\prime } _i/S  _i} \otimes _{\O _{P ' _i}}f ^* _l \left ( \D _{P  _i/ S  _i} ^{(m)}(T) \otimes _{\O _{P _i}} \omega ^{-1} _{P  _i/ S  _i} \right)\right),$$
where the symbol $l$ means that we choose the left structure of left $\D _{P  _i/ S  _i} ^{(m)}(T) $-module.

\item This yields by completion the 
$(f ^{-1} \smash{\widetilde{\D}} _{\fP / \S } ^{(m)} (T),~\smash{\widetilde{\D}} _{\fP ^{\prime }/ \S } ^{(m)}(T') )$-bimodule :
$$\smash{\widetilde{\D}} _{\fP  \leftarrow \fP ^{\prime }/ \S } ^{(m)} ( T, T'):=
\underset{\underset{i}{\longleftarrow}}{\lim}\, \widetilde{\D}_{P  _i \leftarrow P ^{\prime } _i/ S  _i} ^{(m)}( T , T').$$

\item We get the
($f ^{-1} \D ^{\dag }_{\fP  / \S } (\hdag T) _{\Q}$, $\D ^{\dag } _{\fP ^{\prime }/ \S } (\hdag T' )_{\Q}$)-bimodule
$\D ^{\dag} _{\fP  \leftarrow \fP ^{\prime }/ \S } (\hdag T ,T ')_{\Q}:=
\underset{\underset{m}{\longrightarrow}}{\lim}
\smash{\widetilde{\D}} _{\fP  \leftarrow \fP ^{\prime }/ \S } ^{(m)} ( T, T') _\Q$.

\end{enumerate}

\end{ntn}

\begin{empt}
\begin{enumerate}[(a)]
\item The
direct image  by $f$ with overconvergent singularities along $T$ and $T'$ 
is a functor of the form
$f  ^{ (\bullet)} _{T,T',+} \colon
\smash{\underrightarrow{LD}} ^{\mathrm{b}} _{\Q,\mathrm{qc}} ( \smash{\widetilde{\D}} _{\fP ^{\prime }/\S }  ^{(\bullet)}(T'))
\to 
\smash{\underrightarrow{LD}} ^{\mathrm{b}} _{\Q,\mathrm{qc}} ( \smash{\widetilde{\D}} _{\fP /\S } ^{(\bullet)}(T))$
defined by setting,
for any $\E ^{\prime (\bullet)} \in \smash{\underrightarrow{LD}}  ^\mathrm{b} _{\Q, \mathrm{qc}}
( \smash{\widetilde{\D}} _{\fP ^{\prime }/\S }  ^{(\bullet)} (T '))$:
\begin{gather}\notag
f _{T,T',+} ^{ (\bullet)} ( \E ^{\prime (\bullet)} ):= 
\R f _* (
\smash{\widetilde{\D}} ^{(\bullet)} _{\fP  \leftarrow \fP ^{\prime }/\S } (T,T')
\smash{\widehat{\otimes}} ^\L _{\smash{\widetilde{\D}} ^{(\bullet)} _{\fP ^{\prime }/\S }  (T')}
\E ^{\prime (\bullet)}).
\end{gather}

\item The
direct image by $f$ with overconvergent singularities along $T$ and $T'$ is a functor of the form
$f _{T,T',+} \colon
D ^\mathrm{b} _\mathrm{coh} ( \smash{\D} ^\dag _{\fP ^{\prime }/\S }  (\hdag T') _{\Q} )
\to 
 D ^\mathrm{b} ( \smash{\D} ^\dag _{\fP /\S } (\hdag T) _{\Q} )$,
defined by setting,
for any $\E '\in D ^{\mathrm{b}} _{\mathrm{coh}} ( \D ^{\dag} _{\fP ^{\prime }/\S }  (\hdag T' ) _{\Q})$ :
\begin{equation}
\label{ftt'+}
f   _{T , T ', +}( \E '):=
 \R f_* (\D ^{\dag} _{\fP  \leftarrow \fP ^{\prime }/\S }  ( \hdag T , T ') _{\Q}
\otimes ^{\L} _{ \D ^{\dag} _{\fP ^{\prime }/\S }  (\hdag T ') _{\Q}} \E ').
\end{equation}

\item When $T '= f ^{-1} (T)$, we simply write respectively
$f  ^{ (\bullet)} _{T,+} $ and $f   _{T ,+}$. If moreover $T$ is empty, 
we write $f  ^{ (\bullet)} _{+} $ and $f   _{+}$. 
\end{enumerate}

\end{empt}

\begin{empt}
\label{coh-Qcoh}
With notation \ref{fct-qcoh2coh}, 
we have the isomorphism of functors
$\mathrm{Coh} _{T'} (f  ^{ (\bullet)}_{T , T ', +}) \riso f  ^{} _{T , T ', +}$ 
and
$\mathrm{Coh} _{T} ( f _{T',T} ^{ !(\bullet)}) \riso  f _{T',T} ^{ !}$
(this is checked similarly to \cite[4.3.2.2 and 4.3.7.1]{Beintro2}).

\end{empt}

We recall the following fundamental theorem.
\begin{thm}
[Noot-Huyghe]
\label{Noot-Huyghe-finitehomoldim}
The sheaf of rings
$\D ^{\dag} _{\fP  } (\hdag T ) _{\Q}$
is of finite ohomological dimension.
\end{thm}

\begin{proof}
This is proved in \cite{huyghe_finitude_coho}.
\end{proof}

\begin{ntn}
[Duality]
\label{ntn-dualfunctor}
\begin{enumerate}[(a)]
\item Let $\E \in D ^{\mathrm{b}} _{\mathrm{coh}}( \smash{\D} ^\dag _{\fP /\S } (\hdag T) _{\Q} )$.
The $\smash{\D} ^\dag _{\fP /\S } (\hdag T) _{\Q}$-linear dual of $\E$ is defined by setting
$$\DD _T (\E):= \R \mathcal{H} om _{ \smash{\D} ^\dag _{\fP /\S } (\hdag T) _{\Q} }
(\E,
 \smash{\D} ^\dag _{\fP /\S } (\hdag T) _{\Q} 
\otimes  _{\O _{\fP}} \omega _{\fP ^{ }/\S } ^{-1})) [d _{P}].$$
Following \ref{Noot-Huyghe-finitehomoldim}, 
we get 
$D ^{\mathrm{b}} _{\mathrm{coh}}( \smash{\D} ^\dag _{\fP /\S } (\hdag T) _{\Q} )
=
D ^{\mathrm{b}} _{\mathrm{parf}}( \smash{\D} ^\dag _{\fP /\S } (\hdag T) _{\Q} )$, 
where the right category is that of perfect bounded complexes of 
$ \smash{\D} ^\dag _{\fP /\S } (\hdag T) _{\Q} $-modules. 
This yields 
$ \DD _T (\E)
 \in 
 D ^{\mathrm{b}} _{\mathrm{coh}}( \smash{\D} ^\dag _{\fP /\S } (\hdag T) _{\Q} )$.
 Hence, by biduality, we get the equivalence of categories
$ \DD _T
\colon 
 D ^{\mathrm{b}} _{\mathrm{coh}}( \smash{\D} ^\dag _{\fP /\S } (\hdag T) _{\Q} )
 \cong
  D ^{\mathrm{b}} _{\mathrm{coh}}( \smash{\D} ^\dag _{\fP /\S } (\hdag T) _{\Q} )$.
  
\item   We denote by
$\DD ^{(\bullet)} _T
\colon 
\smash{\underrightarrow{LD}} ^{\mathrm{b}} _{\Q,\mathrm{coh}}
( \smash{\widetilde{\D}} _{\fP /\S } ^{(\bullet)}(T))
\to 
\smash{\underrightarrow{LD}} ^{\mathrm{b}} _{\Q,\mathrm{coh}}
( \smash{\widetilde{\D}} _{\fP /\S } ^{(\bullet)}(T))$
the equivalence of categories such that
$\mathrm{Coh} _{T} (\DD ^{(\bullet)} _T)
\riso 
\DD _T$.

\end{enumerate}

\end{ntn}

\subsection{Commutation of pullbacks with localization functors outside of a divisor}

We keep notation \ref{subsection3.1}.

\begin{lemm}
\label{lem-f!B}
Suppose $T':= f ^{-1} (T)$.
We have the canonical isomorphism 
$$\O _{P ' _i} \otimes  ^\L
_{f ^{-1} \O _{P _i}} f ^{-1} \B ^{(m)}  _{P _i} ( T) 
\riso 
\B ^{(m)}  _{P' _i} ( T') .$$
We have also the canonical isomorphism 
$f ^{!(\bullet)} (\widetilde{\B} ^{(\bullet)}  _{\fP } ( T) )
\riso 
\widetilde{\B} ^{(\bullet)}  _{\fP '} ( T')[d _{P'/P}]$
in
$\smash{\underrightarrow{LD}} ^\mathrm{b} _{\Q, \mathrm{qc}}
(\overset{^\mathrm{l}}{} \smash{\widetilde{\D}} _{\fP ^{\prime}/\S } ^{(\bullet)}( T'))$.

\end{lemm}

\begin{proof}
This is checked similarly to  \cite[5.2.1]{caro-stab-sys-ind-surcoh}.
\end{proof}

\begin{empt}
\label{oub-div-opcoh}
\begin{enumerate}
\item   Let  $\E ^{\prime (\bullet)} \in \smash{\underset{^{\longrightarrow}}{LD}} ^{\mathrm{b}} _{\Q ,\mathrm{qc}}
( \smash{\widetilde{\D}} _{\fP ^{\prime  }} ^{(\bullet)}(T'))$.
Similarly to \cite[1.1.9 ]{caro_courbe-nouveau}, we check that we have the canonical isomorphism
$oub _{T} \circ f ^{(\bullet)} _{T,T',+} (\E ^{\prime (\bullet)} )\riso f ^{(\bullet)} _+  \circ oub _{T'} (\E ^{\prime (\bullet)})$.
Hence, it is harmless to write by abuse of notation
$f ^{(\bullet)} _{+}$ instead of  $f ^{(\bullet)} _{T, T',+}$.

Using the remark 
\ref{rema-fct-qcoh2coh}
this yields that the functors 
$\mathrm{Coh} _{T'} (f ^{(\bullet)} _{T,T',+}) $
and
$\mathrm{Coh}  (f ^{(\bullet)} _{+}) $ 
are isomorphic over
$D ^\mathrm{b} _\mathrm{coh} ( \smash{\D} ^\dag _{\fP ^{\prime  },\Q} ) 
\cap D ^\mathrm{b} _\mathrm{coh} ( \smash{\D} ^\dag _{\fP ^{\prime  }} (\hdag T') _{\Q} )$.
Since 
we have the canonical isomorphisms of functors 
$\mathrm{Coh} _{T'} (f  ^{ (\bullet)}_{T , T ', +}) \riso f  _{T , T ', +}$ 
and
$\mathrm{Coh} (f  ^{ (\bullet)}_{+}) \riso f _+$ (\ref{coh-Qcoh}), 
then it is harmless to write 
$f _+$ instead of  $f   _{T, T',+}$
and we get the functor
$f _+
\colon 
D ^\mathrm{b} _\mathrm{coh} ( \smash{\D} ^\dag _{\fP ^{\prime  },\Q} ) 
\cap D ^\mathrm{b} _\mathrm{coh} ( \smash{\D} ^\dag _{\fP ^{\prime  }} (\hdag T') _{\Q} )
\to 
D ^\mathrm{b} _\mathrm{coh} ( \smash{\D} ^\dag _{\fP ,\Q} ) 
\cap D ^\mathrm{b} _\mathrm{coh} ( \smash{\D} ^\dag _{\fP} (\hdag T) _{\Q} )$.

\item 
\label{oub-div-opcohb)}
Let $D$ and $D'$ be some divisors of respectively $P$ and $P'$ such that 
$f ( P '\setminus D' ) \subset P \setminus D$, $D \subset T$, and
$D' \subset T'$.
Let  $\E ^{(\bullet)} \in \smash{\underset{^{\longrightarrow}}{LD}} ^{\mathrm{b}} _{\Q ,\mathrm{qc}}
( \smash{\widetilde{\D}} _{\fP /\S } ^{(\bullet)} (D))$.
Similarly to \cite[1.1.9]{caro_courbe-nouveau}, we check easily the isomorphism
$(\hdag T',D') \circ f ^{!(\bullet)} _{D',D} (\E^{(\bullet)} )
\riso 
f ^{!(\bullet)} _{T',T} \circ (\hdag T, D) (\E^{(\bullet)}) $.

\end{enumerate}
\end{empt}

\begin{empt}
Let
$\FF ^{(\bullet)}
,\E ^{(\bullet)} \in \smash{\underrightarrow{LD}}  ^{\mathrm{b}} _{\Q ,\mathrm{qc}}
(\smash{\widetilde{\D}} _{\fP /\S } ^{(\bullet)}(T))$.
We easily check  (see \cite[2.1.9.1]{caro-stab-prod-tens})  the following isomorphism of
$\smash{\underrightarrow{LD}} ^{\mathrm{b}} _{\Q ,\mathrm{qc}} ( \smash{\widetilde{\D}} _{\fP ^{\prime }/\S } ^{(\bullet)}(T'))$
\begin{equation}
\label{f!T'Totimes}
 f ^{!(\bullet)} _{T',T} ( \FF ^{(\bullet)} )
 \smash{\widehat{\otimes}} ^\L
_{\widetilde{\B} ^{(\bullet)}  _{\fP '} ( T ') } 
 f ^{!(\bullet)} _{T',T} ( \E ^{(\bullet)})
  \riso 
f ^{!(\bullet)} _{T',T}
   \left (
   \FF ^{(\bullet)}
   \smash{\widehat{\otimes}} ^\L
_{\widetilde{\B} ^{(\bullet)}  _{\fP} ( T) } 
   \E ^{(\bullet)}
\right ) [d _{P'/P}].
\end{equation}

\end{empt}

\begin{prop}
\label{f!commoub}
Suppose  $T ' = f ^{-1} (T)$.
\begin{enumerate}
\item Let  $\E ^{(\bullet)} \in \smash{\underset{^{\longrightarrow}}{LD}} ^{\mathrm{b}} _{\Q ,\mathrm{qc}}
( \smash{\widehat{\D}} _{\fP /\S } ^{(\bullet)})$.
We have the canonical isomorphism 
$$f ^{ !(\bullet)} \circ  oub _T \circ (\hdag T) (\E^{(\bullet)}) 
\riso 
oub _{T'} \circ (\hdag T ') \circ f ^{ !(\bullet)}   (\E^{(\bullet)}) ,$$ 
which we can simply write 
$f ^{ !(\bullet)}  \circ (\hdag T) (\E^{(\bullet)}) 
\riso 
 (\hdag T ') \circ f ^{ !(\bullet)}   (\E^{(\bullet)}) $. 
\item Let  $\E ^{(\bullet)} \in \smash{\underset{^{\longrightarrow}}{LD}} ^{\mathrm{b}} _{\Q ,\mathrm{qc}}
( \smash{\widetilde{\D}} _{\fP /\S } ^{(\bullet)} (T))$.
We have the canonical isomorphism
$$oub _{T'} \circ f ^{ !(\bullet)} _{T} (\E^{(\bullet)} )\riso f ^{!(\bullet)} \circ oub _T (\E^{(\bullet)}).$$
Hence, it is harmless to write by abuse of notation
$f ^{!(\bullet)} $ instead of $f ^{!(\bullet)} _T$. 

\end{enumerate}

\end{prop}

\begin{proof}
Using \ref{f!T'Totimes},\ref{lem-f!B}, for any $\E ^{(\bullet)} \in \smash{\underset{^{\longrightarrow}}{LD}} ^{\mathrm{b}} _{\Q ,\mathrm{qc}}
( \smash{\widehat{\D}} _{\fP /\S } ^{(\bullet)})$, 
we get the isomorphism
\begin{equation}
\notag
f ^{ !(\bullet)} \circ  oub _T \circ (\hdag T) (\E^{(\bullet)}) 
=
f ^{!(\bullet)}
   \left (
   \widetilde{\B} ^{(\bullet)}  _{\fP} ( T) 
   \smash{\widehat{\otimes}} ^\L
_{\O ^{(\bullet)}  _{\fP} } 
   \E ^{(\bullet)}
\right ) 
\riso
   \widetilde{\B} ^{(\bullet)}  _{\fP'} ( T') 
   \smash{\widehat{\otimes}} ^\L
_{\O ^{(\bullet)}  _{\fP'} } 
f ^{!(\bullet)} 
(  \E ^{(\bullet)})
=
oub _{T'} \circ (\hdag T ') \circ f ^{ !(\bullet)}   (\E^{(\bullet)}) .
\end{equation}
By using 
\ref{oub-pl-fid-iso1} and \ref{oub-div-opcoh}.\ref{oub-div-opcohb)},
we check the second part from the first one.
\end{proof}

\begin{rem}
\label{rem-f!TornoT}
With notation \ref{f!commoub}, 
using the remark 
\ref{rema-fct-qcoh2coh}
we check that the functors 
$\mathrm{Coh} _{T} (f ^{!(\bullet)} _T) $
and
$\mathrm{Coh}  (f ^{!(\bullet)} ) $ 
are isomorphic over
$D ^\mathrm{b} _\mathrm{coh} ( \smash{\D} ^\dag _{\fP,\Q} ) 
\cap D ^\mathrm{b} _\mathrm{coh} ( \smash{\D} ^\dag _{\fP } (\hdag T) _{\Q} )$.
Since 
we have the canonical isomorphisms of functors 
$\mathrm{Coh} _{T} (f ^{!(\bullet)} _T)  \riso f  _{T} ^!$ 
and
$\mathrm{Coh} _{T} (f ^{!(\bullet)} )  \riso f ^!$ (\ref{coh-Qcoh}), 
then it is harmless to write 
$f ^!$ instead of  $f   _{T} ^!$. 

\end{rem}

\subsection{Projection formula : commutation of pushforwards with localization functors outside of a divisor}
Let $T $ be a noetherian $\Z _{(p)}$-scheme of finite Krull dimension.
Let $u \colon Y  \to X $ be a morphism of quasi-compact smooth $T$-schemes.  
Remark that since $Y$ is noetherian, then 
$u$ is quasi-separated and quasi-compact. 
 Let $\B _X$ be an $\O _X$-algebra endowed with a compatible structure of 
left $\D ^{(m)} _{X/ T }$-module. 
Put 
$\widetilde{\D} ^{(m)} _{X/ T}
:= 
\B _X \otimes _{\O _X} \D ^{(m)} _{X/ T}$,
$\B _Y := u ^* (\B _X)$,
 $\widetilde{\D} ^{(m)} _{Y/ T}
:= 
\B _Y \otimes _{\O _Y} \D ^{(m)} _{Y/ T}$,
$\D _{Y\to X/T } ^{(m)}:=
u ^* \D _{X/T } ^{(m)}$,
$\smash{\widetilde{\D}} _{Y\to X/T} ^{(m)}:=
\B ^{(m)} _{Y}   \otimes _{\O _{Y}} 
\D _{Y\to X/T } ^{(m)}$.

\begin{empt}
Following \cite[3.6.5]{Tohoku}, since
$X$ is noetherian of finite Krull dimension $d _X$, then for $i> d _X$, 
for every sheaf $\E$ of abelian groups
we have 
$H ^i (X, \E)=0$. 
Then, following \cite[12.2.1]{EGAIII1},
we get 
that 
$R ^i u _* (\E) = 0$ for $i> d _X$
and every sheaf $\E$ of abelian groups. 
In particular, by definition (see \cite[12.1.1]{EGAIII1}),
the functor $u _*$ has finite (bounded by $d _X$) cohomological dimension on $\mathrm{Mod} (u ^{-1}\O _X)$, the category 
of $u ^{-1}\O _X$-modules, or on $\mathrm{Mod} (u ^{-1}\widetilde{\D} ^{(m)} _{X / T })$.

Let $P$ be the subset of objects of $\mathrm{Mod} (u ^{-1}\widetilde{\D} ^{(m)} _{X / T })$ which are $u _*$-acyclic. 
Remark that $P$ contains injective $u ^{-1}\widetilde{\D} ^{(m)} _{X / T }$-modules. 
Using the cohomological dimension finiteness of $u _*$, if
\begin{equation}
\G ^0 \to \G ^1 \to \cdots 
\to \G ^{d _X}
\to 
\E
\to 0
\end{equation}
is an exact sequence of $\mathrm{Mod} (u ^{-1}\O _X)$,
and $\G ^0,\dots, \G ^{d _X} \in P$, 
then $\E \in P$.
Using \cite[Lemma I.4.6.2]{HaRD},
this implies that for any complex
$\E \in K (u ^{-1}\widetilde{\D} ^{(m)} _{X / T })$ 
(resp.  $\E \in K ^{-} (u ^{-1}\widetilde{\D} ^{(m)} _{X / T })$,
resp. $\E \in K ^{+} (u ^{-1}\widetilde{\D} ^{(m)} _{X / T })$,
resp. $\E \in K ^{\mathrm{b}} (u ^{-1}\widetilde{\D} ^{(m)} _{X / T })$) 
there exists a quasi-isomorphism
$\E \riso \I$ where $\I \in  K (u ^{-1}\widetilde{\D} ^{(m)} _{X / T })$ 
(resp.  $\I \in K ^{-} (u ^{-1}\widetilde{\D} ^{(m)} _{X / T })$,
resp. $\I \in K ^{+} (u ^{-1}\widetilde{\D} ^{(m)} _{X / T })$,
resp. $\I \in K ^{\mathrm{b}} (u ^{-1}\widetilde{\D} ^{(m)} _{X / T })$) 
is a complex whose modules belong to $P$.
We get the functor
$\R u _* \colon D (u ^{-1}\widetilde{\D} ^{(m)} _{X / T }) \to D (\widetilde{\D} ^{(m)} _{X / T }) $
(resp. 
$\R u _* \colon D ^- ( u ^{-1}\widetilde{\D} ^{(m)} _{X / T }) \to D ^-( \widetilde{\D} ^{(m)} _{X / T }) $,
resp. 
$\R u _* \colon D ^+ (u ^{-1}\widetilde{\D} ^{(m)} _{X / T }) \to D ^+(\widetilde{\D} ^{(m)} _{X / T }) $,
resp. 
$\R u _* \colon D ^\mathrm{b} (u ^{-1}\widetilde{\D} ^{(m)} _{X / T }) \to D ^\mathrm{b}(\widetilde{\D} ^{(m)} _{X / T }) $)
which is computed by taking a resolution with objects in $P$.

Moreover, following \cite[II.2.1]{HaRD}
$\R u_*$ takes
$D ^{?} _{\mathrm{qc}} ( \O _X) $
into 
$D ^{?} _{\mathrm{qc}} ( \O _Y) $
with $? \in \{\emptyset, +,- ,\mathrm{b}\}$.
\end{empt}

\begin{prop}
\label{Prop1.2.21-caro-surc}
Suppose one of the following conditions: 
\begin{enumerate}[(a)]
\item 
Let $\FF 
\in 
D _{\mathrm{qc},\mathrm{tdf}}
( \overset{^\mathrm{r}}{} \widetilde{\D} ^{(m)} _{X / T })$, 
and 
$\G 
\in 
D 
( \overset{^\mathrm{l}}{} u ^{-1} \widetilde{\D} ^{(m)} _{X / T })$.
\item
Let $\FF 
\in 
D ^- _{\mathrm{qc}}
( \overset{^\mathrm{r}}{} \widetilde{\D} ^{(m)} _{X / T })$, 
and 
$\G 
\in 
D ^-
( \overset{^\mathrm{l}}{} u ^{-1} \widetilde{\D} ^{(m)} _{X / T })$.
\end{enumerate}
Then we have the following isomorphism
\begin{gather}
\label{Prop1.2.21-caro-surc-iso2}
\FF \otimes ^{\L} _{\widetilde{\D} ^{(m)} _{X / T }} \R u _* ( \G)
\riso 
\R u _* \left (u ^{-1}  \FF \otimes ^{\L} _{u ^{-1} \widetilde{\D} ^{(m)} _{X / T }} \G \right).
\end{gather}
Inverting $\mathrm{r}$ and $\mathrm{l}$ in the hypotheses, 
we get the isomorphism 
\begin{gather}
\label{Prop1.2.21-caro-surc-iso2bis}
\R u _* ( \G) \otimes ^{\L} _{\widetilde{\D} ^{(m)} _{X / T }} \FF
\riso 
\R u _* \left (\G  \otimes ^{\L} _{u ^{-1} \widetilde{\D} ^{(m)} _{X / T }} u ^{-1}  \FF\right).
\end{gather}

\end{prop}

\begin{proof}
Taking a left resolution of $\FF$ by flat $\widetilde{\D} ^{(m)} _{X / T }$-modules,
and a right resolution of $\G$ by $u ^{-1} \widetilde{\D} ^{(m)} _{X / T }$-modules which are
$u _*$-acyclic, we construct the 
morphism \ref{Prop1.2.21-caro-surc-iso2}.
To check that this is an isomorphism,
using 
\cite[I.7.1 (ii), (iii) and (iv)]{HaRD}
and 
\cite[VI.5.1]{sga4-1},
we reduce to the case where 
$\FF
=
\widetilde{\D} ^{(m)} _{X / T }$, which is obvious.
\end{proof}

\begin{cor}
\label{cor-Prop1.2.21-caro-surc}
Let $*,** \in \{ \mathrm{l},\mathrm{r}\}$ such that both are not equal to $\mathrm{r}$.
Suppose one of the following conditions: 
\begin{enumerate}[(a)]
\item 
Let $\FF 
\in 
D _{\mathrm{qc},\mathrm{tdf}}
( \overset{^\mathrm{*}}{} \widetilde{\D} ^{(m)} _{X / T })$, 
and 
$\G 
\in 
D 
( \overset{^\mathrm{**}}{} u ^{-1} \widetilde{\D} ^{(m)} _{X / T })$.
\item
Let $\FF 
\in 
D ^- _{\mathrm{qc}}
( \overset{^\mathrm{*}}{} \widetilde{\D} ^{(m)} _{X / T })$, 
and 
$\G 
\in 
D ^-
( \overset{^\mathrm{**}}{} u ^{-1} \widetilde{\D} ^{(m)} _{X / T })$.
\end{enumerate}
Then we have the following isomorphism
\begin{gather}
\label{Prop1.2.21-caro-surc-iso1}
\FF \otimes ^{\L} _{\B _X} \R u _* ( \G)
\riso 
\R u _* \left (u ^{-1}  \FF \otimes ^{\L} _{u ^{-1} \B _X} \G \right).
\end{gather}

\end{cor}

\begin{proof}
For instance, if $** = \mathrm{l}$, we get 
\begin{gather}
\notag
\FF \otimes ^{\L} _{\B _X} \R u _* ( \G)
\riso 
(\FF \otimes ^{\L} _{\B _X} 
\widetilde{\D} ^{(m)} _{X / T })
 \otimes ^{\L} _{\widetilde{\D} ^{(m)} _{X / T }} 
\R u _* ( \G)
\\
\notag
\underset{\ref{Prop1.2.21-caro-surc-iso2}}{\riso}
\R u _* \left (u ^{-1} (\FF \otimes ^{\L} _{\B _X} 
\widetilde{\D} ^{(m)} _{X / T }) \otimes ^{\L} _{u ^{-1} \widetilde{\D} ^{(m)} _{X / T }} \G \right)
\riso 
\R u _* \left (u ^{-1}  \FF \otimes ^{\L} _{u ^{-1} \B _X} \G \right)
\end{gather}
\end{proof}

\begin{ntn}
For $\E
\in D ^-
( \overset{^\mathrm{l}}{} \widetilde{\D} ^{(m)} _{X / T })$,
we set 
$\L \widetilde{u}  ^* (\E) 
:= 
\widetilde{\D} ^{(m)} _{Y\to X / T } 
\otimes _{u ^{-1}\widetilde{\D} ^{(m)} _{X / T } }
^{\bbL}
(\E) $.
For $\M \in D ^- 
( \overset{^\mathrm{r}}{} \widetilde{\D} ^{(m)} _{Y / T })$,
we set 
$\widetilde{u} _+ ^{(m)} (\cM )
:= 
 \R u _* \left ( \cM 
\otimes ^{\L} _{\widetilde{\D} ^{(m)} _{Y / T }}
\widetilde{\D} ^{(m)} _{Y  \to X / T }
\right ).$

\end{ntn}

\begin{lemm}
\label{u^*otimes}
For $\E$ and $\FF$ two objects of 
$D ^-
( \overset{^\mathrm{l}}{} \widetilde{\D} ^{(m)} _{X / T })$,
we have the isomorphism of $D ^-
( \overset{^\mathrm{l}}{} \widetilde{\D} ^{(m)} _{Y / T })$
\begin{equation}
\label{u^*otimes-iso}
\L \widetilde{u}  ^* (\E) 
\otimes ^\L _{\B _Y}
\L  \widetilde{u}  ^* (\FF) 
\riso 
\L \widetilde{u}  ^* (\E 
\otimes ^\L _{\B _X}
\FF) .
\end{equation}
\end{lemm}

\begin{proof}
Left to the reader. 
\end{proof}

\begin{prop}
\label{prop-u+otimes-u!}
For $\M \in D ^- 
( \overset{^\mathrm{r}}{} \widetilde{\D} ^{(m)} _{Y / T })$ and
$\E \in D ^- _{\mathrm{qc}}
( \overset{^\mathrm{l}}{} \widetilde{\D} ^{(m)} _{X / T })$,
we have the canonical isomorphism
\begin{equation}
\widetilde{u} _+ ^{(m)}
\left (\M 
\otimes ^{\L} _{\B _Y}
\L \widetilde{u}  ^* (\E) 
\right )
\riso
\widetilde{u} _+ ^{(m)}  (\M )
\otimes ^{\L} _{\B _X}
\E.
\end{equation}

\end{prop}

\begin{proof}
We have the isomorphisms of right 
$u ^{-1}\widetilde{\D} ^{(m)} _{X / T }$-modules:
\begin{gather}
\notag
\left (\M 
\otimes ^{\L} _{\B _Y}
\widetilde{u}  ^* (\E) 
\right )
\otimes ^{\L} _{\widetilde{\D} ^{(m)} _{Y / T }}
\widetilde{\D} ^{(m)} _{Y  \to X / T }
\riso
\left (\M 
\otimes ^{\L} _{\widetilde{\D} ^{(m)} _{Y / T }}
(\widetilde{\D} ^{(m)} _{Y / T } 
\otimes ^{\L} _{\B _Y}
\widetilde{u}  ^* (\E) )
\right )
\otimes ^{\L} _{\widetilde{\D} ^{(m)} _{Y / T }}
\widetilde{\D} ^{(m)} _{Y  \to X / T }
\\
\notag
\underset{\cite[1.3.1]{Be2}}{\riso} 
\left (\M 
\otimes ^{\L} _{\widetilde{\D} ^{(m)} _{Y / T }}
(\widetilde{u}  ^* (\E) 
\otimes ^{\L} _{\B _Y}
\widetilde{\D} ^{(m)} _{Y / T } )
\right )
\otimes ^{\L} _{\widetilde{\D} ^{(m)} _{Y / T }}
\widetilde{\D} ^{(m)} _{Y  \to X / T }
\riso 
\M 
\otimes ^{\L} _{\widetilde{\D} ^{(m)} _{Y / T }}
\left ( \widetilde{u}  ^* (\E) 
\otimes ^{\L} _{\B _Y}
\widetilde{\D} ^{(m)} _{Y  \to X / T }
\right ).
\end{gather}
We have the isomorphism of complexes of
$(\widetilde{\D} ^{(m)} _{Y  / T }, 
u ^{-1}\widetilde{\D} ^{(m)} _{X / T })$-bimodules
\begin{gather}
\notag
\widetilde{u}  ^* (\E) 
\otimes ^{\L} _{\B _Y}
\widetilde{\D} ^{(m)} _{Y  \to X / T }
\underset{\ref{u^*otimes-iso}}{\riso} 
\widetilde{u}  ^* (\E 
\otimes ^{\L} _{\B _X}
\widetilde{\D} ^{(m)} _{X / T })
\underset{\cite[1.3.1]{Be2}}{\liso} 
\widetilde{u}  ^* (
\widetilde{\D} ^{(m)} _{X / T }
\otimes ^{\L} _{\B _X}
\E )
\\
\notag
\riso 
\widetilde{\D} ^{(m)} _{Y  \to X / T }
\otimes ^{\L} _{u ^{-1}\widetilde{\D} ^{(m)} _{X / T }}
u ^{-1} (
\widetilde{\D} ^{(m)} _{X / T }
\otimes ^{\L} _{\B _X}
\E ).
\end{gather}
Hence, 
$\left (\M 
\otimes ^{\L} _{\B _Y}
\widetilde{u}  ^* (\E) 
\right )
\otimes ^{\L} _{\widetilde{\D} ^{(m)} _{Y / T }}
\widetilde{\D} ^{(m)} _{Y  \to X / T }
\riso
\M 
\otimes ^{\L} _{\widetilde{\D} ^{(m)} _{Y / T }}
\widetilde{\D} ^{(m)} _{Y  \to X / T }
\otimes ^{\L} _{u ^{-1}\widetilde{\D} ^{(m)} _{X / T }}
u ^{-1} (
\widetilde{\D} ^{(m)} _{X / T }
\otimes ^{\L} _{\B _X}
\E )$.
By applying the functor $\R u _* $ to this latter isomorphism we get the first one:
\begin{gather}
\notag
\widetilde{u} _+ ^{(m)}
\left (\M 
\otimes ^{\L} _{\B _Y}
\widetilde{u}  ^* (\E) 
\right )
\riso 
\R u _* 
\left (
\left (\M 
\otimes ^{\L} _{\widetilde{\D} ^{(m)} _{Y / T }}
\widetilde{\D} ^{(m)} _{Y  \to X / T }
\right )
\otimes ^{\L} _{u ^{-1}\widetilde{\D} ^{(m)} _{X / T }}
u ^{-1} (
\widetilde{\D} ^{(m)} _{X / T }
\otimes ^{\L} _{\B _X}
\E )\right )
\\
\notag
\underset{\ref{Prop1.2.21-caro-surc-iso2bis}}{\liso} 
\R u _* \left (\M 
\otimes ^{\L} _{\widetilde{\D} ^{(m)} _{Y / T }}
\widetilde{\D} ^{(m)} _{Y  \to X / T }
\right )
\otimes ^{\L} _{\widetilde{\D} ^{(m)} _{X / T }}
(
\widetilde{\D} ^{(m)} _{X / T }
\otimes ^{\L} _{\B _X}
\E )
\riso
\widetilde{u} _+ ^{(m)}  (\M )
\otimes ^{\L} _{\B _X}
\E.
\end{gather}
\end{proof}

Let  $f \colon \fP ^{\prime } \to \fP $ be a morphism of smooth formal schemes over $\S $,
$T$ and $T'$ be some divisors of respectively $P$ and $P'$ such that 
$f ( P '\setminus T' ) \subset P \setminus T$.
We finish this subsection by giving some applications of the projection formula.
\begin{prop}
\label{surcoh2.1.4}
Let $\E ^{ (\bullet)}
\in  
\smash{\underrightarrow{LD}} ^\mathrm{b} _{\Q, \mathrm{qc}}
(\overset{^\mathrm{l}}{} \smash{\widetilde{\D}} _{\fP ^{ }/\S } ^{(\bullet)} (T))$,
and 
$\E ^{\prime (\bullet)}
\in  \smash{\underrightarrow{LD}} ^\mathrm{b} _{\Q, \mathrm{qc}}
(\overset{^\mathrm{l}}{} \smash{\widetilde{\D}} _{\fP ^{\prime }/\S } ^{(\bullet)}(T'))$. 
We have the following isomorphism of $\smash{\underrightarrow{LD}} ^\mathrm{b} _{\Q, \mathrm{qc}}
(\overset{^\mathrm{l}}{} \smash{\widetilde{\D}} _{\fP ^{ }/\S } ^{(\bullet)} (T))$ 
\begin{equation}
\label{surcoh2.1.4-iso}
f _{T,T',+} ^{(\bullet)} ( \E ^{\prime (\bullet)} )
\smash{\widehat{\otimes}}^\L 
_{\widetilde{\B} ^{(\bullet)}  _{\fP} ( T) } 
\E ^{ (\bullet)} [d _{P'/P}]
\riso 
f _{T,T',+} ^{(\bullet)} 
\left ( 
\E ^{\prime (\bullet)} 
\smash{\widehat{\otimes}}^\L 
_{\widetilde{\B} ^{(\bullet)}  _{\fP'} ( T') } 
f _{T',T} ^{!(\bullet)} (\E ^{ (\bullet)} )
\right ) .
\end{equation}

\end{prop}

\begin{proof}
Since the functor $\mathrm{oub} _{T}$ is fully faithful
(see \ref{oub-pl-fid}.\ref{oub-pl-fid-iso3}),
we reduce to check the existence of such isomorphism  
in 
$\smash{\underrightarrow{LD}} ^\mathrm{b} _{\Q, \mathrm{qc}}
(\overset{^\mathrm{l}}{} \smash{\widetilde{\D}} _{\fP ^{ }/\S } ^{(\bullet)})$.
Using \ref{oubTDDotimes} and \ref{oub-div-opcoh}, we get the isomorphism 
$$\mathrm{oub} _{T'} \left ( f _{T,T',+} ^{(\bullet)} ( \E ^{\prime (\bullet)} )
\smash{\widehat{\otimes}}^\L 
_{\widetilde{\B} ^{(\bullet)}  _{\fP} ( T) } 
\E ^{ (\bullet)} 
\right ) 
\riso 
f _{+} ^{(\bullet)} (\mathrm{oub} _{T}(  \E ^{\prime (\bullet)} ))
\smash{\widehat{\otimes}}^\L 
_{\O ^{(\bullet)}  _{\fP} }  
\mathrm{oub} _{T} (\E ^{ (\bullet)}).$$
Using \ref{oubTDDotimes}, \ref{f!commoub} and \ref{oub-div-opcoh}, we get the isomorphism 
$$
\mathrm{oub} _{T'}  \circ 
f _{T,T',+} ^{(\bullet)} 
\left ( \E ^{\prime (\bullet)} 
\smash{\widehat{\otimes}}^\L 
_{\widetilde{\B} ^{(\bullet)}  _{\fP'} ( T') } 
f _{T',T} ^{!(\bullet)} (\E ^{ (\bullet)} )
\right ) 
\riso 
f _{+} ^{(\bullet)} \left (
\mathrm{oub} _{T}(  \E ^{\prime (\bullet)} )
\smash{\widehat{\otimes}}^\L 
_{\O ^{(\bullet)}  _{\fP} }  
f  ^{!(\bullet)} (\mathrm{oub} _{T} (\E ^{ (\bullet)}))
\right ).
$$
Hence, we reduce to check the case where $T$ and $T'$ are empty. 
In that case, this is a consequence of \ref{prop-u+otimes-u!} applied in the case
where $\B _X = \O _X$.
\end{proof}

\begin{cor}
\label{surcoh2.1.4-cor1}
Let $\E ^{ (\bullet)}
\in  \smash{\underrightarrow{LD}} ^\mathrm{b} _{\Q, \mathrm{qc}}
(\overset{^\mathrm{l}}{} \smash{\widetilde{\D}} _{\fP ^{ }/\S } ^{(\bullet)} (T))$.
We have the isomorphism 
\begin{equation}
\label{surcoh2.1.4-isocor1}
f _{T,T',+} ^{(\bullet)} \left (\widetilde{\B} ^{(\bullet)}  _{\fP'} ( T') \right )
\smash{\widehat{\otimes}}^\L 
_{\widetilde{\B} ^{(\bullet)}  _{\fP} ( T) } 
\E ^{ (\bullet)} [d _{P'/P}]
\riso 
f _{T,T',+} ^{(\bullet)} \circ  f _{T',T} ^{!(\bullet)} (\E ^{ (\bullet)} ) .
\end{equation}
\end{cor}

\begin{proof}
We apply \ref{surcoh2.1.4}  
to the case where
$\E ^{\prime (\bullet)}
=
\widetilde{\B} ^{(\bullet)}  _{\fP'} ( T') $.
\end{proof}

\begin{cor}
\label{surcoh2.1.4-cor}
Suppose  $T ' = f ^{-1} (T)$.
Let $\E ^{\prime (\bullet)}
\in  \smash{\underrightarrow{LD}} ^\mathrm{b} _{\Q, \mathrm{qc}}
(\overset{^\mathrm{l}}{} \smash{\widetilde{\D}} _{\fP ^{\prime }/\S } ^{(\bullet)})$.
We have the isomorphism of
$\smash{\underrightarrow{LD}} ^\mathrm{b} _{\Q, \mathrm{qc}}
(\overset{^\mathrm{l}}{} \smash{\widetilde{\D}} _{\fP ^{ }/\S } ^{(\bullet)})$: 
$$f _{T,T'+} ^{(\bullet)}  \circ (\hdag T') ( \E ^{\prime (\bullet)} )
\riso 
(\hdag T ) \circ f _{+} ^{(\bullet)} ( \E ^{\prime (\bullet)} ).$$

\end{cor}

\begin{proof}
Use \ref{surcoh2.1.4}
and \ref{lem-f!B}, we get the isomorphism
\begin{equation}
\label{surcoh2.1.4-cor-iso}
f _{+} ^{(\bullet)} ( \E ^{\prime (\bullet)} )
\smash{\widehat{\otimes}}^\L _{\O ^{(\bullet)}  _{\fP} }  
\widetilde{\B} ^{(\bullet)}  _{\fP} ( T) 
\riso 
f _{+} ^{(\bullet)} ( \E ^{\prime (\bullet)} 
\smash{\widehat{\otimes}}^\L 
_{\O ^{(\bullet)}  _{\fP'} }  
\widetilde{\B} ^{(\bullet)}  _{\fP'} ( T') ).
\end{equation}
We conclude using \ref{oub-div-opcoh}.
\end{proof}

\begin{rem}
Using \ref{oub-div-opcoh}, the isomorphism 
of \ref{surcoh2.1.4-cor} could be written
$f _{+} ^{(\bullet)}  \circ (\hdag T') ( \E ^{\prime (\bullet)} )
\riso 
(\hdag T ) \circ f _{+} ^{(\bullet)} ( \E ^{\prime (\bullet)} ).$

\end{rem}

\subsection{On the stability of the coherence}
Let  $f \colon \fP ^{\prime } \to \fP $ be a morphism of smooth formal schemes over $\S $,
$T$ and $T'$ be some divisors of respectively $P$ and $P'$ such that 
$f ( P '\setminus T' ) \subset P \setminus T$.

\begin{lem}
\label{stab-coh-f^!pre}
Suppose $f _i\colon P _i ^{\prime} \to P  _i$ is   smooth.
For any $\E \in D ^{-} _{\mathrm{coh}} (\widetilde{\D} ^{(m)} _{P _i /S _i }(T))$, 
we have 
$f _{i,T',T}^{(m)!} (\E )\in D ^{-} _{\mathrm{coh}} (\widetilde{\D} ^{(m)} _{P _i ^{\prime}/S _i }(T'))$. 
\end{lem}

\begin{proof}
Since this is local in $P _i ^{\prime}$, using locally free resolution, 
we reduce to the case $\E= \widetilde{\D} ^{(m)} _{P _i /S _i }(T)$. 
Then we compute in local coordinates 
that 
the canonical morphism
$\D ^{(m)} _{P _i ^{\prime}/S _i } \to f _i ^{*} \D ^{(m)} _{P _i /S _i }$ is surjective
whose kernel has the usual description in local 
coordinates. 
\end{proof}

\begin{prop}
\label{stab-coh-f^!}
Suppose $f$ is   smooth.
\begin{enumerate}
\item For  $\E \in D ^{\mathrm{b}} _{\mathrm{coh}} (\widetilde{\D} ^{(m)} _{\fP /\S }(T))$, 
we have 
$f _{T',T}^{(m)!} (\E )\in D ^{\mathrm{b}} _{\mathrm{coh}} (\widetilde{\D} ^{(m)} _{\fP ^{\prime}/\S }(T'))$. 

\item For  $\E \in D ^{\mathrm{b}} _{\mathrm{coh}} (\widetilde{\D} ^{(m)} _{\fP /\S }(T) _\Q)$, 
we have
\begin{equation}
\notag
\widetilde{\D} ^{(m+1)} _{\fP ^{\prime}/\S }(T') _\Q
\otimes ^{\L} _{\widetilde{\D} ^{(m)} _{\fP ^{\prime}/\S }(T')_\Q}
 f _{T',T}^{(m)!} (\E )
\riso 
f _{T',T}^{(m+1)!} (
\widetilde{\D} ^{(m+1)} _{\fP /\S }(T)_\Q
\otimes ^{\L} _{\widetilde{\D} ^{(m)} _{\fP /\S }(T)_\Q}
\E ).
\end{equation}

\item The functor $ f _{T',T}^{!(\bullet)}$
sends 
$\underrightarrow{LD} ^{\mathrm{b}}  _{\Q, \mathrm{coh}} (\smash{\widetilde{\D}} _{\fP /\S } ^{(\bullet)} (T))$
to 
$\underrightarrow{LD} ^{\mathrm{b}}  _{\Q, \mathrm{coh}} (\smash{\widetilde{\D}} _{\fP ^{\prime }/\S } ^{(\bullet)} (T'))$.
\item 
For $\E \in D ^{\mathrm{b}} _{\mathrm{coh}} ( \D ^{\dag} _{\fP  } (\hdag T ) _{\Q})$,
we have
$f  ^{ !} _{T' , T} (\E ) \in
D ^{\mathrm{b}} _{\mathrm{coh}} ( \D ^{\dag} _{\fP ^{\prime }} (\hdag T' ) _{\Q})$.
\end{enumerate}

\end{prop}

\begin{proof}
The first part is a consequence of \ref{stab-coh-f^!pre}. 
We check the second part similarly to  \cite[3.4.6]{Beintro2}.
The third and forth parts are a consequence of the previous ones. 
\end{proof}

\begin{lem}
\label{stab-coh-f_+}
Suppose $f$ is proper, and 
$T ' = f ^{-1}(T)$.
\begin{enumerate}
\item The functor $f ^{(m)} _{i,T+}$ sends 
$D ^{-} _{\mathrm{coh}} (\widetilde{\D} ^{(m)} _{P _i ^{\prime}/S _i }(T'))$
to 
$D ^{-} _{\mathrm{coh}} (\widetilde{\D} ^{(m)} _{P _i /S _i }(T))$. 
\item For $\E ' \in D ^{-} _{\mathrm{coh}} (\widetilde{\D} ^{(m)} _{P _i ^{\prime}/S _i }(T'))$,
we have the canonical isomorphism
\begin{equation}
\notag
\widetilde{\D} ^{(m+1)} _{P _i /S _i }(T')
\otimes ^{\L} _{\widetilde{\D} ^{(m)} _{P _i /S _i }(T')}
f ^{(m)} _{i,T+}
(\E ' )
\riso 
f ^{(m+1)} _{i,T+}
\left (\widetilde{\D} ^{(m+1)} _{P _i ^{\prime}/S _i }(T')
\otimes ^{\L} _{\widetilde{\D} ^{(m)} _{P _i ^{\prime}/S _i }(T')}
\E ' \right ).
\end{equation}

\end{enumerate}
\end{lem}

\begin{proof}
The proof is identical to Berthelot's  one: 
since $f _{i+}$ is ``way out right'' (see \cite[I.7]{HaRD}),
in order to check the coherence of 
$f ^{(m)} _{i,T+} (\E ')$ for  $\E ' \in D ^{-} _{\mathrm{coh}} (\widetilde{\D} ^{(m)} _{P _i ^{\prime}/S _i }(T'))$
 we can suppose that $\E '$ is a module. 
Since $\E '$ is the inductive limit of its coherent $\O _{P _i ^{\prime}}$-submodules,
there exists a coherent $\O _{P _i ^{\prime}}$-submodule $\FF '$ of $\E '$ 
such that the canonical map 
$\widetilde{\D} ^{(m)} _{P _i ^{\prime}/S _i } (T')Â \otimes _{\O _{P _i ^{\prime}}} \FF ' \twoheadrightarrow \E '$
is surjective. 
Hence, we get a resolution of $\E '$ by $\widetilde{\D} ^{(m)} _{P _i ^{\prime}/S _i } (T')$-modules of the form
$\widetilde{\D} ^{(m)} _{P _i ^{\prime}/S _i } (T')Â \otimes _{\O _{P _i ^{\prime}}} \FF '$. 
This yields that we reduce to the case where 
$\widetilde{\D} ^{(m)} _{P _i ^{\prime}/S _i } (T')Â \otimes _{\O _{P _i ^{\prime}}} \FF '=\E '$.
Using the projection isomorphism, we get
\begin{align}
\notag
f _{i+} ( \widetilde{\D} ^{(m)} _{P _i ^{\prime}/S _i } (T')Â \otimes _{\O _{P _i ^{\prime}}} \FF ')&=
\R f _* \left (  \left (  f ^{*} _{i,\mathrm{r}} \left (\widetilde{\D} ^{(m)} _{P _i  /S _i }(T) \otimes _{\O _{P _i}}  Â  \omega _{P _i /S _i } ^{-1}\right )    \otimes _{\O _{P _i ^{\prime}}}  Â  \omega _{P _i ^{\prime}/S _i }  \right )
\otimes ^\L _{\widetilde{\D} ^{(m)} _{P _i ^{\prime}/S _i } (T')}
(\widetilde{\D} ^{(m)} _{P _i ^{\prime}/S _i } (T')Â \otimes _{\O _{P _i ^{\prime}}} \FF ' )\right )\\
\label{stab-coh-f_+-proof1}
& 
\riso 
\widetilde{\D} ^{(m)} _{P _i  /S _i } (T) \otimes _{\O _{P _i}} \left (Â  \omega _{P _i /S _i } ^{-1} 
\otimes _{\O _{P _i}} ^\L Â 
\R f _* \left (  \omega ^{-1}_{P _i ^{\prime}/S _i }  \otimes _{\O _{P _i ^{\prime}}} \FF ' \right )\right) .
\end{align}
Since $f$ is proper and $S _i$ is locally noetherian,
then the functor $\R f _*$ preserves the $\O$-coherence, hence $ \omega _{P _i /S _i } ^{-1} 
\otimes _{\O _{P _i}} ^\L Â  \R f _*  (  \omega ^{-1}_{P _i ^{\prime}/S _i }  \otimes _{\O _{P _i ^{\prime}}} \FF ' ) \in 
D ^{-} _{\mathrm{coh}} (\O _{P _i})$ and we conclude.

Let us check the second statement. 
We construct the morphism by using the canonical semi-linear morphisms 
$\smash{\widetilde{\D}} ^{(m)} _{P _i  \leftarrow P _i ^{\prime }/\S } (T,T')
\to 
\smash{\widetilde{\D}} ^{(m+1)} _{P _i  \leftarrow P _i ^{\prime }/\S } (T,T')$
and 
$\E ' \to 
\widetilde{\D} ^{(m+1)} _{P _i ^{\prime}/S _i }(T')
\otimes ^{\L} _{\widetilde{\D} ^{(m)} _{P _i ^{\prime}/S _i }(T')}
\E ' $. To check that this is an isomorphism,
we reduce to the case where
$\widetilde{\D} ^{(m)} _{P _i ^{\prime}/S _i } (T')Â \otimes _{\O _{P _i ^{\prime}}} \FF '=\E '$. 
We conclude via the isomorphism
\ref{stab-coh-f_+-proof1}.
\end{proof}

\begin{prop}
Suppose $f$ is proper, and 
$T ' = f ^{-1}(T)$.
\begin{enumerate}
\item For  $\E '\in D ^{\mathrm{b}} _{\mathrm{coh}} (\widetilde{\D} ^{(m)} _{\fP ^{\prime}/\S }(T'))$, 
we have 
$f _{T,+}^{(m)} (\E' )\in D ^{\mathrm{b}} _{\mathrm{coh}} (\widetilde{\D} ^{(m)} _{\fP /\S }(T))$. 

\item For  $\E' \in D ^{\mathrm{b}} _{\mathrm{coh}} (\widetilde{\D} ^{(m)} _{\fP ^{\prime}/\S }(T'))$, 
we have
\begin{equation}
\notag
\widetilde{\D} ^{(m+1)} _{\fP /\S }(T)
\otimes ^{\L} _{\widetilde{\D} ^{(m)} _{\fP /\S }(T)}
  f _{T,+}^{(m)} (\E )
\riso 
  f _{T,+}^{(m+1)} (
\widetilde{\D} ^{(m+1)} _{\fP ^{\prime}/\S }(T')
\otimes ^{\L} _{\widetilde{\D} ^{(m)} _{\fP ^{\prime}/\S }(T')}
\E ).
\end{equation}

\item The functor $ f _{T,+} ^{(\bullet)}$
sends 
$\underrightarrow{LD} ^{\mathrm{b}}  _{\Q, \mathrm{coh}} (\smash{\widetilde{\D}} _{\fP ^{\prime }/\S } ^{(\bullet)} (T'))$
to 
$\underrightarrow{LD} ^{\mathrm{b}}  _{\Q, \mathrm{coh}} (\smash{\widetilde{\D}} _{\fP /\S } ^{(\bullet)} (T))$.

\item For $\E '\in D ^{\mathrm{b}} _{\mathrm{coh}} ( \D ^{\dag} _{\fP ^{\prime }} (\hdag T' ) _{\Q})$, 
we have
$f  _{T, +}( \E ')\in D ^{\mathrm{b}} _{\mathrm{coh}} ( \D ^{\dag} _{\fP  } (\hdag T ) _{\Q})$.
\end{enumerate}

\end{prop}

\begin{proof}
This is a consequence of \ref{stab-coh-f_+}.
\end{proof}

\section{Relative duality isomorphism and adjunction}
\label{chapt4-adj-reliso-closedimm}
Following Virrion (see \cite{Vir04}), we have
relative duality isomorphisms and adjoint paires $(f _+, f ^!)$  for proper morphisms $f$.
In this section, we first retrieve theses theorems 
in the case of a closed immersion of smooth schemes.
They are at least four reasons for this specific study:
because this case is much easier than Virrion's one, 
because we can remove the  coherence hypotheses to build the adjunction morphisms,
because these adjunction morphisms have a clear description and will be used in chapter \ref{section-BK-var} (more specifically, 
in Proposition \ref{comp-comp-adj-immf})
and because we will see in another chapter  (see \ref{rel-dual-isom-proj})
that this case of closed immersions  implies 
the case of projective morphisms
(which is enough for instance to check the coherence of the constant coefficient in \ref{coh-ss-div}).
Finally, in the last subsection, we complete Virrion's relative duality isomorphisms by adding some 
overconvergent singularities.

\subsection{The fundamental local isomorphism}
Let $T$ be a noetherian $\Z _{(p)}$-scheme of finite Krull dimension.
Let $u \colon Z \hookrightarrow X$ be a closed immersion of smooth $T$-schemes.
Let $\I$ be the ideal defining $u$. 
The level $m\in\N$ is fixed.

\begin{empt}
[Some notation with local coordinates]
\label{locdesc-climm}
Suppose
$X$ is affine and there exist
$t _{r +1},\dots , t _{d}  \in \Gamma (X,\I)$
generating 
$I:=\Gamma (X,\I)$,
$t _{1},\dots , t _{r}\in \Gamma ( X,\O _{X})$
such that
$t _{1},\dots ,t  _{d}$ are local coordinates of $X$ over $T$,
$\overline{t} _{1},\dots ,\overline{t} _{r}$ 
are local coordinates of $Z$ over $T$,
and 
$\overline{t} _{r +1},\dots ,\overline{t} _{d}$ is a basis of $\I /\I ^2$,
where $\overline{t} _{1},\dots , \overline{t} _{r} \in\Gamma ( Z ,\O _{Z})$
(resp. $\overline{t} _{r +1},\dots ,\overline{t} _{d}\in\Gamma ( X ,\cI /\cI ^2)$)
are the images of 
$t _{1},\dots , t _{r}$
(resp. $t _{r+1},\dots , t _{d}$)
via 
$\Gamma ( X,\O _{X})
\to
\Gamma ( Z ,\O _{Z})$
(resp. 
$\Gamma ( X,\cI)
\to 
\Gamma ( X ,\cI /\cI ^2)$).

We denote by 
$\tau _i := 1 \otimes t _i -t _i \otimes 1$, 
$\overline{\tau} _j := 1 \otimes \overline{t} _j -\overline{t} _j \otimes 1$, 
for any $i= 1,\dots, d$, $j= 1,\dots, r$.
The sheaf of $\O _X$-algebras
$\cP ^n  _{X/T,(m)}$ is 
a free $\O _X$-module with the basis 
$\{ \underline{\tau} ^{\{\underline{k}\} _{(m)}}\ | \ \underline{k}\in \N ^d \text{ such that } | \underline{k}| \leq n\} $,
and
$\cP ^n  _{Z/T,(m)}$ is 
a free $\O _Z$-module with the basis 
$\{ \underline{\overline{\tau}} ^{\{\underline{l}\} _{(m)}}\ | \ \underline{l}\in \N ^r \text{ such that } | \underline{l}| \leq n\} $.
We denote by 
$\{ 
\underline{\partial} ^{<\underline{k}> _{(m)}}
\ | \ \underline{k}\in \N ^d,\ | \underline{k} | \leq n
\}$
the corresponding dual basis  of 
$\D ^{(m)} _{X/T,n}$ 
and
by 
$\{ \underline{\partial} ^{<\underline{l}> _{(m)}}\  
| \ \underline{l}\in \N ^r,\ | \underline{l} | \leq n\} $
the corresponding dual basis of $\D ^{(m)} _{Z/T,n}$ (we hope the similar notation is not too confusing).
The sheaf
$\D ^{(m)} _{X/T}$ is 
a free $\O _X$-module with the basis 
$\{ 
\underline{\partial} ^{<\underline{k}> _{(m)}}
\ | \ \underline{k}\in \N ^d
\}$,
and
$\D ^{(m)} _{Z/T}$ is 
a free $\O _Z$-module with the basis 
$\{ \underline{\partial} ^{<\underline{l}> _{(m)}}\ | \ \underline{l}\in \N ^r\} $.

a) We compute the canonical homomorphism
$u ^* \cP ^n  _{X/T,(m)}
\to \cP ^n  _{Z/T,(m)}$
sends 
$\underline{\tau} ^{\{(\underline{l}, \underline{h})\} _{(m)}}$
where 
$\underline{l} \in \N ^r$
and 
$\underline{h} \in \N ^{d-r}$
to
$\underline{\overline{\tau}} ^{\{\underline{l}\} _{(m)}}$
if $\underline{h} = (0,\dots,0)$
and to $0$ otherwise.

b) We denote by 
$\theta \colon 
\cD ^{(m)} _{Z/T}
\to 
\cD ^{(m)} _{Z \to X/T}$
the canonical homomorphism of left $\cD ^{(m)} _{Z/T}$-modules
(which is built by duality from the 
canonical homomorphisms $u ^* \cP ^n  _{X/T,(m)}
\to \cP ^n  _{Z/T,(m)}$).
For any $P \in D ^{(m)} _{X/T}$, we denote by 
$\overline{P}$ its image via
the canonical morphism of left $D ^{(m)} _{X/T}$-modules
$D ^{(m)} _{X/T}
\to 
D ^{(m)} _{X/T} / I D ^{(m)} _{X/T}
= 
D ^{(m)} _{Z \to X/T}$.
We set 
$\underline{\xi} ^{<\underline{k}> _{(m)}}:= 
\overline{\underline{\partial} ^{<\underline{k}> _{(m)}}}$. 
By duality from a), we compute
$\theta (\underline{\partial} ^{<\underline{l}> _{(m)}})
=
\underline{\xi} ^{<(\underline{l}, \underline{0})> _{(m)}}$,
for any 
$\underline{l}\in \N ^r$.

c) Let $Q \in D ^{(m)} _{Z}$, 
$Q _X\in D ^{(m)} _{X/T}$ 
such that  $\overline{Q _X}= \theta (Q)$.
Let $x \in \Gamma (X, \O _X)$ and
$\overline{x} \in \Gamma (Z, \O _Z)$ its image.
By definition of the action of 
$\cD ^{(m)} _{X/T}$ on $\O _X$ 
and 
of $\cD ^{(m)} _{Z/T}$ on $\cO _Z$ (see \cite[2.2.1.4]{Be1}), 
we check the formula
\begin{equation}
\label{dfn-u!leftclimm-pre2cst}
Q \cdot \overline{x}
=
\overline{Q _X \cdot x}.
\end{equation}

\end{empt}

\begin{empt}
\label{locdesc-climm2}
Suppose we are in the local situation of \ref{locdesc-climm}. 
We denote by 
$\cD ^{(m)} _{X,Z,\underline{t}/T}$ the subring of 
$\cD ^{(m)} _{X/T}$ which is a 
 free $\O _X$-module with the basis 
$\{ \underline{\partial} ^{<(\underline{l}, \underline{0})> _{(m)}}\ | \ \underline{l}\in \N ^r\} $, 
where 
$\underline{0}:=(0,\dots, 0) \in 
\N ^{d-r}$.
If there is no ambiguity concerning the local coordinates (resp. and $T$),
we might simply denote $\cD ^{(m)} _{X,Z,\underline{t}/T}$ by 
$\cD ^{(m)} _{X,Z/T}$
(resp. $\cD ^{(m)} _{X,Z}$).
\begin{enumerate}[(a)]
\item Since 
$t _{r +1},\dots , t _{d}$ generate $I$ and are in the center of 
$D ^{(m)} _{X,Z,\underline{t}/T}$, then we compute 
$\cI\cD ^{(m)} _{X,Z,\underline{t}/T} = \cD ^{(m)} _{X,Z,\underline{t}/T}\cI$.
Hence, we get a canonical ring structure on 
$\cD ^{(m)} _{X,Z,\underline{t}/T}/\cI\cD ^{(m)} _{X,Z,\underline{t}/T}$
induced by that of 
$\cD ^{(m)} _{X,Z,\underline{t}/T}$.
Since $\cI\cD ^{(m)} _{X,Z,\underline{t}/T} 
= 
\cD ^{(m)} _{X,Z,\underline{t}/T} 
\cap 
\cI\cD ^{(m)} _{X/T} $, we get the inclusion
$\cD ^{(m)} _{X,Z,\underline{t}/T}/\cI\cD ^{(m)} _{X,Z,\underline{t}/T} 
\hookrightarrow
\cD ^{(m)} _{X/T} / \cI \cD ^{(m)} _{X/T}$.
Moreover, since $\theta (\underline{\partial} ^{<\underline{l}> _{(m)}})
=
\underline{\xi} ^{<(\underline{l}, \underline{0})> _{(m)}}$
for any 
$\underline{l}\in \N ^r$, 
then we have the following factorization
\begin{equation}
\label{diag-DZ2X-t-comm}
\xymatrix{
{ \cD ^{(m)} _{X,Z,\underline{t}/T}/\cI\cD ^{(m)} _{X,Z,\underline{t}/T}} 
\ar@{^{(}->}[r] ^-{}
& 
{\cD ^{(m)} _{X/T} / \cI \cD ^{(m)} _{X/T}} 
\\ 
{u _* \cD ^{(m)} _{Z/T}} 
\ar[u] ^-{\sim} _-{\theta}
\ar[r] ^-{\theta}
& 
{u _* \cD ^{(m)} _{Z \to X/T},} 
\ar@{=}[u] ^-{}
}
\end{equation}
where both horizontal morphisms are the canonical ones
and where 
the vertical arrow
$u _* \cD ^{(m)} _{Z/T}
\to
 \cD ^{(m)} _{X,Z,\underline{t}/T}/\cI\cD ^{(m)} _{X,Z,\underline{t}/T}$
is in fact a bijection that we still denote by $\theta$.

\item 
Let $P \in D ^{(m)} _{X/T}$, $Q \in D ^{(m)} _{Z}$. 
Choose $Q _X\in D ^{(m)} _{X,Z,\underline{t}/T}$ 
such that  $\overline{Q _X}= \theta (Q)$.
Since $\theta$ is $D ^{(m)} _{Z}$-linear, we get
$Q \cdot \overline{1} = \theta (Q \cdot 1) = \overline{Q _X}$, where
$Q \cdot \overline{1} $ is the action of $Q$ on $\overline{1} \in D ^{(m)} _{Z \to X/T}$
given by its structure of  left $D ^{(m)} _{Z/T}$-module. 
Moreover, since $D ^{(m)} _{Z \to X/T}$ is a $(D ^{(m)} _{Z/T},D ^{(m)} _{X/T})$-bimodule,
then 
we get 
\begin{equation}
\label{eq-QXP}
Q \cdot \overline{P} 
= 
Q \cdot \overline{1} \cdot P
=
\overline{Q _X} \cdot P
 =
\overline{Q _X P}.
\end{equation}

\item \label{(c)}
Since 
$u _* \cD ^{(m)} _{Z \to X/T}= \cD ^{(m)} _{X/T} / \cI \cD ^{(m)} _{X/T}$, 
then $\cD ^{(m)} _{X/T} / \cI \cD ^{(m)} _{X/T} $ is a 
$(u _*\cD ^{(m)} _{Z/T} , \D ^{(m)} _{X/T})$-bimodule.
Using the formula \ref{eq-QXP}, we compute that 
$\cD ^{(m)} _{X,Z,\underline{t}/T}/\cI\cD ^{(m)} _{X,Z,\underline{t}/T} $
is also a left $u _*\cD ^{(m)} _{Z/T} $-submodule  of 
$\cD ^{(m)} _{X/T} / \cI \cD ^{(m)} _{X/T}$, and then 
a $(u _*\cD ^{(m)} _{Z/T} , \D ^{(m)} _{X,Z,\underline{t}/T})$-subbimodule
of $\cD ^{(m)} _{X/T} / \cI \cD ^{(m)} _{X/T}$.
Since $u _* \cD ^{(m)} _{Z/T}
\to 
u _* \cD ^{(m)} _{Z \to X/T}$
is an homomorphism of left $u _*\cD ^{(m)} _{Z/T}$-modules, then
via the commutativity of \ref{diag-DZ2X-t-comm}, 
this implies that the bijection
$\theta \colon u _* \cD ^{(m)} _{Z/T}
\riso
 \cD ^{(m)} _{X,Z,\underline{t}/T}/\cI\cD ^{(m)} _{X,Z,\underline{t}/T}$
is an isomorphism of 
left $u _* \cD ^{(m)} _{Z/T}$-modules.

\item \label{(e)}The canonical bijection
$\theta \colon u _* \cD ^{(m)} _{Z/T}
\riso
 \cD ^{(m)} _{X,Z,\underline{t}/T}/\cI\cD ^{(m)} _{X,Z,\underline{t}/T}$
is  an isomorphism of rings.
Indeed, by $\O _Z$-linearity of this map, 
we reduce to check that 
$\theta 
(
\underline{\partial} ^{<\underline{k}> _{(m)}}
\overline{a}
\underline{\partial} ^{<\underline{l}> _{(m)}} 
) 
=
\theta (\underline{\partial} ^{<\underline{k}> _{(m)}})
\theta ( \overline{a})
\theta (\underline{\partial} ^{<\underline{l}> _{(m)}} )$,
for any $a \in \Gamma (X, \O _X)$, 
$\underline{k}, \underline{l}\in \N ^r$.
Using \cite[2.2.4.(iii)--(iv)]{Be1}, by $\O _Z$-linearity we get the first equality
\begin{gather}
\notag
\theta 
(
\underline{\partial} ^{<\underline{k}> _{(m)}}
\overline{a}
\underline{\partial} ^{<\underline{l}> _{(m)}} 
) =
\sum _{\underline{k}''+\underline{k}'=\underline{k}}
\left\{
\begin{smallmatrix}
\underline{k}   \\
  \underline{k} '
\end{smallmatrix}
\right \} _{(m)}
\left<
\begin{smallmatrix}
\underline{k} ''  +   \underline{l}\\
  \underline{l}
\end{smallmatrix}
\right > _{(m)}
\underline{\partial} ^{<\underline{k}'> _{(m)}}  (\overline{a})
\underline{\xi} ^{<(\underline{k}''+  \underline{l}, \underline{0})> _{(m)}} 
\\
\notag
\overset{\ref{dfn-u!leftclimm-pre2cst}}{=}
\sum _{\underline{k}''+\underline{k}'=\underline{k}}
\left\{
\begin{smallmatrix}
\underline{k}   \\
  \underline{k} '
\end{smallmatrix}
\right \} _{(m)}
\left<
\begin{smallmatrix}
\underline{k} ''  +   \underline{l}\\
  \underline{l}
\end{smallmatrix}
\right > _{(m)}
\overline{ \underline{\partial} ^{<(\underline{k}', \underline{0})> _{(m)}}  (a)
\underline{\partial} ^{<(\underline{k}''+  \underline{l}, \underline{0})> _{(m)}} )}
\\
=
\overline{ \underline{\partial} ^{<(\underline{k}, \underline{0})> _{(m)}} 
a
\underline{\partial} ^{<(\underline{l}, \underline{0})> _{(m)}} )}
=
\overline{ \underline{\partial} ^{<(\underline{k}, \underline{0})> _{(m)}}} 
\overline{a}
\overline{\underline{\partial} ^{<(\underline{l}, \underline{0})> _{(m)}} )}
=
\theta (\underline{\partial} ^{<\underline{k}> _{(m)}})
\theta ( \overline{a})
\theta (\underline{\partial} ^{<\underline{l}> _{(m)}} ).
\end{gather}

\item \label{f}
Let $\E$ be an $\O _X$-module (resp. an $(\cA, \O _X)$-bimodule for some $\O _T$-algebra $\cA$) such that $\I \E = 0$. 
Then $\E$ has a structure of left $\cD ^{(m)} _{X,Z,\underline{t}/T}$-module
(resp. right $\cD ^{(m)} _{X,Z,\underline{t}/T}$-module, 
resp. $(\cD ^{(m)} _{X,Z,\underline{t}/T},\cA)$-bimodule)
extending its structure of $\O _X$-module 
if and only if 
 $\E$ has a structure of left
$\cD ^{(m)} _{X,Z,\underline{t}/T}/ \I \cD ^{(m)} _{X,Z,\underline{t}/T}$-module
(resp. right $\cD ^{(m)} _{X,Z,\underline{t}/T}/ \I \cD ^{(m)} _{X,Z,\underline{t}/T}$-module, 
resp. $(\cD ^{(m)} _{X,Z,\underline{t}/T}/ \I \cD ^{(m)} _{X,Z,\underline{t}/T},\cA)$-bimodule)
extending its structure of $\O _X$-module.

\item \label{g}
Both rings 
$u _* \cD ^{(m)} _{Z/T}$ and 
$\cD ^{(m)} _{X,Z,\underline{t}/T}$ are 
$\O _X$-rings (i.e. they are rings endowed with a structural 
homomorphism of rings 
$\O _X \to \cD ^{(m)} _{X,Z,\underline{t}/T}$
and 
$\O _X \to u _* \cD ^{(m)} _{Z/T}$).
This yields the composite 
\begin{equation}
\label{morp-rho}
\rho\colon \cD ^{(m)} _{X,Z,\underline{t}/T}
\to
\cD ^{(m)} _{X,Z,\underline{t}/T}
/\I \cD ^{(m)} _{X,Z,\underline{t}/T}
\underset{\theta}{\liso}  
u _* \cD ^{(m)} _{Z/T}
\end{equation}
is an homomorphism of $\O _X$-rings.
If $\cM$ is a left (resp. right) $u _* \cD ^{(m)} _{Z/T}$-module
or a (resp. left, resp. right) $u _* \cD ^{(m)} _{Z/T}$-bimodule, 
then we denote by 
$\rho _* (\cM)$ (or simply by $\cM$)
the left (resp. right) $\cD ^{(m)} _{X,Z,\underline{t}/T}$-module
or the (resp. left, resp. right) $\cD ^{(m)} _{X,Z,\underline{t}/T}$-bimodule
structure on $\cM$ induced by $\rho$. 

Conversely  if  $\cM$ is a left (resp. right) $\cD ^{(m)} _{X,Z,\underline{t}/T}
/\I \cD ^{(m)} _{X,Z,\underline{t}/T}
$-module
or a (resp. left, resp. right) $\cD ^{(m)} _{X,Z,\underline{t}/T}
/\I \cD ^{(m)} _{X,Z,\underline{t}/T}$-bimodule, 
then we denote by 
$\theta _* (\cM)$ (or simply by $\cM$)
the left (resp. right) $u _* \cD ^{(m)} _{Z/T}$-module
or the (resp. left, resp. right) $u _* \cD ^{(m)} _{Z/T}$-bimodule
structure on $\cM$ induced by $\theta$. 
By definition, 
$\rho _* \theta _* = id$
and
$\theta _*  \rho _* = id$.

\item With notation (\ref{g}), 
we get 
the isomorphism
$\theta \colon \rho _* u _* \cD ^{(m)} _{Z/T}
\riso
 \cD ^{(m)} _{X,Z,\underline{t}/T}/\cI\cD ^{(m)} _{X,Z,\underline{t}/T}$
 of 
$\cD ^{(m)} _{X,Z,\underline{t}/T}$-bimodules.
This yields the isomorphism
$\theta \colon u _* \cD ^{(m)} _{Z/T}
\riso
\theta _* \left ( 
\cD ^{(m)} _{X,Z,\underline{t}/T}/\cI\cD ^{(m)} _{X,Z,\underline{t}/T}
\right )$
of $u _* \cD ^{(m)} _{Z/T}$-bimodules.

\item \label{locadesc-uflat2-iso1}
Via the remark (\ref{f}), 
we can also view $\rho_ * u _* \cD ^{(m)} _{Z/T}$ as a
$( \cD ^{(m)} _{X,Z,\underline{t}/T}/\I \cD ^{(m)} _{X,Z,\underline{t}/T} , \cD ^{(m)} _{X,Z,\underline{t}/T})$-bimodule. 
Hence, we get a structure of
$( u _* \cD ^{(m)} _{Z/T}, \cD ^{(m)} _{X,Z,\underline{t}/T})$-bimodule
on $u _* \cD ^{(m)} _{Z/T}$.
On the other hand, following (\ref{(c)}),
$ \cD ^{(m)} _{X,Z,\underline{t}/T}/\cI\cD ^{(m)} _{X,Z,\underline{t}/T}$ is a 
$( u _* \cD ^{(m)} _{Z/T}, \cD ^{(m)} _{X,Z,\underline{t}/T})$-subbimodule
of $ \cD ^{(m)} _{X/T}/\cI\cD ^{(m)} _{X/T}$.
Since $\theta$ is an isomorphism of left $u _* \cD ^{(m)} _{Z/T}$-modules,
then 
the 
isomorphism
$\theta \colon u _* \cD ^{(m)} _{Z/T}
\riso
 \cD ^{(m)} _{X,Z,\underline{t}/T}/\cI\cD ^{(m)} _{X,Z,\underline{t}/T}$
is also an isomorphism of 
$( u _* \cD ^{(m)} _{Z/T}, \cD ^{(m)} _{X,Z,\underline{t}/T})$-bimodules
such that the structure of left $u _* \cD ^{(m)} _{Z/T}$-module on 
$u _* \D ^{(m)} _{Z}$ is the canonical one and 
the structure of right $\cD ^{(m)} _{X,Z,\underline{t}/T}$-module on 
$\D ^{(m)} _{X,Z,\underline{t}/T}/\cI \D ^{(m)} _{X,Z,\underline{t}/T}$ is the canonical one.

\item Since 
$\cD ^{(m)} _{X/T} $
is a free left $\cD ^{(m)} _{X,Z,\underline{t}/T}$-module
with the basis
$\{ \underline{\partial} ^{<(\underline{0},\underline{h})> _{(m)}}\ | \ \underline{h}\in \N ^{d-r}\} $,
where 
$\underline{0}:=(0,\dots, 0) \in 
\N ^{r}$,
then from the commutativity of \ref{diag-DZ2X-t-comm},
we get that 
$\cD ^{(m)} _{Z \to X/T}$ is a free 
left 
$\cD ^{(m)} _{Z/T}$-module 
with the basis
$\{ \underline{\xi} ^{<(\underline{0},\underline{h})> _{(m)}}\ | \ \underline{h}\in \N ^{d-r}\} $,
where 
$\underline{0}:=(0,\dots, 0) \in 
\N ^{r}$.

\item We have the transposition automorphism
${} ^t \colon 
D ^{(m)} _{X/T}
\to 
D ^{(m)} _{X/T}$
given by 
$P 
= 
\sum _{\underline{k} \in \N ^d}
a _{\underline{k}} \underline{\partial} ^{<\underline{k}> _{(m)}}
\mapsto
{} ^t P:= 
\sum _{\underline{k} \in \N ^d}
(-1) ^{| \underline{k}|}\underline{\partial} ^{<\underline{k}> _{(m)}}
a _{\underline{k}} $.
Beware that this transposition depends on the choice of the local coordinates $t _1,\dots, t _d$.
This transposition automorphism induces 
${} ^t \colon 
D ^{(m)} _{X,Z,\underline{t}/T}
\to 
D ^{(m)} _{X,Z,\underline{t}/T}$
such that
${} ^t ( ID ^{(m)} _{X,Z,\underline{t}/T})
=ID ^{(m)} _{X,Z,\underline{t}/T}$.
This yields the automorphism
${} ^t \colon 
D ^{(m)} _{X,Z,\underline{t}/T}/I D ^{(m)} _{X,Z,\underline{t}/T}
\to 
D ^{(m)} _{X,Z,\underline{t}/T}
/
ID ^{(m)} _{X,Z,\underline{t}/T}$.
On the other hand, 
via the local coordinates $\overline{t} _{1},\dots ,\overline{t} _{r}$ 
of $Z$ over $T$,
we get
the transposition automorphism
${} ^t \colon 
D ^{(m)} _{Z/T}
\to 
D ^{(m)} _{Z/T}$
given by 
$Q 
= 
\sum _{\underline{k} \in \N ^r}
b _{\underline{k}} \underline{\partial} ^{<\underline{k}> _{(m)}}
\mapsto
{} ^t Q:= 
\sum _{\underline{k} \in \N ^r}
(-1) ^{| \underline{k}|}\underline{\partial} ^{<\underline{k}> _{(m)}}
b _{\underline{k}} $.
We compute the following diagram 
\begin{equation}
\label{transp-closedimmersion}
\xymatrix{
{D ^{(m)} _{X,Z,\underline{t}/T}/I D ^{(m)} _{X,Z,\underline{t}/T}} 
\ar[r] _-{\sim} ^-{{}^t}
& 
{D ^{(m)} _{X,Z,\underline{t}/T}/I D ^{(m)} _{X,Z,\underline{t}/T}} 
\\ 
{D ^{(m)} _{Z/T}} 
\ar[u] ^-{\sim} _-{\theta}
\ar[r] _-{\sim} ^-{{}^t}
& 
{D ^{(m)} _{Z/T}} 
\ar[u] ^-{\sim} _-{\theta}
}
\end{equation}
is commutative.
\end{enumerate}

\end{empt}

\begin{empt}
\label{rightD-moduleflat}
We have the commutative diagram 
\begin{equation}
\label{comdiag-finite}
\xymatrix{
{\Delta ^n  _{X/T,(m)}(2)} 
\ar@<2ex>[r] ^-{p _{12}}
\ar[r] ^-{p _{02}}
\ar@<-2ex>[r] ^-{p _{01}}
& 
{\Delta ^n  _{X/T,(m)}} 
\ar@<1ex>[r] ^-{p _1}
\ar@<-1ex>[r] _-{p _0}
& 
{X} 
\\ 
{\Delta ^n  _{X/T,(m)}(2)} 
\ar@<2ex>[r] ^-{p _{12}}
\ar[r] ^-{p _{02}}
\ar@<-2ex>[r] ^-{p _{01}}
\ar@{^{(}->}[u] ^-{\Delta ^n (u)(2)}
& 
 {\Delta ^n  _{Z/T,(m)}} 
\ar@<1ex>[r] ^-{p _1}
\ar@<-1ex>[r] _-{p _0}
\ar@{^{(}->}[u] ^-{\Delta ^n (u)}
& 
{Z}
\ar@{^{(}->}[u] ^-{u}
}
\end{equation}
where $p _0, p_1$ correspond to the homomorphisms
$\O _X \rightarrow \cP ^n _{X (m)} $ given by respectively the left and right structures of $\O _X$-algebra
on $\cP ^n _{X (m)} $.
Using the local description 
of the homomorphism $u ^* \cP ^n  _{X/T,(m)} \to \cP ^n  _{Z/T,(m)}$
given in \ref{locdesc-climm},
we check that 
$u ^* \cP ^n  _{X/T,(m)} \to \cP ^n  _{Z/T,(m)}$
is an epimorphism i.e. vertical morphisms of \ref{comdiag-finite} are indeed closed immersions.

We denote $\overline{u}\colon (Z, \O _Z) \to (X, u _* \O _Z)$ the morphism of ringed spaces induced by $u$.
We remark that $\overline{u}$ is flat and 
that 
$\overline{u} ^* = u ^{-1} \colon 
D ^{+}  ( u _* \O _Z )
\to 
D ^{+}  (  \O _Z )$.
Recall that for any $\M \in D ^{+}  ( \O _X )$, by definition
$u ^\flat ( \M) := u ^{-1} \R \mathcal{H}om _{\O _X} ( u _* \O _{Z}, \M)$ (see \cite[III.6]{HaRD}).

If $\M$ is a right $\D _X ^{(m)}$-module, we denote by 
$u ^{\flat 0} (\M)
:=
u ^{-1} \mathcal{H}om _{\O _X} ( u _* \O _{Z}, \M)$.
To simplify notation, we will write 
$u ^{\flat 0} (\M)
:=
\mathcal{H}om _{\O _X} ( \O _{Z}, \M)$.
By definition, the right $\D _Z ^{(m)}$-module structure of 
$u ^{\flat 0} (\M)$ is given by the following $m$-PD-costratification
$\epsilon ^{u ^{\flat 0}(\M)} _n$ making commutative the diagram:
\begin{equation}
\label{diag-rightD-moduleflat}
\xymatrix{
{\mathcal{H} om _{\O _Z} ( p ^{n} _{0*} \cP ^n _{Z,\,(m)},\,\mathcal{H}om _{\O _X} ( \O _{Z}, \M))} 
\ar[d] ^-{\sim} _-{\mathrm{ev}(1)}
\ar@{.>}[r] ^-{\sim} _-{\epsilon ^{u ^{\flat 0}(\M)} _n}
&
{\mathcal{H} om _{\O _Z} ( p ^{n} _{1*} \cP ^n _{Z,\,(m)},\,\mathcal{H}om _{\O _X} ( \O _{Z}, \M))} 
\ar[d] ^-{\sim} _-{\mathrm{ev}(1)}
\\ 
{\mathcal{H} om _{\O _X} ( p ^{n} _{0*} \cP ^n _{Z,\,(m)},\,\M)} 
&
{\mathcal{H} om _{\O _X} ( p ^{n} _{1*} \cP ^n _{Z,\,(m)},\,\M)} 
\\ 
{\mathcal{H} om _{\cP ^n _{X,\,(m)}} ( \cP ^n _{Z,\,(m)},\,\mathcal{H}om _{\O _X} ( p ^{n} _{0*} \cP ^n _{X,\,(m)}, \M)) } 
\ar[u] _-{\sim} ^-{\mathrm{ev}(1)}
\ar@{^{(}->}[d] ^-{\mathrm{ev}(1)}
\ar[r] ^-{\sim} _-{\epsilon ^\M _n}
& 
{\mathcal{H} om _{\cP ^n _{X,\,(m)}} ( \cP ^n _{Z,\,(m)},\,\mathcal{H}om _{\O _X} ( p ^{n} _{1*} \cP ^n _{X,\,(m)}, \M)) } 
\ar[u] _-{\sim} ^-{\mathrm{ev}(1)}
\ar@{^{(}->}[d] ^-{\mathrm{ev}(1)}
\\ 
{\mathcal{H}om _{\O _X} ( p ^{n} _{0*} \cP ^n _{X,\,(m)}, \M)} 
\ar[r] ^-{\sim} _-{\epsilon ^\M _n}
& 
{\mathcal{H}om _{\O _X} ( p ^{n} _{1*} \cP ^n _{X,\,(m)}, \M).} 
}
\end{equation}
Both left (resp. right) top vertical isomorphisms correspond both isomorphisms
$p ^{\flat0} _0 \circ u ^{\flat 0}
\riso
(u \circ p  _0) ^{\flat 0}
=
(p _0 \circ \Delta ^n (u)  ) ^{\flat 0} 
\liso 
(\Delta ^n (u) ) ^{\flat 0} 
 \circ 
 p ^{\flat 0} _0$
(resp. 
$p ^{\flat0} _1 \circ u ^{\flat 0}
\riso
(u \circ p  _1) ^{\flat 0}
=
(p _1 \circ \Delta ^n (u)  ) ^{\flat 0} 
\liso
 (\Delta ^n (u) ) ^{\flat 0}
 \circ
 p ^{\flat0} _1 
$).
The bottom vertical arrows are monomorphisms
because 
 $u ^{-1} \cP ^n  _{X/T,(m)} \to \cP ^n  _{Z/T,(m)}$
 is surjective.
We check the cocycle conditions via similar isomorphisms.

Since $\D _X ^{(m)}$ is a flat $\O _X$-module, then an injective 
right $\D _X ^{(m)}$-module is an injective $\O_X$-module. 
Hence, taking an injective resolution of a complex 
of $D ^{+} ( {} ^r \D _X ^{(m)})$, 
we check the functor $u ^\flat$ sends
$D ^{+}  ( {} ^r \D _X ^{(m)})$
to
$D ^{+} ({} ^r \D _Z ^{(m)} )$, i.e. 
it induces
\begin{equation}
\label{dfn-uflat}
u ^\flat
\colon 
D ^{+}  ( {} ^r \D _X ^{(m)})
\to
D ^{+} ({} ^r \D _Z ^{(m)} ).
\end{equation}
When the level $m$ is ambiguous, 
we denote it more specifically by
$u ^{\flat (m)}$.

Since $X$ is locally noetherian, then $u ^\flat$ preserves the quasi-coherence and
sends 
$D ^{+} _{\mathrm{qc}} ( {} ^r \D _X ^{(m)})$
to 
$D ^{+} _{\mathrm{qc}} ({} ^r \D _Z ^{(m)} )$.

\end{empt}

\begin{lem}
\label{locadesc-uflat1} 
Suppose we are in the local situation of \ref{locdesc-climm}. 
Let $\M$ be a right $\D _X ^{(m)}$-module.
Let $x \in \Gamma (Z,  u ^{\flat 0} (\M))$ and 
$Q \in D ^{(m)} _{Z}$. 
Choose any $Q _X\in D ^{(m)} _{X,Z,\underline{t}/T}$ 
such that  $\theta (Q ) = \overline{Q _X}$
(see \ref{diag-DZ2X-t-comm}).
We have the formula
\begin{equation}
\label{locadesc-uflat1-formula}
\mathrm{ev} _1
(x  \cdot Q)
= 
\mathrm{ev} _1 (x)  \cdot Q _X,
\end{equation}
where $\mathrm{ev} _1 \colon \Gamma ( Z, u ^{\flat 0} (\M) )
\hookrightarrow \Gamma (X, \M)$ is 
the evaluation at $1$ homomorphism.

\end{lem}

\begin{proof}
Let $\underline{l}\in \N ^r$.
Let us consider the case where 
$Q  = \underline{\partial} ^{<\underline{l}> _{(m)}}$ and 
$Q _X = \underline{\partial} ^{<(\underline{l}, \underline{0})> _{(m)}}$.
Via the isomorphism
$\mathcal{H}om _{\O _X} ( \O _{Z}, \M)
\otimes _{\O _X}
\D ^{(m)} _{X/T,n}
\riso 
\mathcal{H} om _{\O _Z} ( p ^{n} _{0*} \cP ^n _{Z,\,(m)},\,\mathcal{H}om _{\O _X} ( \O _{Z}, \M))$,
we identify $x \otimes \underline{\partial} ^{<\underline{l}> _{(m)}} $
with a global section of $\mathcal{H} om _{\O _Z} ( p ^{n} _{0*} \cP ^n _{Z,\,(m)},\,\mathcal{H}om _{\O _X} ( \O _{Z}, \M))$.
Then the composition of the left vertical arrows of \ref{diag-rightD-moduleflat}
sends 
$x \otimes \underline{\partial} ^{<\underline{l}> _{(m)}} $
to 
$\mathrm{ev} _1 (x)
\otimes 
\underline{\partial} ^{<(\underline{l}, \underline{0})> _{(m)}}
$.
Using the formula \cite[1.1.6.1]{Be2}, 
the image of 
$\mathrm{ev} _1 (x)
\otimes 
\underline{\partial} ^{<(\underline{l}, \underline{0})> _{(m)}}$
via the composition of 
$\epsilon _n ^{\cM}$ 
with the evaluation at $1$ homomorphism
$\mathcal{H}om _{\O _X} ( p ^{n} _{1*} \cP ^n _{X,\,(m)}, \M)
\to \cM$
is equal to 
$\mathrm{ev} _1 (x)  \cdot\underline{\partial} ^{<(\underline{l}, \underline{0})> _{(m)}}
\in 
 \Gamma (X, \M)$.
 Similarly, 
the image of 
$x \otimes \underline{\partial} ^{<\underline{l}> _{(m)}}$
via the composition of 
$\epsilon _n ^{u ^{\flat 0}\cM}$
with the evaluation at $1$ homomorphism
$\mathcal{H} om _{\O _Z} ( p ^{n} _{1*} \cP ^n _{Z,\,(m)},\,\mathcal{H}om _{\O _X} ( \O _{Z}, \M))
\to 
 \mathcal{H}om _{\O _X} ( \O _{Z}, \M)$
 is 
$x \cdot \underline{\partial} ^{<\underline{l}> _{(m)}}  $.
Hence,
by using the commutativity of the diagram \ref{diag-rightD-moduleflat}
we get 
\begin{equation}
\label{prelocadesc-uflat1-formula}
\mathrm{ev} _1(x  \cdot 
\underline{\partial} ^{<\underline{l}> _{(m)}})
= 
\mathrm{ev} _1 (x)  \cdot
\underline{\partial} ^{<(\underline{l}, \underline{0})> _{(m)}}.
\end{equation}
Finally, we check easily the formula \ref{locadesc-uflat1-formula}
from \ref{prelocadesc-uflat1-formula}.
\end{proof}

\begin{rem}
\label{rem-rightD-moduleflat}
Following \ref{rightD-moduleflat}, 
if $\M$ is a right $\D _X ^{(m)}$-module, 
then we have a canonical structure of 
right $\D _Z ^{(m)}$-module on
$u ^{\flat 0} (\M)
:=
u ^{-1} \mathcal{H}om _{\O _X} ( u _* \O _{Z}, \M)$
which is constructed by using the $m$-PD-costratification associated to 
$\M$.
In fact, using the canonical homomorphism
$\cD ^{(m)} _{Z/T}
\to 
 \cD ^{(m)} _{Z \to X/T}
 =
u ^{-1}( \cD ^{(m)} _{X/T}
 /\cI \cD ^{(m)} _{X/T})$,
there is another canonical way to give a structure of 
right $\D _Z ^{(m)}$-module on
$u ^{\flat 0} (\M)$.
Indeed, suppose $X$ affine. 
Let $x \in \Gamma (Z,  u ^{\flat 0} (\M))$ and 
$Q \in D ^{(m)} _{Z}$. 
For any 
$Q _X\in D ^{(m)} _{X}$ 
such that  $\theta (Q ) = \overline{Q _X}$,
we define $x  \cdot Q$ so that we get the equality
\begin{equation}
\label{rem-locadesc-uflat1-formula}
\mathrm{ev} _1
(x  \cdot Q)
: = 
\mathrm{ev} _1 (x)  \cdot Q _X,
\end{equation}
where $\mathrm{ev} _1 \colon \Gamma ( Z, u ^{\flat 0} (\M) )
\hookrightarrow \Gamma (X, \M)$ is 
the evaluation at $1$ homomorphism (which is injective).
Since $I$ annihilates $\mathrm{ev} _1 (x) $,
we remark that this is well defined. 

Following \ref{locadesc-uflat1-formula}, both canonical structures of
right $\D _Z ^{(m)}$-module on
$u ^{\flat 0} (\M)$ are identical.

\end{rem}

\begin{empt}
[Local description of the right $\D _X ^{(m)}$-module structure of $u ^{\flat 0} (\M) $]
\label{locadesc-uflat2} 
Suppose we are in the local situation of \ref{locdesc-climm}. 
Let $\M$ be a right $\D _X ^{(m)}$-module.
Since $u _* \O _Z \otimes _{\O _X}\D ^{(m)} _{X,Z,\underline{t}/T}
\riso 
\D ^{(m)} _{X,Z,\underline{t}/T}/\cI \D ^{(m)} _{X,Z,\underline{t}/T}$, we have the isomorphism
$u ^{\flat 0} (\M)
\riso 
u ^{-1}
\mathcal{H}om _{\D ^{(m)} _{X,Z,\underline{t}/T} } ( \D ^{(m)} _{X,Z,\underline{t}/T}/\cI \D ^{(m)} _{X,Z,\underline{t}/T}, \M).$
By using Lemma \ref{locadesc-uflat1}, we check that 
this  isomorphism can be seen as an 
 isomorphism of right of $\D ^{(m)} _{X,Z,\underline{t}/T}/\cI \D ^{(m)} _{X,Z,\underline{t}/T}$-modules of the form:
\begin{gather}
\label{locadesc-uflat2-iso2pre}
\rho _* u ^{\flat 0} (\M)
\riso 
u ^{-1}
\mathcal{H}om _{\D ^{(m)} _{X,Z,\underline{t}/T} } ( \D ^{(m)} _{X,Z,\underline{t}/T}/\cI \D ^{(m)} _{X,Z,\underline{t}/T}, \M).
\end{gather}
Following \ref{locdesc-climm2}.(\ref{locadesc-uflat2-iso1}),
from \ref{locadesc-uflat2-iso2pre} 
we get the isomorphism
\begin{gather}
\label{locadesc-uflat2-iso2}
u ^{\flat 0} (\M)
\riso 
u ^{-1}
\mathcal{H}om _{\D ^{(m)} _{X,Z,\underline{t}/T} } ( u _* \D ^{(m)} _{Z}, \M)
\end{gather}
of right $u _*\D ^{(m)} _{Z}$-modules. 
If there is no ambiguity, 
we can avoid writing $u ^{-1}$, $u _*$ and $\rho _*$, 
e.g. we can simply write
$\mathcal{H}om _{\D ^{(m)} _{X,Z,\underline{t}/T} } ( \D ^{(m)} _{Z}, \M)
=
u ^{-1}
\mathcal{H}om _{\D ^{(m)} _{X,Z,\underline{t}/T} } ( u _* \D ^{(m)} _{Z}, \M)
$.
\end{empt}

\begin{empt}
Suppose we are in the local situation of \ref{locdesc-climm}. 
Let $\M$ be a right $\D _X ^{(m)}$-module.

\begin{enumerate}
\item 
For simplicity, 
we remove $\rho _*$ in the notation
and 
we  view 
$u ^{\flat} (\M)$
as an object of 
$D ^{\mathrm{b}} ({} ^r\D ^{(m)} _{X,Z,\underline{t}/T})$.
By derivating 
\ref{locadesc-uflat2-iso2pre},
we get the isomorphism of $D ^{\mathrm{b}} ({} ^r\D ^{(m)} _{X,Z,\underline{t}/T})$ of the form
\begin{equation}
\label{Rlocadesc-uflat2-iso2pre}
u ^{\flat} (\M) \riso \R \mathcal{H}om _{\D ^{(m)} _{X,Z,\underline{t}/T} } (  \D ^{(m)} _{X,Z,\underline{t}/T}/\cI \D ^{(m)} _{X,Z,\underline{t}/T} , \M).
\end{equation}

Let $s := d -r$, and $f _1= t _{r+1},\dots, f _s := t _{d}$. 
Let $K _{\bullet} (\underline{f})$ be the Koszul complex
of $\underline{f}=(f_1,\dots,f_s)$. Let $e _1, \dots, e _s$ be the canonical basis of $\O _X ^s$. Recall 
$K _i (\underline{f}) = \wedge ^i (\O _X ^s)$ and 
$d _{i, \underline{f}} \colon K _i (\underline{f}) \to K _{i-1} (\underline{f})$ 
(or simply $d _i$) is the $\O _X$-linear map defined by 
$$d _i ( e _{n _1} \wedge \cdots \wedge e _{n _i})
= 
\sum _{j=1} ^{i}
(-1) ^{j-1}
f _{n _j} e _{n _1} \wedge \cdots \wedge \widehat{e} _{n _j} \wedge \cdots \wedge e _{n _i}.$$

Since $f _1,\dots, f _s$ is a regular sequence of $\I$, 
then the canonical morphism 
$K _{\bullet} (\underline{f}) \to \O _X/\I$
(given by the canonical map
$K _{0} (\underline{f}) = \O _X \to \O _X/\I$)
is a quasi-isomorphism.
Hence, we get the isomorphism of $D ^{\mathrm{b}} (\O _{X})$ 
$$\phi _{\underline{f}} 
\colon 
 u ^{\flat} (\M) 
 \riso\mathcal{H}om _{\O _{X} } ( K _{\bullet} (\underline{f}) , \M).$$
Since $f _1,\dots f _s$ are in the center of $\D ^{(m)} _{X,Z,\underline{t}/T} $ and 
$\D ^{(m)} _{X,Z,\underline{t}/T}$ is a flat $\O _X$-algebra, then we get the quasi-isomorphism
$\D ^{(m)} _{X,Z,\underline{t}/T} \otimes _{\O _X} K _{\bullet} (\underline{f}) 
\riso 
 \D ^{(m)} _{X,Z,\underline{t}/T}/\cI \D ^{(m)} _{X,Z,\underline{t}/T}$ 
which is a resolution of $ \D ^{(m)} _{X,Z,\underline{t}/T}/\cI \D ^{(m)} _{X,Z,\underline{t}/T}$ by 
free $\D ^{(m)} _{X,Z,\underline{t}/T}$-bimodules of finite type.
We get the commutativity of diagram
\begin{equation}
\label{locadesc-uflat-iso}
\xymatrix{
{u ^{\flat} (\M)} 
\ar[r] ^-{\phi _{\underline{f}}} _-{\sim}
\ar[d] ^-{\sim} _{\ref{Rlocadesc-uflat2-iso2pre}}
& 
{\mathcal{H}om _{\O _{X} } ( K _{\bullet} (\underline{f}) , \M)} 
\ar[d] ^-{\sim}
\\ 
{\R \mathcal{H}om _{\D ^{(m)} _{X,Z,\underline{t}/T} } (  \D ^{(m)} _{X,Z,\underline{t}/T}/\cI \D ^{(m)} _{X,Z,\underline{t}/T} , \M)} 
& 
{\mathcal{H}om _{\D ^{(m)} _{X,Z,\underline{t}/T} } ( \D ^{(m)} _{X,Z,\underline{t}/T} \otimes _{\O _X} K _{\bullet} (\underline{f}) , \M).} 
\ar[l] ^-{\sim}
}
\end{equation}
We denote by 
$\phi _{\underline{t}} 
\colon 
u ^{\flat} (\M)
\riso
\mathcal{H}om _{\D ^{(m)} _{X,Z,\underline{t}/T} } ( \D ^{(m)} _{X,Z,\underline{t}/T} \otimes _{\O _X} K _{\bullet} (\underline{f}) , \M)$
the isomorphism induced by composition with $\phi _{\underline{f}} $. 
By commutativity of \ref{isos-fund-isom1pre},
$\phi _{\underline{t}} $ 
 is an isomorphism of $D ^{\mathrm{b}} ({} ^r \D ^{(m)} _{X,Z,\underline{t}/T})$.
Hence, we get the  isomorphism of right $\D ^{(m)} _{X,Z,\underline{t}/T}$-modules
\begin{equation}
\label{isos-fund-isom1pre}
\phi ^s _{\underline{t}} 
=
\mathcal{H} ^s(\phi _{\underline{t}})
\colon 
R ^s u ^{\flat 0} (\M)
\riso 
\mathcal{H} ^s \mathcal{H}om _{\D ^{(m)} _{X,Z,\underline{t}/T} } ( \D ^{(m)} _{X,Z,\underline{t}/T} \otimes _{\O _X} K _{\bullet} (\underline{f}) , \M).
\end{equation}
We have the homomorphism of right $\D ^{(m)} _{X,Z,\underline{t}/T}$-modules 
$\mathcal{H}om _{\D ^{(m)} _{X,Z,\underline{t}/T} } ( \D ^{(m)} _{X,Z,\underline{t}/T} \otimes _{\O _X} K _{s} (\underline{f}) ,  \M)
\to 
 \M$
(the structure of right $\D ^{(m)} _{X,Z,\underline{t}/T}$-module on the left term comes from the structure of left 
$\D ^{(m)} _{X,Z,\underline{t}/T}$-module on $\D ^{(m)} _{X,Z,\underline{t}/T} \otimes _{\O _X} K _{s} (\underline{f}) $)
given by 
$\phi \mapsto \phi ( e _1 \wedge  \cdots \wedge e _s)$.
Since 
$\cM \to \M / \I \M$ is a morphism of right $\D ^{(m)} _{X,Z,\underline{t}/T}$-modules,
then this induces  the morphism of complex of right $\D ^{(m)} _{X,Z,\underline{t}/T}$-modules of the form
$\mathcal{H}om _{\D ^{(m)} _{X,Z,\underline{t}/T} } ( \D ^{(m)} _{X,Z,\underline{t}/T} \otimes _{\O _X} K _{\bullet} (\underline{f}) , \M)
\to \M / \I \M$.
This yields the isomorphism of right $\D ^{(m)} _{X,Z,\underline{t}/T}$-modules
\begin{equation}
\label{isos-fund-isom2pre}
\mathcal{H} ^s \mathcal{H}om _{\D ^{(m)} _{X,Z,\underline{t}/T} } ( \D ^{(m)} _{X,Z,\underline{t}/T} \otimes _{\O _X} K _{\bullet} (\underline{f}) , {}  \M )
\riso  \M / \I \M.
\end{equation}

\item {\it Change of coordinates.}
\label{ntn-chgcoor}
Make a second choice : 
let $t ' _{r +1},\dots , t '_{d}  \in \Gamma (X,\I)$
generating 
$I:=\Gamma (X,\I)$,
$t '_{1},\dots , t '_{r}\in \Gamma ( X,\O _{X})$
such that
$t ' _{1},\dots ,t  '_{d}$ are local coordinates of $X$ over $T$,
$\overline{t} '_{1},\dots ,\overline{t} '_{r}$ 
are local coordinates of $Z$ over $T$,
and 
$\overline{t} '_{r +1},\dots ,\overline{t} '_{d}$ is a basis of $\I /\I ^2$.
Let $A := \Gamma (X,\O _X)$
and $B : =\Gamma ( Z ,\O _{Z})$.
Set $f '_1= t '_{r+1},\dots, f '_s := t '_{d}$.
Let $K _{\bullet} (\underline{f}')$ be the Koszul complex
of $\underline{f}'=(f '_1,\dots,f '_s)$.
Let $M _I  = ( c _{ij}) \in M _s (A)$ be the matrix such that 
$\sum _{i =1} ^{s} f ' _i c _{ij} = f _j$. 
Let $\vartheta \colon A ^s \to A ^s$ be the morphism associated with 
$M _I$. It corresponds to a morphism
$\vartheta \colon K _{1} (\underline{f})
\to 
K _{1} (\underline{f}')$.
We compute that the composition of 
$\vartheta$ with 
$d _{1, \underline{f}}
\colon 
K _{1} (\underline{f}') \to K _{0} (\underline{f}')=A$
is equal to
$d _{1, \underline{f}'}
\colon 
K _{1} (\underline{f}) \to K _{0} (\underline{f}) =A$.
Since 
$K _{\bullet} (\underline{f}) = \wedge K _{1} (\underline{f}) $
this yields the morphism of complexes
$\wedge \vartheta \colon 
K _{\bullet} (\underline{f}) 
\to K _{\bullet} (\underline{f}')$. 
Hence, we get the commutative diagram
$$\xymatrix{
{u ^{\flat} (\M)} 
\ar[r] _-{\sim} ^-{\phi _{\underline{f}}}
\ar[rd] ^-{\sim} _-{\phi _{\underline{f}'}}
& 
{\mathcal{H}om _{\O _{X} } ( K _{\bullet} (\underline{f}) , \M)}
\\ 
{} 
& 
{\mathcal{H}om _{\O _{X} } (  K _{\bullet} (\underline{f}') , \M).}
\ar[u] ^-{\wedge \vartheta} 
}
$$
This yields the commutativity of the left square of the diagram
\begin{equation}
\label{fund-isom1-diag1pre}
\xymatrix{
{R ^s u ^{\flat 0} (\M)} 
\ar@/^4ex/[rr] ^-{\phi ^s _{\underline{t}}}
\ar[r] ^-{\sim} _-{\cH ^s \phi _{\underline{f}}}
& 
{\mathcal{H} ^s \mathcal{H}om _{\O _{X }} (K _{\bullet} (\underline{f}) , \M)} 
\ar[r] _-{\sim} 
& 
{\mathcal{H} ^s \mathcal{H}om _{\D ^{(m)} _{X,Z,\underline{t}/T} } ( \D ^{(m)} _{X,Z,\underline{t}/T} \otimes _{\O _X} K _{\bullet} (\underline{f}) , \M)} 
\ar[r] _-{\sim} ^-{\ref{isos-fund-isom2pre}}
&
{\M /\I \M} 
\\ 
{R ^s u ^{\flat 0} (\M)} 
\ar@/_4ex/[rr] _-{\phi ^s _{\underline{t'}}}
\ar[r] _-{\sim} ^-{\cH ^s\phi  _{\underline{f}'}}
\ar@{=}[u] ^-{}
& 
{\mathcal{H} ^s \mathcal{H}om _{\O _{X} }(K _{\bullet} (\underline{f}') , \M)}
\ar[u] ^-{\det  \vartheta} 
\ar[r] _-{\sim} 
& 
{\mathcal{H} ^s \mathcal{H}om _{\D ^{(m)} _{X,Z,\underline{t}'/T} } ( \D ^{(m)} _{X,Z,\underline{t}'/T} \otimes _{\O _X} K _{\bullet} (\underline{f}') , \M)}
\ar[r] _-{\sim} ^-{\ref{isos-fund-isom2pre}}
&
{\M /\I \M,}  
\ar[u] ^-{\overline{\det  \vartheta}} 
}
\end{equation}
whose compositions of horizontal morphisms are 
$\D ^{(m)} _{X,Z,\underline{t}/T}$-linear.
Hence, the diagram \ref{fund-isom1-diag1pre} is commutative.

\end{enumerate}

\end{empt}

\begin{ntn}
\label{dfn-u!leftclimm}

\label{ntnu*form}
If $\cE$ is a left $\D _X ^{(m)}$-module, we denote by 
$u ^{*} (\cE)
:=
 \O _{Z}
 \otimes _{ u ^{-1}\O _X}
 u ^{-1}\cE$.
Using $m$-PD-stratifications, 
we get a structure of left $\D _Z ^{(m)}$-module 
on $u ^{*} (\cE)$.
Via the homomorphisms of left $\D _X ^{(m)}$-modules of the form $\D  ^{(m)} _X\to \cE$, 
we check by functoriality that 
the canonical homomorphism 
\begin{equation}
\label{dfn-u!leftclimm-iso}
\O _{Z}
 \otimes _{ u ^{-1}\O _X}
 u ^{-1}\cE
 \to \D  ^{(m)} _{Z \to X}
 \otimes  _{ u ^{-1}\D  ^{(m)} _X}
 u ^{-1}\cE
\end{equation}
 is an isomorphism of 
 left $\D _Z ^{(m)}$-modules.
By deriving, we get the functor
$\L u ^* \colon 
D ^{+}  ( {} ^l \D _X ^{(m)}) 
\to
D ^{+} ({} ^l \D _Z ^{(m)} )$
(or 
$\L u ^* \colon 
D   ( {} ^l \D _X ^{(m)}) 
\to
D ({} ^l \D _Z ^{(m)} )$)
defined by setting 
\begin{equation}
\label{ntnu*form-dfn1}
\L u ^{*} (\cE)
:=
\D  ^{(m)} _{Z \to X}
 \otimes ^\L _{ u ^{-1}\D  ^{(m)} _X}
 u ^{-1}\cE.
\end{equation}
Finally, we set
$u ^! (\cE) := \L u ^{*} (\E ) [d _{Z/X}]$.

Suppose now we are in the local situation of \ref{locdesc-climm}. 
Let $Q \in D ^{(m)} _{Z}$. 
Choose $Q _X\in D ^{(m)} _{X,Z,\underline{t}/T}$ 
such that  $\overline{Q _X}= \theta (Q)$.
From \ref{dfn-u!leftclimm-iso}, we check the formula
\begin{equation}
\label{dfn-u!leftclimm-pre2}
Q (u ^* (x))
=
u ^* (Q _X \cdot x)).
\end{equation}
Let $\E \in D ^{+}  ( {} ^l  \D  ^{(m)} _X) $.
Via the monomorphism of rings
$\D ^{(m)} _{X,Z,\underline{t}}
\hookrightarrow 
\D ^{(m)} _{X}$, 
we check the canonical homomorphism 
$$
 \D  ^{(m)} _{ Z } \otimes ^\L _{u ^{-1} \D  ^{(m)} _{X, Z ,\underline{t}}} u ^{-1}\E
 \to
 \L u ^* (\E) $$
 is an isomorphism of $D ^{+}  ( {} ^l  \D  ^{(m)} _Z) $.
This yields
$$( \D  ^{(m)} _{X, Z ,\underline{t}} \otimes _{\O _{X}} K _{\bullet} (\underline{f}) ) 
\otimes _{ u ^{-1}\D  ^{(m)} _{X, Z ,\underline{t}} } u ^{-1}\E
\riso 
\L u ^* (\E) .$$

\end{ntn}

\begin{prop}
\label{fund-isom}
Let $\E$ be a left $\D _X ^{(m)}$-module (resp. a $\D _X ^{(m)}$-bimodule). 
We have the canonical isomorphism of right $\D _Z ^{(m)}$-modules (resp. 
of right $( \D _{Z} ^{(m)}, u ^{-1}\D _{X} ^{(m)})$-bimodules):
\begin{equation}
\label{fund-isom1}
R ^{-d _{Z/X}} u ^{\flat 0} ( \omega _{X} \otimes _{\O _X} \E)
\riso 
\omega _{Z} \otimes _{\O _Z} u ^{*} (\E ).
\end{equation}
\end{prop}

\begin{proof}
By functoriality, we reduce to check the non respective case.

0) First, let us suppose the conditions of \ref{locdesc-climm} are satisfied, i.e. 
suppose
$X$ is affine and there exist
$t _{r +1},\dots , t _{d}  \in \Gamma (X,\I)$
generating 
$I:=\Gamma (X,\I)$,
$t _{1},\dots , t _{r}\in \Gamma ( X,\O _{X})$
such that
$t _{1},\dots ,t  _{d}$ are local coordinates of $X$ over $T$,
$\overline{t} _{1},\dots ,\overline{t} _{r}$ 
are local coordinates of $Z$ over $T$,
and 
$\overline{t} _{r +1},\dots ,\overline{t} _{d}$ is a basis of $\I /\I ^2$.
In that case $\omega _{Z}$ is a free $\O _Z$-module with the basis
 $d \overline{t} _{1} \wedge  \cdots \wedge d \overline{t} _{r}$,
and $\omega _X $ is a free $\O _X$-module with the basis $d t _{1} \wedge  \cdots \wedge dt _{d}$.

1) Since the isomorphism
$\theta \colon D ^{(m)} _{Z/T} 
\riso D ^{(m)} _{X,Z,\underline{t}/T}$ 
commutes with the transposition automorphisms (see \ref{transp-closedimmersion}), 
we get the isomorphism of right $\D ^{(m)} _{X,Z,\underline{t}/T}$-modules
$$
\psi _{\underline{t}}
\colon 
(\omega _{X} \otimes _{\O _X} \E ) /( \omega _{X} \otimes _{\O _X} \E )  \I 
\riso 
\omega _{Z} \otimes _{\O _Z} 
(\E/\I \E),$$ 
which is given by
$d t _{1} \wedge  \cdots \wedge dt _{d} \otimes x \mod \I 
\mapsto 
d \overline{t} _{1} \wedge  \cdots \wedge d \overline{t} _{r}
\otimes (x \mod \I)$.

2) Let $s := d -r$, and $f _1= t _{r+1},\dots, f _s := t _{d}$. 
Notice that $s = -d _{Z/X}$ ($r = d _Z$ and $d = d _X$).
Using the isomorphism 
$\phi ^s _{\underline{t}} $
constructed in \ref{isos-fund-isom1pre}, 
we get by composition the isomorphism of right $\D ^{(m)} _{Z}$-modules
\begin{gather}
\notag 
R ^s u ^{\flat 0} (\omega _{X} \otimes _{\O _X} \E)
\underset{\phi ^s _{\underline{t}} }{\riso} 
\mathcal{H} ^s \mathcal{H}om _{\D ^{(m)} _{X,Z,\underline{t}/T} } ( \D ^{(m)} _{X,Z,\underline{t}/T} \otimes _{\O _X} K _{\bullet} (\underline{f}) , \omega _{X} \otimes _{\O _X} \E)
\\
\label{isos-fund-isom}
\underset{\ref{isos-fund-isom2pre}}{\riso} 
(\omega _{X} \otimes _{\O _X} \E ) / ( \omega _{X} \otimes _{\O _X} \E ) \I 
\underset{\psi _{\underline{t}}}{\riso} 
\omega _{Z} \otimes _{\O _Z} 
(\E/\I \E).
\end{gather}

3) 
It remains to check  that the composition of the isomorphisms of \ref{isos-fund-isom} does not depend on the choice of the coordinates. 
Make some second choice : Let $t ' _{r +1},\dots , t '_{d}  \in \Gamma (X,\I)$
generating 
$I:=\Gamma (X,\I)$,
$t '_{1},\dots , t '_{r}\in \Gamma ( X,\O _{X})$
such that
$t ' _{1},\dots ,t  '_{d}$ are local coordinates of $X$ over $T$,
$\overline{t} '_{1},\dots ,\overline{t} '_{r}$ 
are local coordinates of $Z$ over $T$,
and 
$\overline{t} '_{r +1},\dots ,\overline{t} '_{d}$ is a basis of $\I /\I ^2$.
Let $A := \Gamma (X,\O _X)$,
$B : =\Gamma ( Z ,\O _{Z})$,
$R: =\Gamma (T, \O _T)$.
Set $f '_1= t '_{r+1},\dots, f '_s := t '_{d}$.
Let $( c _{ij})  \in M _s (A)$ be the matrix such that 
$\sum _{i =1} ^{s} f ' _i c _{ij} = f _j$. 
Let $\vartheta \colon A ^s \to A ^s$ be the morphism associated with 
$( c _{ij})$. 
Let 
$( d _{ij})  \in M _d (A)$ be the matrix such that 
$d t ' _i = \sum _{j =1} ^{d} d _{ij} d t _j $. 
Let $D = ( \overline{d} _{ij}) _{1\leq i,j \leq d} \in M _d (B)$,
$D _1 = ( \overline{d} _{ij}) _{1\leq i,j \leq r} \in M _r (B)$,
and 
$D _2 = ( \overline{d} _{ij}) _{r+1\leq i,j \leq d} \in M _s (B)$.
We denote by 
$\overline{d t _i}$
and 
$\overline{d t ' _i}$
the image of 
$d t _i$ and $d t' _i$ via the homomorphism
$\Omega _{A/R} \to B \otimes _{A}\Omega _{A/R}$.
We get 
$\overline{d t '_{1}} \wedge  \cdots \wedge \overline{d t '_{d}}
=( \det D)  \left ( 
\overline{d t _{1}} \wedge  \cdots \wedge \overline{d t _{d}}
\right )$.

By considering the images via the canonical morphism
$u ^* \Omega _{X} \to \Omega _{Z}$ which sends
$\overline{d t _i }$ (resp. $\overline{d t ' _i }$)
to 
$d\, \overline{t  _i} $ (resp. $d\, \overline{t ' _i} $),
we get the equality
$d\, \overline{t ' _i} = \sum _{j =1} ^{d} \overline{d} _{ij} d\, \overline{t  _j} $
for any $i=1,\dots, d$.
When $i \geq r+1$,
we have $d\, \overline{t  _i} =0$
and $d\, \overline{t ' _i} =0$.
Since $d\, \overline{t  _1}, \dots, d\, \overline{t  _r}$ is a basis of $\Omega _{Z}$, then
$\overline{d} _{ij} =0$ for any $ i \geq r+1$ and $j \leq r$
and then
$d\, \overline{t ' _i} = \sum _{j =1} ^{d} \overline{d} _{ij} d\, \overline{t  _j} $
for any $i=1,\dots, r$.
This yields both equalities $\det D = \det D _1 \det D _2$
and
$d \overline{t' _1} \wedge  \cdots \wedge d \overline{t' _r}
= 
(\det D _1 ) 
d \overline{t _1} \wedge  \cdots \wedge d \overline{t _r}$.

For any $i \geq r+1$, 
we have 
$\overline{d t _{i}}
=
\overline{f _{i-r}}$
and 
$\overline{d t '_{i}}
=
\overline{f '_{i-r}}$.
Since for any $i \geq r+1$, we have
$\overline{d t '_{i}} = \sum _{j =r+1} ^{d} \overline{d _{ij}} \,\overline{d t _{j}} $,
this means that 
$D _2$ is the inverse of the transposition matrix of 
$( \overline{c} _{ij}) $.
Hence, 
$\overline{\det \vartheta} = (\det D _2) ^{-1}$.
This implies the commutativity of the following diagram:
\begin{equation}
\label{fund-isom1-diag2}
\xymatrix{
{( \omega _{X} \otimes _{\O _X} \E ) / ( \omega _{X} \otimes _{\O _X} \E )\I } 
\ar[r] ^-{\psi _{\underline{t}}} _-{\sim}
& 
{\omega _{Z} \otimes _{\O _Z} 
(\E/\I \E) } 
\\ 
{(\omega _{X} \otimes _{\O _X} \E )/ ( \omega _{X} \otimes _{\O _X} \E )\I } 
\ar[u] ^-{\overline{\det \vartheta}} 
\ar[r] ^-{\psi _{\underline{t}'}} _-{\sim}
& 
{\omega _{Z} \otimes _{\O _Z} 
(\E/\I \E) .}
\ar@{=}[u] ^-{} 
}
\end{equation}
By composing the commutative diagram 
\ref{fund-isom1-diag1pre} for $\cM = \omega _{X} \otimes _{\O _X} \E$
with
\ref{fund-isom1-diag2},
we get the independence on the choice of the coordinates of the
composition of the isomorphisms of \ref{isos-fund-isom}.
\end{proof}

\begin{thm}
\label{fund-isom-thmpre}
Let $\E \in D  ( {} ^l \D _X ^{(m)})$ (resp. $\E \in D  ( {} ^l \D _X ^{(m)}, {} ^r \D _X ^{(m)})$). 
We have the canonical isomorphism of 
$D  ({} ^r \D _Z ^{(m)})$
(resp. $D  ({} ^r \D _Z ^{(m)},{} ^r u ^{-1}\D _{X} ^{(m)})$)
\begin{equation}
\label{fund-isom2pre}
\omega _{Z} \otimes _{\O _Z} u ^{!} (\E )
\riso
u ^{\flat} ( \omega _{X} \otimes _{\O _X} \E).
\end{equation}
\end{thm}

\begin{proof}
We already know that there exists an isomorphism of the form 
\ref{fund-isom2pre} in the category $D  (\O _Z)$.
This yields that if $\E$ is a flat left
$\D _X ^{(m)}$-module, then for any $i \not = s$, 
$\mathcal{H} ^i u ^{\flat} ( \omega _{X} \otimes _{\O _X} \E)=0$.
Using
\cite[I.7.4]{HaRD},
we conclude using \ref{fund-isom}.
\end{proof}

\begin{coro}
We have the canonical isomorphism of right $( \D _{Z} ^{(m)}, u ^{-1}\D _{X} ^{(m)}) $-bimodules of the form
\begin{equation}
\label{fund-isom2-corpre}
\omega _{Z} \otimes _{\O _Z} \D _{Z\to X} ^{(m)} 
\riso
u ^{\flat} ( \omega _{X} \otimes _{\O _X}  \D _X ^{(m)}) [-d _{Z/X}].
\end{equation}
\end{coro}

\begin{proof}
We apply Theorem \ref{fund-isom-thmpre} in the case
$\cE= \D _X ^{(m)}$.
\end{proof}

\subsection{Relative duality isomorphism and adjunction for a closed immersion of schemes}
Let $T$ be a noetherian $\Z _{(p)}$-scheme of finite Krull dimension.
Let $u \colon Z \hookrightarrow X$ be a closed immersion of smooth $T$-schemes.
Let $\I$ be the ideal defining $u$. The level $m\in\N$ is fixed.

\begin{ntn}
We get the functors
$u ^{(m)} _+\colon 
D  ( {} ^*  \D _Z ^{(m)}) 
\to
D  ({} ^*  \D _X^{(m)} )$ 
by setting
for any 
$\cE \in D   ( {} ^l  \D _Z ^{(m)}) $
and
$\cN \in D   ( {} ^r  \D _Z ^{(m)}) $
by setting
\begin{gather}
\label{dfnu+(m)}
u ^{(m)} _+ (\cE) := u _* 
\left ( 
\D  _{X \leftarrow Z} ^{(m)}
\otimes _{\D  _Z ^{(m)}} 
\cE 
\right ),
u ^{(m)} _+ (\cN) := u _* 
\left ( \cN \otimes _{\D  _Z ^{(m)}} \D  _{Z \to X} ^{(m)}
\right ).
\end{gather}

Moreover, we get the dual functors
$\DD ^{(m)} \colon 
D  ( {} ^*  \D _X ^{(m)}) 
\to
D  ({} ^*  \D _X^{(m)} )$ by setting
for any 
$\cE \in D   ( {} ^l  \D _X ^{(m)}) $
and
$\cM \in D   ( {} ^r  \D _X ^{(m)}) $,
\begin{gather}
\label{dfnuD(m)}
\DD ^{(m)}  (\cE) := 
\R \mathcal{H}om _{\D _X ^{(m)}}
(\cE, \D _X ^{(m)} \otimes _{\O _X} \omega _{X} ^{-1}) [d _{X}],
\
\
\DD ^{(m)}  (\cM) := 
\R \mathcal{H}om _{\D _X ^{(m)}}
(\cM,\omega _{X} \otimes _{\O _X}  \D _X ^{(m)}) [d _{X}],
\end{gather}
which are respectively computed by taking an injective resolution of
$\D _X ^{(m)} \otimes _{\O _X} \omega _{X} ^{-1}$ 
and $\omega _{X} \otimes _{\O _X}  \D _X ^{(m)}$.
These functors preserve the coherence. 
We can remove $(m)$ in the notation if there is no ambiguity with the level.

These functors are compatible with the quasi-inverse functors
$-\otimes _{\O _X} \omega _{X} ^{-1}$ and
$\omega _{X}  \otimes _{\O _X} -$ exchanging left and right 
$\D ^{(m)} _X$-modules structures.
More precisely, we have the canonical isomorphism
\begin{equation}
\label{u+D-left2right}
\omega _{X}  \otimes _{\O _X} u  ^{(m)} _+ (\cE)
\riso 
u ^{(m)} _+ (\omega _{Z}  \otimes _{\O _Z}  \cE),
\end{equation}
which is constructed as follows :
\begin{gather}
\notag
\omega _{X}  \otimes _{\O _X} u _* 
\left ( 
\D ^{(m)}  _{X \leftarrow Z}  
\otimes _{\D ^{(m)}  _Z } 
\cE
\right )
\riso
u _* 
\left ( 
( u ^{-1}\omega _{X}  \otimes _{u ^{-1} \O _X}  \D ^{(m)}  _{X \leftarrow Z}  )
\otimes _{\D ^{(m)}  _Z } 
\cE
\right )
\\
\notag
\riso
u _* 
\left ( 
(\omega _{Z}  \otimes _{\O _Z}   \cE) 
\otimes _{\D ^{(m)}  _Z } 
( u ^{-1}\omega _{X}  \otimes _{u ^{-1} \O _X}  \D ^{(m)}  _{X \leftarrow Z}  
 \otimes _{\O _Z}  \omega _{Z}  ^{-1})
\right )
\riso
u _* 
\left ( 
(\omega _{Z}  \otimes _{\O _Z}   \cE) 
\otimes _{\D ^{(m)}  _Z } 
 \D ^{(m)}  _{Z \to X}  
\right ).
\end{gather}
More easily, we check the isomorphism
$\omega _{X}  \otimes _{\O _X} \DD ^{(m)} (\cE)
\riso 
\DD^{(m)} (\omega _{X}  \otimes _{\O _X}  \cE)$.

\end{ntn}

\begin{prop}
\label{u+uflat-lem}
Let $\M$ be a 
right $\D _X ^{(m)}$-module, 
$\cN$ be a 
right  $\D _Z ^{(m)}$-module.
\begin{enumerate}
\item We have the canonical adjunction morphisms
$\mathrm{adj} 
\colon 
u _+ u ^{\flat 0} (\M)
\to 
\M$
and 
$\mathrm{adj} 
\colon 
\cN
\to 
 u ^{\flat 0} u _+  (\cN)$.
Moreover, the compositions
$ u ^{\flat 0} (\M)
 \overset{\mathrm{adj}}{\longrightarrow} 
u ^{\flat 0} u _+ u ^{\flat 0} (\M)
\overset{\mathrm{adj}}{\longrightarrow} 
 u ^{\flat 0} (\M)$
and 
$u _+  (\cN)
\overset{\mathrm{adj}}{\longrightarrow} 
u _+ u ^{\flat 0} u _+  (\cN)
\overset{\mathrm{adj}}{\longrightarrow} 
u _+  (\cN)$
are the identity.

\item Using the above  adjunction morphisms, we construct maps 
$$\mathcal{H}om _{\D _X ^{(m)}}
(u _+ (\cN), \M)
\to 
u _* \mathcal{H}om _{\D _Z ^{(m)} }
(\cN, u ^{\flat 0} (\M)),
\
u _* \mathcal{H}om _{\D _Z ^{(m)} }
(\cN, u ^{\flat 0} (\M))
\to 
\mathcal{H}om _{\D _X ^{(m)}}
(u _+ (\cN), \M),$$
which are  inverse of each other.
\item The functor $u ^\flat$ transforms 
$K$-injective complexes into
$K$-injective complexes.
\end{enumerate}

\end{prop}

\begin{proof}
1) Let us check the first assertion. 

a) Since the construction of the {\it canonical} morphism
$\mathrm{adj} 
\colon 
u _+ u ^{\flat 0} (\M)
\to 
\M$
is local, we can suppose the assumptions of \ref{locdesc-climm} are satisfied
(and we use notations \ref{locdesc-climm} and \ref{locdesc-climm2}).
We have
$u _+ u ^{\flat 0} (\M)
:= 
u _* 
\left ( u ^{-1} \mathcal{H}om _{\O _X} ( u _* \O _{Z}, \M) \otimes _{\D _Z ^{(m)}} \D _{Z \to X} ^{(m)}
\right )
=
\mathcal{H}om _{\O _X} ( u _* \O _{Z}, \M) \otimes _{u _* \D _Z ^{(m)}} u _*  \D _{Z \to X} ^{(m)}$.
Let $A := \Gamma (X, \O _X)$.
We reduce  to construct a canonical morphism of the form
$\mathrm{Hom} _{A} ( A/I, M) \otimes _{D _Z ^{(m)}} D _{Z \to X} ^{(m)}
\to 
M$. 
Let us check that  the canonical map 
$\mathrm{Hom} _{A} ( A/I, M) \otimes _{D _Z ^{(m)}} D _{Z \to X} ^{(m)}
\to 
M$,
defined by setting 
$x \otimes \overline{P} 
\mapsto 
\mathrm{ev} _1 (x) P$ for any $x \in \mathrm{Hom} _{A} ( A/I, M)$, $P \in D _{X} ^{(m)}$,
is well defined. 
Let $x \in \mathrm{Hom} _{A} ( A/I, M)$,
$P \in D _{X} ^{(m)}$, 
$Q \in D ^{(m)} _{Z}$. 
Choose $Q _X\in D ^{(m)} _{X,Z,\underline{t}/T}$ 
such that  $\theta (Q ) = \overline{Q _X}$.
Using the formula \ref{locadesc-uflat1-formula}, we get
$\mathrm{ev} _1
(x  \cdot Q) P
= 
\mathrm{ev} _1 (x)  \cdot Q _X P$.
Following \ref{eq-QXP},
we have 
$Q \cdot \overline{P} = \overline{Q _X P}$, and then we get the equality
$x \cdot Q \otimes \overline{P} 
= 
x  \otimes \overline{Q _X P} $
in $\mathrm{Hom} _{A} ( A/I, M) \otimes _{D _Z ^{(m)}} D _{Z \to X} ^{(m)}$.
Hence, we have checked the map is well defined.
The $D _{X} ^{(m)}$-linearity of this canonical map is obvious.

b) Similarly, to construct the morphism
$\mathrm{adj} 
\colon 
\cN
\to 
u ^{\flat 0} u _+  (\cN)$, 
we reduce to the case where the assumptions of \ref{locdesc-climm} are satisfied
and to construct a canonical morphism of the form
$N
\to 
\mathrm{Hom} _{A} ( A/I, N \otimes _{D _Z ^{(m)}} D _{Z \to X} ^{(m)})$.
This latter map is defined by setting 
$y \mapsto ( \overline{a} \mapsto  y \otimes \overline{a})$. 
Let $y \in N$, 
$Q \in D ^{(m)} _{Z}$.
Let $x $ be the element of 
$\mathrm{Hom} _{A} ( A/I, N \otimes _{D _Z ^{(m)}} D _{Z \to X} ^{(m)})$
given by 
$( \overline{a} \mapsto  y \otimes \overline{a})$.
Choose $Q _X\in D ^{(m)} _{X,Z,\underline{t}/T}$ 
such that  $\theta (Q ) = \overline{Q _X}$.
Using \ref{locadesc-uflat1-formula}, 
$\mathrm{ev}  _1 ( x \cdot Q) 
= 
\mathrm{ev} _1 ( x ) \cdot Q _X
=
y \otimes \overline{Q _X}
\overset{\ref{eq-QXP}}{=}
y \cdot Q \otimes \overline{1} $.
This yields 
that the canonical map 
$N
\to 
\mathrm{Hom} _{A} ( A/I, N \otimes _{D _Z ^{(m)}} D _{Z \to X} ^{(m)})$
is 
$D _{Z} ^{(m)}$-linear.

c) Reducing to the local context \ref{locdesc-climm}, we compute easily that both compositions are the identity maps.

2) The fact that the first statement of the proposition implies the second one is standard. 
Finally, since the functor $u _+$ is exact, we get the last statement from the second one
(see 13.30.9 of the stack project).
\end{proof}

\begin{coro}
\label{Radj-u+flat}
Let 
$\M \in D 
( {} ^r \D _X ^{(m)})$,
$\cN\in D 
({} ^r \D _Z ^{(m)} )$. 
Let 
$\cE \in D 
( {} ^l \D _X ^{(m)})$,
$\cN\in D 
({} ^l \D _Z ^{(m)} )$. 
We have the isomorphisms
\begin{gather}
\notag
\R \mathcal{H}om _{\D _X ^{(m)}}
(u _+ (\cN), \M)
\riso 
u _* \R \mathcal{H}om _{\D _Z ^{(m)} }
(\cN, u ^{\flat} (\M));
\\
\notag
\R \mathcal{H}om _{\D _X ^{(m)}}
(u _+ (\cE), \cF)
\riso 
u _* \R \mathcal{H}om _{\D _Z ^{(m)} }
(\cE, u ^{!} (\cF)).
\end{gather}

\end{coro}

\begin{proof}
Taking a K-injective resolution of 
$\M$ (see 13.33.5 of the stack project), 
the first isomorphism is a consequence of \ref{u+uflat-lem}.2--3.
This yields the second one by using \ref{fund-isom-thmpre}
and \ref{u+D-left2right}.
\end{proof}

\begin{coro}
\label{rel-dual-isom-imm}
Let 
$\cN\in 
D ^{\mathrm{b}} _{\mathrm{coh}}({} ^* \D _Z ^{(m)} )$.
We have the isomorphism of $D ^{\mathrm{b}} _{\mathrm{coh}}({} ^* \D _X ^{(m)} )$:
\begin{equation}
\label{rel-dual-isom-imm1}
\DD ^{(m)} \circ u ^{(m)} _+ (\cN)
\riso 
 u ^{(m)} _+ \circ \DD ^{(m)} (\cN).
\end{equation}
\end{coro}

\begin{proof}
By using 
\ref{u+D-left2right}; we reduce to the case $* = r$.
In this case, the isomorphism \ref{rel-dual-isom-imm1} 
is the composition of the following  isomorphisms :
\begin{gather}
\notag
\R \mathcal{H}om _{\D _X ^{(m)}}
(u _+ (\cN),\omega _{X} \otimes _{\O _X}  \D _X ^{(m)}) [d _{X}]
\overset{\ref{Radj-u+flat}}{\riso} 
u _* \R \mathcal{H}om _{\D _Z ^{(m)} }
(\cN, u ^{\flat} ( \omega _{X} \otimes _{\O _X}  \D _X ^{(m)}) )
[d _{X}]
\overset{\ref{fund-isom2-corpre}}{\riso} 
\\
\notag
u _* \R \mathcal{H}om _{\D _Z ^{(m)} }
(\cN, \omega _{Z} \otimes _{\O _Z} \D _{Z\to X} ^{(m)} ) 
[d _{Z}]
\underset{\cite[2.1.17]{caro_comparaison}}{\riso} 
u _* 
\left ( 
\R \mathcal{H}om _{\D _Z ^{(m)} }
(\cN,  \omega _{Z} \otimes _{\O _Z} \D _{Z} ^{(m)} 
[d _{Z}]
)
\otimes _{\D _{Z} ^{(m)} }
 \D _{Z\to X} ^{(m)} 
 \right ).
\end{gather}
\end{proof}

\subsection{Relative duality isomorphism and adjunction for a closed immersion of formal schemes}

Let $u \colon \ZZ \hookrightarrow \X$ be a closed immersion of 
smooth formal $\V$-schemes.
Let $\I$ be the ideal defining $u$. 
The level $m\in \N$ is fixed.
In this subsection, by 
the letter $\widetilde{\D} $ we mean
$\widehat{\D}  ^{(m)}$ or respectively
$\widehat{\D} ^{\dag } \otimes _\bbZ \bbQ$.
For instance, 
$\widetilde{\D} _{\X/\fS} $ is 
$\widehat{\D} ^{(m)} _{\X/\fS} $ 
(resp. 
$\D ^{\dag} _{\X/\fS,\Q} $).

\begin{empt}
[Local description]
\label{locdesc-climm-form}
Suppose
$\fX$ is affine and there exist
$t _{r +1},\dots , t _{d}  \in \Gamma (\fX ,\I)$
generating 
$I:=\Gamma (\fX \,,\I)$,
$t _{1},\dots , t _{r}\in \Gamma (  \fX,\O _{ \fX})$
such that
$t _{1},\dots ,t  _{d}$ are local coordinates of $ \fX$ over $S$,
$\overline{t} _{1},\dots ,\overline{t} _{r}$ 
are local coordinates of $ \fZ$ over $\fS$,
and 
$\overline{t} _{r +1},\dots ,\overline{t} _{d}$ is a basis of $\I /\I ^2$,
where $\overline{t} _{1},\dots , \overline{t} _{r} \in\Gamma (  \fZ ,\O _{ \fZ})$
(resp. $\overline{t} _{r +1},\dots ,\overline{t} _{d}\in\Gamma (  \fX ,\cI /\cI ^2)$)
are the images of 
$t _{1},\dots , t _{r}$
(resp. $t _{r+1},\dots , t _{d}$)
via 
$\Gamma (  \fX,\O _{ \fX})
\to
\Gamma (  \fZ ,\O _{ \fZ})$
(resp. 
$\Gamma (  \fX,\cI)
\to 
\Gamma (  \fX ,\cI /\cI ^2)$).

We denote by 
$\tau _i := 1 \otimes t _i -t _i \otimes 1$, 
$\overline{\tau} _j := 1 \otimes \overline{t} _j -\overline{t} _j \otimes 1$, 
for any $i= 1,\dots, d$, $j= 1,\dots, r$.
The sheaf of $\O _ \fX$-algebras
$\cP ^n  _{ \fX/\fS,(m)}$ is 
a free $\O _ \fX$-module with the basis 
$\{ \underline{\tau} ^{\{\underline{k}\} _{(m)}}\ | \ \underline{k}\in \N ^d \text{ such that } | \underline{k}| \leq n\} $,
and
$\cP ^n  _{ \fZ/\fS,(m)}$ is 
a free $\O _ \fZ$-module with the basis 
$\{ \underline{\overline{\tau}} ^{\{\underline{l}\} _{(m)}}\ | \ \underline{l}\in \N ^r \text{ such that } | \underline{l}| \leq n\} $.
We denote by 
$\{ 
\underline{\partial} ^{<\underline{k}> _{(m)}}
\ | \ \underline{k}\in \N ^d,\ | \underline{k} | \leq n
\}$
the corresponding dual basis  of 
$\D ^{(m)} _{ \fX/\fS,n}$ 
and
by 
$\{ \underline{\partial} ^{<\underline{l}> _{(m)}}\  
| \ \underline{l}\in \N ^r,\ | \underline{l} | \leq n\} $
the corresponding dual basis of $\D ^{(m)} _{ \fZ/ \fS,n}$ (if there is no possible confusion).
The sheaf
$\D ^{(m)} _{ \fX/ \fS}$ is 
a free $\O _ \fX$-module with the basis 
$\{ 
\underline{\partial} ^{<\underline{k}> _{(m)}}
\ | \ \underline{k}\in \N ^d
\}$,
and
$\D ^{(m)} _{ \fZ/ \fS}$ is 
a free $\O _ \fZ$-module with the basis 
$\{ \underline{\partial} ^{<\underline{l}> _{(m)}}\ | \ \underline{l}\in \N ^r\} $.

a) We compute the canonical homomorphism
$u ^* \cP ^n  _{ \fX/ \fS,(m)}
\to \cP ^n  _{ \fZ/ \fS,(m)}$
sends 
$\underline{\tau} ^{\{(\underline{l}, \underline{h})\} _{(m)}}$
where 
$\underline{l} \in \N ^r$
and 
$\underline{h} \in \N ^{d-r}$
to
$\underline{\overline{\tau}} ^{\{\underline{l}\} _{(m)}}$
if $\underline{h} = (0,\dots,0)$
and to $0$ otherwise.

b) We denote by 
$\theta \colon 
\cD ^{(m)} _{ \fZ/ \fS}
\to 
\cD ^{(m)} _{ \fZ \to  \fX/ \fS}$
the canonical homomorphism of left $\cD ^{(m)} _{ \fZ/ \fS}$-modules
(which is built by duality from the 
canonical homomorphisms $u ^* \cP ^n  _{ \fX/ \fS,(m)}
\to \cP ^n  _{ \fZ/ \fS,(m)}$).
For any $P \in D ^{(m)} _{ \fX/ \fS}$, we denote by 
$\overline{P}$ its image via
the canonical morphism of left $D ^{(m)} _{ \fX/ \fS}$-modules
$D ^{(m)} _{ \fX/ \fS}
\to 
D ^{(m)} _{ \fX/ \fS} / I D ^{(m)} _{ \fX/ \fS}
= 
D ^{(m)} _{ \fZ \to  \fX/ \fS}$.
We set 
$\underline{\xi} ^{<\underline{k}> _{(m)}}:= 
\overline{\underline{\partial} ^{<\underline{k}> _{(m)}}}$. 
By duality from a), we compute
$\theta (\underline{\partial} ^{<\underline{l}> _{(m)}})
=
\underline{\xi} ^{<(\underline{l}, \underline{0})> _{(m)}}$,
for any 
$\underline{l}\in \N ^r$.

\end{empt}

\begin{empt}
\label{locdesc-climm2-form}
Suppose we are in the local situation of \ref{locdesc-climm-form}. 
We denote by 
$\cD ^{(m)} _{ \fX, \fZ,\underline{t}/ \fS}$ the subring of 
$\cD ^{(m)} _{ \fX/ \fS}$ which is a 
 free $\O _ \fX$-module with the basis 
$\{ \underline{\partial} ^{<(\underline{l}, \underline{0})> _{(m)}}\ | \ \underline{l}\in \N ^r\} $, 
where 
$\underline{0}:=(0,\dots, 0) \in 
\N ^{d-r}$.
If there is no ambiguity concerning the local coordinates (resp. and $\fS$),
we might simply denote $\cD ^{(m)} _{ \fX, \fZ,\underline{t}/ \fS}$ by 
$\cD ^{(m)} _{ \fX, \fZ/ \fS}$
(resp. $\cD ^{(m)} _{ \fX, \fZ}$).
The properties of \ref{locdesc-climm2} are still valid in the context of formal schemes : 
\begin{enumerate}[(a)]
\item  Since 
$t _{r +1},\dots , t _{d}$ generate $I$ and are in the center of 
$D ^{(m)} _{ \fX, \fZ,\underline{t}/ \fS}$, then we compute 
$\cI\cD ^{(m)} _{ \fX, \fZ,\underline{t}/ \fS} = \cD ^{(m)} _{ \fX, \fZ,\underline{t}/ \fS}\cI$.
Hence, we get a canonical ring structure on 
$\cD ^{(m)} _{ \fX, \fZ,\underline{t}/ \fS}/\cI\cD ^{(m)} _{ \fX, \fZ,\underline{t}/ \fS}$
induced by that of 
$\cD ^{(m)} _{ \fX, \fZ,\underline{t}/ \fS}$.
We have the following factorization
\begin{equation}
\label{diag-DZ2X-t-commf}
\xymatrix{
{ \cD ^{(m)} _{ \fX, \fZ,\underline{t}/ \fS}/\cI\cD ^{(m)} _{ \fX, \fZ,\underline{t}/ \fS}} 
\ar@{^{(}->}[r] ^-{}
& 
{\cD ^{(m)} _{ \fX/ \fS} / \cI \cD ^{(m)} _{ \fX/ \fS}} 
\\ 
{u _* \cD ^{(m)} _{ \fZ/ \fS}} 
\ar[u] ^-{\sim} _-{\theta}
\ar[r] ^-{\theta}
& 
{u _* \cD ^{(m)} _{ \fZ \to  \fX/ \fS},} 
\ar@{=}[u] ^-{}
}
\end{equation}
where both horizontal morphisms are the canonical ones
and where 
the vertical arrow
$u _* \cD ^{(m)} _{ \fZ/ \fS}
\to
 \cD ^{(m)} _{ \fX, \fZ,\underline{t}/ \fS}/\cI\cD ^{(m)} _{ \fX, \fZ,\underline{t}/ \fS}$
 is a bijection that we still denote by $\theta$. 
 
 \item 
Let $P \in D ^{(m)} _{ \fX/ \fS}$, $Q \in D ^{(m)} _{ \fZ}$. 
Choose $Q _ \fX\in D ^{(m)} _{ \fX, \fZ,\underline{t}/ \fS}$ 
such that  $\overline{Q _ \fX}= \theta (Q)$.
We have :
\begin{equation}
\label{eq-QXPf}
Q \cdot \overline{P} 
= 
Q \cdot \overline{1} \cdot P
=
\overline{Q _ \fX} \cdot P
 =
\overline{Q _ \fX P}.
\end{equation}
Using \ref{eq-QXPf}, we check that 
$\cD ^{(m)} _{ \fX, \fZ,\underline{t}/ \fS}/\cI\cD ^{(m)} _{ \fX, \fZ,\underline{t}/ \fS} $
is 
a $(u _*\cD ^{(m)} _{ \fZ/ \fS} , \D ^{(m)} _{ \fX, \fZ,\underline{t}/ \fS})$-subbimodule
of $\cD ^{(m)} _{ \fX/ \fS} / \cI \cD ^{(m)} _{ \fX/ \fS}$
and 
that 
$\theta 
\colon u _* \cD ^{(m)} _{ \fZ/ \fS}
\to
 \cD ^{(m)} _{ \fX, \fZ,\underline{t}/ \fS}/\cI\cD ^{(m)} _{ \fX, \fZ,\underline{t}/ \fS}$
an isomorphism of 
left $u _* \cD ^{(m)} _{ \fZ/ \fS}$-modules.
Finally, 
we compute
$\theta 
\colon u _* \cD ^{(m)} _{ \fZ/ \fS}
\riso
 \cD ^{(m)} _{ \fX, \fZ,\underline{t}/ \fS}/\cI\cD ^{(m)} _{ \fX, \fZ,\underline{t}/ \fS}$
is also an isomorphism of rings.

\item \label{gf}
Both rings 
$u _* \cD ^{(m)} _{ \fZ/ \fS}$ and 
$\cD ^{(m)} _{ \fX, \fZ,\underline{t}/ \fS}$ are 
$\O _ \fX$-rings (i.e. they are rings endowed with a structural 
homomorphism of rings 
$\O _ \fX \to \cD ^{(m)} _{ \fX, \fZ,\underline{t}/ \fS}$
and 
$\O _ \fX \to u _* \cD ^{(m)} _{ \fZ/ \fS}$).
This yields the composite 
\begin{equation}
\label{morp-rhof}
\rho\colon \cD ^{(m)} _{ \fX, \fZ,\underline{t}/ \fS}
\to
\cD ^{(m)} _{ \fX, \fZ,\underline{t}/ \fS}
/\I \cD ^{(m)} _{ \fX, \fZ,\underline{t}/ \fS}
\underset{\theta}{\liso}  
u _* \cD ^{(m)} _{ \fZ/ \fS}
\end{equation}
is an homomorphism of $\O _ \fX$-rings.
We get the isomorphism
$\theta \colon \rho _* u _* \cD ^{(m)} _{ \fZ/ \fS}
\riso
 \cD ^{(m)} _{ \fX, \fZ,\underline{t}/ \fS}/\cI\cD ^{(m)} _{ \fX, \fZ,\underline{t}/ \fS}$
 of 
$\cD ^{(m)} _{ \fX, \fZ,\underline{t}/ \fS}$-bimodules.

Moreover, the isomorphism
$\theta \colon u _* \cD ^{(m)} _{ \fZ/ \fS}
\riso
 \cD ^{(m)} _{ \fX, \fZ,\underline{t}/ \fS}/\cI\cD ^{(m)} _{ \fX, \fZ,\underline{t}/ \fS}$
is also an isomorphism of 
$( u _* \cD ^{(m)} _{ \fZ/ \fS}, \cD ^{(m)} _{ \fX, \fZ,\underline{t}/ \fS})$-bimodules
such the structure of left $u _* \cD ^{(m)} _{ \fZ/ \fS}$-module on 
$u _* \D ^{(m)} _{ \fZ}$ is the canonical one and 
the structure of right $\cD ^{(m)} _{ \fX, \fZ,\underline{t}/ \fS}$-module on 
$\D ^{(m)} _{ \fX, \fZ,\underline{t}/ \fS}/\cI \D ^{(m)} _{ \fX, \fZ,\underline{t}/ \fS}$ is the canonical one.

\item  
The sheaf $\cD ^{(m)} _{ \fX/ \fS} $
is a free left $\cD ^{(m)} _{ \fX, \fZ,\underline{t}/ \fS}$-module
with the basis
$\{ \underline{\partial} ^{<(\underline{0},\underline{h})> _{(m)}}\ | \ \underline{h}\in \N ^{d-r}\} $,
where 
$\underline{0}:=(0,\dots, 0) \in 
\N ^{r}$,
the sheaf $\cD ^{(m)} _{ \fZ \to  \fX/ \fS}$ is a free 
left 
$\cD ^{(m)} _{ \fZ/ \fS}$-module 
with the basis
$\{ \underline{\xi} ^{<(\underline{0},\underline{h})> _{(m)}}\ | \ \underline{h}\in \N ^{d-r}\} $,
where 
$\underline{0}:=(0,\dots, 0) \in 
\N ^{r}$.

\item We have the transposition automorphism
${} ^t \colon 
D ^{(m)} _{ \fX/ \fS}
\to 
D ^{(m)} _{ \fX/ \fS}$
given by 
$P 
= 
\sum _{\underline{k} \in \N ^d}
a _{\underline{k}} \underline{\partial} ^{<\underline{k}> _{(m)}}
\mapsto
{} ^t P:= 
\sum _{\underline{k} \in \N ^d}
(-1) ^{| \underline{k}|}\underline{\partial} ^{<\underline{k}> _{(m)}}
a _{\underline{k}} $.
This induces the automorphisms
${} ^t \colon 
D ^{(m)} _{ \fX, \fZ,\underline{t}/ \fS}
\to 
D ^{(m)} _{ \fX, \fZ,\underline{t}/ \fS}$
and
${} ^t \colon 
D ^{(m)} _{ \fX, \fZ,\underline{t}/ \fS}/I D ^{(m)} _{ \fX, \fZ,\underline{t}/ \fS}
\to 
D ^{(m)} _{ \fX, \fZ,\underline{t}/ \fS}
/
ID ^{(m)} _{ \fX, \fZ,\underline{t}/ \fS}$.
On the other hand, 
via the local coordinates $\overline{t} _{1},\dots ,\overline{t} _{r}$ 
of $ \fZ$ over $ \fS$,
we get
the transposition automorphism
${} ^t \colon 
D ^{(m)} _{ \fZ/ \fS}
\to 
D ^{(m)} _{ \fZ/ \fS}$
given by 
$Q 
= 
\sum _{\underline{k} \in \N ^r}
b _{\underline{k}} \underline{\partial} ^{<\underline{k}> _{(m)}}
\mapsto
{} ^t Q:= 
\sum _{\underline{k} \in \N ^r}
(-1) ^{| \underline{k}|}\underline{\partial} ^{<\underline{k}> _{(m)}}
b _{\underline{k}} $.
We compute the following diagram 
\begin{equation}
\label{transp-closedimmersion-form}
\xymatrix{
{D ^{(m)} _{ \fX, \fZ,\underline{t}/ \fS}/I D ^{(m)} _{ \fX, \fZ,\underline{t}/ \fS}} 
\ar[r] _-{\sim} ^-{{}^t}
& 
{D ^{(m)} _{ \fX, \fZ,\underline{t}/ \fS}/I D ^{(m)} _{ \fX, \fZ,\underline{t}/ \fS}} 
\\ 
{D ^{(m)} _{ \fZ/ \fS}} 
\ar[u] ^-{\sim} _-{\theta}
\ar[r] _-{\sim} ^-{{}^t}
& 
{D ^{(m)} _{ \fZ/ \fS}} 
\ar[u] ^-{\sim} _-{\theta}
}
\end{equation}
is commutative.
\end{enumerate}

\end{empt}

\begin{empt}
\label{form-dfn-Dmstruuflat}
If $\M$ is a right $\D _\X ^{(m)}$-module, we denote by 
$u ^{\flat 0} (\M)
:=
u ^{-1} \mathcal{H}om _{\O _\X} ( u _* \O _{\ZZ}, \M)$.
Using $m$-PD-costratifications, 
we get similarly to \ref{rightD-moduleflat} a structure of right $\D _\ZZ ^{(m)}$-module 
on $u ^{\flat 0} (\M)$.
Since $\D _\X ^{(m)}$ is a flat $\O _\X$-module, then an injective 
right $\D _\X ^{(m)}$-module is an injective $\O_\X$-modules. 
Hence, taking an injective resolution of a complex 
of $D ^{+} ( {} ^r \D _\X ^{(m)})$, 
we check the functor $u ^\flat$ sends
$D ^{+}  ( {} ^r \D _\X ^{(m)})$
to
$D ^{+} ({} ^r \D _\ZZ ^{(m)} )$.

\end{empt}

\begin{empt}
[Local description of $u ^{\flat}$]
\label{locadesc-uflat-form}
Suppose we are in the local situation of \ref{locdesc-climm-form}. 
Let $\M$ be a right $\D _{\X/\fS} ^{(m)}$-module.
Let $x \in \Gamma (\fZ,  u ^{\flat 0} (\M))$ and 
$Q \in D ^{(m)} _{\fZ}$. 
Similarly to 
\ref{locadesc-uflat1},
for any 
$Q _\fX\in D ^{(m)} _{\fX,\fZ,\underline{t}/\fS}$ 
such that  $\theta (Q ) = \overline{Q _\fX}$,
we compute
\begin{equation}
\label{locadesc-uflat1-formula-form}
\mathrm{ev} _1
(x  \cdot Q)
= 
\mathrm{ev} _1 (x)  \cdot Q _\fX.
\end{equation}

\end{empt}

\begin{empt}
\label{rem-rightD-moduleflat-form}
Let $\M$ be a right $\widetilde{\D} _\X $-module. 

\begin{enumerate}
\item There is a canonical way to endow 
$u ^{\flat 0} (\M)$
with a structure of 
right $\widetilde{\D} _\fZ$-module.
Indeed, suppose $\fX$ affine. 
Let $x \in \Gamma (\fZ,  u ^{\flat 0} (\M))$ and 
$Q \in \widetilde{D} _{\fZ}$. 
For any 
$Q _\fX\in \widetilde{D} _{\fX}$ 
such that  $\theta (Q ) = \overline{Q _\fX}$,
we define $x  \cdot Q$ so that we get the equality
\begin{equation}
\label{rem-locadesc-uflat1-formula-form}
\mathrm{ev} _1
(x  \cdot Q)
: = 
\mathrm{ev} _1 (x)  \cdot Q _\fX,
\end{equation}
where $\mathrm{ev} _1 \colon \Gamma ( \fZ, u ^{\flat 0} (\M) )
\hookrightarrow \Gamma (\fX, \M)$ is 
the evaluation at $1$ homomorphism (which is injective).
Since $I$ annihilates $\mathrm{ev} _1 (x) $,
we remark that this is well defined. 

\item 
When $\widetilde{\D} _\X = \D _\fX ^{(m)}$,
we had seen another way to endow $u ^{\flat 0} (\M)$ with a canonical structure of
right $\D _\fZ ^{(m)}$-module
 (see \ref{form-dfn-Dmstruuflat}).
It follows from the computation  \ref{locadesc-uflat1-formula-form}
that both structures are identical.
\end{enumerate}

\end{empt}

\begin{empt}
\label{dfntheta-rhotilde}
Suppose we are in the local situation of \ref{locdesc-climm-form}.
We keep notation \ref{locdesc-climm-form}
and \ref{locdesc-climm2-form}.

\begin{enumerate}\item A section of the sheaf
$\widehat{\D} ^{(m)} _{\X/\S}$ 
can uniquely be written in the form 
$\sum 
_{\underline{k}\in \N ^d}
a _{\underline{k}}
\underline{\partial} ^{<\underline{k}> _{(m)}}$
such that 
$a _{\underline{k}} \in \O _\X$
converges to $0$ when 
$| \underline{k}|\to \infty$.
A section of the sheaf
$\widehat{\D} ^{(m)} _{\fZ/\fS}$ 
can uniquely be written in the form 
$\sum 
_{\underline{l}\in \N ^r}
b _{\underline{l}}
\underline{\partial} ^{<\underline{l}> _{(m)}}$
such that 
$b _{\underline{l}} \in \O _\fZ$
converges to $0$ when 
$| \underline{l}|\to \infty$.
Let 
$\widehat{\D} ^{(m)} _{\X,\fZ,\underline{t}}$ be
the $p$-adic completion of
$\D ^{(m)} _{\X,\fZ,\underline{t}}$.
Then $\widehat{\cD} ^{(m)} _{\X,\fZ,\underline{t}}$
is a subring of 
$\widehat{\cD} ^{(m)} _{\X/\S}$ whose elements can uniquely be written in the form
$\sum 
_{\underline{l}\in \N ^r}
a _{\underline{l}}
\underline{\partial} ^{<(\underline{l}, \underline{0})> _{(m)}}$
(recall $\underline{0}:=(0,\dots, 0) \in 
\N ^{d-r}$)
where 
$a _{\underline{l}} \in \O _\fX$
converges to $0$ when 
$| \underline{l}|\to \infty$.
Taking the $p$-adic completion of the diagram \ref{diag-DZ2X-t-commf},
we get the canonical diagram
\begin{equation}
\label{diag-DZ2X-t-comm-form}
\xymatrix{
{\widehat{\cD} ^{(m)} _{\fX,\fZ,\underline{t}/\fS}/\cI\widehat{\cD} ^{(m)} _{\fX,\fZ,\underline{t}/\fS}} 
\ar@{^{(}->}[r] ^-{}
& 
{\widehat{\cD} ^{(m)} _{\fX/\fS} / \cI \widehat{\cD} ^{(m)} _{\fX/\fS}} 
\\ 
{u _* \widehat{\cD} ^{(m)} _{\fZ/\fS}} 
\ar[u] ^-{\sim} _-{\widehat{\theta}}
\ar[r] ^-{\widehat{\theta}}
& 
{u _* \widehat{\cD} ^{(m)} _{\fZ \to \fX/\fS}} 
\ar@{=}[u] 
}
\end{equation}
where 
$\widehat{\theta}
\colon 
u _* \widehat{\cD} ^{(m)} _{\ZZ/\S}
\riso 
\widehat{\cD} ^{(m)} _{\X,\fZ,\underline{t}}
/\cI\widehat{\cD} ^{(m)} _{\X,\fZ,\underline{t}}
$
is an isomorphism of $\V$-algebras.

\item We set 
$\D ^{\dag} _{\X,\fZ,\underline{t},\Q}
:=
\underrightarrow{\lim}
\widehat{\D} ^{(m)} _{\X,\fZ,\underline{t},\bbQ}$.
We get a similar diagram than \ref{diag-DZ2X-t-comm-form}
by replacing $\widehat{\cD} ^{(m)}$ with $\D ^{\dag}$
and by adding some $\bbQ$.

\item The isomorphism of $\cV$-algebras 
$u _* \widetilde{\cD} _{ \fZ/ \fS}
\riso 
\widetilde{\cD} _{ \fX, \fZ,\underline{t}/ \fS}
/\I \widetilde{\cD} _{ \fX, \fZ,\underline{t}/ \fS}$
induced by $\theta$ will be denoted by 
$\widetilde{\theta}$.
This yields by composition the  homomorphism of $\O _ \fX$-rings :
\begin{equation}
\label{morp-rhof-tilde}
\widetilde{\rho} 
\colon
\widetilde{\cD}  _{ \fX, \fZ,\underline{t}/ \fS}
\to
\widetilde{\cD} _{ \fX, \fZ,\underline{t}/ \fS}
/\I \widetilde{\cD} _{ \fX, \fZ,\underline{t}/ \fS}
\underset{\widetilde{\theta}}{\liso}  
u _* \widetilde{\cD} _{ \fZ/ \fS}.
\end{equation}

\end{enumerate}

\end{empt}

\begin{empt}
Taking the $p$-adic completion of  $u _* \O _\fZ \otimes _{\O _X}\D ^{(m)} _{\fX, \fZ ,\underline{t}} 
\riso 
\D ^{(m)} _{\fX, \fZ ,\underline{t}} /\cI \D ^{(m)} _{\fX, \fZ ,\underline{t}} $
(and taking the inductive limit on the level if necessary),
we get the isomorphism
$u _* \O _\fZ \otimes _{\O _X}
\widetilde{\cD} _{\fX, \fZ ,\underline{t}} 
\riso 
\widetilde{\cD} _{\fX, \fZ ,\underline{t}} 
/
\cI \widetilde{\cD} _{\fX, \fZ ,\underline{t}} $.
Using \ref{rem-locadesc-uflat1-formula-form}, 
we compute the canonical isomorphism
\begin{equation}
\label{locadesc-uflat2-iso2pre-form-tilde}
\widetilde{\rho} _* u ^{\flat 0} (\M)
\riso 
u ^{-1}
\mathcal{H}om _{\widetilde{\cD} _{\fX, \fZ ,\underline{t}} } ( \widetilde{\cD} _{\fX, \fZ ,\underline{t}} /\cI \widetilde{\cD} _{\fX, \fZ ,\underline{t}}, \M)
\end{equation}
is an isomorphism of $\widetilde{\cD} _{\fX,\fZ,\underline{t}/\fS} $-modules .

Similarly to 
\ref{locadesc-uflat2-iso2},
via the isomorphism 
$\widetilde{\theta}
\colon u _* \widetilde{\cD} _{ \fZ/ \fS}
\riso 
\widetilde{\cD} _{ \fX, \fZ,\underline{t}/ \fS}
/\I \widetilde{\cD} _{ \fX, \fZ,\underline{t}/ \fS}$
 (see \ref{dfntheta-rhotilde}) 
 we get a structure of 
right $\widetilde{\cD} _{ \fZ }$-module 
on 
$u ^{-1}
\mathcal{H}om _{\widetilde{\cD} _{\fX, \fZ ,\underline{t}} } ( u _* \widetilde{\cD} _{ \fZ }, \M)$
and the isomorphism
\begin{equation}
\label{locadesc-uflat-form3-iso1preform-tilde}
u ^{\flat 0} (\M)
\riso  u ^{-1}
\mathcal{H}om _{\widetilde{\cD} _{\fX, \fZ ,\underline{t}} } ( u _* \widetilde{\cD} _{ \fZ }, \M)
\end{equation}
of
$\widetilde{\cD} _{ \fZ }$-modules.
If there is no ambiguity, 
we can avoid writing $u ^{-1}$ pr $u _*$, 
e.g. we can simply write
$u ^{\flat 0} (\M)
= 
\mathcal{H}om _{\widetilde{\cD} _{\fX, \fZ ,\underline{t}} } ( \widetilde{\cD} _{ \fZ }, \M)
$.

\end{empt}

\begin{empt}
\label{hat-uflat-desc}
Suppose we are in the local situation of \ref{locdesc-climm-form}. Let $\M$ be a right $\widetilde{\D} _{\X/\fS} $-module.
\begin{enumerate}

\item Let $s := d -r$, and $f _1= t _{r+1},\dots, f _s := t _{d}$. 
Let $K _{\bullet} (\underline{f})$ be the Koszul complex
of $\underline{f}=(f_1,\dots,f_s)$. 
Since $f _1,\dots, f _s$ is a regular sequence of $\I$, 
then the canonical morphism 
$K _{\bullet} (\underline{f}) \to \O _{\X}/\I$
(given by the 
$K _{0} (\underline{f}) = \O _{\X} \to \O _{\X}/\I$)
is a quasi-isomorphism.
Since $f _1,\dots f _s$ are in the center of $ \widetilde{\D} _{\fX, \fZ ,\underline{t}} $, and 
$\widetilde{\D} _{\fX, \fZ ,\underline{t}}$ is a flat $\O _{\X}$-algebra, then
we get the resolution of $ \widetilde{\D} _{\fX, \fZ ,\underline{t}} /\I \widetilde{\D} _{\fX, \fZ ,\underline{t}}$ 
by free $ \widetilde{\D} _{\fX, \fZ ,\underline{t}}$-bimodules of finite type
$ \widetilde{\D} _{\fX, \fZ ,\underline{t}} \otimes _{\O _{\X}} K _{\bullet} (\underline{f}) 
 \riso 
 \widetilde{\D} _{\fX, \fZ ,\underline{t}} /\I \widetilde{\D} _{\fX, \fZ ,\underline{t}}$.
 Using \ref{locadesc-uflat2-iso2pre-form-tilde} (and a avatar diagram of \ref{locadesc-uflat-iso}), this yields this isomorphism 
 of $D ^{\mathrm{b}} ( \widetilde{\D} _{\fX, \fZ ,\underline{t}})$: 
\begin{equation}
\label{locadesc-uflat-iso-form}
\phi _{\underline{t}}
\colon 
u ^{\flat} (\M)
\riso 
\mathcal{H}om _{ \widetilde{\D} _{\fX, \fZ ,\underline{t}} } (  \widetilde{\D} _{\fX, \fZ ,\underline{t}} \otimes _{\O _{\X}} K _{\bullet} (\underline{f}) , \M).
\end{equation}
Hence, we get the  isomorphism of right $ \widetilde{\D} _{\fX, \fZ ,\underline{t}}$-modules
\begin{equation}
\label{isos-fund-isom1pre-form}
\phi ^s _{\underline{t}} 
=
\mathcal{H} ^s(\phi _{\underline{f}})
\colon 
R ^s u ^{\flat 0} (\M)
\riso 
\mathcal{H} ^s \mathcal{H}om _{ \widetilde{\D} _{\fX, \fZ ,\underline{t}} } (  \widetilde{\D} _{\fX, \fZ ,\underline{t}} \otimes _{\O _{\X}} K _{\bullet} (\underline{f}) , \M).
\end{equation}
From the canonical morphism of complex of right $ \widetilde{\D} _{\fX, \fZ ,\underline{t}}$-modules 
$\mathcal{H}om _{ \widetilde{\D} _{\fX, \fZ ,\underline{t}} } (  \widetilde{\D} _{\fX, \fZ ,\underline{t}} \otimes _{\O _{\X}} K _{\bullet} (\underline{f}) , \M)
\to \M / \I \M$,
we get the isomorphism of right $ \widetilde{\D} _{\fX, \fZ ,\underline{t}}$-modules
\begin{equation}
\label{isos-fund-isom2pre-form}
\mathcal{H} ^s \mathcal{H}om _{ \widetilde{\D} _{\fX, \fZ ,\underline{t}} } (  \widetilde{\D} _{\fX, \fZ ,\underline{t}} \otimes _{\O _{\X}} K _{\bullet} (\underline{f}) , {}  \M )
\riso  \M / \I \M.
\end{equation}

\end{enumerate}

\end{empt}

\begin{ntn}
\label{ntnu*form}
If $\cE$ is a left $\D _\X ^{(m)}$-module, we set
$u ^{*} (\cE)
:=
 \O _{\fZ}
 \otimes _{ u ^{-1}\O _\X}
 u ^{-1}\cE$.
Using $m$-PD-stratifications, 
we get a structure of left $\D _\fZ ^{(m)}$-module 
on $u ^{*} (\cE)$.
This yields the functor
$\L u ^* \colon 
D ^{+}  ( {} ^l \D _\X ^{(m)}) 
\to
D ^{+} ({} ^l \D _\ZZ ^{(m)} )$
(resp.
$\L u ^* \colon 
D  ( {} ^l \D _\X ^{(m)}) 
\to
D ({} ^l \D _\ZZ ^{(m)} )$).
Similarly, we get the functor
$\L u ^* \colon 
D ^{+}  ( {} ^l \widetilde{\D} _\X ) 
\to
D ^{+} ({} ^l \widetilde{\D} _\ZZ )$
(resp. $\L u ^* \colon 
D  ( {} ^l \widetilde{\D} _\X ) 
\to
D ({} ^l \widetilde{\D} _\ZZ )$)
defined by setting 
\begin{equation}
\label{ntnu*form-dfn1}
\L u ^{*} (\cE)
:=
\widetilde{\D} _{\fZ \to \fX}
 \otimes ^\L _{ u ^{-1}\widetilde{\D} _\X}
 u ^{-1}\cE.
\end{equation}
Finally, we set
$u ^! (\cE) := \L u ^{*} (\E ) [d _{Z/X}]$.

Suppose we are in the local situation of \ref{locdesc-climm-form}. 
Let $\E \in D   ( {} ^l  \widetilde{\D} _\X) $.
We check the canonical homomorphism 
$$
 \widetilde{\D} _{ \fZ } \otimes ^\L _{u ^{-1} \widetilde{\D} _{\fX, \fZ ,\underline{t}}} u ^{-1}\E
 \to
 \L u ^* (\E) $$
 is an isomorphism of $D ( {} ^l  \widetilde{\D} _\fZ) $.
This yields
$$\L u ^* (\E) \riso ( \widetilde{\D} _{\fX, \fZ ,\underline{t}} \otimes _{\O _{\X}} K _{\bullet} (\underline{f}) ) 
\otimes _{ u ^{-1}\widetilde{\D} _{\fX, \fZ ,\underline{t}} } u ^{-1}\E.$$

\end{ntn}

\begin{prop}
\label{fund-isom-form}
Let $\E$ be a left $ \widetilde{\D} _{\X}$-module (resp. a $ \widetilde{\D} _{\X}$-bimodule). 
Set $n := -d _{Z/X} \in \N$.
We have the canonical isomorphism of right $ \widetilde{\D} _\fZ $-modules (resp. 
of right $(  \widetilde{\D} _{\fZ}, u ^{-1} \widetilde{\D} _{\X})$-bimodules):
\begin{equation}
\label{fund-isom1-form}
R ^{n} u ^{\flat 0} ( \omega _{\X} \otimes _{\O _{\X}} \E)
\riso 
\omega _{\fZ} \otimes _{\O _{\fZ}} u ^{*} (\E ).
\end{equation}
\end{prop}

\begin{proof}
Using \ref{hat-uflat-desc} and \ref{ntnu*form}, we proceed as \ref{fund-isom}.
\end{proof}

\begin{rem}
Since complexes are not coherent, 
Proposition \ref{fund-isom-form} is not a consequence of Theorem \ref{fund-isom}
(but the check is similar).
\end{rem}

\begin{thm}
\label{fund-isom-thm-form}
Let $\E \in D ^{+} ( {} ^l \widetilde{\D} _\fX )$ (resp. $\E \in D ^{+} ( {} ^l \widetilde{\D} _\fX , {} ^r \widetilde{\D} _\fX )$). 
With notation \ref{ntnu*form}, 
we have the canonical isomorphism of 
$D ({} ^r \widetilde{\D} _\fZ )$
(resp. $D ({} ^r \widetilde{\D} _\fZ ,{} ^r u ^{-1}\widetilde{\D} _{\fX} )$)
of the form
\begin{equation}
\label{fund-isom2}
\omega _{\ZZ} \otimes _{\O _\ZZ} \L u ^{!} (\E ) 
\riso
u ^{\flat} ( \omega _{\X} \otimes _{\O _\X} \E).
\end{equation}
\end{thm}

\begin{proof}
Using
\cite[I.7.4]{HaRD},
this is a consequence of \ref{fund-isom-form}.
\end{proof}

\begin{coro}
We have the canonical isomorphism of right $( \widetilde{\D} _{\fZ} , u ^{-1}\widetilde{\D} _{\fX} ) $-bimodules of the form
\begin{equation}
\label{fund-isom2-corpref}
\omega _{\fZ} \otimes _{\O _\fZ} \widetilde{\D} _{\fZ\to \fX}  
\riso
u ^{\flat} ( \omega _{\fX} \otimes _{\O _\fX}  \widetilde{\D} _\fX ) [-d _{Z/X}].
\end{equation}
\end{coro}

\begin{proof}
We apply Theorem \ref{fund-isom-thm-form} in the case
$\cE= \widetilde{\D} _X $.
\end{proof}

\begin{ntn}
We get the functor
$u  _+\colon 
D  ( {} ^*  \widetilde{\D} _\fZ ) 
\to
D  ({} ^*  \widetilde{\D} _\fX )$ by setting
for any 
$\cE \in D   ( {} ^l  \widetilde{\D} _\fZ ) $
and
$\cN \in D   ( {} ^r  \widetilde{\D} _\fZ ) $,
\begin{gather}
\label{dfnu+(m)f}
u  _+ (\cN) := u _* 
\left ( \cN \otimes _{\widetilde{\D}  _\fZ } \widetilde{\D}  _{\fZ \to \fX} 
\right ),
u  _+ (\cE) := u _* 
\left ( 
\widetilde{\D}  _{\fX \leftarrow \fZ}  
\otimes _{\widetilde{\D}  _\fZ } 
\cE
\right ).
\end{gather}
Moreover, we get the functor
$\DD  \colon 
D  ( {} ^*  \widetilde{\D} _\fX ) 
\to
D  ({} ^*  \widetilde{\D} _\fX )$ by setting
for any 
$\cM \in D   ( {} ^r  \widetilde{\D} _\fX ) $,
$\cE \in D   ( {} ^l  \widetilde{\D} _\fX ) $
\begin{gather}
\label{dfnuD(m)f}
\DD   (\cM) := 
\R \mathcal{H}om _{\widetilde{\D} _\fX }
(\cM,\omega _{\fX} \otimes _{\O _\fX}  \widetilde{\D} _\fX ) [d _{X}],
\DD   (\cE) := 
\R \mathcal{H}om _{\D _\fX }
(\cE,  \widetilde{\D} _\fX\otimes _{\O _\fX} \omega _{\fX} ^{-1} ) [d _{X}],
\end{gather}
which are computed respectively by taking an injective resolution of 
$\omega _{\fX} \otimes _{\O _\fX}  \widetilde{\D} _\fX $
and
$\widetilde{\D} _\fX\otimes _{\O _\fX} \omega _{\fX} ^{-1}$.
These functors preserves the coherence and are compatible with the quasi-inverse functors 
$-\otimes _{\O _\fX} \omega _{\fX} ^{-1}$ and
$\omega _{\fX}  \otimes _{\O _\fX} -$ exchanging left and right 
$\widetilde{\D} _\fX$-modules structure.
More precisely, we have the canonical isomorphisms
\begin{equation}
\label{u+D-left2rightf}
\omega _{\fX}  \otimes _{\O _\fX} u  _+ (\cE)
\riso 
u  _+ (\omega _{\fZ}  \otimes _{\O _\fZ}  \cE),
\omega _{\fX}  \otimes _{\O _\fX} \DD (\cE)
\riso 
\DD (\omega _{\fX}  \otimes _{\O _\fX}  \cE)
\end{equation}
whose first one is constructed as \ref{u+D-left2right}.

\end{ntn}

\begin{prop}
\label{u+uflat-lem-form}
Let $\M$ be a 
right $\widetilde{\D} _\fX $-module, 
$\cN$ be a 
right  $\widetilde{\D} _\fZ $-module.
We use notations \ref{ntnu*form-dfn1} and \ref{dfnu+(m)f}.
\begin{enumerate}
\item We have the canonical adjunction morphisms
$\mathrm{adj} 
\colon 
u _+ u ^{\flat 0} (\M)
\to 
\M$
and 
$\mathrm{adj} 
\colon 
\cN
\to 
 u ^{\flat 0} u _+  (\cN)$.
Moreover, the compositions
$ u ^{\flat 0} (\M)
 \overset{\mathrm{adj}}{\longrightarrow} 
u ^{\flat 0} u _+ u ^{\flat 0} (\M)
\overset{\mathrm{adj}}{\longrightarrow} 
 u ^{\flat 0} (\M)$
and 
$u _+  (\cN)
\overset{\mathrm{adj}}{\longrightarrow} 
u _+ u ^{\flat 0} u _+  (\cN)
\overset{\mathrm{adj}}{\longrightarrow} 
u _+  (\cN)$
are the identity.

\item Using the above  adjunction morphisms, we construct maps 
$$\mathcal{H}om _{\widetilde{\D} _\fX }
(u _+ (\cN), \M)
\to 
u _* \mathcal{H}om _{\widetilde{\D} _\fZ  }
(\cN, u ^{\flat 0} (\M)),
\
u _* \mathcal{H}om _{\widetilde{\D} _\fZ  }
(\cN, u ^{\flat 0} (\M))
\to 
\mathcal{H}om _{\widetilde{\D} _{\fX} }
(u _+ (\cN), \M),$$
which are  inverse of each other.
\item The functor $u ^\flat$ transforms 
$K$-injective complexes into
$K$-injective complexes.
\end{enumerate}

\end{prop}

\begin{proof}
We can copy the proof of \ref{u+uflat-lem}.
\end{proof}

\begin{coro}
\label{Radj-u+flatf}
Let 
$\M \in D
( {} ^r \widetilde{\D} _{\fX} )$,
$\cN\in D 
({} ^r \widetilde{\D} _\fZ  )$. 
Let 
$\cE \in D 
( {} ^l \widetilde{\D} _{\fX} )$,
$\cF \in D 
({} ^l \widetilde{\D} _\fZ  )$. 
We have the isomorphisms
\begin{gather}
\R \mathcal{H}om _{\widetilde{\D} _{\fX} }
(u _+ (\cN), \M)
\riso 
u _* \R \mathcal{H}om _{\widetilde{\D} _\fZ  }
(\cN, u ^{\flat} (\M));
\\
\R \mathcal{H}om _{\widetilde{\D} _{\fX} }
(u _+ (\cE), \cF)
\riso 
u _* \R \mathcal{H}om _{\widetilde{\D} _\fZ  }
(\cE, u ^{!} (\cF)).
\end{gather}

\end{coro}

\begin{proof}
Taking a $K$-injective resolution of 
$\M$, the first isomorphism  is a consequence of \ref{u+uflat-lem-form}.2--3.
Using \ref{fund-isom2} and \ref{u+D-left2rightf}, we deduce  the second isomorphism from the first one.
\end{proof}

\begin{coro}
\label{rel-dual-isom-immf}
Let 
$\cN\in D ^{\mathrm{b}} _{\mathrm{coh}}  ({} ^* \widetilde{\D} _\fZ  )$
with $*=r$ or $*=l$.
We have the isomorphism of $D ^{\mathrm{b}} _{\mathrm{coh}}  ({} ^* \widetilde{\D} _\fX  )$:
\begin{equation}
\label{rel-dual-isom-immf1}
\DD  \circ u _+ (\cN)
\riso 
 u _+ \circ \DD  (\cN).
\end{equation}
\end{coro}

\begin{proof}
Using \ref{fund-isom2-corpref} and \ref{Radj-u+flatf}, 
we can copy the proof of \ref{rel-dual-isom-imm}. 
\end{proof}

\begin{prop}
\label{u+uflat-lem-formal}
Let 
$\cN$ be a right coherent $\widetilde{\D} _{\ZZ} $-module without $p$-torsion.
The canonical homomorphism of $\widetilde{\D} _{\ZZ} $-modules 
$\mathrm{adj} 
\colon 
\cN
\to
 u ^{\flat 0} u _+  (\cN)$
 is an isomorphism.
\end{prop}

\begin{proof}
We proceed similarly to \cite[2.3.1]{surcoh-hol}.
\end{proof}

\subsection{Relative duality isomorphism and adjunction for proper morphisms}
Let  $f \colon \fP ^{\prime } \to \fP $ be a morphism of smooth formal schemes over $\S $,
$T$ and $T'$ be some divisors of respectively $P$ and $P'$ such that 
$f ( P '\setminus T' ) \subset P \setminus T$.
We will extend later for realizable morphisms 
the following theorem and its first corollary (see \ref{cor-adj-formulbis}).

\begin{thm}
\label{dualrelative}
Suppose $f$ is proper, and 
$T ' = f ^{-1}(T)$.
Let $\E ' \in D ^\mathrm{b} _{\mathrm{coh}}
(\D ^{\dag} _{\fP ^{\prime }} (\hdag T' ) _{\Q}).$
We have in 
$D ^\mathrm{b} _{\mathrm{coh}}
(\D ^{\dag} _{\fP  } (\hdag T ) _{\Q})$
the isomorphism
\begin{equation}
\label{dualrelative-morp}
f _{T+} \circ \DD _T ( \E ') 
\riso 
\DD _T \circ f _{T+} (\E ').
\end{equation}
\end{thm}

\begin{proof}
We construct the morphism \ref{dualrelative-morp} similarly to 
that of \cite[IV.1.3]{Vir04} (for more details, see 
\cite[1.2.7]{caro_courbe-nouveau}).
The restriction of this morphism \ref{dualrelative-morp} to $\fP \setminus T$ is the same as that 
of Virrion in \cite[IV]{Vir04}. Hence, the restriction to  $\fP \setminus T$ of this morphism is an isomorphism and
then by faithfulness this is an isomorphism.
 \end{proof}

\begin{cor}
\label{cor-adj-formul}
Suppose $f$ is proper, and 
$T ' = f ^{-1}(T)$.
Let $\E ' \in D ^\mathrm{b} _{\mathrm{coh}}
(\D ^{\dag} _{\fP ^{\prime }} (\hdag T' ) _{\Q})$,
and
$\E 
\in 
D ^\mathrm{b} _{\mathrm{coh}}
(\D ^{\dag} _{\fP  } (\hdag T ) _{\Q})$.
We have 
the isomorphisms
\begin{gather}
\label{cor-adj-formul-bij1}
\R \mathcal{H} om _{\D ^{\dag} _{\fP } (\hdag T ) _{\Q}}
( f _{T+} ( \E ') , \E) 
\riso 
\R f _* 
\R \mathcal{H} om _{\D ^{\dag} _{\fP ' } (\hdag T ' ) _{\Q}}
( \E ' ,f ^! _{T}  ( \E)). 
\\
\label{cor-adj-formul-bij2}
\R \mathrm{Hom}  _{\D ^{\dag} _{\fP } (\hdag T ) _{\Q}}
( f _{T+} ( \E ') , \E) 
\riso 
\R \mathrm{Hom}  _{\D ^{\dag} _{\fP ' } (\hdag T' ) _{\Q}}
( \E ' ,f ^! _{T}  ( \E)). 
\end{gather}
\end{cor}

\begin{proof}
The proof is identical to that of 
\cite[IV.4.1 and IV.4.2]{Vir04}: this is a formal consequence of 
the relative duality isomorphism \ref{dualrelative}.
For the reader, let us clarify it below: 
following 
\ref{Noot-Huyghe-finitehomoldim}, we have 
$D ^{\mathrm{b}} _{\mathrm{coh}}( \smash{\D} ^\dag _{\fP /\S } (\hdag T) _{\Q} )
=
D ^{\mathrm{b}} _{\mathrm{parf}}( \smash{\D} ^\dag _{\fP /\S } (\hdag T) _{\Q} )$.
Hence, 
using \cite[2.1.17]{caro_comparaison}, 
we construct the canonical isomorphism:
$$\R \mathcal{H} om _{\D ^{\dag} _{\fP } (\hdag T  ) _{\Q}}
( f _{T+} ( \E ') , \E)  
\riso
\left (
\omega  _{\fP}
\otimes _{\O _{\fP}}
\E
\right )
\otimes ^\L _{\D ^{\dag} _{\fP } (\hdag T ) _{\Q} }
\R \mathcal{H} om _{\D ^{\dag} _{\fP } (\hdag T ) _{\Q}}
\left ( f _{T+} ( \E ') , 
\D ^{\dag} _{\fP } (\hdag T ) _{\Q} 
\otimes _{\O _{\fP}}
\omega ^{-1} _{\fP}
\right ) 
.$$
Hence, using \ref{dualrelative}, 
we get by composition the isomorphism
\begin{equation}
\label{cor-adj-formul-bij1-iso2}
\R \mathcal{H} om _{\D ^{\dag} _{\fP } (\hdag T ) _{\Q}}
( f _{T+} ( \E ') ,\E)  
\riso 
\left (
\omega  _{\fP}
\otimes _{\O _{\fP}}
\E
\right )
\otimes ^\L _{\D ^{\dag} _{\fP } (\hdag T ) _{\Q} }
f _{T+} 
\left (\DD _T ( \E ') 
\right )
[-d _{P}] 
.\end{equation}
Using the projection formula
for the morphism of ringed spaces
$f \colon 
(\fP ' ,\D ^{\dag} _{\fP '} (\hdag T ') _{\Q})
\to 
(\fP, \D ^{\dag} _{\fP } (\hdag T ) _{\Q})$, 
the right term of \ref{cor-adj-formul-bij1-iso2} is isomorphic to 
\begin{equation}
\label{cor-adj-formul-bij1-iso3}
\R f _{*}
\left (
f ^{-1}
\left (
\omega  _{\fP}
\otimes _{\O _{\fP}}
\E
\right )
\otimes ^\L _{f ^{-1}\D ^{\dag} _{\fP } (\hdag T ) _{\Q} }
\D ^{\dag} _{\fP \leftarrow \fP '} (\hdag T ) _{\Q} 
\otimes ^{\L}
_{\D ^{\dag} _{\fP '} (\hdag T ' ) _{\Q}}
\DD _T ( \E ') 
\right ) 
[-d _{P}].
\end{equation}
\end{proof}
Using the isomorphisms
$\left (f ^{-1}
\left (
\omega  _{\fP}
\otimes _{\O _{\fP}}
\E
\right )
\otimes ^\L _{f ^{-1}\D ^{\dag} _{\fP } (\hdag T ) _{\Q} }
\D ^{\dag} _{\fP \leftarrow \fP '} (\hdag T ) _{\Q} \right )
\otimes _{\O _{\fP'}}
\omega ^{-1} _{\fP'}
[d _{P'/P}]
\riso 
f _T ^{!} (\E)$
and
$\omega  _{\fP'}
\otimes _{\O _{\fP'}}
\DD _T ( \E ') 
[-d _{P'}]
\riso 
\R \mathcal{H} om _{\D ^{\dag} _{\fP '} (\hdag T ' ) _{\Q}}
(  \E ' , 
\D ^{\dag} _{\fP '} (\hdag T ') _{\Q} )$, 
the term of \ref{cor-adj-formul-bij1-iso3} is isomorphic to 
$$
\R f _{*}
\left(
\R \mathcal{H} om _{\D ^{\dag} _{\fP '} (\hdag T ' ) _{\Q}}
(  \E ' , 
\D ^{\dag} _{\fP '} (\hdag T ') _{\Q} )
\otimes ^{\L}
_{\D ^{\dag} _{\fP '} (\hdag T ' ) _{\Q}}
f _T ^{!} (\E)
\right)
\underset{\cite[2.1.17]{caro_comparaison}}{\riso}
\R f _{*}
\left(
\R \mathcal{H} om _{\D ^{\dag} _{\fP '} (\hdag T ' ) _{\Q}}
(  \E ' , 
f _T ^{!} (\E)
\right)
.$$

\begin{cor}
\label{adj-morph}
Suppose $f$ is proper, and 
$T ' = f ^{-1}(T)$.
\begin{enumerate}
\item 
Let 
$\E ' 
\in 
D ^\mathrm{b} _{\mathrm{coh}}
(\D ^{\dag} _{\fP ^{\prime }} (\hdag T' ) _{\Q})$. 
We have the adjunction morphism
$\E ' \to f ^! _{T} f _{T+} (\E ')$. 
\item Let $\E \in D ^\mathrm{b} _{\mathrm{coh}}
(\D ^{\dag} _{\fP  } (\hdag T ) _{\Q})$
such that 
$f ^! _{T} (\E) \in 
D ^\mathrm{b} _{\mathrm{coh}}
(\D ^{\dag} _{\fP ^{\prime }} (\hdag T' ) _{\Q})$.
We have the adjunction morphism 
$f _{T+} f ^! _{T} (\E) \to \E$. 

\item Suppose $f$ is proper and smooth. 
Then $f _{T+} \colon 
D ^\mathrm{b} _{\mathrm{coh}}
(\D ^{\dag} _{\fP ^{\prime }} (\hdag T' ) _{\Q})
\to 
D ^\mathrm{b} _{\mathrm{coh}}
(\D ^{\dag} _{\fP  } (\hdag T ) _{\Q})$
is a right adjoint functor of 
$f ^! _{T}
\colon 
D ^\mathrm{b} _{\mathrm{coh}}
(\D ^{\dag} _{\fP  } (\hdag T ) _{\Q})
\to 
D ^\mathrm{b} _{\mathrm{coh}}
(\D ^{\dag} _{\fP ^{\prime }} (\hdag T' ) _{\Q})$.
\end{enumerate}
\end{cor}

\section{Coherent arithmetic $\D$-modules over a realizable smooth scheme}
\label{section-BK-var}

\subsection{Berthelot-Kashiwara's theorem}

\begin{theo}
[Berthelot-Kashiwara]
\label{exact-Berthelot-Kashiwara-full}
Let $u \colon \ZZ  \to \X $ be a closed immersion of smooth 
formal schemes over $\S $.
Let $D$ be a divisor of $X$ such that $Z\cap D $ is a divisor of $Z$.

The functors $u ^{!}$ and $u _{+}$ induce quasi-inverse equivalences between the category of 
coherent $\D ^{\dag} _{\fX /\S } (\hdag D) _{\Q}$-modules with support in $Z$ 
and that of coherent 
$\D ^{\dag} _{\fZ } (\hdag D \cap Z) _{\Q}$-modules.
These functors $u ^{!}$ and  $u  _{+}$ are exact over these categories. 

\end{theo}

\begin{proof}
We can copy word by word the proof of \cite[A.8]{caro-stab-sys-ind-surcoh}.
\end{proof}

\begin{rem}
\label{rem-exact-Berthelot-Kashiwara-full}
With notation \ref{exact-Berthelot-Kashiwara-full}, 
by copying the proof of \cite[A.8]{caro-stab-sys-ind-surcoh},
we check more precisely that the adjunction morphism of 
$u _+ u ^! (\E) \to \E$ (see \ref{u+uflat-lem-form}) is an isomorphism for any 
coherent $\D ^{\dag} _{\fX /\S,\Q}$-module $\E$ with support in $Z$. 
\end{rem}

\begin{thm}
[Inductive system version of Berthelot-Kashiwara's theorem]
\label{u!u+=id}
We keep notation \ref{exact-Berthelot-Kashiwara-full}.
Set $\fY: =\X \setminus Z$. 
Let 
$\FF ^{(\bullet)} 
\in \smash{\underrightarrow{LD}}  ^\mathrm{b} _{\Q, \mathrm{coh}}
(\overset{^\mathrm{l}}{} \smash{\widehat{\D}} _{\ZZ /\S } ^{(\bullet)} (D \cap Z))$,
$\E ^{(\bullet)}  
\in \smash{\underrightarrow{LD}}  ^\mathrm{b} _{\Q, \mathrm{coh}}
(\overset{^\mathrm{l}}{} \smash{\widehat{\D}} _{\X /\S } ^{(\bullet)} (D ))$
such that
$\E ^{(\bullet)} |\fY \riso 0$
in $\smash{\underrightarrow{LD}}  ^\mathrm{b} _{\Q, \mathrm{coh}}
(\overset{^\mathrm{l}}{} \smash{\widehat{\D}} _{\X /\S  } ^{(\bullet)} (D ))$. 

\begin{enumerate}
\item We have the canonical isomorphism in 
$\smash{\underrightarrow{LD}}  ^\mathrm{b} _{\Q, \mathrm{coh}}
(\overset{^\mathrm{l}}{} \smash{\widehat{\D}} _{\ZZ /\S  } ^{(\bullet)} (D \cap Z))$ of the form:
\begin{equation}
\label{u!u+=id-iso}
u ^{ !(\bullet)} \circ u _{+} ^{ (\bullet)} (\FF ^{(\bullet)})
\riso \FF ^{(\bullet)}.
\end{equation}

\item We have $u ^{ !(\bullet)} (\E ^{(\bullet)} )\in \smash{\underrightarrow{LD}}  ^\mathrm{b} _{\Q, \mathrm{coh}}
(\overset{^\mathrm{l}}{} \smash{\widehat{\D}} _{\ZZ /\S } ^{(\bullet)} (D \cap Z))$ and 
we have the canonical isomorphism :
\begin{equation}
\label{u!u+=id-isobis}
u _{+} ^{ (\bullet)} \circ  u ^{ !(\bullet)}  (\E ^{(\bullet)})
\riso \E ^{(\bullet)}.
\end{equation}

\item 
\label{u!u+=id-eq-cat}
The functors $u _{+} ^{ (\bullet)} $ and $u ^{ !(\bullet)}  $ induce t-exact quasi-inverse equivalences of categories 
between 
\begin{enumerate}
\item $\smash{\underrightarrow{LD}}  ^\mathrm{b} _{\Q, \mathrm{coh}}
(\overset{^\mathrm{l}}{} \smash{\widehat{\D}} _{\ZZ /\S } ^{(\bullet)} (D \cap Z))$
(resp. 
$\smash{\underrightarrow{LD}}  ^0 _{\Q, \mathrm{coh}}
(\overset{^\mathrm{l}}{} \smash{\widehat{\D}} _{\ZZ /\S } ^{(\bullet)} (D \cap Z))$)

\item and the subcategory of 
$\smash{\underrightarrow{LD}}  ^\mathrm{b} _{\Q, \mathrm{coh}}
(\overset{^\mathrm{l}}{} \smash{\widehat{\D}} _{\X /\S } ^{(\bullet)} (D ))$
(resp. $\smash{\underrightarrow{LD}}  ^0 _{\Q, \mathrm{coh}}
(\overset{^\mathrm{l}}{} \smash{\widehat{\D}} _{\X /\S } ^{(\bullet)} (D ))$) 
of complexes 
$\E ^{(\bullet)} $ so that 
$\E ^{(\bullet)} |\fY \riso 0$.

\end{enumerate}
\end{enumerate}

\end{thm}

\begin{proof}
Using Theorem \ref{exact-Berthelot-Kashiwara-full}, 
we can copy the proof of  \cite[5.3.7]{caro-stab-sys-ind-surcoh}.
\end{proof}

\subsection{Glueing isomorphisms, base change isomorphisms for pushforwards by a closed immersion}
Let $f,f',f''\colon \X  \to \Y $ be three morphisms of smooth formal schemes over $\S $
such that $f _0=  f' _0=f'' _0$. 
Let $g,g'\colon \Y  \to \ZZ $ be two morphisms of smooth formal schemes over $\S $
such that $g _0=  g' _0$. 
Let $T _Z$ be a divisor of $Z$ such that $T _Y:= g _0 ^{-1} (T)$ is a divisor of $Y$
and 
$T _X := f _0 ^{-1} (T _Y)$ is a divisor of $X$.

\begin{empt}
\label{2.1.5Be2}
Following \cite[2.1.10]{caro-construction} (still valid when $k$ is not supposed to be perfect), 
we have the canonical isomorphism of functors 
$\smash{\underrightarrow{LD}} ^{\mathrm{b}} _{\Q,\mathrm{qc}}
 ( \smash{\widetilde{\D}} _{\Y /\S } ^{(\bullet)}(T _Y))
 \to 
 \smash{\underrightarrow{LD}} ^{\mathrm{b}} _{\Q,\mathrm{qc}}
 ( \smash{\widetilde{\D}} _{\X /\S } ^{(\bullet)}(T _X))$
 of the form
\begin{equation}
\notag
\tau _{f,f'} ^{(\bullet)}
\colon 
f  _{T _Y} ^{\prime !(\bullet)} \riso f _{T _Y} ^{(\bullet) !}.
\end{equation}
These isomorphisms satisfy the following formulas
$\tau ^{(\bullet)} _{f,f}=\mathrm{Id}$,
$\tau ^{(\bullet)} _{f,f''}= \tau ^{(\bullet)} _{f,f'} \circ \tau ^{(\bullet)} _{f',f''}$,
$\tau ^{(\bullet)} _{f,f'} \circ g  _{T _Z} ^{!(\bullet)} = \tau ^{(\bullet)} _{g \circ f ,g \circ f'} $ 
and
$f _{T _Y} ^{(\bullet) !} \circ \tau ^{(\bullet)} _{g,g'}= \tau ^{(\bullet)} _{g \circ f ,g '\circ f} $ 
(see \cite[2.1.3]{caro-construction}).
\end{empt}

\begin{prop}
\label{prop-glueiniso-coh}
\begin{enumerate}
\item There exists a canonical glueing isomorphism of functors 
$D ^{\mathrm{b}} _{\mathrm{coh}} ( \smash{\D} ^{\dag} _{\Y} (\hdag T _Y) _{\Q} )
\to
D ^{\mathrm{b}}  ( \smash{\D} ^{\dag} _{\X } (\hdag T _X ) _{\Q}  )$
of the form
\begin{equation}
\label{prop-glueiniso-coh1}
\tau _{f,f'}\colon 
f  _{T _Y} ^{\prime!} \riso f _{T _Y} ^{!},
\end{equation}
such that $\tau _{f,f}=\mathrm{Id}$,
$\tau _{f,f''}= \tau _{f,f'} \circ \tau _{f',f''}$,
$\tau _{f,f'} \circ g  _{T _Z} ^{!} = \tau _{g \circ f ,g \circ f'} $ 
and
$f _{T _Y} ^{!} \circ \tau _{g,g'}= \tau _{g \circ f ,g '\circ f} $.

\item The diagram 
 of functors
$\smash{\underrightarrow{LD}} ^{\mathrm{b}} _{\Q,\mathrm{coh}}
 ( \smash{\widetilde{\D}} _{\Y /\S } ^{(\bullet)}(T _Y))
 \to 
 D ^{\mathrm{b}}  ( \smash{\D} ^{\dag} _{\X } (\hdag T _X) _{\Q}  )$
 \begin{equation}
 \notag
 \xymatrix {
 {\underrightarrow{\lim} \circ f  _{T _Y} ^{\prime !(\bullet)} } 
 \ar[r] ^-{\sim} _-{\underrightarrow{\lim} \circ \tau _{f,f'} ^{(\bullet)}}
 \ar[d] ^-{\sim}
 & 
 {\underrightarrow{\lim} \circ f _{T _Y} ^{(\bullet) !}} 
 \ar[d] ^-{\sim} 
 \\ 
 {f  _{T _Y} ^{\prime!} \circ \underrightarrow{\lim}} 
\ar[r] ^-{\tau _{f,f'} \circ \underrightarrow{\lim} } & 
 {f  _{T _Y} ^{!}\circ \underrightarrow{\lim}} 
 }
 \end{equation}
is commutative up to canonical isomorphism.
\end{enumerate}

\end{prop}

\begin{proof}
Let $\FF\in D ^{\mathrm{b}} _{\mathrm{coh}} ( \smash{\D} ^{\dag} _{\Y} (\hdag T _Y) _{\Q} )$.
Taking inductive limits of the completion of the glueing isomorphisms  \cite[2.1.5]{Be2}, we get the isomorphism 
$\tau _{f,f '}
\colon  \smash{\D} ^{\dag} _{\X \overset{f'}{\rightarrow}\Y}  (\hdag T _Y) _{ \Q}  
\riso 
 \smash{\D} ^{\dag} _{\X \overset{f}{\rightarrow}\Y} (\hdag T _Y) _{ \Q}   $.
It follows from \cite[2.1.5]{Be2} that these isomorphisms satisfies the desired properties.  
Finally, we still denote by 
$\tau _{f,f '}$
the composition
$f  ^{\prime !} _{T _Y} \FF = 
\smash{\D} ^{\dag} _{\X \overset{f'}{\rightarrow}\Y}(\hdag T _Y) _{ \Q}      \otimes ^{\L} _{f _0 ^{-1}\smash{\D} ^{\dag} _{\Y} (\hdag T _Y) _{ \Q}  }
f _0 ^{-1} \FF [d _{X/Y}]
\underset{\tau _{f,f'} \otimes ^{\L} id}{\riso}
\smash{\D} ^{\dag} _{\X \overset{f}{\rightarrow}\Y} (\hdag T _Y) _{ \Q}     \otimes ^{\L} _{f _0 ^{-1}\smash{\D} ^{\dag} _{\Y} (\hdag T _Y) _{ \Q} }
f _0 ^{-1} \FF [d _{X/Y}]
=
f _{T _Y} ^{!} \FF$.
They also satisfy the desired properties.
\end{proof}

\begin{empt}
\label{2.1.5Be2-empt}
a) For any  $\smash{\D} ^{\dag} _{\Y} (\hdag T _Y) _{\Q}$-module $\G$, 
we set 
$$ f ^{\dag *} _{T _Y} (\G)  := 
\smash{\D} ^{\dag} _{\X \overset{f}{\rightarrow}\Y} (\hdag T _Y) _{ \Q}     \otimes  _{f _0 ^{-1}\smash{\D} ^{\dag} _{\Y} (\hdag T _Y) _{ \Q} }
f _0 ^{-1} \G.$$
Similarly to \ref{prop-glueiniso-coh},
we construct isomorphisms
$\tau _{f,f'}\colon 
f ^{\prime \dag *} _{T _Y}  (\G) \riso 
f ^{\dag *} _{T _Y}(  \G)$
functorial in $\G$ and 
such that such that $\tau _{f,f}=\mathrm{Id}$,
$\tau _{f,f''}= \tau _{f,f'} \circ \tau _{f',f''}$.
We have the isomorphism 
of functors
$D ^{\mathrm{b}} _{\mathrm{coh}} ( \smash{\D} ^{\dag} _{\Y} (\hdag T _Y) _{\Q} )
\to 
D ^{\mathrm{b}}  ( \smash{\D} ^{\dag} _{\X } (\hdag T _X ) _{\Q}  )$
of the form
$ f _{T _Y} ^! \riso \L 
f ^{\dag *} _{T _Y} [d _{X/Y}]$.

b) Suppose $f$ is finite.
Then using \cite[3.2.4]{Be1},
we check that the canonical morphism
$$\widetilde{\B} ^{(m)} _{\X} ( T _X)
\otimes _{f ^{-1} \widetilde{\B} ^{(m)} _{\Y} ( T _Y) }f _0 ^{-1} \smash{\widetilde{\D}} _{\Y / \S } ^{(m)} (T _Y)
\to
\smash{\widetilde{\D}} _{\X \overset{f}{\rightarrow}\Y} ^{(m)} ( T _Y)$$
is an isomorphism.
Hence, so is the canonical morphism
$$
\O _{\X} (\hdag T _X) _{\Q} \otimes _{f _0 ^{-1} \O _{\Y} (\hdag T _Y) _{\Q} } f _0 ^{-1}\smash{\D} ^{\dag} _{\Y}  (\hdag T _Y) _{ \Q}  
\to
\smash{\D} ^{\dag} _{\X \overset{f}{\rightarrow}\Y}  (\hdag T _Y) _{ \Q} . $$
Tensoring by $\Q$ and taking the inductive limit over the level, 
this yields the canonical morphism
$$f _{T _Y} ^* (\G) := \O _{\X} (\hdag T _X) _{\Q} \otimes _{f _0 ^{-1} \O _{\Y} (\hdag T _Y) _{\Q} } f _0 ^{-1}\G 
\to 
f ^{\dag *} _{T _Y} (\G)  $$
is an isomorphism.
Hence, if 
$\FF\in D ^{\mathrm{b}} _{\mathrm{coh}} ( \smash{\D} ^{\dag} _{\Y} (\hdag T _Y) _{\Q} )$ 
has a resolution $\cP$ by $\smash{\D} ^{\dag} _{\Y} (\hdag T _Y) _{\Q}$-modules  which are 
$\O _{\Y} (\hdag T _Y) _{\Q}$-flat , then we get the isomorphism
$f ^{\dag *} _{T _Y} (\FF ) 
\riso 
\L f ^{*} _{T _Y} (\FF) $.
\end{empt}

\begin{rem}
\label{rem-flat-resol-tau}
Let $\FF\in D ^{\mathrm{b}} _{\mathrm{coh}} ( \smash{\D} ^{\dag} _{\Y} (\hdag T _Y) _{\Q} )$.
\begin{enumerate}
\item Suppose $\FF$ has a resolution $\cP$ by flat coherent $\smash{\D} ^{\dag} _{\Y} (\hdag T) _{\Q}$-modules.
Via $ f _{T _Y} ^! (\FF) \riso 
f ^{\dag *} _{T _Y} ( \cP) [d _{X/Y}]$
and 
$f _{T _Y} ^{\prime !} (\FF) \riso 
f ^{\prime \dag *} _{T _Y} ( \cP) [d _{X/Y}]$
(see \ref{2.1.5Be2-empt}), 
the isomorphism 
$\tau _{f,f'}\colon 
f  _{T _Y} ^{\prime !} \FF \riso f _{T _Y} ^{!} \FF$ is the same (up to the shift $[d _{X/Y}]$)
than 
that 
$\tau _{f,f'}\colon 
f ^{\prime \dag *} _{T _Y} ( \cP)  \riso f ^{\dag *} _{T _Y} ( \cP)$,
which is 
computed term by term.

\item  Suppose $\FF$ has a resolution $\cP$ by coherent $\smash{\D} ^{\dag} _{\Y} (\hdag T) _{\Q}$-modules which are 
$\O _{\Y} (\hdag T) _{\Q}$-flat and suppose 
$f$ and $g$ are finite morphisms.
Via $ f _{T _Y} ^! (\FF) \riso 
f ^{ *} _{T _Y} ( \cP) [d _{X/Y}]$
and 
$f _{T _Y} ^{\prime !} (\FF) \riso 
f ^{\prime *} _{T _Y} ( \cP) [d _{X/Y}]$
(see \ref{2.1.5Be2-empt}), 
the isomorphism 
$\tau _{f,f'}\colon 
f  _{T _Y} ^{\prime !} \FF \riso f _{T _Y} ^{!} \FF$ is the same (up to the shift $[d _{X/Y}]$)
than 
that 
$\tau _{f,f'}\colon 
f ^{\prime *} _{T _Y} ( \cP)  \riso f ^{ *} _{T _Y} ( \cP)$,
which is 
computed term by term.

\end{enumerate}
\end{rem}

\begin{prop}
\label{comp-comp-adj-immf}
  Consider the following diagram of smooth formal schemes over $\S $:
\begin{equation}\label{deuxcarresadj}
  \xymatrix  @R=0,3cm {
  {\fP ^{\prime \prime }} \ar[r] ^g
  &
  {\fP ^{\prime }} \ar[r] ^f
  &
  {\fP }
  \\
  {\X ^{\prime \prime }} \ar[u]  ^{u''}\ar[r]^b
  &
  {\X ^{\prime }} \ar[u] ^-{u'} \ar[r]^a
  &
  {\X  ,} \ar[u] ^u
  }
\end{equation}
where $f$, $g$, $a$ and $b$ are  smooth,
where $u$, $u'$ and $u''$ are some closed immersions. 
We suppose that the diagram \ref{deuxcarresadj} 
is commutative modulo $\pi$. 
Moreover, let  $T _{P }$ be a divisor of  $P $ such that $T _{P^{\prime }} := f ^{-1} (T _{P })$
(resp. $T _{P^{\prime \prime }} := g ^{-1} (T _{P^{\prime }})$, $T _{X } := u ^{-1} (T _{P })$, $T _{X^{\prime }} := u ^{\prime -1} (T _{P^{\prime }})$
and $T _{X^{\prime \prime }} := u ^{\prime \prime -1} (T _{P^{\prime \prime }})$) is a divisor of 
$P^{\prime }$ (resp. $P^{\prime \prime }$, $X $, $X^{\prime }$ and  $X^{\prime \prime }$).

\begin{enumerate}[(i)]
 \item We have the canonical adjunction morphism
\begin{equation}
\label{comp-comp-adj-immf-morp1}
u ^\prime _+ \circ a ^!\rightarrow f ^{!}\circ u _+
\end{equation}
of functors
  $D ^{\mathrm{b}} _{\mathrm{coh}} (\smash{\D} ^\dag _{\X  } (\hdag T _{X }) _\Q)
  \to 
  D ^{\mathrm{b}} _{\mathrm{coh}} (\smash{\D} ^\dag _{\fP '  } (\hdag T _{P' }) _\Q)$.
If the right square of \ref{deuxcarresadj} 
is cartesian modulo $\pi$ then 
\ref{comp-comp-adj-immf-morp1}  is an isomorphism.

\item Denoting by 
$\phi \colon u ^\prime _+ \circ a ^! \rightarrow f ^{!}\circ u _+$,
(resp. 
$\phi '\ : \ u ^{''} _+ \circ b ^! \rightarrow g ^! \circ u ^\prime _+$,
resp. $\phi ''\ : \ u ^{''} _+ \circ (a \circ b ) ^! \rightarrow (f \circ g) ^!\circ u  _+$)
the morphism of adjunction of the right square 
 \ref{deuxcarresadj} (resp. the left square, resp. the outline of \ref{deuxcarresadj}),
then the following diagram
$$\xymatrix  @R=0,3cm {
{ u ''  _+ \circ (a\circ b) ^!}
\ar[r] _\sim
\ar[d] ^-{\phi ''}
&
{ u ''  _+ \circ b ^!\circ  a ^! }
\ar[d]^{( g ^! \circ \phi)\circ (\phi ' \circ a ^!)}
\\
{(f\circ g) ^! \circ  u _+  }
\ar[r] _-{\sim}
&
{ g ^!  \circ f ^!\circ  u _+,}
}
$$
is commutative. 
By abuse of notation, 
we get the transitivity equality
$\phi ''=( g ^! \circ \phi)\circ (\phi ' \circ a ^!)  $.

\item Let $a '$ : $ \X ^{\prime } \rightarrow \X $ (resp. $f'$ : $\fP ^{\prime } \rightarrow \fP $)
be a morphism whose reduction $X ^{\prime } \rightarrow \X $ (resp. $P ^{\prime } \rightarrow \fP $)
is equal to that of $a$ (resp. $f$). Then the following diagram
$$\xymatrix  @R=0,3cm {
{ u ^\prime_+ a ^{ !}}
\ar[r] ^-{\phi}
&
{f ^{!}\circ u _+}
\\
{ u' _+ a ^{\prime !} }
\ar[r]^{\psi}
\ar[u] ^-{ u ^\prime_+ (\tau _{a,a'})} _-{\sim}
&
{f ^{\prime !}\circ u  _+,}
\ar[u] ^-{\tau _{f,f'} u _+} _-{\sim}
}$$
where $\psi$ means the morphism of adjunction of the right square of  \ref{deuxcarresadj} whose 
$a$ and  $f$ have been replaced respectively by $a'$ and  $f'$,
is commutative.
\end{enumerate}

\end{prop}

\begin{proof}
We build \ref{comp-comp-adj-immf-morp1} using the adjoint paires
$(u _+, u ^!)$ and $(u '_+, u ^{\prime !})$
(see \ref{Radj-u+flatf}).
If the right square of \ref{deuxcarresadj} 
is cartesian modulo $\pi$ then using Berthelot-Kashiwara's theorem \ref{exact-Berthelot-Kashiwara-full}, 
we check 
\ref{comp-comp-adj-immf-morp1}  is an isomorphism.
We proceed similarly to \cite[2.2.2]{caro-construction} to check the other properties.
\end{proof}

\subsection{Berthelot-Kashiwara's theorem for closed immersions of schemes over $S $}

Let $\fP $ be a smooth formal scheme over $\S $.
Let $u _0\colon X  \to P $ be a closed immersion of smooth schemes over $S $.
Let $T$ be a divisor of $P$ such that $Z:= T \cap X$ is a divisor of $X$. 
We set $Y:= X \setminus Z$. 
\label{ntnPPalpha}
Let $(\fP  _{\alpha}) _{\alpha \in \Lambda}$ be an open covering of  $\fP $.
We set $\fP  _{\alpha \beta}:= \fP  _\alpha \cap \fP  _\beta$,
$\fP  _{\alpha \beta \gamma}:= \fP  _\alpha \cap \fP  _\beta \cap \fP  _\gamma$,
$X  _\alpha := X  \cap P  _\alpha$,
$X _{\alpha \beta } := X  _\alpha \cap X  _\beta$ and
$X _{\alpha \beta \gamma } := X  _\alpha \cap X  _\beta \cap X  _\gamma $.
We denote by $Y _\alpha := X  _\alpha \cap Y$,
$Y _{\alpha \beta} := Y _\alpha \cap Y _\beta$,
$Y _{\alpha \beta \gamma} := Y _\alpha \cap Y _\beta \cap Y _\gamma $,
$Z _\alpha := X  _\alpha \cap Z$,
$Z _{\alpha \beta} := Z _\alpha \cap Z _\beta$,
$Z _{\alpha \beta \gamma} := Z _\alpha \cap Z _\beta \cap Z _\gamma $,
$j _\alpha$ : $ Y _\alpha
\hookrightarrow X  _\alpha$,
$j _{\alpha \beta} $ :
$Y _{\alpha \beta}  \hookrightarrow  X  _{\alpha \beta}$
and
$j _{\alpha \beta \gamma} $ :
$Y _{\alpha \beta \gamma }  \hookrightarrow  X  _{\alpha \beta \gamma} $
the canonical open immersions.

We suppose the covering $(\fP  _{\alpha}) _{\alpha \in \Lambda}$ satisfies the following lifting properties 
(such coverings exist following example \ref{section-BK-var-ex}).
For any 3uple $(\alpha, \, \beta,\, \gamma)\in \Lambda ^3$, 
we suppose there exists 
$\X  _\alpha$ (resp. $\X  _{\alpha \beta}$, $\X  _{\alpha \beta \gamma}$)
some smooth formal $\S $-schemes lifting of $X  _\alpha$
(resp. $X  _{\alpha \beta}$, $X  _{\alpha \beta \gamma}$),
$p _1 ^{\alpha \beta}$ :
$\X   _{\alpha \beta} \rightarrow \X  _{\alpha}$
(resp. $p _2 ^{\alpha \beta}$ :
$\X   _{\alpha \beta} \rightarrow \X  _{\beta}$)
some flat lifting 
of 
$X   _{\alpha \beta} \rightarrow X  _{\alpha}$
(resp. $X   _{\alpha \beta} \rightarrow X  _{\beta}$).
Similarly, for any $(\alpha,\,\beta,\,\gamma )\in \Lambda ^3$, we suppose there exist some lifting 
$p _{12} ^{\alpha \beta \gamma}$ : $\X   _{\alpha \beta \gamma} \rightarrow \X   _{\alpha \beta} $,
$p _{23} ^{\alpha \beta \gamma}$ : $\X   _{\alpha \beta \gamma} \rightarrow \X   _{\beta \gamma} $,
$p _{13} ^{\alpha \beta \gamma}$ : $\X   _{\alpha \beta \gamma} \rightarrow \X   _{\alpha \gamma} $,
$p _1 ^{\alpha \beta \gamma}$ : $\X   _{\alpha \beta \gamma} \rightarrow \X   _{\alpha} $,
$p _2 ^{\alpha \beta \gamma}$ : $\X   _{\alpha \beta \gamma} \rightarrow \X   _{\beta} $,
$p _3 ^{\alpha \beta \gamma}$ : $\X   _{\alpha \beta \gamma} \rightarrow \X   _{\gamma} $,
$u _{\alpha}$ : $\X  _{\alpha } \hookrightarrow \fP  _{\alpha }$,
$u _{\alpha \beta}$ : $\X  _{\alpha \beta} \hookrightarrow \fP  _{\alpha \beta}$
and
$u _{\alpha \beta \gamma}$ : $\X  _{\alpha \beta \gamma } \hookrightarrow \fP  _{\alpha \beta \gamma}$.

\begin{ex}
\label{section-BK-var-ex}
For instance, 
when for every $\alpha\in \Lambda$, $X  _\alpha$ is affine, 
(for instance when the covering $(\fP  _{\alpha}) _{\alpha \in \Lambda}$ is affine).
Indeed, since $P$ is separated (see our convention at the beginning of the paper), for any $\alpha,\beta ,\gamma \in \Lambda$,
$X _{\alpha \beta }$ and  $X _{\alpha \beta \gamma }$ are also affine.
Hence, such liftings exist.
\end{ex}

\begin{dfn}\label{defindonnederecol}
For any $\alpha \in \Lambda$, let $\E _\alpha$ be a coherent
$\D ^{\dag} _{\X  _{\alpha}} (\hdag Z _{\alpha}) _{\Q} $-module.
A \textit{glueing data} on $(\E _{\alpha})_{\alpha \in \Lambda}$
is the data for any $\alpha,\,\beta \in \Lambda$ of a
$\D ^{\dag} _{\X  _{\alpha \beta}} (\hdag Z _{\alpha \beta}) _{\Q} $-linear
isomorphism
$$ \theta _{  \alpha \beta} \ : \  p _2  ^{\alpha \beta !} (\E _{\beta}) \riso p  _1 ^{\alpha \beta !} (\E _{\alpha}),$$
satisfying the cocycle condition:
$\theta _{13} ^{\alpha \beta \gamma }=
\theta _{12} ^{\alpha \beta \gamma }
\circ
\theta _{23} ^{\alpha \beta \gamma }$,
where $\theta _{12} ^{\alpha \beta \gamma }$, $\theta _{23} ^{\alpha \beta \gamma }$
and $\theta _{13} ^{\alpha \beta \gamma }$ are the isomorphisms making commutative the following diagram
\begin{equation}
  \label{diag1-defindonnederecol}
\xymatrix  @R=0,3cm {
{  p _{12} ^{\alpha \beta \gamma !} p  _2 ^{\alpha \beta !}  (\E _\beta )}
\ar[r] ^-{\tau} _-{\sim}
\ar[d] ^-{p _{12} ^{\alpha \beta \gamma !} (\theta _{\alpha \beta})} _-{\sim}
&
{p _2 ^{\alpha \beta \gamma!}  (\E _\beta )}
\ar@{.>}[d] ^-{\theta _{12} ^{\alpha \beta \gamma }}
\\
{ p _{12} ^{\alpha \beta \gamma !}  p  _1 ^{\alpha \beta !}  (\E _\alpha)}
\ar[r]^{\tau} _-{\sim}
&
{p _1 ^{\alpha \beta \gamma!}(\E _\alpha),}
}
\xymatrix  @R=0,3cm {
{  p _{23} ^{\alpha \beta \gamma !} p  _2 ^{\beta \gamma!}  (\E _\gamma )}
\ar[r] ^-{\tau} _-{\sim}
\ar[d] ^-{p _{23} ^{\alpha \beta \gamma !} (\theta _{ \beta \gamma})} _-{\sim}
&
{p _3 ^{\alpha \beta \gamma!}  (\E _\gamma )}
\ar@{.>}[d] ^-{\theta _{23} ^{\alpha \beta \gamma }}
\\
{ p _{23} ^{\alpha \beta \gamma !}  p  _1 ^{ \beta \gamma !}  (\E _\beta)}
\ar[r]^{\tau} _-{\sim}
&
{p _2 ^{\alpha \beta \gamma!}(\E _\beta),}
}
\xymatrix  @R=0,3cm {
{  p _{13} ^{\alpha \beta \gamma !} p  _2 ^{\alpha \gamma !}  (\E _\gamma )}
\ar[r] ^-{\tau} _-{\sim}
\ar[d] ^-{p _{13} ^{\alpha \beta \gamma !} (\theta _{\alpha \gamma})} _-{\sim}
&
{p _3 ^{\alpha \beta \gamma!}  (\E _\gamma )}
\ar@{.>}[d]^{\theta _{13} ^{\alpha \beta \gamma }}
\\
{ p _{13} ^{\alpha \beta \gamma !}  p  _1 ^{\alpha \gamma !}  (\E _\alpha)}
\ar[r]^{\tau} _-{\sim}
&
{p _1 ^{\alpha \beta \gamma!}(\E _\alpha),}
}
\end{equation}
where $\tau$ are the glueing isomorphisms defined in \ref{prop-glueiniso-coh1}.

\end{dfn}
\begin{dfn}
We define the category $\mathrm{Coh} ((\X   _\alpha )_{\alpha \in \Lambda},Z/K)$ as follows: 

\begin{itemize}
\item [-] an object is a family $(\E _\alpha) _{\alpha \in \Lambda}$
of coherent  $\D ^{\dag} _{\X  _{\alpha}} (\hdag Z _{\alpha}) _{\Q} $-modules
together with a glueing data $ (\theta _{\alpha\beta}) _{\alpha ,\beta \in \Lambda}$,

\item [-] a morphism
$((\E _{\alpha})_{\alpha \in \Lambda},\, (\theta _{\alpha\beta}) _{\alpha ,\beta \in \Lambda})
\rightarrow
((\E ' _{\alpha})_{\alpha \in \Lambda},\, (\theta '_{\alpha\beta}) _{\alpha ,\beta \in \Lambda})$
is a familly of morphisms $f _\alpha$ : $\E _\alpha \rightarrow \E '_\alpha$
of coherent  $\D ^{\dag} _{\X  _{\alpha}} (\hdag Z _{\alpha}) _{\Q} $-modules
commuting with glueing data, i.e., such that the following diagrams are commutative : 
\begin{equation}
  \label{diag2-defindonnederecol}
\xymatrix  @R=0,3cm {
{ p _2  ^{\alpha \beta !} (\E _{\beta}) }
\ar[d] _-{p _2  ^{\alpha \beta !} (f _{\beta}) }
\ar[r] ^-{\theta _{\alpha\beta}} _-{\sim}
&
{  p  _1 ^{\alpha \beta !} (\E _{\alpha}) }
\ar[d] ^-{p  _1 ^{\alpha \beta !} (f _{\alpha})}
\\
{p _2  ^{\alpha \beta !} (\E '_{\beta})  }
\ar[r]^{\theta '_{\alpha\beta}} _-{\sim}
&
{ p  _1 ^{\alpha \beta !} (\E '_{\alpha})  .}
}
\end{equation}

\end{itemize}
\end{dfn}

\begin{rem}
\label{rem-ind-ariYXP}
We can consider the category $\mathrm{Coh} ((\X   _\alpha )_{\alpha \in \Lambda},Z/K)$
as the category of arithmetic $\D$-modules over $(Y,\fP)/\V$ or over $(Y,X)/\V$ (we can check that, up to canonical equivalence of categories,
this is independent of 
the choice of the closed immersion $X \hookrightarrow \fP$ and of the liftings $\X _\alpha$ etc.).
\end{rem}

\begin{dfn}
\label{dfnCohXPP}
We denote by 
$\mathrm{Coh} (X, \fP,T /K)$ the category 
of 
coherent $\D ^{\dag} _{\fP} (\hdag T) _{\Q} $-modules with support in $X$. 
When $T$ is the empty divisor, we simply write
$\mathrm{Coh} (X, \fP /K)$.
\end{dfn}

\begin{lem}
[Construction of $u ^! _0$]
\label{const-u0!}
There exists
a canonical functor 
$$u _0 ^! \colon 
\mathrm{Coh} (X, \fP,T /K)
\to 
\mathrm{Coh} ((\X   _\alpha )_{\alpha \in \Lambda},Z/K)$$
extending the usual functor $u _0 ^!$ when $X$ has a smooth formal $\S$-scheme lifting.
\end{lem}

\begin{proof}
For any object
$\E \in \mathrm{Coh} (X, \fP,T /K)$,
we set 
$\E _\alpha : =
\mathcal{H} ^{0} u _{\alpha } ^{!} ( \E | \fP _\alpha)$.
Remark that following 
\ref{u!u+=id},
$\E _\alpha$ is a coherent
$\D ^{\dag} _{\X  _{\alpha}} (\hdag Z _{\alpha}) _{\Q} $-module.
Via the isomorphisms of the form $\tau $ (\ref{prop-glueiniso-coh1}), 
we obtain the glueing 
$\D ^{\dag} _{\X  _{\alpha \beta}} (\hdag Z _{\alpha \beta}) _{\Q} $-linear
isomorphism
$ \theta _{  \alpha \beta} \ : \  p _2  ^{\alpha \beta !} (\E _{\beta}) \riso p  _1 ^{\alpha \beta !} (\E _{\alpha}),$
satisfying the cocycle condition:
$\theta _{13} ^{\alpha \beta \gamma }=
\theta _{12} ^{\alpha \beta \gamma }
\circ
\theta _{23} ^{\alpha \beta \gamma }$.
\end{proof}

\begin{lem}
\label{const-u0+}
There exists a canonical functor  
$$u _{0+}
\colon 
\mathrm{Coh} ((\X   _\alpha )_{\alpha \in \Lambda},Z/K)
\rightarrow
\mathrm{Coh} (X,\, \fP ,T /K)$$
extending the usual functor $u _{0+}$ when $X$ has a smooth formal $\S$-scheme lifting.
\end{lem}

\begin{proof}
This functor was constructed in \cite[2.5.4]{caro-construction} 
and was denoted by $\mathcal{R}ecol$. Let us recall its construction.
1) Let $(\E _{\alpha}) _{\alpha \in \Lambda}$
be a family of coherent $\D ^{\dag} _{\X  _{\alpha}} (\hdag Z _{\alpha}) _{\Q} $-modules 
together with a gluing data
$(\theta _{\alpha \beta})_{\alpha , \beta \in \Lambda}$.
Let's prove that $( u _{\alpha +} ( \E _{\alpha} ))_{\alpha \in \Lambda}$
glues via $(\theta _{\alpha \beta})_{\alpha , \beta \in \Lambda}$
to a coherent $\D ^{\dag} _{\fP} (\hdag T) _{\Q} $-module 
with support in $X$. 
Let $\phi _1 ^{\alpha \beta}$ (resp. $\phi _2 ^{\alpha \beta}$) be the adjunction morphism
(see the definition \ref{comp-comp-adj-immf})
of the left  (resp. the right) square of
\begin{equation}\label{carre0eqcatrecolD-mod}
\xymatrix  @R=0,3cm {
{\fP   _{\alpha \beta}} \ar[rr]
&
&{ \fP   _{\alpha}}
\\
{\X   _{\alpha \beta}} \ar[rr] ^-{p _1 ^{\alpha \beta}} \ar[u] ^-{u _{\alpha \beta}}
&
&
{ \X  _{\alpha},} \ar[u] ^-{u _{\alpha}}
}
\hspace{3cm}
\xymatrix  @R=0,3cm{
{\fP   _{\alpha \beta}} \ar[rr] &  &{ \fP   _{\beta}} \\
{\X   _{\alpha \beta}} \ar[rr] ^-{p _2 ^{\alpha \beta}} \ar[u] ^-{u _{\alpha \beta}}
&  &{ \X  _{\beta}.} \ar[u] ^-{u _{\beta}}
}
\end{equation}
For every $\alpha ,\beta \in \Lambda$ we define the isomorphism 
$\tau _{\alpha \beta}\ : \  {(u _{\beta +} ( \E _{\beta} ))} |_{\fP  _{\alpha \beta}}
\riso {(u _{\alpha +} ( \E _{\alpha} ))} |_{\fP  _{\alpha \beta}}$
to be the one making commutative the following diagram:
\begin{equation}\label{carre1eqcatrecolD-mod}
\xymatrix  @R=0,3cm {
{u _{\alpha \beta +} \circ p  _1 ^{\alpha \beta !} (\E _{\alpha})}
\ar[rr] ^-{\phi _1 ^{\alpha \beta}(\E _{\alpha})} _-{\sim}
& &
 {{(u _{\alpha +} ( \E _{\alpha} ))} |_{ \fP  _{\alpha \beta}}   } \\
{ u _{\alpha \beta +} \circ p  _2 ^{\alpha \beta !} (\E _{\beta}) }
 \ar[rr] ^-{\phi _2 ^{\alpha \beta}(\E _{\beta})} _-{\sim}
 \ar[u] ^-{ u _{\alpha \beta +}(\theta _{  \alpha \beta}) } _-{\sim}
 &  &
{ {(u _{\beta +} ( \E _{\beta} )) }|_{\fP  _{\alpha \beta}}. }
\ar@{.>}[u] _{\tau _{\alpha \beta}}
}
\end{equation}
These isomorphisms $\tau _{\alpha \beta}$
satisfy the glueing condition, 
and we check that 
 $u _{0+}$ is indeed a functor (for the details, see the proof of \cite[2.5.4]{caro-construction}). 
\end{proof}

\begin{thm}
\label{prop1}
The functors $u ^!  _0$ and $ u _{0+}$ constructed in respectively \ref{const-u0!} and \ref{const-u0+} 
are quasi-inverse equivalences of categories between 
$\mathrm{Coh} ((\X   _\alpha )_{\alpha \in \Lambda},Z/K)$
and
$\mathrm{Coh} (X,\, \fP ,T /K)$.

\end{thm}

\begin{proof}
We check that 
the isomorphism given by the quasi-inverse equivalences of categories of \ref{exact-Berthelot-Kashiwara-full} 
are compatible with the glueing isomorphisms. 
For the details of the proof, see \cite[2.5.4]{caro-construction}.
\end{proof}

\section{Arithmetic $\D$-modules associated with overconvergent isocrystals on smooth compactification}

\subsection{Overconvergent isocrystals}

Let $\fP $ be a smooth  formal scheme over $\S $.
Let $i\colon X  \hookrightarrow P $ be a closed immersion of smooth schemes over $S $.
Let $j\colon Y \hookrightarrow X$ be an open immersion. 
\begin{ntn}
[Overconvergent isocrystals]
\label{ntn-realP}
We denote by 
$\mathrm{Isoc} ^{\dag} (Y, X,\fP/K)$ the category of overconvergent 
isocrystals on $(Y, X,\fP)/K$ (see \cite[7.1.2]{LeStum-livreRigCoh}),
by 
$\mathrm{MIC}  (Y, X,\fP/K)$ the category of 
coherent 
$j ^\dag \O _{\X _K}$-modules together with an integrable connection,
and  by 
$\mathrm{MIC} ^{\dag} (Y, X,\fP/K)$ the full subcategory of 
$\mathrm{MIC}  (Y, X,\fP/K)$ of
coherent 
$j ^\dag \O _{\X _K}$-modules together with an overconvergent connection (see \cite[7.2.10]{LeStum-livreRigCoh}).
Following \cite[7.2.13]{LeStum-livreRigCoh} 
the realisation functor
$E \mapsto E _\fP$ induces an equivalence of categories 
\begin{equation}
\label{dfn-real}
\mathrm{real} _\fP
\colon 
\mathrm{Isoc} ^{\dag} (Y, X,\fP/K) \cong \mathrm{MIC} ^{\dag} (Y, X,\fP/K).
\end{equation}
This realization functor is constructed as follows.
Let $E \in \mathrm{Isoc} ^{\dag} (Y, X,\fP/K) $. 
Let 
$q _0 , q _1
\colon 
 \fP ^{ } \times _{\S } \fP ^{ }
 \to 
  \fP ^{ }$
be the left and the right projections. 
This yields the morphisms of frames 
$p _0 := (id,id, q _0), 
p _1:=(id,id, q _1)
\colon 
(Y, X,\fP \times _\S \fP) \to (Y, X,\fP)$.
We denote by $E _{\fP \times _\S \fP}$ the realization of $E$ on the frame
$(Y, X,\fP \times _\S \fP)$.  
We get the isomorphisms
$\phi _{p_0} 
\colon 
p _0 ^* E _{\fP} \riso E _{\fP \times _\S \fP} $
and
$\phi _{p_1} 
\colon 
p _1 ^* E _{\fP} \riso E _{\fP \times _\S \fP} $
(see notation 
\cite[7.1.1.(ii)]{LeStum-livreRigCoh} concerning the isomorphism $\phi$).
This yields the isomorphism 
\begin{equation}
\label{dfn-epsilon}
\epsilon = \phi _{p_0} ^{-1} \circ \phi _{p_1} 
\colon 
p _1 ^* E _{\fP} 
\riso 
p _0 ^* E _{\fP} ,
\end{equation}
which corresponds to the overconvergent connection on 
$E _{\fP}$ and is called the Taylor isomorphism of $E _{\fP} $.
\end{ntn}

\begin{empt}
\label{inv-image}
Let $f = (a,b,u) \colon (Y', X',\fP') \to (Y, X,\fP)$ be a morphism of  smooth $\S$-frames (see Definitions \cite[3.1.6 and 3.3.5]{LeStum-livreRigCoh}).
This yields a functor
$$f ^* \colon \mathrm{Isoc} ^{\dag} (Y, X,\fP/K)
\to 
\mathrm{Isoc} ^{\dag} (Y', X',\fP'/K).$$
On the other hand, 
we have the functor
$f ^* _K\colon 
\mathrm{MIC} ^{\dag} (Y, X,\fP/K)
\to 
\mathrm{MIC} ^{\dag} (Y', X',\fP'/K)$,
{\it the pullback by $f$},
which is defined for an object 
$E_{\fP} \in   \mathrm{MIC} ^{\dag} (Y, X,\fP/K) $ 
by setting
\begin{equation}
\label{inv-image-real}
f _K ^* (E _{\fP}) 
:= 
j ^{\prime \dag} \smash{\O} _{]X'[ _{\fP'}} \otimes  _{ u ^{-1} _K j ^\dag  \smash{\O} _{]X[ _{\fP}}} u ^{-1} _K E _{\fP}
\riso 
j ^{\prime \dag} \smash{\D} _{]X'[ _{\fP'}} \otimes  _{ u ^{-1} _K j ^\dag  \smash{\D} _{]X[ _{\fP}}} u ^{-1} _K E _{\fP},
\end{equation}
where $u _K \colon ]X'[ _{\fP'} \to ]X[ _{\fP}$ is the morphism of ringed spaces induced by $f$. 
When $X '= u ^{-1} (X)$ and $Y '= u ^{-1} (Y)$, 
the functor $f _K ^*$ can simply be denoted by $u _K ^*$.

The functor $\mathrm{real} _\fP$ of \ref{dfn-real} commutes with both pullbacks by $f$:
for any $E \in \mathrm{Isoc} ^{\dag} (Y, X,\fP/K) $ we have 
the natural isomorphism
$f ^* _K \circ \mathrm{real} _\fP (E ) \riso \mathrm{real} _{\fP '} \circ  f ^* (E)$
which can also be written
$f ^* _K (E _\fP) \riso f ^* (E)  _{\fP '} $. 
\end{empt}

\begin{empt}
\label{glueingisocntn}
Let $f = (a,b,u)$ and $g=(a,b,v)$ be two morphisms of smooth $\S$-frames of the form
$(Y', X',\fP') \to (Y, X,\fP)$.
\begin{enumerate}
\item The morphism 
$(u,v)\colon  \fP ^{\prime } \to  \fP ^{ } \times _{\S } \fP ^{ }$
induce the morphism of frames
$\delta  _{u,v}= (b , a, (u,v)) 
\colon 
(Y', X',\fP')
\to 
(Y, X,\fP \times _\S \fP) $.
We get the morphisms of frames
$f =  p _0 \circ \delta  _{u,v}$ and $g = p _1  \circ \delta  _{u,v}$.
Let $E_{\fP} \in   \mathrm{MIC} ^{\dag} (Y, X,\fP/K) $. From the isomorphism $\epsilon 
\colon 
p _1 ^* E _{\fP} 
\riso 
p _0 ^* E _{\fP} $ (see \ref{dfn-epsilon}),
we get the glueing isomorphism
\begin{equation}
\label{glueingisocntn-iso1}
\epsilon _{u,v}:= \delta  _{u,v} ^* (\epsilon) 
\colon 
g ^{*} E _{\fP} 
\riso 
f ^{ *} E _{\fP} .
\end{equation}
Using the property of the isomorphism $\phi $ of 
\cite[7.1.1.(ii)]{LeStum-livreRigCoh}, we have the equality
$\epsilon _{u,v}:= \delta  _{u,v}^* (\epsilon) = \phi _{f} ^{-1} \circ \phi _{g} $.
In particular, $\epsilon _{u,u} =id$.

\item Let $w\colon \fP ^{\prime } \to \fP $ be a third morphism of formal schemes over $\S $
such that 
we get the morphism of  smooth $\S$-frames  $h :=(b,\,a,\,w)
\colon 
(Y', X ',\fP')
\to 
(Y, X,\fP)$.
We have the transitive formula
$\epsilon_{u,\, w} = \epsilon_{u,\,v}\circ \epsilon_{v,\,w}$.

\item Let $f' = (a',b',u')$ and $g'=(a',b',v')$ be two morphisms of  smooth $\S$-frames of the form
$(Y'', X'',\fP'') \to (Y', X',\fP')$.
We check the formulas
$ \epsilon _{u \circ u', v \circ u'} = f ^{\prime *} _K \circ \epsilon _{u,v}$
and
$\epsilon _{u' ,v'} \circ f ^* _K =  \epsilon _{u \circ u', u \circ v'}$.

\end{enumerate}
\end{empt}

\begin{empt}
\label{ntnPPalpha-withoutdiv}
Let $(\fP  _{\alpha}) _{\alpha \in \Lambda}$ be an open covering of  $\fP $.
We set $\fP  _{\alpha \beta}:= \fP  _\alpha \cap \fP  _\beta$,
$\fP  _{\alpha \beta \gamma}:= \fP  _\alpha \cap \fP  _\beta \cap \fP  _\gamma$,
$X  _\alpha := X  \cap P  _\alpha$,
$X _{\alpha \beta } := X  _\alpha \cap X  _\beta$ and
$X _{\alpha \beta \gamma } := X  _\alpha \cap X  _\beta \cap X  _\gamma $.
We denote by $Y _\alpha := X  _\alpha \cap Y$,
$Y _{\alpha \beta} := Y _\alpha \cap Y _\beta$,
$Y _{\alpha \beta \gamma} := Y _\alpha \cap Y _\beta \cap Y _\gamma $.
$j _\alpha$ : $ Y _\alpha
\hookrightarrow X  _\alpha$,
$j _{\alpha \beta} $ :
$Y _{\alpha \beta}  \hookrightarrow  X  _{\alpha \beta}$
and
$j _{\alpha \beta \gamma} $ :
$Y _{\alpha \beta \gamma }  \hookrightarrow  X  _{\alpha \beta \gamma} $
the canonical open immersions.
We suppose that for every $\alpha\in \Lambda$, $X  _\alpha$ is affine, 
(for instance when the covering $(\fP  _{\alpha}) _{\alpha \in \Lambda}$ is affine).

For any 3uple $(\alpha, \, \beta,\, \gamma)\in \Lambda ^3$, fix 
$\X  _\alpha$ (resp. $\X  _{\alpha \beta}$, $\X  _{\alpha \beta \gamma}$)
some smooth formal $\S $-schemes lifting $X  _\alpha$
(resp. $X  _{\alpha \beta}$, $X  _{\alpha \beta \gamma}$),
$p _1 ^{\alpha \beta}$ :
$\X   _{\alpha \beta} \rightarrow \X  _{\alpha}$
(resp. $p _2 ^{\alpha \beta}$ :
$\X   _{\alpha \beta} \rightarrow \X  _{\beta}$)
some flat lifting 
of 
$X   _{\alpha \beta} \rightarrow X  _{\alpha}$
(resp. $X   _{\alpha \beta} \rightarrow X  _{\beta}$).

Similarly, for any $(\alpha,\,\beta,\,\gamma )\in \Lambda ^3$, fix some lifting 
$p _{12} ^{\alpha \beta \gamma}$ : $\X   _{\alpha \beta \gamma} \rightarrow \X   _{\alpha \beta} $,
$p _{23} ^{\alpha \beta \gamma}$ : $\X   _{\alpha \beta \gamma} \rightarrow \X   _{\beta \gamma} $,
$p _{13} ^{\alpha \beta \gamma}$ : $\X   _{\alpha \beta \gamma} \rightarrow \X   _{\alpha \gamma} $,
$p _1 ^{\alpha \beta \gamma}$ : $\X   _{\alpha \beta \gamma} \rightarrow \X   _{\alpha} $,
$p _2 ^{\alpha \beta \gamma}$ : $\X   _{\alpha \beta \gamma} \rightarrow \X   _{\beta} $,
$p _3 ^{\alpha \beta \gamma}$ : $\X   _{\alpha \beta \gamma} \rightarrow \X   _{\gamma} $,
$u _{\alpha}$ : $\X  _{\alpha } \hookrightarrow \fP  _{\alpha }$,
$u _{\alpha \beta}$ : $\X  _{\alpha \beta} \hookrightarrow \fP  _{\alpha \beta}$
and
$u _{\alpha \beta \gamma}$ : $\X  _{\alpha \beta \gamma } \hookrightarrow \fP  _{\alpha \beta \gamma}$.
\end{empt}

\begin{dfn}
\label{dfnMICdagalphabeta}
With notation \ref{ntnPPalpha-withoutdiv},
we define the category
  $\mathrm{MIC} ^\dag (Y,  (\X   _\alpha )_{\alpha \in \Lambda}/K)$ as follows.

\begin{itemize}
\item [-]
An object of   $\mathrm{MIC} ^\dag (Y,  (\X   _\alpha )_{\alpha \in \Lambda}/K)$ is a family $(E _\alpha) _{\alpha \in \Lambda}$
  of objects $E _\alpha $
  of $\mathrm{MIC} ^{\dag} ((Y _\alpha, X _\alpha,\X _\alpha)/K)$
together with a {\it glueing data}, i.e., a collection of isomorphisms in 
$\mathrm{MIC} ^{\dag} ((Y _{\alpha\beta}, X _{\alpha\beta},\X _{\alpha\beta})/K)$
 of the form
$ \eta _{  \alpha \beta} \ : \
p _{2 K}  ^{\alpha \beta *} (E _{\beta})
\riso
p  _{1 K} ^{\alpha \beta *} (E _{\alpha})$
satisfying the cocycle condition:
$\eta _{13} ^{\alpha \beta \gamma }=
\eta _{12} ^{\alpha \beta \gamma }
\circ
\eta _{23} ^{\alpha \beta \gamma }$,
where $\eta _{12} ^{\alpha \beta \gamma }$, $\eta _{23} ^{\alpha \beta \gamma }$
and $\eta _{13} ^{\alpha \beta \gamma }$ are defined so that the following diagrams 
\begin{equation}
  \label{diag1-defindonnederecolK}
\xymatrix  @R=0,3cm @C=0,45cm {
{  p _{12K} ^{\alpha \beta \gamma *} p  _{2K} ^{\alpha \beta *}  (E _\beta )}
\ar[r] ^-{\epsilon} _-{\sim}
\ar[d] ^-{p _{12K} ^{\alpha \beta \gamma *} (\eta _{\alpha \beta})} _-{\sim}
&
{p _{2K} ^{\alpha \beta \gamma*}  (E _\beta )}
\ar@{.>}[d] ^-{\eta _{12} ^{\alpha \beta \gamma }}
\\
{ p _{12K} ^{\alpha \beta \gamma *}  p  _{1K} ^{\alpha \beta *}  (E _\alpha)}
\ar[r]^-{\epsilon} _-{\sim}
&
{p _{1K} ^{\alpha \beta \gamma*}(E _\alpha),}
}
\xymatrix  @R=0,3cm @C=0,45cm {
{  p _{23K} ^{\alpha \beta \gamma *} p  _{2K} ^{\beta \gamma*}  (E _\gamma )}
\ar[r] ^-{\epsilon} _-{\sim}
\ar[d] ^-{p _{23K} ^{\alpha \beta \gamma *} (\eta _{ \beta \gamma})} _-{\sim}
&
{p _{3K} ^{\alpha \beta \gamma*}  (E _\gamma )}
\ar@{.>}[d] ^-{\eta _{23} ^{\alpha \beta \gamma }}
\\
{ p _{23K} ^{\alpha \beta \gamma *}  p  _{1K} ^{ \beta \gamma *}  (E _\beta)}
\ar[r]^-{\epsilon} _-{\sim}
&
{p _{2K} ^{\alpha \beta \gamma*}(E _\beta),}
}
\xymatrix  @R=0,3cm @C=0,45cm {
{  p _{13K} ^{\alpha \beta \gamma *} p  _{2K} ^{\alpha \gamma *}  (E _\gamma )}
\ar[r] ^-{\epsilon} _-{\sim}
\ar[d] ^-{p _{13K} ^{\alpha \beta \gamma *} (\eta _{\alpha \gamma})} _-{\sim}
&
{p _{3K} ^{\alpha \beta \gamma*}  (E _\gamma )}
\ar@{.>}[d]^-{\eta _{13} ^{\alpha \beta \gamma }}
\\
{ p _{13K} ^{\alpha \beta \gamma *}  p  _{1K} ^{\alpha \gamma *}  (E _\alpha)}
\ar[r]^-{\epsilon} _-{\sim}
&
{p _{1K} ^{\alpha \beta \gamma*}(E _\alpha),}
}
\end{equation}
where the isomorphisms $\epsilon$ are those of the form \ref{glueingisocntn-iso1},
are commutative.

\item [-]
A morphism
$f= (f _\alpha)_{\alpha \in \Lambda}$ :
$((E _{\alpha})_{\alpha \in \Lambda},\, (\eta _{\alpha\beta}) _{\alpha ,\beta \in \Lambda})
\rightarrow
((E '_{\alpha})_{\alpha \in \Lambda},\, (\eta '_{\alpha\beta}) _{\alpha ,\beta \in \Lambda})$
of $\mathrm{MIC} ^\dag (Y,  (\X   _\alpha )_{\alpha \in \Lambda}/K)$ 
is by definition a familly of 
morphisms $f _\alpha$ : $ E _\alpha \rightarrow E ' _\alpha$
commuting with glueing data.

\end{itemize}
\end{dfn}

\begin{prop}
\label{eqcat-iso-reco}
With notation \ref{dfnMICdagalphabeta}, 
there exists a canonical equivalence of categories
  $$u ^*  _{0K}
  \colon
  \mathrm{MIC} ^{\dag} (Y, X,\fP/K) 
    \cong
  \mathrm{MIC} ^\dag (Y,  (\X   _\alpha )_{\alpha \in \Lambda}/K).$$
\end{prop}

\begin{proof}
1) Let $\phi _\alpha : =(id,id, u _\alpha) 
\colon 
(Y _\alpha, X _\alpha, \X _\alpha)
\to 
(Y _\alpha, X _\alpha, \fP _\alpha)$ be the proper morphism of frames.
We remark that  $\phi _\alpha$ is the composition of the morphism
of frames
$\gamma _{u _\alpha}
\colon 
(Y _\alpha, X _\alpha, \X _\alpha)
\to 
(Y _\alpha, X _\alpha, \X _\alpha \times _\S \fP _\alpha)$
induced by the graph of $u _\alpha$
and of 
$p _2\colon (Y _\alpha, X _\alpha, \X _\alpha \times _\S \fP _\alpha)
\to 
(Y _\alpha, X _\alpha, \fP _\alpha)$ induced by the second projection.
Since $p _2$ is proper and smooth (see Definitions \cite[3.3.5 and 3.3.10]{LeStum-livreRigCoh}),
then following 
Theorem \cite[7.1.8]{LeStum-livreRigCoh}
we get the equivalence of categories
$p _{2K} ^* \colon 
\mathrm{MIC} ^{\dag} ((Y _\alpha, X _\alpha, \fP _\alpha)/K)
\cong
\mathrm{MIC} ^{\dag} ((Y _\alpha, X _\alpha, \X _\alpha \times _\S \fP _\alpha)/K)$.
Since $\gamma _{u _\alpha}$ has a retraction (the first projection), 
then using 
Corollary \cite[7.1.7]{LeStum-livreRigCoh}
we get that
$\gamma _{u _\alpha K} ^*$ is also an equivalence of categories. 
Hence, by composition, so is 
$\phi _{\alpha K} ^*
\colon 
\mathrm{MIC} ^{\dag} ((Y _\alpha, X _\alpha, \fP _\alpha)/K)
\cong
\mathrm{MIC} ^{\dag} ((Y _\alpha, X _\alpha, \X _\alpha)/K)$.

1') Let $\phi _{\alpha \beta} : =(id,id, u _{\alpha \beta}) 
\colon 
(Y _{\alpha \beta}, X _{\alpha \beta}, \X _{\alpha \beta})
\to 
(Y _{\alpha \beta}, X _{\alpha \beta}, \fP _{\alpha \beta})$ be the morphism of frames. 
Using the same arguments than in 1), we get the equivalence of categories
$\phi _{\alpha \beta K} ^*
\colon 
\mathrm{MIC} ^{\dag} ((Y _{\alpha \beta}, X _{\alpha \beta}, \fP _{\alpha \beta})/K)
\cong
\mathrm{MIC} ^{\dag} ((Y _{\alpha \beta}, X _{\alpha \beta}, \X _{\alpha \beta})/K)$.

2) Let $E _\fP  \in \mathrm{MIC} ^{\dag} (Y, X,\fP/K)$.  
Using the properties of the glueing isomorphisms \ref{glueingisocntn-iso1},
we get canonically on the object
  $ (\phi ^* _{\alpha K} ( E _\fP |_{]X _\alpha[ _{\fP _\alpha}})) _{\alpha \in \Lambda}$ a glueing data 
  making it an object 
  of $ \mathrm{MIC} ^\dag (Y,  (\X   _\alpha )_{\alpha \in \Lambda}/K)$.
The functoriality is obvious and 
this yields the canonical functor 
$u ^*  _{0K}
  \colon
  \mathrm{MIC} ^{\dag} (Y, X,\fP/K) 
    \cong
  \mathrm{MIC} ^\dag (Y,  (\X   _\alpha )_{\alpha \in \Lambda}/K)$.
Since 
$\phi _{\alpha  K} ^*$ is fully faithful 
and 
$\phi _{\alpha \beta K} ^*$ is faithful, we check easily that the functor
$u ^*  _{0K}$ is fully faithful. 
We check the essential surjectivity 
similarly to \cite[2.5.7]{caro-construction}.
\end{proof}

\subsection{Quasi-inverse equivalences of categories via $\sp _*$ and $\sp ^*$}
Let $\X $ be a smooth  formal scheme over $\S $.
Let $Z$ be a divisor of $X$  and $\Y $ the open subset of $\X $ complementary to the support of $Z$.

\begin{thm}
[Berthelot]
\label{thm-eqcat-cvisoc}
\begin{enumerate}
\item 
\label{thm-eqcat-cvisoc1}
The functor $\sp _*$ induces an equivalence of categories between the category of 
convergent isocrystals on $Y$, and the category of 
$\D ^\dag _{\Y, \Q} $-modules
which are $\O _{\Y,\Q} $-coherent.
Moreover, a $\D ^\dag _{\Y, \Q} $-module
which is $\O _{\Y,\Q} $-coherent
is also $\D ^\dag _{\Y, \Q} $-coherent
and 
$\O _{\Y,\Q} $-locally projective of finite type.

\item 
\label{thm-eqcat-cvisoc2}
Let $\E$ be a coherent 
$\D ^\dag _{\Y, \Q} $-module
which is $\O _{\Y,\Q} $-locally projective of finite type. We have the following properties.
\begin{enumerate}
\item For any $m \in \N$, there exists 
a (coherent) $\widehat{\D} ^{(m)} _{\Y} $-module
$\overset{\circ}{\E}$, coherent over $\O _\Y$ together with 
an isomorphism of $\widehat{\D} ^{(m)} _{\Y, \Q} $-modules
$\overset{\circ}{\E} _\Q \riso \E$.
\item 
\label{thm-eqcat-cvisoc-arrows}
The module $\E$ is $\D  _{\Y, \Q} $-coherent 
 and for any $m \in \N$
the canonical homomorphisms 
\begin{gather}
\notag
\E 
\to 
\widehat{\D} ^{(m)} _{\Y, \Q}  \otimes _{\D  _{\Y, \Q} }
\E
,
\
\
\E
\to 
\D ^\dag _{\Y, \Q} 
 \otimes _{\widehat{\D} ^{(m)} _{\Y, \Q} }
\E
\end{gather}
are isomorphisms.

\end{enumerate}

\end{enumerate}

\end{thm}

\begin{proof}
This is \cite[4.1.4]{Be1} and \cite[3.1.2 and 3.1.4]{Be0}.
\end{proof}

\begin{prop}
[Berthelot]
\label{445Be1}
The functor $\sp _*$ induces an exact  fully faithful functor from the category 
$\mathrm{MIC} ^{\dag} (Y, X,\X /K) $
to the category of 
$\D ^\dag _{\X } (\hdag Z) _\Q$-modules.
Its essential image consists in 
$\D ^\dag _{\X } (\hdag Z) _\Q$-modules
$\E$ such that there exists $n _0\in \N$, and 
a coherent $\widehat{\B} ^{(n_0)} _{\X} (Z) _\Q$-module $\E _0$
together with an isomorphism
\begin{equation}
\label{4451Be1}
\underrightarrow{\lim} _{n \geq n _0}
\widehat{\B} ^{(n)} _{\X} (Z) _\Q \otimes _{\widehat{\B} ^{(n _0)} _{\X} (Z) _\Q}
\E _0
\riso 
\E,
\end{equation}
satisfying the following condition:
For any $m \in \N$, there exists an integer $n _m \geq \max ( n _{m-1}, m)$
and a structure of 
$(\widehat{\B} ^{(n _m)} _{\X} (Z) 
\widehat{\otimes}
\widehat{\D} ^{(m)} _{\X /\S }) _\Q$-module
on 
$\widehat{\B} ^{(n _m)} _{\X} (Z) _\Q \otimes _{\widehat{\B} ^{(n _0)} _{\X} (Z) _\Q}
\E _0$
such that 
the homomorphisms
$$
\widehat{\B} ^{(n _m)} _{\X} (Z) _\Q \otimes _{\widehat{\B} ^{(n _0)} _{\X} (Z) _\Q}
\E _0
\to 
\widehat{\B} ^{(n _{m+1})} _{\X} (Z) _\Q \otimes _{\widehat{\B} ^{(n _0)} _{\X} (Z) _\Q}
\E _0
$$
are
$(\widehat{\B} ^{(n _m)} _{\X} (Z) 
\widehat{\otimes}
\widehat{\D} ^{(m)} _{\X /\S }) _\Q$-linear
and the isomorphism
\ref{4451Be1} is 
$\D ^\dag _{\X } (\hdag Z) _\Q$-linear.
\end{prop}

\begin{proof}
See \cite[4.4.5 and 4.4.12]{Be1}.
\end{proof}

\begin{lem}
\label{LetterLem2}
We suppose $\X$ affine.
Let $n \geq m \geq 0$ be two integers. 
Let $\E$ be a coherent 
$\widehat{\B} ^{(n)} _{\X} (Z) 
\widehat{\otimes}
\widehat{\D} ^{(m)} _{\X /\S }$-module. 
Then $\E$ is $\widehat{\B} ^{(n)} _{\X} (Z) $-coherent if and only if
$\Gamma (\X, \E) $ is a 
$\Gamma (\X, \widehat{\B} ^{(n)} _{\X} (Z) )$-module of finite type.
\end{lem}

\begin{proof}
Since we have theorem of type $A$ for 
coherent $\widehat{\B} ^{(n)} _{\X} (Z) $-modules,
then the $\widehat{\B} ^{(n)} _{\X} (Z) $-coherence of $\E$ implies 
that $\Gamma (\X, \E) $ is a 
$\Gamma (\X, \widehat{\B} ^{(n)} _{\X} (Z) )$-module of finite type.
Conversely, suppose
$\Gamma (\X, \E) $ is a 
$\Gamma (\X, \widehat{\B} ^{(n)} _{\X} (Z) )$-module of finite type.
Since we have  theorem of type $A$ for 
coherent 
$\widehat{\B} ^{(n)} _{\X} (Z) \widehat{\otimes}\widehat{\D} ^{(m)} _{\X /\S }$-modules,
then the canonical morphism
$$\widehat{\B} ^{(n)} _{\X} (Z) \widehat{\otimes}\widehat{\D} ^{(m)} _{\X /\S }
\otimes
_{\Gamma (\X,\widehat{\B} ^{(n)} _{\X} (Z) \widehat{\otimes}\widehat{\D} ^{(m)} _{\X /\S } )}
\Gamma (\X, \E)
\to 
\E$$
is an isomorphism.
Since 
$\Gamma (\X, \E)$ is of finite type over 
$\Gamma (\X,\widehat{\B} ^{(n)} _{\X} (Z) \widehat{\otimes}\widehat{\D} ^{(m)} _{\X /\S } )$
and over
$\Gamma (\X,\widehat{\B} ^{(n)} _{\X} (Z) )$, then (taking inverse limits
of the isomorphisms of the form \cite[2.3.5.2]{Be1} we check that)
the canonical morphism 
$$\widehat{\B} ^{(n)} _{\X} (Z) 
\otimes
_{\Gamma (\X,\widehat{\B} ^{(n)} _{\X} (Z) )} 
\Gamma (\X, \E)
\to 
\left ( \widehat{\B} ^{(n)} _{\X} (Z) \widehat{\otimes}\widehat{\D} ^{(m)} _{\X /\S }\right )
\otimes
_{\Gamma (\X,\widehat{\B} ^{(n)} _{\X} (Z) \widehat{\otimes}\widehat{\D} ^{(m)} _{\X /\S } )}
\Gamma (\X, \E)$$
 is an isomorphism.
\end{proof}

\begin{rem}
\label{projec-mdag}
Let $\E$ be a coherent $\widehat{\B} ^{(m _0)} _{\X} (Z) _\Q$-module. 
If $\E | \Y$ is a locally projective $\O _{\Y,\Q}$-module of finite type, then 
for $m \geq m_0$ large enough, 
$\widehat{\B} ^{(m)} _{\X} (Z) _\Q \otimes _{\widehat{\B} ^{(m _0)} _{\X} (Z) _\Q} \E$
is a locally projective $\widehat{\B} ^{(m)} _{\X} (Z) _\Q$-module of finite type.
Indeed, since $\O _{\X} (\hdag Z) _\Q\to j _* \O _{\Y,\Q}$ is faithfully flat (see \cite[4.3.10]{Be1}), 
then $\O _{\X} (\hdag Z) _\Q \otimes _{\widehat{\B} ^{(m _0)} _{\X} (Z) _\Q} \E$
is a projective $\O _{\X} (\hdag Z) _\Q$-module of finite type. 
We conclude using Proposition \cite[3.6.2]{Be1}.

\end{rem}

\begin{ntn}
\label{ntnMICdag2fs}
Let 
$\mathrm{MIC} ^{\dag \dag} (\X,Z/K)$
be the category of coherent 
$\D ^\dag _{\X } (\hdag Z) _\Q$-modules 
which are $\O _{\X} (\hdag Z) _\Q$-coherent.
When $Z$ is empty we remove it in the notation.
We have the morphism of ringed spaces 
$\sp \colon 
(\X _K , j ^\dag \O _{\X _K})
\to 
(\X, \O _{\X} (\hdag Z) _\Q)$
induced by the specialization morphism.
We get the inverse image functor $\sp ^*$ by setting 
$\sp ^* (\E ) := j ^\dag \O _{\X _K} \otimes _{\sp ^{-1} \widetilde{\O} _{\X,Q}}  \sp ^{-1} (\E)$,
for any $\E \in \mathrm{MIC} ^{\dag \dag} (\X,Z/K) $.

\end{ntn}

In order to prove Berthelot's Theorem \ref{letterBerthelotCaro2007},
we need the following unpublished two Berthelot's Lemmas which complete 
Theorem \ref{thm-eqcat-cvisoc}.

\begin{lem}
Let $\E ^{(m)}$ be a coherent $\widehat{\D} ^{(m)} _{\Y/\S, \Q}  $-module. 
For any $m' \geq m$, we set
$\E ^{(m')}:=
\widehat{\D} ^{(m')} _{\Y/\S, \Q}   
\otimes _{\widehat{\D} ^{(m)} _{\Y/\S, \Q}  }
\E ^{(m)}$,
and 
$\E := 
\D ^\dag _{\Y/\S, \Q} \otimes _{\widehat{\D} ^{(m)} _{\Y/\S, \Q}  }
\E ^{(m)}$.

If $\E$ is $\O _{\Y,\Q}$-coherent, then for $m'$ large enough
the canonical homomorphism 
$\E ^{(m')}
\to 
\E$
is an isomorphism.

\end{lem}

\begin{proof}
This is a consequence of Proposition \cite[3.6.2]{Be1} and of \ref{thm-eqcat-cvisoc}.2.
\end{proof}

\begin{lem}
\label{letterBerthelot-Lem3}
Let $\E$ be a coherent $\D ^\dag _{\Y/\S, \Q} $-module which is 
$\O _{\Y,\Q}$-coherent, and 
$\overset{\circ}{\E}$ be a 
coherent $\widehat{\D} ^{(m)} _{\Y/\S}$-module without $p$-torsion together with a 
$\widehat{\D} ^{(m)} _{\Y/\S,\Q}$-linear isomorphism of the form
$\E \riso \overset{\circ}{\E} _\Q$.
Then $\overset{\circ}{\E}$ is $\O _{\Y}$-coherent, and this is a 
locally topologically nilpotent $\widehat{\D} ^{(m)} _{\Y/\S}$-module.
\end{lem}

\begin{proof}
Since this is local, we can suppose $\Y$ affine and that $\Y/\S$ has local coordinates $t _1,\dots, t _d$.
Let $\partial _1, \dots, \partial _d$ be the induced derivations.
Following \ref{thm-eqcat-cvisoc}, there exists a 
coherent $\widehat{\D} ^{(m+1)} _{\Y/\S} $-module
$\FF$, coherent over $\O _\Y$ together with 
an isomorphism of $\widehat{\D} ^{(m+1)} _{\Y/\S, \Q} $-modules
$\FF _\Q \riso \E$.
Using \cite[3.4.4]{Be1}, we can suppose $\FF$ without $p$-torsion.
We get a homomorphism
$\overset{\circ}{\E}  \to \overset{\circ}{\E} _\Q \riso \FF _\Q$.
Multiplying this homomorphism by a power of $p$, we get
the injective $\widehat{\D} ^{(m)} _{\Y/\S} $-linear homomorphism
$\overset{\circ}{\E}  \hookrightarrow \FF$.
Using \ref{LetterLem2} in the case where the divisor is empty, we get the coherence of 
$\overset{\circ}{\E}$ over $\O _\Y$.
Since $\FF$ is $\O _\Y$-coherent,
for $r$ large enough 
we get $p ^r \FF \hookrightarrow \overset{\circ}{\E} \hookrightarrow \FF$ whose composition is the canonical inclusion.
For any $x \in \Gamma (\Y, \overset{\circ}{\E} )$, for any $\underline{k} \in \N ^d$, we get in 
$\Gamma (\Y,\FF)$ the formula
$$ \underline{\partial} ^{<\underline{k}> _{(m)}} x
=
\frac{q _{\underline{k}} ^{(m)} !}{q _{\underline{k}} ^{(m+1)} !}
\underline{\partial} ^{<\underline{k}> _{(m +1)}}x,$$
and $q _{\underline{k}} ^{(m)} !/q _{\underline{k}} ^{(m+1)} !$ 
converges $p$-adically to $0$ when $|\underline{k}|\to \infty$. Hence, we are done.
\end{proof}

\begin{thm}
[Berthelot]
\label{letterBerthelotCaro2007}
\begin{enumerate}
\item The functors $\sp _*$ and $\sp ^*$ induce quasi-inverse equivalences of categories between 
$\mathrm{MIC} ^{\dag} (Y, X,\X /K) $ and 
$\mathrm{MIC} ^{\dag \dag} (\X,Z/K) $.
\item Let $\E$ be a coherent $\D ^\dag _{\X } (\hdag Z) _\Q$-module.
Then $\E \in \mathrm{MIC} ^{\dag \dag} (\X,Z/K) $ if and only if 
$\E | \Y$ is $\O _{\Y,\Q}$-coherent. 
\end{enumerate}
\end{thm}

\begin{proof}
Since this Berthelot's proof is unpublished, let us give it below for the reader. 

I) If $E$ is an overconvergent isocrystal on $(Y,X,\X)/K$,
then 
following \cite[4.4.12]{Be1}, 
$ \sp _* (E)$ is $\D ^\dag _{\X } (\hdag Z) _\Q$-coherent. 
Using \ref{4451Be1},
we get the  $\O _{\X} (\hdag Z) _\Q$-coherence of $ \sp _* (E)$, i.e. 
$ \sp _* (E) \in \mathrm{MIC} ^{\dag \dag} (\X,Z/K) $.
This yields that the adjunction morphism
$\sp ^* \sp _* (E) \to E$ is an isomorphism.

II.1) Let $\E$ be a coherent $\D ^\dag _{\X } (\hdag Z) _\Q$-module such that 
$\E | \Y$ is $\O _{\Y,\Q}$-coherent. 
It remains to check that $\E$ is in the essential image of $\sp _*$.
For $n \geq m$, we set 
$\widehat{\D} ^{(m,n)} _{\X /\S } (Z):= 
\widehat{\B} ^{(n )} _{\X} (Z) 
\widehat{\otimes}
\widehat{\D} ^{(m)} _{\X /\S }$, 
and 
$\widehat{\D} ^{(m)} _{\X /\S } (Z):= 
\widehat{\D} ^{(m,m)} _{\X /\S } (Z)$. 
For $m _0$ large enough, 
there exists a coherent $\widehat{\D} ^{(m _0)} _{\X /\S } (Z) _\Q$-module $\E ^{(m_0)}$ together with a
$\D ^\dag _{\X } (\hdag Z) _\Q$-linear isomorphism
$\D ^\dag _{\X } (\hdag Z) _\Q \otimes _{\widehat{\D} ^{(m_0)} _{\X /\S } (Z) } \E ^{(m_0)}
\riso \E$ (use \cite[3.6.2]{Be1} and the isomorphism
$\underrightarrow{\lim} _m \widehat{\D} ^{(m)} _{\X /\S } (Z) 
\riso
\D ^\dag _{\X } (\hdag Z) _\Q$).
There exists  a coherent $\widehat{\D} ^{(m _0)} _{\X /\S } (Z)$-module $\G ^{(m _0)}$  without $p$-torsion together with 
a $\widehat{\D} ^{(m _0)} _{\X /\S } (Z) _\Q$-linear isomorphism 
$\G ^{(m _0)} _\Q\riso \E ^{(m_0)}$ (see \cite[3.4.5]{Be1}).
For any $n\geq m \geq m _0$, we put
$\G ^{(m,n)}
:=
\widehat{\D} ^{(m,n)} _{\X /\S } (Z) 
\otimes _{\widehat{\D} ^{(m_0)} _{\X /\S } (Z)}\G ^{(m_0)}/ \text{$p$-torsion}$,
$\G ^{(m)}
:=
\G ^{(m,m)}$.
Following \cite[3.4.4]{Be1}, 
$\G ^{(m,n)}$
is $\widehat{\D} ^{(m,n)} _{\X /\S } (Z) $-coherent. 
From \ref{thm-eqcat-cvisoc}, 
we get 
$\G ^{(m,n)} _\Q |\Y  \riso \E |\Y$.
From \ref{letterBerthelot-Lem3}, this yields that 
$\G ^{(m,n)} |\Y $ is $\O _{\Y}$-coherent.

II.2) We will now prove that for $n$ large enough $\G ^{(m,n)}$ is  $\widehat{\B} ^{(n)} _{\X} (Z)$-coherent. 
Since this is local, we can suppose
$\X = \Spf A$ is affine, there exist 
$f \in A$ such that $Z= \Spec \overline{A} / (\overline{f})$
and local coordinates 
$t _1,\dots, t _d\in A$ of 
$\X /\S $.
Let $\partial _{1},\dots, \partial _{d}$ be the induced  derivations.
Following \ref{LetterLem2}, we reduce to check that for $n$ large enough,
$\Gamma (\X, \G ^{(m,n)})$ is a 
$\Gamma (\X, \widehat{\B} ^{(n)} _{\X} (Z) )$-module of finite type.

Put 
$\D ^{(m)} _{X /S } (Z)
:=
\D ^{(m)} _{\X /\S } (Z)/
\pi \D ^{(m)} _{\X /\S } (Z)$,
$\smash{\overline{\G}} ^{(m)}:= 
\G ^{(m)} / \pi \G ^{(m)}$,
and 
$\smash{\overline{\G}} ^{(m,n)}:= 
\G ^{(m,n)} / \pi \G ^{(m,n)}$.
Let $\overline{x} _1, \dots, \overline{x} _r 
\in 
\Gamma (X, \smash{\overline{\G}} ^{(m _0)})$
which generate 
$\smash{\overline{\G}} ^{(m _0)}$ as 
$\D ^{(m_0)} _{X /S } (Z)$-module.

Fix $m \geq m _0$. 
From Lemma \ref{letterBerthelot-Lem3}, 
$\smash{\overline{\G}} ^{(m)}|Y$ is a 
nilpotent 
$\D ^{(m)} _{Y}$-module.
Hence, there exists $h \in \N$ large enough so that
we get in 
$\Gamma ( Y, \smash{\overline{\G}} ^{(m)})$ the relation
$$\forall i=1,\dots, r,
\forall j=1,\dots, d,
\forall l=1,\dots, m,
\
(\partial _{j} ^{[p ^l]}) ^h 
\cdot
\overline{x} _i 
=
0, 
$$
where by abuse of notation we still denote by 
$\overline{x} _i  $ 
(resp. $(\partial _{j} ^{[p ^l]}) ^h $) the image of $\overline{x} _i $ 
(resp. $(\partial _{j} ^{[p ^l]}) ^h $) via the canonical map 
$\Gamma (X, \smash{\overline{\G}} ^{(m _0)})\to 
\Gamma ( Y, \smash{\overline{\G}} ^{(m)})$
(resp. $\Gamma (X, \D ^{(m)} _{X /S } )
\to 
\Gamma (Y, \D ^{(m)} _{Y} )$).
Hence, for $n _m> m$ large enough, we get in 
$\Gamma (X, \smash{\overline{\G}} ^{(m)})$ the relation
$$\forall i=1,\dots, r,
\forall j=1,\dots, d,
\forall l=1,\dots, m,
\
\overline{f} ^{p ^{n _m}}
(\partial _{j} ^{[p ^l]}) ^h 
\cdot
\overline{x} _i 
=
0. 
$$
Fix such $n _m$.
Since $n _m > m$, then following
\cite[2.2.6]{Be1}
$\overline{f} ^{p ^{n _m}}$ is in the center of 
$\Gamma (X, \D ^{(m)} _{X /S } )$.
Let 
$P =
\prod _{j=1} ^{d}
\prod _{l=1} ^{m}
(\partial _{j} ^{[p ^l]}) ^{h _{jl}} \in \Gamma (X, \D ^{(m)} _{X /S } )$
where $h _{jl} \in \N$.
Since $\overline{f} ^{p ^{n _m}}$ is in the center of 
$\Gamma (X, \D ^{(m)} _{X /S } )$, 
if there exist
$j _0$ and $l _0$ such that
$h _{j_0 l _0} \geq h$, then 
we have in 
$\Gamma (X, \smash{\overline{\G}} ^{(m)})$
the relation
$\overline{f} ^{p ^{n _m}} P\cdot
\overline{x} _i 
=
0$,
for any $i $.

Let $x _1, \dots, x _r \in \Gamma (\X,\G ^{(m _0)})$ be some sections lifting
respectively 
$\overline{x} _1, \dots, \overline{x} _r $.
Let 
$P =
\prod _{j=1} ^{d}
\prod _{l=1} ^{m}
(\partial _{j} ^{[p ^l]}) ^{h _{jl}} \in 
\Gamma (\X, \D ^{(m)} _{\X /\S } )$
where $h _{jl} \in \N$ are such that
there exist
$j _0$ and $l _0$ satisfying
$h _{j _0 l _0} \geq h$.
Then, we get in 
$\Gamma (\X,\G ^{(m)})$ the relation
$$\forall i=1,\dots, r,
\
f ^{p ^{n _m}} P\cdot
x _i 
\in 
p 
\Gamma (\X,\G ^{(m)}),$$
where by abuse of notation we still denote by 
$x _i $ the image of $x _i$ via the canonical map 
$\Gamma (\X,\G ^{(m _0)}) \to \Gamma (\X,\G ^{(m)})$.
Let $T _{n _m -1} \in \widehat{\B} ^{(n _m -1)} _{\X} (Z) $
be the element such that 
$f ^{p ^{n _m}}T _{n _m -1} =p$.
Since the $\widehat{\B} ^{(n _m -1)} _{\X} (Z) $-module
$\G ^{(m,n _m -1)}$ has no $\pi$-torsion then 
it has no $f $-torsion. 
Hence, for such $P$, we get in 
$\Gamma (\X, \G ^{(m,n _m -1)})$: 
$$\forall i=1,\dots, r,
\
 P\cdot
x _i 
\in 
T _{n _m -1}
\Gamma (\X, \G ^{(m,n _m -1)}).$$

Let $y _1, \dots, y _s$ be the elements of the form
$\left (\prod _{j=1} ^{d}
\prod _{l=1} ^{m}
(\partial _{j} ^{[p ^l]}) ^{h _{jl}} \right)
\cdot
x _i $
where $h _{jl} \in \{ 0,\dots, h-1\}$ for any $j$ and 
$l$ (beware that these elements and their number depend on $m$).
Following \cite[2.2.5]{Be1},
$\Gamma (\X, \D ^{(m)} _{\X /\S } )$
is generated as $\Gamma (\X, \O _{\X})$-module (for its left or right structure) 
by the elements of the form
$\prod _{j=1} ^{d}
\prod _{l=1} ^{m}
(\partial _{j} ^{[p ^l]}) ^{h _{jl}}$,
where
$h _{jl} \in \N$.
Since $\overline{x} _1, \dots, \overline{x} _r $
generate 
$\smash{\overline{\G}} ^{(m _0)}$ as 
$\D ^{(m_0)} _{X /S } (Z)$-module,
then for any $n\geq n _m -1$,
$\Gamma (\X, \G ^{(m,n)})$
is generated as 
$\Gamma (\X,\widehat{\B} ^{(n)} _{\X} (Z) )$-module by 
$p\Gamma (\X, \G ^{(m,n)})$ and by
the elements of the forms
$\left (\prod _{j=1} ^{d}
\prod _{l=1} ^{m}
(\partial _{j} ^{[p ^l]}) ^{h _{jl}} \right)
\cdot
x _i $
where $h _{jl}\in \N$.
Since $T _{n _m -1}$ divides $p$, then we get 
$$\forall n \geq n _{m} -1,
\
\Gamma (\X, \G ^{(m,n)})
=
\sum _{i= 1} ^s
\Gamma (\X,\widehat{\B} ^{(n)} _{\X} (Z) )
\cdot y _i
+
T _{n _m -1}\Gamma (\X, \G ^{(m,n)}).$$
By iteration, this yields
$$\forall n \geq n _{m} -1,
\
\Gamma (\X, \G ^{(m,n)})
=
\sum _{i= 1} ^s
\Gamma (\X,\widehat{\B} ^{(n)} _{\X} (Z) )
\cdot y _i
+
T _{n _m -1} ^p\Gamma (\X, \G ^{(m,n)}).$$
For any $n \geq n _m$, we have
$T _{n _m -1} ^p
=
p ^{p-1}T _{n _m} $.
We get, 
$$\forall n \geq n _{m} ,
\
\Gamma (\X, \G ^{(m,n)})
=
\sum _{i= 1} ^s
\Gamma (\X,\widehat{\B} ^{(n)} _{\X} (Z) )
\cdot y _i
+
p\Gamma (\X, \G ^{(m,n)}).$$
Since 
$\Gamma (\X, \G ^{(m,n)})$ is $p$-adically separated and complete, this yields that
$\Gamma (\X, \G ^{(m,n)})$ is generated
as $\Gamma (\X,\widehat{\B} ^{(n)} _{\X} (Z) )$-module by 
$y _1,\dots, y _s$.

II.3) We can suppose that the sequence $(n _m) _m$ is increasing. 
Set $\E _m := \G ^{(m,n _m) }_\Q$. Following II.2), the $\widehat{\D} ^{(m,n _m)} _{\X /\S } (Z) _\Q $-module
$\E _m$ is $\widehat{\B} ^{(n _m)} _{\X} (Z) _\Q$-coherent.
Since $\E _{m _0} |\Y \riso\E _{m} |\Y$ is a 
coherent $\D ^\dag _{\Y,\Q}$-module which is $\O _{\Y,\Q}$-coherent, then 
it is locally projective of finite type over $\O _{\Y,\Q}$.
Hence, with the remark \ref{projec-mdag},
increasing $m_0$ is necessary, we can suppose 
$\E _{m _0}$ is a projective 
$\widehat{\B} ^{(n _{m _0})} _{\X} (Z) _\Q$-module of finite type.

Following 
\ref{445Be1},
it is sufficient to check that 
the canonical homomorphism
$$\widehat{\B} ^{(n _m)} _{\X} (Z) _\Q \otimes _{\widehat{\B} ^{(n _0)} _{\X} (Z) _\Q} \E _{m _0}  \to \E _m $$
is an isomorphism for any $m \geq m _0$. 
Fix $m\geq m _0$ and set $B _m := 
\Gamma (\X, \widehat{\B} ^{(n _m)} _{\X} (Z) _\Q)$,
$E _m := \Gamma (\X, \E _m)$,
$E _{m _0,m} := \Gamma (\X, \widehat{\B} ^{(n _m)} _{\X} (Z) _\Q \otimes _{\widehat{\B} ^{(n _0)} _{\X} (Z) _\Q} \E _{m _0} )$.
We get the morphism
$E _{m _0,m}
\to 
E _{m}$ 
of $B _m$-modules of finite type.

Following the part II.2) and its notations,
$\Gamma (\X, \G ^{(m,n _m)})$ is generated as 
$\Gamma (\X,\widehat{\B} ^{(n _m)} _{\X} (Z) )$-module by $y _1,\dots, y _s$.
We remark that $y _1,\dots, y _s \in E _{m _0}$. 
Hence, the morphism $E _{m _0,m}
\to 
E _{m}$
is 
surjective. 
After applying $B _m \to 
\Gamma (\Y,\O _{\Y,\Q})$
to the morphism 
$E _{m _0,m}
\to 
E _{m}$, we get an isomorphism.
Since 
$E _{m _0,m} $ is a projective $B _m$-module of finite type
and since
$B _m 
\to \Gamma (\Y,\O _{\Y,\Q})$
is injective, we get the injectivity of
$E _{m _0,m}
\to 
E _{m}$. We are done.
\end{proof}

\begin{prop}
\label{lem-projff}
Let $\E \in \mathrm{MIC} ^{\dag \dag} (\X,Z/K) $.
\begin{enumerate}
\item 
\label{lem-projff-it1}
If $\X$ is affine, then $\Gamma (\X, \E)$ is a projectif 
$\Gamma (\X, \O _{\X} (\hdag Z) _\Q)$-module of finite type.
\item 
\label{lem-projff-it2}
The object $\E$
is a locally projective $\O _{\X} (\hdag Z) _\Q$-module of finite type. 
\item 
\label{lem-projff-it3}
We have $\E =0$ if and only if there exists an open dense subset $\U $ of $\X $
such that $\E | \U =0$.
\end{enumerate}
\end{prop}

\begin{proof}
1) Following \ref{thm-eqcat-cvisoc}, $\E | \Y$ is a locally projective $\O _{\Y,\Q}$-module of finite type.
Suppose $\X$ affine. Then $\Y$ is affine and 
$\Gamma (\Y, \E)$ is a  projective $\Gamma (\Y, \O _{\Y,\Q})$-module of finite type.
Using theorem of type $A$ concerning 
coherent $\O _{\X} (\hdag Z) _\Q$-modules,
$\Gamma (\Y, \O _{\X} (\hdag Z) _\Q)
\otimes _{\Gamma (\X, \O _{\X} (\hdag Z) _\Q)}
\Gamma (\X, \E)
\to 
\Gamma (\Y, \E)$
is an isomorphism.
Since 
$\Gamma (\Y, \O _{\X} (\hdag Z) _\Q) 
=
\Gamma (\Y,  \O _{\Y,\Q})$,
since $\Gamma (\X, \O _{\X} (\hdag Z) _\Q)
\to 
\Gamma (\Y,  \O _{\Y,\Q})$ is faithfully flat (this is checked in the proof of \cite[4.3.10]{Be1}), 
then this implies that $\Gamma (\X, \E)$ 
is a projectif 
$\Gamma (\X, \O _{\X} (\hdag Z) _\Q)$-module of finite type.

2) Since $\E$ is a coherent $\O _{\X} (\hdag Z) _\Q$-module, then using 
theorem of type $A$, the second assertion is a consequence of the first one. 

3)  Since the third part is local in $\X$, we can suppose $\X$ is affine and that $\E$ is a direct summand 
(in the category of coherent $\O _{\X} (\hdag Z) _\Q$-modules)
of a free 
$\O _{\X} (\hdag Z) _\Q$-module $\cL$ of finite type. 
Suppose   $\E | \U =0$. Replacing $\U$ by a smaller open subset,
we reduce to the case where $\U$ is a principal open subset (i.e. given by a global section of $\X$).
Since $\Gamma (\X, \cL )\to\Gamma (\fU,  \cL)$ is injective, we get
$\Gamma( \fX,\E) = 0$. Using theorem of type $A$, this yields that 
$\E= 0$. The converse is obvious.
\end{proof}

\begin{prop}
\label{cohisosurcv}
We set $\D _{\X} (\hdag Z) _\Q : = \O _{\X , \Q} (\hdag Z )
\otimes _{\O _{\X , \Q }} \D _{\X /\S ,\Q}  $.
Let 
$\E \in \mathrm{MIC} ^{\dag \dag} (\X,Z/K) $.
Then $\E $ is $\D _{\X} (\hdag Z) _\Q$-coherent and
the canonical morphism
\begin{equation}\label{cohisosurcv2}
  \E \rightarrow 
  \D ^\dag _{\X } (\hdag Z) _\Q
   \otimes _{\D _{\X} (\hdag Z) _\Q} \E
\end{equation}
is an isomorphism.
\end{prop}

\begin{proof}
We can copy the proof of \cite[2.2.7]{caro_comparaison}.
\end{proof}

\begin{ntn}
\label{cvisoc-f*}
\begin{enumerate}
\item 
\label{cvisoc-f*1}
Similarly to \ref{defi-M-L},
we denote by $M (\O _{\Y} ^{(\bullet)})$ the category of
$\O _{\Y} ^{(\bullet)}$-modules.
We get a canonical functor 
$\mathrm{cst}
\colon 
M (\O _{\Y} )
\to 
M (\O _{\Y} ^{(\bullet)})$ defined by 
$\FF \mapsto \FF ^{(\bullet)}$ so that 
$\FF ^{(m)}
\to 
\FF ^{(m+1)}$
is the identity of 
$\FF$.
Since this functor is exact, 
this yields the t-exact functor
$\mathrm{cst}
\colon 
D (\O _{\Y} )
\to 
D (\O _{\Y} ^{(\bullet)})$.
Similarly to \ref{defi-M-L},
we define the notion of ind-isogenies  (resp. of lim-ind-isogenies) of
$M (\O _{\Y} ^{(\bullet)})$.
Similarly to \ref{nota-(L)Mcoh},
we define the category
$\underrightarrow{LM} _{\Q, \mathrm{coh}} ( \O ^{(\bullet)} _{\Y})$.
We remark that 
$\underrightarrow{LM} _{\Q, \mathrm{coh}} ( \O ^{(\bullet)} _{\Y})$
is the  subcategory of
$\underrightarrow{LM} _{\Q} ( \O ^{(\bullet)} _{\Y})$
consisting of objects which are locally isomorphic to an object of the form 
$\mathrm{cst} (\G)$ where $\G$ is a coherent $\O _{\Y}$-module
(use analogous versions of \cite[2.1.7 and 2.2.2]{caro-stab-sys-ind-surcoh}).

\item 
\label{cvisoc-f*2}
 Following notation \ref{ntnMICdag2fs}, 
we denote by 
$\mathrm{MIC} ^{\dag \dag}(\Y/\V)$ the category of $\D ^\dag _{\Y/\S, \Q} $-modules which are also
$\O _{\Y,\Q}$-coherent. Recall these objects are necessarily 
$\D ^\dag _{\Y, \Q} $-coherent,
and 
$\O _{\Y,\Q} $-locally projective of finite type.
We denote by 
$\mathrm{MIC}  ^{(\bullet)}(\Y/\V)$
the full subcategory of 
$\underrightarrow{LM} _{\Q, \mathrm{coh}} ( \widehat{\D} ^{(\bullet)} _{\Y/\S})$
consisting of objects 
$\E ^{(\bullet)}$ such that 
$\underrightarrow{\lim} \E ^{(\bullet)}$
are 
$\O _{\Y,\Q}$-coherent.

\end{enumerate}

\end{ntn}

\begin{rem}
\label{remMICY/K}
Let $\E \in \mathrm{MIC} ^{\dag \dag}(\Y/\V)$.
Let 
$\widetilde{\D} := \D ^\dag _{\Y, \Q}$ or 
$\widetilde{\D} := \widehat{\D} ^{(m)} _{\Y, \Q}$. 
Let 
$\D := \D _{\Y, \Q}$ or 
$\D := \widehat{\D} ^{(0)} _{\Y, \Q}$.
By using the isomorphisms of  \ref{thm-eqcat-cvisoc}.\ref{thm-eqcat-cvisoc-arrows}, 
we check that both morphisms 
$\E 
\to 
\widetilde{\D} \otimes _{\D  }
\E
\to 
\E$
are isomorphisms.
This yields that the first morphism 
is in fact
$\widetilde{\D}$-linear. 
Hence, if $\FF$ is a $\widetilde{\D}$-module, 
then any 
$\D $-linear morphism
$\E \to \FF$ is necessarily 
$\widetilde{\D}$-linear. 
\end{rem}

\begin{lem}
\label{remMICY/K2}
Let 
$\FF ^{(m)}$ be a coherent $\widehat{\D} ^{(m)} _{\Y} $-module
et $f \colon \FF ^{(m)} \to \FF ^{(m)}$ be a $\V$-linear morphism
such that $f _\Q \colon \FF ^{(m)} _\Q\to \FF ^{(m)} _\Q$ is equal to $p ^N id$ for some $N \in \N$.
Then, for $N' \in \N$ large enough, we have
$p ^{N'} f = p ^{N'+N} id$.
\end{lem}

\begin{proof}
Since $\X$ is quasi-compact and
$\FF ^{(m)}$ is a coherent $\widehat{\D} ^{(m)} _{\Y} $-module,
then
the $p$-torsion part of  
$\FF ^{(m)}$ is killed by some power of $p$.
Hence, we are done.
\end{proof}

\begin{prop}
\label{MICbullet-lem}
Let $\E \in \mathrm{MIC} ^{\dag \dag}(\Y/\V)$.
Let 
$\FF ^{(0)}$ be a $\widehat{\D} ^{(0)} _{\Y} $-module, coherent over $\O _\Y$ together with 
an isomorphism of $\widehat{\D} ^{(0)} _{\Y, \Q} $-modules of the form
$\FF ^{(0)} _\Q \riso \E$.
For any $m\in \N$, let $\G ^{(m)} $ be the quotient of 
$\widehat{\D} ^{(m)} _{\Y/\S} \otimes _{\widehat{\D} ^{(0)} _{\Y/\S}} \FF ^{(0)}$ by its $p$-torsion part. 
The following conditions are satisfied. 
\begin{enumerate}

\item 
The module $\G ^{(m) }$ 
is $\O _{\Y}$-coherent.

\item The first (resp. second) canonical morphism
$$\FF ^{(0)} \to 
\widehat{\D} ^{(m)} _{\Y} \otimes _{\widehat{\D} ^{(0)} _{\Y} } \FF ^{(0)}
\to \G ^{(m)}$$
is an isogeny in the category of
$\widehat{\D} ^{(0)} _{\Y} $-modules
(resp. of coherent $\widehat{\D} ^{(m)} _{\Y} $-modules).

\item $\widehat{\D} ^{(\bullet)} _{\Y/\S} \otimes _{\widehat{\D} ^{(0)} _{\Y/\S}} \FF ^{(0)}
\in 
\mathrm{MIC}  ^{(\bullet)}(\Y/\V)$
and 
$ \underrightarrow{\lim} 
\,
(\widehat{\D} ^{(\bullet)} _{\Y/\S} \otimes _{\widehat{\D} ^{(0)} _{\Y/\S}} \FF ^{(0)})
\riso 
\E$.

\end{enumerate}

\end{prop}

\begin{proof}
1)
Following \ref{remMICY/K},
the canonical morphism 
$\E 
\to 
\widehat{\D} ^{(m)} _{\Y,\Q} \otimes _{\widehat{\D} ^{(0)} _{\Y,\Q} } \E$ 
is an isomorphism of $\widehat{\D} ^{(m)} _{\Y,\Q}$-modules.
The isomorphism
$\FF ^{(0)} _\Q \riso \E$
of $\widehat{\D} ^{(0)} _{\Y, \Q} $-modules 
induces 
the last isomorphism
of $\widehat{\D} ^{(m)} _{\Y, \Q} $-modules
$\G ^{(m)} _\Q 
\riso 
\widehat{\D} ^{(m)} _{\Y,\Q} \otimes _{\widehat{\D} ^{(0)} _{\Y,\Q} } \FF ^{(0)} _\Q
\riso 
\widehat{\D} ^{(m)} _{\Y,\Q} \otimes _{\widehat{\D} ^{(0)} _{\Y,\Q} } \E$.
Using \ref{letterBerthelot-Lem3}, 
this yields $\G ^{(m)}$ is $\O _{\Y}$-coherent.

2) a) Let us denote by 
$\alpha ^{(m)}
\colon 
\widehat{\D} ^{(m)} _{\Y} \otimes _{\widehat{\D} ^{(0)} _{\Y} } \FF ^{(0)}
\to \G ^{(m)}$ the canonical
epimorphism of coherent 
$\widehat{\D} ^{(m)} _{\Y} $-modules.
Since  $\alpha ^{(m)}$ is a morphism of coherent 
$\widehat{\D} ^{(m)} _{\Y} $-modules which is
an isomorphism after tensoring by $\Q$, 
then this is an isogeny 
in the category of $\widehat{\D} ^{(m)} _{\Y} $-modules
(use \cite[3.4.5]{Be1}).

b) Let
$\iota ^{(m)}
\colon 
\FF ^{(0)} \to 
\widehat{\D} ^{(m)} _{\Y} \otimes _{\widehat{\D} ^{(0)} _{\Y} } \FF ^{(0)}$
be the canonical morphism. 
It remains to check that $\iota ^{(m)}$ is an isogeny.
Using 2a) in the case $m=0$, we get a morphism 
$\beta ^{(0)} \colon 
\G ^{(0)} \to \FF ^{(0)}$ of 
coherent $\widehat{\D} ^{(0)} _{\Y} $-modules 
such that 
$\alpha ^{(0)} \circ \beta ^{(0)} 
=
p ^N id$ 
and 
$\beta ^{(0)} \circ \alpha ^{(0)} 
=
p ^N id$ for some integer $N$. 
Since the canonical morphism
$\E \to \widehat{\D} ^{(m)} _{\Y,\Q} \otimes _{\widehat{\D} ^{(0)} _{\Y,\Q} } \E$
is an isomorphism, then 
the canonical morphism
$\G ^{(0)} _\Q \to \widehat{\D} ^{(m)} _{\Y,\Q} \otimes _{\widehat{\D} ^{(0)} _{\Y,\Q} } \G ^{(0)} _\Q$
is an isomorphism. 
Since the canonical morphism
$\widehat{\D} ^{(m)} _{\Y,\Q} \otimes _{\widehat{\D} ^{(0)} _{\Y,\Q} } \G ^{(0)} _\Q
\to 
\G ^{(m)} _\Q$ is an isomorphism,
this yields by composition that the canonical $\widehat{\D} ^{(0)} _{\Y} $-linear morphism 
$\G ^{(0)}
\to \G ^{(m)}$
is an isomorphism after tensoring by $\Q$.
Let us denote by 
$\gamma ^{(m)}\colon \G ^{(0)}
\to \G ^{(m)}$ this morphism
and by 
$\gamma ^{(m)} _\Q
\colon 
\G ^{(0)} _\Q
\riso 
\G ^{(m)} _\Q$ 
the induced isomorphism.
Since $\gamma ^{(m)}$ is $\widehat{\D} ^{(0)} _{\Y} $-linear (and then 
$\O _{\Y}$-linear), since 
$\G ^{(m)}$ is $\O _{\Y}$-coherent and 
$\G ^{(0)}$ has no $p$-torsion, then,
for $N'$ large enough, 
$p ^{N'} (\gamma ^{(m)} _\Q ) ^{-1}$ induces the 
 morphism 
$\delta ^{(m)}\colon \G ^{(m)}
\to \G ^{(0)}$
of $\widehat{\D} ^{(0)} _{\Y} $-modules. 
We get 
$  \kappa ^{(m)}:= \beta ^{(0)} \circ \delta ^{(m)} \circ \alpha ^{(m)}
\colon 
\widehat{\D} ^{(m)} _{\Y} \otimes _{\widehat{\D} ^{(0)} _{\Y} } \FF ^{(0)}
\to 
\FF ^{(0)}$.
Using the Lemma \ref{remMICY/K2}, increasing $N$ or $N'$ if necessary in the construction of 
$ \kappa ^{(m)}$,  we get 
$\iota ^{(m)} \circ \kappa ^{(m)} 
=
p ^{N+N'} id$ 
and 
$\kappa ^{(m)} \circ \iota ^{(m)} 
=
p ^{N+N'} id$. 
Hence,  this morphism
$\iota ^{(m)}\colon \FF ^{(0)}
\to \G ^{(m)}$
is an isogeny in the category of
$\widehat{\D} ^{(0)} _{\Y} $-modules.

3) Finally 
$\underrightarrow{\lim} 
\,
(\widehat{\D} ^{(\bullet)} _{\Y/\S} \otimes _{\widehat{\D} ^{(0)} _{\Y/\S}} \FF ^{(0)})
\riso 
\D ^\dag _{\Y/\S, \Q} 
 \otimes _{\widehat{\D} ^{(0)} _{\Y/\S},\Q} \FF ^{(0)}_\Q$.
 Following the first part, 
 the canonical morphism
 $\FF ^{(0)}_\Q
 \to 
  \D ^\dag _{\Y/\S, \Q} 
 \otimes _{\widehat{\D} ^{(0)} _{\Y/\S},\Q} \FF ^{(0)}_\Q$
 is an isomorphism.
We endow 
$\FF ^{(0)} _\Q $ with the structure of 
$  \D ^\dag _{\Y/\S, \Q} $-module making $  \D ^\dag _{\Y/\S, \Q} $-linear
the isomorphism
$\FF ^{(0)} _\Q \riso \E$.
Following the remark 
 \ref{remMICY/K}, 
 this yields that 
 the canonical isomorphism
  $\FF ^{(0)}_\Q
 \to 
  \D ^\dag _{\Y/\S, \Q} 
 \otimes _{\widehat{\D} ^{(0)} _{\Y/\S},\Q} \FF ^{(0)}_\Q$
is in fact $  \D ^\dag _{\Y/\S, \Q} $-linear. 
Hence, we get the isomorphism 
$ \D ^\dag _{\Y/\S, \Q} 
 \otimes _{\widehat{\D} ^{(0)} _{\Y/\S},\Q} \FF ^{(0)}_\Q
 \riso \E$ of $ \mathrm{MIC} ^{\dag \dag}(\Y/\V)$.
\end{proof}

\begin{cor}
\label{MICbullet-prop}
Let $\E ^{(\bullet)}
\in
\underrightarrow{LM} _{\Q} ( \widehat{\D} ^{(\bullet)} _{\Y/\S})$.
The object 
$\E ^{(\bullet)}$ belongs to 
$\mathrm{MIC}  ^{(\bullet)}(\Y/\V)$ if and only if the following condition is satisfied: 
There exists  
a $\widehat{\D} ^{(0)} _{\Y/\S}$-module $\FF ^{(0)}$,
coherent over $\O _\Y$ 
such that 
$\widehat{\D} ^{(\bullet)} _{\Y/\S} \otimes _{\widehat{\D} ^{(0)} _{\Y/\S}} \FF ^{(0)}$
is isomorphic in 
$\underrightarrow{LM} _{\Q} ( \widehat{\D} ^{(\bullet)} _{\Y/\S})$ to 
$\E ^{(\bullet)}$
and such that 
the canonical morphism 
$\mathrm{cst} (\FF ^{(0)})
\to 
\widehat{\D} ^{(\bullet)} _{\Y/\S} \otimes _{\widehat{\D} ^{(0)} _{\Y/\S}} \FF ^{(0)}$
is an ind-isogeny in  
$M ( \O ^{(\bullet)} _{\Y})$.
Moreover, when 
$\E ^{(\bullet)}
\in
\mathrm{MIC}  ^{(\bullet)}(\Y/\V)$,
we can choose such $\FF ^{(0)}$ without $p$-torsion.
\end{cor}

\begin{proof}
Let $\E := \underrightarrow{\lim} \E ^{(\bullet)}$
be the corresponding $\D ^\dag _{\Y/\S, \Q} $-module.
If such $\FF ^{(0)}$ exists, then 
$\E$ is in particular 
$\O _{\Y,\Q}$-coherent and then by definition 
$\E ^{(\bullet)}
\in
\mathrm{MIC}  ^{(\bullet)}(\Y/\V)$. 
Conversely, 
suppose 
$\E ^{(\bullet)}
\in
\mathrm{MIC}  ^{(\bullet)}(\Y/\V)$.
Then $\E$ is 
a $\D ^\dag _{\Y/\S, \Q} $-module which is also
$\O _{\Y,\Q}$-coherent.
Using \ref{thm-eqcat-cvisoc}, 
there exists 
a coherent $\widehat{\D} ^{(0)} _{\Y/\S}$-module $\FF ^{(0)}$ without $p$-torsion,
coherent over $\O _\Y$
together with 
an isomorphism of $\widehat{\D} ^{(0)} _{\Y, \Q} $-modules
$\E \riso \FF ^{(0)} _\Q$.
Hence, we conclude by using \ref{MICbullet-lem}.
\end{proof}

\begin{empt}
Let $f \colon \Y ' \to \Y $ be a morphism of smooth formal $\V$-schemes.
Let 
$\E ^{(\bullet)}
\in 
M( \widehat{\D} ^{(\bullet)} _{\Y/\S})$.
We set
$f ^{*(m)} _{\mathrm{alg}}(\E ^{(m)})
:=
\smash{\widehat{\D}} ^{(m)} _{\Y ^{\prime } \rightarrow \Y /\S }
\otimes   _{f ^{-1} \smash{\widehat{\D}} ^{(m)} _{\Y /\S } }
f ^{-1} \E ^{(m)}$.
We denote by 
$f ^{*(\bullet)} _{\mathrm{alg}}(\E ^{(\bullet)}):=
\smash{\widehat{\D}} ^{(\bullet)} _{\Y ^{\prime } \rightarrow \Y /\S }
\otimes   _{f ^{-1} \smash{\widehat{\D}} ^{(\bullet)} _{\Y /\S } }
f ^{-1} \E ^{(\bullet)}$
the object of 
$M( \widehat{\D} ^{(\bullet)} _{\Y'/\S})$
whose transition morphisms are
$f ^{*(m)} _{\mathrm{alg}} ( \E ^{(m)})
\to 
f ^{*(m+1)} _{\mathrm{alg}} ( \E ^{(m+1)})$.
By left deriving the functor
$f ^{*(\bullet)} _{\mathrm{alg}}$, 
this yields the functor
$\L f ^{*(\bullet)} _{\mathrm{alg}}
\colon 
D ^- ( \widehat{\D} ^{(\bullet)} _{\Y/\S})
\to 
D ^- ( \widehat{\D} ^{(\bullet)} _{\Y'/\S})$, 
defined by setting 
$\L f ^{*(\bullet)} _{\mathrm{alg}}
(\cF ^{(\bullet)}):=
\smash{\widehat{\D}} ^{(\bullet)} _{\Y ^{\prime } \rightarrow \Y /\S }
\otimes   ^\L _{f ^{-1} \smash{\widehat{\D}} ^{(\bullet)} _{\Y /\S } }
f ^{-1} \cF ^{(\bullet)}$
for any
$\cF ^{(\bullet)} \in 
D ^- ( \widehat{\D} ^{(\bullet)} _{\Y/\S})$.
Since it preserves lim-ind-isogenies, 
this induces the functor
$\L f ^{*(\bullet)} _{\mathrm{alg}}
\colon 
\underrightarrow{LD} _{\Q}  ^{-} (\widehat{\D} ^{(\bullet)} _{\Y/\S} )
\to 
\underrightarrow{LD} _{\Q}  ^{-} (\widehat{\D} ^{(\bullet)} _{\Y'/\S} )$.

Following notation \ref{ntn-Lf!+*},
we set 
$\L f ^{*(\bullet)} ( \cF ^{(\bullet)}) :=
\smash{\widehat{\D}} ^{(\bullet)} _{\Y ^{\prime } \rightarrow \Y /\S }
\smash{\widehat{\otimes}} ^\L _{f ^{-1} \smash{\widehat{\D}} ^{(\bullet)} _{\Y /\S } }
f ^{-1} \cF ^{(\bullet)}$,
for any
$\cF ^{(\bullet)} \in 
\underrightarrow{LD} _{\Q,\mathrm{qc}}  ^{\mathrm{b}} (\widehat{\D} ^{(\bullet)} _{\Y/\S} )$.
Beware the notation is slightly misleading since $\L f ^{*(\bullet)}$
is not necessarily the left derived functor of a functor.
We get the morphism
$\L f ^{*(\bullet)} _{\mathrm{alg}} (\cF ^{(\bullet)})
\to 
\L f ^{*(\bullet)} ( \cF ^{(\bullet)})$

\end{empt}

\begin{lem}
\label{f*-OcohDm}
Let $f \colon \Y ' \to \Y $ be a morphism of smooth formal $\V$-schemes. 
We have the following properties.
\begin{enumerate}
\item Let 
$\FF ^{(\bullet)} \in 
\underrightarrow{LD} _{\Q,\mathrm{qc}}  ^{\mathrm{b}} (\widehat{\D} ^{(\bullet)} _{\Y/\S} )$.
The canonical morphism
$$\O _{\Y ^{\prime } }  ^{(\bullet)}
\smash{\widehat{\otimes}} ^\L _{f ^{-1}\O   ^{(\bullet)} _{\Y /\S }}
f ^{-1}\FF ^{(\bullet)}
\to
\smash{\widehat{\D}} ^{(\bullet)} _{\Y ^{\prime } \rightarrow \Y /\S }
\smash{\widehat{\otimes}} ^\L _{f ^{-1} \smash{\widehat{\D}} ^{(\bullet)} _{\Y /\S } }
f ^{-1} \FF ^{(\bullet)}$$
is an isomorphism.

\item Let 
$\FF ^{(\bullet)} \in 
\underrightarrow{LD} _{\Q,\mathrm{coh}}  ^{\mathrm{b}} (\widehat{\D} ^{(\bullet)} _{\Y/\S} )$.
The canonical morphism
$$
\L f ^{*(\bullet)} _{\mathrm{alg}} (\FF ^{(\bullet)}):=
\smash{\widehat{\D}} ^{(\bullet)} _{\Y ^{\prime } \rightarrow \Y /\S }
\otimes ^\L _{f ^{-1} \smash{\widehat{\D}} ^{(\bullet)} _{\Y /\S } }
f ^{-1} \FF ^{(\bullet)}
\to
\smash{\widehat{\D}} ^{(\bullet)} _{\Y ^{\prime } \rightarrow \Y /\S }
\smash{\widehat{\otimes}} ^\L _{f ^{-1} \smash{\widehat{\D}} ^{(\bullet)} _{\Y /\S } }
f ^{-1} \FF ^{(\bullet)}
=:\L f ^{*(\bullet)} (\FF ^{(\bullet)})$$
is an isomorphism.

\item 
\label{f*-OcohDm3}
Let 
$\G ^{(\bullet)} \in 
\underrightarrow{LD} _{\Q,\mathrm{coh}}  ^{\mathrm{b}} ( \O ^{(\bullet)} _{\Y})$.
Then, 
the canonical morphism
$$\O  ^{(\bullet)} _{\Y ^{\prime } }
\otimes  ^\L _{f ^{-1} \O   ^{(\bullet)} _{\Y /\S } }
f ^{-1}\G ^{(\bullet)}
\to 
\O  ^{(\bullet)} _{\Y ^{\prime } }
\smash{\widehat{\otimes}} ^\L _{f ^{-1}\O  ^{(\bullet)} _{\Y /\S }}
f ^{-1}\G ^{(\bullet)}$$
is an isomorphism of 
$\underrightarrow{LD} _{\Q,\mathrm{coh}}  ^{\mathrm{b}} ( \O ^{(\bullet)} _{\Y'})$.

\end{enumerate}
\end{lem}

\begin{proof}
This is left to the reader and easy (hint : to check 1) use \cite[2.3.5.2]{Be1}, and
the proof of 2) and 3) is identical to that of  \cite[3.4.2.2]{Beintro2}).
\end{proof}

\begin{prop}
\label{ntn-f*}
Let $f \colon \Y ' \to \Y $ be a morphism of smooth formal $\V$-schemes. 

\begin{enumerate}
\item 
\label{ntn-f*1}
Let $\E \in \mathrm{MIC} ^{\dag \dag}(\Y/\V)$. Then the canonical last morphism
$$
\O _{\Y ^{\prime},\Q}
\otimes  _{ f ^{-1} \O _{\Y /\S,\Q}} f ^{-1} \E
\liso 
\O _{\Y ^{\prime},\Q}
\otimes ^{\L} _{ f ^{-1} \O _{\Y /\S,\Q}} f ^{-1} \E
\to
\D ^{\dag} _{\Y ^{\prime }\rightarrow \Y ,\Q}
\otimes ^{\L} _{ f ^{-1} \D ^{\dag} _{\Y /\S,\Q}} f ^{-1} \E$$
is an isomorphism.
Hence, we can set $f ^* (\E):= \D ^{\dag} _{\Y ^{\prime }\rightarrow \Y ,\Q}
\otimes  _{ f ^{-1} \D ^{\dag} _{\Y /\S,\Q}} f ^{-1} \E$ without ambiguity. 
We have also $f ^* (\E) \in \mathrm{MIC} ^{\dag \dag}(\Y'/\V)$.

\item \label{ntn-f*2}
Let 
$\FF $ be a $\widehat{\D} ^{(m)} _{\Y} $-module, coherent over $\O _\Y$.
Then the morphisms 
\begin{equation}
\notag
\O _{\Y ^{\prime}}
\otimes  _{ f ^{-1} \O _{\Y /\S}} f ^{-1} \FF
\to 
\O _{\Y ^{\prime}}
\widehat{\otimes}  _{ f ^{-1} \O _{\Y /\S}} f ^{-1} \FF
\to 
\smash{\widehat{\D}} ^{(m)} _{\Y ^{\prime } \rightarrow \Y /\S }
\smash{\widehat{\otimes}}  _{f ^{-1} \smash{\widehat{\D}} ^{(m)} _{\Y /\S } }
f ^{-1} \FF 
\leftarrow
\smash{\widehat{\D}} ^{(m)} _{\Y ^{\prime } \rightarrow \Y /\S }
\otimes  _{f ^{-1} \smash{\widehat{\D}} ^{(m)} _{\Y /\S } }
f ^{-1} \FF 
\end{equation}
are isomorphisms.
Hence, we can set $f ^* (\FF):= 
\smash{\widehat{\D}} ^{(m)} _{\Y ^{\prime } \rightarrow \Y /\S }
\otimes  _{f ^{-1} \smash{\widehat{\D}} ^{(m)} _{\Y /\S } }
f ^{-1} \FF $ 
without ambiguity.
Moreover, 
$f ^* (\FF)$ is  a $\widehat{\D} ^{(m)} _{\Y'} $-module, coherent over $\O _{\Y'}$.
\end{enumerate}

\end{prop}

\begin{proof}
1) Following \ref{MICbullet-prop}, there exists
$\E ^{(\bullet)} \in 
\underrightarrow{LM} _{\Q,\mathrm{coh}}   (\widehat{\D} ^{(\bullet)} _{\Y/\S} )
\cap 
\underrightarrow{LD} _{\Q,\mathrm{coh}}   (\O ^{(\bullet)} _{\Y/\S} )$ 
such that 
$\underrightarrow{\lim} (\E ^{(\bullet)}) \riso \E$. 
By applying Lemma \ref{f*-OcohDm}, this yields that the canonical morphism
$$\O  ^{(\bullet)} _{\Y ^{\prime } }
\otimes  ^\L _{f ^{-1} \O   ^{(\bullet)} _{\Y /\S } }
f ^{-1}\E ^{(\bullet)}
\to
\smash{\widehat{\D}} ^{(\bullet)} _{\Y ^{\prime } \rightarrow \Y /\S }
\otimes ^\L _{f ^{-1} \smash{\widehat{\D}} ^{(\bullet)} _{\Y /\S } }
f ^{-1} \E ^{(\bullet)}$$
is an isomorphism. 
By applying the functor $\underrightarrow{\lim} $, we get the desired isomorphism. 
Since 
$f ^* (\E) $ is $\O _{\Y ',\Q}$-coherent, this yields
$f ^* (\E) \in \mathrm{MIC} ^{\dag \dag}(\Y'/\V)$.

2) Since $\FF $ is both $\widehat{\D} ^{(m)} _{\Y} $-coherent and  $\O _\Y$-coherent,
the first and the last morphisms are isomorphisms. 
Since the modulo $\pi ^{n+1}$ reduction of the middle morphism is an isomorphism for any $n\in \N$
(see \cite[2.3.5.2]{Be1}), since this is a morphism of separated complete modules for the $p$-adic topology, 
this implies that the middle morphism is an isomorphism.
\end{proof}

\begin{prop}
\label{stab-MIC-f*}
Let $f \colon \Y ' \to \Y $ be a morphism of smooth formal $\V$-schemes. 
Let $\FF ^{(0)}$ 
be a $\widehat{\D} ^{(0)} _{\Y/\S}$-module,
coherent over $\O _\Y$ 
and such that 
the canonical morphism 
$\mathrm{cst} (\FF ^{(0)})
\to 
\widehat{\D} ^{(\bullet)} _{\Y/\S} \otimes _{\widehat{\D} ^{(0)} _{\Y/\S}} \FF ^{(0)}=:\FF ^{(\bullet)}$
is an ind-isogeny in  
$M ( \O ^{(\bullet)} _{\Y})$.
For any $m\in\N$,
let $\G ^{(m)} $ be the quotient of 
$\widehat{\D} ^{(m)} _{\Y/\S} \otimes _{\widehat{\D} ^{(0)} _{\Y/\S}} \FF ^{(0)}$ by its $p$-torsion part. 

\begin{enumerate}

\item 
The canonical morphism 
$\mathrm{cst} (f ^* (\FF ^{(0)}))
\to 
\widehat{\D} ^{(\bullet)} _{\Y'/\S} \otimes _{\widehat{\D} ^{(0)} _{\Y'/\S}} f ^* (\FF ^{(0)})$
is an ind-isogeny of  
$M ( \O ^{(\bullet)} _{\Y'})$.

\item 
The canonical morphisms
$f ^{*(\bullet)} _{\mathrm{alg}} (\FF ^{(\bullet)})
\to
f ^{*(\bullet)} _{\mathrm{alg}} (\G ^{(\bullet)})$,
and 
$\widehat{\D} ^{(\bullet)} _{\Y'/\S} \otimes _{\widehat{\D} ^{(0)} _{\Y'/\S}} f ^* (\FF ^{(0)})
\to 
f ^{*(\bullet)} _{\mathrm{alg}} (\G ^{(\bullet)})$
are ind-isogenies of 
$M (\widehat{\D} ^{(\bullet)} _{\Y'/\S} )$.
\item The canonical morphisms
$\L f ^{*(\bullet)} _{\mathrm{alg}} (\FF ^{(\bullet)})
\to 
\L f ^{*(\bullet)} ( \FF ^{(\bullet)})$
and 
$\L f ^{*(\bullet)} _{\mathrm{alg}} (\FF ^{(\bullet)}) 
\to 
 f ^{*(\bullet)} _{\mathrm{alg}} (\FF ^{(\bullet)})$ 
are isomorphisms of
$\underrightarrow{LD} ^\mathrm{b} _{\Q} ( \widehat{\D} ^{(\bullet)} _{\Y'/\S})$.
\end{enumerate}
\end{prop}

\begin{proof}
1) Following \ref{MICbullet-prop},
$\FF ^{(\bullet)}
\in
\mathrm{MIC}  ^{(\bullet)}(\Y/\V)$.
Set $\FF := \underrightarrow{\lim}\, \FF ^{(\bullet)}\in \mathrm{MIC} ^{\dag \dag}(\Y/\V)$. 
By applying the functor $\underrightarrow{\lim}$
to  the canonical morphism 
$\mathrm{cst} (\FF ^{(0)})
\to 
\widehat{\D} ^{(\bullet)} _{\Y/\S} \otimes _{\widehat{\D} ^{(0)} _{\Y/\S}} \FF ^{(0)}=:\FF ^{(\bullet)}$, 
we check that the canonical morphism 
$\FF ^{(0)} _\Q
\to 
\D ^\dag _{\Y, \Q}
\otimes _{\widehat{\D} ^{(0)} _{\Y,\Q} }
\FF ^{(0)} _\Q$ is an isomorphism. 
Via this isomorphism, we can view 
$\FF ^{(0)} _\Q$ as an object of 
$\mathrm{MIC} ^{\dag \dag}(\Y/\V)$.
This yields that
$f ^* (\FF ^{(0)})$ is 
a $\widehat{\D} ^{(0)} _{\Y'} $-module, coherent over $\O _{\Y'}$
and such that 
$\left ( f ^* (\FF ^{(0)}) \right )  _\Q
\riso 
f ^* (\FF ^{(0)} _\Q)$ (see both notation \ref{ntn-f*}.\ref{ntn-f*1} and \ref{ntn-f*}.\ref{ntn-f*2})
is an object of 
$\mathrm{MIC} ^{\dag \dag}(\Y'/\V)$.
Hence, via \ref{MICbullet-lem}, 
we get the first statement.

2) Since $\FF ^{(\bullet)}
\to
\G ^{(\bullet)}$
is an ind-isogeny of 
$M (\widehat{\D} ^{(\bullet)} _{\Y/\S} )$,
then
$f ^{*(\bullet)} _{\mathrm{alg}} (\FF ^{(\bullet)})
\to
f ^{*(\bullet)} _{\mathrm{alg}} (\G ^{(\bullet)})$
is an isogeny of 
$M (\widehat{\D} ^{(\bullet)} _{\Y'/\S} )$.
We have the commutative diagram with canonical morphisms
\begin{equation}
\label{stab-MIC-f*-diag}
\xymatrix{
{ f ^* (\FF ^{(0)})} 
\ar[r] ^-{}
\ar[d] ^-{\sim}
&
{\widehat{\D} ^{(m)} _{\Y'/\S} \otimes _{\widehat{\D} ^{(0)} _{\Y'/\S}} f ^* (\FF ^{(0)})} 
\ar[r] ^-{}
\ar[d] ^-{\sim}
& 
{ f ^* (\G ^{(m)})} 
\ar[d] ^-{\sim}
\\ 
{f ^{*(0)} _{\mathrm{alg}} (\FF ^{(0)})} 
\ar[r] ^-{}
&
{\widehat{\D} ^{(m)} _{\Y'/\S} \otimes _{\widehat{\D} ^{(0)} _{\Y'/\S}} f ^{*(0)} _{\mathrm{alg}} (\FF ^{(0)})} 
\ar[r] ^-{}
&
{f ^{*(m)} _{\mathrm{alg}} (\G ^{(m)}).} 
}
\end{equation}
Following \ref{ntn-f*}.2, 
the vertical arrows are isomorphisms.
From the first statement, 
the left horizontal arrows 
are isogenies of $\O _{\Y'}$-modules.
Since $\FF ^{(0)} \to \G ^{(m)}$ are also isogenies, 
then the morphisms of the diagram \ref{stab-MIC-f*-diag}
become isomorphisms after tensoring by $\Q$. 
Since 
$\widehat{\D} ^{(m)} _{\Y'/\S} \otimes _{\widehat{\D} ^{(0)} _{\Y'/\S}} f ^* (\FF ^{(0)})
\to 
f ^{*(m)} _{\mathrm{alg}} (\G ^{(m)})$
is a morphism of coherent 
$\widehat{\D} ^{(m)} _{\Y'/\S}$-modules,
this yields the second statement.

3) By using \ref{f*-OcohDm}, 
since 
$\FF ^{(\bullet)}
\in 
\underrightarrow{LM} _{\Q,\mathrm{coh}}   ( \O ^{(\bullet)} _{\Y})$
and 
$\FF ^{(\bullet)}
\in 
\underrightarrow{LM} _{\Q,\mathrm{coh}}   
(\widehat{\D} ^{(\bullet)} _{\Y/\S} )$,
three arrows of the diagram
\begin{equation}
\notag
\xymatrix{
{\L f ^{*(\bullet)} _{\mathrm{alg}} (\FF ^{(\bullet)})} 
\ar[r] ^-{\sim}
& 
{\L f ^{*(\bullet)} ( \FF ^{(\bullet)})} 
\\ 
{\O _{\Y ^{\prime } }  ^{(\bullet)}
\otimes ^\L _{f ^{-1}\O   ^{(\bullet)} _{\Y /\S }}
f ^{-1}\FF ^{(\bullet)}} 
\ar[u] ^-{}
\ar[r] ^-{\sim}
& 
{\O _{\Y ^{\prime } }  ^{(\bullet)}
\smash{\widehat{\otimes}} ^\L _{f ^{-1}\O   ^{(\bullet)} _{\Y /\S }}
f ^{-1}\FF ^{(\bullet)} } 
\ar[u] ^-{\sim}
}
\end{equation}
are isomorphisms.
Hence so is the forth. 
It remains to check 
that 
the canonical morphism
$
\O _{\Y ^{\prime } }  ^{(\bullet)}
\otimes ^\L _{f ^{-1}\O   ^{(\bullet)} _{\Y /\S }}
f ^{-1}\FF ^{(\bullet)} 
\to 
\O _{\Y ^{\prime } }  ^{(\bullet)}
\otimes  _{f ^{-1}\O   ^{(\bullet)} _{\Y /\S }}
f ^{-1}\FF ^{(\bullet)}
$
is an isomorphism. 
Since this is a morphism in 
$\underrightarrow{LD} _{\Q,\mathrm{coh}}  ^{\mathrm{b}} ( \O ^{(\bullet)} _{\Y'})$,
we reduce to check it after applying the functor 
$\underrightarrow{\lim}$, 
which is a consequence of the flatness 
as $\O _{\Y,\Q}$-module of an object of 
$\mathrm{MIC} ^{\dag \dag}(\Y/\V)$. 
\end{proof}

\begin{coro}
\label{corostab-MIC-f*}
Let $f \colon \Y ' \to \Y $ be a morphism of smooth formal $\V$-schemes. 
Let $\E ^{(\bullet)}
\in
\mathrm{MIC}  ^{(\bullet)}(\Y/\V)$, and
$\E := \underrightarrow{\lim}\, \E ^{(\bullet)}\in \mathrm{MIC} ^{\dag \dag}(\Y/\V)$.

\begin{enumerate}

\item $\L f ^{*(\bullet)} ( \E ^{(\bullet)})
\in \mathrm{MIC}  ^{(\bullet)}(\Y'/\V)$ (i.e. is isomorphic to such an object)
and 
$ \underrightarrow{\lim} \L f ^{*(\bullet)} ( \E ^{(\bullet)})
\riso 
f ^* (\E)$.

\item 
Choose
a $\widehat{\D} ^{(0)} _{\Y/\S}$-module $\FF ^{(0)}$,
coherent over $\O _\Y$ 
such that 
$\widehat{\D} ^{(\bullet)} _{\Y/\S} \otimes _{\widehat{\D} ^{(0)} _{\Y/\S}} \FF ^{(0)}$
is isomorphic in 
$\underrightarrow{LM} _{\Q} ( \widehat{\D} ^{(\bullet)} _{\Y/\S})$ to 
$\E ^{(\bullet)}$
and such that 
the canonical morphism 
$\mathrm{cst} (\FF ^{(0)})
\to 
\widehat{\D} ^{(\bullet)} _{\Y/\S} \otimes _{\widehat{\D} ^{(0)} _{\Y/\S}} \FF ^{(0)}$
is an ind-isogeny in  
$M ( \O ^{(\bullet)} _{\Y})$.
Then 
$\L f ^{*(\bullet)} ( \E ^{(\bullet)})
\riso 
\widehat{\D} ^{(\bullet)} _{\Y'/\S} \otimes _{\widehat{\D} ^{(0)} _{\Y'/\S}} 
f ^* (\FF ^{(0)})$.

\end{enumerate}
\end{coro}

\subsection{Commutation of $\sp _*$ with duality, inverse images and glueing isomorphisms}
Let $\X $ be a smooth log formal scheme over $\S $.
Let $Z$ be a divisor of $X$,
$\Y $ be the open subset of $\X $ complementary to the support of $Z$,
and $j \colon \Y \hookrightarrow \X$ the open immersion.
Let $E \in  \mathrm{Isoc} ^{\dag} (Y, X,\X /K) $.
Let $E _\X \in \mathrm{MIC} ^{\dag} (Y, X,\X /K) $ its realization on 
the frame $(Y, X,\X )$ 
and 
$\E := \mathrm{sp} _* ( E _{\X})$.
We have the equalities 
  $D ^\mathrm{b} _{\mathrm{coh}}(\O _{\X } ( \hdag Z ) _{\Q}) =
  D _{\mathrm{parf}} (\O _{\X } ( \hdag Z ) _{\Q})$,
  $D ^\mathrm{b} _{\mathrm{coh}}(\D _{\X /\S } (\hdag Z) _\Q) =
  D _{\mathrm{parf}} (\D _{\X /\S } (\hdag Z) _\Q)$,
  and 
    $D ^\mathrm{b} _{\mathrm{coh}}(\D ^\dag  _{\X /\S } (\hdag Z) _\Q) =
  D _{\mathrm{parf}} (\D ^\dag _{\X /\S } (\hdag Z) _\Q)$ (this is \cite{huyghe_finitude_coho}).
We get $\E \in D _{\mathrm{parf}} (\O _{\X } ( \hdag Z ) _{\Q})$,
$\E \in D _{\mathrm{parf}} (\D  _{\X /\S } (\hdag Z) _\Q)$
and 
$\E \in D _{\mathrm{parf}} (\D ^\dag _{\X /\S } (\hdag Z) _\Q)$.

\begin{ntn}
For any $\FF \in D  ^\mathrm{b} _{\mathrm{coh}}
( \D  _{\X /\S } (\hdag Z) _\Q ) $, we set
$\DD ^{\mathrm{alg}} _Z (\FF)=
\R \mathcal{H}om  _{\D  _{\X /\S } (\hdag Z) _\Q }
( \FF , \,\D  _{\X /\S } (\hdag Z) _\Q \otimes _{\O _{\X}}\omega _{\X /\S  } ^{-1})[d _X]$ 
and
$\FF ^{\vee} =
\R \mathcal{H}om  _{\O _{\X } ( \hdag Z ) _{\Q} }
( \FF , \,\O_{\X , \Q} ( \hdag Z ))$.
For any $\G \in D  ^\mathrm{b} _{\mathrm{coh}} ( \D ^\dag  _{\X /\S } (\hdag Z) _\Q ) $, we set
$\DD  _Z (\G)=
\R \mathcal{H}om  _{\D ^\dag _{\X /\S } (\hdag Z) _\Q }
( \G , \,\D  ^\dag _{\X /\S } (\hdag Z) _\Q \otimes _{\O _{\X}}\omega _{\X /\S  } ^{-1})[d _X]$. 

\end{ntn}

\begin{prop}\label{Doevee=De}
There exists a canonical isomorphism 
$$ \theta \ : \ \DD ^{\mathrm{alg}} _Z ( \O _{\X } ( \hdag Z ) _{\Q} ) \otimes ^\L _{\O _{\X } ( \hdag Z ) _{\Q}} \E ^\vee
\riso \DD ^{\mathrm{alg}} _Z (\E).$$
\end{prop}
\begin{proof}
This is  \cite[2.2.1]{caro_comparaison}.  
\end{proof}

\begin{lem}
\begin{enumerate}[(i)]
\item  $\O _{\X , \Q } (\hdag Z) \in D _{\mathrm{parf}}( \D  _{\X /\S } (\hdag Z) _\Q)$.

\item We have the canonical isomorphism:
\begin{equation}\label{dualB=B}
  \DD ^{\mathrm{alg}} _Z ( \O _{\X } ( \hdag Z ) _{\Q})
\riso
\O _{\X } ( \hdag Z ) _{\Q}.
\end{equation}

\end{enumerate}
\end{lem}
\begin{proof}
This is  \cite[5.20]{caro_log-iso-hol}.
\end{proof}

\begin{rem}\label{evee=De0}
From \ref{dualB=B} and \ref{Doevee=De},
we get the isomorphism
$\E ^\vee \riso \DD ^{\mathrm{alg}} _Z (\E)$.
\end{rem}

\begin{prop}
\label{sp*dual} We have the isomorphisms
\begin{gather}
\label{sp*dualiso1} \sp ^* \mathcal{H}om _ {\O _{\X , \Q } (\hdag Z)} ( \E , \O _{\X , \Q } (\hdag Z) )
\riso  \mathcal{H} om  _{ j ^{\dag} \O _{\X _K} }(\sp ^*  \E ,  j ^{\dag} \O _{\X _K} ) \\
\label{sp*dualiso2}
 \mathcal{H}om _ {\O _{\X , \Q } (\hdag Z)} ( \sp _* E , \O _{\X , \Q } (\hdag Z) )
\riso \sp _* ( \mathcal{H} om  _{ j ^{\dag} \O _{\X _K} }( E ,  j ^{\dag} \O _{\X _K} )).
\end{gather}
\end{prop}

\begin{proof}
This is similar to \cite[2.2.7]{caro_comparaison}.
\end{proof}

\begin{empt}\label{rhoisom}
Consider the following morphism:
$$\rho ^\dag _Z\ :\ \DD ^{\mathrm{alg}} _Z (\E) \rightarrow
\D _{\X } ^{\dag} ( \hdag Z ) _{\Q} \otimes _{\D _{\X } ( \hdag Z ) _{\Q}} \DD ^{\mathrm{alg}} _Z (\E)
\riso
\DD  _Z ( \D _{\X } ^{\dag} ( \hdag Z ) _{\Q} \otimes _{\D _{\X } ( \hdag Z ) _{\Q}} \E)
\rightarrow \DD  _Z (\E).$$
Since $\E$ is locally projective of finite type over $\O _{\X , \Q} (\hdag Z )$,
then the morphism
$ \mathcal{H}om _ {\O _{\X , \Q } (\hdag Z)} ( \E , \,\O _{\X , \Q } (\hdag Z) )
\rightarrow
\R \mathcal{H}om _ {\O _{\X , \Q } (\hdag Z)} ( \E , \,\O _{\X , \Q } (\hdag Z) )= \E ^\vee$ is 
an isomorphism.
Using \ref{sp*dualiso2}, this yields $\E ^\vee \in \mathrm{MIC} ^{\dag \dag} (Y, X,\X /K) $.
Since $\E ^\vee \riso \DD ^{\mathrm{alg}} _Z (\E)$ (see \ref{evee=De0}), 
via \ref{cohisosurcv2} we check that 
$\rho ^\dag _Z$ is an isomorphism.
\end{empt}

\begin{empt}\label{constrhodag}
Let  $\theta ^\dag$ :
$\DD ^{\mathrm{alg}}  _Z ( \O _{\X } ( \hdag Z ) _{\Q} ) \otimes _{\O _{\X } ( \hdag Z ) _{\Q}} \E ^\vee
\riso \DD ^{\mathrm{alg}} _Z (\E)$
be the isomorphism making commutative the following diagram:
$$\xymatrix {
{\DD ^{\mathrm{alg}} _Z ( \O _{\X } ( \hdag Z ) _{\Q} ) \otimes _{\O _{\X } ( \hdag Z ) _{\Q}} \E ^\vee}
\ar[r] ^(0.68){\theta} _(0.68){\sim}
\ar[d] ^{\rho ^\dag _Z\otimes id} _{\sim}
&
{\DD ^{\mathrm{alg}} _Z (\E)}
\ar[d] ^{\rho ^\dag _Z} _{\sim}
\\
{\DD  _Z ( \O _{\X } ( \hdag Z ) _{\Q} ) \otimes _{\O _{\X } ( \hdag Z ) _{\Q}} \E ^\vee }
\ar@{.>}[r] ^(0.68){\theta ^\dag} _(0.68){\sim}
&
{\DD _Z (\E)}
}$$
\end{empt}

\begin{empt}
  \label{dualisoscvdag} 
  From  \ref{dualB=B} and \ref{rhoisom}, we get the isomorphism
  $\DD  _Z ( \O _{\X } ( \hdag Z ) _{\Q} ) \riso \O _{\X } ( \hdag Z ) _{\Q}$.
Hence, the isomorphism $\theta ^\dag$ induces the following one
  $\E ^\vee \riso \DD _Z (\E)$. 
\end{empt}

\begin{thm}
\label{theoprincipal}
We have the canonical isomorphism in $\mathrm{MIC} ^{\dag \dag} (\X,Z/K)$ 
\begin{equation}
\label{comspdual} 
\mathrm{sp} _* ( E ^{\vee}) \riso \DD _{Z } (\E ).
\end{equation}
\end{thm}
\begin{proof}
This is straightforward from \ref{sp*dual}
and \ref{dualisoscvdag}.
\end{proof}

\begin{empt}
\label{sp*f*com-empt}
Let $u \colon \X' \to \X$ be a morphism of smooth formal $\V$-schemes such that $Z ':= u _0 ^{-1} (Z)$ is the support of  
a divisor of $X'$, and $\Y' := \X ' \setminus Z'$. 
This yields the  morphism of smooth frames
$f :=(b,\,a,\,u)
\colon 
(Y', X ^{\prime },\X')
\to 
(Y, X,\X)$.
If this do not cause too much confusion, we will simply write
 $u _K ^*$ instead of $f _K ^*$.
Let $E _\X \in \mathrm{MIC} ^{\dag} (Y, X,\X/K) $. 
Let $\E \in \mathrm{MIC} ^{\dag \dag} (\X,Z/K) $ (see Notation \ref{ntnMICdag2fs}).
Following Theorem \ref{letterBerthelotCaro2007}, 
the functors $\sp _*$ and $\sp ^*$ induce quasi-inverse equivalences of categories 
between 
$ \mathrm{MIC} ^{\dag} (Y, X,\X/K) $
and
$\mathrm{MIC} ^{\dag \dag} (\X,Z/K) $.
We have the functor $u ^! [-d _{X'/X}] \colon \mathrm{MIC} ^{\dag \dag} (\X,Z/K) 
\to 
\mathrm{MIC} ^{\dag \dag} (\X',Z'/K) $ which is compatible with $u _K ^*$, 
i.e. 
there exist canonical isomorphisms respectively of 
$\mathrm{MIC} ^{\dag \dag} (X', \X',Z'/K) $
and
$ \mathrm{MIC} ^{\dag} (Y', X',\X'/K) $
 of the form
\begin{equation}
\label{sp*f*com}
\sp _* u _K ^* (E _\X)  \riso u ^!  \sp _* (E _\X)[-d _{X'/X}],
\ 
u _K ^* \sp ^* (\E) \riso \sp ^* u ^! (\E) [-d _{X'/X}].
\end{equation}
Moreover, these isomorphisms are transitive with respect to the composition of morphisms (see \cite[2.4.1]{caro-construction}).
\end{empt}

\begin{prop}
\label{sp-eps-tau}
With the notation \ref{sp*f*com-empt}, 
let $u '\colon \X' \to \X$ be a morphism of smooth formal $\V$-schemes
such that $u ' _0 = u _0$. 
Then, the following diagrams
$$\xymatrix  @R=0,3cm@C=1,4cm {
{ \sp _* u ^{\prime *} _K (E )[d _{X'/X}]}
\ar[r] ^-{\sp _*(\epsilon _{u,\,u'})} _-{\sim}
\ar[d] _-{\sim} ^-{\ref{sp*f*com}}
&
{ \sp _* u ^{*} _K (E )[d _{X'/X}]}
\ar[d] _-{\sim}  ^-{\ref{sp*f*com}}
\\
{ u ^{\prime !} \sp _*  (E )}
\ar[r] ^-{\tau _{u ,u' }} _-{\sim}
&
{ u ^{!}  \sp _*  (E ),}
}
\xymatrix  @R=0,3cm@C=1,4cm {
{ u ^{\prime *} _K \sp ^* (\E )[d _{X'/X}]}
\ar[r] ^-{\epsilon _{u,\,u'}} _-{\sim}
\ar[d] _-{\sim}   ^-{\ref{sp*f*com}}
&
{ u^{*} _K   \sp ^*(\E )[d _{X'/X}]}
\ar[d] _-{\sim}   ^-{\ref{sp*f*com}}
\\
{\sp ^*  u ^{\prime !}  (\E )}
\ar[r] ^-{\sp ^* (\tau _{u,u'})} _-{\sim}
&
{ u ^{!} \sp _*  (\E ),}
}
$$
where the glueing isomorphisms $\epsilon _{u,u'}$ and $\tau _{u,u'}$ are that of 
\ref{glueingisocntn-iso1}
and \ref{prop-glueiniso-coh1},
are commutative.
\end{prop}

\begin{proof}
This follows from the fact that both glueing isomorphisms $\epsilon _{u,u'}$ and $\tau _{u,u'}$ are built similarly using some factorization
via the closed imbedding 
$(u,u') \colon \X ^{\prime } \hookrightarrow (\X ^{ } )^{(n)}$
where 
$(\X ^{ } )^{(n)}$
is the $n$th infinitesimal neighborhood of the diagonal immersion
$\X ^{ }\hookrightarrow
\X ^{ } \times _{\S } \X ^{ }$.
\end{proof}

\subsection{Construction of the functor $\sp _+$}
\label{ntnsp+}
Let $\fP $ be a smooth formal scheme over $\S $.
Let $u _0\colon X  \to P $ be a closed immersion of smooth schemes over $S $.
Let $T$ be a divisor of $P$ such that $Z:= T \cap X$ is a divisor of $X$. 
We set $Y:= X \setminus Z$. 
Choose $(\fP  _{\alpha}) _{\alpha \in \Lambda}$ an open affine covering of  $\fP $
and let us use the corresponding notation of \ref{ntnPPalpha} (which are compatible with
that of \ref{ntnPPalpha-withoutdiv}).

\begin{ntn}
We denote by 
$\mathrm{MIC} ^{\dag \dag} ( (\X   _\alpha )_{\alpha \in \Lambda},Z/K)$
the full subcategory of 
$\mathrm{Coh} ((\X   _\alpha )_{\alpha \in \Lambda},Z/K)$ whose objects
$((\E _{\alpha})_{\alpha \in \Lambda},\, (\theta _{\alpha\beta}) _{\alpha ,\beta \in \Lambda})$
are such that, 
for all $\alpha \in \Lambda$,
$\E _{\alpha}$ 
is 
$\O _{\X _\alpha} (\hdag Z _\alpha) _{\Q}$-coherent.
\end{ntn}

\begin{lem}
\label{lem1pre-sp+plfid}
We have the canonical functor 
$\sp _*$ : $\mathrm{MIC} ^\dag (Y,  (\X   _\alpha )_{\alpha \in \Lambda}/K)
  \rightarrow 
  \mathrm{MIC} ^{\dag \dag} ( (\X   _\alpha )_{\alpha \in \Lambda},Z/K)$.
\end{lem}

\begin{proof}
This is checked in \cite[2.5.9.i)]{caro-construction}. 
Let us at least recall its construction. 
Let  $((E _{\alpha})_{\alpha \in \Lambda},\, (\eta _{\alpha\beta}) _{\alpha ,\beta \in \Lambda})\in
  \mathrm{MIC} ^\dag (Y,  (\X   _\alpha )_{\alpha \in \Lambda}/K)$.
Let $\theta _{\alpha \beta}$ be the isomorphism making commutative the diagram
  \begin{equation}
    \label{defdonneesp*}
    \xymatrix  @R=0,3cm
{
    {\sp _*  p _{1 K}  ^{\alpha \beta !} (E _{\alpha})}
    \ar[r] _-{\sim} ^-{\ref{sp*f*com}}
    &
    {p _{1 }  ^{\alpha \beta !} \sp _*    (E _{\alpha})}
    \\
    {\sp _*  p _{2 K}  ^{\alpha \beta !} (E _{\beta})}
    \ar[r] _-{\sim} ^-{\ref{sp*f*com}}
    \ar[u] ^-{\sp _* \eta _{\alpha \beta}} _-{\sim}
    &
    {p _{2 }  ^{\alpha \beta !} \sp _*    (E _{\beta}).}
    \ar@{.>}[u]^{\theta _{\alpha \beta}}
}
  \end{equation}
Using \ref{sp-eps-tau}, 
we check that
  $\sp _* ((E _{\alpha})_{\alpha \in \Lambda},\, (\eta _{\alpha\beta}) _{\alpha ,\beta \in \Lambda})
  := ( (\sp _* E _{\alpha})_{\alpha \in \Lambda},\, (\theta _{\alpha\beta}) _{\alpha ,\beta \in \Lambda})$
  is functorially an object of 
  $\mathrm{MIC} ^{\dag \dag} ( (\X   _\alpha )_{\alpha \in \Lambda},Z/K)$.
\end{proof}

\begin{lem}
\label{lem2pre-sp+plfid}
We have  the canonical functor 
$\sp ^*\colon
\mathrm{MIC} ^{\dag \dag} ( (\X   _\alpha )_{\alpha \in \Lambda},Z/K)
\rightarrow 
\mathrm{MIC} ^\dag (Y,  (\X   _\alpha )_{\alpha \in \Lambda}/K)$.
\end{lem}

\begin{proof}
This is checked in \cite[2.5.9.i)]{caro-construction}. 
The construction is similar to that of \ref{lem1pre-sp+plfid}.
\end{proof}

\begin{prop}
\label{pre-sp+plfid}
The functors  $\sp _*$ and $\sp ^*$ are quasi-inverse equivalences of categories between the category
$\mathrm{MIC} ^\dag (Y,  (\X   _\alpha )_{\alpha \in \Lambda}/K)$ and 
$\mathrm{MIC} ^{\dag \dag} ( (\X   _\alpha )_{\alpha \in \Lambda},Z/K)$.
\end{prop}

\begin{proof} 
The proof is  \cite[2.5.9]{caro-construction}. 
For any object $((E _{\alpha})_{\alpha \in \Lambda},\, (\eta _{\alpha\beta}) _{\alpha ,\beta \in \Lambda})$
of 
  $\mathrm{MIC} ^\dag (Y,  (\X   _\alpha )_{\alpha \in \Lambda}/K)$,
we have first checked that the adjunction isomorphisms
$\sp ^* \sp _* (E _\alpha) \riso E _\alpha$ commute with glueing data.
Similarly, for any object
$((\E _{\alpha})_{\alpha \in \Lambda},\, (\theta _{\alpha\beta}) _{\alpha ,\beta \in \Lambda})$
of
$\mathrm{MIC} ^{\dag \dag} ( (\X   _\alpha )_{\alpha \in \Lambda},Z/K)$
the adjunction isomorphisms
$ \E _\alpha \riso \sp _*  \sp ^*(\E _\alpha)$ commute with glueing data.
\end{proof}

\begin{ntn}
\label{ntnMICdag2fs2}
We denote by $\mathrm{MIC} ^{\dag \dag} (X, \fP,T/K) $ the full subcategory of 
$\mathrm{Coh} (X, \fP,T /K)$ whose objects $\E$ satisfy the following condition:
for any affine open formal subscheme  
 $\fP ^{\prime }$ of $\fP $, for any morphism of formal schemes
  $v$ : $ \Y ^{\prime } \hookrightarrow \fP ^{\prime } $ which reduces modulo $\pi$ to the closed imbedding 
  $Y ^{ } \cap P ^{\prime } \hookrightarrow P ^{\prime }$,
the sheaf $v ^! (\E |_{\fP^{\prime }}) $ is $\O _{\Y',\,\Q}$-coherent.
When $T$ is empty, we remove it in the notation.
Finally, according to notation \ref{ntnMICdag2fs}, 
when $X= P$, we remove $X$ in the notation.

\end{ntn}

\begin{empt}
The functors $u ^!  _0$ and $ u _{0+}$ constructed in respectively \ref{const-u0!} and \ref{const-u0+} 
induce quasi-inverse equivalence of categories between 
$\mathrm{MIC} ^{\dag \dag} (X, \fP,T/K)$
and
$\mathrm{MIC} ^{\dag \dag} ( (\X   _\alpha )_{\alpha \in \Lambda},Z/K)$, i.e., 
we have the commutative diagram
\begin{equation}
\label{eqcat-u0+!}
\xymatrix{
{\mathrm{MIC} ^{\dag \dag} (X, \fP,T/K) } 
\ar@{^{(}->}[r] ^-{}
\ar@{.>}@<4ex>[d] ^-{\cong} _-{u _{0} ^!} 
& 
{\mathrm{Coh} (X, \fP,T /K)} 
\ar@{.>}@<4ex>[d] ^-{\cong} _-{u _{0} ^!} 
\\ 
{\mathrm{MIC} ^{\dag \dag} ( (\X   _\alpha )_{\alpha \in \Lambda},Z/K)}
\ar@{^{(}->}[r] ^-{}
\ar@{.>}@<4ex>[u] ^-{\cong} _-{u _{0+}} 
& {\mathrm{Coh} ((\X   _\alpha )_{\alpha \in \Lambda},Z/K).} 
\ar@{.>}@<4ex>[u] ^-{\cong} _-{u _{0+}} 
}
\end{equation}

\end{empt}

\begin{ntn}
\label{ntn-dfnsp+}
We get the canonical equivalence of categories 
$\sp _+ \colon 
\mathrm{Isoc} ^{\dag} (Y, X,\fP/K)
\cong
\mathrm{MIC} ^{\dag \dag} (X, \fP,T/K)$
by composition of the equivalences
$\mathrm{Isoc} ^{\dag} (Y, X,\fP/K)
\cong
\mathrm{MIC} ^{\dag} (Y, X,\fP/K)
\underset{u _{0K} ^*}{\riso}
\mathrm{MIC} ^\dag (Y,  (\X   _\alpha )_{\alpha \in \Lambda}/K)
\underset{\sp _*}{\riso}
\mathrm{MIC} ^{\dag \dag} ( (\X   _\alpha )_{\alpha \in \Lambda},Z/K)
\underset{u _{0+}}{\riso}
\mathrm{MIC} ^{\dag \dag} (X, \fP,T/K).$

\end{ntn}

\section{Exterior tensor products}

\subsection{Exterior tensor products on schemes}
Let 
$S$ be a noetherian $\Z _{(p)}$-scheme of finite Krull dimension.
Since the base scheme $S$ is fixed, so we can remove it in the notation.
If $\phi \colon X\to S$ is a morphism, 
by abuse of notation, we sometimes denote $\phi ^{-1} \O _S$ simply by
$\O _S$. Moreover, $S$-schemes will be supposed to be quasi-compact and separated.

For $i= 1,\dots, n$, 
let $X _i $ be a smooth $S$-scheme.
Set $X := X _1 \times _S X _2 \times _S \dots \times _S X _n$.
For $i= 1,\dots, n$, let $pr _i  \colon X \to X _i$,
 be the projections. 
 We denote by 
 $\varpi \colon X \to S$ and by $\varpi _i \colon X _i \to S$ the structural morphisms.

\begin{empt}

\begin{enumerate}
\item For $i= 1,\dots, n$, let 
$\E _i $ be a sheaf of $\varpi _i ^{-1}\O _{S}$-module.
We get the $\varpi ^{-1}\O _{S}$-module by setting
$$\underset{i}{\boxtimes} ^{\mathrm{top}}  \E _i
:=
pr _1 ^{-1} \E _1 
\otimes _{\O _S}
pr _2 ^{-1} \E _2
\otimes _{\O _S}
\cdots
\otimes _{\O _S}
pr _n ^{-1} \E _n.$$

\item For $i= 1,\dots, n$, let 
$\E _i $ be an $\O _{X _i}$-module.
The sheaf 
$\underset{i}{\boxtimes} ^{\mathrm{top}}  \E _i$
has a canonical structure of 
$\underset{i}{\boxtimes} ^{\mathrm{top}} 
\O _{X _i}$-module.
We put
$\underset{i}{\boxtimes}  \E _i
:=
\O _{X}
\otimes _{\underset{i}{\boxtimes} ^{\mathrm{top}}  \O _{X _i}}
\underset{i}{\boxtimes} ^{\mathrm{top}}  \E _i$. 
Moreover, by commutativity and associativity of tensor products,
we get the canonical isomorphism of 
$\underset{i}{\boxtimes} ^{\mathrm{top}} 
\O _{X _i}$-modules
\begin{equation}
\label{boxtimestop-dfn2}
\underset{i}{\boxtimes} ^{\mathrm{top}}  \E _i
\riso 
\left ( pr _1 ^{-1} \E _1 
\otimes _{pr _1 ^{-1} \O _{X _1}}
\underset{i}{\boxtimes} ^{\mathrm{top}}  \O _{X _i} \right)
\otimes _{\underset{i}{\boxtimes} ^{\mathrm{top}}  \O _{X _i} }
\cdots
\otimes _{\underset{i}{\boxtimes} ^{\mathrm{top}}  \O _{X _i} }
\left ( pr _n ^{-1} \E _n 
\otimes _{pr _n ^{-1} \O _{X _n}}
\underset{i}{\boxtimes} ^{\mathrm{top}}  \O _{X _i} \right).
\end{equation}
 Using the isomorphism \ref{boxtimestop-dfn2}, 
we get the isomorphism of $\O _X$-modules
\begin{equation}
\label{boxtimes-dfn2}
\underset{i}{\boxtimes}  \E _i 
\riso 
pr _1 ^{*} \E _1 
\otimes _{\O _{X} }
\cdots
\otimes _{\O _{X} }
 pr _n ^{*} \E _n .
\end{equation}

Since $pr _i ^{-1} \D ^{(m)} _{X _i}$ are $\O _S$-algebras,
we get a canonical structure of $\O _S$-algebra on
$\underset{i}{\boxtimes} ^{\mathrm{top}} 
\D ^{(m)} _{X _i}$.

\item For $i= 1,\dots, n$, 
let $\cF _i$ be a left  $\D ^{(m)} _{X _i}$-module 
(resp. $\cG _i$ be a right  $\D ^{(m)} _{X _i}$-module).
Then 
$\underset{i}{\boxtimes} ^{\mathrm{top}}  \cF _i$
(resp. $\underset{i}{\boxtimes} ^{\mathrm{top}}  \cG _i$)
has a canonical structure of left (resp. right)
$\underset{i}{\boxtimes} ^{\mathrm{top}} 
\D ^{(m)} _{X _i}$-module.
The canonical homomorphism 
of $\O _S$-algebras
$\underset{i}{\boxtimes} ^{\mathrm{top}} 
\D ^{(m)} _{X _i}
\to \D ^{(m)} _{X }$ induces the canonical isomorphism 
of $\O _X$-modules
$\underset{i}{\boxtimes} 
\D ^{(m)} _{X _i}
\riso \D ^{(m)} _{X }$.
This yields the isomorphism of $\O _{X }$-modules
$\underset{i}{\boxtimes}  \cF _i 
\riso 
\D ^{(m)} _{X }
\otimes _{\underset{i}{\boxtimes} ^{\mathrm{top}}  \D ^{(m)}  _{X _i}}
\underset{i}{\boxtimes}  ^{\mathrm{top}} \cF _i $
(resp. 
$\underset{i}{\boxtimes}  \cG _i 
\riso 
\underset{i}{\boxtimes}  ^{\mathrm{top}} \cF _i 
\otimes _{\underset{i}{\boxtimes} ^{\mathrm{top}}  \D ^{(m)}  _{X _i}}
\D ^{(m)} _{X }$).
Via this isomorphism, we endowed 
$\underset{i}{\boxtimes}  \cF _i $
(resp. $\underset{i}{\boxtimes}  \cG _i $)
with a structure of left (resp. right) $\D ^{(m)} _{X }$-module.

\item For $i= 1,\dots, n$, 
let $\cF _i$ be a left  $\D ^{(m)} _{X _i}$-module.
Then 
$pr _1 ^{*} \cF _1 
\otimes _{\O _{X} }
\cdots
\otimes _{\O _{X} }
pr _n ^{*} \cF _n$
has a canonical structure of left $\D ^{(m)} _{X}$-module (see \cite[2.3.3]{Be1}).  
We check that the isomorphism 
\ref{boxtimes-dfn2} is in fact an isomorphism of 
left $\D ^{(m)} _{X}$-modules.

\end{enumerate}
\end{empt}

\begin{empt}
\label{rem-ext-prod}
\begin{enumerate}
\item When $S$ is the spectrum of a field, 
the multi-functor $\underset{i}{\boxtimes} ^{\mathrm{top}}  $ is exact. 
Since the extensions
$\underset{i}{\boxtimes} ^{\mathrm{top}}  \O _{X _i}
\to \O _{X}$ 
and 
$\underset{i}{\boxtimes} ^{\mathrm{top}}  \D ^{(m)}  _{X _i}
\to 
\D ^{(m)} _{X }$ are right and left flat, 
this yields that 
the multi-functor $\underset{i}{\boxtimes}  $ is also exact.

\item When $S$ is not the spectrum of a field, 
the multi-functor $\underset{i}{\boxtimes} ^{\mathrm{top}}  $ is not necessarily exact. 
We get the multi-functor 
$\underset{i}{\overset{\L}{\boxtimes}} {}^{\mathrm{top}}  
\colon 
D ^{-} (\varpi _1 ^{-1}\O _{S})
\times \cdots \times  
D ^{-} (\varpi _n ^{-1}\O _{S})
\to 
D ^{-} (\varpi ^{-1}\O _{S})$ by setting 
for any 
$\E _i \in D ^{-} (\varpi _i ^{-1}\O _{S})$
$$\underset{i}{\overset{\L}{\boxtimes}} {}^{\mathrm{top}}   \E _i
:=
pr _1 ^{-1} \E _1 
\otimes _{\O _S} ^{\L}
pr _2 ^{-1} \E _2
\otimes _{\O _S} ^{\L}
\cdots
\otimes _{\O _S} ^{\L}
pr _n ^{-1} \E _n.$$

\item We have the multi-functor 
$\underset{i}{\overset{\L}{\boxtimes}}
\colon 
D ^{-} (\O _{X _1})
\times \cdots \times  
D ^{-} (\O _{X _n})
\to 
D ^{-} (\O _X)$ 
by setting for any 
$\E _i \in D ^{-} (\O _{X _i})$
\begin{equation}
\label{boxtimes-dfn2L}
\underset{i}{\overset{\L}{\boxtimes}}
 \E _i
:=
\O _{X}
\otimes _{\underset{i}{\boxtimes} ^{\mathrm{top}}  \O _{X _i}}
\underset{i}{\overset{\L}{\boxtimes}} {}^{\mathrm{top}} 
\cE _i
\riso 
pr _1 ^{*} \E _1 
\otimes _{\O _{X} } ^{\L}
\cdots
\otimes _{\O _{X} } ^{\L}
 pr _n ^{*} \E _n ,
\end{equation}
where the last isomorphism is, after using flat resolutions,
a consequence of \ref{boxtimes-dfn2}.

\item For any $i = 1,\dots, n$, let 
$\cF _i \in D ^{-} ({}^{l} \D ^{(m)} _{X _i})$,
$\cM _i \in D ^{-} ({}^{r} \D ^{(m)} _{X _i})$.
Since 
we have the canonical isomorphisms 
$\underset{i}{\overset{\L}{\boxtimes}}
\D ^{(m)} _{X _i}
\riso 
\underset{i}{\boxtimes} 
\D ^{(m)} _{X _i}
\riso \D ^{(m)} _{X }$,
then the canonical morphisms
\begin{gather}
\notag
\O _{X}
\otimes _{\underset{i}{\boxtimes} ^{\mathrm{top}}  \O _{X _i}}
\underset{i}{\overset{\L}{\boxtimes}} {}^{\mathrm{top}} 
\cF _i
\to 
\D ^{(m)} _{X }
\otimes _{\underset{i}{\boxtimes} ^{\mathrm{top}}  \D ^{(m)}  _{X _i}}
\underset{i}{\overset{\L}{\boxtimes}} {}^{\mathrm{top}} 
\cF _i
\ \ \
\text{ and }
\\
\notag
\O _{X}
\otimes _{\underset{i}{\boxtimes} ^{\mathrm{top}}  \O _{X _i}}
\underset{i}{\overset{\L}{\boxtimes}} {}^{\mathrm{top}} 
\cM _i
\riso 
\underset{i}{\overset{\L}{\boxtimes}} {}^{\mathrm{top}} 
\cM _i
\otimes _{\underset{i}{\boxtimes} ^{\mathrm{top}}  \O _{X _i}}
\O _{X}
\to 
\underset{i}{\overset{\L}{\boxtimes}} {}^{\mathrm{top}} 
\cM _i
\otimes _{\underset{i}{\boxtimes} ^{\mathrm{top}}  \D ^{(m)}  _{X _i}}
\D ^{(m)} _{X }
\end{gather}
are isomorphisms. 
Hence, there is no problem (up to canonical isomorphim) with respect to \ref{boxtimes-dfn2L}  to set 
$\underset{i}{\overset{\L}{\boxtimes}}
 \cF _i
:=
\D ^{(m)} _{X }
\otimes _{\underset{i}{\boxtimes} ^{\mathrm{top}}  \D ^{(m)}  _{X _i}}
\underset{i}{\overset{\L}{\boxtimes}} {}^{\mathrm{top}} 
\cF _i$
and
$\underset{i}{\overset{\L}{\boxtimes}}
 \cM _i
:=
\underset{i}{\overset{\L}{\boxtimes}} {}^{\mathrm{top}} 
\cM _i
\otimes _{\underset{i}{\boxtimes} ^{\mathrm{top}}  \D ^{(m)}  _{X _i}}
\D ^{(m)} _{X }$.
For $ * \in \{Â l,r\}$, we get the multi-functor
$$\underset{i}{\overset{\L}{\boxtimes}}
\colon 
D ^{-} ({} ^* \D ^{(m)} _{X _1})
\times \cdots \times  
D ^{-} ({} ^* \D ^{(m)} _{X _n})
\to 
D ^{-} ({} ^* \D ^{(m)} _{X }).$$ 

\item If we would like to clarify the basis $S$, 
we will add it in the notation. For instance,
we write 
$\underset{S ,i}{\overset{\L}{\boxtimes}} {}^{\mathrm{top}} $
and 
$\underset{S,i}{\overset{\L}{\boxtimes}}$
(or $\underset{\O _S ,i}{\overset{\L}{\boxtimes}} {}^{\mathrm{top}} $
and 
$\underset{\O _S,i}{\overset{\L}{\boxtimes}}$)
instead of
$\underset{i}{\overset{\L}{\boxtimes}} {}^{\mathrm{top}} $
and 
$\underset{i}{\overset{\L}{\boxtimes}}$.

\end{enumerate}

\end{empt}

\begin{lem}
\label{lem-otimes-assoc-comm-box}
Let $\cR$ be a commutative sheaf on $X$.
For $i = 1,\dots, n$, 
let 
$\cA _i$, 
$\cB _i$ be $\cR$-algebras, 
$\widetilde{\cM} _i$ be a $(\cA _i,\cB _i)$-bimodule,
$\widetilde{\cE} _i$ be a left $\cB _i$-module.
Set 
$\underset{i}{\otimes} \cA _i:= \cA _1 \otimes _{\cR} \cA _2 \otimes _{\cR}\cdots \otimes _{\cR} \cA _n$, 
and similarly 
by replacing $\cA _i$ by 
$\cB _i$,
$\widetilde{\cM} _i$ or 
$\widetilde{\cE} _i$.
We have the canonical isomorphism of 
left $\underset{i}{\otimes} \cA _i$-modules of  the form
\begin{equation}
\notag
\otimes _i (\widetilde{\cM} _i \otimes _{\cB _i} \widetilde{\cE} _i)
\riso 
\otimes _i (\widetilde{\cM} _i ) 
\otimes _{\otimes _i (\cB _i)} 
\otimes _i (\widetilde{\cE} _i).
\end{equation}
\end{lem}

\begin{proof}
This is an exercise. 
\end{proof}

\begin{lem}
\label{com-botimestop}
For $i= 1,\dots, n$, 
let $\D _i$ be a sheaf of $\varpi _i ^{-1} \O _S$-algebras,
$\cM _i\in D ^{-} (\D _i, \O _{X _i}) $,
$\E _i \in D ^- (\O _{X _i})$,
$\cN _i \in D ^{-} (\D _i,  \D ^{(m)} _{X _i})$,
$\cF _i \in D ^-(\D ^{(m)}  _{X _i})$.
\begin{enumerate}
\item We have the canonical isomorphism of 
$D ^- (\underset{i}{\boxtimes} ^{\mathrm{top}} \D _i, 
\underset{i}{\boxtimes} ^{\mathrm{top}}  \O _{X _i} )$
\begin{equation}
\label{com-botimestop-iso1}
\underset{i}{\overset{\L}{\boxtimes}} {}^{\mathrm{top}}   
(\cM _i \otimes ^{\L} _{\O _{X _i}} \E _i)
\riso 
\underset{i}{\overset{\L}{\boxtimes}} {}^{\mathrm{top}}   
\cM _i \otimes ^{\L} _{\underset{i}{\boxtimes} ^{\mathrm{top}}  \O _{X _i}} 
\underset{i}{\overset{\L}{\boxtimes}} {}^{\mathrm{top}}   \E _i.
\end{equation}

\item We have the canonical isomorphism of $\underset{i}{\boxtimes} ^{\mathrm{top}} \D _i$-modules
\begin{equation}
\label{com-botimestop-iso2}
\underset{i}{\overset{\L}{\boxtimes}} {}^{\mathrm{top}}   
(\cN _i \otimes ^{\L} _{\D ^{(m)}  _{X _i}} \cF _i)
\riso 
\underset{i}{\overset{\L}{\boxtimes}} {}^{\mathrm{top}}   
\cN _i \otimes ^{\L} _{\underset{i}{\boxtimes} ^{\mathrm{top}}  \D ^{(m)}  _{X _i}} \underset{i}{\overset{\L}{\boxtimes}} {}^{\mathrm{top}}   \cF _i.
\end{equation}

\end{enumerate}
\end{lem}

\begin{proof}
By using flat resolutions, 
we can remove $\L$. 
Then, this is a consequence of Lemma \ref{lem-otimes-assoc-comm-box}
in the case where $\cR = \cO _S$,
$\cA _i =pr _i ^{-1} \cD _i$, 
$\cB _i =\cO _{X _i}$
(resp. $\cB _i =pr _i ^{-1} \cD ^{(m)}_{X _i}$),
$\widetilde{\cM} _i= pr _i ^{-1}\cM _i$,
$\widetilde{\cE} _i= pr _i ^{-1} \cE _i$.
\end{proof}

\begin{lemm}
\label{com-botimes}
For $i= 1,\dots, n$, 
let $\D _i$ be a sheaf of $\varpi _i ^{-1} \O _S$-algebras.

\begin{enumerate}[(i)]
\item 
\label{com-botimes1} 
For $i= 1,\dots, n$, 
for $ * \in \{Â l,r\}$, let 
$\cM _i\in D ^{-} ({}^*\D _i, \O _{X _i}) $,
$\E _i \in D ^- (\O _{X _i})$.
We have the canonical isomorphism of the form
$\underset{i}{\overset{\L}{\boxtimes}} 
(\cM _i \otimes ^{\L} _{\O _{X _i}} \E _i)
\riso
\underset{i}{\overset{\L}{\boxtimes}}  
\cM _i \otimes ^{\L} _{\O _X} \underset{i}{\overset{\L}{\boxtimes}}   \E _i$ 
of 
$D ^- ({}^*\underset{i}{\boxtimes} ^{\mathrm{top}}\D _i, \O _{X} )$.
Moreover, this isomorphism is compatible with that of \ref{com-botimestop-iso1}, i.e. 
the following diagram 
of $D ^- ({}^*\underset{i}{\boxtimes} ^{\mathrm{top}} \D _i, 
\underset{i}{\boxtimes} ^{\mathrm{top}} \O _{X _i})$
\begin{equation}
\label{com-botimes-diag2}
\xymatrix {
{\underset{i}{\overset{\L}{\boxtimes}} {}^{\mathrm{top}}   
(\cM _i \otimes ^{\L} _{\O _{X _i}} \E _i)}
\ar[r] ^-{\sim}  _-{\ref{com-botimestop-iso1}}
\ar[d] ^-{}
&
{\underset{i}{\overset{\L}{\boxtimes}} {}^{\mathrm{top}}   
\cM _i \otimes ^{\L} _{\underset{i}{\boxtimes} ^{\mathrm{top}}  \O _{X _i}} 
\underset{i}{\overset{\L}{\boxtimes}} {}^{\mathrm{top}}   \E _i} 
\ar[d] 
\\ 
{\underset{i}{\overset{\L}{\boxtimes}} 
(\cM _i \otimes ^{\L} _{\O _{X _i}} \E _i)} 
\ar[r] ^-{\sim}  
&
{\underset{i}{\overset{\L}{\boxtimes}}  
\cM _i \otimes ^{\L} _{\O _X} \underset{i}{\overset{\L}{\boxtimes}}   \E _i}
}
\end{equation}
is commutative.
\item 
\label{com-botimes2} 
For $i= 1,\dots, n$, 
for $ * \in \{Â l,r\}$, let
$\cM _i \in D ^{-} ({}^*  \D _i,  {}^l \D ^{(m)} _{X _i})$,
$\cE _i \in D ^-({}^l \D ^{(m)}  _{X _i})$.
Then, 
the isomorphism
$\underset{i}{\overset{\L}{\boxtimes}} 
(\cM _i \otimes ^{\L} _{\O _{X _i}} \cE _i)
\riso
\underset{i}{\overset{\L}{\boxtimes}}  
\cM _i \otimes ^{\L} _{\O _X} \underset{i}{\overset{\L}{\boxtimes}}   \cE _i$ 
constructed in  \ref{com-botimes}.(\ref{com-botimes1})
is in fact an isomorphism
of 
$D ^- ({}^*  \underset{i}{\boxtimes} ^{\mathrm{top}} \D _i,{}^l \D ^{(m)}  _{X} )$.

\end{enumerate}

\end{lemm}

\begin{proof}
We construct the isomorphism of the form
$\underset{i}{\overset{\L}{\boxtimes}} 
(\cM _i \otimes ^{\L} _{\O _{X _i}} \E _i)
\riso
\underset{i}{\overset{\L}{\boxtimes}}  
\cM _i \otimes ^{\L} _{\O _X} \underset{i}{\overset{\L}{\boxtimes}}   \E _i$ 
as follows :
\begin{gather}
\notag
\underset{i}{\overset{\L}{\boxtimes}} 
(\cM _i \otimes ^{\L} _{\O _{X _i}} \E _i)
\underset{\ref{boxtimes-dfn2L}}{\riso}
pr _1 ^{*} ( \cM _1 \otimes ^{\L} _{\O _{X _1}} \E _1) 
\otimes ^{\L} _{\O _{X} }
\cdots
\otimes ^{\L} _{\O _{X} }
 pr _n ^{*} (\cM _n \otimes ^{\L} _{\O _{X _n}} \E _n)
 \\
 \label{com-botimes-diag1}
 \riso
\left ( pr _1 ^{*} \cM _1 
\otimes ^{\L} _{\O _{X} }
\cdots
\otimes ^{\L} _{\O _{X} }
 pr _n ^{*} \cM _n 
 \right) 
 \otimes ^{\L} _{\O _X}
\left(  pr _1 ^{*} \E _1
\otimes ^{\L} _{\O _{X} }
\cdots
\otimes ^{\L} _{\O _{X} }
 pr _n ^{*} \E _n
 \right) 
\underset{\ref{boxtimes-dfn2L}}{\riso}
\underset{i}{\overset{\L}{\boxtimes}}  
\cM _i \otimes ^{\L} _{\O _X} \underset{i}{\overset{\L}{\boxtimes}}   \E _i.
\end{gather}
We check by an easy computation the commutativity of the diagram \ref{com-botimes-diag2}.
Finally, when 
$\cM _i \in D ^{-} ({}^*  \D _i,  {}^l \D ^{(m)} _{X _i})$,
$\cE _i \in D ^-({}^l \D ^{(m)}  _{X _i})$, the isomorphisms of 
\ref{com-botimes-diag1} are 
$\D ^{(m)}  _{X}$-linear. 
\end{proof}

\begin{empt}
\label{boxtimes-left-right}
We have the splitting of $\O _X$-modules
$\oplus _{i=1} ^n pr _i ^*  \Omega ^1 _{X _i} \riso \Omega ^1 _{X}$. By applying determinants,
this yields the isomorphism of $\O _X$-modules
$\underset{i}{\boxtimes}~  \omega _{X _i} \riso \omega  _{X}$.
Using the canonical structure of 
right $\D  ^{(m)} _{X _i}$-module on $\omega _{X _i}$, we get a structure of 
right $\D  ^{(m)} _{X}$-module on $\underset{i}{\boxtimes} \omega _{X _i}$.
By local computations, we check the canonical isomorphism
$\underset{i}{\boxtimes} \omega _{X _i} \riso \omega  _{X}$ 
is in fact an isomorphism of right $\D  ^{(m)} _{X}$-modules.

For $i= 1,\dots, n$, 
$\E _i $ be a left $\D ^{(m)}  _{X _i}$-module, and 
$\cF _i $ be a right $\D ^{(m)}  _{X _i}$-module.
Then we have the canonical morphism of right $\D  ^{(m)} _{X}$-modules (resp. left $\D  ^{(m)} _{X}$-modules)
$\underset{i}{\boxtimes} ( \omega _{X _i} \otimes _{\O _{X _i}} \cE _i)
\riso \omega _{X } \otimes _{\O _{X }} \underset{i}{\boxtimes}  \cE _i$
(resp. $\underset{i}{\boxtimes}  ( \cF _i \otimes _{\O _{X _i}} \omega _{X _i} ^{-1} )
\riso \underset{i}{\boxtimes}  \cF _i \otimes _{\O _{X}} \omega _{X} ^{-1} $).
Taking flat resolutions, we have similar isomorphisms in derived categories.

\end{empt}

\subsection{Commutation with pull-backs and push forwards}
\label{subsec4.2}
Let 
$S$ be a noetherian $\Z _{(p)}$-scheme of finite Krull dimension.
For $i= 1,\dots, n$, 
let $f _i \colon X _i \to Y _i$ be a (quasi-separated and quasi-compact) morphism of smooth $S$-schemes.
Set $X := X _1 \times _S X _2 \times _S \dots \times _S X _n$,
$Y := Y _1 \times _S Y _2 \times _S \dots \times _S Y _n$,
and $f: = f _1 \times \dots \times f _n \colon X \to Y$. 
For $i= 1,\dots, n$, let $pr _i  \colon X \to X _i$,
$pr '_i  \colon Y \to Y _i$ be the projections.
We denote by 
 $\varpi \colon X \to S$ and $\varpi _i \colon X _i \to S$,
 $\varpi ' \colon Y \to S$ and $\varpi _i ' \colon Y _i \to S$ the structural morphisms

\begin{rem}
\label{n=2,f2=id-diag-parag}
Suppose $n = 2$ and $f _2 $ is the identity. 
In that case, denoting by $T:= X _2=Y_2$, we get the cartesian square
\begin{equation}
\label{n=2,f2=id-diag}
\xymatrix@ C=2cm {
{X= X _1 \times _S T} 
\ar[r] ^-{f= f _1 \times id}
\ar[d] ^-{pr _1}
& 
{Y= Y _1 \times _S T} 
\ar[d] ^-{pr ' _1}
\\ 
{X _1} 
\ar[r] ^-{f _1}
& 
{Y _1.} 
}
\end{equation}

\end{rem}

\begin{empt}
 \label{f_*/otimes}
We recall the following fact. 
Let 
$F\colon \fA\to \fB$ and 
$G\colon \fB\to \fA$ be functors of abelian categories such
that
$F$ is a right adjoint to $G$. 
Let
$\cM $ 
be a complex of $\fA$ 
and let $\cN  $
be a complex of 
$\fB$. If $\R F$ is defined at $\cM$ and $\L G$ is defined at $\cN$, then there is a canonical isomorphism 
$$\mathrm{Hom} _{D(\fB)}(\cN,\R F(\cM)) 
\riso
\mathrm{Hom} _{D(\fA)}
(\L G(\cN),\cM).$$ 
This isomorphism is functorial in both variables on the triangulated subcategories of 
$D(\fA)$ and $D(\fB)$ where $\R F$ and $\L G$ are defined.
In particular, let $u\colon (U, \O _U) \to (V, \O _V)$ 
be a morphism of ringed spaces. Since the functors
$ \R u_*$   
and $\L u ^*$ are well defined, we get
\begin{equation}
\label{f_*/otimes-adjpre}
\mathrm{Hom} _{D ( \O _{V} )}(\cN,\R u _*(\cM)) 
\riso
\mathrm{Hom} _{D ( \O _{U} )}
(\L u ^*(\cN),\cM)
\end{equation}
bifunctorially in $\cM \in \mathrm{Ob} (D ( \O _{U} ))$ and 
$\cN \in \mathrm{Ob} (D( \O _{V} ))$ (see 20.28.1 of the stack project).

Suppose the functor 
$\R u _*$ induces the functor
$\R u _* \colon
D ^- ( \O _{U} )
\to 
D ^- ( \O _{V} )$.
Then, for any
 $\cM ,\cM ' \in D ^- ( \O _{U} )$, we construct the canonical morphism
of 
$D ^- ( \O _{V} )$ :
\begin{equation}
\label{f_*/otimes-iso}
\R u_* ( \cM )Â \otimes ^{\L} _{\O _V}\R u _* ( \cM ')
\to 
\R u _* ( \cMÂ \otimes ^{\L} _{\O _U}\cM ') 
 \end{equation}
bifunctorially in $\cM,\cM ' \in \mathrm{Ob} (D ^- (\O _{U} ))$
as follows.
Using \ref{f_*/otimes-adjpre}, since tensor products commute with inverses images, 
we reduce to construct a canonical morphism
$\mathrm{adj} \otimes \mathrm{adj}
\colon 
\L u ^* \R u_* ( \cM )Â \otimes ^{\L} _{\O _U } \L u ^* \R u _* ( \cM ')
\to 
\cMÂ \otimes ^{\L} _{\O _U}\cM ',$ which again a consequence of \ref{f_*/otimes-adjpre}.

For any $\cM \in D ^- ( \O _{U} )$, 
$\cN  \in D ^- ( \O _{V} )$,
we construct the canonical morphism 
\begin{equation}
\label{f_*/otimes-iso-projform}
\R u_* ( \cM )Â \otimes ^{\L} _{\O _V} \cN
\to 
\R u _* ( \cMÂ \otimes ^{\L} _{\O _U} \L u ^* \cN).
 \end{equation}
as follows. 
Using \ref{f_*/otimes-adjpre}, since tensor products commute with inverses images, 
we reduce to construct a canonical morphism
$\mathrm{adj} \otimes id
\colon 
\L u ^* \R u_* ( \cM )Â \otimes ^{\L} _{\O _U } \L u ^* \cN 
\to 
\cMÂ \otimes ^{\L} _{\O _U}\L u ^* \cN ,$ which again a consequence of \ref{f_*/otimes-adjpre}.

Let $\cI$ be a $K$-injective resolution of $\cM$, 
$\cQ$ be a $K$-flat resolution of $\cN$
(see $20.26.11$ of stack project).
Then \ref{f_*/otimes-iso-projform} is defined by 
the composition 
$\R u_* ( \cM )Â \otimes ^{\L} _{\O _V} \cN
\riso 
 u_* ( \cI )Â \otimes _{\O _V} \cQ
\to 
u _* ( \cIÂ \otimes _{\O _U} u ^* \cQ)
\to 
\R u _* ( \cIÂ \otimes _{\O _U} u ^* \cQ)
\riso 
\R u _* ( \cMÂ \otimes ^{\L} _{\O _U} \L u ^* \cN)$.
Indeed, let $\cP$ be a $K$-flat resolution of $f _* \cI$.
Then 
$\L u ^* \R u_* ( \cM ) \to \cM$ is defined by 
$\L u ^* \R u_* ( \cM ) 
\riso 
u ^* \cP 
\to 
u ^* u_* ( \cI)
\to 
\cI
\riso
\M$.
Hence,
$\mathrm{adj} \otimes id
\colon 
\L u ^* \R u_* ( \cM )Â \otimes ^{\L} _{\O _U } \L u ^* \cN 
\to 
\cMÂ \otimes ^{\L} _{\O _U}\L u ^* \cN $
is given by 
$\L u ^* \R u_* ( \cM )Â \otimes ^{\L} _{\O _U } \L u ^* \cN 
\riso 
u ^* \cP \otimes _{\O _U } u ^* \cQ
\to 
u ^* u_* ( \cI) \otimes _{\O _U } u ^* \cQ
\to 
\cI \otimes _{\O _U } u ^* \cQ
\riso
\M \otimes ^{\L} _{\O _U } \L u ^* \cN $.

\begin{equation}
\xymatrix{
{\R u_* ( \cM )Â \otimes ^{\L} _{\O _V} \cN} 
\ar[d] ^-{\mathrm{adj}}
\ar[r] ^-{\sim}
&
{\cP \otimes _{\O _V } \cQ} 
\ar[r] ^-{\sim}
\ar[d] ^-{\mathrm{adj}}
& 
{u_* \cI \otimes _{\O _V } \cQ} 
\ar[r] ^-{}
\ar[d] ^-{\mathrm{adj}}
&
{ u _* ( \cI \otimes _{\O _U } u ^* \cQ)} 
\ar@{=}[dd] ^-{}
\\ 
{\R u _* \L u ^* ( \R u_* ( \cM )Â \otimes ^{\L} _{\O _V} \cN)} 
\ar[ddd] ^-{\sim}
&
{u _* u ^* ( \cP \otimes _{\O _V}  \cQ)} 
\ar[r] ^-{}
\ar[l] ^-{}
\ar[d] ^-{\sim}
& 
{u _* u ^* ( u _* \cI  \otimes _{\O _V}  \cQ)} 
\ar[d] ^-{\sim}
& 
{} 
\\ 
&
{u _* (u ^* \cP \otimes _{\O _U } u ^* \cQ)} 
\ar[r] ^-{}
\ar[d] ^-{}
& 
{u _* (u ^* u_* ( \cI) \otimes _{\O _U } u ^* \cQ)} 
\ar[r] ^-{\mathrm{adj}}
\ar[d] ^-{}
& 
{ u _* ( \cI \otimes _{\O _U } u ^* \cQ)} 
\ar[d] ^-{}
\\ 
{}
&
{\R u _* (u ^* \cP \otimes _{\O _U } u ^* \cQ)} 
\ar[r] ^-{}
& 
{\R u _* (u ^* u_* ( \cI) \otimes _{\O _U } u ^* \cQ)} 
\ar[r] ^-{\mathrm{adj}}
& 
{\R u _* ( \cI \otimes _{\O _U } u ^* \cQ)} 
\ar[d] ^-{\sim}
\\ 
{\R u _* ( \L u ^* \R u_* ( \cM )Â \otimes ^{\L} _{\O _U } \L u ^* \cN )} 
\ar[ur] ^-{\sim}
\ar[rrr] ^-{\mathrm{adj}}
&&&
{\R u _* ( \cMÂ \otimes ^{\L} _{\O _U} \L u ^* \cN)} 
}
\end{equation}
Hence, we get the commutativity of the canonical diagram
\begin{equation}
\label{f_*/otimes-iso-projformter}
\xymatrix{
{\R u_* ( \cM )Â \otimes ^{\L} _{\O _V} \cO _V}
\ar[r] ^-{\ref{f_*/otimes-iso-projform}}
& 
{\R u _* ( \cMÂ \otimes ^{\L} _{\O _U} \L u ^* \cO _V )} 
\\ 
{\R u_* ( \cM )Â .} 
\ar[u] ^-{\sim}
\ar[ur] ^-{\sim}
& 
{ } 
}
\end{equation}
This yields (using standard methods) that when $u$ is a quasi-separated and quasi-compact morphism of 
 noetherian schemes of finite Krull dimension,
the projection morphism \ref{f_*/otimes-iso-projform} is an isomorphism for complexes with 
bounded above and quasi-coherent cohomology.
The commutativity of \ref{f_*/otimes-iso-projformter}
will be used to check the commutativity of 
\ref{sch-prop-boxtimes-v+isoproofbis2}.

\end{empt}

\begin{lem}
\label{lem-f_*/otimes-iso-projform2}
Let $u\colon (U, \O _U) \to (V, \O _V)$ 
be a morphism of ringed spaces. 
For any $\cM \in D ^- ( \O _{U} )$, 
the following diagram is commutative:
\begin{equation}
\label{f_*/otimes-iso-projform2}
\xymatrix{
{\R u_* ( \cM )Â \otimes ^{\L} _{\O _V} \R u_*  \L u ^*  ( \cN ) } 
\ar[r] ^-{\ref{f_*/otimes-iso}}
& 
{\R u _* ( \cMÂ \otimes ^{\L} _{\O _U} \L u ^* \cN)} 
\\ 
{\R u_* ( \cM )Â \otimes ^{\L} _{\O _V} \cN.} 
\ar[u] ^-{id \otimes \mathrm{adj}}
\ar[ur] ^-{\ref{f_*/otimes-iso-projform}}
& 
{ } 
}
\end{equation}
\end{lem}

\begin{proof}
Since the composition 
$ \L u ^*  ( \cN ) 
\overset{\mathrm{adj}}{\longrightarrow}
\L u ^* \R u_*  \L u ^*  ( \cN ) 
\overset{\mathrm{adj}}{\longrightarrow}
 \L u ^*  ( \cN ) $ is the identity, 
we get the commutative diagram :
\begin{equation}
\label{f_*/otimes-iso-projform2pre}
\xymatrix{
{\L u ^* \R u_* ( \cM )Â \otimes ^{\L} _{\O _U} \L u ^* \R u_*  \L u ^*  ( \cN ) } 
\ar[r] ^-{\mathrm{adj} \otimes \mathrm{adj}}
& 
{\cMÂ \otimes ^{\L} _{\O _U} \L u ^* \cN} 
\\ 
{\L u ^* \R u_* ( \cM )Â \otimes ^{\L} _{\O _U} \L u ^* \cN.} 
\ar[u] ^-{id \otimes \mathrm{adj}}
\ar[ur] ^-{\mathrm{adj} \otimes id}
& 
{ } 
}
\end{equation}
By construction of both morphisms
\ref{f_*/otimes-iso} and 
\ref{f_*/otimes-iso-projform}, 
from the commutative diagram
\ref{f_*/otimes-iso-projform2pre} 
we get by adjunction 
(via \ref{f_*/otimes-adjpre}) 
the commutative diagram
\ref{f_*/otimes-iso-projform2}. 
\end{proof}

\begin{lemm}
\label{comm-boxtimes-f*}
For $i= 1,\dots, n$, 
let $\D _i$ be a sheaf 
of $\varpi _i ^{\prime -1} \O _S$-algebras,
$\cN _i \in D ^- (\D _i,  \D ^{(m)} _{Y _i})$.
We have the canonical isomorphism of 
$D ^- (\underset{i}{\boxtimes} ^{\mathrm{top}} \D _i,\D ^{(m)}  _{X _i} )$:
$$\L f ^* ( \underset{i}{\overset{\L}{\boxtimes}}  \cN _i )
\riso 
\underset{i}{\overset{\L}{\boxtimes}}  \,
\L f _i ^* (  \cN _i ) .$$
\end{lemm}

\begin{proof}
This is a consequence of \ref{boxtimes-dfn2L} and of the commutation of inverse images with tensor products.
\end{proof}

\begin{lem}
\label{lem-Rf*adjf-1*}
Let $f \colon X \to Y$,
$g \colon Y \to Y'$,
$f' \colon X ' \to Y '$
and 
$g ' \colon X  \to X'$
 be some morphism of $S$-schemes such that 
 $g \circ f = f ' \circ g'$.
We suppose $g$ and $g'$ flat. 
For any $\E ' \in D  ( \O _{X'} )$, 
the canonical diagram 
\begin{equation}
\xymatrix{
{g ^{-1} \R f  '_* (\E')} 
\ar[r] ^-{\mathrm{adj}}
\ar[d] ^-{}
& 
{\R f _* g ^{\prime -1}  (\E')} 
\ar[d] ^-{}
\\ 
{g ^{*} \R f ' _* (\E')} 
\ar[r] ^-{\mathrm{adj}}
& 
{\R f _* g ^{\prime *}  (\E')} 
}
\end{equation}
is commutative. 
 
\end{lem}

\begin{proof}
Since $f$ has finite cohomological dimension, then 
$\E'$ has a resolution $\I' $ by $f _*$-acyclic modules 
(see \cite[Lemma I.4.6]{HaRD}). 
Then, the lemma follows from the commutative diagram.
\begin{equation}
\xymatrix{
{g ^{-1}  f ' _* (\cI ')} 
\ar[r] ^-{\mathrm{adj}}
\ar[d] ^-{}
& 
{f _* f ^{-1} g ^{-1}  f '_* (\cI')} 
\ar[r] ^-{\sim}
\ar[d] ^-{}
& 
{f _*  g ^{\prime -1}  f ^{\prime -1} f ' _* (\cI')} 
\ar[r] ^-{\mathrm{adj}}
\ar[d] ^-{}
& 
{ f _* g ^{\prime -1}  (\cI')} 
\ar[r] ^-{}
\ar[d] ^-{}
& 
{\R f _* g ^{\prime -1}  (\E')} 
\ar[d] ^-{}
\\ 
{g ^{*}  f ' _* (\cI')} 
\ar[r] ^-{\mathrm{adj}}
& 
{f _* f ^{*} g ^{*}  f '_* (\cI')} 
\ar[r] ^-{\sim}
& 
{f _* g ^{\prime *}  f ^{\prime *}  f ' _* (\cI)'} 
\ar[r] ^-{\mathrm{adj}}
& 
{ f _* g ^{\prime*}  (\cI')} 
\ar[r] ^-{}
& 
{\R f _* g ^{\prime*}  (\E').} 
}
\end{equation}
\end{proof}

\begin{lem}
\label{lem-Rfi*2Rf*top}
For $i= 1,\dots, n$, 
let $\D _i$ be a sheaf 
of $\varpi _i ^{\prime -1} \O _S$-algebras.

\begin{enumerate}[(a)]
\item \label{lem-Rfi*2Rf*top(a)}
\label{lem-Rfi*2Rf*top1}
For $i= 1,\dots, n$, let 
$\cE _i\in D ^{-}  (f _i ^{-1}\D _i)$.
Then we have the canonical morphism 
$\underset{i}{\overset{\L}{\boxtimes}} {}^{\mathrm{top}} ( \R f _{i*} \E _i)
\to 
\R f _{*} (\underset{i}{\overset{\L}{\boxtimes}} {}^{\mathrm{top}}  \E _i)$
of 
$D ^{-} (
\underset{i}{\boxtimes} ^{\mathrm{top}} 
\D _i)$.

\item 
\label{lem-Rfi*2Rf*top2}
For $i= 1,\dots, n$, let 
$\cE _i\in D ^{-}  (f _i ^{-1}\D _i, \O _{X _i})$.
Then the canonical morphism 
$\underset{i}{\overset{\L}{\boxtimes}} {}^{\mathrm{top}} ( \R f _{i*} \E _i)
\to 
\R f _{*} (\underset{i}{\overset{\L}{\boxtimes}} {}^{\mathrm{top}}  \E _i)$
is also a morphism of 
$D ^{-} (
\underset{i}{\boxtimes} ^{\mathrm{top}} 
\D _i,
\underset{i}{\boxtimes} ^{\mathrm{top}} 
\O _{Y _i})$.
Moreover we have  the canonical morphism
$\underset{i}{\overset{\L}{\boxtimes}} ( \R f _{i*} \E _i)
\to 
\R f _{*} ( \underset{i}{\overset{\L}{\boxtimes}} \E _i) $
of 
$D ^{-}(
\underset{i}{\boxtimes} ^{\mathrm{top}} 
\D _i,
\O _{Y })$
making commutative the diagram
\begin{equation}
\label{lem-Rfi*2Rf*top-diag}
\xymatrix{
{\underset{i}{\overset{\L}{\boxtimes}} {}^{\mathrm{top}} ( \R f _{i*} \E _i} )
\ar[r] ^-{}
\ar[d] ^-{}
& 
{\underset{i}{\overset{\L}{\boxtimes}} ( \R f _{i*} \E _i)} 
\ar[d] ^-{}
\\ 
{\R f _{*} (\underset{i}{\overset{\L}{\boxtimes}} {}^{\mathrm{top}}  \E _i)} 
\ar[r] ^-{}
& 
{\R f _{*} ( \underset{i}{\overset{\L}{\boxtimes}} \E _i) .} 
}
\end{equation}

\end{enumerate}

\end{lem}

\begin{proof}
0) If 
$\G _i \in  D ^{-} (\varpi _i ^{-1}\O _{S})$ 
(resp. 
$\G _i \in  D ^{-} (\varpi _i ^{\prime -1}\O _{S})$),
then we set 
$\underset{i}{\overset{\L}{\otimes}} \G _i 
:= 
\G _1 
\otimes ^{\L} _{\O _S}
\G _2
\otimes ^{\L} _{\O _S}
\cdots
\otimes ^{\L} _{\O _S}
\G _n$. 

a) 
 By applying \ref{f_*/otimes-iso} to the morphism of ringed spaces
$(X, \varpi ^{-1} \O _S)
\to 
(Y, \varpi ^{\prime -1} \O _S)$, we get the morphism 
$\underset{i}{\overset{\L}{\otimes}} ( \R f _{*} pr ^{-1} _i  \E _i)
\to 
\R f _{*} \left (\underset{i}{\overset{\L}{\otimes}} ( pr ^{-1} _i  \E _i) \right )$.
By functoriality, since 
$pr ^{-1} _i  \E _i \in D ^- (f ^{-1} pr ^{\prime -1} _i\D _i , \varpi ^{-1} \O _S)$
this latter morphism belongs to 
$D ^{-} (
\underset{i}{\boxtimes} ^{\mathrm{top}} 
\D _i)$.

b) Similarly, by applying \ref{f_*/otimes-iso} to the morphism of ringed spaces
$(X, \varpi ^{-1} \O _S)
\to 
(Y, \varpi ^{\prime -1} \O _S)$, we get the morphisms 
$\underset{i}{\overset{\L}{\otimes}} ( \R f _{*} pr ^{-1} _i  \E _i)
\to 
\R f _{*} \left (\underset{i}{\overset{\L}{\otimes}} ( pr ^{-1} _i  \E _i) \right )$
and
$\underset{i}{\overset{\L}{\otimes}} ( \R f _{*} pr ^{*} _i  \E _i)
\to 
\R f _{*} \left (\underset{i}{\overset{\L}{\otimes}} ( pr ^{*} _i  \E _i) \right )$
of 
$D ^{-} (
\underset{i}{\boxtimes} ^{\mathrm{top}} 
\D _i,
\underset{i}{\boxtimes} ^{\mathrm{top}} 
\O _{Y _i})$.
Finally, by applying \ref{f_*/otimes-iso} to the morphism $X \to Y$,
we get the morphisms 
$\underset{\O _Y}{\otimes ^{\L}} ( \R f _{*} pr ^{*} _i  \E _i)
\to 
\R f _{*} \left (\underset{i}{\overset{\L}{\otimes}} ( pr ^{*} _i  \E _i) \right )$
of 
$D ^{-} (
\underset{i}{\boxtimes} ^{\mathrm{top}} 
\D _i,
\O _Y)$.

1) The  canonical morphism 
$\underset{i}{\overset{\L}{\boxtimes}} {}^{\mathrm{top}} ( \R f _{i*} \E _i)
\to 
\R f _{*} \underset{i}{\overset{\L}{\boxtimes}} {}^{\mathrm{top}} (  \E _i)$
(resp. 
$\underset{i}{\overset{\L}{\boxtimes}} ( \R f _{i*} \E _i)
\to 
\R f _{*} ( \underset{i}{\overset{\L}{\boxtimes}} \E _i) $
is by definition 
is the one making commutative the right (resp. left) rectangle of 
the diagram \ref{Rfi*2Rf*dfn} below:
\begin{equation}
\label{Rfi*2Rf*dfn}
\xymatrix{
{\underset{i}{\overset{\L}{\boxtimes}} {}^{\mathrm{top}} ( \R f _{i*} \E _i)}  
\ar@{=}[r] ^-{}
\ar@{.>}[dd] ^-{}
& 
{\underset{i}{\overset{\L}{\otimes}} ( pr ^{\prime -1} _i \R f _{i*} \E _i)}  
\ar[r] ^-{}
\ar[d] ^-{\mathrm{adj}}
& 
{\underset{i}{\overset{\L}{\otimes}} ( pr ^{\prime *} _i \R f _{i*} \E _i)}  
\ar[r] 
\ar[d] ^-{\mathrm{adj}}
& 
{\underset{\O _Y}{\overset{\L}{\otimes}} ( pr ^{\prime *} _i \R f _{i*} \E _i)}  
\ar[r] ^-{\sim}
\ar[d] ^-{\mathrm{adj}}
&
{\underset{i}{\overset{\L}{\boxtimes}}  ( \R f _{i*} \E _i)}  
\ar@{.>}[dd] ^-{}
\\ 
& 
{\underset{i}{\overset{\L}{\otimes}} ( \R f _{*} pr ^{-1} _i  \E _i)}  
\ar[r] ^-{}
\ar[d] ^-{\ref{f_*/otimes-iso}}
& 
{\underset{i}{\overset{\L}{\otimes}} ( \R f _{*} pr ^{*} _i  \E _i)}  
\ar[d] ^-{\ref{f_*/otimes-iso}}
\ar[r] ^-{}
& 
{\underset{\O _Y}{\overset{\L}{\otimes}} ( \R f _{*} pr ^{*} _i  \E _i)}  
\ar[d] ^-{\ref{f_*/otimes-iso}}
&
\\
{\R f _{*} \underset{i}{\overset{\L}{\boxtimes}} {}^{\mathrm{top}} (  \E _i)}  
\ar@{=}[r] ^-{}
& 
{\R f _{*} \left (\underset{i}{\overset{\L}{\otimes}} ( pr ^{-1} _i  \E _i) \right )}  
\ar[r] ^-{}
& 
{\R f _{*} \left (\underset{i}{\overset{\L}{\otimes}} ( pr ^{*} _i  \E _i) \right )}  
\ar[r] 
& 
{\R f _{*} \left (\underset{\O _X}{\overset{\L}{\otimes}} ( pr ^{*} _i  \E _i) \right )}  
\ar[r] ^-{\sim}
&
{\underset{i}{\overset{\L}{\boxtimes}}  ( \R f _{i*} \E _i)}  
}
\end{equation}
Using \ref{lem-Rf*adjf-1*}, we get the commutativity of the top left square.
We check the commutativity of the bottom right square by construction of 
both vertical arrows. 
The other squares are commutative by functoriality. 
Since the right and left rectangles are commutative by definition,
then the diagram \ref{Rfi*2Rf*dfn} is commutative. 
Finally, the composition of the top (resp. bottom) morphisms of \ref{Rfi*2Rf*dfn}
is the canonical morphism
$\underset{i}{\overset{\L}{\boxtimes}} {}^{\mathrm{top}} ( \R f _{i*} \E _i)
\to \underset{i}{\overset{\L}{\boxtimes}}  ( \R f _{i*} \E _i)$
(resp.
$\R f _{*} \underset{i}{\overset{\L}{\boxtimes}} {}^{\mathrm{top}} (  \E _i)
\to 
\underset{i}{\overset{\L}{\boxtimes}}  ( \R f _{i*} \E _i)$).
\end{proof}

\begin{empt}
\label{lem-boxtimesDleftarrow}
By flatness, the morphism
$\underset{i}{\overset{\L}{\boxtimes}}
 \cD _{Y _i \leftarrow X _i} ^{(m)} 
 \to \underset{i}{\boxtimes}  
 \cD _{Y _i \leftarrow X _i} ^{(m)} $
 is an isomorphism.
Moreover, 
using \ref{boxtimes-left-right},
 we get by functoriality the canonical isomorphism of 
$(f ^{-1} \underset{i}{\boxtimes} ^{\mathrm{top}} \D ^{(m)}  _{Y _i} , \D  ^{(m)} _{X})$-bimodules
\begin{equation}
\label{boxtimesDleftarrow}
\underset{i}{\boxtimes}  
 \cD _{Y _i \leftarrow X _i} ^{(m)} 
 =
 \underset{i}{\boxtimes}  
 \left (
\omega _{X _i} \otimes _{\O _{X _i}} f ^* _i ( \cD _{Y _i} ^{(m)} \otimes _{\O _{Y _i}} \omega _{Y _i} ^{-1} )
\right) 
\underset{\ref{comm-boxtimes-f*}}{\riso}
\omega _{X} \otimes _{\O _{X}}  
f ^*  ( \underset{i}{\boxtimes}  
\cD _{Y _i} ^{(m)} 
\otimes _{\O _{Y}} \omega _{Y} ^{-1} )
\riso
 \cD _{Y \leftarrow X } ^{(m)},
\end{equation}
where the last isomorphism is a consequence of 
the canonical isomorphism
$\underset{i}{\boxtimes} 
\D ^{(m)} _{Y _i}
\riso \D ^{(m)} _{Y }$.
This yields a canonical structure of 
$(f ^{-1} \D ^{(m)}  _{Y} , \D  ^{(m)} _{X})$-bimodule
on $\underset{i}{\boxtimes}  
 \cD _{Y _i \leftarrow X _i} ^{(m)} $ making 
$(f ^{-1} \D ^{(m)}  _{Y} , \D  ^{(m)} _{X})$-bilinear
the composite isomorphism \ref{boxtimesDleftarrow}. 
Similarly for $f$.
\end{empt}

\begin{thm}
\label{sch-prop-boxtimes-v*}
For $i= 1,\dots, n$, let 
$\cE _i\in D ^{\mathrm{b}} _{\mathrm{qc}} (\O _{X _i})$.
The canonical morphism
\begin{equation}
\label{sch-prop-boxtimes-v*iso}
\underset{i}{\overset{\L}{\boxtimes}}  
\R f _{i*} (\cE _i )
\to
\R f _* (\underset{i}{\overset{\L}{\boxtimes}}   \cE _i ).
\end{equation}
is an isomorphism.
\end{thm}

\begin{proof}
i) By construction of the morphism
\ref{sch-prop-boxtimes-v*iso} (i.e. the one making commutative the right rectangle of \ref{Rfi*2Rf*dfn}), 
we reduce to the case where $n=2$. We have to check that  the composition
\begin{equation}
\label{sch-prop-boxtimes-v+isoproof}
pr ^{\prime *} _1 \R f _{1*}( \E _1)
\otimes ^\L _{\O _Y}
pr ^{\prime *} _2 \R f _{2*} \E _2
\underset{\mathrm{adj}}{\longrightarrow}
\R f _{*} pr ^{*} _1  (\E _1)
\otimes ^\L _{\O _Y}
\R f _{*}  pr ^{*} _2 \E _2
\underset{\ref{f_*/otimes-iso}}{\longrightarrow}
\R f _{*} ( 
pr ^{*} _1  (\E _1)
\otimes ^\L _{\O _X}
pr ^{*} _2 \E _2)
\end{equation}
is an isomorphism.

ii) We reduce to the case where $f _2= id$ as follows. 
Consider the commutative diagram
\begin{equation}
\label{sch-prop-boxtimes-v*-diagXYtildeY}
\xymatrix{
{X _1}
\ar@{=}[d] ^-{} 
& 
{X=X _1 \times X _2} 
\ar[r] ^-{pr _2}
\ar[l] _-{pr _1}
\ar[d] ^-{\widetilde{f} _2:= id \times f _2}
&
{X _2} 
\ar[d] ^-{f _2}
\\ 
{X _1}
\ar[d] ^-{f _1}
& 
{\widetilde{Y}:=X _1 \times Y _2} 
\ar[r] ^-{\widetilde{pr}_2}
\ar[l] _-{\widetilde{pr} _1}
\ar[d] ^-{\widetilde{f} _1:= f _1 \times id}
&
{X _2} 
\ar@{=}[d] ^-{} 
\\
{Y _1}
& 
{Y:=Y _1 \times Y _2} 
\ar[r] ^-{pr ' _2}
\ar[l] _-{pr ' _1}
&
{X _2.} 
}
\end{equation}
By adjunction with respect to the left bottom square of \ref{sch-prop-boxtimes-v*-diagXYtildeY}
(resp. the left top square of \ref{sch-prop-boxtimes-v*-diagXYtildeY})
the morphism 
$pr ^{\prime *} _1 \R f _{1*}( \E _1)
\overset{\mathrm{adj}}{\longrightarrow} 
\R \widetilde{f} _{1*} \widetilde{pr} ^{*} _1 ( \E _1)$
and 
$\widetilde{pr} ^{*} _1 ( \E _1)
=
\widetilde{pr} ^{*} _1 id _*  ( \E _1)
\overset{\mathrm{adj}}{\longrightarrow}
\R \widetilde{f} _{2*} pr ^{*} _1 ( \E _1)$.
By transitivity,
we get that the composition 
$pr ^{\prime *} _1 \R f _{1*}( \E _1)
\overset{\mathrm{adj}}{\longrightarrow} 
\R \widetilde{f} _{1*} \widetilde{pr} ^{*} _1 ( \E _1)
\overset{\mathrm{adj}}{\longrightarrow} 
\R \widetilde{f} _{1*} \R \widetilde{f} _{2*} pr ^{*} _1 ( \E _1)
\riso 
\R f_{*} pr ^{*} _1 ( \E _1)$
is the adjunction morphism with respect to the left rectangle of \ref{sch-prop-boxtimes-v*-diagXYtildeY}
(i.e. the composite of both left squares).
Similarly, 
we get by transitivity that the composition
$pr ^{\prime *} _2 \R f _{2*} \E _2
\overset{\mathrm{adj}}{\longrightarrow} 
\R \widetilde{f} _{1*} 
\widetilde{pr} ^{*}   _2 
\R f _{2*} \E _2
\overset{\mathrm{adj}}{\longrightarrow} 
\R \widetilde{f} _{1*} 
\R \widetilde{f} _{2*}
pr ^{*}   _2  \E _2
\riso 
\R f _{*} 
pr ^{*}   _2  \E _2$ is the adjunction morphism.
This yields the commutativity of the top left square of the following diagram:
\begin{equation}
\label{sch-prop-boxtimes-v*-diagXYtildeY2}
\xymatrix{
{pr ^{\prime *} _1 \R f _{1*}( \E _1)
\underset{\O _Y}{\overset{\L}{\otimes}} 
pr ^{\prime *} _2 \R f _{2*} \E _2} 
\ar[r] ^-{\mathrm{adj}}
\ar[d] ^-{\mathrm{adj}}
& 
{\R \widetilde{f} _{1*} \widetilde{pr} ^{*} _1 ( \E _1)
\underset{\O _{Y}}{\overset{\L}{\otimes}} 
\R \widetilde{f} _{1*} 
\widetilde{pr} ^{*}   _2 
\R f _{2*} \E _2} 
\ar[r] ^-{\ref{f_*/otimes-iso}}
\ar[d] ^-{\mathrm{adj}}
&
{\R \widetilde{f} _{1*} \left (
\widetilde{pr} ^{*} _1 ( \E _1)
\underset{\O _{\widetilde{Y}}}{\overset{\L}{\otimes}} 
\widetilde{pr} ^{*}   _2 
\R f _{2*} \E _2
\right) } 
\ar[d] ^-{\mathrm{adj}}
\\ 
{\R f _{*} pr ^{*} _1 ( \E _1)
\underset{\O _{Y}}{\overset{\L}{\otimes}} 
\R f _{*} 
pr ^{*}   _2  \E _2} 
\ar[r] ^-{\sim}
\ar[d] ^-{\ref{f_*/otimes-iso}}
& 
{\R \widetilde{f} _{1*} \R \widetilde{f} _{2*} pr ^{*} _1 ( \E _1)
\underset{\O _{Y}}{\overset{\L}{\otimes}} 
\R \widetilde{f} _{1*} 
\R \widetilde{f} _{2*}
pr ^{*}   _2  \E _2} 
\ar[r] ^-{\ref{f_*/otimes-iso}}
& 
{\R \widetilde{f} _{1*} 
\left (
\R \widetilde{f} _{2*} pr ^{*} _1 ( \E _1)
\underset{\O _{\widetilde{Y}}}{\overset{\L}{\otimes}} 
\R \widetilde{f} _{2*}
pr ^{*}   _2 
 \E _2
\right)  } 
\ar[d] ^-{\ref{f_*/otimes-iso}}
\\
{\R f _{*} \left (pr ^{*} _1 ( \E _1)
\underset{\O _{X}}{\overset{\L}{\otimes}} 
pr ^{*}   _2  \E _2\right) } 
\ar[rr] ^-{\sim}
& 
{ } 
&
{\R \widetilde{f} _{1*} \R \widetilde{f} _{2*} 
\left (
pr ^{*} _1 ( \E _1)
\underset{\O _{X}}{\overset{\L}{\otimes}} 
pr ^{*}   _2 
 \E _2
\right),  }
}
\end{equation}
where the left vertical morphisms (resp. the top horizontal morphisms, resp. the right vertical morphisms), 
are the morphisms of \ref{sch-prop-boxtimes-v+isoproof} (resp. 
the morphisms of \ref{sch-prop-boxtimes-v+isoproof} in the case where $f _2= id$, resp. where $f _1 = id$).
By transitivity of the morphism of the form \ref{f_*/otimes-iso}, 
we get the commutativity of the bottom rectangle. The top right square is commutative by functoriality.
Hence, the diagram \ref{sch-prop-boxtimes-v*-diagXYtildeY2} is commutative.
By stability of the quasi-coherence under (topological) push-forwards, 
$\R f _{2*} \E _2$ is quasi-coherent. Hence, the case where $f _1$ or $f _2$ is the identity implies the general case. 
By symmetry, we can suppose $f _2=id$.

iii) Consider the commutative diagram
\small
\begin{equation}
\label{sch-prop-boxtimes-v+isoproofbis}
\xymatrix @C=0,45cm {
{pr ^{\prime *} _1 \R f _{1*}( \E _1)
\otimes ^\L _{\O _Y}
pr ^{\prime *} _2 \R f _{2*} \E _2}
\ar[r] ^-{\mathrm{adj}}
\ar[d] ^-{\sim}
&
{\R f _{*} pr ^{*} _1  (\E _1)
\otimes ^\L _{\O _Y}
\R f _{*}  pr ^{*} _2 \E _2}
\ar[r] ^-{\ref{f_*/otimes-iso}}
&
{\R f _{*} ( 
pr ^{*} _1  (\E _1)
\otimes ^\L _{\O _X}
pr ^{*} _2 \E _2)} 
\\ 
{\R f _{*} pr ^{*} _1  (\E _1)
\otimes ^\L _{\O _Y}
pr ^{\prime *} _2 \R f _{2*} \E _2}
\ar@{=}[d] 
\ar[ru] ^-{}
&
{\R f _{*} pr ^{*} _1  (\E _1)
\otimes ^\L _{\O _Y}
\R f _{*} pr ^{ *} _2 \L f _2 ^{*}  \R f _{2*} \E _2}
\ar[r] ^-{}
\ar[u] ^-{}
&
{\R f _{*} ( 
pr ^{*} _1  (\E _1)
\otimes ^\L _{\O _X}
pr ^{ *} _2 \L f _2 ^{*}  \R f _{2*} \E _2)}
\ar[u] ^-{}
\\ 
{\R f _{*} pr ^{*} _1  (\E _1)
\otimes ^\L _{\O _Y}
pr ^{\prime *} _2 \R f _{2*} \E _2}
\ar[r] ^--{id \otimes \mathrm{adj}}
&
{\R f _{*} pr ^{*} _1  (\E _1)
\otimes ^\L _{\O _Y}
\R f _{*} \L f ^{*} pr ^{\prime *} _2 \R f _{2*} \E _2}
\ar[r] ^-{\ref{f_*/otimes-iso}}
\ar[u] ^-{\sim}
&
{\R f _{*} ( 
pr ^{*} _1  (\E _1)
\otimes ^\L _{\O _X}
\L f ^{*} pr ^{\prime *} _2 \R f _{2*} \E _2)}
\ar[u] ^-{\sim}
}
\end{equation}
\normalsize
where the top horizontal morphisms are
\ref{sch-prop-boxtimes-v+isoproof},
where the left vertical isomorphism is induced by the base change morphism
$pr ^{\prime *} _1 \R f _{1*}( \E _1)
\to
\R f _{*} pr ^{*} _1  (\E _1)$ which is in our case an isomorphism (see \cite[II.5.12]{HaRD}),
where the trapeze  is commutative
by construction of the base change isomorphism
$pr ^{\prime *} _2 \R f _{2*} \E _2 \riso \R f _{*}  pr ^{*} _2 \E _2$,
where the composition of 
the bottom horizontal morphisms is the projection morphism
and hence is an isomorphism (see \ref{lem-f_*/otimes-iso-projform2}).
Since $f _2 =id$, then the top middle and right vertical arrows are identity morphisms.
Hence, the composition of the top arrows of the diagram \ref{sch-prop-boxtimes-v+isoproofbis} is an isomorphism. 
\end{proof}

\begin{ntn}
\label{ntn-Tf}
Let 
$\cE _i\in D ^{\mathrm{b}} _{\mathrm{qc}} (\D ^{(m)} _{X _i})$.
We denote by 
$\mathrm{T} _{f _i}
\colon 
\colon \R f _{i,*}  (\cE _i)
\to
\R f _{i,*}  (
 \cD _{Y _i  \leftarrow X _i} ^{(m)} \otimes _{  \cD _{X _i} ^{(m)}} ^{\L}
\cE _i)
=
f _{i,+} ^{ (m)}(\cE_i)$,
the canonical morphism induced by
the homomorphism 
$\cD _{X _i} ^{(m)}
\to 
\cD _{Y _i  \leftarrow X _i} ^{(m)}  $
given by the left $\cD _{X _i} ^{(m)}$-module structure of 
$\cD _{Y _i  \leftarrow X _i} ^{(m)}  $.

Let 
$\cM _i\in D ^{\mathrm{b}} _{\mathrm{qc}} ({} ^r \D ^{(m)} _{X _i})$.
We denote by 
$\mathrm{T} _{f _i}
\colon 
\colon \R f _{i,*}  (\cM _i)
\to
\R f _{i,*}  (
\cM _i \otimes _{  \cD _{X _i} ^{(m)}} ^{\L}
 \cD _{X _i  \to  Y _i} ^{(m)})
=
f _{i,+} ^{ (m)}(\cM_i)$,
the canonical morphism induced by
the homomorphism 
$\cD _{X _i} ^{(m)}
\to 
\cD _{X _i  \to Y _i} ^{(m)}  $
given by the right $\cD _{X _i} ^{(m)}$-module structure of 
$\cD _{X _i  \to Y _i} ^{(m)}  $.

\end{ntn}

\begin{thm}
\label{sch-prop-boxtimes-v+}
For $i= 1,\dots, n$, let 
$\cE _i\in D ^{\mathrm{b}} _{\mathrm{qc}} ({} ^{*}\D ^{(m)} _{X _i})$, with 
$*= r$ or $*=l$.
We have the canonical isomorphism
\begin{equation}
\label{sch-prop-boxtimes-v+iso}
\underset{i}{\overset{\L}{\boxtimes}}  f _{i+} ^{ (m)} (\cE _i )
\riso
f _+ ^{ (m)} (\underset{i}{\overset{\L}{\boxtimes}}   \cE _i )
\end{equation}
of $D ^{\mathrm{b}} _{\mathrm{qc}} ({} ^{*}\D ^{(m)} _{Y})$
making commutative the canonical diagram
\begin{equation}
\label{sch-prop-boxtimes-v+iso2}
\xymatrix @ R=0,3cm{
{\underset{i}{\overset{\L}{\boxtimes}} {}^{\mathrm{top}} ( \R f _{i*} \E _i} )
\ar[r] ^-{}
\ar[d] ^-{}
&
{\underset{i}{\overset{\L}{\boxtimes}}  
\R f _{i*} (\cE _i )} 
\ar[d] ^-{\ref{sch-prop-boxtimes-v*iso}} _-{\sim}
\ar[r] ^-{\underset{i}{\overset{\L}{\boxtimes}} \mathrm{T} _{f_i}} 
_-{\ref{ntn-Tf}}
& 
{\underset{i}{\overset{\L}{\boxtimes}}  f _{i+} ^{ (m)} (\cE _i )} 
\ar[d] ^-{\sim}
\\
{\R f _{*} (\underset{i}{\overset{\L}{\boxtimes}} {}^{\mathrm{top}}  \E _i)} 
\ar[r] ^-{}
&  
{\R f _* (\underset{i}{\overset{\L}{\boxtimes}}   \cE _i )} 
\ar[r] ^-{\mathrm{T} _{f}} _-{\ref{ntn-Tf}}
& 
{f _+ ^{ (m)} (\underset{i}{\overset{\L}{\boxtimes}}   \cE _i ). } 
}
\end{equation}

\end{thm}

\begin{proof}
0) Since the case where $*=r$ is similar, we reduce to suppose $*=l$.

I) 1) We have the morphisms of 
$D ^{\mathrm{b}}  (\underset{i}{\boxtimes} {}^{\mathrm{top}} \D ^{(m)}  _{Y _i}) $:
\begin{gather}
\notag
\underset{i}{\overset{\L}{\boxtimes}}  {}^{\mathrm{top}} 
\R f _{i*} ( \cD _{Y _i \leftarrow X _i} ^{(m)} \otimes _{\cD _{X _i} ^{(m)}} ^{\L}
\cE _i )
\underset{\ref{lem-Rfi*2Rf*top}.\ref{lem-Rfi*2Rf*top(a)}}{\longrightarrow} 
\R f _{*}
\underset{i}{\overset{\L}{\boxtimes}} {}^{\mathrm{top}}
 ( \cD _{Y _i \leftarrow X _i} ^{(m)} \otimes _{\cD _{X _i} ^{(m)}} ^{\L}
\cE _i )
\underset{\ref{com-botimestop-iso2}}{\riso} 
\R f _{*}  \left (
\underset{i}{\boxtimes}  ^{\mathrm{top}} 
 \cD _{Y _i \leftarrow X _i} ^{(m)} \otimes _{\underset{i}{\boxtimes}  ^{\mathrm{top}}  \cD _{X _i} ^{(m)}} ^{\L}
\underset{i}{\overset{\L}{\boxtimes}} {}^{\mathrm{top}}  \cE _i 
\right)
\\
\notag
\to 
\R f _{*}  \left (
\underset{i}{\boxtimes}  
 \cD _{Y _i \leftarrow X _i} ^{(m)} \otimes _{\underset{i}{\boxtimes}   \cD _{X _i} ^{(m)}} ^{\L}
\underset{i}{\overset{\L}{\boxtimes}} \cE _i 
\right)
\underset{\ref{boxtimesDleftarrow}}{\riso}
\R f _{*}  \left (
 \cD _{Y  \leftarrow X } ^{(m)} \otimes _{  \cD _{X } ^{(m)}} ^{\L}
\underset{i}{\overset{\L}{\boxtimes}}    \cE _i 
\right)
=
f _+ ^{ (m)}(\underset{i}{\overset{\L}{\boxtimes}}   \cE _i ).
\end{gather}
Since 
$f _+ ^{ (m)} (\underset{i}{\overset{\L}{\boxtimes}}   \cE _i )
\in D ^{\mathrm{b}} _{\mathrm{qc}} (\D ^{(m)} _{Y})$,
this yields the morphism 
of $D ^{\mathrm{b}} _{\mathrm{qc}} (\D ^{(m)} _{Y})$:
\begin{gather}
\label{const-sch-prop-boxtimes-v+}
\underset{i}{\overset{\L}{\boxtimes}}  f _{i+} ^{ (m)} (\cE _i )
=
\D ^{(m)} _{Y }
\otimes _{\underset{i}{\boxtimes} {}^{\mathrm{top}} \D ^{(m)}  _{Y _i}}
\underset{i}{\overset{\L}{\boxtimes}}  {}^{\mathrm{top}} 
\R f _{i*} ( \cD _{Y _i \leftarrow X _i} ^{(m)} \otimes _{\cD _{X _i} ^{(m)}} ^{\L}
\cE _i )
\to
f _+ ^{ (m)}(\underset{i}{\overset{\L}{\boxtimes}}   \cE _i ).
\end{gather}
Since the left square is commutative (see \ref{lem-Rfi*2Rf*top-diag}), to check 
the commutativity of the diagram \ref{sch-prop-boxtimes-v+iso2}, we notice it is enough 
to check the commutativity of the outline. This latter fact is easy.

II) It remains to check the morphism 
\ref{const-sch-prop-boxtimes-v+}
is an isomorphism.
Using
\cite[I.7.1.(i)]{HaRD}, we reduce to the case where 
$\cE _i$ is a quasi-coherent $\D ^{(m)} _{X _i}$-module.
We remark that such $\cE _i$ is a quotient of 
a $\D ^{(m)} _{X _i}$-module of the form 
$\D ^{(m)} _{X _i} \otimes _{\O _{X _i}} \cL _i$, 
where $\cL _i$ is a quasi-coherent $\O _{X _i}$-module
(e.g. take $\cL _i = \cE _i$). 
Since both functors of \ref{sch-prop-boxtimes-v+iso2} are way-out right, 
using \cite[I.7.1.(iv)]{HaRD}
we reduce to the case where 
$\cE _i = \D ^{(m)} _{X _i} \otimes _{\O _{X _i}} \cL _i$.

3) To simplify notation, 
put 
$
\underset{i}{\overset{\L}{\widetilde{\boxtimes}}}:= 
\underset{i}{\overset{\L}{\boxtimes}} {}^{\mathrm{top}}$,
$\cM _i := \cD _{Y _i \leftarrow X _i} ^{(m)}$,
$\cD _i := \D ^{(m)}  _{X _i}$,
$\O _{X _i}:= \O _i$,
$\cD ' _i := \D ^{(m)}  _{Y _i}$,
$ \O ' _i :=\O _{Y _i}$.
Since $\underset{i}{\boxtimes} ^{\mathrm{top}} f _i ^{-1}\D ^{(m)}  _{Y _i}=
f ^{-1} \underset{i}{\boxtimes} ^{\mathrm{top}} \D ^{(m)}  _{Y _i}$, we get the 
the following diagram in the category 
$D ^- (f ^{-1} \underset{i}{\boxtimes} ^{\mathrm{top}} \D ^{(m)}  _{Y _i} , \O _{X} )$:
\begin{equation}
\label{boxtimes-big-diag1}
\xymatrix{
{
\underset{i}{\overset{\L}{\widetilde{\boxtimes}}}  
\left (
\cM _i 
\otimes _{ \cD _i}   ^{\L}
(\cD _i \otimes _{\cO _i} \cL _i )
\right)
} 
\ar[d] ^-{\sim}
\ar[r] ^-{\sim} _-{\ref{com-botimestop-iso2}}
&
{
\underset{i}{\widetilde{\boxtimes}}  
\cM _i 
\otimes ^{\L} _{\underset{i}{\widetilde{\boxtimes}}  \cD _i} 
\underset{i}{\overset{\L}{\widetilde{\boxtimes}}}
(\cD _i \otimes _{\cO _i} \cL _i )
}
\ar[r] ^-{\sim} _-{\ref{com-botimestop-iso1}}
&
{
\underset{i}{\widetilde{\boxtimes}}  
\cM _i 
\otimes ^{\L} _{\underset{i}{\widetilde{\boxtimes}}  \cD _i} 
( \underset{i}{\widetilde{\boxtimes}}  
\cD _i 
\otimes _{\underset{i}{\widetilde{\boxtimes}}  \cO _i} 
\underset{i}{\overset{\L}{\widetilde{\boxtimes}}}  \cL _i )
} 
\ar[d] ^-{\sim}
\\
{
\underset{i}{\overset{\L}{\widetilde{\boxtimes}}}  
(
\cM _i 
\otimes _{\cO _i} \cL _i 
)
}
\ar[rr] ^-{\sim} _-{\ref{com-botimestop-iso1}}
\ar[d] ^-{}
&&
{
\underset{i}{\widetilde{\boxtimes}}  
\cM _i 
\otimes _{\underset{i}{\widetilde{\boxtimes}}  \cO _i} 
\underset{i}{\overset{\L}{\widetilde{\boxtimes}}}  \cL _i 
} 
\ar[d] ^-{}
\\ 
{
\underset{i}{\overset{\L}{\boxtimes}}  
(
\cM _i 
\otimes _{\cO _i} \cL _i 
)
}
\ar[rr] ^-{\sim}  _-{\ref{com-botimes-diag1}}
&&
{
\underset{i}{\boxtimes}  
\cM _i 
\otimes _{\underset{i}{\boxtimes}  \cO _i} 
\underset{i}{\overset{\L}{\boxtimes}}  \cL _i .
} 
}
\end{equation}
Using \ref{com-botimes-diag2},
we get the commutativity of the bottom rectangle. 
The commutativity of the top rectangle is straightforward.
Consider now the following diagram:
\begin{equation}
\label{boxtimes-big-diag2}
\xymatrix{
{
\underset{i}{\widetilde{\boxtimes}}  
\cM _i 
\otimes ^{\L} _{\underset{i}{\widetilde{\boxtimes}}  \cD _i} 
\underset{i}{\overset{\L}{\widetilde{\boxtimes}}}  
(\cD _i \otimes _{\cO _i} \cL _i )
} 
\ar[r] ^-{}
\ar[d] ^-{\sim} _-{\ref{com-botimestop-iso1}}
& 
{
\underset{i}{\widetilde{\boxtimes}}  
\cM _i 
\otimes ^{\L} _{\underset{i}{\widetilde{\boxtimes}}  \cD _i} 
\underset{i}{\overset{\L}{\boxtimes}}
(\cD _i \otimes _{\cO _i} \cL _i )
}
\ar[d] ^-{\sim} _-{\ref{com-botimes-diag1}}
\ar[r] ^-{}
&
{
\underset{i}{\boxtimes}  
\cM _i 
\otimes ^{\L} _{\underset{i}{\boxtimes}  \cD _i} 
\underset{i}{\overset{\L}{\boxtimes}}
(\cD _i \otimes _{\cO _i} \cL _i )
}
\ar[d] ^-{\sim}
\\
{
\underset{i}{\widetilde{\boxtimes}}  
\cM _i 
\otimes ^{\L} _{\underset{i}{\widetilde{\boxtimes}}  \cD _i} 
(\underset{i}{\widetilde{\boxtimes}}  
\cD _i 
\otimes _{\underset{i}{\widetilde{\boxtimes}}  \cO _i} 
\underset{i}{\overset{\L}{\widetilde{\boxtimes}}}  \cL _i )
} 
\ar[r] ^-{}
\ar[d] ^-{\sim}
&
{
\underset{i}{\widetilde{\boxtimes}}  
\cM _i 
\otimes ^{\L} _{\underset{i}{\widetilde{\boxtimes}}  \cD _i} 
(\underset{i}{\boxtimes}  
\cD _i 
\otimes _{\underset{i}{\boxtimes}  \cO _i} 
\underset{i}{\overset{\L}{\boxtimes}}  \cL _i )
}
\ar[r] ^-{}
&
{
\underset{i}{\boxtimes}  
\cM _i 
\otimes ^{\L} _{\underset{i}{\boxtimes}  \cD _i} 
(\underset{i}{\boxtimes}  
\cD _i 
\otimes _{\underset{i}{\boxtimes}  \cO _i} 
\underset{i}{\overset{\L}{\boxtimes}}  \cL _i )
}
\ar[d] ^-{\sim}
\\ 
{
\underset{i}{\widetilde{\boxtimes}}  
\cM _i 
\otimes _{\underset{i}{\widetilde{\boxtimes}}  \cO _i} 
\underset{i}{\overset{\L}{\widetilde{\boxtimes}}}  \cL _i 
} 
\ar[rr] ^-{}
&& 
{
\underset{i}{\boxtimes}  
\cM _i 
\otimes _{\underset{i}{\boxtimes}  \cO _i} 
\underset{i}{\overset{\L}{\boxtimes}}  \cL _i .
}
}
\end{equation}
The left top square is commutative because of that of \ref{com-botimes-diag2}.
The right top square is commutative by functoriality. 
Taking  $\O _{X _i}$-flat resolutions of $\cL _i$, 
we check the commutativity of the bottom rectangle.

Compositing both diagrams \ref{boxtimes-big-diag1} and \ref{boxtimes-big-diag2},
we get the commutative diagram
\begin{equation}
\label{boxtimes-big-diag1+2}
\xymatrix{
{
\underset{i}{\overset{\L}{\widetilde{\boxtimes}}}  
\left (
\cM _i 
\otimes ^{\L} _{ \cD _i}   
(\cD _i \otimes _{\cO _i} \cL _i )
\right)
} 
\ar[d] ^-{\sim}
\ar[r] ^-{} 
&
{\underset{i}{\boxtimes}  
\cM _i 
\otimes ^{\L} _{\underset{i}{\boxtimes}  \cD _i} 
\underset{i}{\overset{\L}{\boxtimes}}
(\cD _i \otimes _{\cO _i} \cL _i )}
\ar[d] ^-{\sim}
\\
{\underset{i}{\overset{\L}{\widetilde{\boxtimes}}}  
(
\cM _i 
\otimes _{\cO _i} \cL _i )}
\ar[d] ^-{}
&
{\underset{i}{\boxtimes}  
\cM _i 
\otimes ^{\L} _{\underset{i}{\boxtimes}  \cD _i} 
( \underset{i}{\boxtimes}  
\cD _i 
\otimes _{\underset{i}{\boxtimes}  \cO _i} 
\underset{i}{\overset{\L}{\boxtimes}}  \cL _i )}
\ar[d] ^-{\sim}
\\ 
{
\underset{i}{\overset{\L}{\boxtimes}}  
(
\cM _i 
\otimes _{\cO _i} \cL _i 
)
}
\ar[r] ^-{\sim}  _-{\ref{com-botimes-diag1}}
&
{\underset{i}{\boxtimes}  
\cM _i 
\otimes _{\underset{i}{\boxtimes}  \cO _i} 
\underset{i}{\overset{\L}{\boxtimes}}  \cL _i ,}
}
\end{equation}
where the top arrow is the composite of the left top horizontal arrow of 
\ref{boxtimes-big-diag1} with top horizontal arrows of \ref{boxtimes-big-diag2}.
We get  from the commutativity of \ref{boxtimes-big-diag1+2}
that of the right rectangle of the diagram \ref{boxtimes-big-diag3} below:
\small
\begin{equation}
\label{boxtimes-big-diag3}
\xymatrix@ C=0,6cm{
{\underset{i}{\overset{\L}{\boxtimes}}  f _{i+} ^{ (m)} (\cE _i )}
\ar@{=}[d] ^-{}
\ar[rr] ^-{}&&
{f _+ ^{ (m)} (\underset{i}{\overset{\L}{\boxtimes}}   \cE _i )}
\\
{ \D' 
\underset{\underset{i}{\widetilde{\boxtimes}}  \D ' _i}{\otimes} 
\underset{i}{\overset{\L}{\widetilde{\boxtimes}}}
 \R f _{i* } 
 ( \cM _i  
 \underset{ \cD _i}{\overset{\L}{\otimes}}
  (\cD _i \underset{\cO _i}{\otimes} \cL _i ))}
\ar[r] ^-{\ref{lem-Rfi*2Rf*top}.\ref{lem-Rfi*2Rf*top1}}
\ar[d] ^-{\sim} 
&
{ \D' 
\underset{\underset{i}{\widetilde{\boxtimes}}  \D ' _i}{\otimes} 
 \R f _*  \underset{i}{\overset{\L}{\widetilde{\boxtimes}}}  
\left (
\cM _i 
\underset{ \cD _i}{\overset{\L}{\otimes}}
(\cD _i 
\underset{\cO _i}{\otimes}
 \cL _i )
\right)} 
\ar[d] ^-{\sim}
\ar[r] ^-{} 
&
{\R f _* (\underset{i}{\boxtimes}  
\cM _i 
\underset{\underset{i}{\boxtimes}  \cD _i} {\otimes ^{\L}} 
\underset{i}{\overset{\L}{\boxtimes}}
(\cD _i \underset{\cO _i}{\otimes} \cL _i ))}
\ar[d] ^-{\sim}
\ar[u] ^-{\sim} _-{\ref{boxtimesDleftarrow}}
\\
{ \D' 
\underset{\underset{i}{\widetilde{\boxtimes}}  \D ' _i}{\otimes} 
 \underset{i}{\overset{\L}{\widetilde{\boxtimes}}}  \R f _{i* } (
\cM _i \underset{\cO _i}{\otimes} \cL _i )}
\ar[r] ^-{\ref{lem-Rfi*2Rf*top}.\ref{lem-Rfi*2Rf*top1}}
&										
{\D' 
\underset{\underset{i}{\widetilde{\boxtimes}}  \D ' _i}{\otimes}   
\R f _* \underset{i}{\overset{\L}{\widetilde{\boxtimes}}}  
(
\cM _i 
\underset{\cO _i}{\otimes} \cL _i )}
\ar[d] ^-{\sim}
&
{\R f _* (\underset{i}{\boxtimes}  
\cM _i 
\underset{\underset{i}{\boxtimes}  \cD _i} {\otimes ^{\L}}  
(\underset{i}{\boxtimes}  
\cD _i 
\underset{\underset{i}{\boxtimes}  \cO _i}{\otimes} 
\underset{i}{\overset{\L}{\boxtimes}}  \cL _i ))}
\ar[d] ^-{\sim}
\\ 
{ 
\underset{i}{\overset{\L}{\boxtimes}} \R f _{i* } (
\cM _i \underset{\cO _i}{\otimes} \cL _i )}
\ar[r] ^-{\ref{lem-Rfi*2Rf*top}.\ref{lem-Rfi*2Rf*top2}}
\ar[u] ^-{\sim}
&
{ \R f _* 
\underset{i}{\overset{\L}{\boxtimes}}  
(
\cM _i 
\underset{\cO _i}{\otimes} \cL _i 
)
}
\ar[r] ^-{\sim} 
&
{\R f _* (\underset{i}{\boxtimes}  
\cM _i 
\underset{\underset{i}{\boxtimes}  \cO _i}{\otimes} 
\underset{i}{\overset{\L}{\boxtimes}}  \cL _i ).}
}
\end{equation}
 \normalsize
The commutativity of the top rectangle is by construction of the top arrow (see \ref{const-sch-prop-boxtimes-v+}).
The commutativity of the top square is checked by functoriality.
Using the commutativity of the diagram \ref{lem-Rfi*2Rf*top-diag}, we obtain the commutativity of the bottom square of \ref{boxtimes-big-diag3}.
Hence, the diagram  \ref{boxtimes-big-diag3} is commutative. 
Following Theorem \ref{sch-prop-boxtimes-v*}, 
the left bottom morphism is an isomorphism.
Hence, using the commutativity of the diagram  \ref{boxtimes-big-diag3}, 
this yields that the top morphism is an isomorphism.
\end{proof}

\subsection{Application : base change in the projection case}
\label{subsec4.3}
We keep notation \ref{subsec4.2} and we 
suppose $n = 2$ and $f _2 $ is the identity.
\begin{prop}
\label{theo-iso-chgtbase2-pre} 
For any $\cE _1\in D ^{\mathrm{b}} _{\mathrm{qc}} (\D ^{(m)} _{X _1})$, we have the canonical isomorphism
$pr  _1 ^{\prime ! (m)}  \circ f ^{ (m)} _{1,+}( \E _1)
\riso 
f ^{ (m)} _{+}  \circ pr _1 ^{ ! (m)} (\E _1)$
of 
$D ^{\mathrm{b}} _{\mathrm{qc}} (\D ^{(m)} _{Y})$
making commutative the diagram 
\begin{equation}
\label{theo-iso-chgtbase2-pre-iso1}
\xymatrix{
{pr  _1 ^{\prime*}  \circ \R f  _{1,*}( \E _1)} 
\ar[r] ^-{\sim}
\ar[d] ^-{}
& 
{\R f _{*}  \circ pr _1 ^{*} (\E _1)} 
\ar[d] ^-{}
\\ 
{pr  _1 ^{\prime *}  \circ f ^{ (m)} _{1,+}( \E _1)} 
\ar[r] ^-{\sim}
& 
{f ^{ (m)} _{+}  \circ pr _1 ^{ *} (\E _1),} 
}
\end{equation}
where the top isomorphism is the usual base change isomorphism.
\end{prop}

\begin{proof}
This is an easy consequence of Theorem \ref{sch-prop-boxtimes-v+}.
Indeed, with notation \ref{n=2,f2=id-diag-parag},
recall
$pr  _1 ^{ (m)!} = pr  _1 ^{*} [\dim T]$, 
$pr  _1 ^{ \prime (m)!} = pr  _1 ^{\prime*} [\dim T]$.
Next, consider the following diagram 
\begin{equation}
\label{theo-iso-chgtbase2-pre-iso2}
\xymatrix{
{pr  _1 ^{\prime*}  \circ \R f  _{1,*}( \E _1)} 
\ar[r] ^-{\sim}
\ar@{=}[d] ^-{}
& 
{pr ^{\prime *} _1 \R f _{1*}( \E _1)
\otimes ^\L _{\O _Y}
pr ^{\prime *} _2 \O _T}
\ar[r] ^-{}
\ar@{=}[d] ^-{}
&
{\R f _{*} ( 
pr ^{*} _1  (\E _1)
\otimes ^\L _{\O _X}
pr ^{*} _2 \O _T )} 
\ar[r] ^-{\sim}
\ar@{=}[d] ^-{}
& 
{\R f _{*}  \circ pr _1 ^{*} (\E _1)} 
\ar@{=}[d] ^-{}
\\ 
{pr  _1 ^{\prime *}  \circ \R f  _{1,*}( \E _1)} 
\ar[r] ^-{\sim}
\ar[d] ^-{}
& 
{\R f  _{1,*}( \E _1)\overset{\L}{\boxtimes} \O _T } 
\ar[r] ^-{\ref{sch-prop-boxtimes-v*iso}} _-{\sim}
\ar[d] ^-{}
&
{\R f _{*}   (\E _1\overset{\L}{\boxtimes} \O _T )} 
\ar[r] ^-{\sim}
\ar[d] ^-{}
& 
{\R f _{*}  \circ pr _1 ^{*} (\E _1)} 
\ar[d] ^-{}
\\ 
{pr  _1 ^{\prime*}  \circ f ^{ (m)} _{1,+}( \E _1)} 
\ar[r] ^-{\sim}
& 
{f ^{ (m)} _{1,+}( \E _1) 
\overset{\L}{\boxtimes} \O _T 
}
\ar[r] ^-{\sim} _-{\ref{sch-prop-boxtimes-v+}}
& 
{f ^{ (m)} _{+} (\E _1
\overset{\L}{\boxtimes} \O _T )} 
\ar[r] ^-{\sim}
& 
{f ^{ (m)} _{+} ( pr _1 ^{ *} (\E _1))} 
}
\end{equation}
Following \ref{sch-prop-boxtimes-v+}, the middle square (of the bottom) is commutative.
The left and right squares are commutative by functoriality. 
By compositing the bottom isomorphisms, 
we get 
$pr  _1 ^{*}  \circ f ^{ (m)} _{1,+}( \E _1)
\riso
f ^{ (m)} _{+}  \circ pr _1 ^{ *} (\E _1)$.
It remains to check that the composition of the top isomorphisms is the base change isomorphism.

Consider the commutative diagram
\begin{equation}
\label{sch-prop-boxtimes-v+isoproofbis2}
\xymatrix @C=0,45cm {
{pr  _1 ^{\prime*}  \circ \R f  _{1,*}( \E _1)} 
\ar[r] ^-{\sim}
\ar[d] ^-{\sim}
&
{pr ^{\prime *} _1 \R f _{1*}( \E _1)
\otimes ^\L _{\O _Y}
pr ^{\prime *} _2 \O _T}
\ar[r] ^-{\sim}
\ar[d] ^-{\sim}
&
{\R f _{*} ( 
pr ^{*} _1  (\E _1)
\otimes ^\L _{\O _X}
pr ^{*} _2 \O _T )} 
\ar[r] ^-{\sim}
&
{\R f _{*} pr ^{*} _1  (\E _1)} 
\ar@{=}[d] ^-{}
\\ 
{\R f _{*} pr ^{*} _1  (\E _1)} 
\ar[r] ^-{\sim}
&
{\R f _{*} pr ^{*} _1  (\E _1)
\otimes ^\L _{\O _Y}
pr ^{\prime *} _2 \O _T}
\ar[r] ^--{\ref{f_*/otimes-iso-projform}}
&
{\R f _{*} ( 
pr ^{*} _1  (\E _1)
\otimes ^\L _{\O _X}
\L f ^{*} pr ^{\prime *} _2 \O _T)}
\ar[u] ^-{\sim}
\ar[r] ^-{\sim}
&
{\R f _{*} pr ^{*} _1  (\E _1),} 
}
\end{equation}
where the middle square is commutative (this is the outline of the diagram \ref{sch-prop-boxtimes-v+isoproofbis}),
the top arrows are the same than that of \ref{theo-iso-chgtbase2-pre-iso2},
where the right and left squares are commutative by functoriality.
Using \ref{f_*/otimes-iso-projformter}, we get that the composition of the bottom morphisms is the identity.
Hence we are done.
\end{proof}

\begin{ntn}
\label{ntn-psharp}
Let $g \colon Z \to T$ be a smooth morphism of $S$-schemes.
Following \cite[III.2]{HaRD},
we denote by 
$g ^\sharp \colon D ( \O _T) \to D ( \O _Z)$
the functor defined by setting
$g ^\sharp (\cM ) : = g ^* ( \cM) \otimes _{\O _T} \omega _{Z/T} [ d _{Z/T}]$.
We remark that 
if 
$\cM \in D ^{\mathrm{b}} _{\mathrm{qc}} ({} ^{r}\D ^{(m)} _{T})$ 
then 
$g ^\sharp (\cM ) \riso 
g ^{ ! (m)} (\cM)$.
\end{ntn}

\begin{prop}
\label{theo-iso-chgtbase2-pre2}
We keep notation \ref{ntn-psharp}.
\begin{enumerate}
\item For any $\cM _1\in D ^{\mathrm{b}} _{\mathrm{qc}} (\O  _{X _1})$, 
we have the  isomorphism
\begin{equation}
\label{theo-iso-chgtbase2-pre-isopre}
pr  _1 ^{\prime \sharp}  \circ \R f  _{1,*}( \cM _1) \riso \R f _{*}  \circ pr _1 ^{\sharp} (\cM _1)
\end{equation}
of 
$D ^{\mathrm{b}} _{\mathrm{qc}} (\O _{Y})$
canonically induced by the usual base change isomorphism.

\item For any $\cM _1\in D ^{\mathrm{b}} _{\mathrm{qc}} ({} ^r \D ^{(m)} _{X _1})$, 
we have the  isomorphism
 the canonical
$pr  _1 ^{\prime ! (m)}  \circ f ^{ (m)} _{1,+}( \cM _1)
\riso 
f ^{ (m)} _{+}  \circ pr _1 ^{ ! (m)} (\cM _1)$
of 
$D ^{\mathrm{b}} _{\mathrm{qc}} ({} ^r \D ^{(m)} _{Y})$
making commutative the diagram 
\begin{equation}
\label{theo-iso-chgtbase2-pre-iso}
\xymatrix{
{pr  _1 ^{\prime \sharp}  \circ \R f  _{1,*}( \cM _1)} 
\ar[r] ^-{\sim} _-{\ref{theo-iso-chgtbase2-pre-isopre}}
\ar[d] ^-{}
& 
{\R f _{*}  \circ pr _1 ^{\sharp} (\cM _1)} 
\ar[d] ^-{}
\\ 
{pr  _1 ^{\prime  ! (m)}  \circ f ^{ (m)} _{1,+}( \cM _1)} 
\ar[r] ^-{\sim}
& 
{f ^{ (m)} _{+}  \circ pr _1 ^{ ! (m)} (\cM _1).} 
}
\end{equation}

\end{enumerate}

\end{prop}

\begin{proof}
Again, this is an easy consequence of Theorem \ref{sch-prop-boxtimes-v+}.
Indeed, with notation \ref{n=2,f2=id-diag-parag}, 
since 
$pr ^{\prime *} _2 \omega _T 
\riso \omega _{Y/Y_1}$
and 
$pr ^{ *} _2 \omega _T 
\riso \omega _{X/X_1}$,
we get the commutative diagram
\begin{equation}
\label{theo-iso-chgtbase2-pre-iso2bis}
\xymatrix @ C =0,3cm {
{pr  _1 ^{\prime \sharp}  \circ \R f  _{1,*}( \cM _1) } 
\ar[r] ^-{\sim} 
\ar@{=}[d] ^-{}
& 
{pr ^{\prime *} _1 \R f _{1*}( \cM _1)
\otimes ^\L _{\O _Y}
pr ^{\prime *} _2 \omega _T [ d _{T}]}
\ar[r] ^-{\sim} _-{\ref{sch-prop-boxtimes-v+isoproof}}
\ar@{=}[d] ^-{}
&
{\R f _{*} ( 
pr ^{*} _1  (\cM _1)
\otimes ^\L _{\O _X}
pr ^{*} _2 \omega _T ) [ d _{T}]} 
\ar[r] ^-{\sim}
\ar@{=}[d] ^-{}
& 
{\R f _{*}  \circ pr _1 ^{\sharp} (\cM _1)} 
\ar@{=}[d] ^-{}
\\ 
{pr  _1 ^{\prime \sharp}  \circ \R f  _{1,*}( \cM _1)} 
\ar[r] ^-{\sim}
\ar[d] ^-{}
& 
{\R f  _{1,*}( \cM _1)\overset{\L}{\boxtimes} \omega _T [ d _{T}]} 
\ar[r] ^-{\ref{sch-prop-boxtimes-v*iso}} _-{\sim}
\ar[d] ^-{}
&
{\R f _{*}   (\cM _1\overset{\L}{\boxtimes} \omega _T )[ d _{T}]} 
\ar[r] ^-{\sim}
\ar[d] ^-{}
& 
{\R f _{*}  \circ pr _1 ^{\sharp} (\cM _1)} 
\ar[d] ^-{}
\\ 
{pr  _1 ^{\prime ! (m)}  \circ f ^{ (m)} _{1,+}( \cM _1)} 
\ar[r] ^-{\sim}
& 
{f ^{ (m)} _{1,+}( \cM _1) 
\overset{\L}{\boxtimes} \omega _T 
[ d _{T}]}
\ar[r] ^-{\sim} _-{\ref{sch-prop-boxtimes-v+}}
& 
{f ^{ (m)} _{+} (\cM _1
\overset{\L}{\boxtimes} \omega _T )
[ d _{T}]} 
\ar[r] ^-{\sim}
& 
{f ^{ (m)} _{+} ( pr _1 ^{!(m)} (\cM _1)).} 
}
\end{equation}
The commutative diagram \ref{theo-iso-chgtbase2-pre-iso} corresponds to the outline of 
\ref{theo-iso-chgtbase2-pre-iso2bis}.
\end{proof}

\subsection{Application : relative duality isomorphism and adjunction for projective morphisms}
Following Virrion, we have a 
relative duality isomorphism and the adjoint paire $(f _+, f ^!)$ for a proper morphism $f$.
The key point is to construct a trace map which is compatible
with Grothendieck's one for coherent $\O$-modules (i.e. we have a commutative diagram of the form \ref{theo-iso-chgtbase2-pre-iso3ex}). 
In general, this key point is highly technical and corresponds to the hard part of the proof of a relative duality isomorphism (see \cite{Vir04}).
In this subsection, we show how  
this is much easier to construct such a trace map in the case of projective morphisms 
by using base change in the projection case and by using the case of a closed immersion (see \ref{rel-dual-isom-imm}).

We keep notation \ref{subsec4.2} and we 
suppose $n = 2$, $f _2 $ is the identity, 
$X _1 = \bbP ^{d} _{Y _1}$,
$f _1 \colon \bbP ^{d} _{Y _1} \to Y _1$ is the canonical projection.
We set $T:= X _2=Y_2$.
We suppose there exists an integer $n \geq 0$ such that 
$S = \Spec ( \cV / \pi ^{n+1})$. 
In that case, for instance since $Y _1/S$ is smooth then 
$D ^{\mathrm{b}} _{\mathrm{qc}} (\O  _{Y _1})
= 
D ^{\mathrm{b}} _{\mathrm{qc}, \mathrm{tdf}} (\O  _{Y _1})$.
\begin{lem}
With notation \ref{ntn-Tf} and \ref{ntn-psharp}, 
for any $\cN _1\in D ^{\mathrm{b}} _{\mathrm{qc}} (\O  _{Y _1})$, 
we have the commutative diagram:
\begin{equation}
\label{diag-comm-Trf1-Trf}
\xymatrix{
{pr  _1 ^{\prime \sharp}  \circ \R f  _{1,*} \circ f _1 ^{\sharp} ( \cN _1)} 
\ar[r] ^-{\sim} _-{\ref{theo-iso-chgtbase2-pre-isopre}}
\ar[d] ^-{\mathrm{Tr} _{f _1}} _-{\sim}
& 
{\R f  _{*} \circ  pr  _1 ^{\sharp}  \circ f _1 ^{\sharp} ( \cN _1)} 
\ar[r] ^-{\sim}
& 
{\R f  _{*} \circ  f  ^{\sharp}   \circ pr  _1 ^{\prime \sharp} ( \cN _1)} 
\ar[d] ^-{\mathrm{Tr} _{f}} _-{\sim}
\\ 
{pr  _1 ^{\prime \sharp}  ( \cN _1)} 
\ar@{=}[rr] ^-{}
&& 
{ pr  _1 ^{\prime \sharp} ( \cN _1),} 
}
\end{equation}
where $\mathrm{Tr} _{f}$ and 
$\mathrm{Tr} _{f _1}$ are the trace map isomorphisms
(see \cite[III.4.3]{HaRD}).
\end{lem}

\begin{proof}
Following the commutativity of the diagram 
\ref{sch-prop-boxtimes-v+isoproofbis}, since 
$\omega _{Y/Y _1}
\riso
pr ^{\prime *} _2 \omega _T $
and 
$\omega _{X/X _1}
\riso
pr ^{ *} _2 \omega _T $
the isomorphism 
$pr  _1 ^{\prime \sharp}  \circ \R f  _{1,*} \circ f _1 ^{\sharp} ( \cN _1)
\underset{\ref{theo-iso-chgtbase2-pre-isopre}}{\riso}
\R f  _{*} \circ  pr  _1 ^{\sharp}  \circ f _1 ^{\sharp} ( \cN _1)$
is given 
by the composition of the left verticales arrows of the diagram
\begin{equation}
\notag
\xymatrix{
{pr  _1 ^{\prime \sharp}  \circ \R f  _{1,*} \circ f _1 ^{\sharp} ( \cN _1)}
\ar[d] ^-{\sim}
\ar[r] ^-{\mathrm{Tr} _{f _1} }
&
{pr  _1 ^{\prime \sharp}  ( \cN _1)}
\ar[d] ^-{\sim}
\\
{pr ^{\prime *} _1 \R f _{1*} f _1 ^{\sharp} ( \cN _1)
\otimes ^\L _{\O _Y}
pr ^{\prime *} _2 \omega _T [ d _{T}]}
\ar[r] ^-{\mathrm{Tr} _{f _1}\otimes id }
\ar[d] ^-{\sim}
 & 
{pr ^{\prime *} _1( \cN _1)
\otimes ^\L _{\O _Y}
pr ^{\prime *} _2 \omega _T [ d _{T}]}
\ar@{=}[rd] ^-{}
 \\ 
{\R f _{*}  pr ^{*} _1  f _1 ^{\sharp} ( \cN _1)
\otimes ^\L _{\O _Y}
pr ^{\prime *} _2 \omega _T [ d _{T}]}
\ar[r] ^-{\sim}
\ar[d] ^-{\sim}
 & 
{\R f _{*}  f ^{\sharp} pr ^{\prime *} _1   ( \cN _1)
\otimes ^\L _{\O _Y}
pr ^{\prime *} _2 \omega _T [ d _{T}]}
\ar[u] ^-{\mathrm{Tr} _{f}\otimes id }
\ar[r] ^-{\mathrm{Tr} _{f}\otimes id }
\ar[d] ^-{\sim}
 & 
{pr ^{\prime *} _1( \cN _1)
\otimes ^\L _{\O _Y}
pr ^{\prime *} _2 \omega _T [ d _{T}]}
 \\ 
{\R f _{*}  ( pr ^{*} _1  f _1 ^{\sharp} ( \cN _1)
\otimes ^\L _{\O _Y}
f ^* pr ^{\prime *} _2 \omega _T ) [ d _{T}]}
\ar[r] ^-{\sim}
\ar[d] ^-{\sim}
 & 
{\R f _{*}  (f ^{\sharp} pr ^{\prime *} _1   ( \cN _1)
\otimes ^\L _{\O _Y}
f ^*  pr ^{\prime *} _2 \omega _T ) [ d _{T}]}
\ar[d] ^-{\sim}
\ar[r] ^-{\sim}
 & 
{\R f _{*}  f ^{\sharp} ( pr ^{\prime *} _1   ( \cN _1)
\otimes ^\L _{\O _Y}
 pr ^{\prime *} _2 \omega _T ) [ d _{T}]}
 \ar[u] ^-{\mathrm{Tr} _{f}}
 \ar[dd] ^-{\sim}
 \\ 
{\R f _{*}  ( pr ^{*} _1  f _1 ^{\sharp} ( \cN _1)
\otimes ^\L _{\O _Y}
 pr ^{*} _2 \omega _T ) [ d _{T}]}
\ar[r] ^-{\sim}
\ar[d] ^-{\sim}
 & 
{\R f _{*}  (f ^{\sharp} pr ^{\prime *} _1   ( \cN _1)
\otimes ^\L _{\O _Y}
 pr ^{*} _2 \omega _T ) [ d _{T}]}
 \\
{ \R f  _{*} \circ  pr  _1 ^{\sharp}  \circ f _1 ^{\sharp} ( \cN _1)}
\ar[rr] ^-{\sim}
&&
{\R f  _{*} \circ  f  ^{\sharp}   \circ pr  _1 ^{\prime \sharp} ( \cN _1)}  }
\end{equation}
Following \cite[III.10.5.Tra 4)]{HaRD}, the square of the second line is commutative. 
Following \cite[III.4.4]{HaRD}, the right square of the third line 
is commutative. The commutativity of the other squares, of the triangle or trapeze  are obvious. 
Hence, we are done. 
\end{proof}

\begin{prop}
\label{proptheo-iso-chgtbase2-pre-iso3}
Let $\cN _1\in D ^{\mathrm{b}} _{\mathrm{qc}} ({} ^r \D ^{(m)} _{Y _1})$.
Suppose we have 
 the canonical morphism 
$\mathrm{Tr} _{+,f _1}
\colon 
f ^{ (m)} _{1,+}\circ f _1 ^{!(m)} ( \cN _1)
\to 
\cN _1$
of 
$D ^{\mathrm{b}} _{\mathrm{qc}} ({} ^r \D ^{(m)} _{Y_1})$
making commutative the diagram 
\begin{equation}
\label{theo-iso-chgtbase2-pre-iso3pre}
\xymatrix{
{ \R f  _{1,*}\circ f _1 ^{\sharp} ( \cN _1)} 
\ar[d] ^-{}
\ar[r] ^-{\mathrm{Tr} _{f _1}}
&
{ \cN _1}
\\
{ f ^{ (m)} _{1,+}\circ f _1 ^{!(m)} ( \cN _1).} 
\ar[ur] _-{\mathrm{Tr} _{+,f _1}}
}
\end{equation}

Then, there exists a  canonical morphism 
$\mathrm{Tr} _{+,f }
\colon 
f ^{ (m)} _{+}  \circ f  ^{! (m)} \circ pr _1 ^{\prime  ! (m)}  ( \cN _1)
\to 
pr _1 ^{\prime  ! (m)}  ( \cN _1)$
of 
$D ^{\mathrm{b}} _{\mathrm{qc}} ({} ^r \D ^{(m)} _{Y})$
making commutative the diagram 
\begin{equation}
\label{theo-iso-chgtbase2-pre-iso3}
\xymatrix{
{\R f  _{*} \circ  f  ^{\sharp}   \circ pr  _1 ^{\prime \sharp} ( \cN _1)} 
\ar[r] ^-{\mathrm{Tr} _{f}}
\ar[d] ^-{}
&
{pr  _1 ^{\prime \sharp} ( \cN _1)} 
\ar[d] ^-{\sim}
\\
{f ^{ (m)} _{+}  \circ f  ^{! (m)} \circ pr _1 ^{\prime  ! (m)}  ( \cN _1)} 
\ar[r] ^-{\mathrm{Tr} _{+,f }}
&
{pr  _1 ^{\prime !(m)} ( \cN _1).} 
}
\end{equation}
\end{prop}

\begin{proof}
By definition, we define the morphism
$\mathrm{Tr} _{+,f }
\colon 
f ^{ (m)} _{+}  \circ f  ^{! (m)} \circ pr _1 ^{\prime  ! (m)}  ( \cN _1)
\to 
pr _1 ^{\prime  ! (m)}  ( \cN _1)$
to be the one making commutative the bottom of the diagram:
\begin{equation}
\label{theo-iso-chgtbase2-pre-iso3-proof1}
\xymatrix{
{pr  _1 ^{\prime \sharp}  \circ \R f  _{1,*}\circ f _1 ^{\sharp} ( \cN _1)} 
\ar[r] ^-{\sim} _-{\ref{theo-iso-chgtbase2-pre-isopre}}
\ar[d] ^-{}
\ar@/^5ex/[rrr] ^-{\mathrm{Tr} _{f _1}}
& 
{\R f _{*}  \circ pr _1 ^{\sharp} \circ f _1 ^{\sharp} ( \cN _1)} 
\ar[d] ^-{}
\ar[r] ^-{\sim}
& 
{\R f  _{*} \circ  f  ^{\sharp}   \circ pr  _1 ^{\prime \sharp} ( \cN _1)} 
\ar[r] ^-{\mathrm{Tr} _{f}}
\ar[d] ^-{}
&
{pr  _1 ^{\prime \sharp} ( \cN _1)} 
\ar[d] ^-{\sim}
\\
{pr  _1 ^{\prime  ! (m)}  \circ f ^{ (m)} _{1,+}\circ f _1 ^{! (m)} ( \cN _1)} 
\ar[r] ^-{\sim}
\ar@/_5ex/[rrr] ^-{\mathrm{Tr} _{+,f _1}}
& 
{f ^{ (m)} _{+}  \circ pr _1 ^{ ! (m)} \circ f _1 ^{! (m)} ( \cN _1)} 
\ar[r] ^-{\sim}
& 
{f ^{ (m)} _{+}  \circ f  ^{! (m)} \circ pr _1 ^{\prime  ! (m)}  ( \cN _1)} 
\ar@{.>}[r] ^-{\mathrm{Tr} _{+,f }}
&
{pr  _1 ^{\prime !(m)} ( \cN _1).} 
}\end{equation}
Following \ref{diag-comm-Trf1-Trf}, 
the top of the diagram \ref{theo-iso-chgtbase2-pre-iso3-proof1} is commutative.
From the commutativity of \ref{theo-iso-chgtbase2-pre-iso3pre},
we get the commutativity of the outline of \ref{theo-iso-chgtbase2-pre-iso3-proof1}.
From the commutative diagram \ref{theo-iso-chgtbase2-pre-iso}, we get the commutativity of the left square of \ref{theo-iso-chgtbase2-pre-iso3-proof1}. 
The commutativity of the middle square of \ref{theo-iso-chgtbase2-pre-iso3-proof1} is easy. 
This yields that the right square of \ref{theo-iso-chgtbase2-pre-iso3-proof1} is indeed commutative.
\end{proof}

\begin{empt}
Suppose 
$Y _1 =S$, $X _1 = \bbP ^{d} _S$,
$f _1 \colon \bbP ^{d} _S \to S$ is the canonical projection
and 
$\cN _1= \O _S \in D ^{\mathrm{b}} _{\mathrm{qc}} ({} ^r \D ^{(m)} _{Y_1/S})
=
D ^{\mathrm{b}} _{\mathrm{qc}} (\O _S)$.
We have 
$f _1 ^{!(m)} ( \cO _S)=f _1 ^\sharp (\O _S) =  \omega _{\bbP ^{d} _S/S}  [d]$
and 
the trace map 
$Tr _{f _1}
\colon 
\R f _* ( \omega _{\bbP ^{d} _S/S} ) [d]
\to \O _S$ 
is an isomorphism of 
$D ^{\mathrm{b}} _{\mathrm{qc}} (\O _S)$.
Since the canonical morphism 
$\R f _* ( \omega _{\bbP ^{d} _S/S} ) [d]
\to 
f ^{ (m)} _{1,+}
( \omega _{\bbP ^{d} _S/S} ) [d]$ is an isomorphism after applying the trunctation
functor 
$\tau _{\geq 0}$, 
we get the morphism
$\mathrm{Tr} _{+,f _1}
\colon 
f ^{ (m)} _{1,+}
( \omega _{\bbP ^{d} _S/S} ) [d]
\to 
 \O _S$ making commutative the diagram
\begin{equation}
\label{theo-iso-chgtbase2-ex}
\xymatrix{
{ \R f  _{1,*}( \omega _{\bbP ^{d} _S/S} ) [d]} 
\ar[d] ^-{}
\ar[r] ^-{\mathrm{Tr} _{f _1}}
&
{ \cO _S}
\\
{ f ^{ (m)} _{1,+} ( \omega _{\bbP ^{d} _S/S} ) [d].} 
\ar[ur] _-{\mathrm{Tr} _{+,f _1}}
}
\end{equation}
Hence, following Proposition \ref{proptheo-iso-chgtbase2-pre-iso3},
there exists a  canonical morphism 
$\mathrm{Tr} _{+,f }
\colon 
f ^{ (m)} _{+}  ( \omega _{\bbP ^{d} _T/S} ) [d]
\to 
( \omega _{T/S} )$
of 
$D ^{\mathrm{b}} _{\mathrm{qc}} ({} ^r \D ^{(m)} _{T/S})$
making commutative the diagram 
\begin{equation}
\label{theo-iso-chgtbase2-pre-iso3ex}
\xymatrix{
{\R f  _{*} \circ  ( \omega _{\bbP ^{d} _T/S} ) [d]} 
\ar[r] ^-{\mathrm{Tr} _{f}}
\ar[d] ^-{}
&
{\omega _{T/S}.} 
\\
{f ^{ (m)} _{+}  ( \omega _{\bbP ^{d} _T/S} ) [d]} 
\ar[ur] ^-{\mathrm{Tr} _{+,f }}
}
\end{equation}

\end{empt}

\begin{thm}
\label{rel-dual-isom-proj}
Let $f\colon X \to Y$ be a morphism of smooth $S$-schemes which is the composition of a closed immersion of the form
$X \hookrightarrow \bbP ^d _Y$ and of the projection $\bbP ^d _Y \to Y$.

\begin{enumerate}
\item Let 
$\cE \in D ^\mathrm{b}  _{\mathrm{coh}}
({} ^l \D _X ^{(m)} )$.
We have the isomorphism of $D ^\mathrm{b}  _{\mathrm{coh}}
({} ^l \D _Y ^{(m)} )$:
\begin{equation}
\label{rel-dual-isom-proj1}
\DD ^{(m)} \circ f _+ (\cE)
\riso 
 f _+ \circ \DD ^{(m)} (\cE).
\end{equation}

\item Let $\cE \in D ^\mathrm{b}  _{\mathrm{coh}}
({} ^l \D _X ^{(m)} )$,
and
$\cF \in D ^\mathrm{b}  _{\mathrm{coh}}
({} ^l \D _Y ^{(m)} )$.
We have 
the isomorphisms
\begin{gather}
\label{cor-adj-formul-proj-bij1}
\R \mathcal{H} om _{\D _Y ^{(m)}}
( f _{+} ( \E ) , \cF) 
\riso 
\R f _* 
\R \mathcal{H} om _{\D _X ^{(m)}}
( \E  ,f ^!   ( \cF)),
\\
\label{cor-adj-formul-proj-bij2}
\R \mathrm{Hom}  _{\D _Y ^{(m)}}
( f _{+} ( \E ) , \cF) 
\riso 
\R \mathrm{Hom}  _{\D _X ^{(m)}}
( \E  ,f ^!   ( \cF)). 
\end{gather}
\end{enumerate}

\end{thm}

\begin{proof}
1) Following \ref{rel-dual-isom-imm}, 
the case of a closed immersion is easily checked.
Hence, we reduce to the case where $f$ is the projection $\bbP ^d _Y \to Y$.
Using \ref{theo-iso-chgtbase2-pre-iso3ex}, to check such an isomorphism,
we can copy Virrion's proof (more precisely :
a) the construction is given in \cite[IV.1.3]{Vir04},
b) for induced modules, using Grothedieck's duality isomorphism for coherent 
$\O$-modules, we construct in another way such an isomorphism : see \cite[IV.2.2.4]{Vir04},
c) the equality between both constructions is a consequence of 
the commutativity of \ref{theo-iso-chgtbase2-pre-iso3ex}: see \cite[IV.2.2.5]{Vir04}).

2) By copying word by word the proof of \ref{cor-adj-formul}, the second part of the theorem is a consequence of 
\ref{rel-dual-isom-proj1}. 
\end{proof}

\subsection{Going to formal $\S$-schemes}
Let $\fP$ and $\fQ$ be two  smooth formal schemes over $\S $,
$p _1 \colon \fP \times _\S \fQ \to \fP$, 
$p_2 \colon \fP \times _\S \fQ \to \fQ$
be the canonical projections.

\begin{empt}
Using the tensor product defined in \ref{def-otimes-coh1qc}, 
we get the bifunctor
\begin{equation}
\label{dfnboxtimes}
 \smash{\widehat{\boxtimes}}
^\L _{\O _{\S }}
\colon 
\smash{\underrightarrow{LD}}  ^\mathrm{b} _{\Q, \mathrm{qc}}
( \smash{\widetilde{\D}} _{\fP ^{ }/\S }  ^{(\bullet)} )
\times 
\smash{\underrightarrow{LD}}  ^\mathrm{b} _{\Q, \mathrm{qc}}
( \smash{\widetilde{\D}} _{\fQ ^{ }/\S }  ^{(\bullet)} )
\to 
\smash{\underrightarrow{LD}}  ^\mathrm{b} _{\Q, \mathrm{qc}}
( \smash{\widetilde{\D}} _{\fP \times \fQ/\S }  ^{(\bullet)} )
\end{equation}
defined as follows: 
for any $\E ^{ (\bullet)} 
\in 
\smash{\underrightarrow{LD}}  ^\mathrm{b} _{\Q, \mathrm{qc}}
( \smash{\widetilde{\D}} _{\fP ^{ }/\S }  ^{(\bullet)} )$,
$\FF  ^{ (\bullet)} 
\in 
\smash{\underrightarrow{LD}}  ^\mathrm{b} _{\Q, \mathrm{qc}}
( \smash{\widetilde{\D}} _{\fQ ^{ }/\S }  ^{(\bullet)} )$, 
we set
$$ \E ^{ (\bullet)}
 \smash{\widehat{\boxtimes}}
^\L _{\O _{\S }}
\FF ^{ (\bullet)}
:= 
p _1 ^{(\bullet) *} \E ^{ (\bullet)}
 \smash{\widehat{\otimes}}
^\L _{\O ^{(\bullet)}  _{\fP \times \fQ} }
p _2 ^{(\bullet) *} \FF ^{ (\bullet)}.$$
As for \cite[4.3.5]{Beintro2}, this functor induces 
the following one
\begin{equation}
\label{boxtimesLDcoh}
 \smash{\widehat{\boxtimes}} 
^\L _{\O _{\S }}
\colon 
\smash{\underrightarrow{LD}}  ^\mathrm{b} _{\Q, \mathrm{coh}}
( \smash{\widetilde{\D}} _{\fP ^{ }/\S }  ^{(\bullet)} )
\times 
\smash{\underrightarrow{LD}}  ^\mathrm{b} _{\Q, \mathrm{coh}}
( \smash{\widetilde{\D}} _{\fQ ^{ }/\S }  ^{(\bullet)} )
\to 
\smash{\underrightarrow{LD}}  ^\mathrm{b} _{\Q, \mathrm{coh}}
( \smash{\widetilde{\D}} _{\fP \times \fQ/\S }  ^{(\bullet)} ).
\end{equation}

\end{empt}

\begin{empt}
For any $\E ^{ (\bullet)} 
\in 
\smash{\underrightarrow{LD}}  ^\mathrm{b} _{\Q, \mathrm{qc}}
( \smash{\widetilde{\D}} _{\fP ^{ }/\S }  ^{(\bullet)} )$,
$\FF  ^{ (\bullet)} 
\in 
\smash{\underrightarrow{LD}}  ^\mathrm{b} _{\Q, \mathrm{qc}}
( \smash{\widetilde{\D}} _{\fQ ^{ }/\S }  ^{(\bullet)} )$, 
we have the isomorphism
\begin{equation}
\label{boxtimesalg-formal}
\E ^{ (\bullet)}
 \smash{\widehat{\boxtimes}}
^\L _{\O _{\S }}
\FF ^{ (\bullet)}
\riso 
\R \underleftarrow{\lim}_i \,  
\left (
\E _i ^{ (\bullet)}
 \smash{\widehat{\boxtimes}}
^\L _{\O _{S _i}}
\FF _i ^{ (\bullet)}
\right ),
\end{equation}
where as usual we set 
$\E ^{ (\bullet)} _i := \smash{\widetilde{\D}} _{P  _i/ S  _i} ^{(\bullet)}  \otimes ^\L _{\smash{\widetilde{\D}} _{\fP /\S } ^{(\bullet)} } \E ^{ (\bullet)}$,
and
$\cF ^{ (\bullet)} _i := \smash{\widetilde{\D}} _{P  _i/ S  _i} ^{(\bullet)}  \otimes ^\L _{\smash{\widetilde{\D}} _{\fP /\S } ^{(\bullet)} } \cF ^{ (\bullet)}$.

\end{empt}

\begin{lem}
\label{exact-boxtimes}
The bifunctor \ref{boxtimesLDcoh} induces the
 exact bifunctor
$$ \smash{\widehat{\boxtimes}} ^\L _{\O _{\S }}
\colon 
\smash{\underrightarrow{LM}}  _{\Q, \mathrm{coh}}
( \smash{\widetilde{\D}} _{\fP ^{ }/\S }  ^{(\bullet)} )
\times 
\smash{\underrightarrow{LM}}   _{\Q, \mathrm{coh}}
( \smash{\widetilde{\D}} _{\fQ ^{ }/\S }  ^{(\bullet)} )
\to 
\smash{\underrightarrow{LM}}  _{\Q, \mathrm{coh}}
( \smash{\widetilde{\D}} _{\fP \times \fQ/\S }  ^{(\bullet)} ).$$ 
\end{lem}

\begin{proof}
Let 
$\E ^{ (\bullet)} 
\in 
\smash{\underrightarrow{LM}}  _{\Q, \mathrm{coh}}
( \smash{\widetilde{\D}} _{\fP ^{ }/\S }  ^{(\bullet)} )$,
$\FF  ^{ (\bullet)} 
\in 
\smash{\underrightarrow{LM}}   _{\Q, \mathrm{coh}}
( \smash{\widetilde{\D}} _{\fQ ^{ }/\S }  ^{(\bullet)} )$.
Let 
$\E:= \underrightarrow{\lim} 
\,
\E ^{ (\bullet)}$,
$\FF:= \underrightarrow{\lim} 
\,
\FF ^{ (\bullet)}$,
where $\underrightarrow{\lim} $ is the equivalence of categories
of \ref{M-eq-coh-lim}.
Choose $m _0$ large enough so that there exists
a coherent 
$ \smash{\widetilde{\D}} _{\fP ^{ }/\S }  ^{(m _0)} $-module 
$\mathscr{E} ^{ (m _0)} $ without $p$-torsion such that 
$\D ^\dag _{\fP/\S, \Q}
\otimes _{\smash{\widetilde{\D}} _{\fP ^{ }/\S }  ^{(m _0)} }
\mathscr{E} ^{ (m _0)}  
\riso 
\E $, 
and 
a coherent 
$ \smash{\widetilde{\D}} _{\fQ ^{ }/\S }  ^{(m _0)} $-module 
$\mathscr{F} ^{ (m _0)} $ without $p$-torsion such that 
$\D ^\dag _{\fQ/\S, \Q}
\otimes _{\smash{\widetilde{\D}} _{\fQ ^{ }/\S }  ^{(m _0)} }
\mathscr{F} ^{ (m _0)}  
\riso 
\FF $. 
For any $m \geq m _0$, 
let 
$\mathscr{E} ^{ (m )} $ and 
(resp. $\mathscr{F} ^{ (m )} $) 
 be 
the quotient of 
$\smash{\widetilde{\D}} _{\fP ^{ }/\S }  ^{(m )} 
\otimes _{\smash{\widetilde{\D}} _{\fP ^{ }/\S }  ^{(m _0)} }
\mathscr{E} ^{ (m _0)} $
(resp. 
$\smash{\widetilde{\D}} _{\fQ ^{ }/\S }  ^{(m )} 
\otimes _{\smash{\widetilde{\D}} _{\fQ ^{ }/\S }  ^{(m _0)} }
\mathscr{F} ^{ (m _0)} $)
by its torsion part. 
 We get 
 $\mathscr{E} ^{ (\bullet + m _0)} 
\in 
\smash{\underrightarrow{LM}}   _{\Q, \mathrm{coh}}
( \smash{\widetilde{\D}} _{\fP ^{ }/\S }  ^{(\bullet)} )$,
$\mathscr{F}  ^{ (\bullet+ m _0)} 
\in 
\smash{\underrightarrow{LM}}  _{\Q, \mathrm{coh}}
( \smash{\widetilde{\D}} _{\fQ ^{ }/\S }  ^{(\bullet)} )$
such that 
$\underrightarrow{\lim} 
\,
\mathscr{E} ^{ (\bullet + m _0)} 
\riso 
\E$,
and 
$\underrightarrow{\lim} 
\,
\mathscr{F} ^{ (\bullet + m _0)} 
\riso 
\FF$.
Hence, we obtain the isomorphisms 
$\E ^{ (\bullet)}  \riso \mathscr{E} ^{ (\bullet + m _0)} $
and 
$\FF ^{ (\bullet)}  \riso \mathscr{F} ^{ (\bullet + m _0)} $.
Let 
$r \colon \fP \times _\S \fQ \to \fS$ be the structural morphism.
If no confusion is possible, the sheaf $r ^{-1}\O _{S _i}$
will be denoted by $\O _{S _i}$.
Since 
$\mathscr{E} ^{ (m)}$ and 
$\mathscr{F} ^{ (m)}$ have no $p$-torsion, then 
the canonical morphism
$p _1 ^{-1} \mathscr{E} ^{ (m)} _i
\otimes ^{\L}
 _{\O _{S _i}}
p _2 ^{-1} \mathscr{F} ^{ (m)} _i
\to 
p _1 ^{-1} \mathscr{E} ^{ (m)} _i
\otimes
 _{\O _{S _i}}
p _2 ^{-1} \mathscr{F} ^{ (m)} _i$
is an isomorphism.
Since
\begin{gather}
\notag
p _1 ^{-1} \mathscr{E} ^{ (m)} _i
\otimes
 _{\O _{S _i}}
p _2 ^{-1} \mathscr{F} ^{ (m)} _i
\riso
\\
\notag
\left (
p _1 ^{-1} \mathscr{E} ^{ (m)} _i
\otimes _{p _1  ^{-1} \O _{P _i}}
(p _1  ^{-1} \O _{P _i}
	\otimes 
	 _{\O _{S _i}}
	p _2 ^{-1} \O _{Q _i})
\right )
\otimes 
	_{
	(p _1  ^{-1} \O _{P _i}
	\otimes 
	 _{\O _{S _i}}
	p _2 ^{-1} \O _{Q _i})
	}
\left (	
(p _1  ^{-1} \O _{P _i}
	\otimes 
	 _{\O _{S _i}}
	p _2 ^{-1} \O _{Q _i})
\otimes _{	p _2 ^{-1} \O _{Q _i}}
p _2 ^{-1} \mathscr{F} ^{ (m)} _i
\right),
\end{gather}
and since the extension
$(p _1  ^{-1} \O _{P _i}
	\otimes 
	 _{\O _{S _i}}
	p _2 ^{-1} \O _{Q _i})
\to \O   _{P _i \times Q _i} 	$
is flat, we get the first isomorphism
\begin{gather}
\notag
p _1 ^{*} \mathscr{E} ^{ (m)}
 \smash{\widehat{\otimes}}
^\L _{\O   _{\fP \times \fQ} }
p _2 ^{*} \mathscr{F} ^{ (m)}
\riso 
\R \underleftarrow{\lim}_i \, 
\O   _{P _i \times Q _i} 
\otimes 
	_{
	(p _1  ^{-1} \O _{P _i}
	\otimes 
	 _{\O _{S _i}}
	p _2 ^{-1} \O _{Q _i})
	}
(p _1 ^{-1} \mathscr{E} ^{ (m)} _i
\otimes
 _{\O _{S _i}}
p _2 ^{-1} \mathscr{F} ^{ (m)} _i)
\\
\notag
\riso
\underleftarrow{\lim}_i \, 
\O   _{P _i \times Q _i} 
\otimes 
	_{
	(p _1  ^{-1} \O _{P _i}
	\otimes 
	 _{\O _{S _i}}
	p _2 ^{-1} \O _{Q _i})
	}
(p _1 ^{-1} \mathscr{E} ^{ (m)} _i
\otimes
 _{\O _{S _i}}
p _2 ^{-1} \mathscr{F} ^{ (m)} _i)
\riso
p _1 ^{*} \mathscr{E} ^{ (m)}
 \smash{\widehat{\otimes}}
_{\O   _{\fP \times \fQ} }
p _2 ^{*} \mathscr{F} ^{ (m)},
\end{gather}
the second isomorphism is checked using 
Mittag-Leffler.
\end{proof}

\begin{coro}
We get the t-exact bifunctor
\begin{equation}
\label{dfn-boxtimesDLMcoh}
 \smash{\widehat{\boxtimes}} 
^\L _{\O _{\S }}
\colon 
D ^{\mathrm{b}}  (\underrightarrow{LM} _{\Q,\mathrm{coh}} (\smash{\widetilde{\D}} _{\fP /\S } ^{(\bullet)} ))
\times 
D ^{\mathrm{b}}  (\underrightarrow{LM} _{\Q,\mathrm{coh}} (\smash{\widetilde{\D}} _{\fQ /\S } ^{(\bullet)} ))
\to 
D ^{\mathrm{b}}  (\underrightarrow{LM} _{\Q,\mathrm{coh}} (\smash{\widetilde{\D}} _{\cP \times \fQ /\S } ^{(\bullet)})).
\end{equation}

\end{coro}

\begin{prop}
\label{prop-boxtimes}
\begin{enumerate}
\item 
\label{prop-boxtimes1}
Let 
$\E ^{ (\bullet)} 
\in 
D ^{\mathrm{b}}  (\underrightarrow{LM} _{\Q,\mathrm{coh}} (\smash{\widetilde{\D}} _{\fP /\S } ^{(\bullet)} ))$,
$\FF ^{ (\bullet)} 
\in 
D ^{\mathrm{b}}  (\underrightarrow{LM} _{\Q,\mathrm{coh}} (\smash{\widetilde{\D}} _{\fQ /\S } ^{(\bullet)} ))$. 
We get
the spectral sequence in 
$\underrightarrow{LM} _{\Q,\mathrm{coh}} (\smash{\widetilde{\D}} _{\cP \times \fQ /\S } ^{(\bullet)})$
of the form
$$\H ^r (\E  ^{ (\bullet)} )
\smash{\widehat{\boxtimes}} 
^\L _{\O _{\S }}
\H ^s
 ( 
\FF  ^{ (\bullet)}
)
=:
E _{2} ^{r,s}
\Rightarrow
E ^n :=
\H ^n 
\left ( 
\E  ^{ (\bullet)} 
\smash{\widehat{\boxtimes}} 
^\L _{\O _{\S }}
\FF  ^{ (\bullet)}
\right ) .$$
In particular,
when
$\E ^{ (\bullet)} 
\in 
\smash{\underrightarrow{LM}}  _{\Q, \mathrm{coh}}
( \smash{\widetilde{\D}} _{\fP ^{ }/\S }  ^{(\bullet)} )$,
this yields 
$\H ^n 
\left ( 
\E  ^{ (\bullet)} 
\smash{\widehat{\boxtimes}} 
^\L _{\O _{\S }}
\FF  ^{ (\bullet)}
\right ) 
\riso 
\E  ^{ (\bullet)} 
\smash{\widehat{\boxtimes}} 
^\L _{\O _{\S }}
\H ^n 
 ( 
\FF  ^{ (\bullet)}
) $.

\item 
\label{prop-boxtimes2}
Suppose $\fQ$ affine.
Let 
$\E ^{ (\bullet)} 
\in 
\smash{\underrightarrow{LM}}  _{\Q, \mathrm{coh}}
( \smash{\widetilde{\D}} _{\fP ^{ }/\S }  ^{(\bullet)} )$,
$\FF  ^{ (\bullet)} 
\in 
\smash{\underrightarrow{LD}}  ^\mathrm{b} _{\Q, \mathrm{coh}}
( \smash{\widetilde{\D}} _{\fQ ^{ }/\S }  ^{(\bullet)} )$.
We have 
$\H ^n 
\left ( 
\E  ^{ (\bullet)} 
\smash{\widehat{\boxtimes}} 
^\L _{\O _{\S }}
\FF  ^{ (\bullet)}
\right ) 
\riso 
\E  ^{ (\bullet)} 
\smash{\widehat{\boxtimes}} 
^\L _{\O _{\S }}
\H ^n 
 ( 
\FF  ^{ (\bullet)}
) $.

\end{enumerate}

\end{prop}

\begin{proof}
The fist statement is a consequence of the t-exactness of the functor
\ref{dfn-boxtimesDLMcoh}. Moreover, when $\fQ$ is affine,
following \ref{cor-eq-cat-coh-m-2},
there exists 
$\G  ^{ (\bullet)}  \in 
D ^{\mathrm{b}}  (\underrightarrow{LM} _{\Q,\mathrm{coh}} (\smash{\widetilde{\D}} _{\fQ /\S } ^{(\bullet)} ))$
such that  $\G  ^{ (\bullet)}  $ 
is isomorphic in 
$\smash{\underrightarrow{LD}}  ^\mathrm{b} _{\Q, \mathrm{coh}}
( \smash{\widetilde{\D}} _{\fQ ^{ }/\S }  ^{(\bullet)} )$
to 
$\FF  ^{ (\bullet)} $.
We get 
$\H ^n 
\left ( 
\E  ^{ (\bullet)} 
\smash{\widehat{\boxtimes}} 
^\L _{\O _{\S }}
\G  ^{ (\bullet)}
\right ) 
\riso 
\E  ^{ (\bullet)} 
\smash{\widehat{\boxtimes}} 
^\L _{\O _{\S }}
\H ^n 
 ( 
\G  ^{ (\bullet)}
) $.
Beware that we have two distincts bifunctors \ref{boxtimesLDcoh}, \ref{dfn-boxtimesDLMcoh}.
Hence, using the commutative diagram 
\ref{diag-Hn-comp-coh}, we get
$\H ^n 
\left ( 
\E  ^{ (\bullet)} 
\smash{\widehat{\boxtimes}} 
^\L _{\O _{\S }}
\FF  ^{ (\bullet)}
\right ) 
\riso 
\E  ^{ (\bullet)} 
\smash{\widehat{\boxtimes}} 
^\L _{\O _{\S }}
\H ^n 
 ( 
\FF  ^{ (\bullet)}
) $.
\end{proof}

\begin{prop}
\label{prop-boxtimes-v+}
Let $u\colon \fP' \to \fP$ and  $v\colon \fQ' \to \fQ$ 
be two morphisms of smooth formal  $\V$-schemes.
Let
$\ZZ: =\fP \times _\S \fQ $,
$\ZZ ' := \fP '\times _\S \fQ'$,
and $w:= (u, v) \colon \ZZ ' \to \ZZ$ be the induced morphism.
\begin{enumerate}
\item For any
$\E ^{(\bullet)}
\in \underrightarrow{LD}  ^\mathrm{b} _{\Q, \mathrm{qc}}
(\overset{^\mathrm{g}}{} \smash{\widehat{\D}} _{\fP} ^{(\bullet)} )$
and 
$\FF ^{ (\bullet)}
\in \underrightarrow{LD}  ^\mathrm{b} _{\Q, \mathrm{qc}}
(\overset{^\mathrm{g}}{} \smash{\widehat{\D}} _{\fQ} ^{(\bullet)} )$, 
with notation \ref{ntn-Lf!+*}, we have in 
 $\underrightarrow{LD}  ^\mathrm{b} _{\Q, \mathrm{qc}}
(\overset{^\mathrm{g}}{} \smash{\widehat{\D}} _{\ZZ '} ^{(\bullet)})$
the isomorphism:
\begin{equation}
\label{boxtimes-v!formal}
\L w ^{*(\bullet)}  (\E ^{(\bullet)}
\smash{\widehat{\boxtimes}} ^\L _{\O _{\S }}
\FF ^{(\bullet)})
\riso
\L u ^{*(\bullet)} (\E ^{(\bullet)})
\smash{\widehat{\boxtimes}} ^\L _{\O _{\S }}
\L v ^{*(\bullet)} (\FF ^{(\bullet)}).
\end{equation}

\item For any
$\E ^{\prime(\bullet)}
\in \underrightarrow{LD}  ^\mathrm{b} _{\Q, \mathrm{qc}}
(\overset{^\mathrm{g}}{} \smash{\widehat{\D}} _{\fP'} ^{(\bullet)} )$
and 
$\FF ^{\prime (\bullet)}
\in \underrightarrow{LD}  ^\mathrm{b} _{\Q, \mathrm{qc}}
(\overset{^\mathrm{g}}{} \smash{\widehat{\D}} _{\fQ'} ^{(\bullet)} )$, 
we have in 
 $\underrightarrow{LD}  ^\mathrm{b} _{\Q, \mathrm{qc}}
(\overset{^\mathrm{g}}{} \smash{\widehat{\D}} _{\ZZ} ^{(\bullet)})$
the isomorphism:
\begin{equation}
\label{boxtimes-v+}
w ^{(\bullet)} _+ (\E ^{\prime(\bullet)}
\smash{\widehat{\boxtimes}} ^\L _{\O _{\S }}
\FF ^{\prime(\bullet)})
\riso
u ^{(\bullet)}_+ (\E ^{\prime(\bullet)})
\smash{\widehat{\boxtimes}} ^\L _{\O _{\S }}
v ^{(\bullet)}_+ (\FF ^{\prime(\bullet)}).
\end{equation}

\end{enumerate}

\end{prop}

\begin{proof}
The first statement is a consequence of \ref{comm-boxtimes-f*} and \ref{boxtimesalg-formal}.
The second one is a consequence of \ref{sch-prop-boxtimes-v+} and \ref{boxtimesalg-formal}. 
\end{proof}

\begin{coro}
\label{theo-iso-chgtbase2}
We keep notation \ref{prop-boxtimes-v+} and we suppose
$v $ is the identity. 
Let $\pi \colon  \fZ  \to \fP $, 
and $\pi '\colon  \fZ ^{ \prime } \to \fP ^{ \prime }$
be the projections. 
Let 
$\E ^{\prime (\bullet)}
\in  \smash{\underrightarrow{LD}} ^\mathrm{b} _{\Q, \mathrm{qc}}
(\overset{^\mathrm{l}}{} \smash{\widetilde{\D}} _{\fP ^{\prime }} ^{(\bullet)})$. 
There exists a canonical isomorphism in 
$\smash{\underrightarrow{LD}} ^\mathrm{b} _{\Q, \mathrm{qc}}
(\overset{^\mathrm{l}}{} \smash{\widetilde{\D}} _{\fZ } ^{(\bullet)})$ of the form:
\begin{equation}
\label{iso-chgtbase2}
\pi ^{ !(\bullet)} \circ u ^{(\bullet)}_{ +} (\E ^{\prime (\bullet)})
\riso
w  ^{(\bullet)}_{+}  \circ \pi  ^{\prime (\bullet) !} (\E ^{\prime (\bullet)}). 
\end{equation}
\end{coro}

\begin{proof}
This is a consequence of \ref{theo-iso-chgtbase2-pre} (or we can deduce it from \ref{prop-boxtimes-v+}). 
\end{proof}

\begin{rem}
\label{remoftheo-iso-chgtbase2}
We will prove later (see \ref{theo-iso-chgtbase}) a coherent version of Corollary \ref{theo-iso-chgtbase2}.
In this version, we can use for instance Berthelot-Kashiwara theorem which allow us to extend geometrically 
the context. 
\end{rem}

The following corollary is weaker than \ref{dualrelative} but the reader can check that
is the proof is much easier. 
\begin{cor}
\label{rel-dual-isom-proj-f}
Let $\fX \to \fY$ be a morphism of smooth formal $\fS$-schemes which is the composition of a closed immersion of the form
$\fX \hookrightarrow \widehat{\bbP} ^d _\fY$ and of the projection $\widehat{\bbP} ^d _\fY \to \fY$. 

\begin{enumerate}
\item For any
$\cE ^{(\bullet)} \in  \smash{\underrightarrow{LD}} ^\mathrm{b} _{\Q, \mathrm{coh}}
(\overset{^\mathrm{l}}{} \smash{\widetilde{\D}} _{\fX } ^{(\bullet)})$,
we have a canonical isomorphism of 
$\smash{\underrightarrow{LD}} ^{\mathrm{b}} _{\Q,\mathrm{coh}} ( \smash{\widehat{\D}} _{\fY} ^{(\bullet)})$
of the form :
\begin{equation}
\label{rel-dual-isom-proj-f1}
\DD ^{(\bullet)} \circ f ^{(\bullet)} _+ (\cE ^{(\bullet)})
\riso 
 f ^{(\bullet)} _+ \circ \DD ^{(\bullet)} (\cE ^{(\bullet)}).
\end{equation}

\item Let $\E  \in D ^\mathrm{b} _{\mathrm{coh}}
(\D ^{\dag} _{\fX,\Q})$,
and
$\cF 
\in 
D ^\mathrm{b} _{\mathrm{coh}}
(\D ^{\dag} _{\fY ,\Q})$.
We have 
the isomorphisms
\begin{gather}
\label{cor-adj-formul-proj-formal-bij1}
\R \mathcal{H} om _{\D ^{\dag} _{\fY,\Q}}
( f _{+} ( \E ) , \cF) 
\riso 
\R f _* 
\R \mathcal{H} om _{\D ^{\dag} _{\fX ,\Q}}
( \E  ,f ^!   ( \cF)),
\\
\label{cor-adj-formul-proj-formal-bij2}
\R \mathrm{Hom}  _{\D ^{\dag} _{\fY ,\Q}}
( f _{+} ( \E) , \cF) 
\riso 
\R \mathrm{Hom}  _{\D ^{\dag} _{\fX,\Q}}
( \E  ,f ^!   ( \cF)). 
\end{gather}
\end{enumerate}

\end{cor}

\begin{proof}
The first statement is a consequence of \ref{rel-dual-isom-proj}.
Similarly to \ref{cor-adj-formul},
we check that 
\ref{rel-dual-isom-proj-f1}
implies the second statement.
\end{proof}

\section{Differential coherence of $\O _{\X} ({}^\dag Z) _\Q$ when $Z$ is a divisor}

\subsection{Local cohomology with support in a  smooth closed subscheme of the constant coefficient}

\begin{empt}
\label{com-invMIC-ntn}
Let $\fP $ be a smooth formal scheme over $\S $.
Let $u _0\colon X  \hookrightarrow P $ be a closed immersion of smooth schemes over $S $ purely of codimension $r$.
Let $T$ be a divisor of $P$ such that $Z:= T \cap X$ is a divisor of $X$. 
We set $\U := \fP \setminus T$, $Y:= X \setminus Z$, 
 $v _0 \colon Y \to U$ be the morphism induced by $u_ 0$.
Choose $(\fP  _{\alpha}) _{\alpha \in \Lambda}$ an open affine covering of  $\fP $
and let us use the corresponding notation of \ref{ntnPPalpha} (which are compatible with
that of \ref{ntnPPalpha-withoutdiv}).

\begin{enumerate}
\item Denoting by 
$f :=(v _0,\,u _0,\,id)
\colon 
(Y, X,\fP)
\to 
(U, P,\fP)$ the morphism of frames,
we get the inverse image $f _K ^* = | _{]X[ _\fP}
  \colon
  \mathrm{MIC} ^{\dag} (U, P,\fP/K) 
\to 
    \mathrm{MIC} ^{\dag} (Y, X,\fP/K)$ (see notation \ref{inv-image-real}). 
    Hence, we get the functor
$u _{0K}  ^*\circ  | _{]X[ _\fP}
  \colon
 \mathrm{MIC} ^{\dag} (U, P,\fP/K)
\to 
\mathrm{MIC} ^{\dag} (Y, (\X   _\alpha )_{\alpha \in \Lambda}/K)$ (see \ref{eqcat-iso-reco}).

\item Similarly to the construction of
$u _0 ^! \colon 
\mathrm{Coh} (X, \fP,T /K)
\to 
\mathrm{Coh} ((\X   _\alpha )_{\alpha \in \Lambda},Z/K)$
of   \ref{const-u0!}, 
we can define the functor 
$u _0  ^*
  \colon
\mathrm{MIC} ^{\dag \dag} (\fP,T/K) 
\to 
\mathrm{MIC} ^{\dag \dag} ( (\X   _\alpha )_{\alpha \in \Lambda},Z/K)$
    as follows. 
    Let $\E \in  \mathrm{MIC} ^{\dag \dag} (\fP,T/K) $, i.e. 
    a coherent $\D ^{\dag} _{\fP  } (\hdag T ) _{\Q}$ which is also 
    $\O _{\fP  } (\hdag T ) _{\Q}$-coherent.     
We set 
$\E _\alpha : =
u _{\alpha } ^{*} ( \E | \fP _\alpha) := 
\mathcal{H} ^{-r}u _{\alpha } ^{!} ( \E | \fP _\alpha)
\riso
u _{\alpha } ^{!} ( \E | \fP _\alpha) [-r]$.
Then 
$\E _\alpha$ is a coherent
$\D ^{\dag} _{\X  _{\alpha}} (\hdag Z _{\alpha}) _{\Q} $-module, 
which is also 
$\O  _{\X _{\alpha} }(\hdag Z _\alpha ) _{\Q}$-coherent.
Via the isomorphisms of the form $\tau $ 
(\ref{prop-glueiniso-coh1}),
we obtain the glueing 
$\D ^{\dag} _{\X  _{\alpha \beta}} (\hdag Z _{\alpha \beta}) _{\Q} $-linear
isomorphism
$ \theta _{  \alpha \beta} \ : \  p _2  ^{\alpha \beta !} (\E _{\beta}) \riso p  _1 ^{\alpha \beta !} (\E _{\alpha}),$
satisfying the cocycle condition:
$\theta _{13} ^{\alpha \beta \gamma }=
\theta _{12} ^{\alpha \beta \gamma }
\circ
\theta _{23} ^{\alpha \beta \gamma }$.

\end{enumerate}

\end{empt}

\begin{prop}
\label{spinvim*+}
With the notation \ref{com-invMIC-ntn}, 
we have the canonical isomorphism
$$\sp _* \circ u _{0 K} ^* 
\circ  | _{]X[ _\fP} \riso 
u _0 ^* \circ \sp _*$$ 
of functors
$\mathrm{MIC} ^{\dag} (U, P,\fP/K)
\to \mathrm{MIC} ^{\dag \dag} (X, \fP,T/K).$

\end{prop}

\begin{proof}
Using 
\ref{sp-eps-tau},
we check that glueing data are compatible.
\end{proof}

For the reader, let us recall the 
differential coherence of $\O _{\fP} ({}^\dag T) _\Q$ when $T$ is a strict NCD.
\begin{prop}
[Berthelot]
\label{NCDgencoh}
Let $\fP $ be a smooth formal scheme over $\S $.
Suppose 
there exist local coordinates
$t _1, \dots, t _d $ of $\fP$ over $\fS$.
Let $T$ be the divisor of $P$ defined by setting
$T: = V ( \overline{t} _1\cdots \overline{t} _r)$ with $r \leq d$, where
$\overline{t} _1, \dots, \overline{t} _r$ are the images of $t _1,\dots, t _r$ in $\Gamma (P ,\O _{P})$.
We have  the exact sequence
\begin{equation}
(\D ^\dag _{\fP /\S ,\Q}) ^{d}
\overset{\psi}{\longrightarrow} 
\D ^\dag _{\fP /\S ,\Q}
\overset{\phi}{\longrightarrow}
\O _{\fP} (\hdag T ) _\Q
\to 0,
\end{equation}
where 
$\phi (P)= P \cdot (1/t _1\cdots t _r)$, and 
$\psi$ is defined by
\begin{equation}
\psi ( P _1,\dots, P _d) = \sum _{i=1} ^{r} P _i \partial _i t _i + \sum _{i=r+1} ^d P _i \partial _i.
\end{equation}
\end{prop}

\begin{proof}
Similar to \cite[4.3.2]{Be0}.
\end{proof}

\begin{dfn}
\label{sscd}
Let $P $ be a smooth  scheme over $\Spec k$.
Let $T$ be the divisor of $P$. 

\begin{enumerate}
\item We say that $T$ 
is a {\it strict smooth crossing divisor} of $P/\Spec k$
if 
Zariski locally in $P$,
there exist local coordinates
$t _1,\cdots ,t _d$ of $P$ over $S$
such that 
$T = V ( t _1\cdots t _r)$ with $r \leq d$.
Following \ref{NCDgencoh}, 
when $T$ is a strict smooth crossing divisor of $P/\Spec k$, then
the left $\D ^\dag _{\fP /\S ,\Q}$-module
$\O _{\fP} (\hdag T ) _\Q$ is coherent. 
Remark also that this notion 
of strict smooth crossing divisor of $P/\Spec k$ is stable under any base change 
$\Spec k' \to \Spec k$.

\item We say that $T$ 
is a {\it geometric strict normal crossing divisor} of $P/\Spec k$,
if $T \times _{\Spec k} \Spec l$
is a strict normal crossing divisor of $P \times _{\Spec k} \Spec l$
for some  perfect closure $l $  of $k$ (i.e.  
$l$ is an algebraic extension of $k$ which is perfect).
\end{enumerate}

\end{dfn}

\begin{lemm}
\label{sscd-gsncd}
Let $P $ be a smooth scheme over $\Spec k$.

\begin{enumerate}

\item A strict smooth crossing divisor of $P/\Spec k$
is a geometric strict normal crossing divisor $P/\Spec k$.
When $k$ is perfect, both notions are equal to 
the (absolute) notion of strict normal crossing divisor.

\item Let $T$ be the geometric strict normal crossing divisor of $P/\Spec k$. 
Let  $l $ be a perfect closure of $k$.
Then there exists a finite subextension $k'/k$ of $l/k$ such that
$T \times _{\Spec k} \Spec k'$
is a strict smooth crossing divisor of $P \times _{\Spec k} \Spec k'/\Spec k'$.
\end{enumerate}

\end{lemm}

\begin{proof}
Let us check the first statement. 
Since the notion of strict smooth crossing divisor is closed under base change, 
we reduce to the case where the field $k$ is perfect.
Recall (e.g. see 40.21.1 of the stack project), 
that the property $T$ is a strict normal crossing divisor of $P$
means that for any point $x \in T$, 
 there exists a regular system of parameters
$t _{1,x},\dots, t _{d,x}$ 
 of 
$\fm _x$ (the maximal ideal of $\O _{P,x}$) and
 $1\leq r\leq d$
 such that $T$ is cut out by 
 $t _1\dots t _r$ in $\O _{P,x}$.
Hence, the first statement is a consequence of \cite[17.15.4]{EGAIV4}  

Let us prove the second one.
Let $T$ be the geometric strict normal crossing divisor of $P/\Spec k$ and
let  $l $ be a perfect closure of $k$.
Since $l$ is perfect, 
then $T \times _{\Spec k} \Spec l$
is a strict smooth crossing divisor of $P \times _{\Spec k} \Spec l/\Spec l$.
Hence, Zariski locally we have an étale morphism of $\Spec l$-varieties of the form
$\phi \colon P \times _{\Spec k} \Spec l \to \bbA ^{d} _{\Spec l}$ given by 
$t _1, \dots, t _d$ 
such that 
$T \times _{\Spec k} \Spec l = V ( t _1\cdots t _r)$ with $r \leq d$.
Thanks to \cite[8.8.2.(ii)]{EGAIV3} and \cite[17.7.8]{EGAIV4},
there exists a finite subextension $k'/k$ of $l/k$ and 
 an étale morphism of $\Spec k'$-varieties of the form
$\phi ' \colon P \times _{\Spec k} \Spec k' \to \bbA ^{d} _{\Spec k'}$ given by 
$t '_1, \dots, t '_d$ such that $\phi = \phi ' \times _{\Spec k'} \Spec l$.
Again by  using  \cite[8.8.2.(ii)]{EGAIV3},
increasing $k'$ if necessary, we can suppose that 
$T \times _{\Spec k} \Spec k' = V ( t' _1\cdots t '_r)$ with $r \leq d$.
\end{proof}

\begin{lemm}
\label{40.21.2StPrj}
Let $P $ be a smooth scheme over $\Spec k$.
Let $T$ be a (reduced) divisor of $P$.
Let $T _i \subset T$, $i\in I$ be its irreducible components viewed as closed (reduced) subschemes of $P$.
The following are equivalent. 
\begin{enumerate}
\item $T$
be a strict smooth crossing divisor of $P/\Spec k$, and 
\item for any  $J \subset I$ finite the scheme theoretic intersection 
$T _J := \cap _{j\in J} T _j$ is smooth over $\Spec k$ each of whose irreducible 
components has codimension $|J|$ in $P$.
\end{enumerate}

\end{lemm}

\begin{proof}
The check is easy and analogue to the proof of  40.21.2 of the stack project.
\end{proof}

\begin{ntn}
\label{ntn-GammaZO-rig}
Before defining local cohomology in the context of quasi-coherent complexes
(see \ref{localcoho-section}), we will need to 
focus on the case of a smooth closed subscheme for the constant coefficient as follows. 
We will see via \ref{cor-induction-div-coh2} that both local cohomology are canonically compatible, 
which justifies using the same notation. 
Let $\fP $ be a smooth formal scheme over $\S $.
Let $X$ be a smooth closed subscheme of $P$ and $j _X \colon  P   \setminus X\to P$ be the open immersion. 
We set 
$(\hdag X) ( \O _{\fP,\Q}) 
:=\R \sp _* j _X ^\dag (\O _{\fP _K})$
and
$\R \underline{\Gamma} ^\dag _X \O _{\fP,\Q} 
:=\R \sp _* \underline{\Gamma} ^\dag _X ( \O _{\fP _K})$.
By definition, 
$\R \underline{\Gamma} ^\dag _X \O _{\fP,\Q} $ is the 
local cohomology with support in $X$ of $\O _{\fP,\Q}$.
The exact sequence 
$0 \to \underline{\Gamma} ^\dag _X (\O _{\fP _K}) \to   \O _{\fP _K}
\to 
j _X ^\dag (  \O _{\fP _K})
\to 
0$
induces the exact triangle 
\begin{equation}
\label{extriangleloc}
\R \underline{\Gamma} ^\dag _X \O _{\fP,\Q} 
\to 
\O _{\fP,\Q}
\to 
(\hdag X) ( \O _{\fP,\Q}) 
\to 
\R \underline{\Gamma} ^\dag _X \O _{\fP,\Q}  [1].
\end{equation}
For any integer $i\in \Z$, we set 
$\mathcal{H} ^{\dag i } _X (\O _{\fP,\Q}  ):= 
\mathcal{H} ^{i} \R \underline{\Gamma} ^\dag _X \O _{\fP,\Q} $.
\end{ntn}

\begin{prop}
[Berthelot]
\label{coh-smoothsubsch}
Let $\fP $ be a smooth formal scheme over $\S $.
Let $u \colon X  \to P $ be a closed immersion  of smooth schemes over $S $ purely of codimension $r$.

\begin{enumerate}
\item $(\hdag X) ( \O _{\fP,\Q}) ,
\R \underline{\Gamma} ^\dag _X \O _{\fP,\Q}  
\in 
D ^{\mathrm{b}} _{\mathrm{coh}}
(\D ^\dag _{\fP, \Q})$, and
$\mathcal{H} ^{\dag i } _X (\O _{\fP,\Q}  )=0$
for any $i\not = r$.

\item Let $x \in P$. 
Choose  an open affine formal subscheme $\U $ of $\fP $ containing $x$
such that 
there exist coordinates
$t _1, \dots, t _d \in \Gamma (\U ,\O _{\U})$ 
such that $X\cap U = V ( \overline{t} _1, \dots, \overline{t} _r)$ where
$r \leq d$ and
$\overline{t} _1, \dots, \overline{t} _r$ are the image of $t _1,\dots, t _r$ in $\Gamma (U ,\O _{U})$.
We have the exact sequence
\begin{equation}
\label{exseq-HrZ}
(\D ^\dag _{\U /\S ,\Q}) ^{d}
\overset{\psi}{\longrightarrow} 
\D ^\dag _{\U /\S ,\Q}
\overset{\phi}{\longrightarrow}
\mathcal{H} ^{\dag r} _{X \cap U} (\O _{\U,\Q}  )
\to 0,
\end{equation}
where 
$\phi (P)= P \cdot (1/t _1\cdots t _r)$, and 
$\psi$ is defined by
\begin{equation}
\psi ( P _1,\dots, P _d) = \sum _{i=1} ^{r} P _i t _i + \sum _{i=r+1} ^d P _i \partial _i.
\end{equation}
\end{enumerate}
\end{prop}

\begin{proof}
See \cite[4.3.4]{Be0}.
\end{proof}

\begin{empt}
\label{empt-resolRsP*}
With the notation \ref{coh-smoothsubsch}, 
suppose, $\U  = \fP $.
For $i = 1,\dots, r$, put 
$X _i:  = V (\overline{t} _i)$, and 
$X _{i _0,\dots, i _{k}}
:= 
X _{i _0}\cup \dots \cup X _{i _k}$
(i.e. $V ( \overline{t} _{i _0} \cdots \overline{t} _{i _k})= X _{i _0,\dots, i _{k}}$).
Then $(\hdag X) ( \O _{\fP,\Q})$
is represented by the complex 
\begin{equation}
\label{coh-smoothsubsch-es1}
\prod _{i=1} ^{d}
\O _{\fP} (\hdag X _i) _{\Q}
\to 
\prod _{i _0 < i _1}
\O _{\fP} (\hdag X _{i _0 i _1}) _{\Q}
\to
\dots
\to 
\O _{\fP} (\hdag X _{1 \dots r}) _{\Q}
\to 0,
\end{equation}
whose first term is at degree $0$.
This yields that
$\R \underline{\Gamma} ^\dag _X \O _{\fP,\Q}$ is represented by the complex 
\begin{equation}
\label{coh-smoothsubsch-es2}
\O _{\fP,\Q}
\to 
\prod _{i=1} ^{d}
\O _{\fP} (\hdag X _i) _{\Q}
\to 
\prod _{i _0 < i _1}
\O _{\fP} (\hdag X _{i _0 i _1}) _{\Q}
\to
\dots
\to 
\O _{\fP} (\hdag X _{1 \dots r}) _{\Q}
\to 0,
\end{equation}
whose first term is at degree $0$.
Using \ref{NCDgencoh}, 
this is how Berthelot checked in \cite[4.3.4]{Be0} that 
$\R \underline{\Gamma} ^\dag _X \O _{\fP,\Q}  \in 
D ^{\mathrm{b}} _{\mathrm{coh}}
(\D ^\dag _{\fP, \Q})$.
\end{empt}

\begin{coro}
\label{coro-trace-upre}
Let $u \colon \X  \hookrightarrow \fP $ be a closed immersion of smooth formal schemes over $\S $.

\begin{enumerate}
\item We have $u ^! (\hdag X) ( \O _{\fP,\Q})=0$, i.e. 
by applying the functor $u ^!$ to the canonical morphism
$\R \underline{\Gamma} ^\dag _X \O _{\fP,\Q}  
\to 
\O _{\fP,\Q}$, we get an isomorphism.

\item We have the canonical isomorphism
$u ^! (\O _{\fP,\Q})
\riso 
(\O _{\X,\Q}) [d _{X/P}]$.
We have the canonical isomorphism 
$\R \underline{\Gamma} ^\dag _X \O _{\fP,\Q}  
\riso 
u _+ u ^! (\O _{\fP,\Q})$ 
making commutative the canonical diagram
\begin{equation}
\label{coro-trace-upre-BK}
\xymatrix @ R=0,3cm {
{\R \underline{\Gamma} ^\dag _X \O _{\fP,\Q}  } 
\ar[r] ^-{\sim}
\ar[d] ^-{}
& 
{u _+ u ^! (\O _{\fP,\Q})} 
\ar[d] ^-{\mathrm{adj}} _-{\ref{Radj-u+flatf}}
\\ 
{\O _{\fP,\Q}} 
\ar@{=}[r] ^-{}
& 
{\O _{\fP,\Q}} 
}
\end{equation}

\end{enumerate}

\end{coro}

\begin{proof}
a) Let $\Y $ be the open formal subscheme 
of $\fP $  complementary to $X$. 
First we check that 
$\R \underline{\Gamma} ^\dag _X \O _{\fP,\Q} |\Y  =0$.
Since this  is local, 
we can suppose with the notation of Proposition \ref{coh-smoothsubsch} that
$\U =\fP $.
In that case, we compute that the restriction to $\Y $ of the map $\psi$ of the exact sequence \ref{exseq-HrZ}
 is surjective. Hence, 
 $\mathcal{H} ^{\dag r } _X (\O _{\fP,\Q}  )  |\Y  =0$
 (or we can also compute the $r$th cohomological space of the complex \ref{coh-smoothsubsch-es2}).

b) Now, let us check that 
$u ^! (\hdag X) ( \O _{\fP,\Q})=0$. 
Since this local, 
we can suppose with the notation of Proposition \ref{coh-smoothsubsch} that
$\U =\fP $, and we use notation \ref{empt-resolRsP*}.
Using the resolution \ref{coh-smoothsubsch-es1} of 
$ (\hdag X) ( \O _{\fP,\Q})$,
we reduce to check 
$ \L u ^* (\hdag X _{i _0 \dots i _k}) ( \O _{\fP,\Q})=0$.
Let $u _i\colon \mathfrak{X} _i\hookrightarrow \fP$ be the closed immersion of formal schemes whose corresponding ideal is generated by
$t _i$. If we choose $i \in \{ i _0 ,\dots, i _k\}$, then the multiplication by
$t _i \colon (\hdag X _{i _0 \dots i _k}) ( \O _{\fP,\Q}) \overset{t _i}{\longrightarrow} (\hdag X _{i _0 \dots i _k}) ( \O _{\fP,\Q})$ is an isomorphism.
Hence, $ \L u _i ^* (\hdag X _{i _0 \dots i _k}) ( \O _{\fP,\Q})=0$.
This yields $ \L u ^* (\hdag X _{i _0 \dots i _k}) ( \O _{\fP,\Q})=0$.

c) Hence, by applying the functor 
$u ^!$ to the exact triangle
\ref{extriangleloc}, 
using part b) of the proof,
we get the isomorphism
$u ^ !\R \underline{\Gamma} ^\dag _X \O _{\fP,\Q} 
\riso 
u ^ ! \O _{\fP,\Q} $.
Since $\R \underline{\Gamma} ^\dag _X \O _{\fP,\Q} $ has his support in $X$, 
we get from Theorem \ref{exact-Berthelot-Kashiwara-full} the first isomorphism
$\R \underline{\Gamma} ^\dag _X \O _{\fP,\Q}
\riso 
u _+ u ^ !\R \underline{\Gamma} ^\dag _X \O _{\fP,\Q} 
\riso 
u _+ u ^ !  \O _{\fP,\Q} $.
The commutativity of \ref{coro-trace-upre-BK} is a consequence of  \ref{rem-exact-Berthelot-Kashiwara-full}.
\end{proof}

\begin{coro}
\label{coro-sp+jdagO}
Let $\fP $ be a smooth formal scheme over $\S $.
Let $X$ be a closed smooth $k$-subvariety of $P$ purely of codimension $r$.
We have the isomorphism of $\mathrm{Coh} (X, \fP/K) $
of the form
$$
\sp _+ (\smash{\O} _{]X[ _{\fP}}) 
\riso 
\mathcal{H} ^{\dag ,r} _X \O _{\fP,\Q} .$$

\end{coro}

\begin{proof}
0) Choose $(\fP  _{\alpha}) _{\alpha \in \Lambda}$ an open affine covering of  $\fP $
and let us use the corresponding notation of \ref{ntnPPalpha} (which are compatible with
that of \ref{ntnPPalpha-withoutdiv}).
Choosing a finer covering if necessarily, 
we can suppose furthermore that 
there exist local coordinates
$t  _{\alpha 1}, \dots, t  _{\alpha d} \in \Gamma (\fP _\alpha ,\O _{\fP})$ 
such that $\X _\alpha= V ( t  _{\alpha 1} ,\dots, t  _{\alpha r})$.

1) Following \ref{coh-smoothsubsch},
$\mathcal{H} ^{\dag ,r} _X \O _{\fP,\Q} 
\in 
\mathrm{Coh} (X, \fP/K)$
and 
$\mathcal{H} ^{\dag ,r} _X \O _{\fP,\Q} 
\riso 
\R \underline{\Gamma} ^\dag _X \O _{\fP,\Q}   [r]$.
By functoriality, we have the commutative diagram
\begin{equation}
\notag
\xymatrix{
{p _2  ^{\alpha \beta !}  u _\beta ^! (\R \underline{\Gamma} ^\dag _{X _\beta} \O _{\fP _\beta,\Q}) [-r]} 
\ar[d] ^-{\sim} _-{\tau}
\ar[r] ^-{}
&
{p _2  ^{\alpha \beta !}  u _\beta ^! (\O _{\fP _\beta,\Q} [-r] )} 
\ar[d] ^-{\sim} _-{\tau}
\\ 
{p _1  ^{\alpha \beta !}  u _\alpha ^! (\R \underline{\Gamma} ^\dag _{X _\alpha} \O _{\fP _\alpha,\Q} [-r])} 
\ar[r] ^-{}
& 
{p _1  ^{\alpha \beta !}  u _\alpha ^! (\O _{\fP _\alpha,\Q} [-r]),} 
}
\end{equation}
where the horizontal morphisms are induced by 
$\R \underline{\Gamma} ^\dag _{X } \O _{\fP ,\Q} 
\to 
\O _{\fP ,\Q} $
and where the vertical isomorphisms are the canonical glueing ones (i.e. of the form \ref{prop-glueiniso-coh1}).
By applying $\mathcal{H} ^{0}$, 
since the glueing isomorphisms induced by $\O _{\fP,\Q}$ are the identity,
we get the commutative diagram
\begin{equation}
\label{tau-HrXOP-diag1}
\xymatrix{
{p _2  ^{\alpha \beta !}  u _\beta ^! (\mathcal{H} ^{\dag ,r} _X \O _{\fP,\Q} |\fP _\beta)} 
\ar[d] ^-{\sim} _-{\tau}
\ar[r] ^-{}
&
{\O _{\X _{\alpha\beta},\Q} } 
\ar@{=}[d] 
\\ 
{p _1  ^{\alpha \beta !}  u _\alpha ^! (\mathcal{H} ^{\dag ,r} _X \O _{\fP,\Q} |\fP _\alpha)} 
\ar[r] ^-{}
& 
{\O _{\X _{\alpha\beta},\Q} ,} 
}
\end{equation}
where we omit indicating $\mathcal{H} ^0$ to simplify notation.
Set 
$\E := \mathcal{H} ^{\dag ,r} _X \O _{\fP,\Q} \in \mathrm{Coh} (X, \fP/K)$,
and
$\E _\alpha : =
\mathcal{H} ^{0} u _{\alpha } ^{!} ( \E | \fP _\alpha)$.
We denote by 
$ \theta _{  \alpha \beta} \ : \  p _2  ^{\alpha \beta !} (\E _{\beta}) \riso p  _1 ^{\alpha \beta !} (\E _{\alpha})$,
the glueing 
$\D ^{\dag} _{\X  _{\alpha \beta},\Q} $-linear
isomorphism (which is equal to the left one of \ref{tau-HrXOP-diag1}).

2) For $i = 1,\dots, r$, put 
$X _{\alpha i}:  = V (\overline{t} _{\alpha i})$, and 
$X _{\alpha i _0,\dots, i _{k}}
:= 
X _{\alpha i _0}\cup \dots \cup X _{\alpha i _k}$.
Consider the diagram 
\begin{equation}
\label{locdesc-Gamma2O}
\xymatrix{
{\O _{\fP _\alpha,\Q}} 
\ar[r] ^-{}
\ar[d] ^-{}
& 
{\prod _{i=1} ^{d}
\O _{\fP _\alpha} (\hdag X _i) _{\Q}} 
\ar[r] ^-{}
\ar[d] ^-{}
& 
{\prod _{i _0 < i _1}
\O _{\fP _\alpha} (\hdag X _{i _0 i _1}) _{\Q}} 
\ar[r] ^-{}
\ar[d] ^-{}
&
{\cdots}
\ar[r] ^-{}
\ar[d] ^-{}
& 
{\O _{\fP _\alpha} (\hdag X _{1 \dots r}) _{\Q}} 
\ar[r] ^-{}
\ar[d] ^-{}
& 
{0} 
\ar[d] ^-{}
\\ 
{\O _{\fP _\alpha,\Q}} 
\ar[r] ^-{}
& 
{0} 
\ar[r] ^-{}
& 
{0} 
\ar[r] ^-{}
&
{\cdots}
\ar[r] ^-{}
& 
{0} 
\ar[r] ^-{}
& 
{0.} 
}
\end{equation}
We denote by $\FF ^\bullet _\alpha$ the complex of the top of \ref{locdesc-Gamma2O}
such that $\FF ^{0} _\alpha =\O _{\fP _\alpha,\Q}$. 
Then,  $\FF ^\bullet _\alpha [r]$ is a left resolution of 
$\E | \fP _\alpha$ by coherent $\D ^{\dag} _{\fP _\alpha} (\hdag T _\alpha) _{\Q} $-modules
which are $\O_{\fP _\alpha} (\hdag T _\alpha) _{\Q} $-flat, 
and the canonical morphism
$\R \underline{\Gamma} ^\dag _{X _\alpha} \O _{\fP _\alpha,\Q} 
\to 
\O _{\fP _\alpha,\Q} $
is represented by 
$\E | \fP _\alpha [-r] \overset{\sim}{\longleftarrow}
\FF ^\bullet _\alpha \to  
\O _{\fP _\alpha,\Q} $, where the first arrow is a quasi-isomorphism and the second one
corresponds to the vertical morphism of complexes of \ref{locdesc-Gamma2O}.
Hence, using \ref{2.1.5Be2-empt}.b),
we get the isomorphism 
$\E '_\alpha \riso \E _\alpha$, where
$\E '_\alpha : =
\mathcal{H} ^{r}
u _{\alpha } ^{*} ( \FF ^\bullet _\alpha)$.
Then, we denote by $ \theta '_{  \alpha \beta} \ : \  p _2  ^{\alpha \beta !} (\E '_{\beta}) \riso p  _1 ^{\alpha \beta !} (\E '_{\alpha})$,
the glueing 
$\D ^{\dag} _{\X  _{\alpha \beta}} (\hdag Z _{\alpha \beta}) _{\Q} $-linear
isomorphism
making commutative the diagram
\begin{equation}
\label{tau-HrXOP-diag2}
\xymatrix{
{p _2  ^{\alpha \beta !} (\E' _{\beta})} 
\ar[r] ^-{\sim}
\ar@{.>}[d] ^-{\sim} _-{ \theta '_{  \alpha \beta} }
& 
{p _2  ^{\alpha \beta !} (\E _{\beta})} 
\ar[d] ^-{\sim} _-{ \theta _{  \alpha \beta} }
\\ 
{p  _1 ^{\alpha \beta !} (\E '_{\alpha})} 
\ar[r] ^-{\sim}
& 
{p  _1 ^{\alpha \beta !} (\E _{\alpha}).} 
}
\end{equation}
Since 
$ \theta _{  \alpha \beta} $
satisfy the cocycle condition, then so are 
$ \theta '_{  \alpha \beta} $ (this is just a matter of writing some commutative cubes).
Hence, 
$u _0 ^! (\E)$ is isomorphic to 
$((\E ' _{\alpha})_{\alpha \in \Lambda},\, (\theta '_{\alpha\beta}) _{\alpha ,\beta \in \Lambda})$.

3) For any $s \not =0$, we have
$u _{\alpha } ^{*} ( \FF ^s _\alpha)
=0$.
Moreover,
$\E '_\alpha =
 u _{\alpha } ^{*} ( \FF ^0 _\alpha)
 =\O _{\X _{\alpha\beta},\Q} $.
Hence, 
by applying the functor $\L u _\alpha ^* $ to 
$\FF ^\bullet _\alpha \to  
\O _{\fP _\alpha,\Q} $
we get the identity
$\O _{\X _\alpha ,\Q} \to \O _{\X _\alpha,\Q}$.
This yields that composing \ref{tau-HrXOP-diag1} with \ref{tau-HrXOP-diag2},
we get a square whose morphisms are the identity of $\O _{\X _{\alpha\beta},\Q}$.
In particular 
$ \theta '_{  \alpha \beta}  $ is the identity of $\O _{\X _{\alpha\beta},\Q}$.
By construction of the functor
$u ^* _0$ (see \ref{com-invMIC-ntn})
this means that
$((\E ' _{\alpha})_{\alpha \in \Lambda},\, (\theta '_{\alpha\beta}) _{\alpha ,\beta \in \Lambda})
= u _0 ^*( \O _{\fP,\Q})$.

4) Using 2) and 3), 
we get
the canonical isomorphism 
\begin{equation}
\label{coro-trace-upre-proof4)}
u ^! _0 (\mathcal{H} ^{\dag ,r} _X \O _{\fP,\Q} )
\riso 
u _0 ^*( \O _{\fP,\Q})
\end{equation}
of 
$\mathrm{MIC} ^{\dag \dag} ( (\X   _\alpha )_{\alpha \in \Lambda},Z/K)$.
Using \ref{spinvim*+} and its notation, 
since 
$\sp _* ( \O _{\fP _K}) = \O _{\fP,\Q}$,
then  we get
$\sp _+( \smash{\O} _{]X[ _{\fP}}) 
=u _{0+} \sp _* u _{0K} ^*( \smash{\O} _{]X[ _{\fP}}) 
\underset{\ref{spinvim*+}}{\riso} 
u _{0+} u _{0} ^* \sp _* (  \O  _{\fP _K}) 
=u _{0+} u _{0} ^*  (  \O  _{\fP,\Q})
\riso 
 u _{0+}  u ^! _0 (\mathcal{H} ^{\dag ,r} _X \O _{\fP,\Q} )
\underset{\ref{prop1}}{\riso} 
\mathcal{H} ^{\dag ,r} _X \O _{\fP,\Q} $. 
\end{proof}

\subsection{Commutation of $\sp _+$  with duality}

\begin{empt}
\label{propspetdualsansfrob721}
With notation \ref{dfnMICdagalphabeta},
let $((E _{\alpha}) _{\alpha \in \Lambda}, (\eta _{\alpha \beta }) _{\alpha, \beta \in \Lambda})
\in 
\mathrm{MIC} ^\dag (Y,  (\X   _\alpha )_{\alpha \in \Lambda}/K)$.
The $j ^{\dag} \O _{\X _{\alpha K}}$-linear dual 
of $E _{\alpha} $ is denoted by
$E _{\alpha} ^\vee : =
\mathcal{H} om  _{ j ^{\dag} \O _{\X _{\alpha K} }}( E ,  j ^{\dag} \O _{\X _{\alpha K}} )$. 
Since the $j ^\dag \O$-linear dual commutes with pullbacks, 
the inverse of the isomorphism
$(\eta _{\alpha \beta}) ^\vee$ is canonically isomorphic to 
$ \eta ^* _{  \alpha \beta} \ : \  p _{2K}  ^{\alpha \beta *} ((E _{\beta}) ^\vee) 
\riso 
p  _{1K} ^{\alpha \beta *} ( (E _{\alpha}) ^\vee)$.
These isomorphisms satisfy the cocycle condition (for more details, see  
 \cite[4.3.1]{caro-construction}).
Hence, we get the dual functor 
$(-)^\vee \colon \mathrm{MIC} ^\dag (Y,  (\X   _\alpha )_{\alpha \in \Lambda}/K)
\to \mathrm{MIC} ^\dag (Y,  (\X   _\alpha )_{\alpha \in \Lambda}/K)$
defined by 
$((E _{\alpha}) _{\alpha \in \Lambda}, (\eta _{\alpha \beta }) _{\alpha, \beta \in \Lambda}) ^\vee
:=
(( (E _{\alpha}) ^\vee) _{\alpha \in \Lambda}, (\eta ^*_{\alpha \beta }) _{\alpha, \beta \in \Lambda})$.

Let $E \in \mathrm{MIC} ^{\dag} (Y, X,\fP/K)$,
and  $E ^\vee$ be its dual.
We check the isomorphism
$$u _{0K} ^* (E ^\vee) \riso 
(u _{0K} ^* (E )  ) ^\vee,$$ i.e. that the isomorphisms coming from the commutation of the dual with the pullbacks 
are compatible with glueing data, which is easy.

\end{empt}

\begin{empt}
Let $f \colon \X ^{\prime } \to \X $ be an open immersion of 
smooth formal schemes over $\S $.
Let $Z$ be a divisor of $X$ and $Z ':= f ^{-1} (Z)$.
Let $\E \in D ^{\mathrm{b}} _{\mathrm{coh}} ( \D ^\dag _{\X} (\hdag Z) _{\Q})$.
Following \cite[3.2.8]{caro-construction}, we define the following isomorphism
\begin{gather}\notag
\xi \ :\   f ^!  \DD   (\E)
\riso
\R \mathcal{H} om _{\D ^\dag _{\X^{\prime }} (\hdag Z') _{\Q}}
(f ^{!} ( \E),
f ^! _{\mathrm{r}}   (\D ^\dag _{\X} (\hdag Z) _{\Q}  \otimes _{\O _{\X}} \omega _{\X, \Q} ^{-1}))[d _X]
\\
\label{defDf!=f!D1bis}
\riso
\R \mathcal{H} om _{\D ^\dag _{\X^{\prime }} (\hdag Z') _{\Q}}
(f ^{!}  ( \E),
(\D ^\dag _{\X^{\prime }} (\hdag Z') _{\Q} \otimes  _{\O _{\X'}} \omega _{\X ^{\prime }/\S } ^{-1} ) _\mathrm{t})[d _X]
\underset{\beta}{\riso}
\DD   f ^!  (\E),
\end{gather}
where $\beta$ is the transposition isomorphism exchanging both structures of left 
$\D ^\dag _{\X^{\prime }} (\hdag Z') _{\Q} $-modules of 
$\D ^\dag _{\X^{\prime }} (\hdag Z') _{\Q} \otimes  _{\O _{\X'}} \omega _{\X ^{\prime }/\S } ^{-1} $.
\end{empt}

\begin{empt}
With notation \ref{ntnPPalpha},
let 
$((\E _{\alpha}) _{\alpha \in \Lambda}, (\theta _{\alpha \beta }) _{\alpha, \beta \in \Lambda})\in 
\mathrm{MIC} ^{\dag \dag} ( (\X   _\alpha )_{\alpha \in \Lambda},Z/K)$.
Via the isomorphisms 
\ref{defDf!=f!D1bis}, the inverse of the isomorphism
$\DD (\theta _{\alpha \beta})$ is canonically isomorphic to 
$ \theta ^* _{  \alpha \beta} \ : \  p _2  ^{\alpha \beta !} (\DD (\E _{\beta})) \riso p  _1 ^{\alpha \beta !} (\DD (\E _{\alpha}))$.
These isomorphisms satisfy the cocycle condition 
(for more details, see  
 \cite[4.3.1]{caro-construction}).
Hence, we get the dual functor 
$$\DD \colon \mathrm{MIC} ^{\dag \dag} ( (\X   _\alpha )_{\alpha \in \Lambda},Z/K)
\to \mathrm{MIC} ^{\dag \dag} ( (\X   _\alpha )_{\alpha \in \Lambda},Z/K)$$
defined by 
$\DD ((\E _{\alpha}) _{\alpha \in \Lambda}, (\theta _{\alpha \beta }) _{\alpha, \beta \in \Lambda})
:=
((\DD (\E _{\alpha})) _{\alpha \in \Lambda}, (\theta ^*_{\alpha \beta }) _{\alpha, \beta \in \Lambda})$.

\end{empt}

\begin{empt}
\label{propspetdualsansfrob724}
With notation \ref{ntnPPalpha},
let $((E _{\alpha}) _{\alpha \in \Lambda}, (\eta _{\alpha \beta }) _{\alpha, \beta \in \Lambda})
\in 
\mathrm{MIC} ^\dag (Y,  (\X   _\alpha )_{\alpha \in \Lambda}/K)$.
Following \ref{theoprincipal}, we have the canonical isomorphism
$\mathrm{sp} _* ( E _\alpha ^{\vee}) \riso \DD (\mathrm{sp} _* (E _\alpha))$
of 
$\mathrm{MIC} ^{\dag \dag} ( (\X   _\alpha )_{\alpha \in \Lambda},Z/K)$.
These isomorphisms satisfy the cocycle condition 
(for more details, see  
 \cite[4.3.1]{caro-construction}).
Hence, we get the isomorphism 
$$\mathrm{sp} _* 
(((E _{\alpha}) _{\alpha \in \Lambda}, (\eta _{\alpha \beta }) _{\alpha, \beta \in \Lambda}) ^\vee )
\riso
\DD \circ \mathrm{sp} _*
((E _{\alpha}) _{\alpha \in \Lambda}, (\eta _{\alpha \beta }) _{\alpha, \beta \in \Lambda}).$$
\end{empt}

\begin{empt}
\label{propspetdualsansfrob725}
With notation \ref{ntnPPalpha},
let 
$((\E _{\alpha}) _{\alpha \in \Lambda}, (\theta _{\alpha \beta }) _{\alpha, \beta \in \Lambda})\in 
\mathrm{MIC} ^{\dag \dag} ( (\X   _\alpha )_{\alpha \in \Lambda},Z/K)$.
From the relative duality isomorphism (see \ref{dualrelative}), 
we have the isomorphism
$u _{\alpha +} \circ \DD (\E _\alpha) 
\riso 
\DD \circ u _{\alpha +}  (\E _\alpha) $
These isomorphisms satisfy the cocycle condition 
(for more details, see  
 \cite[4.3.1]{caro-construction}), i.e.
we get the commutation isomorphism :
$$u _{0+} \circ \DD ((\E _{\alpha}) _{\alpha \in \Lambda}, (\theta _{\alpha \beta }) _{\alpha, \beta \in \Lambda}) 
\riso \DD \circ u _{0+} ((\E _{\alpha}) _{\alpha \in \Lambda}, (\theta _{\alpha \beta }) _{\alpha, \beta \in \Lambda})).$$
\end{empt}

\begin{prop}
\label{propspetdualsansfrob}
We keep notation \ref{ntnsp+}.
Let $E \in \mathrm{MIC} ^{\dag} (Y, X,\fP/K)$,
and  $E ^\vee$ be its dual. We have the functorial canonical isomorphism
in $E$ : $ \sp _{+} (E ^\vee ) \riso
\DD  \circ \sp _{+} (E )$.
\end{prop}
\begin{proof}
Since $\sp _{+} 
=
u _{0+} \circ \sp _{*} \circ u _{0K} ^*$, 
the proposition is a consequence of 
\ref{propspetdualsansfrob721},
\ref{propspetdualsansfrob724}, 
and
\ref{propspetdualsansfrob725}.
\end{proof}

\subsection{Differential coherence result}

\begin{lem}
\label{lem-desc-coh-chgbase}
Let $\V \to \V'$ be a finite morphism of complete discrete valuation rings of mixed characteristics $(0,p)$.
We get the finite morphism $\S ' := \Spf \V ' \to \S$.
Let $\X $ be a smooth formal scheme over $\S $, 
$\X ' := \X \times _{\S} \S '$, 
and $f \colon \X' \to \X$ be the canonical projection.
Let $Z$ be a divisor of $X$ and $Z':= f ^{-1} (Z)$.
\begin{enumerate}
\item  
\label{lem-desc-coh-chgbase0}
The canonical 
homomorphism 
$\D ^\dag _{\X'/\S '} (\hdag Z') _{\Q} 
\to 
\D ^\dag _{\X'\to \X/\S '\to \S} (\hdag Z') _{\Q} $
is an isomorphism. 
The composite morphism
$f ^{-1}\D ^\dag _{\X/\S} (\hdag Z) _{\Q} 
\to 
\D ^\dag _{\X'\to \X/\S '\to \S} (\hdag Z') _{\Q} 
\liso
\D ^\dag _{\X'/\S '} (\hdag Z') _{\Q} $
is a homomorphism of rings. 
Hence, if $\E$ is a coherent $\D ^\dag _{\X/\S} (\hdag Z) _{\Q}$-module, then
$f _Z ^! (\E) \riso \D ^\dag _{\X'/\S '} (\hdag Z') _{\Q} \otimes _{f ^{-1}\D ^\dag _{\X/\S} (\hdag Z) _{\Q}}
f ^{-1}\E$, where $f ^! _Z$ is the extraordinary inverse image of $\X' \to \X$ above $\S' \to \S$
 with overconvergent singularities along $Z$, i.e. $f ^! _Z$ is the base change inverse image.

\item 
\label{lem-desc-coh-chgbase1}
Suppose $\X$ is affine.
Let $\E$ be a coherent $\D ^\dag _{\X/\S} (\hdag Z) _{\Q}$-module.
 Then the canonical morphisms
$$ \V ' \otimes _\V \Gamma (\X, \E) 
\to 
D ^\dag _{\X'/\S '} (\hdag Z') _{\Q} 
\otimes _{D ^\dag _{\X/\S} (\hdag Z) _{\Q}}
\Gamma (\X, \E) 
\to 
\Gamma (\X', f _Z ^! (\E)) $$
 are isomorphisms.
 Moreover, 
$D ^\dag _{\X'/\S'} (\hdag Z') _{\Q}$ is a faithfully flat 
$D ^\dag _{\X/\S} (\hdag Z) _{\Q}$-module for both left or right structure.

\item 
\label{lem-desc-coh-chgbase2}
For any  $\D ^\dag _{\X/\S} (\hdag Z) _{\Q}$-module $\E$,
the canonical morphisms
\begin{equation}
\notag
f ^* (\E) := \O _{\X'} \otimes _{f ^{-1}\O _{\X}}
f ^{-1}\E 
\to
\O _{\X'} (\hdag Z') _{\Q} \otimes _{f ^{-1}\O _{\X} (\hdag Z) _{\Q}}
f ^{-1}\E 
\to
\D ^\dag _{\X'/\S '} (\hdag Z') _{\Q} \otimes _{f ^{-1}\D ^\dag _{\X/\S} (\hdag Z) _{\Q}}
f ^{-1}\E 
\end{equation}
are isomorphisms.

 \item 
 \label{lem-desc-coh-chgbase3}
 Let $\phi \colon \E' \to \E$ be a morphism of $\O _{\X}$-modules.
 Then $\phi$ is an isomorphism if and only if 
$f ^* (\phi)$ 
is an isomorphism.
\end{enumerate}

\end{lem}

\begin{proof}
1) Let us prove the first assertion. 
Since this is local, we can suppose $\X$ is affine 
and $\X$ has local coordinates $t _1,\dots, t _d$ over $\S$.
They induce local coordinates $t ' _1,\dots, t ' _d$ of $\X' /\S'$.
We denote by 
$\underline{\partial} ^{[\underline{k}]}$
(resp. $\underline{\partial} ^{\prime [\underline{k}]}$) 
the corresponding $\O _{\X,\Q}$ basis
(resp. $\O _{\X',\Q}$ basis)
of $\D _{\X/\S, \Q}$
(resp. $\D _{\X'/\S', \Q}$).
We compute the canonical homomorphism
$\V ' \otimes _{\V} D ^\dag _{\X/\S} (\hdag Z) _{\Q}
\to D ^\dag _{\X'\to \X/\S '\to \S} (\hdag Z') _{\Q}$ is an isomorphism.
The composition morphism
$D ^\dag _{\X'/\S '} (\hdag Z') _{\Q} 
\to 
D ^\dag _{\X'\to \X/\S '\to \S} (\hdag Z') _{\Q} 
\liso 
\V ' \otimes _{\V} D ^\dag _{\X/\S} (\hdag Z) _{\Q}$
is the morphism sending
$\underline{\partial} ^{\prime [\underline{k}]}$
to
$1 \otimes \underline{\partial} ^{[\underline{k}]}$,
which is also an isomorphism.
Composing the inverse of this isomorphism of left $D ^\dag _{\X'/\S'} (\hdag Z') _{\Q}$-modules
with
the canonical homomorphism of rings
$D ^\dag _{\X/\S} (\hdag Z) _{\Q}
\to 
\V ' \otimes _{\V} D ^\dag _{\X/\S} (\hdag Z) _{\Q}$
given by 
$P \mapsto 1 \otimes P$,
we get the 
homomorphism
$D ^\dag _{\X/\S} (\hdag Z) _{\Q}
\to 
D ^\dag _{\X'/\S'} (\hdag Z') _{\Q}$
sending
$\underline{\partial} ^{\prime [\underline{k}]}$
to
$\underline{\partial} ^{[\underline{k}]}$.
We check easily this is a homomorphism of rings
 (similarly to \cite[2.2.2]{Beintro2}).

2)
 a) Since  $\E$ is a coherent $\D ^\dag _{\X/\S} (\hdag Z) _{\Q}$-module
and $f _Z ^! (\E)$ is a coherent $\D ^\dag _{\X'/\S} (\hdag Z') _{\Q}$,
then via the corresponding theorems of type $A$ the canonical morphisms
$\D ^\dag _{\X/\S} (\hdag Z) _{\Q}
\otimes _{D ^\dag _{\X/\S} (\hdag Z) _{\Q}}
\Gamma (\X, \E) \to \E$,
and
$\D ^\dag _{\X'/\S'} (\hdag Z') _{\Q}
\otimes _{D ^\dag _{\X'/\S'} (\hdag Z') _{\Q}}
\Gamma (\X', f _Z ^! (\E))
\to 
f _Z ^! (\E)$
are isomorphisms. 
This yields by associativity of tensor products that 
the canonical morphism
$D ^\dag _{\X'/\S '} (\hdag Z') _{\Q} 
\otimes _{D ^\dag _{\X/\S} (\hdag Z) _{\Q}}
\Gamma (\X, \E) 
\to 
\Gamma (\X', f _Z ^! (\E)) $
is an isomorphism. 

b) By associativity of tensor products, to check the first isomorphism,
we reduce to the case where $\E=\D ^\dag _{\X/\S} (\hdag Z) _{\Q}$, which has already been checked.
Since the homomorphism $\V \to \V'$ is faithfully flat (e.g. use \cite[I.3.5, Proposition 9 and III.5.2, Theorem 1]{bourbaki3-4}),
then 
$D ^\dag _{\X'/\S'} (\hdag Z') _{\Q}$ is a faithfully flat 
$D ^\dag _{\X/\S} (\hdag Z) _{\Q}$-module for both left or right structure.

3) By associativity of tensor products, 
to check the third statement we reduce to the case where
$\E=\D ^\dag _{\X/\S} (\hdag Z) _{\Q}$, which is easy.

4) Let us prove the forth assertion. Since $\V '$ is a finite faithfully flat $\V$-algebra, since $\V$ and $\V'$ are complete for the $p$-adic topology
then $\V'$ is a free $\V$-module of finite type (e.g. use \cite[II.3.2, Proposition 5]{bourbaki3-4}). 
Let $\U$ be an affine open subset of $\X$
and $\U' := f ^{-1} (\U)$. 
Then,
$\Gamma( \U',\O _{\X'})$ is 
also a free $\Gamma( \U,\O _{\X}) $-module of finite type, and we can conclude. 
\end{proof}

\begin{prop}
\label{desc-coh-chgbase}
With notation \ref{lem-desc-coh-chgbase}, let
 $\E$ be a   $\D ^\dag _{\X/\S} (\hdag Z) _{\Q}$-coherent module.
Then $\E$ is a coherent $\D ^\dag _{\X/\S, \Q}$-module if and only if 
$f _Z ^! (\E) $ is a coherent $\D ^\dag _{\X'/\S ', \Q}$-module.  
\end{prop}

\begin{proof}
1) Using Lemma \ref{lem-desc-coh-chgbase}.\ref{lem-desc-coh-chgbase2},
we can check that the canonical morphism 
$$
\D ^\dag _{\X'/\S ',\Q} \otimes _{f ^{-1}\D ^\dag _{\X/\S,\Q}}
f ^{-1}\E 
\to 
\D ^\dag _{\X'/\S '} (\hdag Z') _{\Q} \otimes _{f ^{-1}\D ^\dag _{\X/\S} (\hdag Z) _{\Q}}
f ^{-1}\E 
=
f _Z ^! (\E)$$
is an isomorphism.
If $\E$ is also a coherent $\D ^\dag _{\X/\S, \Q}$-module, 
this implies that $f _Z ^! (\E)$ is 
$\D ^\dag _{\X'/\S ', \Q}$-coherent.

2) Conversely, suppose 
$f _Z ^! (\E)$ is a coherent $\D ^\dag _{\X'/\S ', \Q}$-module.  

a) Since the 
 $\D ^\dag _{\X/\S, \Q}$-coherence of $\E$ is local in $\X$, we can suppose $\X$ is affine.
Using theorem of type $A$, this yields that
$\Gamma (\X', f _Z ^! (\E)) $
is  
a $D ^\dag _{\X'/\S ', \Q}$-module of finite presentation.
Set 
$E: =\Gamma (\X, \E)$. 
Following \ref{lem-desc-coh-chgbase}.\ref{lem-desc-coh-chgbase1},
this yields that 
$D ^\dag _{\X'/\S '} (\hdag Z') _{\Q} 
\otimes _{D ^\dag _{\X/\S} (\hdag Z) _{\Q}}
E$
is a $D ^\dag _{\X'/\S',\Q}$-module of finite presentation.
Using again \ref{lem-desc-coh-chgbase}.\ref{lem-desc-coh-chgbase1}, 
we get both isomorphism
$\V ' \otimes _{\V} D ^\dag _{\X/\S} (\hdag Z) _{\Q}
\riso
D ^\dag _{\X'/\S '} (\hdag Z') _{\Q} $, 
and $\V ' \otimes _{\V} D ^\dag _{\X/\S,\Q}
\riso
D ^\dag _{\X'/\S ',\Q} $.
Hence, the canonical morphisms
$\V ' \otimes _\V E
\to 
D ^\dag _{\X'/\S ',\Q} 
\otimes _{D ^\dag _{\X/\S,\Q}}
E
\to 
D ^\dag _{\X'/\S '} (\hdag Z') _{\Q} 
\otimes _{D ^\dag _{\X/\S} (\hdag Z) _{\Q}}
E$
are isomorphisms. 
This yields that 
$D ^\dag _{\X'/\S ',\Q} 
\otimes _{D ^\dag _{\X/\S,\Q}}
E$
is  a $D ^\dag _{\X'/\S ', \Q}$-module of finite presentation.
By full faithfulness of 
$D ^\dag _{\X/\S,\Q}
\to 
D ^\dag _{\X'/\S ',\Q} $,
then 
$E$
is 
a $D ^\dag _{\X/\S, \Q}$-module of finite presentation.

c) Let $E ^{\Delta}: = \D ^\dag _{\X/\S, \Q} \otimes_{D ^\dag _{\X/\S, \Q}}E$.
By applying the functor $f ^* = \O _{\X '}
\otimes _{f ^{-1} \O _{\X}}
-$ to the morphism
$E ^{\Delta}
\to \E$, 
we get (up to canonical isomorphisms) 
the homomorphism of coherent 
$\D ^\dag _{\X'/\S', \Q}$-modules
$f ^{!} (E ^{\Delta}) 
\to 
f _Z ^! (\E)$.
Since 
$\Gamma (\X', f ^{!} (E ^{\Delta}) )
\riso \V ' \otimes _\V \Gamma (\X, E ^{\Delta}) 
\riso \V ' \otimes _\V E$
and 
$\Gamma (\X', f _Z ^! (\E))
\riso
\V ' \otimes _\V E$, 
then 
by applying the functor $\Gamma (\X', -)$
to 
$f ^{!} (E ^{\Delta}) 
\to 
f _Z ^! (\E)$, 
we get an isomorphism.
Since $\X '$ is affine, using the theorem of type $A$
satisfied by coherent 
$\D ^\dag _{\X'/\S', \Q}$-modules, 
this yields that 
the morphism $f ^{!} (E ^{\Delta}) 
\to 
f _Z ^! (\E)$
of coherent 
$\D ^\dag _{\X'/\S', \Q}$-modules
is an isomorphism.
Using \ref{lem-desc-coh-chgbase}.\ref{lem-desc-coh-chgbase3}, 
this implies that 
the morphism
$E ^{\Delta}
\to \E$
is an isomorphism.
\end{proof}

\begin{thm}
[Berthelot]
\label{coh-ss-div}
Let $\X $ be a smooth formal scheme over $\S $.
Let $Z$ be a divisor of $X$.
Then 
$\O _{\X} (\hdag Z) _\Q$ is a coherent 
$\D ^\dag _{\X, \Q}$-module.
\end{thm}

\begin{proof}
We can adapt the proof of Berthelot of \cite{Becohdiff} as follows. 
0) Using \cite[8.8.2, 8.10.5]{EGAIV3}
and \cite[17.7.8]{EGAIV4}, it follows from Theorem \cite[4.1]{dejong} 
(see also the construction and \cite[4.5]{dejong} to justify the fact that the divisor 
is a geometric strict normal crossing divisor
and not only a  strict normal crossing divisor),
that there exists a finite extension $k'$ of $k$ satisfying the following property : 
for any irreducible component $\widetilde{X}$ of $X \times _k l$, setting $\widetilde{Z}:= \widetilde{X} \cap (Z \times _k l)$,
there exist a smooth integral $l$-variety $X'$, a projective morphism of $l$-varieties 
$\phi \colon X ' \to \widetilde{X}$ which is generically finite and étale 
such that  $X ' $ is quasi-projective 
and $Z':= \phi  ^{-1} (\widetilde{Z}) $ is a geometric strict normal crossing divisor of $X'/\Spec l$ 
(see definition \ref{sscd}).
Increasing $l$ is necessary, thanks to \ref{sscd-gsncd},
 we can suppose that, locally in $X'$,  
$Z'$ is a strict smooth crossing divisor of $X'/\Spec l$ (see definition \ref{sscd}).

1) Using Lemma \ref{desc-coh-chgbase},
we can suppose $k= l$ and $X$ integral.

2) i) There exists a closed immersion of the form
$u _0 \colon X ' \hookrightarrow \bbP ^n _{X}$
whose composition with the projection 
$\bbP ^n _{X} \to X$ is $\phi$.
Let 
$\fP := \widehat{\bbP} ^n _{\X}$, 
$f \colon \fP \to \X$ be the projection. 
Since $f$ is proper and smooth, we have the adjoint morphism
$ f _{+}  \circ f ^{!} (\O _{\X,\Q})
 \to 
\O _{\X,\Q}$
in 
$D ^{\mathrm{b}} _{\mathrm{coh}}( \smash{\D} ^\dag _{\X ,\Q} )$ (see \ref{adj-morph}).
Following \ref{extriangleloc} and \ref{coh-smoothsubsch}, 
we have in $D ^{\mathrm{b}} _{\mathrm{coh}}( \smash{\D} ^\dag _{\fP /\S ,\Q} )$
the morphism
$\R \underline{\Gamma} ^{\dag} _{X '}  (\O _{\fP,\Q})
\to 
\O _{\fP,\Q}$.
Since 
$f ^{!} (\O _{\X,\Q}) 
\riso 
\O _{\fP,\Q} [n]$, then we get 
the morphism
in 
$D ^{\mathrm{b}} _{\mathrm{coh}}( \smash{\D} ^\dag _{\X ,\Q} )$
\begin{equation}
\label{cstrcf+GammaO->O}
 f _{+} ( \R \underline{\Gamma} ^{\dag} _{X '} \O _{\fP,\Q} [n] )
 \to 
\O _{\X,\Q}.
\end{equation}

ii) In this step, 
we construct 
the morphism
$\O _{\X,\Q}
\to 
f _{+} ( \R \underline{\Gamma} ^{\dag} _{X '} \O _{\fP,\Q} [n])$ as follows:
 we have 
\begin{equation}
\label{RGammadual}
 \DD  (  \R \underline{\Gamma} ^{\dag} _{X '}  \O _{\fP,\Q} [n])
\underset{\ref{coro-sp+jdagO}}{\riso}
\DD  (  \sp _+ (j ^\dag  \smash{\O} _{]X '[ _{\fP}}))
\underset{\ref{propspetdualsansfrob}}{\riso}
\sp _+ ((j ^\dag  \smash{\O} _{]X '[ _{\fP}}) ^\vee)
\riso 
\sp _+ (j ^\dag  \smash{\O} _{]X '[ _{\fP}})
\underset{\ref{coro-sp+jdagO}}{\riso}
 \R \underline{\Gamma} ^{\dag} _{X '}  \O _{\fP,\Q} [n].
\end{equation}
This yields
$$\O _{\X,\Q} 
\underset{\ref{dualisoscvdag}}{\riso}
\DD (\O _{\X,\Q})
\underset{\ref{cstrcf+GammaO->O}}{\longrightarrow} 
\DD f _{+}  ( \R \underline{\Gamma} ^{\dag} _{X '} \O _{\fP,\Q} [n])
\underset{\ref{dualrelative}}{\riso} 
f _{+} \DD  (  \R \underline{\Gamma} ^{\dag} _{X '}  \O _{\fP,\Q} [n])
\underset{\ref{RGammadual}}{\riso} 
f _{+} ( \R \underline{\Gamma} ^{\dag} _{X '} \O _{\fP,\Q} [n]).$$

iii) The composite morphism
$\O _{\X,\Q}
\to 
f _{+} ( \R \underline{\Gamma} ^{\dag} _{X '} \O _{\fP,\Q} [n])
 \to 
\O _{\X,\Q}$
in 
$D ^{\mathrm{b}} _{\mathrm{coh}}( \smash{\D} ^\dag _{\X ,\Q} )$
is an isomorphism.
Indeed, using the third part of Proposition \ref{lem-projff}.\ref{lem-projff-it3}, since this composition is a morphism of 
the abelian category $\mathrm{MIC} ^{\dag\dag} (\X /K) $,
we reduce to check that  its restriction to an open dense subset is an isomorphism.
Hence, we can suppose that 
$\phi  \colon X ' \to X$ 
is the special fiber of a finite and étale 
morphism $g \colon \X ' \to \X$.
By using \ref{coro-trace-upre-BK}, 
we get that the morphism \ref{cstrcf+GammaO->O} 
corresponds to  the trace map 
$g _+ g ^! (\O _{\X,\Q}) 
\to 
\O _{\X,\Q}$
and 
$\O _{\X,\Q}
\to 
g _+ g ^! (\O _{\X,\Q}) $ is induced by duality, i.e.
is the adjunction morphism of $g _! = g_+$ and $g ^! = g ^+$.
This is well known that the composition is an isomorphism.

3) 
Following the step 2), 
$\O _{\X,\Q}$ 
is a direct summand of
$f _{+} ( \R \underline{\Gamma} ^{\dag} _{X '} \O _{\fP,\Q} [n])$
in the category 
$D ^{\mathrm{b}} _{\mathrm{coh}}( \smash{\D} ^\dag _{\X /\S ,\Q} )$.
This yields that 
$\O _{\X} (\hdag Z) _\Q$ 
is a direct summand of 
$(\hdag Z) f _{+} ( \R \underline{\Gamma} ^{\dag} _{X '} \O _{\fP,\Q} [n])$
in the category 
$D ^{\mathrm{b}} _{\mathrm{coh}}( \smash{\D} ^\dag _{\X /\S } (\hdag Z) _{\Q} )$.
Using \ref{surcoh2.1.4-cor},
we get in 
$D ^{\mathrm{b}} _{\mathrm{coh}}( \smash{\D} ^\dag _{\X /\S } (\hdag Z) _{\Q} )$
the following isomorphism
$$(\hdag Z) f _{+} ( \R \underline{\Gamma} ^{\dag} _{X '} \O _{\fP,\Q} [n]) \riso 
f _{Z,+} \circ (\hdag f ^{-1} (Z))  ( \R \underline{\Gamma} ^{\dag} _{X '} \O _{\fP,\Q} [n]).$$
Hence,  thanks to \ref{rema-fct-qcoh2coh},
it is sufficient to
check that this latter object is 
$ \smash{\D} ^\dag _{\X /\S ,\Q} $-coherent.
Since $f$ is proper and since 
$(\hdag f ^{-1} (Z))  ( \R \underline{\Gamma} ^{\dag} _{X '} \O _{\fP,\Q} [n])$ is 
already known to be $ \smash{\D} ^\dag _{\fP /\S} (\hdag f ^{-1} (Z)) _{\Q} $-coherent,
using the remark \ref{oub-div-opcoh}.1,
we reduce to check that 
$(\hdag f ^{-1} (Z))  ( \R \underline{\Gamma} ^{\dag} _{X '} \O _{\fP,\Q} [n])$ is 
$ \smash{\D} ^\dag _{\fP /\S ,\Q} $-coherent.
Since this is local in $\fP $, we can suppose $\fP $ affine. 
Hence, there exists  a morphism 
$u \colon \X '  \to \fP  $
of smooth formal schemes over $\S $
 which is $u _0 \colon X '  \to P $ modulo $\pi$.
 We get 
\begin{gather}
\notag
(\hdag f ^{-1} (Z)) (\R \underline{\Gamma} ^\dag _{X'} \O _{\fP,\Q}   [n])
\underset{\ref{coro-trace-upre}}{\riso} 
(\hdag f ^{-1} (Z)) (u _+ (\O _{\X',\Q}))
\underset{\ref{surcoh2.1.4-cor}}{\riso} 
 u _{f ^{-1} (Z),+} (\O _{\X'}(\hdag \phi ^{-1} (Z)) _{\Q}).
\end{gather}
Since $\phi  ^{-1} (Z )$ is a strict smooth crossing divisor of $X '$ which can in fact be described locally as in  \ref{NCDgencoh}, then
$\O _{\X'}(\hdag \phi ^{-1} (Z)) _{\Q})$ is 
$ \smash{\D} ^\dag _{\X' /\S ,\Q} $-coherent.
Hence, 
using the remark \ref{oub-div-opcoh},
$ u _{f ^{-1} (Z),+} (\O _{\X'}(\hdag \phi ^{-1} (Z)) _{\Q})
\riso
 u _{+} (\O _{\X'}(\hdag \phi ^{-1} (Z)) _{\Q})$
 is 
$ \smash{\D} ^\dag _{\fP/\S ,\Q} $-coherent. 
\end{proof}

\begin{coro}
\label{coh-Bbullet}
With notation \ref{coh-ss-div}, 
we have
$\smash{\widetilde{\B}} _{\X} ^{(\bullet)} (Z ) \in 
\underrightarrow{LM}  _{\Q, \mathrm{coh}} (\smash{\widehat{\D}} _{\X } ^{(\bullet)})
\cap 
\underrightarrow{LM}  _{\Q, \mathrm{coh}} (\smash{\widetilde{\D}} _{\X } ^{(\bullet)} (Z))$.
\end{coro}

\begin{proof}
We already know that
$\smash{\widetilde{\B}} _{\X} ^{(\bullet)} (Z ) 
\in \underrightarrow{LM}  _{\Q, \mathrm{coh}} (\smash{\widetilde{\D}} _{\X } ^{(\bullet)} (Z))$.
Following \ref{coh-ss-div}, 
$\O _{\X} (\hdag Z) _{\Q}=
\underrightarrow{\lim}
\smash{\widetilde{\B}} _{\X} ^{(\bullet)} (Z ) $
is a coherent
$\smash{\D} ^\dag _{\X,\Q}$-module.
Using  \ref{coro1limTouD}, we can conclude.
\end{proof}

We will need later the following proposition.

\begin{prop}
\label{ind-desc-coh-chgbase}
With notation \ref{desc-coh-chgbase},
let 
$\E ^{(\bullet)}
\in 
\smash{\underrightarrow{LD}} ^{\mathrm{b}} _{\Q,\mathrm{coh}} ( \smash{\widehat{\D}} _{\X} ^{(\bullet)})$.
Let 
$ \E ^{\prime (\bullet)}
 :=
  \V' \otimes 
_{\V}  \E ^{(\bullet)}$.
If $(\hdag Z ' )  ( \E ^{\prime (\bullet)})
\in 
\smash{\underrightarrow{LD}} ^{\mathrm{b}} _{\Q,\mathrm{coh}} ( \smash{\widehat{\D}} _{\X '} ^{(\bullet)})$,
then
$(\hdag Z )  ( \E ^{(\bullet)})
\in 
\smash{\underrightarrow{LD}} ^{\mathrm{b}} _{\Q,\mathrm{coh}} ( \smash{\widehat{\D}} _{\X} ^{(\bullet)})$.
\end{prop}

\begin{proof}
Using  \ref{coro1limTouD}, this is a consequence of 
Lemma \ref{desc-coh-chgbase}.
\end{proof}

\subsection{Miscellaneous on base changes}

The purpose of this subsection is to introduce the notion of perfectifications
and special morphisms
(see Definition \ref{def-perfectification}).
For completeness, we add  Proposition \ref{desc-coh-chgbase2}, 
which is useless in this paper but  which 
extends somehow Lemma \ref{desc-coh-chgbase}.

\begin{empt}
\label{ntn-desc-rad}
Let $k \to l$ be an extension of the field $k$.
Following \cite[IX, App, Corollary of Theorem 1]{Bourbaki-AC89} and its terminology,
there exists a ``gonflement'' $\V \to \W ^\heartsuit$ lifting $k \to l$.
Following
\cite[IX, App, Proposition 2 and its Corollary]{Bourbaki-AC89},
$\W ^\heartsuit$ is a faithfully flat $\V$-algebra, $\W ^\heartsuit$ is local, noetherian, regular of dimension $1$ and
its maximal ideal is generated by a uniformizer of $\V$.
Let $\W$ be the $p$-adic completion of $\W ^\heartsuit$. 
Then $\W$ is faithfully flat $\V$-flat, 
is local, noetherian, regular of dimension $1$ and its maximal ideal is generated 
by a uniformizer of $\V$ (e.g. see \cite[VIII.5, Proposition 1 and its Corollary]{Bourbaki-AC89}). 
Hence, 
$\W$ is a 
 $\V$-algebra whose underlying morphism 
 $\V \to \W$   is a morphism of complete discrete valuation rings 
of unequal characteristics $(0,p)$ such that $ \mathfrak{m} _\V \W = \mathfrak{m} _{\W}$.

\end{empt}

\begin{rem}
\label{gonflement-cohen}
With notation \ref{ntn-desc-rad}, suppose $k \to l$ is a separable extension.
Then, following the terminology of
\cite[0.19.8.1]{EGAIV1}, $\W$ is a $\V$-algebra of Cohen.
Following \cite[0.19.8.2.(ii)]{EGAIV1}, such 
$\V$-algebra is unique up to (non unique) isomorphism.
\end{rem}

\begin{lem}
With notation \ref{ntn-desc-rad}, if $l$ is algebraic over $k$ then
$\W$ is integral over $\V$.
\end{lem}

\begin{proof}
Let $x\in \W$, and $\overline{x} $ its image in $\W /\pi \W=l$.
Let $\V [x]$ be the sub $\V$-algebra of $\W$ generated by $x$, 
and $\V \{ x\}$ be the $p$-adic completion of $\V [x]$.
Let $k [\overline{x} ]$ be the sub $k$-algebra of $l$ generated by $\overline{x} $.
Since $l/k$ is algebraic, then  $k [\overline{x} ]$ is a finite $k$-vector space.
Since $\V \{ x\}/ \pi \V \{ x\}\riso k [\overline{x} ]$ 
and since $\V \{ x\}$ is $p$-adically complete, then
$\V \{ x\}$ is a finite $\V$-algebra. 
By noetherianity, $\V [x]$ is a finite $\V$-algebra, i.e. $x$ is integral over $\V$.
\end{proof}

\begin{empt}
\label{special-filtered-pre}
With notation \ref{ntn-desc-rad},
let $L$ be the fraction field of $\W$.
Then, the homomorphism 
$\V \to \W$ induces  the homomorphism 
$K \to L$. If $K '$ a finite sub-extension of $K$ in $L$, then $\V'$, the integral closure of $\V$ in $K'$, 
is a complete discrete valuation such that $\V'$ is a free $\V$ module of rank $[K':K]$ (see \cite[II.2, Proposition 3]{Serre-corpslocaux}).
Conversely, a complete discrete valuation ring $\V'$ which a finite sub $\V$-algebra of $\W$ 
is the integral closure of $\V$ in the fraction field of $\V'$.

\end{empt}

\begin{dfn}
\label{def-perfectification}
With notation \ref{ntn-desc-rad},
when $l$ is perfect and $k \to l$ is algebraic and radicial, we say that 
$\V \to \W$ is a ``perfectification'' of $\V$.
A complete discrete valuation ring $\V'$ which a finite sub $\V$-algebra of $\W$ 
will be called a ``special'' $\V$-algebra (in $\W$) and 
the morphism $\V \to \V'$ will be called a ``special'' morphism of 
complete discrete valuation rings of unequal characteristics  $(0,p)$.
\end{dfn}

\begin{rem}
\label{special-filtered}
Suppose $\V \to \W$ is a perfectification.
Then, using the description of 
special $\V$-algebras of \ref{special-filtered-pre}, 
we check that the preordered (by the inclusion) set of special $\V$-algebras in $\W$ is directed. 
\end{rem}

\begin{prop}
\label{desc-coh-chgbase2}
With notation \ref{ntn-desc-rad}, 
suppose $l$ is algebraic over $k$.
Let $\T := \Spf \W \to \S$ be the corresponding morphism of formal $p$-adic schemes.
Let $\X $ be a smooth formal scheme over $\S $, 
$\Y := \X \times _{\S} \T$, 
and $f \colon \Y \to \X$ be the canonical projection.
Let $Z _X$ be a divisor of $X$ and  
$Z _Y := f ^{-1} (Z _X)$ be the corresponding divisor of $Y$.

The homomorphisms 
$\widehat{\D} ^{(m)} _{\X/\S} ( Z _X) 
\to 
f _* \widehat{\D} ^{(m)} _{\Y/\T} ( Z _Y) $
and 
$\D ^\dag _{\X/\S} (\hdag Z _X) _{\Q}
\to 
f _* \D ^\dag _{\Y/\T} (\hdag Z _Y ) _{\Q}$
are right and left faithfully flat (in the sense of the definition after \cite[Lemma 4.3.8]{Be1}). 
\end{prop}

\begin{proof}
Following \cite[Lemma 4.3.8]{Be1},
it is sufficient to check that
if $\U$ is an open affine formal subscheme of $\X$, 
the homomorphisms 
$\Gamma (\U ,\widehat{\D} ^{(m)} _{\X/\S} ( Z _X) )
\to 
\Gamma (\U , f _* \widehat{\D} ^{(m)} _{\Y/\T} ( Z _Y) )$
and 
$\Gamma (\U , \D ^\dag _{\X/\S} (\hdag Z _X) _{\Q})
\to 
\Gamma (\U , f _* \D ^\dag _{\Y/\T} (\hdag Z _Y ) _{\Q})$
are faithfully flat.
Hence, we reduce to the case where $\X$ is affine. 
First, let us check the first statement. 
Since 
$\widehat{D} ^{(m)} _{\X/\S} ( Z _X) 
\to 
 \W \otimes _{\V}\widehat{D} ^{(m)} _{\X/\S} ( Z _X) $
 is faithfully flat, it is sufficient to prove that 
$ \W \otimes _{\V}\widehat{D} ^{(m)} _{\X/\S} ( Z _X) 
\to
\widehat{D} ^{(m)} _{\Y/\T} ( Z _Y) $
is faithfully flat. 
We remark 
$ \W \widehat{\otimes} _{\V}\widehat{D} ^{(m)} _{\X/\S} ( Z _X) 
\riso
\widehat{D} ^{(m)} _{\Y/\T} ( Z _Y) $.
In other word, we have to check that 
$ \W \otimes _{\V}\widehat{D} ^{(m)} _{\X/\S} ( Z _X) $
is Zariskian for the $p$-adic topology.
Following \cite[13.9]{Isaacs}, 
to check that $ p ( \W \otimes _{\V}\widehat{D} ^{(m)} _{\X/\S} ( Z _X) )$
is included in the Jacobson ideal of 
$\W \otimes _{\V}\widehat{D} ^{(m)} _{\X/\S} ( Z _X) $
it is enough to check that for any 
$P \in \W \otimes _{\V}\widehat{D} ^{(m)} _{\X/\S} ( Z _X) $, 
$1 -\pi  P$ is invertible in 
$\W \otimes _{\V}\widehat{D} ^{(m)} _{\X/\S} ( Z _X) $.
Let $P \in \W \otimes _{\V}\widehat{D} ^{(m)} _{\X/\S} ( Z _X) $.
Since $\V \to \W $ is integral, 
there exists a finite sub $\V$-algebra $\V'$ of $\W$ such that 
$P \in \V ' \otimes _{\V}\widehat{D} ^{(m)} _{\X/\S} ( Z _X) $.
Since $\V \to \V'$ is finite, then 
$\V ' \otimes _{\V}\widehat{D} ^{(m)} _{\X/\S} ( Z _X) = 
\V ' \widehat{\otimes} _{\V}\widehat{D} ^{(m)} _{\X/\S} ( Z _X)$
(e.g. see \cite[3.2.4]{Be1}).
Hence, 
$1-\pi P $ is invertible in 
$\V ' \otimes _{\V}\widehat{D} ^{(m)} _{\X/\S} ( Z _X) $.
Since, 
$\V ' \otimes _{\V}\widehat{D} ^{(m)} _{\X/\S} ( Z _X) 
\subset
\W \otimes _{\V}\widehat{D} ^{(m)} _{\X/\S} ( Z _X) $, 
we conclude.

Following \cite[I.2, Proposition 2 and III.5, Proposition 9.(b)]{Bourbaki-AC17}, 
we check that a filtrante inductive limits of faithfully flat homomorphisms of rings 
is faithfully flat. 
Hence, 
$D ^\dag _{\X/\S} (\hdag Z _X) _{\Q}
\to 
D ^\dag _{\Y/\T} (\hdag Z _Y ) _{\Q}$
is faithfully flat.
\end{proof}

\section{Local cohomological functors}
\subsection{Local cohomological functor with strict support over a divisor}

Let  $\fP $ be a smooth  formal scheme over $\S $.
Let $T$ be a divisor of $P$.
We keep notation of \ref{ntn-tildeD(Z)}.
We have already defined in \ref{hdagT-nota}
the localisation functor $(\hdag T)$ outside $T$. 
In this subsection, we define and study 
the local cohomological functor with support in $T$, which we denote by
$\R \underline{\Gamma} ^\dag _{T} $.

\begin{lemm}
\label{annulationHom-hdag}
\begin{enumerate}
\item 
\label{annulationHom-hdag-item1}
Let 
$\FF ^{(\bullet)}
\to 
\E ^{(\bullet)} 
\to 
(\hdag T) (\E ^{(\bullet)} )  
 \to 
\FF ^{(\bullet)} [1] 
$
be a distinguished triangle of 
$\smash{\underrightarrow{LD}} ^\mathrm{b} _{\Q,\mathrm{qc}} ( \smash{\widehat{\D}} _{\fP /\S } ^{(\bullet)})$
where the second arrow is the canonical morphism. 
For any divisor $T \subset T'$, 
we have the isomorphism 
$(\hdag T') (\FF ^{(\bullet)} )\riso 0$
of 
$\smash{\underrightarrow{LD}} ^\mathrm{b} _{\Q,\mathrm{qc}} ( \smash{\widehat{\D}} _{\fP /\S } ^{(\bullet)})$.

\item 
Let 
$\E ^{(\bullet)}
\in \smash{\underrightarrow{LD}} ^\mathrm{b} _{\Q,\mathrm{qc}} ( \smash{\widehat{\D}} _{\fP /\S } ^{(\bullet)})$
et 
$\FF ^{(\bullet)}
\in \smash{\underrightarrow{LD}} ^\mathrm{b} _{\Q,\mathrm{qc}} ( \smash{\widetilde{\D}} _{\fP /\S } ^{(\bullet) } (T))$.
We suppose we have in 
$\smash{\underrightarrow{LD}} ^\mathrm{b} _{\Q,\mathrm{qc}} ( \smash{\widehat{\D}} _{\fP /\S } ^{(\bullet)})$
the isomorphism
$(\hdag T) (\E ^{(\bullet)} )\riso 0$.
Then
$\mathrm{Hom} _{\smash{\underrightarrow{LD}}  _{\Q} ( \smash{\widehat{\D}} _{\fP /\S } ^{(\bullet)})}
(\E ^{(\bullet)} , \FF ^{(\bullet)}) =0.$

\end{enumerate}

\end{lemm}

\begin{proof}
Using \ref{hdagT'T=cup}, this is checked similarly to
\cite[4.1.2 and 4.1.3]{caro-stab-sys-ind-surcoh}.
\end{proof}

\begin{empt}
\label{exist-RrmHom}
Let  $\mathrm{Ab}$ be the category of abelian groups. 
Similarly to 
\cite[1.4.2]{caro-stab-sys-ind-surcoh}, we construct the bifunctor
(which is the standard construction bifunctor of homomorphims of the abelian category
$\underrightarrow{LM} _{\Q} (\smash{\widetilde{\D}} _{\fP /\S } ^{(\bullet)} (T))$):
$$\mathrm{Hom} ^{\bullet} _{\underrightarrow{LM} _{\Q} (\smash{\widetilde{\D}} _{\fP /\S } ^{(\bullet)} (T))}
(-,-)
\colon 
K 
(\underrightarrow{LM} _{\Q} (\smash{\widetilde{\D}} _{\fP /\S } ^{(\bullet)} (T))) ^\circ 
\times
K 
(\underrightarrow{LM} _{\Q} (\smash{\widetilde{\D}} _{\fP /\S } ^{(\bullet)} (T))) 
\to
K (\mathrm{Ab}).$$
Similarly to 
\cite[1.4.7]{caro-stab-sys-ind-surcoh},
we check that the bifunctor 
$\mathrm{Hom} ^{\bullet} _{\underrightarrow{LM} _{\Q} (\smash{\widetilde{\D}} _{\fP /\S } ^{(\bullet)} (T))} (-,-)$ 
is right localizable. We get the bifunctor 
$$\R \mathrm{Hom} _{D (\underrightarrow{LM} _{\Q} (\smash{\widetilde{\D}} _{\fP /\S } ^{(\bullet)} (T)))} (-,-)
\colon 
D ^{\mathrm{b}}(\underrightarrow{LM} _{\Q} (\smash{\widetilde{\D}} _{\fP /\S } ^{(\bullet)} (T))) ^\circ 
\times
D ^{\mathrm{b}} (\underrightarrow{LM} _{\Q} (\smash{\widetilde{\D}} _{\fP /\S } ^{(\bullet)} (T)))
\to 
D (\mathrm{Ab}).$$
Moreover, we have the isomorphism of bifunctors
$D ^{\mathrm{b}}(\underrightarrow{LM} _{\Q} (\smash{\widetilde{\D}} _{\fP /\S } ^{(\bullet)} (T))) ^\circ 
\times
D ^{\mathrm{b}} (\underrightarrow{LM} _{\Q} (\smash{\widetilde{\D}} _{\fP /\S } ^{(\bullet)} (T)))
\to 
\mathrm{Ab} $ of the form:
\begin{equation}
\label{H0Homrm-DLM}
\mathcal{H} ^{0} (\R \mathrm{Hom} _{D (\underrightarrow{LM} _{\Q} (\smash{\widetilde{\D}} _{\fP /\S } ^{(\bullet)} (T)))} (-,-))
\riso 
\mathrm{Hom} _{D (\underrightarrow{LM} _{\Q} (\smash{\widetilde{\D}} _{\fP /\S } ^{(\bullet)} (T)))} (-,-).
\end{equation}
\end{empt}

\begin{empt}
\label{GammaT}
Let  $T \subset T'$ be a second divisor.
Suppose we have the commutative diagram in 
$\smash{\underrightarrow{LD}} ^\mathrm{b} _{\Q,\mathrm{qc}} ( \smash{\widehat{\D}} _{\fP /\S } ^{(\bullet)})$
of the form
\begin{equation}
\label{prefonct-GammaT}
\xymatrix @=0,4cm{
{\FF ^{(\bullet)}  } 
\ar[r] ^-{}
& 
{\E ^{(\bullet)}  } 
\ar[r] ^-{}
\ar[d] ^-{\phi}
& 
{(\hdag T) (\E ^{(\bullet)} )} 
\ar[d] ^-{(\hdag T)(\phi)}
\ar[r] ^-{}
&
{\FF ^{(\bullet)}  [1]} 
\\ 
{\FF ^{\prime (\bullet)}  } 
\ar[r] ^-{}
& 
{\E ^{\prime (\bullet)}  } 
\ar[r] ^-{}
& 
{(\hdag T) (\E ^{\prime (\bullet)} )} 
\ar[r] ^-{}
&
{\FF ^{\prime (\bullet)}   [1]} 
}
\end{equation}
where middle horizontal morphisms are the canonical ones and
where both horizontal triangles are distinguished. 
Modulo  the equivalence of categories
$\underrightarrow{LD} ^{\mathrm{b}} _{\Q} (\smash{\widetilde{\D}} _{\fP /\S } ^{(\bullet)} (T))
\cong 
D ^{\mathrm{b}}
(\underrightarrow{LM} _{\Q} (\smash{\widetilde{\D}} _{\fP /\S } ^{(\bullet)} (T)))$
(see \ref{eqcatLD=DSM-fonct})
which allows us to see \ref{prefonct-GammaT} as a diagram of 
$D ^{\mathrm{b}}
(\underrightarrow{LM} _{\Q} (\smash{\widetilde{\D}} _{\fP /\S } ^{(\bullet)} (T)))$,
we have
$$H ^{-1} (\R \mathrm{Hom} _{D (\underrightarrow{LM} _{\Q} (\smash{\widetilde{\D}} _{\fP /\S } ^{(\bullet)} (T)))}
( \FF ^{(\bullet)}  ,(\hdag T) (\E ^{\prime (\bullet)} )))
\underset{\ref{H0Homrm-DLM}}{\riso}
\mathrm{Hom} _{D (\underrightarrow{LM} _{\Q} (\smash{\widetilde{\D}} _{\fP /\S } ^{(\bullet)} (T)))}
( \FF ^{(\bullet)}  ,(\hdag T) (\E ^{\prime (\bullet)}) [-1]) 
\underset{\ref{annulationHom-hdag}}{=}0.$$
Following 
\cite[1.1.9]{BBD},
this implies there exists a unique morphism
$\FF ^{(\bullet)}\to \FF ^{\prime (\bullet)} $
making commutative in 
$\smash{\underrightarrow{LD}} ^\mathrm{b} _{\Q,\mathrm{qc}} ( \smash{\widehat{\D}} _{\fP /\S } ^{(\bullet)})$
the diagram:
\begin{equation}
\label{fonct-hdagT}
\xymatrix @=0,4cm{
{\FF ^{(\bullet)}  } 
\ar[r] ^-{}
\ar@{.>}[d] ^-{\exists !}
& 
{\E ^{(\bullet)}  } 
\ar[r] ^-{}
\ar[d] ^-{\phi}
& 
{(\hdag T) (\E ^{(\bullet)} )} 
\ar[d] ^-{(\hdag T)(\phi)}
\ar[r] ^-{}
& 
{\FF ^{(\bullet)}  [1]} 
\ar@{.>}[d] ^-{\exists !}
\\ 
{\FF ^{\prime (\bullet)}  } 
\ar[r] ^-{}
& 
{\E ^{\prime (\bullet)}  } 
\ar[r] ^-{}
& 
{(\hdag T) (\E ^{\prime (\bullet)} )} 
\ar[r] ^-{}
& 
{\FF ^{\prime (\bullet)} [1].} 
}
\end{equation}
Similarly to \cite[1.1.10]{BBD}, 
this implies that the cone of 
$\E ^{(\bullet)}
\to  
(\hdag T) (\E ^{(\bullet)} ) $
is unique up to canonical isomorphism. 
Hence, such a complex $\FF ^{ (\bullet)}$ is unique up to canonical isomorphism.
We denote it by 
$ \R \underline{\Gamma} ^\dag _{T} (\E ^{(\bullet)})$.
Moreover, 
the complex 
$\R \underline{\Gamma} ^\dag _{T} (\E ^{(\bullet)}) $ is functorial in
$\E ^{(\bullet)}$.
\end{empt}

\begin{dfn}
With notation \ref{GammaT}, 
the functor
$\R \underline{\Gamma} ^\dag _{T} 
\colon 
\smash{\underrightarrow{LD}} ^\mathrm{b} _{\Q,\mathrm{qc}} ( \smash{\widehat{\D}} _{\fP /\S } ^{(\bullet)})
\to 
\smash{\underrightarrow{LD}} ^\mathrm{b} _{\Q,\mathrm{qc}} ( \smash{\widehat{\D}} _{\fP /\S } ^{(\bullet)})
$
is 
the ``local cohomological functor with strict support over the divisor $T$''. 
For $\E ^{(\bullet)} \in 
\smash{\underrightarrow{LD}} ^\mathrm{b} _{\Q,\mathrm{qc}} ( \smash{\widehat{\D}} _{\fP /\S } ^{(\bullet)})$, 
we denote by $\Delta _{T} (\E ^{(\bullet)})$  the canonical exact triangle 
\begin{equation}
\label{tri-loc-berthelot}
 \R \underline{\Gamma} ^\dag _{T} (\E ^{(\bullet)})
\to 
\E ^{(\bullet)}
\to 
(\hdag T) (\E ^{(\bullet)})
\to 
 \R \underline{\Gamma} ^\dag _{T} (\E ^{(\bullet)})
 [1].
\end{equation}
\end{dfn}

\begin{lemm}
Let  $T \subset T'$ be a second divisor,
and 
$\E ^{(\bullet)} 
\in 
\smash{\underrightarrow{LD}} ^\mathrm{b} _{\Q,\mathrm{qc}} ( \smash{\widehat{\D}} _{\fP /\S } ^{(\bullet)})$.
There exists a unique morphism
$\R \underline{\Gamma} ^\dag _{T} (\E ^{(\bullet)}) 
\to
\R \underline{\Gamma} ^\dag _{T'} (\E ^{(\bullet)}) $
making commutative the following diagram
\begin{equation}
\label{fonct-hdagX2}
\xymatrix @=0,4cm{
{\R \underline{\Gamma} ^\dag _{T} (\E ^{(\bullet)}) } 
\ar[r] ^-{}
\ar@{.>}[d] ^-{\exists !}
& 
{\E ^{(\bullet)}  } 
\ar[r] ^-{}
\ar@{=}[d] ^-{}
& 
{(\hdag T) (\E ^{(\bullet)} )} 
\ar[d] ^-{}
\ar[r] ^-{}
& 
{\R \underline{\Gamma} ^\dag _{T} (\E ^{(\bullet)})[1]} 
\ar@{.>}[d] ^-{\exists !}
\\ 
{\R \underline{\Gamma} ^\dag _{T'} (\E ^{ (\bullet)})  } 
\ar[r] ^-{}
& 
{\E ^{ (\bullet)}  } 
\ar[r] ^-{}
& 
{(\hdag T') (\E ^{ (\bullet)} )} 
\ar[r] ^-{}
& 
{\R \underline{\Gamma} ^\dag _{T'} (\E ^{ (\bullet)})[1].} 
}
\end{equation}
In other words, 
$\R \underline{\Gamma} ^\dag _{T} (\E ^{(\bullet)}) $
is functorial in $T$.
\end{lemm}

\begin{proof}
We can copy 
\cite[4.1.4.3]{caro-stab-sys-ind-surcoh}.
\end{proof}

\begin{empt}
[Commutation with tensor products]
\label{iso-comm-locaux-prod-tens}
Let  $\E ^{(\bullet)},~\FF ^{(\bullet)}
\in \smash{\underrightarrow{LD}} ^\mathrm{b} _{\Q,\mathrm{qc}} ( \smash{\widehat{\D}} _{\fP /\S } ^{(\bullet)})$.
By commutativity and associativity of tensor products, 
we have the canonical isomorphisms
$(\hdag T) (\E ^{(\bullet)})
\smash{\widehat{\otimes}}^\L  _{\O ^{(\bullet)}  _{\fP} } 
\FF ^{(\bullet)} 
\riso 
(\hdag T) (\E ^{(\bullet)}
\smash{\widehat{\otimes}}^\L  _{\O ^{(\bullet)}  _{\fP} } 
\FF ^{(\bullet)}) 
\riso 
 \E ^{(\bullet)}
\smash{\widehat{\otimes}}^\L  _{\O ^{(\bullet)}  _{\fP} } 
(\hdag T) (\FF ^{(\bullet)}) $.
Hence, there exists a unique isomorphism 
$\R \underline{\Gamma} ^\dag _{T} 
(\E ^{(\bullet)}
\smash{\widehat{\otimes}}^\L  _{\O ^{(\bullet)}  _{\fP} } 
\FF ^{(\bullet)})
\riso 
\R \underline{\Gamma} ^\dag _{T}( \E ^{(\bullet)})
\smash{\widehat{\otimes}}^\L  _{\O ^{(\bullet)}  _{\fP} } 
\FF ^{(\bullet)}$
(resp. $\R \underline{\Gamma} ^\dag _{T} 
(\E ^{(\bullet)}
\smash{\widehat{\otimes}}^\L  _{\O ^{(\bullet)}  _{\fP} } 
\FF ^{(\bullet)})
\riso 
\E ^{(\bullet)}
\smash{\widehat{\otimes}}^\L  _{\O ^{(\bullet)}  _{\fP} } 
\R \underline{\Gamma} ^\dag _{T}
(\FF ^{(\bullet)})$)
making commutative the following diagram
\begin{equation}
\label{fonct-hdagXbis}
\xymatrix @=0,4cm{
{\R \underline{\Gamma} ^\dag _{T} 
(\E ^{(\bullet)})
\smash{\widehat{\otimes}}^\L  _{\O ^{(\bullet)}  _{\fP} } 
\FF ^{(\bullet)}} 
\ar[r] ^-{}
& 
{\E ^{(\bullet)}
\smash{\widehat{\otimes}}^\L  _{\O ^{(\bullet)}  _{\fP} } 
\FF ^{(\bullet)} } 
\ar@{=}[d] ^-{}
\ar[r] ^-{}
& 
{(\hdag T) (\E ^{(\bullet)})
\smash{\widehat{\otimes}}^\L  _{\O ^{(\bullet)}  _{\fP} } 
\FF ^{(\bullet)} } 
\ar[r] ^-{}
& 
{\R \underline{\Gamma} ^\dag _{T} 
(\E ^{(\bullet)} )
\smash{\widehat{\otimes}}^\L  _{\O ^{(\bullet)}  _{\fP} } 
\FF ^{(\bullet)} [1]}
\\
{\R \underline{\Gamma} ^\dag _{T} 
(\E ^{(\bullet)}
\smash{\widehat{\otimes}}^\L  _{\O ^{(\bullet)}  _{\fP} } 
\FF ^{(\bullet)})} 
\ar[r] ^-{}
\ar@{.>}[d] ^-{\exists !}
\ar@{.>}[u] ^-{\exists !}
& 
{\E ^{(\bullet)}
\smash{\widehat{\otimes}}^\L  _{\O ^{(\bullet)}  _{\fP} } 
\FF ^{(\bullet)} } 
\ar@{=}[d] ^-{}
\ar[r] ^-{}
& 
{(\hdag T) (\E ^{(\bullet)}
\smash{\widehat{\otimes}}^\L  _{\O ^{(\bullet)}  _{\fP} } 
\FF ^{(\bullet)})  } 
\ar[d] ^-{\sim}
\ar[r] ^-{}
\ar[u] ^-{\sim}
& 
{\R \underline{\Gamma} ^\dag _{T} 
(\E ^{(\bullet)}
\smash{\widehat{\otimes}}^\L  _{\O ^{(\bullet)}  _{\fP} } 
\FF ^{(\bullet)}) [1]}
\ar@{.>}[d] ^-{\exists !}
\ar@{.>}[u] ^-{\exists !}
\\ 
 {\E ^{(\bullet)}
\smash{\widehat{\otimes}}^\L  _{\O ^{(\bullet)}  _{\fP} } 
\R \underline{\Gamma} ^\dag _{T}
(\FF ^{(\bullet)})}
\ar[r] ^-{}
& 
{\E ^{(\bullet)}
\smash{\widehat{\otimes}}^\L  _{\O ^{(\bullet)}  _{\fP} } 
\FF ^{(\bullet)} } 
\ar[r] ^-{}
& 
{ \E ^{(\bullet)}
\smash{\widehat{\otimes}}^\L  _{\O ^{(\bullet)}  _{\fP} } 
(\hdag T) (\FF ^{(\bullet)})  }\ar[r] ^-{}
& 
{\E ^{(\bullet)}
\smash{\widehat{\otimes}}^\L  _{\O ^{(\bullet)}  _{\fP} } 
\R \underline{\Gamma} ^\dag _{T}
(\FF ^{(\bullet)})[1].} 
}
\end{equation}
Theses isomorphisms are functorial in 
$\E ^{(\bullet)},~\FF ^{(\bullet)},~ T$ (for the meaning of the functoriality in $T$, see  \ref{fonct-hdagX2}).

\end{empt}

\begin{empt}
[Commutation between local cohomological functors and localization functors]
\label{iso-comm-locaux}
Let $T _1, T _2$ be two divisors of  $P$, 
$\E ^{(\bullet)}
\in \smash{\underrightarrow{LD}} ^\mathrm{b} _{\Q,\mathrm{qc}} ( \smash{\widehat{\D}} _{\fP /\S } ^{(\bullet)})$.

\begin{enumerate}[(a)]
\item By commutativity of the tensor product, 
we have the functorial in  $T _1$, $T _2$ and $\E ^{(\bullet)}$ canonical isomorphism
\begin{equation}
\label{commhdagT1T2}
(\hdag T _2) \circ (\hdag T _1) (\E ^{(\bullet)})
\riso 
(\hdag T _1) \circ (\hdag T _2) (\E ^{(\bullet)}).
\end{equation}

\item  There exists a unique  isomorphism
$(\hdag T _2) \circ \R \underline{\Gamma} ^\dag _{T _1}(\E ^{(\bullet)} )
\riso 
\R \underline{\Gamma} ^\dag _{T _1}
\circ
(\hdag T _2)(\E ^{(\bullet)} )$
inducing the canonical morphism of triangles
$ (\hdag T _2) ( \Delta _{T _1} (\E ^{(\bullet)} ))
\to 
\Delta _{T _1} ((\hdag T _2) (\E ^{(\bullet)} ))$ (see \cite[4.2.2.2]{caro-stab-sys-ind-surcoh}).
This isomorphism is functorial in  $T _1$, $T _2$, $\E ^{(\bullet)}$.

\item Similarly there exists a unique isomorphism 
$\R \underline{\Gamma} ^\dag _{T _2} \circ \R \underline{\Gamma} ^\dag _{T _1} (\E ^{(\bullet)})
\riso 
\R \underline{\Gamma} ^\dag _{T _1}\circ \R \underline{\Gamma} ^\dag _{T _2} (\E ^{(\bullet)})$
functorial in $T _1$, $T _2$, $\E ^{(\bullet)}$
and inducing the canonical morphism of triangles 
$\Delta _{T _2} (\R \underline{\Gamma} ^\dag _{T _1} (\E ^{(\bullet)} ))
\to 
\R \underline{\Gamma} ^\dag _{T _1} ( \Delta _{T _2} (\E ^{(\bullet)} ))$.

\end{enumerate}

\end{empt}

\begin{empt}
The three isomorphisms of \ref{iso-comm-locaux} are compatible with that 
of \ref{iso-comm-locaux-prod-tens} (for more precision, see  \cite[4.2.3]{caro-stab-sys-ind-surcoh}).
\end{empt}

We will need the following Lemmas (e.g. see the construction of \ref{dfn-4.3.4} or Proposition \ref{prop2.2.9}) in the next section.
\begin{lemm}
\label{lemme2.2.3}
Let  $D,~T$ be two divisors of $P$, 
$\E ^{(\bullet)}
\in 
\smash{\underrightarrow{LD}} ^{\mathrm{b}} _{\Q,\mathrm{coh}} ( \smash{\widetilde{\D}} _{\fP /\S } ^{(\bullet)} (D))$,
$\U$ be the open subset of  $\fP$ complementary to the support of $T$.
The following assertions are equivalent :
\begin{enumerate}
\item We have in  $\smash{\underrightarrow{LD}} ^{\mathrm{b}} _{\Q,\mathrm{coh}} ( \smash{\widetilde{\D}} _{\U /\S } ^{(\bullet)} (D\cap U))$
the isomorphism $\E ^{(\bullet)}|\U \riso 0$.
\item The canonical morphism 
$\R \underline{\Gamma} ^\dag _{T} (\E ^{(\bullet)})
\to 
\E ^{(\bullet)}$ of $\smash{\underrightarrow{LD}} ^{\mathrm{b}} _{\Q,\mathrm{qc}} ( \smash{\widetilde{\D}} _{\fP /\S } ^{(\bullet)} (D))$
is an isomorphism.
\item We have in $\smash{\underrightarrow{LD}} ^{\mathrm{b}} _{\Q,\mathrm{qc}} ( \smash{\widetilde{\D}} _{\fP /\S } ^{(\bullet)} (D))$
the isomorphism $(\hdag T) (\E ^{(\bullet)} )\riso 0$.
\end{enumerate}
\end{lemm}

\begin{proof}
We can copy the proof of \cite[4.3.2]{caro-stab-sys-ind-surcoh}.
\end{proof}

\begin{lemm}
\label{induction-div-coh}
Let $T _1, \dots, T _{r}$ be some divisors of $P$.
Let $T$ be a divisor of $P$.
Then
$\R \underline{\Gamma} ^\dag _{T _r} \circ \cdots \circ 
\R \underline{\Gamma} ^\dag _{T _1} (\smash{\widetilde{\B}} _{\fP} ^{(\bullet)} (T  ))
\in 
\smash{\underrightarrow{LD}} ^{\mathrm{b}} _{\Q,\mathrm{coh}} ( \smash{\widehat{\D}} _{\fP /\S } ^{(\bullet)} )$.
\end{lemm}

\begin{proof}
We can copy \cite[4.3.3]{caro-stab-sys-ind-surcoh}, i.e. 
this is an easy consequence by devissage of Theorem 
\ref{coh-Bbullet}.
\end{proof}

\begin{cor}
\label{cor-induction-div-coh}
For $r \in \N$, let $T _1, \dots, T _{r}$ be some divisors of $P$ (by convention, $r=0$ means there is no divisors).
Let $T$ be a divisor of $P$.
Then there exists 
a canonical isomorphism
$\R \sp _* 
\left (
\underline{\Gamma} ^\dag _{T _r} 
\circ \cdots \circ 
\underline{\Gamma} ^\dag _{T _1} ( j ^\dag _{T}\O _{\fP _K})
\right )
\riso 
\underrightarrow{\lim}\,
\R \underline{\Gamma} ^\dag _{T _r} 
\circ \cdots \circ 
\R \underline{\Gamma} ^\dag _{T _1} (\smash{\widetilde{\B}} _{\fP} ^{(\bullet)} (T  ))$
which are functorial in $T _i$ and $T$, i.e. making commutative the diagram of 
$D ^{\mathrm{b}} _{\mathrm{coh}}( \smash{\D} ^\dag _{\fP,\Q} )$
\begin{gather}
\notag
\xymatrix{
{\R \sp _* 
\left (
\underline{\Gamma} ^\dag _{T _r} 
\circ \cdots \circ 
\underline{\Gamma} ^\dag _{T _1} ( \O _{\fP _K})
\right )} 
\ar[r] ^-{\sim}
\ar[d] ^-{}
& 
{\underrightarrow{\lim}\,
\R \underline{\Gamma} ^\dag _{T _r} 
\circ \cdots \circ 
\R \underline{\Gamma} ^\dag _{T _1} (\O _{\fP} ^{(\bullet)} )} 
\ar[d] ^-{}
\\ 
{\R \sp _* 
\left (
\underline{\Gamma} ^\dag _{T _r} 
\circ \cdots \circ 
\underline{\Gamma} ^\dag _{T _1} ( j ^\dag _{T}\O _{\fP _K})
\right )} 
\ar[r] ^-{\sim}
\ar[d] ^-{}
& 
{\underrightarrow{\lim}\,
\R \underline{\Gamma} ^\dag _{T _r} 
\circ \cdots \circ 
\R \underline{\Gamma} ^\dag _{T _1} (\smash{\widetilde{\B}} _{\fP} ^{(\bullet)} (T  ))} 
\ar[d] ^-{}
\\
{\R \sp _* 
\left (
\underline{\Gamma} ^\dag _{T _{r-1}} 
\circ \cdots \circ 
\underline{\Gamma} ^\dag _{T _1} ( j ^\dag _{T}\O _{\fP _K})
\right )} 
\ar[r] ^-{\sim}
& 
{\underrightarrow{\lim}\,
\R \underline{\Gamma} ^\dag _{T _{r-1}} 
\circ \cdots \circ 
\R \underline{\Gamma} ^\dag _{T _1} (\smash{\widetilde{\B}} _{\fP} ^{(\bullet)} (T  )),} 
}
\end{gather}
where the vertical arrows are the canonical ones induced by 
$\O _{\fP _K} \to j ^\dag _{T}\O _{\fP _K}$,
$\O _{\fP} ^{(\bullet)} 
\to \smash{\widetilde{\B}} _{\fP} ^{(\bullet)} (T  )$,
$\underline{\Gamma} ^\dag _{T _r}  \to id$, 
$\R \underline{\Gamma} ^\dag _{T _r} \to id$, 
and where $\underrightarrow{\lim} $
is the equivalence of categories
$\underrightarrow{\lim} 
\colon
\underrightarrow{LD} ^{\mathrm{b}} _{\Q, \mathrm{coh}} (\smash{\widehat{\D}} _{\fP /\S } ^{(\bullet)})
\cong
D ^{\mathrm{b}} _{\mathrm{coh}}( \smash{\D} ^\dag _{\fP,\Q} )$
(see \ref{eqcat-limcoh}).
\end{cor}

\begin{proof}
This is checked by induction on $r\in \N$. 
When $r = 0$, this corresponds to the functorial in $T$  isomorphism 
$\R \sp _* (j ^\dag _{T}\O _{\fP _K})
\riso 
\O _{\fP} (\hdag T) _\Q
\riso 
\underrightarrow{\lim}\, \smash{\widetilde{\B}} _{\fP} ^{(\bullet)} (T  )$.
Since
$j ^\dag _{T _r}  j ^\dag _{T}\O _{\fP _K}= j ^\dag _{T _r\cup T}\O _{\fP _K}$, we get the exact sequence
\begin{equation}
0\to 
\underline{\Gamma} ^\dag _{T _r} 
\circ \cdots \circ 
\underline{\Gamma} ^\dag _{T _1} ( j ^\dag _{T}\O _{\fP _K})
\to 
\underline{\Gamma} ^\dag _{T _{r-1}} 
\circ \cdots \circ 
\underline{\Gamma} ^\dag _{T _1} ( j ^\dag _{T}\O _{\fP _K})
\to 
\underline{\Gamma} ^\dag _{T _{r-1}} 
\circ \cdots \circ 
\underline{\Gamma} ^\dag _{T _1} (j ^\dag _{T _r \cup T}\O _{\fP _K})
\to 0.
\end{equation}
Hence, by induction hypothesis, we get a unique 
(using again \cite[1.1.9]{BBD}) 
isomorphism making commutative the diagram
of 
$D ^{\mathrm{b}} _{\mathrm{coh}}( \smash{\D} ^\dag _{\fP,\Q} )$:
\begin{equation}
\xymatrix{
{\R \sp _* 
\underline{\Gamma} ^\dag _{T _r} 
\circ \cdots \circ 
\underline{\Gamma} ^\dag _{T _1} ( j ^\dag _{T}\O _{\fP _K})} 
\ar@{.>}[r] ^-{\sim}
\ar[d] ^-{}
& 
{\underrightarrow{\lim}\,
\R \underline{\Gamma} ^\dag _{T _r} 
\circ \cdots \circ 
\R \underline{\Gamma} ^\dag _{T _1} (\smash{\widetilde{\B}} _{\fP} ^{(\bullet)} (T  ))} 
\ar[d] ^-{}
\\
{\R \sp _* 
\underline{\Gamma} ^\dag _{T _{r-1}} 
\circ \cdots \circ 
\underline{\Gamma} ^\dag _{T _1} ( j ^\dag _{T}\O _{\fP _K})}
\ar[r] ^-{\sim}
\ar[d] ^-{}
& 
{\underrightarrow{\lim}\,
\R \underline{\Gamma} ^\dag _{T _{r-1}} 
\circ \cdots \circ 
\R \underline{\Gamma} ^\dag _{T _1} (\smash{\widetilde{\B}} _{\fP} ^{(\bullet)} (T  ))}
\ar[d] ^-{}
\\
{\R \sp _* 
\underline{\Gamma} ^\dag _{T _{r-1}} 
\circ \cdots \circ 
\underline{\Gamma} ^\dag _{T _1} (j ^\dag _{T _r \cup T}\O _{\fP _K})} 
\ar[r] ^-{\sim}
\ar[d] ^-{}
&
{\underrightarrow{\lim}\,
\R \underline{\Gamma} ^\dag _{T _{r-1}} 
\circ \cdots \circ 
\R \underline{\Gamma} ^\dag _{T _1} (\smash{\widetilde{\B}} _{\fP} ^{(\bullet)} (T _r \cup T  ))}
\ar[d] ^-{}
\\
{\R \sp _* 
\underline{\Gamma} ^\dag _{T _r} 
\circ \cdots \circ 
\underline{\Gamma} ^\dag _{T _1} ( j ^\dag _{T}\O _{\fP _K})[1]} 
\ar@{.>}[r] ^-{\sim}
& 
{\underrightarrow{\lim}\,
\R \underline{\Gamma} ^\dag _{T _r} 
\circ \cdots \circ 
\R \underline{\Gamma} ^\dag _{T _1} (\smash{\widetilde{\B}} _{\fP} ^{(\bullet)} (T  )) [1]} 
}
\end{equation}
where the triangle given by the right vertical arrows is distinguished since
we have
$\smash{\widetilde{\B}} _{\fP} ^{(\bullet)} (T _r \cup T  )
\riso 
(\hdag T _r) ( \smash{\widetilde{\B}} _{\fP} ^{(\bullet)} (T  ))$
(use 
\ref{hdagT'T=cup}).
\end{proof}

\subsection{Local cohomological functor with strict support over a closed subvariety}
Let  $\fP $ be a smooth  formal scheme over $\S $.
\label{localcoho-section}
\begin{dfn}
\label{dfn-4.3.4}
Let $X$ be a (reduced) closed subscheme of $P$.
Similarly to \cite[2.2]{caro_surcoherent}, we define
the local cohomological functor  
$\R \underline{\Gamma} ^\dag _{X}
\colon 
\smash{\underrightarrow{LD}} ^{\mathrm{b}} _{\Q,\mathrm{qc}} 
(\overset{^\mathrm{l}}{} \smash{\widehat{\D}} _{\fP /\S  } ^{(\bullet)} )
\to 
\smash{\underrightarrow{LD}} ^{\mathrm{b}} _{\Q,\mathrm{qc}} 
(\overset{^\mathrm{l}}{} \smash{\widehat{\D}} _{\fP /\S  } ^{(\bullet)} )$
with strict support in $X$ as follows.
Since $P$ is the sum of its irreducible components, 
we reduce to the case where $P$ is integral. 
\begin{enumerate}
\item When $X= P$, the functor $\R \underline{\Gamma} ^\dag _{X}$ 
is by definition the identity.

\item Suppose now $X \not = P$. 
Following \cite[2.2.5]{caro_surcoherent} (there was a typo: we need to add the hypothesis {`` $P$ is integral''})
the underlying space of $X$ is equal to a finite intersection of (the support of some) divisors of $P$.
Choose some divisors
$T _1, \dots, T _{r}$ of $P$ such that
$X = \cap _{i =1} ^{r} T _i$.
For $\E ^{(\bullet)}
\in \smash{\underrightarrow{LD}} ^\mathrm{b} _{\Q,\mathrm{qc}} ( \smash{\widehat{\D}} _{\fP /\S } ^{(\bullet)} )$,
the complex
$\R \underline{\Gamma} ^\dag _{X} (\E ^{(\bullet)} ):=
\R \underline{\Gamma} ^\dag _{T _r} \circ \cdots \circ 
\R \underline{\Gamma} ^\dag _{T _1} (\E ^{(\bullet)} )$
does not depend canonically on the choice of the divisors  $T _1,\dots, T _r$
satisfying $X = \cap _{i =1} ^{r} T _i$.
(Indeed, 
by using 
\ref{iso-comm-locaux-prod-tens}, 
we reduce to the case
$\E ^{(\bullet)} = \widehat{\cO} _{\fP}^{(\bullet)}  $.
Next, using Lemmas \ref{lemme2.2.3} and \ref{induction-div-coh},
it is useless to add divisors containing $ X$.)

\end{enumerate}

\end{dfn}

\begin{prop}
\label{cor-induction-div-coh2}
Let $X$ be a smooth closed subscheme of $P$.
The complex
$\R \underline{\Gamma} ^\dag _X \O _{\fP,\Q} 
:=\R \sp _* \underline{\Gamma} ^\dag _X ( \O _{\fP _K})$
defined at \ref{ntn-GammaZO-rig} 
is canonically isomorphic to 
$\underrightarrow{\lim}\,
\R \underline{\Gamma} ^\dag _{X} 
 (\O _{\fP} ^{(\bullet)} )$, which confirms the compatibility of our notation.

\end{prop}

\begin{proof}
By the construction explained in \ref{dfn-4.3.4}, 
this is a consequence of \ref{cor-induction-div-coh}. 
\end{proof}

\begin{prop}
Let $X,~X'$ be two closed subschemes of  $P$, 
$\E ^{(\bullet)},~\FF ^{(\bullet)}
\in \smash{\underrightarrow{LD}} ^\mathrm{b} _{\Q,\mathrm{qc}} ( \smash{\widehat{\D}} _{\fP /\S } ^{(\bullet)})$.

\begin{enumerate}
\item 
We have the canonical isomorphism functorial in $\E ^{(\bullet)},~ X$, and $X'$ :
\begin{equation}
\label{theo2.2.8}
\R \underline{\Gamma} ^\dag _{X} \circ \R \underline{\Gamma} ^\dag _{X'} 
(\E ^{(\bullet)} )
\riso
\R \underline{\Gamma} ^\dag _{X\cap X'}  (\E ^{(\bullet)} ).
\end{equation}

\item We have the canonical isomorphism functorial in $\E ^{(\bullet)},~\FF ^{(\bullet)},~ X$, and $X'$ :
\begin{equation}
\label{fonctXX'Gamma-iso}
\R \underline{\Gamma} ^\dag _{X \cap X'} (\E ^{(\bullet)}
\smash{\widehat{\otimes}}^\L  _{\O ^{(\bullet)}  _{\fP} }  \FF ^{(\bullet)} )
\riso 
\R \underline{\Gamma} ^\dag _{X} 
(\E ^{(\bullet)})
\smash{\widehat{\otimes}}^\L  _{\O ^{(\bullet)}  _{\fP} }  
\R \underline{\Gamma} ^\dag _{X'}
(\FF ^{(\bullet)}).
\end{equation}
\end{enumerate}

\end{prop}

\begin{proof}
The first statement is obvious by construction of the local cohomological functor with strict support.
We can copy \cite[4.3.6]{caro-stab-sys-ind-surcoh} for the last one.
\end{proof}

\subsection{Localisation outside a closed subscheme functor}
Let  $\fP $ be a smooth  formal scheme over $\S $.

\begin{dfn}
Let $\E ^{(\bullet)}
\in \smash{\underrightarrow{LD}} ^\mathrm{b} _{\Q,\mathrm{qc}} ( \smash{\widehat{\D}} _{\fP /\S } ^{(\bullet)})$.
Let $X$ be a closed subscheme of  $P$.
Using \cite[1.1.10]{BBD} and Lemma \ref{annulationHom-hdag}, 
we check that the cone of the morphism
$\R \underline{\Gamma} ^\dag _{X} (\E ^{(\bullet)} )
\to  
\E ^{(\bullet)} $
is unique up to canonical isomorphism (for more details, 
see \cite[4.4.3]{caro-stab-sys-ind-surcoh}).
We will denote it by 
$(\hdag X) (\E ^{(\bullet)} )$.
We check that 
$(\hdag X) (\E ^{(\bullet)} )$
is functorial in $X$, and $\E ^{(\bullet)} $.
We get by construction the distinguished triangle
\begin{equation}
\label{caro-stab-sys-ind-surcoh4.4.3}
\R \underline{\Gamma} ^\dag _{X} (\E ^{(\bullet)} )
\to 
\E ^{(\bullet)} 
\to 
(\hdag X) (\E ^{(\bullet)} )
 \to 
\R \underline{\Gamma} ^\dag _{X} (\E ^{(\bullet)} )[1] .
\end{equation}
\end{dfn}

\begin{thm}
\label{coh-OGammaY}
Let  $X, X'$ be two closed subschemes of $P$.
We have
$(\hdag X ') \circ \R \underline{\Gamma} ^\dag _{X}  (\O _{\fP}^{(\bullet)}  )
\in 
\smash{\underrightarrow{LD}} ^{\mathrm{b}}  
_{\Q, \mathrm{coh}}
(\overset{^\mathrm{l}}{} \smash{\widehat{\D}} _{\fP /\S  } ^{(\bullet)} )$.
\end{thm}

\begin{proof}
By devissage (use \ref{caro-stab-sys-ind-surcoh4.4.3}),
this is a straightforward consequence of Theorem \ref{induction-div-coh}.
\end{proof}

\begin{empt}
\label{(hdagX)otimes}
For a closed subscheme  $X$ of $P$, 
for
$\E ^{(\bullet)},~\FF ^{(\bullet)}
\in \smash{\underrightarrow{LD}} ^\mathrm{b} _{\Q,\mathrm{qc}} ( \smash{\widehat{\D}} _{\fP /\S } ^{(\bullet)})$,
there exists a canonical isomorphism 
$(\hdag X) (\E ^{(\bullet)}
\smash{\widehat{\otimes}}^\L
_{\O ^{(\bullet)}  _{\fP}} 
\FF ^{(\bullet)}) 
\riso
\E ^{(\bullet)}
\smash{\widehat{\otimes}}^\L
_{\O ^{(\bullet)}  _{\fP}} 
(\hdag X) (\FF ^{(\bullet)}) $, 
which is moreover 
functorial in $X,~\E ^{(\bullet)},~\FF ^{(\bullet)}$ (for more details see 
\cite[4.4.4]{caro-stab-sys-ind-surcoh}).
\end{empt}

\begin{empt}
Let  $X, X'$ be two closed subschemes of $P$,
$\E ^{(\bullet)}
\in \smash{\underrightarrow{LD}} ^\mathrm{b} _{\Q,\mathrm{qc}} ( \smash{\widehat{\D}} _{\fP /\S } ^{(\bullet)})$.
There exists a canonical isomorphism 
$(\hdag X') \circ \R \underline{\Gamma} ^\dag _{X}(\E ^{(\bullet)} )
\riso
\R \underline{\Gamma} ^\dag _{X} \circ (\hdag X') (\E ^{(\bullet)})$
functorial in  $X,~X ',~\E ^{(\bullet)}$.

Similarly to  \cite[2.2.14]{caro_surcoherent}, 
we get the canonical isomorphism
\begin{equation}
\label{2.2.14-surcoh}
(\hdag X) \circ (\hdag X') (\E ^{(\bullet)})
\riso
(\hdag X \cup X')(\E ^{(\bullet)}),
\end{equation}
functorial in $X,~X ',~\E ^{(\bullet)}$.
Similarly to \cite[2.2.16]{caro_surcoherent},
we get the Mayer-Vietoris distinguished triangles :
\begin{gather}
\notag
  \R \underline{\Gamma} ^\dag _{X \cap X'}(\E ^{(\bullet)}) \rightarrow
  \R \underline{\Gamma} ^\dag _{X }(\E ^{(\bullet)}) \oplus
\R \underline{\Gamma} ^\dag _{X ' }(\E ^{(\bullet)})  \rightarrow
\R \underline{\Gamma} ^\dag _{X \cup X '}(\E^{(\bullet)} ) \rightarrow
\R \underline{\Gamma} ^\dag _{X \cap X'}(\E^{(\bullet)} )[1],\\
\label{eq1mayer-vietoris}
 (\hdag X \cap X')(\E ) \rightarrow  (\hdag X )(\E^{(\bullet)} ) \oplus
 (\hdag X ') (\E^{(\bullet)} )  \rightarrow   (\hdag X \cup X ')(\E^{(\bullet)} ) \rightarrow (\hdag X \cap X')(\E ^{(\bullet)})[1].
\end{gather}

\end{empt}

 \begin{prop}
\label{prop2.2.9}
Let  $D$ be a divisor of  $P$, 
$X$ be a closed subscheme of $P$,
$\U$ be the open subset of  $\fP$ complementary to the support of $X$.
Let $\E ^{(\bullet)}
\in 
\smash{\underrightarrow{LD}} ^{\mathrm{b}} _{\Q,\mathrm{coh}} ( \smash{\widetilde{\D}} _{\fP /\S } ^{(\bullet)} (D))$.
The following assertions are equivalent :
\begin{enumerate}
\item We have in  $\smash{\underrightarrow{LD}} ^{\mathrm{b}} _{\Q,\mathrm{coh}} ( \smash{\widetilde{\D}} _{\U /\S } ^{(\bullet)} (D\cap U))$
the isomorphism $\E ^{(\bullet)}|\U \riso 0$.
\item The canonical morphism 
$\R \underline{\Gamma} ^\dag _{X} (\E ^{(\bullet)})
\to 
\E ^{(\bullet)}$ is an isomorphism in 
$\smash{\underrightarrow{LD}} ^{\mathrm{b}} _{\Q} ( \smash{\widetilde{\D}} _{\fP /\S } ^{(\bullet)} (D))$.
\item We have in $\smash{\underrightarrow{LD}} ^{\mathrm{b}} _{\Q,\mathrm{coh}} ( \smash{\widetilde{\D}} _{\fP /\S } ^{(\bullet)} (D))$
the isomorphism $(\hdag X) (\E ^{(\bullet)} )\riso 0$.
\end{enumerate}

\end{prop}

\begin{proof}
Using Lemma \ref{lemme2.2.3}, 
we can copy the proof of 
\cite[4.4.6]{caro-stab-sys-ind-surcoh}.
\end{proof}

\begin{empt}
[Support]
\label{dfn-support}
Let  $D$ be a divisor of  $P$, 
$\E ^{(\bullet)}
\in 
\smash{\underrightarrow{LD}} ^{\mathrm{b}} _{\Q,\mathrm{coh}} ( \smash{\widetilde{\D}} _{\fP /\S } ^{(\bullet)} (D))$.
The support of $\E ^{(\bullet)} $ 
is by definition 
the biggest closed subscheme $X$ of $P$ such that 
$(\hdag X) (\E ^{(\bullet)} )\riso 0$ (one of the equivalent conditions of \ref{prop2.2.9}).

Remark if 
$\E ^{(\bullet)}
\in 
\smash{\underrightarrow{LM}} _{\Q,\mathrm{coh}} ( \smash{\widetilde{\D}} _{\fP /\S } ^{(\bullet)} (D))$, 
then 
this is equal to the  support (for the usual definition) of the coherent 
$\smash{\D} ^\dag _{\fP} (\hdag D) _{\Q} $-module
$\underrightarrow{\lim}\, \E ^{(\bullet)} $, which justifies the terminology.

\end{empt}

\subsection{Local cohomological functor with strict support over a subvariety}
Let  $\fP $ be a smooth  formal scheme over $\S $.

\begin{empt}
\label{3.2.1caro-2006-surcoh-surcv}
Let  $X$, $X'$, $T$, $T'$ be closed subschemes of  $P$ such that
  $X \setminus T = X' \setminus T'$.
For any $\E  ^{(\bullet)} \in \smash{\underset{^{\longrightarrow}}{LD}} ^{\mathrm{b}} _{\Q ,\mathrm{qc}}
(\smash{\widehat{\D}} _{\fP} ^{(\bullet)})$,
we have the canonical isomorphism:
\begin{equation}
  \label{xtx't'}
\R\underline{\Gamma} ^\dag _{X} (\hdag T ) (\E  ^{(\bullet)})
\riso
\R\underline{\Gamma} ^\dag _{X'} (\hdag T ') (\E  ^{(\bullet)}).
\end{equation}
Indeed, $\R\underline{\Gamma} ^\dag _{X} (\hdag T ) (\E  ^{(\bullet)}) \riso
\R\underline{\Gamma} ^\dag _{X} (\hdag T ) (\O _{\fP} ^{(\bullet)})
\smash{\widehat{\otimes}}^\L
_{\O ^{(\bullet)}  _{\fP}} 
\E  ^{(\bullet)}$,
and similarly with some primes.
Hence, we reduce to the case $\E  ^{(\bullet)} = \O _{\fP}   ^{(\bullet)}$.
Using \ref{coh-OGammaY},
\ref{theo2.2.8},
\ref{2.2.14-surcoh},
\ref{prop2.2.9},
we get the isomorphism
$\R\underline{\Gamma} ^\dag _{X} (\hdag T ) (\O _{\fP}   ^{(\bullet)})
\riso
\R\underline{\Gamma} ^\dag _{X \cap X'} (\hdag T \cup T ') (\O _{\fP}   ^{(\bullet)})$.
We conclude by symmetry. 

Setting $ Y := X \setminus T$, we denote by 
$\R\underline{\Gamma} ^\dag _{Y} (\E  ^{(\bullet)}) $ one of both complexes
of \ref{xtx't'}.

\end{empt}

\begin{prop}
\label{prop-induction-div-coh}
We have 
$\R\underline{\Gamma} ^\dag _{Y} (\O _\fP  ^{(\bullet)}) 
\in 
\smash{\underrightarrow{LD}} ^{\mathrm{b}} _{\Q,\mathrm{coh}} ( \smash{\widehat{\D}} _{\fP /\S } ^{(\bullet)} )$.
\end{prop}
\begin{proof}
This is a translation of \ref{coh-OGammaY}.
\end{proof}

\begin{empt}
Let  $Y $ and $Y'$ be two subschemes of  $P$. 
Let $\E  ^{(\bullet)} ,\FF ^{(\bullet)}\in \smash{\underset{^{\longrightarrow}}{LD}} ^{\mathrm{b}} _{\Q ,\mathrm{qc}}
(\smash{\widehat{\D}} _{\fP} ^{(\bullet)})$.
\begin{itemize}
\item [-]
Using \ref{theo2.2.8},
\ref{2.2.14-surcoh},
we get the canonical isomorphism functorial in $\E ^{(\bullet)},~ Y$, and $Y'$ :
\begin{equation}
  \label{gammayY'}
  \R\underline{\Gamma} ^\dag _{Y} \circ \R\underline{\Gamma} ^\dag _{Y'} (\E  ^{(\bullet)})
  \riso
  \R\underline{\Gamma} ^\dag _{Y \cap Y'} (\E  ^{(\bullet)}).
\end{equation}

\item [-] Using \ref{fonctXX'Gamma-iso} and \ref{(hdagX)otimes}
we get the canonical isomorphism functorial in $\E ^{(\bullet)},~\FF ^{(\bullet)},~ Y$, and $Y'$ :
\begin{equation}
\label{fonctYY'Gamma-iso}
\R \underline{\Gamma} ^\dag _{Y \cap Y'} (\E ^{(\bullet)}
\smash{\widehat{\otimes}}^\L  _{\O ^{(\bullet)}  _{\fP} }  \FF ^{(\bullet)} )
\riso 
\R \underline{\Gamma} ^\dag _{Y} 
(\E ^{(\bullet)})
\smash{\widehat{\otimes}}^\L  _{\O ^{(\bullet)}  _{\fP} }  
\R \underline{\Gamma} ^\dag _{Y'}
(\FF ^{(\bullet)}).
\end{equation}
\item [-] If $Y '$ is an open (resp. a closed) subscheme of  $Y$, we have the canonical homomorphism
$\R\underline{\Gamma} ^\dag _{Y} (\E  ^{(\bullet)}) \rightarrow \R\underline{\Gamma} ^\dag _{Y'} (\E  ^{(\bullet)})$
(resp. $\R\underline{\Gamma} ^\dag _{Y'} (\E  ^{(\bullet)}) \rightarrow \R\underline{\Gamma} ^\dag _{Y} (\E  ^{(\bullet)})$).
If $Y '$ is a closed subscheme of  $Y$, we have the localization distinguished triangle
$\R\underline{\Gamma} ^\dag _{Y'} (\E  ^{(\bullet)}) \rightarrow \R\underline{\Gamma} ^\dag _{Y} (\E  ^{(\bullet)})
\rightarrow \R\underline{\Gamma} ^\dag _{Y \setminus Y'} (\E  ^{(\bullet)}) \rightarrow +1.$

\end{itemize}

\end{empt}

\section{Commutation with local cohomological functors and applications}
\subsection{Commutation with local cohomological functors}

\begin{theo}
\label{2.2.18}
Let  $f \colon \X ^{\prime } \to \X $ be a morphism of smooth formal $\V$-schemes.
Let $Y$ be a subscheme of $X$, $Y':= f ^{-1} (Y)$, 
$\E ^{(\bullet)} \in 
\smash{\underrightarrow{LD}} ^{\mathrm{b}} _{\Q,\mathrm{qc}} 
(\overset{^\mathrm{l}}{} \smash{\widehat{\D}} _{\X /\S  } ^{(\bullet)} )$ 
and
$\E ^{\prime (\bullet)} \in\smash{\underrightarrow{LD}} ^{\mathrm{b}} _{\Q,\mathrm{qc}} 
(\overset{^\mathrm{l}}{} \smash{\widehat{\D}} _{\X ^{\prime }/\S  } ^{(\bullet)} )$. 
We have the functorial in $Y$ isomorphisms :
\begin{gather}
\label{commutfonctcohlocal1}
  f ^{ !(\bullet)}  \circ\R \underline{\Gamma} ^\dag _{Y}(\E ^{(\bullet)}) 
  \riso
   \R \underline{\Gamma} ^\dag _{Y' }\circ f ^{ !(\bullet)}  (\E ^{(\bullet)}), 
\\
\label{commutfonctcohlocal2}
\R \underline{\Gamma} ^\dag _{Y}\circ f ^{ (\bullet)} _{+} (\E ^{\prime (\bullet)})
\riso
f ^{ (\bullet)}  _{+} \circ \R \underline{\Gamma} ^\dag _{Y'}(\E ^{\prime (\bullet)}).
\end{gather}

\end{theo}

\begin{proof}
1) Let us check \ref{commutfonctcohlocal1}.
Using \ref{f!T'Totimes} and 
\ref{fonctYY'Gamma-iso}, we reduce to the case where $Y$ is the complement of a divisor $T$ and 
$\E ^{(\bullet)} = \O _{\X} ^{(\bullet)}$.
The morphism $f$ is the composition of its graph
$\gamma _f \colon 
 \X ^{\prime } \hookrightarrow 
  \X ^{\prime } \times  \X ^{ }$
  with the  projection 
  $\X ^{\prime } \times  \X ^{ } \to\X ^{ }$.
  Since the case where $f$ is a flat morphism is known (see \ref{f!commoub}), 
we reduce to the case where $f$ is a closed immersion. 
We conclude by using again \ref{f!commoub} (indeed, either $T \cap X'$ is a divisor and we can use \ref{f!commoub}, or 
  $T\cap X'= X'$ and then the isomorphism \ref{commutfonctcohlocal1} is $0\riso 0$). 
  
2) Let us check that \ref{commutfonctcohlocal2} is a consequence of
\ref{commutfonctcohlocal1}.
\begin{equation}
\R \underline{\Gamma} ^\dag _{Y}\circ f ^{ (\bullet)} _{+} (\E ^{\prime (\bullet)})
\underset{\ref{fonctYY'Gamma-iso}}{\riso}
\R \underline{\Gamma} ^\dag _{Y}(\O ^{(\bullet)}  _{\X})
\smash{\widehat{\otimes}}^\L  _{\O ^{(\bullet)}  _{\X} }
 f ^{ (\bullet)} _{+} (\E ^{\prime (\bullet)})
\underset{\ref{surcoh2.1.4}}{\riso} 
 f ^{ (\bullet)} _{+} 
 (
f  ^{!(\bullet)} ( \R \underline{\Gamma} ^\dag _{Y}(\O ^{(\bullet)}  _{\X}) )
\smash{\widehat{\otimes}}^\L  _{\O ^{(\bullet)}  _{\X'} }
\E ^{\prime (\bullet)}))[-d _{X'/X}]
\end{equation}
Using \ref{commutfonctcohlocal1}, we get
$f  ^{!(\bullet)} ( \R \underline{\Gamma} ^\dag _{Y}(\O ^{(\bullet)}  _{\X}) )
[-d _{X'/X}]
\riso 
   \R \underline{\Gamma} ^\dag _{Y' } \O ^{(\bullet)}  _{\X'}$.
   Hence we are done.
\end{proof}

We can extend Corollary \ref{coro-trace-upre} for quasi-coherent complexes :

\begin{coro}
\label{pre-loc-tri-B-t1T}
Let $u \colon \ZZ  \to \X $ be a closed immersion of smooth formal $\V$-schemes.
For any 
$\E ^{(\bullet)} \in \smash{\underrightarrow{LD}} ^{\mathrm{b}} _{\Q,\mathrm{qc}}(\overset{^\mathrm{l}}{} \smash{\widehat{\D}} _{\X /\S  } ^{(\bullet)} )$,
we have the isomorphism
\begin{equation}
\label{pre-loc-tri-B-t1T-iso}
\R \underline{\Gamma} ^\dag _{Z} (\E ^{(\bullet)}) 
\riso 
u _{+} ^{ (\bullet)} \circ  u ^{ !(\bullet)} (\E ^{(\bullet)}).
\end{equation}
\end{coro}

\begin{proof}
Using \ref{surcoh2.1.4-cor1} and \ref{fonctXX'Gamma-iso}, we reduce to the case where
$\E ^{(\bullet) }= \O _{\X}^{(\bullet)} $.
From Berthelot-Kashiwara's theorem \ref{u!u+=id},
since 
$\R \underline{\Gamma} ^\dag _{Z } (\O _{\X}^{(\bullet)}) $
is coherent with support in $Z $ (see \ref{coh-OGammaY}), we get 
$$u _{+} ^{ (\bullet)}u ^{ !(\bullet)}
\R \underline{\Gamma} ^\dag _{Z } (\O _{\X}^{(\bullet)}) 
\riso 
\R \underline{\Gamma} ^\dag _{Z } (\O _{\X}^{(\bullet)})  .$$
On the other hand,
$$u ^{ !(\bullet)}
\R \underline{\Gamma} ^\dag _{Z } (\O _{\X}^{(\bullet)}) 
\underset{\ref{commutfonctcohlocal1}}{\riso} 
\R \underline{\Gamma} ^\dag _{Z } 
u ^{ !(\bullet)} (\O _{\X}^{(\bullet)} ) 
\riso 
u ^{ !(\bullet)} (\O _{\X}^{(\bullet)} ) .
$$
Hence 
$u _{+} ^{ (\bullet)}u ^{ !(\bullet)}
\R \underline{\Gamma} ^\dag _{Z } (\O _{\X}^{(\bullet)}) 
\riso 
u _{+} ^{ (\bullet)}u ^{ !(\bullet)}
 (\O _{\X}^{(\bullet)} ) $,
 and we are done. 
\end{proof}

\subsection{Coherence of convergent isocrystals, 
inverse images of convergent isocrystals}

\begin{ntn}
\label{ntnMICdag2fs3}
Let  $\fP $ be a smooth  formal scheme over $\S $.
Let $X$ be a smooth closed subscheme of $P$.
We denote by $\mathrm{MIC} ^{(\bullet)} (X, \fP/K) $ the full subcategory of 
$\smash{\underrightarrow{LM}}  _{\Q, \mathrm{coh}}
(\smash{\widehat{\D}} _{\fP /\S } ^{(\bullet)})$
consisting of objects 
$\E ^{(\bullet)}$ with support in $X$ 
and such that 
$\underrightarrow{\lim} (\E ^{(\bullet)})
\in \mathrm{MIC} ^{\dag \dag} (X, \fP/K) $
where
$\underrightarrow{\lim} 
\colon
\smash{\underrightarrow{LM}}  _{\Q, \mathrm{coh}}
(\smash{\widehat{\D}} _{\fP /\S } ^{(\bullet)})
\cong
\mathrm{Coh} ( \smash{\D} ^\dag _{\fP,\Q} )$
is the equivalence of categories
of \ref{M-eq-coh-lim}, 
and where 
$\mathrm{MIC} ^{\dag \dag} (X, \fP/K) $
is defined in \ref{ntnMICdag2fs2}.
When $X=P$, we remove $X$ in the notation so that in this case 
we retrieve Notation \ref{cvisoc-f*}.\ref{cvisoc-f*2}.
\end{ntn}

\begin{prop}
\label{st-isoc-boxtimes}
Let  $\fP $ and $\fQ$ be two smooth  formal schemes over $\S $.
Let $X$ (resp. $Y$) be a smooth closed subscheme of $P$ (resp. $Q$).
Let 
$\E ^{(\bullet)}$ be an object of 
$\mathrm{MIC} ^{(\bullet)} (X, \fP/K) $,
and  
$\cF ^{(\bullet)}$ be an object of 
$\mathrm{MIC} ^{(\bullet)} (Y, \fQ/K) $.
Then 
$\E ^{(\bullet)} \smash{\widehat{\boxtimes}} ^\L _{\O _{\S }} \cF ^{(\bullet)}
\in 
\mathrm{MIC} ^{(\bullet)} (X\times Y, \fP\times \fQ/K) $.
\end{prop}

\begin{proof}
Following Lemma \ref{exact-boxtimes}, we already know
$\E ^{(\bullet)} \smash{\widehat{\boxtimes}} ^\L _{\O _{\S }} \cF ^{(\bullet)}
\in 
\smash{\underrightarrow{LM}}  _{\Q, \mathrm{coh}}
( \smash{\widehat{\D}} _{\fP \times \fQ/\S }  ^{(\bullet)} )$.
Since the proposition is local, 
using 
\ref{prop-boxtimes-v+}, 
we reduce to the case where $X=P$ and $Y=Q$. 
Then this is obvious.
\end{proof}

\begin{thm}
\label{coh-ss-div-bis}
Let  $\fP $ be a smooth  formal scheme over $\S $.
Let $X$ be a smooth closed subscheme of $P$, and $T$ be a divisor of $X$.
Let 
$\E ^{(\bullet)}$ be an object of 
$\mathrm{MIC} ^{(\bullet)} (X, \fP/K) $.
Then 
$(\hdag T ) (\E ^{(\bullet)})
\in 
\smash{\underrightarrow{LM}}  _{\Q, \mathrm{coh}}
(\smash{\widehat{\D}} _{\fP /\S } ^{(\bullet)})$.
\end{thm}

\begin{proof}
1) Using the inductive system version of Berthelot-Kashiwara's theorem
(see \ref{u!u+=id}),
we reduce to the case where $X = P$. 
In this case, we write $\X$ (resp. $Z$) instead of $\fP$
(resp. $T$) and we will use the notation of the proof of \ref{coh-ss-div}.
Now, following the part 0),1) and 2) of the proof of \ref{coh-ss-div},
modulo the equivalence of categories 
\ref{eqcat-limcoh} and the compatibility of \ref{cor-induction-div-coh2}, 
the object
$\O _{\X} ^{(\bullet)}$
is a direct summand of 
$ f _{+} ^{(\bullet)} ( \R \underline{\Gamma} ^{\dag} _{X '} \O _{\fP} ^{(\bullet)} [n] )$
in 
$\smash{\underrightarrow{LD}} ^{\mathrm{b}} _{\Q, \mathrm{coh}}
(\smash{\widehat{\D}} _{\X /\S } ^{(\bullet)})$.
This yields that 
$(\hdag Z ) (\E ^{(\bullet)})$
is a direct summand of 
$$(\hdag Z ) \left (
\E ^{(\bullet)}
\smash{\widehat{\otimes}}
^\L _{\O ^{(\bullet)}  _{\X} }
 f _{+} ^{(\bullet)} ( \R \underline{\Gamma} ^{\dag} _{X '} \O _{\fP} ^{(\bullet)} [n] )
 \right) 
\underset{\ref{surcoh2.1.4-iso}}{\riso}
(\hdag Z )  f _{+} ^{(\bullet)} 
 \left ( 
 f  ^{!(\bullet)} (   \E ^{(\bullet)})
\smash{\widehat{\otimes}}
^\L _{\O ^{(\bullet)}  _{\fP} }
\R \underline{\Gamma} ^{\dag} _{X '} \O _{\fP} ^{(\bullet)} 
\right )
\riso 
 f _{+} ^{(\bullet)} 
\R \underline{\Gamma} ^{\dag} _{X '\setminus Z'}  f  ^{!(\bullet)} (   \E ^{(\bullet)}).$$
Hence, we reduce to check that 
$\R \underline{\Gamma} ^{\dag} _{X '\setminus Z'}  f  ^{!(\bullet)} (   \E ^{(\bullet)})$
is an object of 
$\smash{\underrightarrow{LD}} ^{\mathrm{b}} _{\Q, \mathrm{coh}}
(\smash{\widehat{\D}} _{\fP /\S } ^{(\bullet)})$.
Since this is local in $\fP$, we can suppose there exists a closed immersion of smooth 
$\V$-formal schemes $u \colon \X ' \hookrightarrow \fP$ which reduces modulo $\pi$ to 
$X 'Â \hookrightarrow P$. Following \ref{pre-loc-tri-B-t1T-iso}, 
$\R \underline{\Gamma} ^{\dag} _{X '}  f  ^{!(\bullet)} (   \E ^{(\bullet)})
\riso 
u _{+} ^{ (\bullet)} \circ  u ^{ !(\bullet)} \circ f  ^{!(\bullet)}(\E ^{(\bullet)})$.
Hence, 
$\R \underline{\Gamma} ^{\dag} _{X '\setminus Z'}  f  ^{!(\bullet)} (   \E ^{(\bullet)})
\riso 
u _{+} ^{ (\bullet)} \circ   (\hdag Z') \circ u ^{ !(\bullet)} \circ f  ^{!(\bullet)}(\E ^{(\bullet)})$.
Following \ref{corostab-MIC-f*}, 
we get 
$\E ^{\prime (\bullet)} := u ^{ !(\bullet)} \circ f  ^{!(\bullet)}(\E ^{(\bullet)})
\in \mathrm{MIC} ^{(\bullet)} (\X '/K)$.
Since $u$ is proper, then 
$u _{+} ^{ (\bullet)}$ preserves the coherence. 
Hence, we reduce to check 
that 
$ (\hdag Z') (\E ^{\prime (\bullet)})
\in 
\smash{\underrightarrow{LM}}  _{\Q, \mathrm{coh}}
(\smash{\widehat{\D}} _{\X' /\S } ^{(\bullet)})$.

2) Since this is local, we can suppose $\X'$ is integral and affine. 
We proceed by induction on the number of irreducible component of $Z'$. 
Let $Z ' _1$ be one irreducible component of $Z'$ and $Z''$ be the union of the other irreducible components.
Let $ u _1 \colon \ZZ ' _1 \hookrightarrow \X'$ be a lifting of $Z ' _1 \hookrightarrow X'$,
and $Z ' _2 := Z ' _1 \cap Z ''$.
\begin{equation}
\label{coh-ss-div-bis-extri}
\R \underline{\Gamma} ^{\dag} _{Z' _1}   ( (\hdag Z'')   \E ^{\prime (\bullet)})
\to 
 (\hdag Z'')   (\E ^{\prime (\bullet)})
\to 
 (\hdag Z')   (\E ^{\prime (\bullet)})
 \to +1
 \end{equation}
Following \ref{corostab-MIC-f*},
we get 
$\L u _1 ^{ *(\bullet)}
( \E ^{\prime (\bullet)})
\in 
 \mathrm{MIC} ^{(\bullet)} (\ZZ '_1 /K)$.
 Since 
$\R \underline{\Gamma} ^{\dag} _{Z' _1}   ( (\hdag Z'')   \E ^{\prime (\bullet)}) 
\riso 
u _{1+} ^{ (\bullet)} \circ  u _1 ^{ !(\bullet)}
( (\hdag Z'')   \E ^{\prime (\bullet)})
\riso 
u _{1+} ^{ (\bullet)} \circ (\hdag Z' _2)   
( u _1 ^{ !(\bullet)}
\E ^{\prime (\bullet)})$, 
by induction hypothesis and preservation of the coherence under
$u _{1+} ^{ (\bullet)}$,
this yields 
$\R \underline{\Gamma} ^{\dag} _{Z' _1}   ( (\hdag Z'')   \E ^{\prime (\bullet)}) 
\in 
\smash{\underrightarrow{LD}} ^{\mathrm{b}} _{\Q, \mathrm{coh}}
(\smash{\widehat{\D}} _{\X' /\S } ^{(\bullet)})$.
By induction hypothesis 
$ (\hdag Z'')   (\E ^{\prime (\bullet)})
\in 
\smash{\underrightarrow{LD}} ^{\mathrm{b}} _{\Q, \mathrm{coh}}
(\smash{\widehat{\D}} _{\X' /\S } ^{(\bullet)})$.
We conclude using the exact triangle \ref{coh-ss-div-bis-extri}.
\end{proof}

\begin{prop}
\label{com-sp+-f*}
Consider the following diagram 
\begin{equation}
\label{com-sp+-f*-squ}
\xymatrix  @R=0,3cm {
{Y ' }
\ar[r] ^-{j '}
\ar[d]^b
&
{X ^{\prime }}
\ar[r] ^-{u'}
\ar[d]^a
&
{\fP ^{\prime }}
\ar[d]^f
\\
{Y }
\ar[r]^{j }
&
{X }
\ar[r]^{u }
&
{\fP ,}
}
\end{equation}
where $f$ is a smooth morphism of smooth formal schemes over $\S $, 
$a$ is a morphism of smooth $S $-schemes,
 $u $, $u'$  are closed immersions,
and $j$, $j'$ are open immersions. We suppose there exists a divisor 
$T$ 
of 
$P$ 
such that $Y = X \setminus T$
and $Z:= T\cap X$ is a divisor of $X$
(resp. a divisor 
$T'$ 
of 
$P'$ 
such that $Y' = X '\setminus T'$
and $Z':= T'\cap X'$ is a divisor of $X'$).
We get the  morphism of smooth frames
$\theta :=(b,\,a,\,f)
\colon 
(Y', X ^{\prime },\fP')
\to 
(Y, X,\fP)$.
Let $E \in \mathrm{MIC} ^{\dag} (Y, X,\fP/K)$ (see notation \ref{ntn-realP}).
We have the isomorphism in 
$\mathrm{MIC} ^{\dag \dag} (X', \fP',T'/K) $  (see notation \ref{ntnMICdag2fs2}) :
\begin{equation}
\label{com-sp+-f*iso}
\sp _+ ( \theta ^* (E)) 
\riso 
\R \underline{\Gamma} ^\dag _{Y'} f ^!\sp _+ (E) [-d _{X'/X}].
\end{equation}

\end{prop}

\begin{proof}
1) First suppose the squares of the diagram
\ref{com-sp+-f*-squ} are cartesian.
Let $(\fP  _{\alpha}) _{\alpha \in \Lambda}$ be an open covering of  $\fP $
satisfying the condition of \ref{ntnPPalpha}.
We fix some liftings as in \ref{ntnPPalpha} and we use the same notation.
Moreover, we denote by
$\fP ' _\alpha := f ^{-1} (\fP _\alpha)$,
$\fX ' _\alpha := \fP ' _\alpha   \times _{\fP  _\alpha  }\fX _\alpha$,
$a _\alpha \colon \fX ' _\alpha \to \fX  _\alpha$ the projection, and similarly for other notations.
If $((E _{\alpha})_{\alpha \in \Lambda},\, (\eta _{\alpha\beta}) _{\alpha ,\beta \in \Lambda})$
is an object of 
$\mathrm{MIC} ^\dag (Y,  (\X   _\alpha )_{\alpha \in \Lambda}/K)$, 
we get canonically an object  of 
 $\mathrm{MIC} ^\dag (Y',  (\X '  _\alpha )_{\alpha \in \Lambda}/K)$
 of the form
$((a _{\alpha K } ^* E _{\alpha})_{\alpha \in \Lambda},\, (\eta '_{\alpha\beta}) _{\alpha ,\beta \in \Lambda})$.
This yields the functor 
$
\mathrm{MIC} ^\dag (Y,  (\X   _\alpha )_{\alpha \in \Lambda}/K)
\to 
\mathrm{MIC} ^\dag (Y',  (\X '  _\alpha )_{\alpha \in \Lambda}/K)$
that we will denote by $a ^*  _{K}$.
Similarly, we construct the functor 
$a ^* 
\colon 
 \mathrm{MIC} ^{\dag \dag} ( (\X   _\alpha )_{\alpha \in \Lambda},Z/K)
 \to 
  \mathrm{MIC} ^{\dag \dag} ( (\X  ' _\alpha )_{\alpha \in \Lambda},Z'/K)$.
  Consider the following diagram.
\begin{equation}
\label{com-sp+-f*diag1}
\xymatrix{
{\mathrm{MIC} ^{\dag} (Y, X,\fP/K)} 
\ar[d] ^-{\ref{eqcat-iso-reco}} _-{u ^*  _{0K}}
\ar[r] ^-{f ^* _K}
& 
{\mathrm{MIC} ^{\dag} (Y', X',\fP'/K)} 
\ar[d] ^-{\ref{eqcat-iso-reco}} _-{u ^{\prime *}  _{0K}}
\\ 
{\mathrm{MIC} ^\dag (Y,  (\X   _\alpha )_{\alpha \in \Lambda}/K)} 
\ar[r] ^-{a ^*  _{K}} 
\ar@<4ex>[d] ^-{\ref{lem1pre-sp+plfid}} _-{\sp _*}
& 
{\mathrm{MIC} ^\dag (Y',  (\X '  _\alpha )_{\alpha \in \Lambda}/K)} 
\ar@<4ex>[d] ^-{\ref{lem1pre-sp+plfid}} _-{\sp _*}
\\
{ \mathrm{MIC} ^{\dag \dag} ( (\X   _\alpha )_{\alpha \in \Lambda},Z/K)} 
\ar[r] ^-{a ^*}
\ar@<4ex>[u] ^-{\ref{lem1pre-sp+plfid}} _-{\sp ^*}
\ar@<4ex>[d] ^-{\ref{eqcat-u0+!}} _-{u _{0+}}
& 
{ \mathrm{MIC} ^{\dag \dag} ( (\X  ' _\alpha )_{\alpha \in \Lambda},Z'/K)} 
\ar@<4ex>[u] ^-{\ref{lem1pre-sp+plfid}} _-{\sp ^*}
\ar@<4ex>[d] ^-{\ref{eqcat-u0+!}} _-{u ' _{0+}}
\\
{\mathrm{MIC} ^{\dag \dag} (X, \fP,T/K) }
\ar@<4ex>[u] ^-{\ref{eqcat-u0+!}} _-{u _{0} ^!}
\ar[r] ^-{f ^*}
&
{\mathrm{MIC} ^{\dag \dag} (X', \fP',T'/K) .}
\ar@<4ex>[u] ^-{\ref{eqcat-u0+!}} _-{u _{0} ^{\prime !}}
}
\end{equation}
By transitivity of the inverse image with respect to the composition, 
the top square is commutative up to canonical isomorphism. 
For the same reason, the middle square involving $\sp ^*$ is commutative up to canonical isomorphism. 
Since $\sp _* $ and $\sp ^*$ are canonically quasi-inverse equivalences of categories,
this yields the middle square involving $\sp _*$ is commutative up to canonical isomorphism. 
Using similar arguments, we check the commutativity up to canonical isomorphism of the bottom square. 

2) Now suppose $f = id$ and $a$ is a closed immersion and the left square is cartesian.   
Let $(\fP  _{\alpha}) _{\alpha \in \Lambda}$ be an open covering of  $\fP $.
Then, 
we fix some liftings (separately) for both $u$ and $u'$ (for the later case, add some primes in notation)
and we use notation \ref{ntnPPalpha}.
Then choose some lifting morphisms
$a _\alpha \colon \fX ' _\alpha \to \fX  _\alpha$, 
and similarly for other notations.
Consider the following diagram.
\begin{equation}
\label{com-sp+-f*diag2}
\xymatrix{
{\mathrm{MIC} ^{\dag} (Y, X,\fP/K)} 
\ar[d] ^-{\ref{eqcat-iso-reco}} _-{u ^*  _{0K}}
\ar[r] ^-{| _{]X'[ _{\fP}}}
& 
{\mathrm{MIC} ^{\dag} (Y', X',\fP/K)} 
\ar[d] ^-{\ref{eqcat-iso-reco}} _-{u ^{\prime *}  _{0K}}
\\ 
{\mathrm{MIC} ^\dag (Y,  (\X   _\alpha )_{\alpha \in \Lambda}/K)} 
\ar[r] ^-{a ^*  _{K}} 
\ar@<4ex>[d] ^-{\ref{lem1pre-sp+plfid}} _-{\sp _*}
& 
{\mathrm{MIC} ^\dag (Y',  (\X '  _\alpha )_{\alpha \in \Lambda}/K)} 
\ar@<4ex>[d] ^-{\ref{lem1pre-sp+plfid}} _-{\sp _*}
\\
{ \mathrm{MIC} ^{\dag \dag} ( (\X   _\alpha )_{\alpha \in \Lambda},Z/K)} 
\ar[r] ^-{a ^*}
\ar@<4ex>[u] ^-{\ref{lem1pre-sp+plfid}} _-{\sp ^*}
\ar@<4ex>[d] ^-{\ref{eqcat-u0+!}} _-{u _{0+}}
& 
{ \mathrm{MIC} ^{\dag \dag} ( (\X  ' _\alpha )_{\alpha \in \Lambda},Z'/K)} 
\ar@<4ex>[u] ^-{\ref{lem1pre-sp+plfid}} _-{\sp ^*}
\ar@<4ex>[d] ^-{\ref{eqcat-u0+!}} _-{u ' _{0+}}
\\
{\mathrm{MIC} ^{\dag \dag} (X, \fP,T/K) }
\ar@<4ex>[u] ^-{\ref{eqcat-u0+!}} _-{u _{0} ^!}
\ar[r] ^-{\R \underline{\Gamma} ^\dag _{X'} [-d _{X'/X}]}
&
{\mathrm{MIC} ^{\dag \dag} (X', \fP,T/K) .}
\ar@<4ex>[u] ^-{\ref{eqcat-u0+!}} _-{u _{0} ^{\prime!}}
}
\end{equation}  
The commutativity up to canonical isomorphism
of the top and middle squares of \ref{com-sp+-f*diag2} is checked as for 
\ref{com-sp+-f*diag1}. It remains to look at the bottom square. 
 Let 
 $\E \in \mathrm{MIC} ^{\dag \dag} (X, \fP,T/K) $. 
 The canonical morphism
 $$ u ^{\prime !} _{\alpha}
 \left (
 \R \underline{\Gamma} ^\dag _{X' }( \E) | \fP _{\alpha} 
 \right)  
 [-d _{X'/X}]
 \to
u ^{\prime !} _{\alpha}
 \left (
  \E| \fP _{\alpha} 
 \right)
 [-d _{X'/X}] 
 $$
 is an isomorphism.
Moreover, 
$u ^{\prime !} _{\alpha}
 \left (
  \E| \fP _{\alpha}
 \right)
  [-d _{X'/X}] 
\riso
a ^{ !} _{\alpha} u ^{ !} _{\alpha}
 \left (
 \E| \fP _{\alpha}  
 \right)
 [-d _{X'/X}]
 \riso 
a ^{ *} _{\alpha}  \left ( u ^{ !} _{\alpha}
 ( \E| \fP _{\alpha})
 \right)$.
These isomorphisms glue, hence we get the commutativity up to canonical isomorphism of the bottom square.

3) The case where $f = id$ and $a=id$ is checked similarly.
This yields the general case by decomposition the diagram \ref{com-sp+-f*-squ}.      
  
\end{proof}

\begin{rem}
The isomorphism 
$
\sp _+ (\smash{\O} _{]X[ _{\fP}}) 
\riso 
\mathcal{H} ^{\dag ,r} _X \O _{\fP,\Q} $
of 
\ref{coro-sp+jdagO}
can be viewed (modulo the compatibility of Proposition \ref{cor-induction-div-coh2}) 
as a particular case of \ref{com-sp+-f*} (in the case where $f=id$ and for the constant coefficient).
Notice that in order to give a meaning of the  isomorphism \ref{com-sp+-f*iso}, 
first we had to construct the local cohomological functor with support over a closed subscheme
(not only for the constant coefficient as for \ref{ntn-GammaZO-rig}). 
Recall also that in order to define our local cohomological functor in its general context, 
we do need the coherent theorem 
\ref{coh-ss-div}. But, the check of 
this latter theorem only uses the definition of \ref{ntn-GammaZO-rig}, which explains why we
have given a preliminary different construction of the local cohomological functor in a very special context.
\end{rem}

\begin{cor}
\label{cor-com-sp+-f*}
We keep notation \ref{com-sp+-f*}.
Let 
$\E ^{(\bullet)}$ and $\cF ^{(\bullet)}$ be two objects of 
$\mathrm{MIC} ^{(\bullet)} (X, \fP/K) $.
\begin{enumerate}
\item $\R \underline{\Gamma} ^\dag _{X'} f ^{!(\bullet)} \E ^{(\bullet)} [-d _{X'/X}]
\in 
\mathrm{MIC} ^{(\bullet)} (X', \fP'/K) $.
\item $\DD ^{(\bullet)} (\E ^{(\bullet)})
\in
\mathrm{MIC} ^{(\bullet)} (X, \fP/K) $.
\item We have the isomorphism
$\DD ^{(\bullet)}
\left (
\R \underline{\Gamma} ^\dag _{X'} f ^{!(\bullet)} \E ^{(\bullet)} [-d _{X'/X}]
\right )
\riso 
\R \underline{\Gamma} ^\dag _{X'} f ^{!(\bullet)} (\DD ^{(\bullet)} \E ^{(\bullet)} ) [-d _{X'/X}]$.
\item 
We have 
$\E ^{(\bullet)}
\smash{\widehat{\otimes}}^\L  _{\O ^{(\bullet)}  _{\fP} }  \FF ^{(\bullet)} [-d _{X/P}]
\in \mathrm{MIC} ^{(\bullet)} (X, \fP/K) $.
\end{enumerate}
\end{cor}

\begin{proof}
The fact that
$\R \underline{\Gamma} ^\dag _{X'} f ^{!(\bullet)} \E ^{(\bullet)} [- d _{X'/X}]
\in 
\mathrm{MIC} ^{(\bullet)} (X', \fP'/K) $
is local in $\fP '$.
Hence, we can suppose there exists a closed immersion of smooth 
$\V$-formal schemes $\mathfrak{u} \colon \X \hookrightarrow \fP$ 
(resp. $\mathfrak{u} ' \colon \X ' \hookrightarrow \fP' $, 
resp. $\mathfrak{a} \colon \X' \to \X$)
which reduces modulo $\pi$ to $u _0$ (resp. $u ' _0$, resp. $a$).
Following \ref{pre-loc-tri-B-t1T-iso}, 
$\R \underline{\Gamma} ^{\dag} _{X '}  f  ^{!(\bullet)} (   \E ^{(\bullet)})
\riso 
\mathfrak{u} _{+} ^{\prime (\bullet)} \circ  \mathfrak{u} ^{\prime !(\bullet)} \circ f  ^{!(\bullet)}(\E ^{(\bullet)})
\riso 
\mathfrak{u} _{+} ^{\prime (\bullet)} \circ  \mathfrak{a}   ^{!(\bullet)} \circ \mathfrak{u} ^{!(\bullet)} (\E ^{(\bullet)})$.
Since
$\mathfrak{u} ^{!(\bullet)} (\E ^{(\bullet)}) 
\in 
\mathrm{MIC} ^{(\bullet)} (\X/K) $,
then
$\L \mathfrak{a}   ^{*(\bullet)} \circ \mathfrak{u} ^{!(\bullet)} (\E ^{(\bullet)})
\in 
\mathrm{MIC} ^{(\bullet)} (\X'/K) $ (see \ref{corostab-MIC-f*}).
Since 
$\L \mathfrak{a}   ^{*(\bullet)} 
=
\mathfrak{a}   ^{!(\bullet)}[- d _{X'/X}] $, we get the first statement. 

The second statement is a consequence of 
\ref{propspetdualsansfrob}.
Using moreover 
\ref{com-sp+-f*}, since the dual functor of a isocrystal 
commutes with its inverse image, 
then we get the third one. 
The last one is a consequence of \ref{st-isoc-boxtimes} and of the first statement.
\end{proof}

\subsection{Base change isomorphism for coherent complexes and realizable morphisms}

\begin{dfn}
\label{realizablefscheme}
Let $f \colon \fP' \to \fP$ be a morphism of smooth formal $\V$-schemes. 
We say that $f$ is {\it realizable} if 
there exist an immersion 
$u \colon \fP' \hookrightarrow \fP''$ of smooth formal $\V$-schemes, 
a proper morphism 
$\pi \colon \fP'' \to \fP$ of smooth formal $\V$-schemes
such that
$f = \pi \circ u$.
When $\fP= \Spf \V$, we say that 
$\fP'$ is a realizable smooth formal $\V$-scheme.
We remark that is $\fP'$ is realizable, then any 
$f$ is a realizable morphism.
\end{dfn}

\begin{prop}
\label{stab-propersupp}
Let $g \colon \X '\to \X$ be a realizable morphism of smooth formal $\V$-schemes.
For any 
$\E ^{\prime (\bullet)} \in \smash{\underrightarrow{LD}} ^{\mathrm{b}} _{\Q,\mathrm{coh}} ( \smash{\widehat{\D}} _{\X '/\V} ^{(\bullet)})$
with proper support over $X$
(i.e., if $Z'$ is the support of $\E ^{\prime (\bullet)}$ in the sense \ref{dfn-support} then the composite 
$Z ' \hookrightarrow  X' \overset{g}{\to} X$ is proper), 
 the object 
$g _{+} (\E ^{\prime (\bullet)} ) $
belongs to 
$\smash{\underrightarrow{LD}} ^{\mathrm{b}} _{\Q,\mathrm{coh}} ( \smash{\widehat{\D}} _{\X/\V} ^{(\bullet)})$.
\end{prop}

\begin{proof}
Let $Z'$ be the support of 
$\E ^{\prime (\bullet)} \in \smash{\underrightarrow{LD}} ^{\mathrm{b}} _{\Q,\mathrm{coh}} ( \smash{\widehat{\D}} _{\X '/\V} ^{(\bullet)})$. By assumption, $Z'$ is proper over $X$ via $g$. 
Let $u \colon \X' \hookrightarrow \X''$ be an immersion of smooth formal $\V$-schemes, 
and $\pi \colon \X'' \to \X$ 
be a proper morphism of smooth formal $\V$-schemes
such that
$g = \pi \circ u$.
Let $v \colon \X' \hookrightarrow \U'' $ 
be a closed immersion, 
and $j\colon \U'' \hookrightarrow \X''$ be an open immersion
such that $u = j \circ v$.
Since $Z'$ is proper over $X$ via $g$ and
since $\pi$ is proper, then 
$Z'$ is proper over $X''$ via $u$.
Since $v$ is proper, 
$\E ^{\prime \prime (\bullet)}  :=
v _+ ^{(\bullet)}
(\E ^{\prime (\bullet)} )
\in 
\smash{\underrightarrow{LD}} ^{\mathrm{b}} _{\Q,\mathrm{coh}} ( \smash{\widehat{\D}} _{\U''/\V} ^{(\bullet)})$
with proper support over $X''$ via $j$.

a) We check in this step that 
$\R j _* \E ^{\prime \prime (\bullet)} 
\in 
\smash{\underrightarrow{LD}} ^{\mathrm{b}} _{\Q,\mathrm{coh}} ( \smash{\widehat{\D}} _{\X''/\V} ^{(\bullet)})$
(with support in $Z'$).
Following \ref{eqcatLD=DSM-fonct-coh},
we have the equivalence of triangulated categories
$\underrightarrow{LD} ^{\mathrm{b}} _{\Q, \mathrm{coh}} (\smash{\widehat{\D}} _{\U''/\V} ^{(\bullet)})
\cong
D ^{\mathrm{b}} _{\mathrm{coh}}
(\underrightarrow{LM} _{\Q} (\smash{\widehat{\D}} _{\U''/\V} ^{(\bullet)} ))$.
Hence, we reduce by devissage to the case where
$\E ^{\prime \prime (\bullet)} 
\in 
\smash{\underrightarrow{LM}}  _{\Q,\mathrm{coh}} ( \smash{\widehat{\D}} _{\U''/\V} ^{(\bullet)})$.
Following 
\cite[2.4.5]{caro-stab-sys-ind-surcoh}, there exists 
$m_0 \in \N$ large enough such that
$\E ^{\prime \prime (\bullet)} $ 
is isomorphic 
in $\underrightarrow{LM} _{\Q} (\smash{\widehat{\D}} _{\U''} ^{(\bullet)})$
to a locally finitely presented 
$\smash{\widehat{\D}} _{\U ''} ^{(\bullet + m_0)}$-module
$\G ^{(\bullet)} $.
Since $\E ^{\prime \prime (\bullet)} $  is isomorphic to zero on the open complementary to $Z'$,
then, taking larger $m_0$ is necessary, we can suppose
$\G ^{(0)} $ is a coherent $\smash{\widehat{\D}} _{\U ''} ^{(m_0)}$-module with support in $Z'$
(i.e. $\G ^{(0)} $ is zero on the open complementary to $Z '$).
This yields that 
$\G ^{(m)} $ is a coherent
$\smash{\widehat{\D}} _{\U ''} ^{(m + m_0)}$-module
with support in $Z'$.
Since $Z'$ is closed in $X''$, 
a coherent $\smash{\widehat{\D}} _{\U ''} ^{(m + m_0)}$-module
with support in $Z '$ is acyclic for 
the functor $j _*$. 
Moreover, the functors $j _*$ and $j ^*$ induce quasi-inverse equivalences of categories
between the category of 
coherent $\smash{\widehat{\D}} _{\U ''} ^{(m + m_0)}$-modules
with support in $Z '$
and that 
of coherent $\smash{\widehat{\D}} _{\X ''} ^{(m + m_0)}$-modules
with support in $Z '$.
This yields that the canonical morphism
$\smash{\widehat{\D}} _{\X ''} ^{(\bullet + m_0)}
\otimes _{\smash{\widehat{\D}} _{\X ''} ^{(m_0)}} 
j _* \G ^{(0)} 
\to
 j_* \G ^{(\bullet)} $
is an isomorphim and that 
$ j_* \G ^{(\bullet)} $
is a locally finitely presented 
$\smash{\widehat{\D}} _{\X ''} ^{(\bullet + m_0)}$-module with support in $Z'$.
Moreover, since the functor $j _*$ is exact over the category of sheaves with support in $Z'$ (because
$Z'$ is closed in $X''$ and in $U''$),
then the canonical morphism
$ j_* \G ^{(\bullet)} 
 \to
\R j_* \G ^{(\bullet)} $
is an isomorphism.
Hence, 
$j _+ ^{(\bullet) } \E ^{\prime \prime (\bullet)} 
=
\R j _* \E ^{\prime \prime (\bullet)} 
\in 
\smash{\underrightarrow{LM}}  _{\Q,\mathrm{coh}} ( \smash{\widehat{\D}} _{\X''/\V} ^{(\bullet)})$. 

b) Since $\pi $ is proper, the part a)
yields
$\pi _+ ^{(\bullet)} 
j _+ ^{(\bullet) }( \E ^{\prime \prime (\bullet)} )
\in \smash{\underrightarrow{LD}} ^{\mathrm{b}} _{\Q,\mathrm{coh}} ( \smash{\widehat{\D}} _{\X/\V} ^{(\bullet)})$.
By transitivity of the push forwards, we get the isomorphism
$g ^{(\bullet)}  _{+} (\E ^{\prime (\bullet)} ) 
\riso 
\pi _+ ^{(\bullet)} 
j _+ ^{(\bullet) }( \E ^{\prime \prime (\bullet)} )$.
Hence, we are done.
\end{proof}

\begin{empt}
\label{rem-gopenimm}
Let $g \colon \X '\to \X$ be an open immersion of smooth formal $\V$-schemes.
Let $Z'$ be a closed subscheme of $X'$ such that 
the composite 
$Z ' \hookrightarrow  X' \overset{g}{\to} X$ is proper.
Then the functors $g _+$ and $g ^!$ induce quasi-inverse equivalences of categories 
between the subcategory of 
$\smash{\underrightarrow{LD}} ^{\mathrm{b}} _{\Q,\mathrm{coh}} ( \smash{\widehat{\D}} _{\X '/\V} ^{(\bullet)})$
consisting of objects
with support in $Z '$
and 
the subcategory of 
$\smash{\underrightarrow{LD}} ^{\mathrm{b}} _{\Q,\mathrm{coh}} ( \smash{\widehat{\D}} _{\X /\V} ^{(\bullet)})$
consisting of objects
with support in $Z'$.
Indeed, from \ref{stab-propersupp}, 
such functor $g _+=\R g _*$ is well defined.
This is also obvious for $g ^! =g ^*$.
We check easily that the canonical morphism 
$g ^* \R g _* \to id$ is an isomorphism and that 
the functor $g ^*$ is faithful. 
Hence, the morphism of functors 
$id \to \R g _* g ^* $, defined over 
the subcategory 
of $\smash{\underrightarrow{LD}} ^{\mathrm{b}} _{\Q,\mathrm{coh}} ( \smash{\widehat{\D}} _{\X /\V} ^{(\bullet)})$
consisting of objects
with support in $Z '$,
is an isomorphism.

\end{empt}

\begin{theo}
\label{theo-iso-chgtbase}
Let  $f \colon \X ^{\prime } \to \X $, $g \colon \Y \to \X$ be two morphisms of  smooth formal $\V$-schemes such that 
$g$ is  smooth and $f$ is realizable.
Set $\Y ^{\prime } := \X ^{\prime } \times _{\X} \Y$, 
$f ' \colon \Y ^{\prime } \to \Y$, $g ' \colon \Y ^{\prime }\to \X ^{\prime }$ 
be the canonical projections. 
Let
$\E ^{\prime (\bullet)}
\in  \smash{\underrightarrow{LD}} ^\mathrm{b} _{\Q, \mathrm{coh}}
(\overset{^\mathrm{l}}{} \smash{\widehat{\D}} _{\X ^{\prime }/\S  } ^{(\bullet)} )$ with proper support over $X$. 
There exists a canonical isomorphism in 
$\smash{\underrightarrow{LD}} ^\mathrm{b} _{\Q, \mathrm{coh}}
(\overset{^\mathrm{l}}{} \smash{\widehat{\D}} _{\X ^{\prime }/\S  } ^{(\bullet)} )$:
\begin{equation}
\label{basechange}
g ^{ !(\bullet)} \circ f ^{(\bullet)}_{+} (\E ^{\prime (\bullet)})
\riso
f  ^{\prime (\bullet)}_{+}  \circ g ^{\prime(\bullet) !} (\E ^{\prime (\bullet)}). 
\end{equation}
\end{theo}

\begin{proof}
Using 
\ref{stab-propersupp}
and \ref{stab-coh-f^!}  we check the complexes of \ref{basechange} are indeed coherent.
The morphism $g$ is the composite of its graph 
$\gamma \colon 
\Y   \to \X \times \Y  $ 
with the canonical projection 
$\pi \colon  \X \times \Y \to \X$.
The morphism $g'$ is the composite of a morphism 
$\gamma' \colon 
\Y ^{\prime }  \to \X ^{\prime }\times \Y  $ 
with the canonical projection 
$\pi '\colon  \X ^{\prime }\times \Y \to \X^{\prime }$.
Set
$\U : = \X \times \Y $, $\U ^{\prime }: = \X ^{\prime }\times \Y $,
$f ^{\prime \prime}= f \times \mathrm{id} _\Y \colon  \U ^{\prime } \to  \U $.
Using Theorem \ref{u!u+=id},
we reduce to check we have a canonical isomorphism of the form
 $$\gamma ^{ (\bullet)} _{+}  \circ  g ^{ !(\bullet)} \circ f ^{(\bullet)}_{+} (\E ^{\prime (\bullet)})
\riso 
\gamma ^{ (\bullet)} _{+}  \circ  f  ^{\prime (\bullet)}_{+}  \circ g ^{\prime (\bullet) !} (\E ^{\prime (\bullet)}).$$
Concerning the left hand side, by transitivity of extraordinary pullbacks we get the canonical isomorphism
  $\gamma ^{ (\bullet)} _{+}  \circ  g ^{ !(\bullet)} \circ f ^{(\bullet)}_{+} (\E ^{\prime (\bullet)})
\riso 
\gamma ^{ (\bullet)} _{+}  \circ  
\gamma ^{ !(\bullet)} \circ
\pi ^{ !(\bullet)} \circ
 f ^{(\bullet)}_{+} (\E ^{\prime (\bullet)})
 $.
Concerning the right hand side, by transitivity of extraordinary pullbacks and of push-forwards 
and by using \ref{commutfonctcohlocal2}, \ref{pre-loc-tri-B-t1T-iso} 
we obtain the canonical isomorphisms
$\gamma ^{ (\bullet)} _{+}  \circ  f  ^{\prime (\bullet)}_{+}  \circ g ^{\prime (\bullet) !} (\E ^{\prime (\bullet)})
 \riso 
f  ^{\prime \prime (\bullet)}_{+}  \circ  \gamma ^{\prime  (\bullet)} _{+}   
 \circ \gamma ^{\prime (\bullet) !} 
 \circ \pi ^{\prime (\bullet) !} 
 (\E ^{\prime (\bullet)})
 \riso
\gamma ^{ (\bullet)} _{+}   
 \circ \gamma ^{ (\bullet) !}
 \circ   
 f  ^{\prime \prime (\bullet)}_{+}  
 \circ \pi ^{\prime (\bullet) !} 
 (\E ^{\prime (\bullet)})
 $.
Hence, we reduce to check the isomorphism
 $\pi ^{ !(\bullet)} \circ
 f ^{(\bullet)}_{+} (\E ^{\prime (\bullet)})
 \riso
 f  ^{\prime \prime (\bullet)}_{+}  
 \circ \pi ^{\prime (\bullet) !} 
 (\E ^{\prime (\bullet)})$, which is already known following 
 Theorem \ref{theo-iso-chgtbase2}.
\end{proof}

\subsection{Relative duality isomorphism and adjunction for realizable morphisms}

\begin{thm}
[Relative duality isomorphism]
\label{rel-dual-isom}
Let $g \colon \fP '\to \fP$ be a realizable morphism of smooth formal $\V$-schemes.
For any 
$\E ^{\prime (\bullet)} 
\in 
\smash{\underrightarrow{LD}} ^{\mathrm{b}} _{\Q,\mathrm{coh}} ( \smash{\widehat{\D}} _{\fP '} ^{(\bullet)})$
with proper support over $P$, 
we have the isomorphism of $\smash{\underrightarrow{LD}} ^{\mathrm{b}} _{\Q,\mathrm{coh}} ( \smash{\widehat{\D}} _{\fP} ^{(\bullet)})$
of the form 
$$g _{+} \circ \DD 
(\E ^{\prime (\bullet)} ) 
\riso 
\DD \circ g _{+} 
(\E ^{\prime (\bullet)} ) .
$$
\end{thm}

\begin{proof}
Let $\E ^{\prime (\bullet)} \in \smash{\underrightarrow{LD}} ^{\mathrm{b}} _{\Q,\mathrm{coh}} (\fP'/\V)$
whose support $X'$ is proper over $P$ via $g$.
Let $u \colon \fP' \hookrightarrow \fP''$ be an immersion of smooth formal $\V$-schemes, 
and $\pi \colon \fP'' \to \fP$ 
be a proper morphism of smooth formal $\V$-schemes
such that
$g = \pi \circ u$.
Let $v \colon \fP' \hookrightarrow \U'' $ 
be a closed immersion, 
and $j\colon \U'' \hookrightarrow \fP''$ be an open immersion
such that $u = j \circ v$.

From the relative duality isomorphism in the proper case (see \cite{Vir04}), 
$\DD  v _+ (\E ^{\prime (\bullet)}) \riso 
v _+   \DD (\E ^{\prime (\bullet)})$.
Set $\FF ^{\prime (\bullet)} := v _+ (\E ^{\prime (\bullet)})$.
Using \ref{stab-propersupp}, we get
$j _+ \FF ^{\prime (\bullet)} 
\in 
\smash{\underrightarrow{LD}} ^{\mathrm{b}} _{\Q,\mathrm{coh}} (\fP '' /\V) $
and it has his support in $X'$.
Hence, 
$ \DD  j _+ \FF ^{\prime (\bullet)} 
\in 
 \smash{\underrightarrow{LD}} ^{\mathrm{b}} _{\Q,\mathrm{coh}} ( \smash{\widehat{\D}} _{\fP ''} ^{(\bullet)})$
 and has its support in $X'$. 
Hence, 
$ \DD  j _+ \FF ^{\prime (\bullet)} 
\riso 
j _+ j ^!  \DD  j _+ \FF ^{\prime (\bullet)}$ (use \ref{rem-gopenimm}).
Moreover, this is obvious that
$j ^!  \DD  j _+ \FF ^{\prime (\bullet)}
\riso 
\DD  j ^!   j _+ \FF ^{\prime (\bullet)}
\riso 
\DD \FF ^{\prime (\bullet)}
$.
Hence, 
$ \DD  j _+ \FF ^{\prime (\bullet)} 
\riso 
j _+  \DD  \FF ^{\prime (\bullet)}$. 
By composition we get 
$\DD  u _+ (\E ^{\prime (\bullet)}) \riso 
u _+   \DD (\E ^{\prime (\bullet)})$.
Since $\pi$ is proper, 
from the relative duality isomorphism in the proper case (see \ref{dualrelative}),
we obtain the first isomorphism
$\DD  \pi _+  u _+ (\E ^{\prime (\bullet)}) 
\riso 
\pi _+ \DD  u _+ (\E ^{\prime (\bullet)}) 
\riso
\pi _+  u _+ \DD (\E ^{\prime (\bullet)}) $. 
Hence, we are done.
\end{proof}

\begin{cor}
\label{cor-adj-formulbis}
Let $g \colon \fP '\to \fP$ be a realizable morphism of smooth formal $\V$-schemes.
Let $\E ' \in D ^\mathrm{b} _{\mathrm{coh}}
(\D ^{\dag} _{\fP ^{\prime },\Q})$
with proper support over $P$,
and
$\E 
\in 
D ^\mathrm{b} _{\mathrm{coh}}
(\D ^{\dag} _{\fP  ,\Q})$.
We have 
the isomorphisms
\begin{gather}
\label{cor-adj-formulbis-bij1}
\R \mathcal{H} om _{\D ^{\dag} _{\fP ,\Q}}
( g _{+} ( \E ') , \E) 
\riso 
\R g _* 
\R \mathcal{H} om _{\D ^{\dag} _{\fP ' ,\Q}}
( \E ' ,g ^! ( \E)). 
\\
\label{cor-adj-formulbis-bij2}
\R \mathrm{Hom}  _{\D ^{\dag} _{\fP ,\Q}}
( g _{+} ( \E ') , \E) 
\riso 
\R \mathrm{Hom}  _{\D ^{\dag} _{\fP ' ,\Q}}
( \E ' ,g ^!   ( \E)). 
\end{gather}
\end{cor}

\begin{proof}
Using \ref{rel-dual-isom}, 
we can copy word by word the proof of \ref{cor-adj-formul}. 
\end{proof}

\section{Stability under Grothendieck's six operations}

\subsection{Data of coefficients}

\begin{dfn}
\label{DVR}
Fix  a perfectification $\V \to \V ^\flat$ of $\V$ (see definition \ref{def-perfectification}).
We define the category 
$\mathrm{DVR}  (\V,\V ^\flat)$ 
as follows : 
an object is the data of  
a complete discrete valued ring $\W$ of mixed characteristic $(0,p)$, together with 
a perfectification $\W \to \W ^\flat$ and two morphisms of local algebras
$\V \to \W$ and 
$\V ^\flat
\to 
\W ^\flat$
making commutative the diagram
\begin{equation}
\notag
\xymatrix{
{\V ^\flat } 
\ar[r] ^-{}
&
{\W ^\flat} 
\\ 
{\V} 
\ar[r] ^-{}
\ar[u] ^-{}
& 
{\W.} 
\ar[u] ^-{}
}
\end{equation}
Such object is simply denoted by 
$(\W, \W ^\flat)$.
A morphism 
$(\W, \W ^\flat)
\to 
(\W ', \W ^{\prime \flat})$ 
is the data of 
a morphism of local $\V$-algebras
$\W \to \W'$ and a morphism of local $\V ^\flat$-algebras
$\W ^\flat\to \W ^{\prime \flat}$
making commutative the following diagram:
\begin{equation}
\notag
\xymatrix{
{\W ^\flat } 
\ar[r] ^-{}
&
{\W ^{\prime \flat} } 
\\ 
{\W}
\ar[r] ^-{}
\ar[u] ^-{}
& 
{\W'.}  
\ar[u] ^-{}
}
\end{equation}

A special morphism 
$(\W, \W ^\flat)
\to 
(\W ', \W ^\flat)$
of $\mathrm{DVR}  (\V,\V ^\flat)$
is a morphism such that
the underlying morphism 
$\W ^\flat
\to \W ^\flat$ is the identity
and 
such that 
$\W \to \W' $ is a special morphism in the sense of \ref{def-perfectification}.

When working exclusively with perfect residue fields, 
we retrieve the definition of \cite[1.1.1]{caro-unip}.
In the proof of \ref{S(D,C)stability8}, 
we will need to consider ``fixed'' perfectifications which gives a justification to work with 
$\mathrm{DVR}  (\V,\V ^\flat)$.

\end{dfn}

\begin{empt}
[Base change and their commutation with cohomological operations]
\label{comm-chg-base}

Let $\alpha \colon (\W, \W ^\flat)
\to 
(\W ', \W ^{\prime \flat})$ be a morphism of $\mathrm{DVR}  (\V,\V ^\flat)$,
let $\X$ be a smooth formal scheme over $\W$, 
$\E ^{(\bullet)} \in \smash{\underrightarrow{LD}} ^{\mathrm{b}} _{\Q,\mathrm{qc}} 
( \smash{\widehat{\D}} _{\X/\Spf (\W)} ^{(\bullet)})$,
$\X' := \X \times _{\Spf (\W)} \Spf \W'$, and $\pi  \colon \X' \to \X$ 
be the projection.
The base change of $\E ^{(\bullet)} $ by $\alpha$ is 
an object 
$\pi  ^{! (\bullet)} (\E ^{(\bullet)})$ of 
$\smash{\underrightarrow{LD}} ^{\mathrm{b}} _{\Q,\mathrm{qc}} ( \smash{\widehat{\D}} _{\X '/\Spf (\W')} ^{(\bullet)})$ (see \cite[2.2.2]{Beintro2}).
Similarly to \cite[2.2.2]{Beintro2}, 
it will simply be denoted by 
$ \W'  \smash{\widehat{\otimes}}^\L
_{\W}  \E^{(\bullet) }$.

From \cite[2.4.2]{Beintro2}, push forwards commute with base change. 
The commutation of base change with extraordinary pullbacks, local cohomological functors,
duals functors (for coherent complexes), and tensor products is 
straightforward. 

\end{empt}

\begin{empt}
\label{t-structure-coh}
Let $(\W, \W ^\flat)$ be an object of $\mathrm{DVR}  (\V,\V ^\flat)$, and 
$\X$ be a smooth formal $\W$-scheme.
If there is no possible confusion (some confusion might arise if for example we do know that $\V \to \W$ is finite and etale), 
for any integer $m \in \N$, 
we denote 
$\smash{\widehat{\D}} _{\X/\Spf (\W)} ^{(m)}$
(resp. $\smash{\D} ^\dag _{\X/\Spf (\W), \Q}$)
simply by 
$\smash{\widehat{\D}} _{\X} ^{(m)}$
(resp. $\smash{\D} ^\dag _{\X\Q}$).
Berthelot checked the following equivalence of categories (see \cite[4.2.4]{Beintro2}, or \ref{eq-catLDBer-LD-D}):
\begin{equation}
\label{limeqcat}
\underrightarrow{\lim}
\colon 
\smash{\underrightarrow{LD}} ^{\mathrm{b}} _{\Q,\mathrm{coh}} ( \smash{\widehat{\D}} _{\X} ^{(\bullet)})
\cong 
D ^{\mathrm{b}}  _{\mathrm{coh}} (\smash{\D} ^\dag _{\X\Q}).
\end{equation}

The category 
$D ^{\mathrm{b}}  _{\mathrm{coh}} (\smash{\D} ^\dag _{\X\Q})$ 
is endowed with its usual t-structure.
Via \ref{limeqcat}, we get a t-structure on 
$\smash{\underrightarrow{LD}} ^{\mathrm{b}} _{\Q,\mathrm{coh}} ( \smash{\widehat{\D}} _{\X} ^{(\bullet)})$
whose heart is 
$\smash{\underrightarrow{LM}}  _{\Q,\mathrm{coh}} ( \smash{\widehat{\D}} _{\X} ^{(\bullet)})$
(see Notation \ref{nota-(L)Mcoh}).
Recall, following \ref{empt-diag-Hn-comp},
we have canonical explicit cohomological functors
$\mathcal{H} ^n 
\colon 
\smash{\underrightarrow{LD}} ^{\mathrm{b}} _{\Q,\mathrm{coh}} ( \smash{\widehat{\D}} _{\X} ^{(\bullet)})
\to 
\smash{\underrightarrow{LM}} _{\Q,\mathrm{coh}} ( \smash{\widehat{\D}} _{\X} ^{(\bullet)})$.
The equivalence of categories 
\ref{limeqcat} commutes with the 
cohomogical functors $\mathcal{H} ^n $
(where the cohomogical functors $\mathcal{H} ^n $ on 
$D ^{\mathrm{b}}  _{\mathrm{coh}} (\smash{\D} ^\dag _{\X\Q})$
are the obvious ones),
 i.e. 
$\underrightarrow{\lim}
\mathcal{H} ^n (\E ^{(\bullet)})$
is  canonically isomorphic
to 
$\mathcal{H} ^n (\underrightarrow{\lim} \,\E ^{(\bullet)})$.

Last but not least, 
following \ref{eqcat-limcoh} 
we have the equivalence of categories 
$\smash{\underrightarrow{LD}} ^{\mathrm{b}} _{\Q,\mathrm{coh}} ( \smash{\widehat{\D}} _{\X} ^{(\bullet)})
\cong 
D ^{\mathrm{b}} _{\mathrm{coh}}
(
\smash{\underrightarrow{LM}} _{\Q} ( \smash{\widehat{\D}} _{\X} ^{(\bullet)})
)
$
which is also compatible with t-structures, 
where the t-structure on 
$D ^{\mathrm{b}} _{\mathrm{coh}}
(\smash{\underrightarrow{LM}} _{\Q} ( \smash{\widehat{\D}} _{\X} ^{(\bullet)}))$
is the canonical one as the derived category of an abelian category.
\end{empt}

\begin{dfn}
\label{dfn-datacoef}
A {\it data of coefficients $\mathfrak{C}$ over $(\V,\V ^\flat)$}
will be the data for any object
$(\W, \W ^\flat)$ of $\mathrm{DVR}  (\V,\V ^\flat)$, 
for any  
smooth formal scheme $\X$ over $\W$ 
of a strictly full subcategory of 
$\smash{\underrightarrow{LD}} ^{\mathrm{b}} _{\Q,\mathrm{coh}} ( \smash{\widehat{\D}} _{\X} ^{(\bullet)})$,
which will be denoted by $\mathfrak{C} (\X/(\W, \W ^\flat))$,
or simply $\mathfrak{C} (\X)$ if there is no ambiguity with the base $(\W, \W ^\flat)$.
If there is no ambiguity with $(\V,\V ^\flat)$, 
we simply say a {\it data of coefficients}.
\end{dfn}

\begin{exs}
\label{ex-Dcst}
We have the following data of coefficients.

\begin{enumerate}

\item  We define $\fB _\emptyset$ as follows: 
for any object $(\W, \W ^\flat)$ of $\mathrm{DVR}  (\V,\V ^\flat)$, 
for any  smooth formal scheme $\X$ over $\W$,  
the category $\fB _\emptyset (\X)$ is the full subcategory of 
$\smash{\underrightarrow{LD}} ^{\mathrm{b}} _{\Q,\mathrm{coh}} ( \smash{\widehat{\D}} _{\X} ^{(\bullet)})$
whose unique object is 
$\O _{\X} ^{(\bullet)}$ 
(where $\O _{\X} ^{(\bullet)}$ is the constant object $\O _{\X} ^{(m)} = \O _{\X}$ for any $m\in \N$ with the identity as transition maps).

\item We will need the larger data of coefficients $\fB _\mathrm{div}$ 
defined as follows: 
for any object $(\W, \W ^\flat)$ of $\mathrm{DVR}  (\V,\V ^\flat)$, 
for any  smooth formal scheme $\X$ over $\W$,  
the category $\fB _\mathrm{div}(\X)$ is the full subcategory of 
$\smash{\underrightarrow{LD}} ^{\mathrm{b}} _{\Q,\mathrm{coh}} ( \smash{\widehat{\D}} _{\X} ^{(\bullet)})$
whose objects are of the form  
$\widehat{\B} ^{(\bullet)} _{\X} (T)$, 
where $T$ is any divisor of the special fiber of $\X$.
From Corollary 
\ref{coh-Bbullet}, 
we have
$\widehat{\B} ^{(\bullet)} _{\X} (T)\in 
\smash{\underrightarrow{LD}} ^{\mathrm{b}} _{\Q,\mathrm{coh}} ( \smash{\widehat{\D}} _{\X} ^{(\bullet)})$. 

\item We define $\fB _\mathrm{cst}$ 
as follows: 
for any object $(\W, \W ^\flat)$ of $\mathrm{DVR}  (\V,\V ^\flat)$, 
for any  smooth formal scheme $\X$ over $\W$,  
the category $\fB _\mathrm{cst}(\X)$ is the full subcategory of 
$\smash{\underrightarrow{LD}} ^{\mathrm{b}} _{\Q,\mathrm{coh}} ( \smash{\widehat{\D}} _{\X} ^{(\bullet)})$
whose objects are of the form  
$\R \underline{\Gamma} ^\dag _{Y} \O _\X ^{(\bullet)} $, 
where $Y$ is a subvariety of the special fiber of $\X$
and the functor 
$\R \underline{\Gamma} ^\dag _{Y}$ is defined in 
\ref{3.2.1caro-2006-surcoh-surcv}.
Recall following \ref{prop-induction-div-coh}, 
theses objects are coherent.

\item We define 
$\fM _\emptyset$
(resp. $\fM _\mathrm{sscd}(\X)$,
resp. $\fM _\mathrm{gsncd}(\X)$,
resp. $\fM _\mathrm{div}$)
as follows: 
for any object $(\W, \W ^\flat)$ of $\mathrm{DVR}  (\V,\V ^\flat)$, 
for any  smooth formal scheme $\X$ over $\W$,  
the category 
$\fM _\emptyset (\X)$
(resp. $\fM _\mathrm{sscd}(\X)$,
resp. $\fM _\mathrm{gsncd}(\X)$,
resp. $\fM _\mathrm{div}(\X)$)
is the full subcategory of 
$\smash{\underrightarrow{LD}} ^{\mathrm{b}} _{\Q,\mathrm{coh}} ( \smash{\widehat{\D}} _{\X} ^{(\bullet)})$
consisting of objects 
of the form  
$(\hdag T)  (\E ^{(\bullet)} )$, 
where 
$\E ^{(\bullet)}
\in 
\mathrm{MIC} ^{(\bullet)} (Z, \X/K) $ (see notation
\ref{ntnMICdag2fs3}),
with $Z$ is a smooth subvariety of 
the special fiber of $\X$,
and where $T$ is an empty divisor
(resp. a strict smooth crossing divisor, 
resp. a geometric strict normal crossing divisor, 
resp. a divisor) of $Z$  over the special fiber of $\W$ (see definition \ref{sscd}).
Recall that following \ref{coh-ss-div-bis}, 
these objects are indeed coherent.

\end{enumerate}

\end{exs}

\begin{dfn}
\label{dfn-stable-data}
In order to be precise, let us fix some terminology.
Let $\mathfrak{C}$ and $\mathfrak{D}$ be two data of coefficients over $(\V,\V ^\flat)$. 
\begin{enumerate}
\item We will say that the data of coefficients $\mathfrak{C}$  is
stable under 
pushforwards 
if for any object $(\W, \W ^\flat)$ of $\mathrm{DVR}  (\V,\V ^\flat)$, 
for any 
{\it realizable} morphism $g \colon \X ' \to \X$ of  smooth formal schemes over $\W$, 
for any object $\E ^{\prime (\bullet)}$ of $\mathfrak{C} (\X')$ with proper support over $X$ via $g$,
the complex $g _{+} (\E ^{\prime (\bullet)})$ is an object of  $\mathfrak{C} (\X)$.

\item We will say that the data of coefficients $\mathfrak{C}$  is stable under extraordinary pullbacks 
(resp. under {\it smooth} extraordinary pullbacks)
if for any object $(\W, \W ^\flat)$ of $\mathrm{DVR}  (\V,\V ^\flat)$, 
for any morphism (resp. {\it smooth} morphism)
$f \colon \Y \to \X$ of  smooth formal schemes over $\W$, 
for any object $\E ^{(\bullet)}$ of $\mathfrak{C} (\X)$, 
we have $f ^{!} (\E ^{(\bullet)})\in \mathfrak{C} (\Y)$.

\item We still say that the data of coefficients $\mathfrak{C}$  satisfies 
the first property (resp. the second property) of Berthelot-Kashiwara theorem
or satisfies $BK ^!$ (resp. $BK _+$) for short if the following property is satisfied:
for any object $(\W, \W ^\flat)$ of $\mathrm{DVR}  (\V,\V ^\flat)$, 
for any closed immersion  $u \colon \ZZ \hookrightarrow \X$ of  smooth formal schemes over $\W$, 
for any object $\E ^{(\bullet)}$ of $\mathfrak{C} (\X)$ with support in $\ZZ$, 
we have $u ^{!} (\E ^{(\bullet)})\in \mathfrak{C} (\ZZ)$
(resp. 
for any object $\G ^{(\bullet)}$ of $\mathfrak{C} (\ZZ)$, 
we have $u _{+} (\G ^{(\bullet)})\in \mathfrak{C} (\X)$).
Remark that $BK ^!$ and $BK _+$ hold if and only if the data of coefficients $\mathfrak{C}$ satisfies 
(an analogue of) Berthelot-Kashiwara theorem, which justifies the terminology. 

\item We will say that the data of coefficients $\mathfrak{C}$ is stable under base change 
(resp. special base change) if
for any morphism 
$(\W, \W ^\flat)
\to 
(\W ', \W ^{\prime \flat})$ 
(resp. for any special morphism 
$(\W, \W ^\flat)
\to 
(\W ', \W ^\flat)$ 
of $\mathrm{DVR}  (\V,\V ^\flat)$,
for any  smooth formal scheme $\X$ over $\W$, 
for any object $\E ^{(\bullet)}$ of $\mathfrak{C} (\X)$, 
we have $ \W'  \smash{\widehat{\otimes}}^\L
_{\W}  \E^{(\bullet) }\in \mathfrak{C} (\X \times _{\Spf \W} \Spf \W ')$.

\item We will say that the data of coefficients $\mathfrak{C}$ is stable under special  descent of the base if, 
for any object $(\W, \W ^\flat)$ of $\mathrm{DVR}  (\V,\V ^\flat)$, 
for any  smooth formal scheme $\X$ over $\W$, 
for any object $\E ^{(\bullet)}\in \smash{\underrightarrow{LD}} ^{\mathrm{b}} _{\Q,\mathrm{coh}} ( \smash{\widehat{\D}} _{\X} ^{(\bullet)})$, 
if
there exists a 
special morphism 
$(\W, \W ^\flat)
\to 
(\W ', \W ^\flat)$
such that $ \W'  \smash{\widehat{\otimes}}^\L
_{\W}  \E^{(\bullet) }\in \mathfrak{C} (\X \times _{\Spf \W} \Spf \W ')$,
then 
$\E^{(\bullet) }\in \mathfrak{C} (\X )$.

\item We will say that the data of coefficients $\mathfrak{C}$ is stable under tensor products (resp. duals) if
for any object $(\W, \W ^\flat)$ of $\mathrm{DVR}  (\V,\V ^\flat)$, 
for any  smooth formal scheme $\X$ over $\W$, 
for any objects $\E ^{(\bullet)}$ and $\FF ^{(\bullet)}$ of $\mathfrak{C} (\X)$
we have 
$\FF ^{(\bullet)}
\smash{\widehat{\otimes}}^\L
_{\O  _{\X}} \E^{(\bullet) }
\in \mathfrak{C} (\X)$
(resp. 
$\DD  _{\X}(\E^{(\bullet) }) \in  \mathfrak{C} (\X)$).

\item We will say that the data of coefficients $\mathfrak{C}$ is stable under local cohomological functors
(resp. under localizations outside a divisor), if 
for any object $(\W, \W ^\flat)$ of $\mathrm{DVR}  (\V,\V ^\flat)$, 
for any  smooth formal scheme $\X$ over $\W$, 
for any object $\E ^{(\bullet)}$ of $\mathfrak{C} (\X)$, 
for any subvariety $Y$ (resp. for any divisor $T$) of the special fiber of $\X$, 
we have
$\R \underline{\Gamma} ^\dag _{Y} \E ^{(\bullet)} \in \mathfrak{C} (\X)$ 
(resp. $(\hdag T) (\E ^{(\bullet)} )\in \mathfrak{C} (\X)$).

\item We will say that the data of coefficients $\mathfrak{C}$ is stable under cohomology if,
for any object $(\W, \W ^\flat)$ of $\mathrm{DVR}  (\V,\V ^\flat)$, 
for any  smooth formal scheme $\X$ over $\W$,
for any object  $\E ^{(\bullet)}$ of $\smash{\underrightarrow{LD}} ^{\mathrm{b}} _{\Q,\mathrm{coh}} ( \smash{\widehat{\D}} _{\X} ^{(\bullet)})$, the property 
$\E ^{(\bullet)}$ is an object of 
$\mathfrak{C} (\X)$ is equivalent to the fact that, for any integer $n$, 
$\H ^n (\E ^{(\bullet)})$ is an object of $\mathfrak{C} (\X)$. 

\item We will say that the data of coefficients $\mathfrak{C}$ is stable under shifts if,
for any object $(\W, \W ^\flat)$ of $\mathrm{DVR}  (\V,\V ^\flat)$, 
for any  smooth formal scheme $\X$ over $\W$,
for any object  $\E ^{(\bullet)}$ of $\mathfrak{C} (\X)$, for any integer $n$, 
$\E ^{(\bullet)} [n]$ is an object of $\mathfrak{C} (\X)$.

\item We will say that the data of coefficients $\mathfrak{C}$ is stable by devissages if
 $\mathfrak{C}$ is stable by shifts and if
for any object $(\W, \W ^\flat)$ of $\mathrm{DVR}  (\V,\V ^\flat)$, 
for any  smooth formal scheme $\X$ over $\W$, 
for any exact triangle 
$\E ^{(\bullet)} _1
\to 
\E ^{(\bullet)} _2
\to 
\E ^{(\bullet)} _3
\to 
\E ^{(\bullet)} _1 [1]$
of 
$\smash{\underrightarrow{LD}} ^{\mathrm{b}} _{\Q,\mathrm{coh}} ( \smash{\widehat{\D}} _{\X} ^{(\bullet)})$, 
if two objects are in $\mathfrak{C} (\X)$, then so is the third one. 

\item We will say that the data of coefficients $\mathfrak{C}$ is stable under direct summands if,
for any object $(\W, \W ^\flat)$ of $\mathrm{DVR}  (\V,\V ^\flat)$, 
for any  smooth formal scheme $\X$ over $\W$ we have the following property: 
any direct summand in $\smash{\underrightarrow{LD}} ^{\mathrm{b}} _{\Q,\mathrm{coh}} ( \smash{\widehat{\D}} _{\X} ^{(\bullet)})$
of an object of $\mathfrak{C} (\X)$ is an object of
$\mathfrak{C} (\X)$.

\item We say that $\mathfrak{C}$ contains $\mathfrak{D}$ 
(or $\mathfrak{D}$ is contained in $\mathfrak{C}$) 
if for any object $(\W, \W ^\flat)$ of $\mathrm{DVR}  (\V,\V ^\flat)$, 
for any  smooth formal scheme $\X$ over $\W$
the category $\mathfrak{D} (\X)$ is a full subcategory of $\mathfrak{C} (\X)$.

\item We say that the data of coefficients $\mathfrak{C}$ is local 
if for any object $(\W, \W ^\flat)$ of $\mathrm{DVR}  (\V,\V ^\flat)$, 
for any  smooth formal scheme $\X$ over $\W$, 
for any open covering $(\X _i) _{i\in I}$ of $\X$, 
for any object $\E ^{(\bullet)}$ of
$\smash{\underrightarrow{LD}} ^{\mathrm{b}} _{\Q,\mathrm{qc}} ( \smash{\widehat{\D}} _{\X} ^{(\bullet)})$, 
we have 
$\E ^{(\bullet)}\in \mathrm{Ob} \mathfrak{C} (\X)$ if and only if 
$\E ^{(\bullet)}| \X _i \in \mathrm{Ob} \mathfrak{C} (\X _i)$ for any $i \in I$. 
For instance, it follows from 
\ref{thick-subcat}.\ref{thick-subcat2}
that 
the data of coefficients 
$\smash{\underrightarrow{LD}} ^{\mathrm{b}} _{\Q,\mathrm{coh}}$ is local. 

\item We say that the data of coefficients $\mathfrak{C}$ is quasi-local 
if for any object $(\W, \W ^\flat)$ of $\mathrm{DVR}  (\V,\V ^\flat)$, 
for any  smooth formal scheme $\X$ over $\W$, 
for any open immersion $j \colon \Y \hookrightarrow \X$ 
for any object 
$\E ^{(\bullet)}\in  \mathfrak{C} (\X)$, 
we have 
$ j ^{ !(\bullet)}\E ^{(\bullet)} \in  \mathfrak{C} (\Y)$.
\end{enumerate}

\end{dfn}

We finish the subsection with some notation.

\begin{empt}
[Duality]
\label{ntn-dual}
Let $\mathfrak{C}$ be a data of coefficients. We define its dual data of coefficients 
$\mathfrak{C}  ^{\vee}$ as follows: 
for any object $(\W, \W ^\flat)$ of $\mathrm{DVR}  (\V,\V ^\flat)$, 
for any   smooth formal scheme $\X$ over $\W$, 
the category $\mathfrak{C}  ^{\vee} (\X)$ is the subcategory of 
$\smash{\underrightarrow{LD}} ^{\mathrm{b}} _{\Q,\mathrm{coh}} ( \smash{\widehat{\D}} _{\X} ^{(\bullet)})$
of objects $\E^{(\bullet) }$
such that $\DD _{\X} (\E^{(\bullet) }) \in \mathfrak{C} (\X)$.
\end{empt}

\begin{ntn}
\label{dfn-Delta(C)}
Let $\mathfrak{C}$ be a  data of coefficients.
We denote by $\mathfrak{C} ^+$ 
the smallest data of coefficients containing $\fC$ and stable under shifts.
We define by induction on $n\in \N$ the data of coefficients
$\Delta _n (\mathfrak{C} )$  as follows: 
for $n=0$, we put $\Delta _0 (\mathfrak{C} )= \mathfrak{C} ^+$.
Suppose $\Delta _n (\mathfrak{C} )$ constructed for $n\in\N$. 
For any object $(\W, \W ^\flat)$ of $\mathrm{DVR}  (\V,\V ^\flat)$, 
for any  smooth formal scheme $\X$ over $\W$,  
the category 
$\Delta _{n+1} (\mathfrak{C} ) (\X)$
is the full subcategory of 
$\smash{\underrightarrow{LD}} ^{\mathrm{b}} _{\Q,\mathrm{coh}} ( \smash{\widehat{\D}} _{\X} ^{(\bullet)})$
of objects $\E^{(\bullet) }$
such that 
there exists an exact triangle of the form 
$\E^{(\bullet) } \to \FF^{(\bullet) } \to \G ^{(\bullet) } \to \E^{(\bullet) } [1]$ 
such that 
$\FF^{(\bullet) } $ 
and
$\G ^{(\bullet) } $ 
are objects of 
$\Delta _{n} (\mathfrak{C} ) (\X)$.
Finally, we put 
$\Delta  (\mathfrak{C} ): = \cup _{n\in\N} \Delta _n (\mathfrak{C} )$. 
The data of coefficients 
$\Delta  (\mathfrak{C} )$ is the smallest data of coefficients
containing $\mathfrak{C}$ and stable under devissage. 
\end{ntn}

\begin{ex}
\label{stab-cst}
Using the isomorphisms \ref{fonctYY'Gamma-iso}, and Theorem \ref{2.2.18}, we check 
that $\fB _\mathrm{cst} ^+$ satisfies $BK _+$, and is stable under local cohomological functors, 
extraordinary pull-backs 
and tensor products.
\end{ex}

The following lemma is obvious.
\begin{lemm}
\label{Delta-lemm-stab}
Let $\mathfrak{D}$ be a data of coefficients over $(\V,\V ^\flat)$. 
If $\mathfrak{D}$ is stable under pushforwards
(resp. extraordinary pullbacks, resp. smooth extraordinary  pullbacks, 
resp. tensor products, resp. base change, 
resp. local cohomological functors, resp. localisation outside a divisor)
then so is 
$\Delta(\mathfrak{D})$.
If $\mathfrak{D}$ satisfies $BK _+$  (resp. is quasi-local) then so is 
$\Delta(\mathfrak{D})$.
If $\mathfrak{D}$ satisfies $BK ^!$ and is stable under
local cohomological functors then so is 
$\Delta(\mathfrak{D})$.
\end{lemm}

\begin{empt}
Beware also that if $\mathfrak{D}$ is local (resp. stable under cohomology, 
resp. stable under special  descent of the base, resp. satisfies $BK ^!$), then it is not clear that 
so is $\Delta(\mathfrak{D})$.

\end{empt}

Since the converse of \ref{Delta-lemm-stab} is not true, let us introduce the following definition.
\begin{dfn}
\label{dfn-DeltaP}
Let $\mathfrak{D}$ be a data of coefficients over $(\V,\V ^\flat)$. 
Let $P$ be one of the stability property of 
\ref{dfn-stable-data}.
We say that $\mathfrak{D}$ is 
$\Delta$-stable under $P$ if there exists
a data of coefficients $\fD'$ over $(\V,\V ^\flat)$ 
such that 
$\Delta(\fD ') =\Delta(\mathfrak{D})$  and $\fD '$ is stable under $P$.

\end{dfn}

\begin{lem}
\label{Delta-lemm-stab-bis}
The data of coefficients $\mathfrak{D}$ is $\Delta$-stable under pushforwards
(resp. extraordinary pullbacks, resp. smooth extraordinary  pullbacks, 
resp. tensor products, resp. base change, 
resp. local cohomological functors, resp. localisation outside a divisor)
if and only if 
$\Delta(\mathfrak{D})$ is stable under pushforwards
(resp. extraordinary pullbacks, resp. smooth extraordinary  pullbacks, 
resp. tensor products, resp. base change, 
resp. local cohomological functors, resp. localisation outside a divisor).
The data of coefficients $\mathfrak{D}$ satisfies $\Delta$-$BK _+$  
(resp. is $\Delta$-quasi local) if and only if 
$\Delta(\mathfrak{D})$ satisfies $BK _+$
(resp. is quasi-local). 

\end{lem}

\begin{proof}
This is a translation of Lemma \ref{Delta-lemm-stab}. 
\end{proof}

Beware, it is not clear that if $\mathfrak{D}$ satisfies $\Delta$-$BK ^!$ and is $\Delta$-stable under
local cohomological functors then 
$\Delta(\mathfrak{D})$ satisfies $BK ^!$.

\subsection{Overcoherence, (over)holonomicity (after any base change) and complements}

\begin{dfn}
\label{dfnS(D,C)}
Let $\mathfrak{C}$ and $\mathfrak{D}$ be two data of coefficients.

\begin{enumerate}
\item We denote by 
$S _0 (\mathfrak{D}, \mathfrak{C})$
the data of coefficients defined as follows: 
for any object $(\W, \W ^\flat)$ of $\mathrm{DVR}  (\V,\V ^\flat)$, 
for any  smooth formal scheme $\X$ over $\W$,  
the category 
$S _0 (\mathfrak{D}, \mathfrak{C}) (\X)$
is the full subcategory of 
$\smash{\underrightarrow{LD}} ^{\mathrm{b}} _{\Q,\mathrm{coh}} ( \smash{\widehat{\D}} _{\X} ^{(\bullet)})$
of objects  $\E ^{(\bullet)}$
satisfying the following properties :
\begin{enumerate}
\item [($\star$)] for any smooth morphism $f\colon \Y \to \X$ of  smooth formal $\W$-schemes, 
for any object 
$\FF ^{(\bullet)}
\in
\mathfrak{D} (\Y)$,
we have 
$\FF ^{(\bullet)}
\smash{\widehat{\otimes}}^\L
_{\O  _{\Y}} f ^{!} (\E^{(\bullet) })
\in
\mathfrak{C} (\Y)$.
\end{enumerate}

\item We denote by 
$S(\mathfrak{D}, \mathfrak{C})$ 
the data of coefficients defined as follows: 
for any object $(\W, \W ^\flat)$ of $\mathrm{DVR}  (\V,\V ^\flat)$, 
for any  smooth formal scheme $\X$ over $\W$,  
the category 
$S  (\mathfrak{D}, \mathfrak{C}) (\X)$ 
is the full subcategory of 
$\smash{\underrightarrow{LD}} ^{\mathrm{b}} _{\Q,\mathrm{coh}} ( \smash{\widehat{\D}} _{\X} ^{(\bullet)})$
of objects  $\E ^{(\bullet)}$
satisfying the following property :
\begin{enumerate}
\item [($\star \star$)] for any morphism 
$(\W, \W ^\flat)
\to 
(\W ', \W ^{\prime \flat})$ 
of $\mathrm{DVR}  (\V,\V ^\flat)$, 
we have 
$ \W'  \smash{\widehat{\otimes}}^\L_{\W}  \E^{(\bullet) }\in 
S _0 (\mathfrak{D}, \mathfrak{C}) (\X \times _{\Spf \W} \Spf \W ')$.

\end{enumerate}
 \item Let $\sharp$ be a symbol so that 
either $S _\sharp = S _0$
or
$S _\sharp = S $.

\end{enumerate}

\end{dfn}

\begin{exs}
\label{ex-cst-surcoh}
\begin{enumerate}

\item
We denote by 
$ \smash{\underrightarrow{LD}} ^{\mathrm{b}} _{\Q,\mathrm{ovcoh}}= 
S _0 (\fB _\mathrm{div}, \smash{\underrightarrow{LD}} ^{\mathrm{b}} _{\Q,\mathrm{coh}})$
(see the second example  of \ref{ex-Dcst}).
This notion corresponds in the perfect residue fields case to that of overcoherence as defined in
\cite[5.4]{caro-stab-sys-ind-surcoh}.
We denote by 
$ \smash{\underrightarrow{LD}} ^{\mathrm{b}} _{\Q,\mathrm{oc}}= 
S (\fB _\mathrm{div}, \smash{\underrightarrow{LD}} ^{\mathrm{b}} _{\Q,\mathrm{coh}})$.
This notion corresponds in the perfect residue fields case to that of overcoherence after any base change as defined in
\cite{surcoh-hol}.

\item \label{hstab} 
We put 
$\mathfrak{H} _0 :=S (\fB _\mathrm{div}, \smash{\underrightarrow{LD}} ^{\mathrm{b}} _{\Q,\mathrm{coh}})$ 
and 
by induction on $i \in \N$, 
we put $\mathfrak{H} _{i+1} :=
\mathfrak{H} _{i} \cap S(\fB _\mathrm{div}, \mathfrak{H} _{i} ^{\vee})$
(see Notation \ref{ntn-dual}).
The coefficients of 
$\mathfrak{H} _{i}$ are called 
{\it $i$-overholonomic after any base change}.
We get the data of coefficients 
$\smash{\underrightarrow{LD}} ^{\mathrm{b}} _{\Q,\mathrm{h}}:= \mathfrak{H} _{\infty}
:=  \cap _{i\in \N} \mathfrak{H} _{i}$
whose objects are called {\it overholonomic after any base change}.

\item \label{ovholstab}
Replacing $S $ by $S _0$ in the definition of $\smash{\underrightarrow{LD}} ^{\mathrm{b}} _{\Q,\mathrm{h}}$, 
we get a data of coefficients that we will denote by 
$\smash{\underrightarrow{LD}} ^{\mathrm{b}} _{\Q,\mathrm{ovhol}}$.

\item Finally, 
we set
$ \smash{\underrightarrow{LM}}  _{\Q,\star}:= 
 \smash{\underrightarrow{LD}} ^{\mathrm{b}} _{\Q,\star} \cap 
  \smash{\underrightarrow{LM}} _{\Q,\mathrm{coh}}$,
  for $\star \in \{\mathrm{ovcoh}, \mathrm{oc}, \mathrm{h}, \mathrm{ovhol} \}$.
\end{enumerate}

\end{exs}

\begin{rem}
\label{rem-overhol}
\begin{enumerate}

\item Let $\mathfrak{C}$ be a data of coefficients.
The data of coefficients $\mathfrak{C}$ is stable under smooth extraordinary inverse images, localizations outside a divisor
(resp. under smooth extraordinary inverse images, localizations outside a divisor, and base change) if and only if 
$S _0 (\fB _\mathrm{div}, \mathfrak{C})=\mathfrak{C}$
(resp. $S  (\fB _\mathrm{div}, \mathfrak{C})=\mathfrak{C}$).

\item By construction, 
we remark that 
$\smash{\underrightarrow{LD}} ^{\mathrm{b}} _{\Q,\mathrm{ovhol}}$
is the biggest data of coefficients 
which contains 
$\fB _\mathrm{div}$, 
is stable by devissage, 
dual functors
and the operation
$S _{0} (\fB _\mathrm{div}, -)$. 
Moreover, 
$\smash{\underrightarrow{LD}} ^{\mathrm{b}} _{\Q,\mathrm{h}}$
is the biggest data of coefficients 
which contains 
$\fB _\mathrm{div}$, 
is stable by devissage, 
dual functors
and the operation
$S  (\fB _\mathrm{div}, -)$.

\end{enumerate}
\end{rem}

\

We will need later the following Lemmas.

\begin{lem}
\label{rem-div-cst}
Let $\mathfrak{C}$ be a data of coefficients stable under devissage.

\begin{enumerate}
\item 
\label{rem-div-cst1}
We have the equality 
$\Delta (\fB _\mathrm{div})=\Delta (\fB _\mathrm{cst})$ (see Notation \ref{ex-Dcst}).

\item \label{S(D,C)stability3-pre}
The data of coefficients $\mathfrak{C}$ is stable under local cohomological functors if and only if it 
is stable under localizations outside a divisor (see Definitions \ref{dfn-stable-data}).

\end{enumerate}
 
\end{lem}

\begin{proof}
Both  statements are checked by using exact triangles of localisation 
\ref{caro-stab-sys-ind-surcoh4.4.3}
and Mayer-Vietoris exact triangles \ref{eq1mayer-vietoris}.
\end{proof}

\begin{lem}
\label{lem-stabextpullback}
Let $\mathfrak{C}$ be a data of coefficients stable under local cohomological functors.
Then the data of coefficients 
$\mathfrak{C}$
is stable under smooth extraordinary pullbacks and  satisfies 
$BK ^!$ 
if and only if $\mathfrak{C}$ is stable under extraordinary pullbacks.

\end{lem}

\begin{proof}
Since the converse is obvious, 
let us check that if $\mathfrak{C}$
is stable under smooth extraordinary pullbacks and  satisfies 
$BK ^!$ 
then $\mathfrak{C}$ is stable under extraordinary pullbacks.
Let $(\W, \W ^\flat)$ be an object of $\mathrm{DVR}  (\V,\V ^\flat)$, 
$f \colon \Y \to \X$ be a morphism of  smooth formal schemes over $\W$, 
and $\E ^{(\bullet)}$ be an object of $\mathfrak{C} (\X)$. 
Since $f $ is the composition of its graph $\Y \hookrightarrow \Y\times \X$ 
followed by the projection 
$\Y \times \X \to \X$ which is smooth, 
then using the stability under smooth extraordinary pullbacks, we reduce to the case where 
$f$ is a closed immersion. 
From the stability under local cohomological functors, 
$\R \underline{\Gamma} ^\dag _{Y} \E ^{(\bullet)} \in \mathfrak{C} (\X)$. 
Since 
 $\mathfrak{C}$ satisfies $BK ^!$,
 then 
$f ^! \R \underline{\Gamma} ^\dag _{Y} \E ^{(\bullet)}
\in 
\mathfrak{C} (\Y)$.
We conclude using 
 the isomorphism 
$f ^! \R \underline{\Gamma} ^\dag _{Y} \E ^{(\bullet)} \riso 
f ^! ( \E ^{(\bullet)})$ 
(use \ref{2.2.18}). 
\end{proof}

\begin{lem}
\label{S(D,C)stability3bis}
Let $\mathfrak{D}$ be a data of coefficients over $(\V,\V ^\flat)$.
If $\mathfrak{D}$ contains $\fB _\mathrm{div}$,
and if $\mathfrak{D}$ is stable under 
tensor products,  
then 
$\mathfrak{D}$ is stable 
under localizations outside a divisor.
\end{lem}

\begin{proof}
This is a consequence of the isomorphism 
\ref{fonctYY'Gamma-iso} (we use the case where 
$\E ^{ (\bullet)}= \O _{\X} ^{(\bullet)}$).
\end{proof}

\begin{lem}
\label{stab-Dvee-3prop}
Let $\mathfrak{C}$ be a data of coefficients. 
If the data of coefficients $\mathfrak{C}$ is local (resp. is stable under devissages, resp. is stable under direct summands,
resp. is stable under special descent of the base, resp. is stable under pushforwards, 
resp. is stable under base change,
resp. satisfies $BK ^!$),
then so is $\mathfrak{C} ^{\vee}$  (see Notation \ref{ntn-dual}).

\end{lem}

\begin{proof}
The stability under pushforwards  is a  consequence of 
the relative duality isomorphism (see \ref{rel-dual-isom}).
Since the other stability properties are straighforward,
let us check the last one. 
Let  $(\W, \W ^\flat)$ be an object of $\mathrm{DVR}  (\V,\V ^\flat)$, 
  $u \colon \ZZ \hookrightarrow \X$ be a closed immersion of  smooth formal schemes over $\W$, 
and $\E ^{(\bullet)}$ be an object of $\mathfrak{C} ^\vee (\X)$ with support in $\ZZ$.
From Berthelot-Kashiwara theorem (see \ref{u!u+=id}),
there exists $\G ^{(\bullet)} \in 
\smash{\underrightarrow{LD}} ^{\mathrm{b}} _{\Q,\mathrm{coh}} ( \smash{\widehat{\D}} _{\ZZ} ^{(\bullet)})$
such that 
$u _+ ( \G ^{(\bullet)} ) \riso \E ^{(\bullet)} $.
Since $\DD _{\X} (\E^{(\bullet) }) \in \mathfrak{C} (\X)$ has his support in $\ZZ$, 
since $BK ^!$ property holds, we get 
$u ^! \DD _{\X} (\E^{(\bullet) }) \in  \mathfrak{C} (\ZZ)$.
From the relative duality isomorphism,
we get $\DD _{\X} (\E^{(\bullet) })
\riso
\DD _{\X} u _+ ( \G ^{(\bullet)} ) 
\riso u _+ (\DD _{\ZZ} ( \G ^{(\bullet)} ))$.
Hence, 
$u ^! u _+ (\DD _{\ZZ} ( \G ^{(\bullet)} )) \in  \mathfrak{C} (\ZZ)$.
From Berthelot-Kashiwara theorem (see \ref{u!u+=id}),
we have $u ^!u _+ (\DD _{\ZZ} ( \G ^{(\bullet)} ))\riso \DD _{\ZZ} ( \G ^{(\bullet)} )$.
This yields
$\DD _{\ZZ} ( \G ^{(\bullet)} )
\in  \mathfrak{C} (\ZZ)$.
Since
$u ^! (\E ^{(\bullet)} ) \riso u ^!u _+ (\G ^{(\bullet)} )\riso \G ^{(\bullet)} $, 
this implies that $u ^! (\E ^{(\bullet)} ) \in  \mathfrak{C} ^{\vee}(\ZZ)$.
\end{proof}

\begin{lem}
\label{stab-Dvee-3propbis}
Let $\mathfrak{C}$ and $\mathfrak{D}$ be two data of coefficients.
\begin{enumerate}
\item If $\mathfrak{D} \subset \mathfrak{C} $ then 
$\mathfrak{D} ^\vee \subset \mathfrak{C} ^\vee$.
\item We have the equality
$\Delta (\fC) ^\vee =\Delta (\fC ^\vee)$. 
\end{enumerate}
\end{lem}

\begin{proof}
The first statement is obvious.
Moreover, since 
$\fC \subset \Delta (\fC )$ then 
$\fC ^\vee
\subset 
\Delta (\fC) ^\vee $. From \ref{stab-Dvee-3prop}, 
$\Delta (\fC) ^\vee $ is stable under devissage. 
Hence
$\Delta (\fC ^\vee)
\subset 
\Delta (\fC) ^\vee $.
By replacing in the inclusion $\fC$ by
$\fC ^\vee$,
since $(\fC ^\vee) ^\vee = \fC$, we get
$\Delta (\fC) 
\subset 
\Delta (\fC ^\vee) ^\vee $.
Hence,
$\Delta (\fC) ^\vee
\subset 
(\Delta (\fC ^\vee) ^\vee ) ^\vee
=
\Delta (\fC ^\vee)
$.
\end{proof}

\begin{lem}
\label{preS(D,C)stability}
Let $\mathfrak{C}$ and $\mathfrak{D}$ be two data of coefficients.
With the notation of \ref{dfnS(D,C)}, 
we have the following properties.

\begin{enumerate}
\item 
\label{S(D,C)stability1}
With Notation \ref{ex-Dcst}, 
if $\mathfrak{D}$ contains $\fB _\emptyset$
then  $S _\sharp (\mathfrak{D}, \mathfrak{C})$
 is contained in  $ \mathfrak{C}$.
If  $\mathfrak{D}$ contains $\fB _\mathrm{div}$, then 
$S  _0(\mathfrak{D}, \mathfrak{C})$ 
(resp $S(\mathfrak{D}, \mathfrak{C})$)
is included in 
$\smash{\underrightarrow{LD}} ^{\mathrm{b}} _{\Q,\mathrm{ovcoh}}$
 (resp.  $\smash{\underrightarrow{LD}} ^{\mathrm{b}} _{\Q,\mathrm{oc}}$). 
 
 \item 
\label{S(D,C)stability1bis}
If $\fC \subset \fC '$ and $\fD ' \subset \fD$, then 
$S _\sharp (\mathfrak{D}, \mathfrak{C}) \subset
S _\sharp (\mathfrak{D}', \mathfrak{C}') $.

 \item 
 \label{S(D,C)stabilitynew3}
 If either $\mathfrak{C}$ or $\mathfrak{D}$ is stable under devissages (resp. shifts),
then so is
$S _\sharp (\mathfrak{D}, \mathfrak{C})$
and we have the equality
$S _\sharp (\Delta  (\mathfrak{D} ), \mathfrak{C})
=
S _\sharp (\mathfrak{D}, \mathfrak{C})$
(resp.
$S _\sharp (\mathfrak{D}^+, \mathfrak{C})
=
S _\sharp (\mathfrak{D}, \mathfrak{C})$).

\item \label{S(D,C)stability2}
Suppose that $\mathfrak{D} $ is stable under smooth extraordinary pullbacks, 
  tensor products (resp. and base change),
 and that $\mathfrak{C} $ contains $\mathfrak{D} $.
\begin{enumerate}
\item The data of coefficients
$S _0  (\mathfrak{D}, \mathfrak{C})$
contains $\mathfrak{D}$
(resp. $S   (\mathfrak{D}, \mathfrak{C})$
contains $\mathfrak{D}$).

\item If $\mathfrak{D}$ contains $\fB _\emptyset$, 
if either $\mathfrak{C} $ or $\mathfrak{D} $ is stable under shifts,
then 
$S _0 (\mathfrak{D}, \mathfrak{C})
=
S _0  \left (\mathfrak{D}, S _0 (\mathfrak{D}, \mathfrak{C}) \right)$ 
(resp. 
$S(\mathfrak{D}, \mathfrak{C})
=S \left (\mathfrak{D}, S(\mathfrak{D}, \mathfrak{C}) \right)$ ).

\item If either $\mathfrak{C} $ or $\mathfrak{D} $ is stable under shifts then
$S _0  \left (S _0 (\mathfrak{D}, \mathfrak{C}), S _0 (\mathfrak{D}, \mathfrak{C}) \right)$ 
(resp. $S \left (S(\mathfrak{D}, \mathfrak{C}), S(\mathfrak{D}, \mathfrak{C}) \right)$ )
contains $\mathfrak{D}$.

\end{enumerate}

\end{enumerate}
\end{lem}

\begin{proof}
The 
assertions 1), 2), 3), 4.a) and 4.b) are obvious. 
 Let us prove 4)c). 
Let us suppose moreover $\mathfrak{C}$ stable under shifts. 
Since tensor products and extraordinary inverse images commute with base change, 
to check the second part, we reduce to establish the non respective case.
Let  $(\W, \W ^\flat)$ be an object of $\mathrm{DVR}  (\V,\V ^\flat)$, 
 $\X$ be a  smooth formal scheme over $\W$, 
and   $\E ^{(\bullet)} \in \mathfrak{D} (\X)$. 
Let   $f\colon \Y \to \X$  be a smooth morphism of  smooth formal $\W$-schemes.
Let 
$\FF ^{(\bullet)}
\in
S _0 (\mathfrak{D}, \mathfrak{C})(\Y)$.
We have to check that 
$\FF ^{(\bullet)}
\smash{\widehat{\otimes}}^\L
_{\O  _{\Y}} f ^{!} (\E^{(\bullet) })
\in
S _0 (\mathfrak{D}, \mathfrak{C})(\Y)$.
Let $g\colon \ZZ \to \Y$  be a smooth morphism of  smooth formal $\W$-schemes, 
and $\G ^{(\bullet)}
\in
\mathfrak{D} (\ZZ)$.
We have the isomorphisms
\begin{align}
\notag
\G ^{(\bullet)}
\smash{\widehat{\otimes}}^\L
_{\O  _{\ZZ}} 
 g ^! \left (
 \FF ^{(\bullet)}
\smash{\widehat{\otimes}}^\L
_{\O  _{\Y}} 
f ^{!} (\E^{(\bullet) })
\right )
&
\underset{\ref{f!T'Totimes}}{\riso} 
\left ( \G ^{(\bullet)}
\smash{\widehat{\otimes}}^\L
_{\O  _{\ZZ}} 
 g ^!  \FF ^{(\bullet)} \right )
\smash{\widehat{\otimes}}^\L
_{\O  _{\ZZ}} 
(f \circ g) ^{!} (\E^{(\bullet) })
[-d _{Z/Y}]
\\
\notag
&
\riso
\left ( \G ^{(\bullet)}
\smash{\widehat{\otimes}}^\L
_{\O  _{\ZZ}} 
(f \circ g) ^{!} (\E^{(\bullet) })
 \right )
\smash{\widehat{\otimes}}^\L
_{\O  _{\ZZ}} 
 g ^!  \FF ^{(\bullet)}
 [-d _{Z/Y}]. 
 \hspace{2cm} 
 (\star)
\end{align}
Since $\mathfrak{D} $ is stable under smooth extraordinary pullbacks 
and tensor products, then 
$ \G ^{(\bullet)}
\smash{\widehat{\otimes}}^\L
_{\O  _{\ZZ}} 
(f \circ g) ^{!} (\E^{(\bullet) })
\in \mathfrak{D} (\ZZ)$.
Since 
$ \FF ^{(\bullet)} 
\in 
S _0 (\mathfrak{D}, \mathfrak{C})(\Y)$,
then
$\left ( \G ^{(\bullet)}
\smash{\widehat{\otimes}}^\L
_{\O  _{\ZZ}} 
(f \circ g) ^{!} (\E^{(\bullet) })
 \right )
\smash{\widehat{\otimes}}^\L
_{\O  _{\ZZ}} 
 g ^!  \FF ^{(\bullet)}
[-d _{Z/Y}]
 \in 
 \mathfrak{C} (\ZZ)$.
Hence, using $(\star)$ we conclude.
\end{proof}

\begin{rem}
\label{rem-div-cst2}
Let $\mathfrak{C}$,
$\mathfrak{D} $ be two data of coefficients.

\begin{enumerate}
[(a)]
\item 
\label{rem-div-cst21}
If $\fC$ is stable under devissages, then 
using \ref{preS(D,C)stability}.\ref{S(D,C)stabilitynew3}
and \ref{rem-div-cst}
we get
$S _\sharp (\fB _\mathrm{div},\mathfrak{C}) 
=
S _\sharp (\fB _\mathrm{cst} ^+,\mathfrak{C}) 
$.

\item \label{rem-div-cst21bis} 
Let $\fD'$ be a data of coefficients
such that 
$\Delta(\fD ') =\Delta(\mathfrak{D})$.
If $\fC$ is stable under devissages, then 
$S _\sharp (\mathfrak{D}', \mathfrak{C})
=
S _\sharp (\mathfrak{D}, \mathfrak{C})$.
Hence, in the case of stable properties appearing in Lemma \ref{Delta-lemm-stab-bis}
and when $\fC$ is stable under devissages,
to study 
$S _\sharp (\mathfrak{D}, \mathfrak{C})$ 
 it is enough to consider 
$\Delta$-stable properties  instead of stable properties satisfied by 
$\mathfrak{D}$ (e.g. see the beginning of the proof of \ref{ovcoh-invim-prop}).

\item \label{rem-div-cst22} If
$\mathfrak{D} $ is stable under smooth extraordinary pullbacks, 
  tensor products,
 and that $\mathfrak{D} $ contains $\fB _{\mathrm{div}}$ and is contained in $\mathfrak{C} $,
if moreover 
either $\mathfrak{C} $ or $\mathfrak{D} $ is stable under shifts, 
then using \ref{preS(D,C)stability} (1, 2 and 4.b), we get
\begin{equation}
\label{rem-div-cst22eq1}S _0 (\mathfrak{D}, \mathfrak{C})
=
S _0  \left (\mathfrak{D}, S _0 (\fB _{\mathrm{div}}, \mathfrak{C}) \right)
=
S _0  \left (\mathfrak{D}, S _0 (\mathfrak{D}, \mathfrak{C}) \right).
\end{equation}

 If moreover $\fD$ is stable under base change,
 then 
\begin{equation}
\label{rem-div-cst22eq2}
S  (\mathfrak{D}, \mathfrak{C})
=
S   \left (\mathfrak{D}, S  (\fB _{\mathrm{div}}, \mathfrak{C}) \right)
=S   \left (\mathfrak{D}, S  (\mathfrak{D}, \mathfrak{C}) \right).
\end{equation}

\end{enumerate}
\end{rem}

\begin{lem}
\label{S(D,C)stability}
Let $\mathfrak{C}$ and $\mathfrak{D}$ be two data of coefficients.
We have the following properties.

\begin{enumerate}

\item 
\label{S(D,C)stability3}
If  $\mathfrak{C}$ is local
and if $\mathfrak{D}$ is quasi-local
then 
$S _\sharp  (\mathfrak{D}, \mathfrak{C})$ is local. 
If  $\mathfrak{C}$ is stable under direct summands,
then so is
$S _\sharp  (\mathfrak{D}, \mathfrak{C})$.

\item \label{S(D,C)stability4}
The data of coefficients 
$S _0 (\mathfrak{D}, \mathfrak{C})$ 
(resp. $S(\mathfrak{D}, \mathfrak{C})$)
is stable under smooth extraordinary pullbacks 
(resp. and under base change).

\item 
\label{S(D,C)stability5}
If 
$\mathfrak{D}$ is stable under local cohomological functors (resp. localizations outside a divisor), 
then so is
$S _\sharp (\mathfrak{D}, \mathfrak{C})$.

\item \label{S(D,C)stability6}
Suppose that $\mathfrak{C}$ 
is stable under pushforwards and shifts. Suppose that
$\mathfrak{D}$ 
is stable under extraordinary pullbacks.   
Then the data of coefficients 
$S _\sharp  (\mathfrak{D}, \mathfrak{C})$
 are stable under pushforwards.

\item \label{S(D,C)stability7}
Suppose that $\mathfrak{C}$ is 
stable under shifts, and satisfies $BK ^!$.
Moreover, suppose that $\mathfrak{D}$ 
satisfies $BK _+$. 
Then 
the data of coefficients 
$S _\sharp (\mathfrak{D}, \mathfrak{C})$
satisfies $BK ^!$.

\end{enumerate}
\end{lem}

\begin{proof}
a) Using \ref{thick-subcat}  (beware that tensor products do not preserve coherence), we check that 
if $\mathfrak{C}$ is local
and if $\mathfrak{D}$ is quasi-local
then 
$S _\sharp  (\mathfrak{D}, \mathfrak{C})$ is local. 
The rest of the 
assertions 1) and 2) are obvious.

b) Let us check 3). From the commutation of the base change with local cohomological functors, we reduce to check that 
$S _0(\mathfrak{D}, \mathfrak{C})$ is stable under local cohomological functors (resp. localisations outside a divisor).
Using \ref{fonctYY'Gamma-iso} and the commutation of local cohomological functors with extraordinary inverse images 
(see \ref{2.2.18}), we check the desired properties.

c) Let us check 4). 
From the commutation of base changes with pushforwards 
(see \ref{comm-chg-base}),
we reduce to check the stability of 
$S _0 (\mathfrak{D}, \mathfrak{C}) $ under  
pushforwards.
Let $(\W, \W ^\flat)$ be an object of $\mathrm{DVR}  (\V,\V ^\flat)$. 
Let 
$g \colon \X '\to \X$ be a 
realizable morphism of smooth formal $\W$-schemes.
Let 
$\E ^{\prime (\bullet)} \in S _0 (\mathfrak{D}, \mathfrak{C}) (\X')$
with proper support over $X$. 
We have to check that 
$g _+ (\E^{\prime (\bullet) })
\in
S _0(\mathfrak{D}, \mathfrak{C})(\X)$.
Let $f \colon \Y \to \X$ be a smooth morphism of smooth formal $\W$-schemes.
Let  
$f '\colon \Y \times _{\X} \X '\to \X'$,  and 
$g '\colon \Y \times _{\X} \X '\to \Y$ be the structural projections. 
Let   
$\FF ^{(\bullet)}
\in
\mathfrak{D}(\Y)$.
We have to check 
$\FF ^{(\bullet)}
\smash{\widehat{\otimes}}^\L
_{\O  _{\Y}} f ^{!} g _+ (\E^{\prime (\bullet) })
\in
 \mathfrak{C}(\Y)$.
Using the hypotheses on 
$\mathfrak{C}$
and 
$\mathfrak{D}$,
via the isomorphisms
$$
\FF ^{(\bullet)}
\smash{\widehat{\otimes}}^\L _{\O  _{\Y}} 
f ^{!} g _+ (\E^{\prime (\bullet) })
\underset{\ref{basechange}}{\riso} 
\FF ^{(\bullet)}
\smash{\widehat{\otimes}}^\L _{\O  _{\Y}} 
g '_{+}  f ^{\prime !} (\E ^{\prime (\bullet)} )
\underset{\ref{surcoh2.1.4-iso}}{\riso} 
g '_{+} \left ( 
  g ^{\prime !} (\FF ^{(\bullet)})
\smash{\widehat{\otimes}}^\L _{\O  _{\Y}} 
  f ^{\prime !} (\E ^{\prime (\bullet)} )
  \right )
  [-d _{X'/X}],$$
we check  that 
$  \FF ^{(\bullet)}
\smash{\widehat{\otimes}}^\L _{\O  _{\Y}} 
f ^{!} g _+ (\E^{\prime (\bullet) }) \in
 \mathfrak{C}(\Y)$.

d) Let us check 5) (we might remark the similarity with the proof of \cite[3.1.7]{caro_surcoherent}). 
Since extraordinary pullbacks commute with base change, 
we reduce to check that 
$S _0(\mathfrak{D}, \mathfrak{C})$
satisfies $BK ^!$.
Let $(\W, \W ^\flat)$ be an object of $\mathrm{DVR}  (\V,\V ^\flat)$, 
and $u\colon \X \hookrightarrow \fP$ be a closed immersion of  smooth formal schemes  over $\W$. 
Let $\E ^{(\bullet)} \in S _0  (\mathfrak{D}, \mathfrak{C}) (\fP)$ with support in $\X$.
We have to check that $u ^! (\E ^{(\bullet)} ) \in S _0  (\mathfrak{D}, \mathfrak{C}) (\X)$.
We already know that
$u ^! (\E ^{(\bullet)})
\in 
\smash{\underrightarrow{LD}} ^{\mathrm{b}} _{\Q,\mathrm{coh}} ( \smash{\widehat{\D}} _{\X} ^{(\bullet)})$
(thanks to Berthelot-Kashiwara theorem
\ref{u!u+=id}).
Let $f\colon \Y \to \X$ be a smooth morphism of  smooth formal $\W$-schemes, 
and $\FF ^{(\bullet)}
\in
\mathfrak{D} (\Y)$.
We have to check 
$\FF ^{(\bullet)}
\smash{\widehat{\otimes}}^\L
_{\O  _{\Y}} f ^{!} (u ^! \E^{(\bullet) })
\in
\mathfrak{C} (\Y)$.
The morphism $f$ is the composite of its graph
$\Y \hookrightarrow \Y \times \X $
with the projection  
$\Y \times \X  \to \X$.
We denote by $v$ the composite of 
$\Y \hookrightarrow \Y \times \X $
with 
$id \times u \colon \Y \times \X \hookrightarrow \Y \times\fP$.
Let $g \colon \Y \times\fP \to \fP$ be the projection. 
Set $\U:=  \Y \times\fP$.
Since 
$\mathfrak{D} $ 
satisfies $BK _+$, 
then
$v _+ (\FF ^{(\bullet)}) \in \mathfrak{D} ( \U)$.
Since 
$\E ^{(\bullet)} \in S _0  (\mathfrak{D}, \mathfrak{C}) (\fP)$
and $g$ is smooth,
this yields 
$v _+ (\FF ^{(\bullet)})
\smash{\widehat{\otimes}}^\L
_{\O  _{\U}}  g ^{!} (\E^{(\bullet) })
\in \mathfrak{C} (\U)$.
Since $\mathfrak{C} $  satisfies $BK ^!$, 
this implies 
$v ^!\left (v _+ (\FF ^{(\bullet)})
\smash{\widehat{\otimes}}^\L
_{\O  _{\U}} g ^{!} (\E^{(\bullet) }) \right ) 
\in \mathfrak{C} (\Y)$.
Since 
$v ^!\left (v _+ (\FF ^{(\bullet)})
\smash{\widehat{\otimes}}^\L
_{\O  _{\U}}   g ^{!} (\E^{(\bullet) }) \right ) 
\riso 
v ^! v _+ (\FF ^{(\bullet)})
\smash{\widehat{\otimes}}^\L
_{\O  _{\Y}} 
v ^! g ^{!} (\E^{(\bullet) }) [r]$ with $r$ an integer (see \ref{f!T'Totimes}), 
since 
$v ^! v _+ (\FF ^{(\bullet)}) \riso 
\FF ^{(\bullet)}$
(see Berthelot-Kashiwara theorem
\ref{u!u+=id}),
since $\mathfrak{C}$ is stable under shifts, 
since by transitivity
$v ^! g ^{!}  \riso f ^! u ^!$,
we get
$\FF ^{(\bullet)}
\smash{\widehat{\otimes}}^\L
_{\O  _{\Y}} 
f ^! u ^! (\E^{(\bullet) })
\in 
\mathfrak{C} (\Y)$.
\end{proof}

\begin{prop}
\label{ovcoh-invim-prop}
Let $\mathfrak{C}$ and $\mathfrak{D}$ be two data of coefficients.
We suppose $\fD$ contains $\fB _\mathrm{div}$, 
satisfies $\Delta$-$BK _+$, and is $\Delta$-stable under 
extraordinary pullbacks
and tensor products
(resp. $\fD$ contains $\fB _\emptyset$, 
satisfies $\Delta$-$BK _+$, and is $\Delta$-stable under 
extraordinary pullbacks 
and
local cohomological functors).
 We suppose 
 $\fC$ is local,
satisfies $BK ^!$, is stable under devissages, pushforwards, and direct summands.

In both cases, the data of coefficients 
$S _0 (\mathfrak{D}, \mathfrak{C})$ 
(resp. $S(\mathfrak{D}, \mathfrak{C})$)
is local, stable under devissages, direct summands, 
local cohomological functors,
extraordinary pullbacks, pushforwards (resp. and base change).

\end{prop}

\begin{proof}
Let us check the non respective case.
Using
\ref{Delta-lemm-stab-bis}, 
$\Delta (\fD)$ satisfies the same properties than 
$\fD$ without the symbol $\Delta$.
Following
\ref{preS(D,C)stability}.\ref{S(D,C)stabilitynew3},
since $\fC$ is stable under devissage, 
then $S _\sharp (\mathfrak{D}, \mathfrak{C})=
S _\sharp (\Delta (\fD), \mathfrak{C})$.
Hence, we reduce to the case 
$\Delta (\fD)=\fD$.
Using 
\ref{S(D,C)stability} (except \ref{S(D,C)stability5}),
we get $S _0 (\mathfrak{D}, \mathfrak{C})$ 
(resp. $S(\mathfrak{D}, \mathfrak{C})$)
is local, is stable under devissages, direct summands, 
smooth extraordinary pullbacks  (resp. and base change), pushforwards, 
satisfies $BK ^!$.
Using 
\ref{S(D,C)stability3bis},
we check $\fD$ is stable under localizations outside a divisor.
This yields by \ref{rem-div-cst}.\ref{S(D,C)stability3-pre} that
$\fD$ is stable under under local cohomological functors.
Hence, using
\ref{S(D,C)stability}.\ref{S(D,C)stability5}
so is $S _\sharp (\mathfrak{D}, \mathfrak{C})$. 
By applying 
\ref{lem-stabextpullback} 
this yields that 
 $S _\sharp (\mathfrak{D}, \mathfrak{C})$ 
 is stable under 
 extraordinary pullbacks.
Similarly, 
the respective case is a straightforward consequence of 
\ref{lem-stabextpullback}  and \ref{S(D,C)stability}.
\end{proof}

\begin{coro}
\label{ovcoh-invim}
Let $i \in \N \cup \{ \infty\}$.
The data of coefficients
$ \smash{\underrightarrow{LD}} ^{\mathrm{b}} _{\Q,\mathrm{ovcoh}}$
(resp. $ \smash{\underrightarrow{LD}} ^{\mathrm{b}} _{\Q,\mathrm{oc}}$,
resp. 
$\mathfrak{H} _{i}$)
contains
$\fB _\mathrm{cst}$,
is local, stable under devissages, direct summands, 
local cohomological functors,
extraordinary pullbacks, pushforwards (resp. and base change).
Moreover, 
$\smash{\underrightarrow{LD}} ^{\mathrm{b}} _{\Q,\mathrm{h}}$
is stable under duality.
\end{coro}

\begin{proof}
For any $i \in \N $,
since 
$\mathfrak{H} _{i+1}\subset
\mathfrak{H} _{i}\cap \mathfrak{H} _{i} ^\vee$,
we get the stability under duality 
of $\mathfrak{H} _{\infty}$.
Using \ref{rem-div-cst2}.\ref{rem-div-cst21},
we get 
$ \smash{\underrightarrow{LD}} ^{\mathrm{b}} _{\Q,\mathrm{ovcoh}}
=
S _0 (\fB _\mathrm{cst} ^+, \smash{\underrightarrow{LD}} ^{\mathrm{b}} _{\Q,\mathrm{coh}})$
(resp. 
$ \smash{\underrightarrow{LD}} ^{\mathrm{b}} _{\Q,\mathrm{oc}}
=
S  (\fB _\mathrm{cst} ^+, \smash{\underrightarrow{LD}} ^{\mathrm{b}} _{\Q,\mathrm{coh}})$,
resp. 
$S(\fB _\mathrm{div}, \mathfrak{H} _{i} ^{\vee})
=S(\fB _{\mathrm{cst}}  ^+, \mathfrak{H} _{i} ^{\vee})$).
Since $\fB _\mathrm{cst} ^+$ satisfies $BK _+$, and is stable under local cohomological functors, 
extraordinary pullbacks
and tensor products (see \ref{stab-cst}), 
since 
$\smash{\underrightarrow{LD}} ^{\mathrm{b}} _{\Q,\mathrm{coh}}$ is local,
satisfies $BK ^!$, is stable under devissages, pushforwards, direct summands
then
we conclude by applying 
\ref{stab-Dvee-3prop} and \ref{ovcoh-invim-prop}.
\end{proof}

\subsection{On the stability under cohomology}

\begin{ntn}
\label{ntnC0}
Let $\mathfrak{C}$ be a data of coefficients.
We denote by $\mathfrak{C} ^0 $ the data of coefficients
defined as follows. 
Let $(\W, \W ^\flat)$ be an object of $\mathrm{DVR}  (\V,\V ^\flat)$,
$\X$ be a smooth formal scheme  over $\W$.
Then
$\mathfrak{C} ^0 (\X)
:=
\mathfrak{C}(\X)
\cap
\smash{\underrightarrow{LM}}  _{\Q,\mathrm{coh}} 
( \smash{\widehat{\D}} _{\X} ^{(\bullet)})
$.
\end{ntn}

\begin{lem}
\label{ntnC0-lemm}
Let $\mathfrak{C}$ be a data of coefficients.
Let $(\W, \W ^\flat)$ be an object of $\mathrm{DVR}  (\V,\V ^\flat)$,
$\X$ be a smooth formal scheme  over $\W$.
\begin{enumerate}
\item 
\label{ntnC0-lemm1}
If $\mathfrak{C}$ is stable under cohomology, then 
$\Delta (\mathfrak{C})
= 
\Delta (\mathfrak{C} ^0) $.
\item 
\label{ntnC0-lemm2}
If $\mathfrak{C}$ is stable under devissages and cohomology,
then the category $\mathfrak{C} ^0 (\X)$
is an abelian strictly full subcategory of 
$\smash{\underrightarrow{LM}}  _{\Q,\mathrm{coh}} 
( \smash{\widehat{\D}} _{\X} ^{(\bullet)})$ which is 
stable under extensions.

\end{enumerate}

\end{lem}

\begin{proof}
1) Since $\Delta (\mathfrak{C} ^0) 
\subset 
\Delta (\mathfrak{C})$, it remains to check 
$\mathfrak{C}
\subset
\Delta (\mathfrak{C} ^0) $.
Let 
$\E ^{(\bullet)}
\in 
\mathfrak{C}(\X)$. 
By using some exact triangles of truncations 
(for the canonical t-structure on 
$ \smash{\underrightarrow{LD}} ^{\mathrm{b}} _{\Q,\mathrm{coh}} 
( \smash{\widehat{\D}} _{\X} ^{(\bullet)})$ as explained in \ref{t-structure-coh}),
we check 
$\E ^{(\bullet)}
\in \Delta _n (\mathfrak{C} ^0) (\X)$
where $n$ is the cardinal of $\{ j\in \Z \;;\; \mathcal{H} ^j  (\E ^{(\bullet)}) \not = 0\}$.

2) Let 
$f \colon 
\E ^{(\bullet)}
\to 
\FF ^{(\bullet)}$
be a morphism of 
$\mathfrak{C} ^0 (\X)$.
By considering the mapping cone of $f$ and by using the stability properties of $\fC$, 
we check that $\Ker f$ and $\Coker f$ are objects
of 
$\mathfrak{C} ^0 (\X)$.
The stability under extensions is obvious. Hence, we are done.
\end{proof}

\begin{prop}
\label{Bcst-st-cohom}
Let $\mathfrak{C}$ 
be a data of coefficients.
Suppose that 
$\fC$ is stable under cohomology, and devissage.
Then $S _\sharp (\fB _\mathrm{cst}^+ , \mathfrak{C})
=S _\sharp (\fB _\mathrm{div}, \mathfrak{C})$ is stable under devissages  and cohomology.
\end{prop}

\begin{proof}
Following \ref{rem-div-cst2},
$S _\sharp (\fB _\mathrm{cst} ^+, \mathfrak{C})
=
S _\sharp (\fB _\mathrm{div}, \mathfrak{C})$.
Since localizations outside a divisor and the functor $f ^{(\bullet)*}$ when $f$ is any smooth morphism
are t-exact (for the canonical t-structure of $\smash{\underrightarrow{LD}}  ^\mathrm{b} _{\Q, \mathrm{coh}}$), 
then this is straightforward.
\end{proof}

\begin{coro}
\label{coro-ovcoh-oc-tstr}
The data of coefficients 
$\smash{\underrightarrow{LD}} ^{\mathrm{b}} _{\Q,\mathrm{ovcoh}}$,
and 
$\smash{\underrightarrow{LD}} ^{\mathrm{b}} _{\Q,\mathrm{oc}}$ 
are stable under cohomology.
\end{coro}

\begin{lem}
\label{lem-boxtimes}
Let $\mathfrak{C}$ be a data of coefficients stable under devissages  and cohomology and which is local.
Let $(\W, \W ^\flat)$ be an object of $\mathrm{DVR}  (\V,\V ^\flat)$, 
$\fT: =\Spf \, \W$, 
$\X$ and $\Y$ be two smooth formal schemes  over $\W$, 
$\E ^{0(\bullet)}
\in 
\smash{\underrightarrow{LM}} _{\Q,\mathrm{coh}} 
( \smash{\widehat{\D}} _{\X} ^{(\bullet)})$,
$\E ^{(\bullet)}
\in 
\smash{\underrightarrow{LD}} ^{\mathrm{b}} _{\Q,\mathrm{coh}} 
( \smash{\widehat{\D}} _{\X} ^{(\bullet)})$,
and 
$\FF ^{(\bullet)}
\in 
\smash{\underrightarrow{LD}} ^{\mathrm{b}} _{\Q,\mathrm{coh}} 
( \smash{\widehat{\D}} _{\Y} ^{(\bullet)})$.
Recall exterior tensor products is defined in \ref{boxtimesLDcoh}.
\begin{enumerate}
\item \label{lem-boxtimes1} 
The following properties are equivalent 
\begin{enumerate}
\item for any $n\in \Z$,
$\E ^{0(\bullet)} \smash{\widehat{\boxtimes}} 
^\L _{\O _{\fT }} \H ^n (\FF ^{(\bullet)}) \in \fC ^0  (\X)$ ;
\item $\E ^{0(\bullet)} \smash{\widehat{\boxtimes}} 
^\L _{\O _{\fT }} \FF ^{(\bullet)} \in \fC (\X)$.
\end{enumerate}
\item 
\label{lem-boxtimes2}
If for any $n\in \Z$ we have
$\E ^{(\bullet)} \smash{\widehat{\boxtimes}} 
^\L _{\O _{\fT }} \H ^n (\FF ^{(\bullet)}) \in \fC (\X)$,
then 
$\E ^{(\bullet)} \smash{\widehat{\boxtimes}} 
^\L _{\O _{\fT }} \FF ^{(\bullet)} \in \fC (\X)$.
\end{enumerate}
\end{lem}

\begin{proof}
Since this is local, we can suppose $\X$ affine.
The first statement is an obvious consequence of \ref{prop-boxtimes}.\ref{prop-boxtimes2}.
Following \ref{cor-eq-cat-coh-m-2},
there exists 
$\G  ^{ (\bullet)}  \in 
D ^{\mathrm{b}}  (\underrightarrow{LM} _{\Q,\mathrm{coh}} (\smash{\widehat{\D}} _{\Y /\S } ^{(\bullet)} ))$
such that  $\G  ^{ (\bullet)}  $ 
is isomorphic in 
$\smash{\underrightarrow{LD}}  ^\mathrm{b} _{\Q, \mathrm{coh}}
( \smash{\widehat{\D}} _{\Y ^{ }/\S }  ^{(\bullet)} )$
to 
$\FF  ^{ (\bullet)} $.
Following \ref{prop-boxtimes}.\ref{prop-boxtimes1},
we get
the spectral sequence in 
$\underrightarrow{LM} _{\Q,\mathrm{coh}} (\smash{\widehat{\D}} _{\X \times \Y /\S } ^{(\bullet)})$
of the form
$\H ^r (\E  ^{ (\bullet)} )
\smash{\widehat{\boxtimes}} 
^\L _{\O _{\T }}
\H ^s
 ( 
\G  ^{ (\bullet)}
)
=:
E _{2} ^{r,s}
\Rightarrow
E ^n :=
\H ^n 
\left ( 
\E  ^{ (\bullet)} 
\smash{\widehat{\boxtimes}} 
^\L _{\O _{\T }}
\G  ^{ (\bullet)}
\right ) $.
Since 
$\fC ^0 (\X \times \Y )$
is an abelian strictly full subcategory of 
$\underrightarrow{LM} _{\Q,\mathrm{coh}} (\smash{\widehat{\D}} _{\X \times \Y /\S } ^{(\bullet)})$ closed under
extensions (see \ref{ntnC0-lemm}.\ref{ntnC0-lemm2}),
since 
$\H ^r (\E  ^{ (\bullet)} )
\smash{\widehat{\boxtimes}} 
^\L _{\O _{\T }}
\H ^s
 ( 
\G  ^{ (\bullet)}
)
\in 
\fC ^0 (\X \times \Y )$, then 
$\H ^n 
\left ( 
\E  ^{ (\bullet)} 
\smash{\widehat{\boxtimes}} 
^\L _{\O _{\T }}
\G  ^{ (\bullet)}
\right ) 
\in 
\fC ^0 (\X \times \Y )$.
Using \ref{diag-Hn-comp-coh}, 
this yields, 
$\H ^n 
\left ( 
\E  ^{ (\bullet)} 
\smash{\widehat{\boxtimes}} 
^\L _{\O _{\T }}
\FF  ^{ (\bullet)}
\right ) 
\in 
\fC ^0 (\X \times \Y )$.
\end{proof}

\begin{prop}
\label{prop-st-cohom}
Let $\mathfrak{C}$ and $\mathfrak{D}$ be two data of coefficients.
Suppose that $\fD$ is stable under cohomology, smooth extraordinary pullbacks,
and that $\fC$ is local and stable under cohomology, devissage, extraordinary pullbacks.
Then $S _\sharp (\mathfrak{D}, \mathfrak{C})$ is stable under devissages  and cohomology.
\end{prop}

\begin{proof}
1) We prove the case where $\sharp = 0$.
Let $(\W, \W ^\flat)$ be an object of $\mathrm{DVR}  (\V,\V ^\flat)$,
$\X$ be a smooth formal scheme  over $\W$, 
$\E ^{(\bullet)}\in 
\smash{\underrightarrow{LD}} ^{\mathrm{b}} _{\Q,\mathrm{coh}} 
( \smash{\widehat{\D}} _{\X} ^{(\bullet)})$.
Let 
$f\colon \Y \to \X$ 
be a smooth morphism of smooth formal $\W$-schemes, 
$\FF ^{(\bullet)}
\in
\mathfrak{D} (\Y)$.

a) Suppose that 
$\E ^{(\bullet)}\in 
S  _0 (\mathfrak{D}, \mathfrak{C}) (\X)$.
Since 
$\fD$ is stable under cohomology and smooth extraordinary pullbacks,
then 
$\H ^r (\FF ^{(\bullet)})
\smash{\widehat{\boxtimes}}^\L
_{\W} f ^{!} (\E ^{(\bullet) })
\in 
\mathfrak{C}
( \Y \times _{\Spf \W} \Y)$, for any $r \in \Z$.
Since 
$\fC$ is stable under cohomology and devissage, then
using \ref{lem-boxtimes}.\ref{lem-boxtimes1},
we get 
$\H ^r (\FF ^{(\bullet)})
\smash{\widehat{\boxtimes}}^\L
_{\W} 
\H ^s (f ^{!(\bullet)} (\E ^{(\bullet) }))
\in 
\mathfrak{C} ^{0}
( \Y \times _{\Spf \W} \Y)$,
for any $r ,s\in \Z$.
Hence, since 
$\H ^s (f ^{!(\bullet)} (\E ^{(\bullet) }))
\riso 
f ^{*(\bullet)} \H ^{s-d _Y} ( \E ^{(\bullet) })$, 
using \ref{lem-boxtimes}.\ref{lem-boxtimes1},
we get 
$\FF ^{(\bullet)}
\smash{\widehat{\boxtimes}}^\L
_{\W} 
f ^{*(\bullet)} \H ^s (\E ^{(\bullet) })
\in 
\mathfrak{C}
( \Y \times _{\Spf \W} \Y)$,
for any $s\in \Z$.
Since  $\fC$ is stable under extraordinary pullbacks  and shifts,
this yields 
$\FF ^{(\bullet)}
\smash{\widehat{\otimes}}^\L
_{\O  _{\Y}} 
f ^{*(\bullet)} \H ^s ( \E ^{(\bullet) }))
\riso 
\L \delta ^{*(\bullet) }
(\FF ^{(\bullet)}
\smash{\widehat{\boxtimes}}^\L
_{\W} 
f ^{*(\bullet)} \H ^s ( \E ^{(\bullet) }))
\in
\mathfrak{C} (\Y)$,
where 
$\delta \colon \Y \hookrightarrow \Y \times _{\Spf \W} \Y$ 
is the diagonal immersion. 
Hence, 
$\H ^s ( \E ^{(\bullet) })
\in 
S  _0 (\mathfrak{D}, \mathfrak{C}) (\X)$,
for any $s\in \Z$.

b ) Conversely, suppose
$\H ^s ( \E ^{(\bullet) })
\in 
S  _0 (\mathfrak{D}, \mathfrak{C}) (\X)$,
for any $s\in \Z$.
Then 
$\FF ^{(\bullet)}
\smash{\widehat{\boxtimes}}^\L
_{\W} 
f ^{*(\bullet)} \H ^s (\E ^{(\bullet) })
\in 
\mathfrak{C}
( \Y \times _{\Spf \W} \Y)$.
Using \ref{lem-boxtimes}.\ref{lem-boxtimes2},
this yields 
$\FF ^{(\bullet)}
\smash{\widehat{\boxtimes}}^\L
_{\W} f ^{!(\bullet)} (\E ^{(\bullet) })
\in 
\mathfrak{C}
( \Y \times _{\Spf \W} \Y)$.
Hence, 
$\FF ^{(\bullet)}
\smash{\widehat{\otimes}}^\L
_{\O  _{\Y}} 
f ^{!(\bullet)}( \E ^{(\bullet) })
\in
\mathfrak{C} (\Y)$.

2) 
For any morphism   $(\W, \W ^\flat)
\to 
(\widetilde{\W}, \widetilde{\W} ^{\flat})$ 
of 
$\mathrm{DVR}  (\V,\V ^\flat)$, 
set $\widetilde{\X} := \X \times _{\Spf \W} \Spf \widetilde{\W}$,
and
$\widetilde{\E} ^{(\bullet)}:= \widetilde{\W}  \smash{\widehat{\otimes}}^\L_{\W}  \E^{(\bullet) }$.
The property 
$\E ^{(\bullet)}\in S  (\mathfrak{D}, \mathfrak{C})  (\X)$ 
is equivalent to the property
$\widetilde{\E} ^{(\bullet)}\in 
S  _0 (\mathfrak{D}, \mathfrak{C}) (\widetilde{\X})$ 
(for any such morphism
$(\W, \W ^\flat)
\to 
(\widetilde{\W}, \widetilde{\W} ^{\flat})$).
The property
$\widetilde{\E} ^{(\bullet)}\in 
S  _0 (\mathfrak{D}, \mathfrak{C}) (\widetilde{\X})$ 
is equivalent from the first part to 
$\H ^n (\widetilde{\E} ^{(\bullet)}) \in 
S  _0 (\mathfrak{D}, \mathfrak{C}) (\widetilde{\X})$ for any $n\in \Z$. 
Since the cohomology commutes with base change, 
this latter property is equivalent to 
$ \widetilde{\W}  \smash{\widehat{\otimes}}^\L_{\W} \H ^n (\E^{(\bullet)}) 
\in 
S  _0 (\mathfrak{D}, \mathfrak{C}) (\widetilde{\X})$ for any $n\in \Z$, i.e.
$\H ^n (\E)
\in 
S   (\mathfrak{D}, \mathfrak{C}) (\widetilde{\X})$ for any $n\in \Z$. 
 
\end{proof}

\begin{empt}
Let $(\W, \W ^\flat)$ be an object of $\mathrm{DVR}  (\V,\V ^\flat)$,
$\X$ be a smooth formal scheme  over $\W$, 
$\E ^{(\bullet)}\in 
\smash{\underrightarrow{LM}}  _{\Q,\mathrm{coh}} 
( \smash{\widehat{\D}} _{\X} ^{(\bullet)})$.
Following \ref{ntn-dualfunctor}, 
we have the dual functor
$\DD ^{(\bullet)}
\colon 
\smash{\underrightarrow{LD}} ^{\mathrm{b}} _{\Q,\mathrm{coh}}
( \widehat{\D} _{\X /\S } ^{(\bullet)})
\to 
\smash{\underrightarrow{LD}} ^{\mathrm{b}} _{\Q,\mathrm{coh}}
( \widehat{\D} _{\X /\S } ^{(\bullet)})$.
Similarly to \cite[2.8]{caro-holo-sansFrob},
we say that 
$\E ^{(\bullet)}$ is holonomic if 
for any $i \not = 0$, 
$\H ^i (\DD ^{(\bullet)}(\E ^{(\bullet)})) =0$.
We denote by 
$\smash{\underrightarrow{LM}}  _{\Q,\mathrm{hol}} ( \widehat{\D} _{\X /\S } ^{(\bullet)})$ the strictly subcategory 
of 
$\smash{\underrightarrow{LM}}  _{\Q,\mathrm{coh}} ( \widehat{\D} _{\X /\S } ^{(\bullet)})$
of holonomic 
$\smash{\widehat{\D}} _{\X} ^{(\bullet)}$-modules.
By copying \cite[2.14]{caro-holo-sansFrob}, we check 
$\smash{\underrightarrow{LM}}  _{\Q,\mathrm{hol}} ( \widehat{\D} _{\X /\S } ^{(\bullet)})$ 
is in fact a Serre subcategory 
of 
$\smash{\underrightarrow{LM}}  _{\Q,\mathrm{coh}} ( \widehat{\D} _{\X /\S } ^{(\bullet)})$.

We denote by 
$\smash{\underrightarrow{LD}} ^{\mathrm{b}} _{\Q,\mathrm{hol}}
( \widehat{\D} _{\X /\S } ^{(\bullet)})$ the strictly full subcategory of
$\smash{\underrightarrow{LD}} ^{\mathrm{b}} _{\Q,\mathrm{coh}}
( \widehat{\D} _{\X /\S } ^{(\bullet)})  $ 
consisting of complexes 
$\E ^{(\bullet)}$ such that 
$\H ^n \E ^{(\bullet)} \in 
\smash{\underrightarrow{LM}}  _{\Q,\mathrm{hol}} ( \widehat{\D} _{\X /\S } ^{(\bullet)})$
for any 
$n\in\Z$.
This yields the t-exact equivalence of categories
$\DD ^{(\bullet)}
\colon 
\smash{\underrightarrow{LD}} ^{\mathrm{b}} _{\Q,\mathrm{hol}}
( \widehat{\D} _{\X /\S } ^{(\bullet)})
\cong 
\smash{\underrightarrow{LD}} ^{\mathrm{b}} _{\Q,\mathrm{hol}}
( \widehat{\D} _{\X /\S } ^{(\bullet)})$.
Copying word by word the proof of 
\cite[3.3.5]{surcoh-hol},
we check that 
$\smash{\underrightarrow{LM}}  _{\Q,\mathrm{oc}} 
\subset
\smash{\underrightarrow{LM}}  _{\Q,\mathrm{hol}}$.
This yields
\begin{equation}
\label{ocinchol}
\smash{\underrightarrow{LD}} ^{\mathrm{b}} _{\Q,\mathrm{oc}} 
\subset
\smash{\underrightarrow{LD}} ^{\mathrm{b}}  _{\Q,\mathrm{hol}}.
\end{equation}
\end{empt}

\begin{lem}
\label{cor-oc-inc-hol}
Let $\mathfrak{C}$ be a data of coefficients stable under cohomology and included in 
$\smash{\underrightarrow{LD}} ^{\mathrm{b}} _{\Q,\mathrm{oc}}$.
Then $\mathfrak{C} ^{\vee}$ is stable under cohomology.
\end{lem}

\begin{proof}
This is a straightforward consequence of
\ref{ocinchol}.
\end{proof}

\begin{coro}
\label{coro-h-tstr}
Let $i \in \N \cup \{ \infty\}$.
The data of coefficients
$\mathfrak{H} _{i}$
is stable under cohomology.
\end{coro}

\begin{empt}
\label{t-structure-ovcoh-oc-hol-h}
Let $\mathfrak{C}$ be a data of coefficients stable under devissages  and cohomology. 
Let $(\W, \W ^\flat)$ be an object of $\mathrm{DVR}  (\V,\V ^\flat)$,
$\X$ be a smooth formal scheme  over $\W$.
Recall that following 
\ref{t-structure-coh}
we have a canonical t-structure on 
$ \smash{\underrightarrow{LD}} ^{\mathrm{b}} _{\Q,\mathrm{coh}} 
( \smash{\widehat{\D}} _{\X} ^{(\bullet)})$.
We get a canonical t-structure on 
$\fC (\X/\W)$ whose heart is $\fC  ^0  (\X/\W)$ 
and so that the t-structure of 
$\fC (\X/\W)$ is induced 
by that of 
$ \smash{\underrightarrow{LD}} ^{\mathrm{b}} _{\Q,\mathrm{coh}} 
( \smash{\widehat{\D}} _{\X} ^{(\bullet)})$,
i.e. the truncation functors are the same
and 
$\fC  ^{\geq n}  (\X/\W):=
\smash{\underrightarrow{LD}} ^{\geq n} _{\Q,\mathrm{coh}} 
( \smash{\widehat{\D}} _{\X} ^{(\bullet)})
\cap 
\fC  (\X/\W)$,
$\fC  ^{\leq n}  (\X/\W):=
\smash{\underrightarrow{LD}} ^{\leq n} _{\Q,\mathrm{coh}} 
( \smash{\widehat{\D}} _{\X} ^{(\bullet)})
\cap 
\fC  (\X/\W)$.

For instance, using \ref{coro-ovcoh-oc-tstr}
and 
\ref{coro-h-tstr}, 
we get 
  for $\star \in \{\mathrm{ovcoh}, \mathrm{oc}, \mathrm{h},\mathrm{hol} \}$
  a canonical t-structure on 
$\smash{\underrightarrow{LD}} ^{\mathrm{b}} _{\Q,\star} $.
The heart of 
$\smash{\underrightarrow{LD}} ^{\mathrm{b}} _{\Q,\star} $
is  $\smash{\underrightarrow{LM}} ^{\mathrm{b}} _{\Q,\star}$.
\end{empt}

\subsection{On the stability under special descent of base}
\begin{prop}
\label{oc-st-sp-desc}
Let $\mathfrak{C}$ 
be a data of coefficients.
Suppose that 
$\fC$ is stable under  special descent of the base.
Then so is $S _\sharp (\fB _\mathrm{div}, \mathfrak{C})$.
\end{prop}

\begin{proof}
Since the case where $\sharp= 0$ is similar and easier, let us treat the other case.
Let $(\W, \W ^\flat)$ be an object of $\mathrm{DVR}  (\V,\V ^\flat)$,
$\X$ be a smooth formal scheme  over $\W$, 
$\E ^{(\bullet)}\in 
\smash{\underrightarrow{LD}} ^{\mathrm{b}} _{\Q,\mathrm{coh}} 
( \smash{\widehat{\D}} _{\X} ^{(\bullet)})$.
We suppose that
there exists 
a special morphism
$(\W, \W ^\flat)
\to 
(\W ', \W ^\flat)$
such that 
$ \W'  \smash{\widehat{\otimes}}^\L
_{\W}  \E^{(\bullet) }\in 
S (\fB _\mathrm{div}, \mathfrak{C})  (\X \times _{\Spf \W} \Spf \W ')$.
Let us check that
$\E^{(\bullet) }\in S (\fB _\mathrm{div}, \mathfrak{C})  (\X )$.

Let $(\W, \W ^\flat)
\to 
(\widetilde{\W}, \widetilde{\W} ^{\flat})$ 
be a morphism of 
$\mathrm{DVR}  (\V,\V ^\flat)$.
Set $\widetilde{\X} := \X \times _{\Spf \W} \Spf \widetilde{\W}$.
We have to check 
$\widetilde{\E} ^{(\bullet)}:= \widetilde{\W}  \smash{\widehat{\otimes}}^\L_{\W}  \E^{(\bullet) }\in 
S  _0 (\mathfrak{D}, \mathfrak{C}) (\widetilde{\X})$.
Let 
$f\colon \widetilde{\Y} \to \widetilde{\X}$ 
be a smooth morphism of smooth formal $\widetilde{\W}$-schemes, 
$\widetilde{Z} $ be a divisor of 
$\widetilde{Y}$.
Let $\widetilde{\W} '$ be the integral closure of the ring
$\mathrm{im} (\W '\to \widetilde{\W} ^{\flat})$.
We get the special morphism
$(\widetilde{\W}, \widetilde{\W} ^{\flat})
\to 
(\widetilde{\W}', \widetilde{\W} ^{\flat})$ (see \ref{special-filtered-pre}).
This yields
$\widetilde{\W}'  \smash{\widehat{\otimes}}^\L_{\widetilde{\W}}   \widetilde{\E} ^{(\bullet)}
\riso 
\widetilde{\W}'  \smash{\widehat{\otimes}}^\L_{\W'}  
( \W'  \smash{\widehat{\otimes}}^\L
_{\W}  
\E^{(\bullet) })
\in 
S  _0 (\fB _\mathrm{div}, \mathfrak{C}) (\widetilde{\X} ')$,
where
$\widetilde{\X} ' := \X \times _{\Spf \W} \Spf \widetilde{\W} '$.
Hence,
$\widetilde{\W}'  \smash{\widehat{\otimes}}^\L_{\widetilde{\W}}  
\left (
\smash{\widetilde{\B}} _{\widetilde{\Y}} ^{(\bullet)} (\widetilde{Z} )
\smash{\widehat{\otimes}}^\L
_{\O  _{\widetilde{\Y}}} f ^{!(\bullet)} (\widetilde{\E}^{(\bullet) })
\right) 
\in
\fC ( \widetilde{\Y} \times _{\Spf \widetilde{\W}} \Spf \widetilde{\W} ')$. 
Since 
$f ^{!(\bullet)} (\widetilde{\E}^{(\bullet) })
\in
\smash{\underrightarrow{LD}} ^{\mathrm{b}} _{\Q,\mathrm{coh}} (\widetilde{\Y})$,
using \ref{ind-desc-coh-chgbase}, this yields
$\smash{\widetilde{\B}} _{\widetilde{\Y}} ^{(\bullet)} (\widetilde{Z} )
\smash{\widehat{\otimes}}^\L
_{\O  _{\widetilde{\Y}}} f ^{!(\bullet)} (\widetilde{\E}^{(\bullet) })
\in
\smash{\underrightarrow{LD}} ^{\mathrm{b}} _{\Q,\mathrm{coh}} (\widetilde{\Y})$.
Since $\fC$ is stable under special descent of the base, we are done.
\end{proof}

\begin{coro}
\label{oc-st-sp-desc2}
The data of coefficients $\smash{\underrightarrow{LD}} ^{\mathrm{b}} _{\Q,\mathrm{coh}}$,
$\smash{\underrightarrow{LD}} ^{\mathrm{b}} _{\Q,\mathrm{ovcoh}}$,
and 
$\smash{\underrightarrow{LD}} ^{\mathrm{b}} _{\Q,\mathrm{oc}}$  
are stable under special descent of the base.
\end{coro}

\begin{prop}
\label{S(D,C)stability8}
Let $\mathfrak{C}$ and $\mathfrak{D}$ be two data of coefficients.
Suppose that $\fD$ is stable under smooth extraordinary pullbacks and special base change, 
and that $\fC$ 
is included in 
$\smash{\underrightarrow{LD}} ^{\mathrm{b}} _{\Q,\mathrm{ovcoh}}$ 
and 
is 
stable under  shifts, 
and special descent of the base. 
Then  $S(\mathfrak{D}, \mathfrak{C})$ is 
stable under special descent of the base. 
\end{prop}

\begin{proof}
Let $(\W, \W ^\flat)$ be an object of $\mathrm{DVR}  (\V,\V ^\flat)$,
$\X$ be a smooth formal scheme  over $\W$, 
$\E ^{(\bullet)}\in 
\smash{\underrightarrow{LD}} ^{\mathrm{b}} _{\Q,\mathrm{coh}} 
( \smash{\widehat{\D}} _{\X} ^{(\bullet)})$.
We suppose that
there exists 
a special morphism
$(\W, \W ^\flat)
\to 
(\W ', \W ^\flat)$
such that 
$ \W'  \smash{\widehat{\otimes}}^\L
_{\W}  \E^{(\bullet) }\in 
S (\mathfrak{D}, \mathfrak{C})  (\X \times _{\Spf \W} \Spf \W ')$.
Let us check that
$\E^{(\bullet) }\in S (\mathfrak{D}, \mathfrak{C})  (\X )$.

Let $(\W, \W ^\flat)
\to 
(\widetilde{\W}, \widetilde{\W} ^{\flat})$ 
be a morphism of 
$\mathrm{DVR}  (\V,\V ^\flat)$.
Set $\widetilde{\X} := \X \times _{\Spf \W} \Spf \widetilde{\W}$.
We have to check 
$\widetilde{\E} ^{(\bullet)}:= \widetilde{\W}  \smash{\widehat{\otimes}}^\L_{\W}  \E^{(\bullet) }\in 
S  _0 (\mathfrak{D}, \mathfrak{C}) (\widetilde{\X})$.

Let $\widetilde{\W} '$ be the integral closure of the ring
$\mathrm{im} (\W '\to \widetilde{\W} ^{\flat})$.
We get the special morphism
$(\widetilde{\W}, \widetilde{\W} ^{\flat})
\to 
(\widetilde{\W}', \widetilde{\W} ^{ \flat})$ (see \ref{special-filtered-pre}).
This yields
$\widetilde{\W}'  \smash{\widehat{\otimes}}^\L_{\widetilde{\W}}   \widetilde{\E} ^{(\bullet)}
\riso 
\widetilde{\W}'  \smash{\widehat{\otimes}}^\L_{\W'}  
( \W'  \smash{\widehat{\otimes}}^\L
_{\W}  \E^{(\bullet) })
\in 
S  _0 (\mathfrak{D}, \mathfrak{C}) (\widetilde{\X} ')$,
where
$\widetilde{\X} ' := \X \times _{\Spf \W} \Spf \widetilde{\W} '$.
Let 
$f\colon \widetilde{\Y} \to \widetilde{\X}$ 
be a smooth morphism of smooth formal $\widetilde{\W}$-schemes, 
$\widetilde{\FF} ^{(\bullet)}
\in
\mathfrak{D} (\widetilde{\Y})$.
Then 
$\widetilde{\FF} ^{(\bullet)}
\smash{\widehat{\boxtimes}}^\L
_{\widetilde{\W}} f ^{!(\bullet)} (\widetilde{\E}^{(\bullet) })
\in 
\smash{\underrightarrow{LD}} ^{\mathrm{b}} _{\Q,\mathrm{coh}} 
( \smash{\widehat{\D}} _{\widetilde{\Y} \times _{\Spf \widetilde{\W}} \widetilde{\Y}} ^{(\bullet)})$.
Since
$\fD$ is stable under smooth extraordinary pullbacks and special base change,
since base changes commute with (exterior) tensor products and extraordinary pullbacks,
since 
$\widetilde{\W}'  \smash{\widehat{\otimes}}^\L_{\widetilde{\W}}   \widetilde{\E} ^{(\bullet)}
\in 
S  _0 (\mathfrak{D}, \mathfrak{C}) (\widetilde{\X} ')$,
and since 
$\fC$ is stable under shifts, 
then 
$\widetilde{\W}'  \smash{\widehat{\otimes}}^\L_{\widetilde{\W}}   \left (
\widetilde{\FF} ^{(\bullet)}
\smash{\widehat{\boxtimes}}^\L
_{\widetilde{\W}} f ^{!(\bullet)} (\widetilde{\E}^{(\bullet) })
\right )
\in 
\fC( \Spf \widetilde{\W} '\times _{\Spf \widetilde{\W}} \widetilde{\Y} \times _{\Spf \widetilde{\W}} \widetilde{\Y}) $.
Since 
$\fC$ is stable under  special descent of the base, 
this yields 
$\widetilde{\FF} ^{(\bullet)}
\smash{\widehat{\boxtimes}}^\L
_{\widetilde{\W}} f ^{!(\bullet)} (\widetilde{\E}^{(\bullet) })
\in 
\fC( \widetilde{\Y} \times _{\Spf \widetilde{\W}} \widetilde{\Y}) $.
Since $\fC$
is included in 
$\smash{\underrightarrow{LD}} ^{\mathrm{b}} _{\Q,\mathrm{ovcoh}}$,
since 
$\widetilde{\FF} ^{(\bullet)}
\smash{\widehat{\otimes}}^\L
_{\widetilde{\W}} f ^{!(\bullet)} (\widetilde{\E}^{(\bullet) })$
and
$\delta ^{!(\bullet)}
\left (
\widetilde{\FF} ^{(\bullet)}
\smash{\widehat{\boxtimes}}^\L
_{\widetilde{\W}} f ^{!(\bullet)} (\widetilde{\E}^{(\bullet) })
\right )$
are isomorphic up to a shift, 
then this implies 
$\widetilde{\FF} ^{(\bullet)}
\smash{\widehat{\otimes}}^\L
_{\O  _{\widetilde{\Y}}} f ^{!(\bullet)} (\widetilde{\E}^{(\bullet) })
\in
\smash{\underrightarrow{LD}} ^{\mathrm{b}} _{\Q,\mathrm{coh}} 
(\widetilde{\Y})$.
Since 
$\widetilde{\W}'  \smash{\widehat{\otimes}}^\L_{\widetilde{\W}}   \widetilde{\E} ^{(\bullet)}
\in 
S  _0 (\mathfrak{D}, \mathfrak{C}) (\widetilde{\X} ')$,
then 
$\widetilde{\W}'  \smash{\widehat{\otimes}}^\L_{\widetilde{\W}}  
\left (
\widetilde{\FF} ^{(\bullet)}
\smash{\widehat{\otimes}}^\L
_{\O  _{\widetilde{\Y}}} f ^{!(\bullet)} (\widetilde{\E}^{(\bullet) })
\right )
\in
\mathfrak{C} 
(\widetilde{\Y}\times _{\widetilde{\W}}\widetilde{\W}' )$.
Since $\fC$
is closed under special descent of the base, 
we get
$\widetilde{\FF} ^{(\bullet)}
\smash{\widehat{\otimes}}^\L
_{\O  _{\Y}} f ^{!(\bullet)} (\widetilde{\E}^{(\bullet) })
\in 
\mathfrak{C} 
(\widetilde{\Y})$
i.e., 
$\widetilde{\E} ^{(\bullet)}\in 
S  _0 (\mathfrak{D}, \mathfrak{C}) (\widetilde{\X})$.
\end{proof}

\subsection{External product stability}

\begin{dfn}
\label{dfnSboxtimes(D,C)}
Let $\mathfrak{C}$ and $\mathfrak{D}$ be two data of coefficients.

\begin{enumerate}
\item We denote by 
$\boxtimes _0 (\mathfrak{D}, \mathfrak{C})$
the data of coefficients defined as follows: 
for any object $(\W, \W ^\flat)$ of $\mathrm{DVR}  (\V,\V ^\flat)$, 
for any  smooth formal scheme $\X$ over $\W$,  
the category 
$\boxtimes _0 (\mathfrak{D}, \mathfrak{C}) (\X)$
is the full subcategory of 
$\smash{\underrightarrow{LD}} ^{\mathrm{b}} _{\Q,\mathrm{coh}} ( \smash{\widehat{\D}} _{\X} ^{(\bullet)})$
consisting of objects  $\E ^{(\bullet)}$
satisfying the following property :
\begin{enumerate}
\item [($\star$)]
for any smooth formal $\W$-scheme $\Y$, 
for any object 
$\FF ^{ (\bullet)}
\in
\mathfrak{D} (\Y)$,
we have 
$\E ^{(\bullet)}
\smash{\widehat{\boxtimes}}^\L
_{\O  _{\Spf \W }} 
\FF^{ (\bullet) }
\in
\mathfrak{C} (\X\times _{\Spf \W} \Y)$.
\end{enumerate}

\item We denote by 
$\boxtimes (\mathfrak{D}, \mathfrak{C})$
the data of coefficients defined as follows: 
for any object $(\W, \W ^\flat)$ of $\mathrm{DVR}  (\V,\V ^\flat)$, 
for any  smooth formal scheme $\X$ over $\W$,  
the category 
$\boxtimes (\mathfrak{D}, \mathfrak{C}) (\X)$
is the full subcategory of 
$\smash{\underrightarrow{LD}} ^{\mathrm{b}} _{\Q,\mathrm{coh}} ( \smash{\widehat{\D}} _{\X} ^{(\bullet)})$
consisting of objects  $\E ^{(\bullet)}$
satisfying the following property :
\begin{enumerate}
\item 
[($\star \star$)]
for any morphism 
$(\W, \W ^\flat)
\to 
(\W ', \W ^{\prime \flat})$ 
of $\mathrm{DVR}  (\V,\V ^\flat)$, 
$ \W'  \smash{\widehat{\otimes}}^\L_{\W}  \E^{(\bullet) }
\in 
\boxtimes _0 (\mathfrak{D}, \mathfrak{C})
(\X \times _{\Spf \W} \Spf \W ')$.
\end{enumerate}
 \item Let $\sharp$ be a symbol so that 
either $\boxtimes _\sharp = \boxtimes _0$
or
$\boxtimes _\sharp = \boxtimes $.

\end{enumerate}

\end{dfn}

\begin{lem}
\label{lem-boxtimesDC}
Let $\mathfrak{C}$ and $\mathfrak{D}$ be two data of coefficients.

\begin{enumerate}
\item 
\label{lem-boxtimesDC1}
If $\mathfrak{D}$ contains $\fB _\emptyset$, then 
$\boxtimes _\sharp(\mathfrak{D}, \mathfrak{C})$ is contained in $\fC$.

 \item 
\label{lem-boxtimesDC2}
If $\fC \subset \fC '$ and $\fD ' \subset \fD$, then 
$\boxtimes _\sharp(\mathfrak{D}, \mathfrak{C}) \subset
\boxtimes _\sharp(\mathfrak{D}', \mathfrak{C}')$.

\item If $\fC$ is stable under devissage
then so is 
$\boxtimes _\sharp(\mathfrak{D}, \mathfrak{C})$.
Moreover, 
$\boxtimes _\sharp(\mathfrak{D}, \mathfrak{C})=
\boxtimes _\sharp(\Delta(\mathfrak{D}), \mathfrak{C})$.

\item The data $\boxtimes (\mathfrak{D}, \mathfrak{C})$ is stable under base change. 
Moreover, if $\fD$ is stable under special base change and $\fC$ is stable under special descent of the base, 
then $\boxtimes _\sharp(\mathfrak{D}, \mathfrak{C})$ is stable under special descent of the base.

\item If $\fC$ is stable under pushforwards (resp. satisfies  $BK ^!$, resp. is local, 
resp. is stable under direct summands), 
then so is 
$\boxtimes _\sharp(\mathfrak{D}, \mathfrak{C})$.

\item If $\fC$ is local and is stable under devissages  and cohomology, 
and if 
$\fD$ is stable under cohomology, 
then 
$\boxtimes _\sharp(\mathfrak{D}, \mathfrak{C})$
is stable under cohomology.
\end{enumerate}

\end{lem}

\begin{proof}
The first forth statements are obvious. 
The non respective case and the first respective case of the fifth one is a consequence of \ref{prop-boxtimes-v+}.
The other cases are obvious.
It remains to check the sixth one.
Following \ref{ntnC0-lemm}.\ref{ntnC0-lemm1},
since 
$\mathfrak{D}$ is stable under cohomology, 
then $\Delta (\mathfrak{D} ) 
= 
\Delta (\mathfrak{D} ^0) $.
Hence, from the third part of our lemma,
we get 
$\boxtimes _\sharp(\mathfrak{D}, \mathfrak{C})
=
\boxtimes _\sharp(\mathfrak{D}^0, \mathfrak{C})$.
Then, we conclude by using \ref{lem-boxtimes}.\ref{lem-boxtimes1}.
\end{proof}

\subsection{Constructions of stable data of coefficients}

\begin{dfn}
\label{dfn-almostdual}
Let $\mathfrak{D}$ be a data of coefficients over $(\V,\V ^\flat)$.
We say that $\mathfrak{D}$ is almost stable under dual functors if the following property holds:
for any data of coefficients $\mathfrak{C}$ over $(\V,\V ^\flat)$ which is stable under special descent of the base,
devissages, 
direct summands and 
pushforwards,
if $\mathfrak{D} \subset \mathfrak{C}$
then 
$\mathfrak{D} ^{\vee} \subset \mathfrak{C}$.
Remark from the biduality isomorphism that 
the inclusion 
$\mathfrak{D} ^{\vee} \subset \mathfrak{C}$
is equivalent to 
the following one 
$\mathfrak{D} \subset  \mathfrak{C} ^\vee$.
\end{dfn}

\begin{lem}
\label{almostdual-delta}
Let $\mathfrak{D}$ be a data of coefficients over $(\V,\V ^\flat)$.
The data $\mathfrak{D}$ is almost stable under dual functors
if and only if $\Delta (\mathfrak{D})$ is almost stable under dual functors.
\end{lem}
\begin{proof}
This is a consequence of \ref{stab-Dvee-3propbis}.
\end{proof}

\begin{lem}
\label{lem-prop=div-almostst}
With notation \ref{ex-Dcst},
we have the equalities
$\fM _\emptyset ^\vee
=
\fM _\emptyset $,
$(\Delta (\fM _\emptyset) ) ^\vee
=
\Delta (\fM _\emptyset)$
and 
$\Delta (\fM _\mathrm{sscd})
= 
\Delta (\fM _\emptyset)$.
\end{lem}

\begin{proof}
The first equality is a consequence of
\ref{propspetdualsansfrob}.
The second one follows from \ref{stab-Dvee-3propbis}.
It remains to  check the inclusion
$\fM _\mathrm{sscd}
\subset \Delta (\fM _\emptyset)$.
Let $(\W, \W ^\flat)$ be an object of $\mathrm{DVR}  (\V,\V ^\flat)$,
$\X$ be a smooth formal scheme  over $\W$, 
$Z$ be a smooth subvariety of the special fiber of $\X$, 
$T$ be a strict smooth crossing divisor of $Z/\Spec k$,
and 
$\E ^{(\bullet)}
\in 
\mathrm{MIC} ^{(\bullet)} (Z, \X/K) $.
We have to prove that 
$(\hdag T ) (\E ^{(\bullet)})
\in 
\Delta (\fM _\emptyset) (\X)$.
We proceed by induction on the dimension of $T$ and next the number of irreducible components of $T$. 
Let $T _1$ be one irreducible component of $T$ and $T'$ be the union of the other irreducible components.
We have the localisation triangle
\begin{equation}
\label{coh-ss-div-bis-extri2}
(\hdag T'\cap T _1)   \R \underline{\Gamma} ^{\dag} _{T _1}   (  \E ^{(\bullet)})
\to 
 (\hdag T')   (\E ^{(\bullet)})
\to 
 (\hdag T)   (\E ^{(\bullet)})
 \to +1
 \end{equation}
Following \ref{cor-com-sp+-f*} and \ref{40.21.2StPrj}, we have
$\R \underline{\Gamma} ^{\dag} _{T _1}   (  \E ^{(\bullet)}) [1]
\in 
\mathrm{MIC} ^{(\bullet)} (T _1, \X/K) $.
Hence, since $T'\cap T _1$ is a strict smooth crossing divisor of $T _1/\Spec k$, 
by induction hypothesis we get
$(\hdag T'\cap T _1)   \R \underline{\Gamma} ^{\dag} _{T _1}   (  \E ^{(\bullet)}) \in
\Delta (\fM _\emptyset) (\X)$.
By induction hypothesis, we have also
$ (\hdag T')   (\E ^{(\bullet)}) \in \Delta (\fM _\emptyset) (\X)$.
Hence, by devissage, we get 
$ (\hdag T)   (\E ^{(\bullet)}) \in \Delta (\fM _\emptyset) (\X)$.
\end{proof}

\begin{prop}
\label{prop=div-almostst}
The data of coefficients $\fB _\mathrm{div}$,
$\fM _\mathrm{div}$,
and 
$\fB _\mathrm{cst}$
are almost stable under duality.
\end{prop}

\begin{proof}
I) Since $\Delta (\fB _\mathrm{cst} )= \Delta (\fB _\mathrm{div})$ (see \ref{rem-div-cst}.\ref{rem-div-cst1}) and using \ref{almostdual-delta}, 
since the case $\fB _\mathrm{div}$ is checked similarly, 
we reduce to prove the almost dual stability of 
$\fM _\mathrm{div}$.

II) Let $\mathfrak{C}$
be a data of coefficients  over $(\V,\V ^\flat)$ which 
contains $\fM _\mathrm{div}$, and which 
is stable under special descent of the base,
devissages, 
direct summands and 
pushforwards.
Let $(\W, \W ^\flat)$ be an object of $\mathrm{DVR}  (\V,\V ^\flat)$,
$\X$ be a smooth formal scheme  over $\W$, 
$Z$ be a smooth subvariety of the special fiber of $\X$, 
$T$ be a divisor of $Z$,
and 
$\E ^{(\bullet)}
\in 
\mathrm{MIC} ^{(\bullet)} (Z, \X/K) $.
We have to check that 
$
(\hdag T ) (\E ^{(\bullet)})
\in 
\fC ^\vee (\X)$.

0) Let $l$ be the residue field of $\W$. 
Since 
$\fC$ 
is stable under special descent of the base, 
then similarly to the part 0),1) of the proof of \ref{coh-ss-div}, 
we can suppose $Z$ integral and that there
exist a smooth integral $l$-variety $Z'$, a projective generically finite and étale morphism of $l$-varieties 
$\phi \colon Z ' \to Z$ 
such that  $Z ' $ is quasi-projective 
and $T':= \phi  ^{-1} (T) $ is a strict smooth crossing divisor of $Z'/\Spec l$.
Hence, 
there exists a closed immersion of the form
$u\colon Z ' \hookrightarrow \bbP ^n _{Z}$
whose composition with the projection 
$\bbP ^n _{Z} \to Z$ is $\phi$.
Let 
$\X ' := \widehat{\bbP} ^n _{\X}$, 
$f \colon \X ' \to \X$ be the projection. 

1) By adjunction (use \ref{adj-morph} and \ref{limeqcat}), we get the morphism
$ f  ^{(\bullet)} _{+}  \R \underline{\Gamma} ^{\dag} _{Z '}  f ^{!(\bullet)}(\E ^{(\bullet)} )
 \to 
\E ^{(\bullet)}$.

1') We construct 
$\E ^{(\bullet)}
 \to 
 f  ^{(\bullet)} _{+}  \R \underline{\Gamma} ^{\dag} _{Z '}  f ^{!(\bullet)}(\E ^{(\bullet)} )$ as follows:
we get the morphism
$ f  ^{(\bullet)} _{+}  \R \underline{\Gamma} ^{\dag} _{Z '}  f ^{!(\bullet)}(\DD ^{(\bullet)} \E ^{(\bullet)}  )
 \to 
\DD ^{(\bullet)} \E ^{(\bullet)} $ by adjunction.
This yields by duality
$\E ^{(\bullet)} 
\riso 
\DD ^{(\bullet)} \DD ^{(\bullet)} \E ^{(\bullet)} 
\to 
\DD  ^{(\bullet)}  f  ^{(\bullet)} _{+}  \R \underline{\Gamma} ^{\dag} _{Z '}  f ^{!(\bullet)}
(\DD ^{(\bullet)} \E ^{(\bullet)} )
\underset{\ref{dualrelative-morp}}{\riso} 
f  ^{(\bullet)} _{+} 
\DD  ^{(\bullet)}   \R \underline{\Gamma} ^{\dag} _{Z '}  f ^{!(\bullet)}(\DD ^{(\bullet)} \E ^{(\bullet)} )$.
Following \ref{cor-com-sp+-f*}, we have the following isomorphism
$\DD ^{(\bullet)}
\left (
\R \underline{\Gamma} ^\dag _{Z'} f ^{!(\bullet)} \E ^{(\bullet)} 
\right )
\riso 
\R \underline{\Gamma} ^\dag _{Z'} f ^{!(\bullet)} (\DD ^{(\bullet)} \E ^{(\bullet)} )$,
and we are done.

2) Composing both morphisms of 1) and 1'),
we get an isomorphism (same reason than for the proof of \ref{coh-ss-div}). In particular, 
$\E ^{(\bullet)}$
is a direct summand of 
$ f  ^{(\bullet)} _{+}  \R \underline{\Gamma} ^{\dag} _{Z '}  f ^{!(\bullet)}(\E ^{(\bullet)} )$.
Hence, 
$(\hdag T ) ( \E ^{(\bullet)})$
is a direct summand of 
$(\hdag T )  f  ^{(\bullet)} _{+}  \R \underline{\Gamma} ^{\dag} _{Z '}  f ^{!(\bullet)}(\E ^{(\bullet)} )$
Using the commutation of localisation functor with pushforwards,
this yields
$(\hdag T )  (\E ^{(\bullet)})$
is a direct summand of 
$  f  ^{(\bullet)} _{+} (\hdag T ') \R \underline{\Gamma} ^{\dag} _{Z '}  f ^{!(\bullet)}(\E ^{(\bullet)} )$.

3)
Since
$\E ^{\prime (\bullet)} := 
\R \underline{\Gamma} ^{\dag} _{Z '}  f ^{!(\bullet)}(\E ^{(\bullet)} )
\in 
\mathrm{MIC} ^{(\bullet)} (Z', \X'/K) $ (use \ref{cor-com-sp+-f*}), then
$(\hdag T ' )(\E ^{\prime (\bullet)} )
\in 
\fM _\mathrm{sscd} (\X')$.
Since 
$\mathfrak{C}$ contains 
$\fM _\mathrm{sscd}$ and 
is stable under 
devissages, 
then 
using \ref{lem-prop=div-almostst}
we get 
$\fM _\mathrm{sscd} \subset 
\fC ^\vee$.
Hence, 
$(\hdag T ' )(\E ^{\prime (\bullet)} )
\in 
\fC ^\vee ( \X')$.
Since 
$\mathfrak{C}$ is stable under 
direct summands and 
pushforwards, we are done.
\end{proof}

\begin{ntn}
\label{dfnqupre}
Let $\mathfrak{C}, \mathfrak{D}$ be two data of coefficients.
We put 
$T _0 (\mathfrak{D} ,\mathfrak{C}) := 
S  (\mathfrak{D}  ,\mathfrak{C})$.
By induction on $i \in \N$, we set 
$U _i (\mathfrak{D}  ,\mathfrak{C}):= T _i (\mathfrak{D}  ,\mathfrak{C}) \cap T _i (\mathfrak{D}  ,\mathfrak{C} ) ^{\vee}$, 
$\widetilde{T} _{i} (\mathfrak{D}  ,\mathfrak{C}) := 
S (\mathfrak{D} , U _i (\mathfrak{D}  ,\mathfrak{C}))$
and
$T _{i+1} (\mathfrak{D}  ,\mathfrak{C}):= 
S  (\widetilde{T} _{i} (\mathfrak{D}  ,\mathfrak{C}), \widetilde{T} _{i} (\mathfrak{D}  ,\mathfrak{C}) )$.
We put $T (\mathfrak{D}  ,\mathfrak{C}) := \cap _{i\in \N} T _{i} (\mathfrak{D}  ,\mathfrak{C}) $.
\end{ntn}

\begin{thm}
\label{dfnquprop}
Let $\fB _\mathrm{div} \subset \mathfrak{D}\subset \fC$ be two data of coefficients.
We suppose 
\begin{enumerate}[a)]
\item The data $\fD$ is stable under cohomology, smooth extraordinary pullbacks and special base change ;
\item The data $\Delta (\fD)$ satisfies $BK _+$, is stable under extraordinary pullbacks, base change, 
tensor products 
and  is almost stable under dual functors ;
\item The data $\mathfrak{C}$ 
satisfies $BK ^!$, 
is local and stable under devissages, direct summands, 
pushforwards,
cohomology
and
special descent of the base. 
\end{enumerate}
Then, 
the data of coefficients $T(\mathfrak{D}, \mathfrak{C})$ 
(see Definition \ref{dfnqupre})
 is included in 
$\mathfrak{C} $,
contains 
$\mathfrak{D} $,
is 
local,
stable by devissages, direct summands, 
local cohomological functors, 
pushforwards, extraordinary pullbacks, base change, tensor products, duals,
cohomology
and
special descent of the base. 

\end{thm}

\begin{proof}
I) First, we check by induction on $i\in \N $ that the data of coefficients 
$T _i (\mathfrak{D},\mathfrak{C})$ 
contains $\mathfrak{D}$, 
is contained in $\mathfrak{C}$, 
is local,
stable under devissages, direct summands, 
local cohomological functors,
pushforwards, extraordinary pullbacks, base change,
cohomology
and
special descent of the base
(which implies such stability properties for $T(\mathfrak{D},\mathfrak{C})$). 

a)
Let us verify that $T _0 (\mathfrak{D},\mathfrak{C})$ satisfies these properties.
Using \ref{preS(D,C)stability}.\ref{S(D,C)stabilitynew3}, 
we get from \ref{preS(D,C)stability}.\ref{S(D,C)stability1} (resp. \ref{preS(D,C)stability}.\ref{S(D,C)stability2}) 
that 
$T _0 (\mathfrak{D},\mathfrak{C})$ is included in $\mathfrak{C}$
(resp. contains $\mathfrak{D}$).
Using \ref{stab-cst} and (the non respective case of) 
\ref{ovcoh-invim-prop}, 
we check that
$S  (\mathfrak{D}, \mathfrak{C})$ and
$S  (\fB _{\mathrm{cst}} ^+, \mathfrak{C})$
are both
local, stable under devissages, direct summands, 
local cohomological functors,
extraordinary pullbacks, pushforwards and base change.
Following
\ref{Bcst-st-cohom} and 
\ref{oc-st-sp-desc}, 
$S  (\fB _{\mathrm{cst}} ^+, \mathfrak{C})$ is also stable under cohomology and 
special descent of the base.
Hence, since 
$S  (\fB _{\mathrm{cst}} ^+, \mathfrak{C}) \subset \smash{\underrightarrow{LD}} ^{\mathrm{b}} _{\Q,\mathrm{ovcoh}}$,
(see \ref{preS(D,C)stability}.\ref{S(D,C)stability1})
by applying \ref{S(D,C)stability8},
this yields that $S   \left (\mathfrak{D}, S  (\fB _{\mathrm{cst}} ^+, \mathfrak{C}) \right)$
is  stable under special descent of the base.
Moreover, it follows from \ref{prop-st-cohom}
that $S   \left (\mathfrak{D}, S  (\fB _{\mathrm{cst}} ^+, \mathfrak{C}) \right)$ 
is also stable under cohomology.
With the remark \ref{rem-div-cst2},
we get the equalities
$S  (\mathfrak{D}, \mathfrak{C})
=
S  (\Delta (\mathfrak{D}), \mathfrak{C})
=
S   \left (\Delta (\mathfrak{D}), S  (\fB _{\mathrm{cst}} ^+, \mathfrak{C}) \right)
=
S   \left (\mathfrak{D}, S  (\fB _{\mathrm{cst}} ^+, \mathfrak{C}) \right)$.
Hence, we are done.

b) Suppose that this is true for $T _i (\mathfrak{D},\mathfrak{C})$ for some $i \in \N$.

i) Since 
$\mathfrak{D}$ is almost stable under duals, 
then $U _i (\mathfrak{D}  ,\mathfrak{C})$
contains 
$\mathfrak{D}$.
Since $\Delta (\mathfrak{D})$ is stable by tensor products, extraordinary pullbacks, and base change 
then, using \ref{preS(D,C)stability}.\ref{S(D,C)stability2}
(where $\mathfrak{C}$ is replaced by
$U _i (\mathfrak{D}  ,\mathfrak{C})$ which is stable under devissable),
this implies that $\mathfrak{D}$ is contained in
$\widetilde{T} _{i} (\mathfrak{D}  ,\mathfrak{C})$
and
$T _{i+1} (\mathfrak{D},\mathfrak{C})$.
Using \ref{preS(D,C)stability}.\ref{S(D,C)stability1}, we get that 
$\widetilde{T} _{i} (\mathfrak{D}  ,\mathfrak{C})$
and
$T _{i+1} (\mathfrak{D},\mathfrak{C})$
are included in $\mathfrak{C}$.

ii) 
From Lemma \ref{stab-Dvee-3prop},
$U _i (\mathfrak{D}  ,\mathfrak{C})$
satisfies $BK ^!$,
is local,
stable under pushforwards,
under devissages, direct summands, base change and 
special descent of the base.
It follows from \ref{preS(D,C)stability}.\ref{S(D,C)stability1}
that 
$T _i (\mathfrak{D}  ,\mathfrak{C})$ is 
included in 
$\smash{\underrightarrow{LD}} ^{\mathrm{b}} _{\Q,\mathrm{oc}}$.
Using \ref{cor-oc-inc-hol}, 
this yields that 
$U _i (\mathfrak{D}  ,\mathfrak{C})$
is stable under cohomology.
Hence, by applying the step I)a) in the case where 
$\mathfrak{C}$ is replaced by 
$U _i (\mathfrak{D}  ,\mathfrak{C})$,
we get 
that $\widetilde{T} _{i} (\mathfrak{D}  ,\mathfrak{C})$
is local, 
stable under devissages, direct summands, 
local cohomological functors,
pushforwards, extraordinary pullbacks, base change,
cohomology
and
special descent of the base.

Using (the respective case of) 
\ref{ovcoh-invim-prop}, 
this yields  that 
$T _{i+1} (\mathfrak{D},\mathfrak{C})$
is
local, stable under devissages, direct summands, 
local cohomological functors,
extraordinary pullbacks, pushforwards and base change.
By applying \ref{prop-st-cohom}
(resp. \ref{S(D,C)stability8}),
we check moreover that 
$T _{i+1} (\mathfrak{D},\mathfrak{C})$
is
stable under cohomology 
(resp. 
special descent of the base).

II) From \ref{preS(D,C)stability}.\ref{S(D,C)stability1}, 
$T _{i+1} (\mathfrak{D},\mathfrak{C})$ is contained in 
$\widetilde{T} _{i} (\mathfrak{D},\mathfrak{C})$
and 
$\widetilde{T} _{i} (\mathfrak{D},\mathfrak{C})$
is contained in 
$T _i (\mathfrak{D},\mathfrak{C}) \cap T _i (\mathfrak{D},\mathfrak{C}) ^{\vee}$.
Hence, by construction, the tensor product of two objects of 
$T _{i+1} (\mathfrak{D},\mathfrak{C})$ is an object of 
$T _{i} (\mathfrak{D},\mathfrak{C})$ and
the dual of an object of 
$T _{i+1} (\mathfrak{D},\mathfrak{C})$ is an object of 
$T _{i} (\mathfrak{D},\mathfrak{C})$.
\end{proof}

\begin{ex}
\label{ex-datastableevery}
We can choose 
$\fD = \fB _\mathrm{div} ^+$
and 
$\fC = \smash{\underrightarrow{LD}} ^{\mathrm{b}} _{\Q,\mathrm{coh}}$.
\end{ex}

\begin{lem}
\label{lemdfnqupropbis}
Let $\mathfrak{C}$ be a data of coefficients which contains $\mathfrak{M} _{\mathrm{div}} $ and is stable under shifts.
We have the inclusions :
\begin{enumerate}
\item The data $\mathfrak{M} _{\mathrm{div}}  ^+$ is stable under base change, smooth extraordinary pullbacks and tensor products. 
\item 
\label{lemdfnqupropbis1}
$\mathfrak{M} _{\mathrm{div}} 
\subset
S (\mathfrak{B} _{\mathrm{div}}  ,\fC)$.

\item 
$\mathfrak{M} _{\mathrm{div}}
\subset
\boxtimes \left (
\boxtimes ( \mathfrak{M} _{\mathrm{div}} , \fC ),
\boxtimes ( \mathfrak{M} _{\mathrm{div}} , \fC) \right)$.

\end{enumerate}

\end{lem}

\begin{proof}
The first statement is a consequence of \ref{cor-com-sp+-f*}.
The other ones are easy consequences of the first statement.
\end{proof}

\begin{ntn}
\label{dfnquprebis}
We put 
$T _0 := 
S  (\mathfrak{B} _{\mathrm{div}}  , \smash{\underrightarrow{LD}} ^{\mathrm{b}} _{\Q,\mathrm{coh}} )$.
By induction on $i \in \N$, we set 
$U _i:= T _i \cap T _i ^{\vee}$, 
$\widetilde{U}Â _i := 
\boxtimes (
\boxtimes ( \mathfrak{M} _{\mathrm{div}} , U _i ),
\boxtimes ( \mathfrak{M} _{\mathrm{div}} , U _i ))$,
and 
$T _{i+1}:= S (\mathfrak{B} _{\mathrm{div}}  ,\widetilde{U}Â _i)$.
We put $T  := \cap _{i\in \N} T _{i}  $.
\end{ntn}

\begin{thm}
\label{dfnqupropbis}
The data of coefficients $T$ 
contains 
$\mathfrak{M} _{\mathrm{div}}$,
is 
local,
stable by devissages, direct summands, 
local cohomological functors, 
pushforwards, extraordinary pullbacks, 
base change, tensor products, duals,
cohomology
and
special descent of the base. 

\end{thm}

\begin{proof}
I) We prove by induction on $i$ that 
$T _i$ 
contains $\mathfrak{M} _{\mathrm{div}}$, 
is local,
stable under devissages, direct summands, 
local cohomological functors,
pushforwards, extraordinary pullbacks, base change,
cohomology
and
special descent of the base.

a) 
For $T _0=\smash{\underrightarrow{LD}} ^{\mathrm{b}} _{\Q,\mathrm{oc}}$, 
this comes from the step I)a) of the proof of \ref{dfnquprop},
and of \ref{lemdfnqupropbis}.\ref{lemdfnqupropbis1}.

b) Suppose that this is true for $T _i$ for some $i \in \N$.

i) Since $\mathfrak{M} _{\mathrm{div}}$ is almost stable under duality (see \ref{prop=div-almostst}),
then $U _i$ 
contains $\mathfrak{M} _{\mathrm{div}}$.
Hence, using \ref{lemdfnqupropbis}, 
$\widetilde{U}Â _i$ 
contains $\mathfrak{M} _{\mathrm{div}}$.

ii) Similarly to  the first part of the step I)b)ii) of the proof of \ref{dfnquprop}, we check that 
$U _i$
satisfies $BK ^!$,
is local,
stable under pushforwards,
under devissages, direct summands, base change, 
special descent of the base,
and  under cohomology.
Hence, following \ref{lem-boxtimesDC}, 
we check that so is 
$\boxtimes ( \mathfrak{M} _{\mathrm{div}} , U _i )$
and then so is 
$\widetilde{U}Â _i$.
Hence, using the step I)a) of the proof of \ref{dfnquprop}, we get 
that $T _{i+1}$ satisfied the desired properties.

II) Using first  \ref{preS(D,C)stability}.\ref{S(D,C)stability1}, 
and next
\ref{lem-boxtimesDC}.\ref{lem-boxtimesDC1} we get the inclusions
$T _{i+1} \subset 
\widetilde{U}Â _i\subset \boxtimes ( \mathfrak{M} _{\mathrm{div}} , U _i ) 
\subset 
U _i
\subset
T _i 
 $. 
Hence, by construction, the external tensor product of two objects of 
$T _{i+1}$ is an object of 
$T _{i} $ and
the dual of an object of 
$T _{i+1}$ is an object of 
$T _{i}$.
\end{proof}

\section{Formalism of Grothendieck six operations for arithmetic $\D$-modules over couples}

Let $(\W, \W ^\flat)$ be an object of $\mathrm{DVR}  (\V,\V ^\flat)$, and $l$ be its residue field.

\subsection{Data of coefficients over frames}

\begin{dfn}
\begin{enumerate}
\item We define the category of frames over $\W$ as follows. 
A {\it frame} $(Y,X,\fP)$ 
over $\W$ means that 
$\fP$ is a realizable smooth formal scheme over $\W$, 
$X$ is a closed subscheme of the special fiber $P$ of $\fP$ and $Y$ is an open subscheme of $X$. 
Let 
$(Y', X', \fP')$ 
and $(Y,X,\fP)$ be two frames over $\W$. 
A morphism $\theta= (b,a,f) \colon (Y', X', \fP')\to (Y,X,\fP)$ of frames over $\W$ 
is the data of a  
morphism $f\colon \fP' \to \fP$ of realizable smooth formal schemes over $\W$,
a morphism $a\colon X' \to X$ of $l$-schemes, 
and a morphism $b \colon Y' \to Y$ of $l$-schemes 
inducing the commutative diagram
$$\xymatrix{
{Y'} 
\ar[d] ^-{b}
\ar@{^{(}->}[r] ^-{}
& 
{X'}
\ar[d] ^-{a}
\ar@{^{(}->}[r] ^-{}
& 
{\fP'} 
\ar[d] ^-{f}
\\
{Y} 
\ar@{^{(}->}[r] ^-{}
& 
{X}
\ar@{^{(}->}[r] ^-{}
& 
{\fP.} 
}$$
If there is no ambiguity with $\W$, we simply say frame or morphism of frames. 

\item A morphism 
$\theta= (b,a,f) \colon (Y', X', \fP')\to (Y,X,\fP)$ of frames over $\W$ 
is said to be {\it complete} 
(resp. {\it strictly complete})
if $a$ is proper (resp. 
$f$ and $a$ are proper). 

\end{enumerate}

\end{dfn}

\begin{dfn}
\begin{enumerate}
\item We define the category of {\it couples} over $\W$ as follow. 
A couple $(Y, X)$ over $\W$ means the two first data of a frame 
over $\W$
of the form $(Y,X, \fP)$.
A frame of the form $(Y,X, \fP)$ is said to be enclosing $(Y,X)$. 
A morphism of couples $u=(b,a)\colon (Y', X') \to (Y, X)$ over $\W$
is the data of a morphism of $l$-schemes of the form 
$a\colon X' \to X$ such that $a (Y' ) \subset Y$ and $b\colon Y'Â \to Y$ is the induced morphism.

\item A morphism of couples $u=(b,a)\colon (Y', X') \to (Y, X)$ over $\W$
is said to be {\it complete} if $a$ is proper.

\end{enumerate}
 
\end{dfn}

\begin{rem}
\label{rem-complete-frame-coup}
\begin{enumerate}

\item Let $u=(b,a)\colon (Y', X') \to (Y, X)$ be a complete morphism of couples  over $\W$.
Then there exists a strictly complete morphism  of frames over $\W$ of the form
$\theta= (b,a,f) \colon (Y', X', \fP')\to (Y,X,\fP)$.
Indeed, by definition, there exist some frames over $\W$ of the form
$(Y', X', \fP'')$ and 
$(Y,X,\fP)$. 
There exists an immersion 
$\fP'' \hookrightarrow \fQ ''$ with $\fQ''$ a proper and smooth formal $\W$-scheme. 
Hence, put $\fP':= \fQ '' \times \fP$ and let $f\colon \fP' \to \fP$ be the projection.
Since $a$ is proper, $X \hookrightarrow \fP$ is proper, 
and $f$ is proper, then the immersion $X' \hookrightarrow \fP'$ is also proper. 

\item Let $u=(b,a)\colon (Y', X') \to (Y, X)$ be a morphism of couples  over $\W$.
Similarly, we check that there exists a morphism  of frames over $\W$ of the form
$\theta= (b,a,f) \colon (Y', X', \fP')\to (Y,X,\fP)$. 

\end{enumerate}

\end{rem}

\begin{ntn}
\label{ntn-6operations}
Let $\mathfrak{C} $ be a data of coefficients over $(\W,\W ^\flat)$.
Let  $(Y, X,\fP)$ be a frame over $\W$. 
We denote by $\mathfrak{C} (Y, \fP/\W)$
the full subcategory of 
$\mathfrak{C}  (\fP)$ 
of objects $\E$ such that there exists an isomorphism of the form 
$\E \riso \R \underline{\Gamma} ^\dag _{Y} (\E)$.
We remark that $\mathfrak{C} (Y, \fP/\W)$ only depend on the immersion 
$Y \hookrightarrow \fP$ which explains the notation.
We might choose $X$ equal to the closure of $Y$ in $P$.

\end{ntn}

\begin{ntn}
\label{ntn-t-structureovcoh}
Let $\mathfrak{C}$ be a data of coefficients stable under devissages  and cohomology. 
Let $(Y, X,\fP)$ be a frame over $\W$.
Choose $\U$ an open set of $\fP$ such that 
$Y$ is closed in $\U$.

\begin{enumerate}
\item Similarly to \cite[1.2.1-5]{Abe-Caro-weights} , 
we define a canonical t-structure on 
$\mathfrak{C} (Y , \fP /\W)$
as follows.
We denote by 
$\fC  ^{\leq n}    (Y,\fP/\W)$
(resp. 
$\fC   ^{\geq n}  (Y,\fP/\W)$)
the full subcategory of 
$\mathfrak{C} (Y , \fP /\W)$
of complexes 
$\E$ such that
$\E |\U 
\in 
\fC  ^{\leq n}  (Y,\U/\W):=
\fC   (Y,\U/\W)
\cap 
\fC  ^{\leq n}  (\U/\W)$
(resp. 
$\E |\U 
\in 
\fC  ^{\geq n}  (Y,\U/\W)
:=
\fC   (Y,\U/\W)
\cap 
\fC  ^{\geq n}  (\U/\W)$),
where the t-structure on 
$\fC  (\U/\W)$
is the canonical one (see \ref{t-structure-ovcoh-oc-hol-h}).
The heart of this t-structure
will be denoted by
$\mathfrak{C} ^0   (Y,\fP/\W)$.
Finally, we denote by 
$\mathcal{H} ^i _{\mathrm{t}}$
the $i$th space of cohomology with respect to this canonical t-structure. 

\item Suppose $Y$ is smooth. 
Then, we denote by
$\fC _{\mathrm{isoc}}   (Y,\fP/\W)$
(resp. $\fC ^{\geq n}  _{\mathrm{isoc}}   (Y,\fP/\W)$,
resp. $\fC ^{\leq n}  _{\mathrm{isoc}}   (Y,\fP/\W)$,
resp. $\fC ^0 _{\mathrm{isoc}}   (Y,\fP/\W)$)
the full subcategory of 
(resp. $\fC ^{\geq n}  (Y,\fP/\W)$,
resp. $\fC ^{\leq n}   (Y,\fP/\W)$,
resp. $\fC ^0  (Y,\fP/\W)$)
consisting of complexes
$\E ^{(\bullet)}$ such that 
$\mathcal{H} ^i  (\E ^{(\bullet)} |\U ) 
\in 
\mathrm{MIC} ^{(\bullet)} (Y, \U/K)$.
We refer ``$\mathrm{isoc}$'' as isocrystals. 
The reason is that when we work over perfect fields,
we have the equivalence of categories
of \cite[5.4.6.1]{caro-pleine-fidelite} (since we use Shiho's work on proper descent, 
this is not clear if we can extend it without this perfectness condition).
In this paper, we avoid trying to check such equivalence of categories
(other than the easier case where the partial compactification is smooth).

\end{enumerate}

\end{ntn}

\begin{rem}
\label{rem-tstructure-exact}
Let $\mathfrak{C}$ be a data of coefficients stable under devissages  and cohomology. 
Let $\fP$ be a smooth formal $\W$-scheme,
$Y$ be a subscheme of $P$,
$Z$ be a closed subscheme of $Y$,
and 
$Y':= Y \setminus Z$.

\begin{enumerate}
\item We get the t-exact functor 
$(\hdag Z) 
\colon
\mathfrak{C} (Y, \fP/\W)
\to 
\mathfrak{C} (Y', \fP/\W)$.
Beware the functor 
$(\hdag Z) 
\colon
\mathfrak{C} (Y, \fP/\W)
\to 
\mathfrak{C} (Y, \fP/\W)$
is not always t-exact.

\item 
\label{rem-tstructure-exact2}
We say that $Z$ locally comes from a divisor of $P$ if 
locally in $P$, there exists a divisor $T$ of $P$ such that 
$Z = Y \cap T$ (this is equivalent to saying that locally in $P$, 
the ideal defining $Z \hookrightarrow Y$ is generated by one element).
In that case, 
we get the t-exact functor 
$(\hdag Z) 
\colon
\mathfrak{C} (Y, \fP/\W)
\to 
\mathfrak{C} (Y, \fP/\W)$.
Indeed, 
by construction of our t-structures, 
we can suppose $Y$ is closed in $\fP$ (and then we reduce to the case where the t-structure on 
$\mathfrak{C} (Y, \fP/\W)$ is induced by 
the standard t-structure
of $\smash{\underrightarrow{LD}} ^{\mathrm{b}} _{\Q,\mathrm{coh}} ( \smash{\widehat{\D}} _{\fP} ^{(\bullet)})$).
Since the property is local, 
we can suppose there exists a divisor
 $T$ such that $Z = T \cap Y$. 
Then both functors 
$(\hdag Z) $ and $(\hdag T)$ 
of 
$\mathfrak{C} (Y, \fP/\W)
\to 
\mathfrak{C} (Y, \fP/\W)$
are isomorphic. Since 
$(\hdag T)$ is exact,
we are done.
\end{enumerate}

\end{rem}

\begin{rem}
\label{rem-sp-desc-123}

Let $(\W, \W ^\flat)
\to 
(\W ', \W ^\flat)$ be a special morphism
of $\mathrm{DVR}  (\V,\V ^\flat)$. 
Let $\fP$ be a smooth formal $\W$-scheme,
$\fP ^\flat := \fP \times _{\Spf \W} \Spf \W ^\flat$ be the induced smooth formal $\W ^\flat$-scheme.
Let $\E^{(\bullet) }, \FF^{(\bullet) }\in
\smash{\underrightarrow{LD}} ^{\mathrm{b}} _{\Q,\mathrm{coh}} ( \smash{\widehat{\D}} _{\fP} ^{(\bullet)})$.
\begin{enumerate}
\item 
\label{rem-sp-desc-123-1}
$\E^{(\bullet) } =0$ if and only if 
$ \W'  \smash{\widehat{\otimes}}^\L
_{\W}  \E^{(\bullet) }
=0$. 
\item 
\label{rem-sp-desc-123-2}
A morphism $f \colon \E^{(\bullet) }\to  \FF^{(\bullet) }$ is an isomorphism,
if and only if 
$ \W'  \smash{\widehat{\otimes}}
_{\W} f$ is an isomorphism.

\item
\label{rem-sp-desc-123-3}
Let $l ^{\flat}$ be the residue field of $\W ^{\flat}$.
Let $Y$ be a reduced $l$-subscheme of $P$.
Let $Y^{\flat}:= Y \times _{\Spec l} \Spec l ^{\flat}$,
and 
$\widetilde{Y}  ^{\flat}:= 
Y^{\flat}_{\mathrm{red}}:= (Y \times _{\Spec l} \Spec l ^{\flat}) _{\mathrm{red}}$
be the corresponding reduced subscheme of 
$P ^\flat$.
Let $l'$  be a finite (radicial) extension of $l$ included in $l ^\flat$.
We put $Y ': =Y \times _{\Spec l} \Spec (l')$.

\begin{enumerate}

\item  By using
 \cite[8.7.2]{EGAIV3}, 
 \cite[8.8.2.(ii)]{EGAIV3} and \cite[8.10.5.(v)]{EGAIV3}, 
for $l'$ large enough,
there exist a reduced $l'$-variety $\widetilde{Y} '$ 
satisfying 
$\widetilde{Y} ^\flat \riso \widetilde{Y} ' \times _{\Spec (l')} \Spec (l ^\flat)$.
For $l'$ large enough, 
it follows from \cite[8.8.2.(i)]{EGAIV3} 
that 
there exists a morphism 
$\widetilde{Y} ' \to Y '$
inducing the closed immersion 
$\widetilde{Y} ^\flat \hookrightarrow Y ^\flat$.
By using \cite[8.10.5]{EGAIV3}, 
for $l'$ large enough, 
we can suppose that $\widetilde{Y} ' \to Y '$ is a surjective closed immersion.
Since $\widetilde{Y} '$ is reduced, this yields
$\widetilde{Y} ' = Y' _{\mathrm{red}}$, for $l'$ large enough.

\item 
\label{rem-sp-desc-123-3b}
Hence, by using \cite[17.7.8]{EGAIV4}, 
we check that if 
$Y^{\flat}_{\mathrm{red}}$ is smooth 
(resp. étale) over $l ^\flat$, then so is 
$Y' _{\mathrm{red}}$ over $l'$
for $l'$ large enough.

\item 
\label{rem-sp-desc-123-3c}
Hence, if $Y$ is of dimension $0$, then
$Y' _{\mathrm{red}}$ is a finite and étale $l'$-variety
for $l'$ large enough.
Indeed, 
since $l ^{\flat}$ is perfect and 
since
$Y^{\flat}_{\mathrm{red}}$
is a reduced $l ^{\flat}$-scheme of finite type of dimension $0$, then 
$Y^{\flat}_{\mathrm{red}}$
is a finite and étale $l ^{\flat}$-scheme.
We conclude using the previous remark.

\end{enumerate}

\end{enumerate}

\end{rem}

\begin{empt}
[Devissage in isocrystals]
\label{ovcoharedev}
Let $\mathfrak{C}$ be a data of coefficients over $ (\V,\V ^\flat)$ 
stable under devissages, cohomology, local cohomological functors,
extraordinary pull back and base change. 
Let $(\W, \W ^\flat)$ be an object of $\mathrm{DVR}  (\V,\V ^\flat)$,
let $(Y, X,\fP)$ be a frame over $\W$.
Let $\E ^{(\bullet)} \in \fC  (Y,\fP/\W)$.
Using \ref{rem-sp-desc-123}.\ref{rem-sp-desc-123-3} and its notations, 
since $l ^{\flat}$ is perfect, 
there exists
an open dense $l ^{\flat}$-smooth subscheme 
$\widetilde{Y} _0 ^{\flat}$ of 
$Y^{\flat}_{\mathrm{red}}$.
Using similar arguments than \ref{rem-sp-desc-123}.\ref{rem-sp-desc-123-3b}, 
this yields there exist
$l'$ a finite (radicial) extension of $l$ included in $l ^\flat$,
and an open dense smooth subscheme $Y ' _0$ of $ (Y \times _{\Spec l} \Spec l ') _{\mathrm{red}}$
such that 
$\widetilde{Y} _0 ^{\flat} \riso Y _0 ' \times _{\Spec (l')} \Spec (l ^\flat)$.
With \ref{special-filtered-pre}, 
increasing $l'$ if necessary, 
we can suppose there exists a special morphism 
$(\W, \W ^\flat)
\to 
(\W ', \W ^\flat)$ 
of 
$\mathrm{DVR}  (\V,\V ^\flat)$
such that the residue field of 
$\W '$ is $l'$.
Hence, 
shrinking $Y ' _0$ if necessary, 
using Theorem \cite[3.4]{caro-holo-sansFrob} (which is still valid without the hypothesis on the perfectness
of the residue fields) 
we get 
$ \W'  \smash{\widehat{\otimes}}^\L
_{\W}  \E^{(\bullet) }\in \mathfrak{C}  _{\mathrm{isoc}} (Y' _0, \fP \times _{\Spf \W} \Spf \W ')$.
By devissage, 
this yields that there exists
a special morphism 
$(\W, \W ^\flat)
\to 
(\W ', \W ^\flat)$ such that,
denoting by 
$\fP ' := \fP \times _{\Spf \W} \Spf \W '$,
$Y':= (Y \times _{\Spf \W} \Spf \W ') _{\mathrm{red}}$ and
$\E^{\prime (\bullet) }
:=
 \W'  \smash{\widehat{\otimes}}^\L
_{\W}  \E^{(\bullet) }$,
there exists a smooth stratification 
$(Y ' _{i}) _{i=1, \dots , r}$ of $Y'$ in $P'$ 
(see Definition \cite[2.2.1]{Abe-Caro-weights})
such that we have
$ \R \underline{\Gamma} ^\dag _{Y ' _i}  (\E ^{\prime (\bullet)}) 
\in 
\mathfrak{C} _{\mathrm{isoc}} (Y' _i, \fP')$
 for any $i =1,\dots, r$.
\end{empt}

\subsection{Formalism of Grothendieck six operations over couples}

\begin{thm}
[Independence]
\label{ind-CYW}
Let $\mathfrak{C}$ be a data of coefficients over $(\V,\V ^\flat)$
which contains $\fB _\mathrm{div}$, 
which is 
stable  under devissages, 
pushforwards, extraordinary pullbacks,
and under local cohomological functors.

Let
$\theta= (b,a,f) \colon (Y', X', \fP')\to (Y,X,\fP)$ be  a morphism of frames over $\W$
such that $a$ and $b$ are proper.

\begin{enumerate}
\item For any 
$\E ^{(\bullet)} 
\in 
\fC (Y , \fP /\W)$, 
for any $\E ^{\prime (\bullet)} \in \fC (Y ', \fP '/\W)$  (recall notation \ref{ntn-6operations},
we have 
$$\mathrm{Hom}  _{\fC (Y , \fP /\W)}
( f ^{(\bullet) } _{+} ( \E ^{\prime (\bullet) }) , \E ^{(\bullet) }) 
\riso 
 \mathrm{Hom}  _{\fC (Y ', \fP '/\W)}
( \E ^{\prime (\bullet) } ,
\R \underline{\Gamma} ^{\dag} _{Y'}
f ^{!(\bullet )}  ( \E ^{ (\bullet) })).$$

\item Suppose that 
$Y'= Y$ and $b$ is the identity,
and that $\fC$ is stable under cohomology.
 Then, 
for any 
$\E ^{(\bullet)} \in \fC ^0 (Y , \fP /\W)$,
for any 
$\E ^{\prime(\bullet)} \in 
\fC ^0 (Y , \fP '/\W)$, 
for any $n \in \Z \setminus \{ 0\}$, we have
$$\mathcal{H} _t ^n \R \underline{\Gamma} ^{\dag} _{Y} f ^{!(\bullet)} ( \E ^{(\bullet)}) =0,
\hspace{1 cm}
\mathcal{H} ^n _t f _{+} ^{(\bullet)} ( \E ^{\prime (\bullet)}) =0.$$

\item Suppose that 
$Y'= Y$ and $b$ is the identity. 
For any 
$\E ^{(\bullet)} 
\in 
\mathfrak{C} (Y , \fP /\W)$, 
for any $\E ^{\prime (\bullet)} \in \mathfrak{C} (Y , \fP '/\W)$, 
the adjunction morphisms 
$\R \underline{\Gamma} ^{\dag} _{Y} f ^{!(\bullet)}  f _{+} ^{(\bullet)} (\E ^{\prime (\bullet)})
\to 
\E ^{\prime (\bullet)}$ 
and 
$f _{+} ^{(\bullet)}  \R \underline{\Gamma} ^{\dag} _{Y} f ^{!(\bullet)}  
(\E ^{(\bullet)} )
\to 
\E ^{(\bullet)} $
are isomorphisms.
In particular, the functors
$\R \underline{\Gamma} ^{\dag} _{Y} f ^{!(\bullet)}$
and 
$f _{+} ^{(\bullet)}$ induce quasi-inverse equivalences of categories between 
$\mathfrak{C} (Y , \fP /\W)$
and 
$\mathfrak{C} (Y , \fP '/\W)$.

\end{enumerate}

\end{thm}

\begin{proof}
I) Let us check the first statement. 
Replacing $X$ and $X'$ by the closure of $Y$ in $P$ and $Y'$ in $P'$,
we can suppose $Y$ is dense in $X$ and $Y'$ is dense in $X'$.
Let 
$\E ^{(\bullet)} 
\in 
\fC (Y , \fP /\W)$, 
and $\E ^{\prime (\bullet)} \in \fC (Y ', \fP '/\W)$.
Since $a$ is proper,
using \ref{cor-adj-formulbis-bij2}, 
the stability of $\fC$  under extraodinary pullbacks,
and the equivalence of categories
\ref{eqcat-limcoh}, we get the bijection
$$\mathrm{Hom}  _{\smash{\underrightarrow{LD}} ^{\mathrm{b}} _{\Q,\mathrm{coh}}
(\smash{\widehat{\D}} _{\fP} ^{(\bullet)})}
( f ^{(\bullet) } _{+} ( \E ^{\prime (\bullet) }) , \E ^{(\bullet) }) 
\riso 
 \mathrm{Hom}  _{\smash{\underrightarrow{LD}} ^{\mathrm{b}} _{\Q,\mathrm{coh}}
(\smash{\widehat{\D}} _{\fP'} ^{(\bullet)})}
( \E ^{\prime (\bullet) } ,
f ^{!(\bullet )}  ( \E ^{ (\bullet) })).$$
Since $a$ and $b$ are proper, then 
the open immersion 
$Y' \subset a ^{-1} (Y)$ is proper. 
Since $Y'$ is dense in $X'$, then 
$Y'=a ^{-1} (Y)$.
Hence,
the functors 
$\R \underline{\Gamma} ^{\dag} _{X'}
f ^{!(\bullet )}$
and
$\R \underline{\Gamma} ^{\dag} _{Y'}
f ^{!(\bullet )}$
(resp. 
$f ^{(\bullet) } _{+} $ 
and 
$\R \underline{\Gamma} ^{\dag} _{Y} f ^{(\bullet) } _{+} $)
are isomorphic over 
$\fC(Y , \fP /\W)$
(resp. $\fC(Y ', \fP '/\W)$).
Hence,
the functor 
$\R \underline{\Gamma} ^{\dag} _{X'}
f ^{!(\bullet )}$
(resp. $f ^{(\bullet) } _{+} $)
induces
$\R \underline{\Gamma} ^{\dag} _{X'}
f ^{!(\bullet )}
\colon
\fC (Y , \fP /\W)
\to 
\fC (Y ', \fP '/\W)$
(resp. 
$f ^{(\bullet) } _{+} 
\colon 
\fC (Y ', \fP '/\W)
\to 
\fC (Y , \fP /\W)$).
Since $\fC (Y' , \fP' /\W)$ 
is a strictly full subcategory of
$\smash{\underrightarrow{LD}} ^{\mathrm{b}} _{\Q,\mathrm{coh}} 
(\smash{\widehat{\D}} _{\fP'} ^{(\bullet)})$,
we conclude using the equality
$$ \mathrm{Hom}  _{\smash{\underrightarrow{LD}} ^{\mathrm{b}} _{\Q,\mathrm{coh}}
(\smash{\widehat{\D}} _{\fP'} ^{(\bullet)})}
( \E ^{\prime (\bullet) } ,
f ^{!(\bullet )}  ( \E ^{ (\bullet) }))
=
 \mathrm{Hom}  _{\fC (Y ', \fP '/\W)}
( \E ^{\prime (\bullet) } ,
\R \underline{\Gamma} ^{\dag} _{X'}
f ^{!(\bullet )}  ( \E ^{ (\bullet) })).$$

II) Now let us check at the same time the last two statements.
Using the stability properties that $\mathfrak{C}$ satisfies, 
we check that the functors 
$f _+ ^{(\bullet )}\colon 
\mathfrak{C} (Y , \fP' /\W)
\to 
\mathfrak{C} (Y , \fP /\W)$
and
$\R \underline{\Gamma} ^{\dag} _{Y} f ^{!(\bullet )}
\colon 
\mathfrak{C} (Y , \fP /\W)
\to 
\mathfrak{C} (Y , \fP '/\W)$
are well defined. 
Since $\mathfrak{C}$ is included in $\smash{\underrightarrow{LD}} ^{\mathrm{b}} _{\Q,\mathrm{ovcoh}} $, 
we reduce to check the case where 
$\mathfrak{C}=\smash{\underrightarrow{LD}} ^{\mathrm{b}} _{\Q,\mathrm{ovcoh}} $.
We proceed similarly to \cite[3.2.6]{caro_surcoherent}:
Choose 
$\U$ (resp. $\U'$) an open set of $\fP$ (resp. $\fP '$) such that 
$Y$ is closed in $\U$ (resp. $Y$ is closed in $\U'$),
and such that $f (\U') \subset \U$.
The functor 
$| \U \colon 
\smash{\underrightarrow{LD}} ^{\mathrm{b}} _{\Q,\mathrm{ovcoh}} (Y , \fP /\W)
\to 
\smash{\underrightarrow{LD}} ^{\mathrm{b}} _{\Q,\mathrm{ovcoh}} (Y , \U /\W)$
is t-exact, and the same with some primes.
Moreover, 
for any $\E ^{(\bullet)}
\in 
\smash{\underrightarrow{LM}}  _{\Q,\mathrm{ovcoh}} (Y , \fP /\W)$
(or 
$\E ^{(\bullet)}
\in \smash{\underrightarrow{LD}} ^{\mathrm{b}} _{\Q,\mathrm{ovcoh}} (Y , \fP /\W)$), 
the property 
$\E ^{(\bullet)}= 0$ is equivalent to 
$\E ^{(\bullet)}|\U= 0$. Hence, we can suppose 
$\U= \fP$ and $\U'= \fP'$, i.e.
$Y \hookrightarrow P$
and 
$Y \hookrightarrow P'$ are closed immersions.

1) Using  notation \ref{rem-sp-desc-123}.\ref{rem-sp-desc-123-3}, 
suppose that $Y' _{\mathrm{red}}$ is smooth over $l'$
for $l'$ large enough. Then, using 
\ref{rem-sp-desc-123}.\ref{rem-sp-desc-123-1} 
and \ref{rem-sp-desc-123}.\ref{rem-sp-desc-123-2}, 
we reduce to the case where $l' = l$, i.e. $Y$ is smooth. 
Since the theorem is local, we can suppose there exists a smooth formal scheme
$\Y$ which is a lifting of $Y$. 
Hence, 
this is an obvious consequence of Berthelot-Kashiwara theorem 
\ref{u!u+=id}.

2) Let us go back to the general case. We proceed by induction on the lexicographic order 
$(\dim Y, N _Y)$ where $\dim Y$ is the dimension of $Y$ and $N _Y$ is the number of irreducible components of maximal dimension
of $Y$. Remark that 
$\dim Y = \dim Y^{\flat}_{\mathrm{red}}$
and 
$N _Y = N _{Y^{\flat}_{\mathrm{red}}}$.
When $\dim Y =0$, then using \ref{rem-sp-desc-123}.\ref{rem-sp-desc-123-3c}, 
this is a consequence of the step 1). 

 Now, suppose $(\dim Y, N _Y)$ is such that $\dim Y\geq 1$.
Since $l ^{\flat}$ is perfect, since the theorem is local in $\fP$, we can suppose 
 $\fP$ integral and affine, and 
there exists $T ^\flat$ a (reduced) divisor of $P ^\flat$ 
such that, putting $Z ^\flat := Y^{\flat}_{\mathrm{red}}\cap T ^\flat$, 
we have
$Y^{\flat}_{\mathrm{red}} \setminus Z ^\flat$ is a smooth $l ^\flat$-variety containing
a generic point of an irreducible component of maximal ideal.
Hence, 
$(\dim Z^\flat, N _{Z ^\flat})
<
(\dim Y, N _Y)$.

By using 
\cite[8.7.2]{EGAIV3}, 
 \cite[8.8.2.(ii)]{EGAIV3} and \cite[8.10.5.(v)]{EGAIV3}, 
for $l'$ large enough,
there exist a reduced closed $l'$-subvariety $T _{l'}$ 
of 
$P _{l'}: =P \times _{\Spec l} \Spec (l')$
satisfying 
$T ^\flat \riso T _{l'}\times _{\Spec (l')} \Spec (l ^\flat)$.
Since 
$T  _{l'}$ (resp. $P _{l'}$) is homeomorphic to $T ^\flat$
(resp. $P ^\flat$), then they have the same dimension
and the same number of irreducible components.
This yields
$T _{l'}$ is a divisor of $P _{l'}$.
Then, using 
\ref{rem-sp-desc-123}.\ref{rem-sp-desc-123-1} 
and \ref{rem-sp-desc-123}.\ref{rem-sp-desc-123-2}, 
we reduce to the case where $l' = l$, i.e. 
we can suppose 
there exists a divisor $T$ of $P$ such that, putting $Z:= T \cap Y$,  
we have 
$Y  \setminus Z $ is a smooth $l$-variety, and 
$(\dim Z, N _{Z})
<
(\dim Y, N _Y)$.

3) We check in this step that 
for any 
    $\E ^{\prime (\bullet)} \in \smash{\underrightarrow{LM}}  _{\Q,\mathrm{ovcoh}} (Y , \fP '/\W)$, 
for any integer $r \neq 0$,
 $\mathcal{H} ^r f _{+} ^{(\bullet)} ( \E ^{\prime(\bullet)}) = 0$.
\medskip

The localisation triangle in $Z$ of  $\E ^{\prime(\bullet)}$ induces the exact sequence in 
$\smash{\underrightarrow{LM}} _{\Q,\mathrm{ovcoh}} (Y , \fP' /\W)$:
\begin{equation}\label{s.e.loca}
0 \rightarrow \mathcal{H} ^{\dag , 0} _{Z} (\E ^{\prime(\bullet)} ) \rightarrow
\E ^{\prime(\bullet)} \rightarrow (\hdag Z ) ( \E ^{\prime(\bullet)}) \rightarrow \mathcal{H}
^{\dag , 1} _{Z} (\E ^{\prime(\bullet)} ) \rightarrow 0 .
\end{equation}
Since $Z$ locally comes from a divisor of $P'$, then the functor
$(\hdag Z )\colon 
\smash{\underrightarrow{LD}} ^{\mathrm{b}} _{\Q,\mathrm{ovcoh}} (Y , \fP '/\W)
\to 
\smash{\underrightarrow{LD}} ^{\mathrm{b}} _{\Q,\mathrm{ovcoh}} (Y , \fP '/\W)$
is exact (see \ref{rem-tstructure-exact}.\ref{rem-tstructure-exact2}). 
Let  $\FF ^{\prime(\bullet)}$ be the kernel of the epimorphism
$(\hdag Z ) ( \E ^{\prime(\bullet)}) 
\rightarrow  \mathcal{H} ^{\dag , 1} _{Z} (\E ^{\prime(\bullet)} )$.
We get the exact sequence in $\smash{\underrightarrow{LM}} _{\Q,\mathrm{ovcoh}} (Y , \fP' /\W)$
$$0 \rightarrow \FF ^{\prime(\bullet)} \rightarrow (\hdag Z ) ( \E ^{\prime(\bullet)})
\rightarrow  \mathcal{H} ^{\dag , 1} _{Z} (\E ^{\prime(\bullet)} ) \rightarrow
0.$$ 
By applying the functor 
$f ^{(\bullet)} _{+}$ to this latter exact sequence, we get a long exact sequence.
We have
$(\hdag Z)(\E ^{\prime(\bullet)}) 
\in 
\smash{\underrightarrow{LM}} _{\Q,\mathrm{ovcoh}} (Y \setminus Z, \fP' /\W)$ 
and
$\mathcal{H} ^{\dag , 1} _{Z} (\E ^{\prime(\bullet)}) \in 
\smash{\underrightarrow{LM}} _{\Q,\mathrm{ovcoh}} (Z , \fP' /\W)$.
Hence, following the step 1), using the induction hypothesis, 
using the long exact sequence, we check that 
for any integer 
$r \not \in \{0,1\}$, 
we have 
$\mathcal{H} ^r (f _{+} ^{(\bullet)} ) ( \FF ^{\prime(\bullet)} ) = 0$. 
Moreover, 
 $\mathcal{H} ^1 (f _{+} ^{(\bullet)} ) ( \FF ^{\prime(\bullet)} ) = 0$ if and only if 
the morphism
$s\colon \mathcal{H} ^0 (f _{+} ^{(\bullet)} )
((\hdag Z ) ( \E ^{\prime(\bullet)}))  \rightarrow \mathcal{H} ^0 (f _{+} ^{(\bullet)} )(
\mathcal{H} ^{\dag , 1} _{Z} (\E ^{\prime(\bullet)}))$ 
is an epimorphism. 
We split the check of this latter property in the following two steps a) and b).

3.a) In this step, we check that
the morphism $s ':=
    \mathcal{H} ^0(\R \underline{\Gamma} ^\dag _{Y}
    \circ f ^{!(\bullet)})  (s ) $
is an epimorphism.
Since  $(\hdag Z)(\E ^{\prime(\bullet)}) \in 
\smash{\underrightarrow{LM}} _{\Q,\mathrm{ovcoh}} (Y \setminus Z, \fP' /\W)$, 
since the functors
$ \R \underline{\Gamma} ^\dag _{Y}
    \circ f ^{!(\bullet)} $
and
$\R \underline{\Gamma} ^\dag _{Y\setminus Z}
    \circ f ^{!(\bullet)}$
    are canonically isomorphic  
    over 
    $\smash{\underrightarrow{LD}} ^{\mathrm{b}} _{\Q,\mathrm{ovcoh}} (Y \setminus Z, \fP /\W)$
then following the step 1), 
the canonical morphism
$$(\hdag Z)(\E ^{\prime(\bullet)}) \rightarrow 
\mathcal{H} ^0( \R \underline{\Gamma} ^\dag _{Y}
    \circ f ^{!(\bullet)} )\circ \mathcal{H} ^0(f _{+} ^{(\bullet)}) ((\hdag Z)(\E ^{\prime(\bullet)}) )$$ 
    is an isomorphism.
Since 
$\mathcal{H} ^{\dag , 1} _{Z} (\E ^{\prime(\bullet)}) \in 
\smash{\underrightarrow{LM}} _{\Q,\mathrm{ovcoh}} (Z , \fP' /\W)$,
since the functors
$ \R \underline{\Gamma} ^\dag _{Y}
    \circ f ^{!(\bullet)} $
and
$\R \underline{\Gamma} ^\dag _{Z}
    \circ f ^{!(\bullet)}$
    are canonically isomorphic  
    over 
    $\smash{\underrightarrow{LD}} ^{\mathrm{b}} _{\Q,\mathrm{ovcoh}} (Z , \fP /\W)$
then by induction hypothesis the canonical morphism
$$\mathcal{H} ^{\dag ,1} _{Z} (\E ^{\prime(\bullet)}) 
\rightarrow
    \mathcal{H} ^0(\R \underline{\Gamma} ^\dag _{Y}
    \circ f ^{!(\bullet)} ) \circ \mathcal{H} ^0(f _{+} ^{(\bullet)} )(\mathcal{H} ^{\dag ,1} _{Z} (\E ^{\prime(\bullet)}))$$ 
    is an isomorphism.
Since $(\hdag Z)(\E ^{\prime(\bullet)}) \rightarrow \mathcal{H} ^{\dag ,1} _{Z} (\E ^{\prime(\bullet)})$ 
is an epimorphism,
this yields that so is $s '$.

\bigskip

3.b)
Let us check that  $s $ is an epimorphism.
Let  
$\FF ^{(\bullet)} 
\in
\smash{\underrightarrow{LM}} _{\Q,\mathrm{ovcoh}} (Y, \fP /\W)$ 
be the image of  $s $, and $i $ be the canonical monomorphism 
$\FF ^{(\bullet)}  
\hookrightarrow 
\mathcal{H} ^0 (f _{+} ^{(\bullet)} )( \mathcal{H} ^{\dag , 1} _{Z} (\E ^{\prime(\bullet)}))$.
Since  
$\mathcal{H} ^0 (f _{+} ^{(\bullet)} )( \mathcal{H} ^{\dag , 1} _{Z} (\E ^{\prime(\bullet)}))$
has his support in $Z$, then $i $ is in fact a monomorphism of 
$\smash{\underrightarrow{LM}} _{\Q,\mathrm{ovcoh}} (Z, \fP /\W)$.
Using the induction hypothesis,
since the functors
$ \R \underline{\Gamma} ^\dag _{Y}
    \circ f ^{!(\bullet)} $
and
$\R \underline{\Gamma} ^\dag _{Z}
    \circ f ^{!(\bullet)}$
    are canonically isomorphic  
    over 
    $\smash{\underrightarrow{LD}} ^{\mathrm{b}} _{\Q,\mathrm{ovcoh}} (Z , \fP /\W)$
this yields that 
$i ' : =\mathcal{H} ^0( \R \underline{\Gamma} ^\dag _{Y}
    \circ f ^{!(\bullet)} ) (i )$
    is a monomorphism.
Since  $s '$ is an epimorphism, then so is $i '$. 
Hence, the morphism $i'$ is an isomorphism.
Using the induction hypothesis,
this implies that $i $ is an isomorphism.
This yields that $s$ is an epimorphism.
    
\medskip

3.c) Hence, we have checked that 
for any integer 
$r \not= 0$, 
we have 
$\mathcal{H} ^r (f _{+} ^{(\bullet)} ) ( \FF ^{\prime(\bullet)} ) = 0$. 
From \ref{s.e.loca}, we get the exact sequence 
    $0 \rightarrow \mathcal{H} ^{\dag , 0} _{Z} (\E ^{\prime(\bullet)} ) \rightarrow \E ^{\prime(\bullet)} \rightarrow \FF ^{\prime(\bullet)} \rightarrow 0 .$
By applying the functor  $f _{+} ^{(\bullet)}$ 
to this latter sequence, we get a long exact sequence.
Looking at this later one, we remark that the property 
``for any $r \neq 0$, $\mathcal{H}
    ^r (f _{+} ^{(\bullet)} ) ( \FF ^{\prime(\bullet)} ) = 0$ 
    and $\mathcal{H} ^r ( f _{+} ^{(\bullet)} ) (\mathcal{H} ^{\dag , 0} _{Z} (\E ^{\prime(\bullet)} ))=0$'',
implies that 
    "for any $r \neq 0$, $\mathcal{H} ^r (f _{+} ^{(\bullet)} ) (\E ^{\prime(\bullet)} ) = 0$". 
    \medskip

    4) Similarly to the step 3), we check that for any  $r\neq 0$, for any 
    $\E ^{(\bullet)} \in \smash{\underrightarrow{LM}}  _{\Q,\mathrm{ovcoh}} (Y , \fP /\W)$, 
    we have 
    $\mathcal{H} ^r (\R \underline{\Gamma} ^{\dag} _{Y}\circ  f ^{!(\bullet)}  )  ( \E ^{(\bullet)}) = 0$. 
    \medskip

5)    
It remains to check the last statement of the theorem. 
Let $\E ^{(\bullet)}
\in \smash{\underrightarrow{LD}} ^{\mathrm{b}} _{\Q,\mathrm{ovcoh}} (Y , \fP /\W)$.
    Using the localisation triangle with respect to $Z$, 
     to check that the morphism
     $f _{+} ^{(\bullet)}\circ \R \underline{\Gamma} ^{\dag} _{{Y} }\circ  f ^{!(\bullet)}( \E   ^{(\bullet)} ) \rightarrow  \E   ^{(\bullet)}  $
     is an isomorphism, we reduce to check we get an isomorphism after applying 
     $\R \underline{\Gamma} ^{\dag} _{{Z} }$ and $(\hdag Z)$.
     Using \ref{2.2.18} and \ref{gammayY'}, 
     after applying $\R \underline{\Gamma} ^{\dag} _{{Z} }$, we get a morphism canonically
     isomorphic to the canonical morphism 
     $f _{+} ^{(\bullet)}\circ \R \underline{\Gamma} ^{\dag} _{{Z} }\circ  f ^{!(\bullet)}
     ( \R \underline{\Gamma} ^{\dag} _{{Z} } \E   ^{(\bullet)} ) \rightarrow  
     \R \underline{\Gamma} ^{\dag} _{{Z} } \E   ^{(\bullet)}  $. By induction hypothesis, this latter is an isomorphism.
     Moreover, after applying $(\hdag Z)$, we get the morphism
     $f _{+} ^{(\bullet)}\circ \R \underline{\Gamma} ^{\dag} _{{Y  \setminus Z} }\circ  f ^{!(\bullet)}
     ( \R \underline{\Gamma} ^{\dag} _{{Y  \setminus Z} } \E   ^{(\bullet)} ) \rightarrow  
     \R \underline{\Gamma} ^{\dag} _{{Y  \setminus Z} } \E   ^{(\bullet)}  $, which is an isomorphism following the step 1).

We proceed similarly to check 
that the canonical morphism
     $\E   ^{\prime (\bullet)}  
     \rightarrow  
\R \underline{\Gamma} ^{\dag} _{{Y} }\circ  f ^{!(\bullet)} \circ f _{+} ^{(\bullet)}( \E   ^{\prime (\bullet)} )      $
     is an isomorphism for any 
     $\E ^{\prime (\bullet)}
\in \smash{\underrightarrow{LD}} ^{\mathrm{b}} _{\Q,\mathrm{ovcoh}} (Y , \fP '/\W)$.     
\end{proof}

\begin{cor}
\label{ind-cat-overcouples}
Let $\mathfrak{C}$ be a data of coefficients over $(\V,\V ^\flat)$
which contains $\fB _\mathrm{div}$, 
which is stable  under devissages, 
pushforwards, extraordinary pullbacks,
and local cohomological functors.
Let $\mathbb{Y}:= (Y,X)$ be a couple over $\W$. 

\begin{enumerate}
\item Choose a frame of the form $(Y, X,\fP)$. 
The category 
$\mathfrak{C} (Y, \fP/\W)$
 does not depend, up to a canonical equivalence of categories, 
on the choice of the frame $(Y,X, \fP)$ over $\W$ enclosing $(Y,X)$.
Hence, we can simply write 
$\mathfrak{C} (\mathbb{Y}/\W)$ 
instead of 
$\mathfrak{C} (Y, \fP/\W)$
without ambiguity (up to canonical equivalence of categories).

\item If moreover $\fC$ is stable under cohomology, 
then we get a canonical t-structure on 
$\mathfrak{C} (\mathbb{Y}/\W)$. 
\end{enumerate}

\end{cor}

\begin{proof}
Let $(Y,X, \fP )$ and $(Y,X, \fP ')$ be
two frames over $\W$ enclosing $(Y,X)$.
The closed immersions 
$X \hookrightarrow \fP $
and
$X \hookrightarrow \fP '$
induce
$X \hookrightarrow \fP  \times \fP '$.
Denoting by 
$\pi _1\colon \fP  \times \fP ' \to \fP $
and 
$\pi _2\colon \fP  \times \fP ' \to \fP '$
the structural projections, 
we get two morphisms of frames over $\W$ of the form
$(id,id, \pi _1) \colon 
(Y,X, \fP  \times \fP ')
\to 
(Y,X, \fP )$
and 
$(id,id, \pi _2) \colon 
(Y,X, \fP  \times \fP ')
\to 
(Y,X, \fP ')$.
From \ref{ind-CYW}, 
the functors 
$\pi ^{(\bullet)} _{2 +} \R \underline{\Gamma} ^{\dag} _{Y} \pi _1 ^{!(\bullet )} $
and
$ \pi _{1+} ^{(\bullet)} \R \underline{\Gamma} ^{\dag} _{Y} \pi _2 ^{!(\bullet)}$
are canonically quasi-inverse equivalences of categories 
between 
$\mathfrak{C} (Y, \fP/\W)$
and 
$\mathfrak{C} (Y, \fP '/\W)$.
When $\fC$ is stable under cohomology then these equivalences 
are t-exact. Hence we are done.
\end{proof}

\begin{lem}
\label{ind-dual}
Let $\mathfrak{C}$ be a data of coefficients over $(\V,\V ^\flat)$
which contains $\fB _\mathrm{div}$, 
which is stable  under devissages, 
pushforwards, extraordinary pullbacks, 
local cohomological functors, and duals.
Let $\mathbb{Y}:= (Y,X)$ be a couple over $\W$. 
 Choose a frame of the form $(Y, X,\fP)$. 
The functor 
$\R \underline{\Gamma} ^\dag _{Y} \DD _{\fP} 
\colon 
\mathfrak{C} (Y,\fP/\W) \to \mathfrak{C} (Y,\fP/\W)$
does not depend,
up to canonical equivalences of categories of 
\ref{ind-cat-overcouples} (more precisely, we have the commutative diagram
\ref{ind-dual-diag1} up to canonical isomorphisms), 
on the choice of the frame enclosing $(Y,X)$. 
Hence, 
we will denote by 
$\DD _{\mathbb{Y}}
\colon \mathfrak{C} (\mathbb{Y}/\W) \to \mathfrak{C} (\mathbb{Y}/\W)$
the functor 
$\R \underline{\Gamma} ^\dag _{Y} \DD _{\fP}$. 
\end{lem}

\begin{proof}
As in the beginning of the proof, 
\ref{ind-cat-overcouples},
let $(Y,X, \fP _1)$ and $(Y,X, \fP _2)$ be
two frames over $\W$ enclosing $(Y,X)$.
Let $\pi _1\colon \fP _1 \times \fP _2 \to \fP _1$
and 
$\pi _2\colon \fP _1 \times \fP _2 \to \fP _2$
be the structural projections.
We have to check that the diagram
\begin{equation}
\label{ind-dual-diag1}
\xymatrix{
{\mathfrak{C} (Y, \fP_1/\W)} 
\ar[r] _-{\R \underline{\Gamma} ^{\dag} _{Y} \pi _1 ^{!(\bullet )} } ^-{\cong}
\ar[d] _-{\R \underline{\Gamma} ^\dag _{Y} \DD _{\fP _1} }
&
{\mathfrak{C} (Y, \fP_1\times \fP _2/\W)} 
\ar[r] _-{\pi ^{(\bullet)} _{2 +} } ^-{\cong}
\ar[d] ^-{\R \underline{\Gamma} ^\dag _{Y} \DD _{\fP _1 \times \fP _2} }
& 
{\mathfrak{C} (Y, \fP_2/\W)} 
\ar[d] ^-{\R \underline{\Gamma} ^\dag _{Y} \DD _{\fP _2} }
\\ 
{\mathfrak{C} (Y, \fP_1/\W) } 
\ar[r] ^-{ \R \underline{\Gamma} ^{\dag} _{Y} \pi _1 ^{!(\bullet )} } _-{\cong}
&
{\mathfrak{C} (Y, \fP_1\times \fP _2/\W)} 
\ar[r] ^-{\pi ^{(\bullet)} _{2 +}} _-{\cong}
& 
{\mathfrak{C} (Y, \fP_2/\W) } 
}
\end{equation}
is commutative, up to canonical isomorphisms.
Let 
$\E ^{(\bullet)}
\in 
\mathfrak{C} (Y, \fP_1\times \fP _2/\W) $.
From
\ref{rel-dual-isom}, we have the isomorphism
$\DD _{\fP _2} \pi ^{(\bullet)} _{2+} ( \E ^{(\bullet)})
\riso 
\pi ^{(\bullet)} _{2+} \DD _{\fP _1 \times \fP _2}  ( \E ^{(\bullet)}) $.
Hence, by applying the functor $\R \underline{\Gamma} ^{\dag} _{Y}$ to this isomorphism, 
we get the first one 
$\R \underline{\Gamma} ^{\dag} _{Y}\DD _{\fP _2} \pi ^{(\bullet)} _{2+} ( \E ^{(\bullet)})
\riso 
 \R \underline{\Gamma} ^{\dag} _{Y}\pi ^{(\bullet)} _{2+} \DD _{\fP _1 \times \fP _2}  ( \E ^{(\bullet)}) 
\underset{\ref{2.2.18}}{\riso}  
\pi ^{(\bullet)} _{2+}  \R \underline{\Gamma} ^{\dag} _{\pi _2 ^{-1}(Y)} \DD _{\fP _1 \times \fP _2}  ( \E ^{(\bullet)}) $.
Since $\fP _1$ is a realizable smooth formal scheme over $\W$,
then we check that the immersion
$Y \hookrightarrow\pi _2 ^{-1}(Y)$ is in fact a closed immersion.
This yields $Y = \overline{Y}\cap \pi _2 ^{-1}(Y)$, where 
$\overline{Y}$ is the closure of $Y$ in $P _1 \times P _2$.
Since  $ \DD _{\fP _1 \times \fP _2}  ( \E ^{(\bullet)}) $
has in support in $\overline{Y}$,
then 
$\R \underline{\Gamma} ^{\dag} _{\pi _2 ^{-1}(Y)} \DD _{\fP _1 \times \fP _2}  ( \E ^{(\bullet)}) 
\riso 
\R \underline{\Gamma} ^{\dag} _{\pi _2 ^{-1}(Y)}
\R \underline{\Gamma} ^{\dag} _{\overline{Y}} \DD _{\fP _1 \times \fP _2}  ( \E ^{(\bullet)}) 
\riso 
\R \underline{\Gamma} ^{\dag} _{Y} \DD _{\fP _1 \times \fP _2}  ( \E ^{(\bullet)}) 
$.
Hence, we have checked the commutativity, up to commutative isomorphisms,
of the right square of \ref{ind-dual-diag1}. 
From \ref{ind-CYW},  $\pi ^{(\bullet)} _{1+}$ is canonically a quasi-inverse 
of the equivalence of categories $\R \underline{\Gamma} ^{\dag} _{Y} \pi _1 ^{!(\bullet )} 
\colon 
\mathfrak{C} (Y, \fP_1\times \fP _2/\W)
\cong
\mathfrak{C} (Y, \fP_1/\W) $
(we means that we have canonical isomorphisms
$\pi ^{(\bullet)} _{1+} \R \underline{\Gamma} ^{\dag} _{Y} \pi _1 ^{!(\bullet )} 
\riso id$
and 
$id \riso 
\R \underline{\Gamma} ^{\dag} _{Y} \pi _1 ^{!(\bullet )} \pi ^{(\bullet)} _{1+}$).
Hence, 
we get the commutativity, up to canonical isomorphism,
of the left square of \ref{ind-dual-diag1}. 
\end{proof}

\begin{lem}
\label{ind-pushforward-extinv}
Let $\mathfrak{C}$ be a data of coefficients over $(\V,\V ^\flat)$
which contains $\fB _\mathrm{div}$, 
which is stable  under devissages, 
pushforwards, extraordinary pullbacks,
and local cohomological functors.
Let  $u=(b,a)\colon (Y', X') \to (Y, X)$ be a morphism of couples over $\W$.
Put $\mathbb{Y}:= (Y,X)$ and $\mathbb{Y}':= (Y',X')$.
Let us choose a morphism of frames $\theta= (b,a,f) \colon (Y', X', \fP')\to (Y,X,\fP)$ over $\W$ enclosing $u$.

\begin{enumerate}
\item 
The functor $\theta ^{!(\bullet )} := \R \underline{\Gamma} ^\dag _{Y'} \circ f ^{!(\bullet)}
\colon
\mathfrak{C} (Y,\fP/\W) \to \mathfrak{C} (Y',\fP'/\W)$ 
does not depend on the choice of such $\theta$ enclosing $u$
(up to canonical equivalences
of categories).
Hence, it will be denoted by $u ^{!}
\colon \mathfrak{C} (\mathbb{Y}/\W) \to \mathfrak{C} (\mathbb{Y}'/\W)$. 

\item Suppose that $u$ is complete, i.e. that $a \colon X' \to X$ is proper. 
The functor $\theta _{+}:= f _+ ^{(\bullet )}
\colon 
\mathfrak{C} (Y',\fP'/\W) \to \mathfrak{C} (Y,\fP/\W)$
does not depend on the choice of such $\theta$ enclosing $u$
(up to canonical equivalences
of categories).
Hence, it will be denoted by 
$u _{+} \colon \mathfrak{C} (\mathbb{Y}'/\W) \to \mathfrak{C} (\mathbb{Y}/\W)$.
\end{enumerate}

\end{lem}

\begin{proof}
To check the first assertion, 
we proceed as in the proof of 
\ref{ind-dual}
(use also the commutation of local cohomological functors with extraordinary inverse images 
given in \ref{2.2.18}).
Let us check that the functor 
$f _+ ^{(\bullet )}
\colon 
\mathfrak{C} (Y',\fP'/\W) \to \mathfrak{C} (Y,\fP/\W)$
is well defined.
Let $\E ^{(\bullet)} \in \mathfrak{C} (Y',\fP'/\W) $.
Since $a$ is proper, then 
$f _+ ^{(\bullet )} (\E ^{(\bullet)} ) \in \mathfrak{C} (\fP) $.
We compute 
$\R \underline{\Gamma} ^\dag _{Y}f _+ ^{(\bullet )} (\E ^{(\bullet)} ) 
\riso 
f _+ ^{(\bullet )} \R \underline{\Gamma} ^\dag _{ f ^{-1}Y} (\E ^{(\bullet)} ) $.
Since $Y'$ is included in $f ^{-1}Y$ and 
$\E ^{(\bullet)} \in \mathfrak{C} (Y',\fP'/\W) $,
then 
$\R \underline{\Gamma} ^\dag _{ f ^{-1}Y} (\E ^{(\bullet)} ) 
\riso 
\E ^{(\bullet)}$.
Hence, 
$\R \underline{\Gamma} ^\dag _{Y}f _+ ^{(\bullet )} (\E ^{(\bullet)} ) 
\riso f _+ ^{(\bullet )} (\E ^{(\bullet)} )$, which implies that
$f _+ ^{(\bullet )} (\E ^{(\bullet)} ) \in  \mathfrak{C} (Y,\fP/\W)$.
To check that the functor does not depend on the choice of $\theta$ enclosing $u$, 
we proceed as in the proof of 
\ref{ind-dual}.
\end{proof}

\begin{lem}
\label{ind-prod-tensor}
Let $\mathfrak{C}$ be a data of coefficients over $(\V,\V ^\flat)$
which contains $\fB _\mathrm{div}$, 
which is stable under devissages, 
pushforwards, extraordinary pullbacks, 
and tensor products.
Let $\mathbb{Y}:= (Y,X)$ be a couple over $\W$. 
 Choose a frame of the form $(Y, X,\fP)$. 
The bifunctor $-\smash{\widehat{\otimes}}^\L
_{\O  _{\fP}} - [-\dim P]
\colon
\mathfrak{C} (Y,\fP/\W) \times \mathfrak{C} (Y,\fP/\W) \to \mathfrak{C} (Y,\fP/\W)$
does not depend, up to the canonical equivalences of categories of \ref{ind-cat-overcouples}, 
on the choice of the frame enclosing $(Y,X)$.
It will be denoted by 
$\widetilde{\otimes} _{\mathbb{Y}}
\colon 
\mathfrak{C} (\mathbb{Y}/\W) \times \mathfrak{C} (\mathbb{Y}/\W) \to \mathfrak{C} (\mathbb{Y}/\W)$. 
\end{lem}

\begin{proof}
From Lemmas  \ref{rem-div-cst}.\ref{S(D,C)stability3-pre}
and 
\ref{S(D,C)stability3bis}, 
the data of coefficients 
$\mathfrak{C}$  is also stable under local cohomological functors.
From \ref{f!T'Totimes} 
(resp. \ref{fonctYY'Gamma-iso}), 
extraordinary inverse images 
(resp. local cohomological functors) 
commute with tensor products (up to a shift).
Proceeding as in the proof of \ref{ind-dual} with its notation,
$\R \underline{\Gamma} ^{\dag} _{Y} \pi _1 ^{!(\bullet )} $ 
and 
$\R \underline{\Gamma} ^{\dag} _{Y} \pi _2 ^{!(\bullet )} $ 
commute with tensor products
and then so are 
$\pi ^{(\bullet )} _{1+}$ and $\pi ^{(\bullet )} _{2+}$.
\end{proof}

\begin{empt}
[Formalism of Grothendieck six operations]
\label{6operations}
Let $\mathfrak{C}$ be a data of coefficients over $(\V,\V ^\flat)$
which contains $\fB _\mathrm{div}$, 
which is 
stable  under devissages, 
pushforwards, extraordinary pullbacks, 
duals,
and tensor products.
To sum-up the above Lemmas
we can define a formalism of Grothendieck six operations on couples 
as follows.
Let  $u=(b,a)\colon (Y', X') \to (Y, X)$ be a morphism of couples over $\W$.
Put $\mathbb{Y}:= (Y,X)$ and $\mathbb{Y}':= (Y',X')$.
\begin{enumerate}

\item We have the dual functor $\DD _{\mathbb{Y}}
\colon \mathfrak{C} (\mathbb{Y}/\W) \to \mathfrak{C} (\mathbb{Y}/\W)$ (see \ref{ind-dual}).

\item 
We have the extraordinary pullback $u ^{!}
\colon
\mathfrak{C} (\mathbb{Y}/\W) \to \mathfrak{C} (\mathbb{Y}'/\W)$ (see \ref{ind-pushforward-extinv}).
We get the pullbacks $u ^{+}:= \DD _{\mathbb{Y}'} \circ u ^{!} \circ \DD _{\mathbb{Y}}$.

\item Suppose that $u$ is complete.
Then, we have the functor $u _+
\colon \mathfrak{C} (\mathbb{Y}'/\W) \to \mathfrak{C} (\mathbb{Y}/\W)$ (see \ref{ind-pushforward-extinv}). 
We denote by  $u _{!}:= \DD _{\mathbb{Y}} \circ u _{+} \circ \DD _{\mathbb{Y}'}$, 
the extraordinary pushforward by $u$.

\item 
We have the tensor product 
$-\widetilde{\otimes} _{\mathbb{Y}}-
\colon 
\mathfrak{C} (\mathbb{Y}/\W) \times \mathfrak{C} (\mathbb{Y}/\W) \to \mathfrak{C} (\mathbb{Y}/\W)$
(see \ref{ind-prod-tensor})
\end{enumerate}

\end{empt}

\begin{exs}
\label{nota-h-ovhol}
\begin{enumerate}
\item  We recall
the data of coefficients $\smash{\underrightarrow{LD}} ^{\mathrm{b}} _{\Q,\mathrm{ovhol}}$
and $\smash{\underrightarrow{LD}} ^{\mathrm{b}} _{\Q,\mathrm{h}}$
are defined respectively in \ref{ex-cst-surcoh}.\ref{hstab} and \ref{ex-cst-surcoh}.\ref{ovholstab}.
Using Lemmas \ref{lem-stabextpullback} and \ref{S(D,C)stability} (and  \ref{rem-div-cst2}),
they 
are stable 
under local cohomological functors, 
pushforwards, extraordinary pullbacks, and duals.
Hence, 
with the notation 
\ref{ind-cat-overcouples},
using Lemmas \ref{ind-pushforward-extinv}, \ref{ind-prod-tensor}, and \ref{ind-dual},
for any frame 
$(Y,X, \fP)$ over $\W$,
we get the categories of the forms 
$\smash{\underrightarrow{LD}} ^{\mathrm{b}} _{\Q,\mathrm{h}} (Y, \fP /\W)$,
$\smash{\underrightarrow{LD}} ^{\mathrm{b}} _{\Q,\mathrm{h}}(\mathbb{Y}/\W)$,
$\smash{\underrightarrow{LD}} ^{\mathrm{b}} _{\Q,\mathrm{ovhol}} (Y, \fP /\W)$
or
$\smash{\underrightarrow{LD}} ^{\mathrm{b}} _{\Q,\mathrm{ovhol}}(\mathbb{Y}/\W)$
endowed with five of Grothendieck cohomological operations (the tensor product is a priori missing).

\item Following theorem \ref{dfnqupropbis},
there exist a data of coefficients $T$ 
which contains 
$\mathfrak{M} _{\mathrm{div}}$,
is 
local,
stable by devissages, direct summands, 
local cohomological functors, 
pushforwards, extraordinary pullbacks, 
base change, tensor products, duals,
cohomology
and
special descent of the base. 
Hence, 
for any frame 
$(Y,X, \fP)$ over $\W$,
we get the triangulated category 
$T (Y, \fP /\W)$ or $T(\mathbb{Y}/\W)$,
endowed with a t-structure and a formalism of Grothendieck six operations.

\end{enumerate}

\end{exs}

\subsection{Formalism of Grothendieck six operations over realizable varieties}

\begin{dfn}
[Proper compactification]
\begin{enumerate}
\item A frame $(Y,X,\fP)$ over $\W$ is said to be {\it proper} if $\fP$ is proper. 
The category of proper frames over $\W$ is the full subcategory of the category
of frames over $\W$ whose objects are proper frames over $\W$.

\item The category of {\it proper couples} over $\W$ is the full subcategory of 
the category of couples over $\W$ whose objects $(Y, X)$ are such that 
$X$ is proper. 
We remark that if $(Y, X)$ is a proper couple over $\W$  then 
there exists a proper frame over $\W$ of the form $(Y,X, \fP)$.

\item A {\it realizable variety} over $\W$ is a $l$-scheme $Y$ such that there exists a proper frame over $\W$ of the form 
$(Y,X,\fP)$. For such frame $(Y,X,\fP)$, we say that 
the proper frame $(Y,X,\fP)$ encloses $Y$ or that the proper couple $(Y,X)$ encloses $Y$. 
\end{enumerate}

\end{dfn}

\begin{empt}
[Formalism of Grothendieck six operations]
\label{6operations-variety}
Let $\mathfrak{C}$ be a data of coefficients over $(\V,\V ^\flat)$
which contains $\fB _\mathrm{div}$, 
which is 
stable  under devissages, 
pushforwards, extraordinary pullbacks, 
duals,
and tensor products.
Similarly to Lemma \ref{ind-cat-overcouples}, we check using Theorem \ref{ind-CYW}
that the category 
$\mathfrak{C} (Y, \fP/\W)$
(resp. $\mathfrak{C} (Y,X/\W)$)
 does not depend, up to a canonical equivalence of categories, 
on the choice of the proper frame $(Y,X, \fP)$ (resp. the proper couple $(Y,X)$) over $\W$ enclosing $Y$.
Hence, we simply denote it by
$\mathfrak{C} (Y/\W)$.
As for \ref{6operations}, we can define a formalism of Grothendieck six operations on realizable varieties 
as follows.
Let  $u\colon Y'\to Y$ be a morphism of realizable varieties over $\W$.
\begin{enumerate}

\item We have the dual functor $\DD _{Y}
\colon \mathfrak{C} (Y/\W) \to \mathfrak{C} (Y/\W)$ (see \ref{ind-dual}).

\item 
We have the extraordinary pullback $u ^{!}
\colon
\mathfrak{C} (Y/\W) \to \mathfrak{C} (Y'/\W)$ (see \ref{ind-pushforward-extinv}).
We get the pullbacks $u ^{+}:= \DD _{Y'} \circ u ^{!} \circ \DD _{Y}$.

\item 
We have the functor $u _+
\colon \mathfrak{C} (Y'/\W) \to \mathfrak{C} (Y/\W)$ (see \ref{ind-pushforward-extinv}). 
We denote by  $u _{!}:= \DD _{Y} \circ u _{+} \circ \DD _{Y'}$, 
the extraordinary pushforward by $u$.

\item 
We have the tensor product 
$-\widetilde{\otimes} _{Y}-
\colon 
\mathfrak{C} (Y/\W) \times \mathfrak{C} (Y/\W) \to \mathfrak{C} (Y/\W)$
(see \ref{ind-prod-tensor})
\end{enumerate}

\end{empt}

\subsection{Constructible t-structure}

For completeness,
we introduce the notion of constructibility.
Let $\mathfrak{C}$ be a data of coefficients over $(\V,\V ^\flat)$
which contains $\fB _\mathrm{div}$, 
which is 
stable  under devissages, 
pushforwards, extraordinary pullbacks, 
duals,
tensor products, and cohomology.

\begin{empt}
[Constructible t-structure]
\label{t-structure}

Let $\mathbb{Y}:= (Y,X)$ be a couple. Choose a frame $(Y, X,\fP)$. 
If $Y' \rightarrow Y$ is an immersion, then 
we denote by 
$i _{Y'} \colon (Y', X', \fP) \to  (Y, X, \fP)$ the induced morphism
where $X'$ is the closure of $Y'$ in $X$. 
We define on $\fC(\mathbb{Y}/\W)$
the constructible t-structure as follows. 
An object $\E\in \fC(\mathbb{Y}/\W) $
belongs to 
$\fC^{c, \geq 0}(\mathbb{Y}/\W)$ 
(resp. $\fC^{c, \leq 0}(\mathbb{Y}/\W)$)
if there exists 
a special morphism 
$(\W, \W ^\flat)
\to 
(\W ', \W ^\flat)$ such that,
denoting by 
$\fP ' := \fP \times _{\Spf \W} \Spf \W '$,
$Y':= (Y \times _{\Spf \W} \Spf \W ') _{\mathrm{red}}$ and
$\E^{\prime (\bullet) }
:=
 \W'  \smash{\widehat{\otimes}}^\L
_{\W}  \E^{(\bullet) }$,
there exists 
a smooth stratification 
(see Definition \cite[2.2.1]{Abe-Caro-weights})
$(Y ' _{i}) _{i=1, \dots , r}$ of $Y'$ such that 
for any $i $, the complex $i _{Y '_i} ^{+} (\E^{\prime (\bullet) } ) [ d _{Y '_i}]$ (see notation \ref{6operations})
belongs to
$\fC ^{\geq 0}_{\textrm{isoc}}(Y ' _i,\fP'/\W)$
(resp. 
$\fC ^{\leq 0}_{\textrm{isoc}}(Y ' _i,\fP'/\W)$).

\end{empt}

\begin{prop}
Let $\mathbb{Y}:= (Y,X)$ be a couple.
\begin{enumerate}

\item Let $\E ^{\prime(\bullet)}\to \E \to \E ^{\prime \prime(\bullet)} \to \E ^{\prime(\bullet)} [1]$ be an exact triangle in $\fC(\mathbb{Y}/\W) $. 
If $\E ^{\prime(\bullet)}$ and $\E ^{\prime \prime(\bullet)}$are in $\fC^{c, \geq 0}(\mathbb{Y}/\W)$
(resp. $\fC^{c, \leq 0}(\mathbb{Y}/\W)$)
then so is $\E$.

\item 
Suppose that $Y$ is smooth. 
Let $\E\in \fC _{\textrm{isoc}}(\mathbb{Y}/\W) $.
Then $\E \in \fC^{c, \geq 0}(\mathbb{Y}/\W)$
(resp. $\E \in \fC^{c, \leq 0}(\mathbb{Y}/\W)$)
if and only if 
$\E \in \fC^{\geq d _X}_{\textrm{isoc}}(\mathbb{Y}/\W)$
(resp. $\E \in \fC^{\leq d _X}_{\textrm{isoc}}(\mathbb{Y}/\W)$). 

\end{enumerate}
\end{prop}

\begin{proof}
This is left to the reader. 
\end{proof}

\begin{rem}
When $k$ is perfect, 
we retrieve the constructibility as defined for overholonomic complexes in 
\cite[2.4]{caro-unip} or that of T. Abe.
\end{rem}

\small
\bibliographystyle{alpha}
\newcommand{\etalchar}[1]{$^{#1}$}
\def\cprime{$'$}

\bigskip
\noindent Daniel Caro\\
Laboratoire de Mathématiques Nicolas Oresme\\
Université de Caen
Campus 2\\
14032 Caen Cedex\\
France.\\
email: daniel.caro@unicaen.fr

\newpage

\end{document}